%% file: main.tex
\pdfoutput=1
\DeclareSymbolFont{AMSb}{U}{msb}{m}{n}
\documentclass[12pt, english, letterpaper]{smfbook}
\usepackage[left=1.4in, right=1.4in, bottom=1.2in, top=1.2in]{geometry}
\usepackage{mathtools}
\usepackage{braket}
\usepackage{enumitem}
\usepackage[bitstream-charter,expert]{mathdesign}
\usepackage[scaled=.96,osf]{XCharter}
\usepackage[tracking]{microtype}

\usepackage[dvipsnames]{xcolor}

\usepackage{tikz-cd}
\usepackage{tikz}
\usetikzlibrary{decorations.pathmorphing}
\usetikzlibrary{positioning}
\usepackage{stmaryrd}
\usepackage{comment}
\usepackage[utf8]{inputenc}
\usepackage{CJKutf8}

\usepackage{hyperref}
\hypersetup{
  colorlinks = true,
  linkcolor = Fuchsia,
  urlcolor = Fuchsia,
  citecolor = ForestGreen,
  linkbordercolor = {white}}

\newtheoremstyle{pineapple}%
  {1em}{1em}%
  {\itshape}{}%
  {\bfseries}{. ---}
  {0.5em}{}

\newtheoremstyle{durian}%
  {1em}{1em}%
  {}{}%
  {\bfseries}{. ---}
  {0.5em}{}

\theoremstyle{pineapple}
\newtheorem{IntroTheorem}{Theorem}

\swapnumbers
\newtheorem{Theorem}[subsection]{Theorem}
\newtheorem{Lemma}[subsection]{Lemma}
\newtheorem{Proposition}[subsection]{Proposition}
\newtheorem{Corollary}[subsection]{Corollary}

\theoremstyle{durian}
\newtheorem{Definition}[subsection]{Definition}
\newtheorem{Remark}[subsection]{Remark}
\newtheorem{Example}[subsection]{Example}

\usepackage{enumitem}
\setlist[1]{labelindent=\parindent}
\setlist[1]{labelsep=0.5em}
\setlist[enumerate,1]{label={\upshape (\roman*)}, ref={\upshape (\roman*)}}

\makeatletter
\let\c@equation\c@subsection
\makeatother
\renewcommand{\theequation}{\textbf{\thesubsection}}

\tikzset{>={Straight Barb[length=2pt,width=4pt]}, commutative diagrams/arrow style=tikz}

\tikzset{
  symbol/.style={
    draw=none,
    every to/.append style={
      edge node={node [sloped, allow upside down, auto=false]{$#1$}}}
  }
}

\newcommand{\punct}[1]{\makebox[0pt][l]{\,#1}} 

\makeatletter
\newcommand*{\coloneqq}{\mathrel{\rlap{%
           \raisebox{0.3ex}{$\m@th\cdot$}}%
           \raisebox{-0.3ex}{$\m@th\cdot$}}%
           =}
\newcommand{\eqqcolon}{=%
           \mathrel{\rlap{%
           \raisebox{0.3ex}{$\m@th\cdot$}}%
           \raisebox{-0.3ex}{$\m@th\cdot$}}}
\makeatother

\DeclarePairedDelimiter{\abs}{\lvert}{\rvert}
\DeclareMathOperator{\Kosz}{Kosz}
\DeclareMathOperator{\idx}{idx}
\DeclareMathOperator{\tr}{tr}
\DeclareMathOperator{\Pic}{Pic}
\DeclareMathOperator{\ch}{ch}
\DeclareMathOperator{\Sing}{Sing}
\DeclareMathOperator{\Gram}{Gram}
\DeclareMathOperator{\Fitt}{Fitt}

\DeclareMathOperator{\Fil}{Fil}
\DeclareMathOperator{\gr}{gr}
\DeclareMathOperator{\SL}{\mathbf{SL}}

\DeclareMathOperator{\Aut}{Aut}
\DeclareMathOperator{\AutSch}{\mathbf{Aut}}

\DeclareMathOperator{\coker}{coker}
\DeclareMathOperator{\image}{image}
\DeclareMathOperator{\corank}{corank}
\DeclareMathOperator{\codim}{codim}
\DeclareMathOperator{\Sym}{Sym}
\DeclareMathOperator{\Div}{Div}
\DeclareMathOperator{\Hom}{Hom}
\DeclareMathOperator{\Ext}{Ext}
\DeclareMathOperator{\HomSch}{\mathbf{Hom}}
\DeclareMathOperator{\pr}{pr}
\DeclareMathOperator{\Fr}{Fr}
\DeclareMathOperator{\id}{id}
\DeclareMathOperator{\colim}{colim}
\DeclareMathOperator{\rank}{rank}

\DeclareMathOperator{\Spec}{Spec}

\DeclareMathOperator{\proj}{proj}

\DeclareMathOperator{\rad}{rad}

\newcommand{\GL}{\mathbf{GL}}

\newcommand{\PP}{\mathbf{P}}
\newcommand{\sO}{\mathcal{O}}

\DeclareMathOperator{\Gr}{Gr}
\newcommand{\Flag}{\mathrm{Flag}}

\newcommand{\kk}{\mathbf{k}}
\newcommand{\qbics}{q\operatorname{\bf\!-bics}}

\newcommand{\parref}[1]{{\bf\ref{#1}}}
\newcommand{\hrefSP}[1]{\href{https://stacks.math.columbia.edu/tag/#1}{#1}}
\newcommand{\citeSP}[1]{\cite[\hrefSP{#1}]{stacks-project}}

\usetikzlibrary{calc}

\newcommand\FanoPlane[1][1cm]{%
\begin{tikzpicture}[
mydot/.style={
  draw,
  circle,
  fill=black,
  inner sep=1.5pt}
]
\draw
  (0,0) coordinate (A) --
  (#1,0) coordinate (B) --
  ($ (A)!.5!(B) ! {sin(60)*2} ! 90:(B) $) coordinate (C) -- cycle;
\coordinate (O) at
  (barycentric cs:A=1,B=1,C=1);
\draw (O) circle [radius=#1*1.717/6];
\draw (C) -- ($ (A)!.5!(B) $) coordinate (LC);
\draw (A) -- ($ (B)!.5!(C) $) coordinate (LA);
\draw (B) -- ($ (C)!.5!(A) $) coordinate (LB);
\foreach \Nodo in {A,B,C,O,LC,LA,LB}
  \node[mydot] at (\Nodo) {};
\end{tikzpicture}%
}

\title{Geometry of \(q\)-bic Hypersurfaces}
\author{Raymond Cheng}
\begin{document}

\linespread{1.15}\selectfont
\input{titlepage}

\linespread{1.5}\selectfont
\input{copyright}
\input{abstract}

\pagenumbering{roman}
\setcounter{page}{1}

{
\hypersetup{linkcolor=black}
\tableofcontents
}
\phantomsection
\input{acknowledgements}\addcontentsline{toc}{chapter}{Acknowledgments}%

\input{dedications}

\pagenumbering{arabic}%
\setcounter{page}{1}%

\include{intro}
\include{forms}

\include{hypersurfaces}

\include{lowdim}

\include{threefolds}


\linespread{1.15}\selectfont
\bibliographystyle{amsalpha}
\bibliography{main}


\linespread{1.5}\selectfont

\appendix
\newcommand{\chaptername}{Appendix}
\include{proj-geom}
\include{representations}

\end{document}

%% file: titlepage.tex
\begin{titlepage}
\begin{center}

\vspace*{4em}
\rule{\textwidth}{1pt} \\
\vspace{1em}
\LARGE
\textbf{Geometry of \emph{q}-bic Hypersurfaces} \\
\vspace{1em}

{\large Raymond Cheng} \\
\rule{0.6\textwidth}{1pt}

\vspace{14em}
\normalsize
Submitted in partial fulfillment of the\\
requirements for the degree of\\
Doctor of Philosophy\\
under the Executive Committee\\
of the Graduate School of Arts and Sciences\\
\vspace{2em}
\textsc{Columbia University}\\
\vspace{1em}
\the\year
\vfill

\end{center}
\end{titlepage}

%% file: copyright.tex
\begin{titlepage}
\begin{center}
\vspace*{46em}

\textcopyright  \,  \the\year\\
Raymond Cheng\\
All Rights Reserved
\vfill

\end{center}
\end{titlepage}

%% file: abstract.tex
\begin{titlepage}

\global\topskip 0pt\relax
\advance\chapterheight 6cm
\vbox to \chapterheight{\centering
\vskip 0pt plus 0.7fil\relax
{\Large\bfseries Abstract}\vskip2pc
{\large Geometry of \(q\)-bic Hypersurfaces}\par
\vspace{0.5em}
{\normalsize\normalfont Raymond Cheng}
\vfil}

Traditional algebraic geometric invariants lose some of their potency in
positive characteristic. For instance, smooth projective hypersurfaces may be
covered by lines despite being of arbitrarily high degree. The purpose of this
dissertation is to define a class of hypersurfaces that exhibits such
classically unexpected properties, and to offer a perspective with which to
conceptualize such phenomena.

Specifically, this dissertation proposes an analogy between the eponymous
\emph{\(q\)-bic hypersurfaces}---special hypersurfaces of degree \(q+1\), with \(q\)
any power of the ground field characteristic, a familiar example given by
the corresponding Fermat hypersurface---and low degree hypersurfaces, especially quadrics and
cubics. This analogy is substantiated by concrete results such as: \(q\)-bic
hypersurfaces are moduli spaces of isotropic vectors for a bilinear form; the
Fano schemes of linear spaces contained in a smooth \(q\)-bic hypersurface are
smooth, irreducible, and carry structures similar to orthogonal Grassmannian;
and the intermediate Jacobian of a \(q\)-bic threefold is purely inseparably
isogenous to the Albanese variety of its smooth Fano surface of lines.

\vspace*{\fill}
\end{titlepage}

%% file: acknowledgements.tex
\vspace*{1.72cm}
\begin{center}
{\Large\bfseries Acknowledgements}
\end{center}
\vspace{3.34cm}

Johan, what a pleasure and privilege it has been. You bring such
clarity, and imbue such spirit to everything that one cannot help but
brave the darkness and indifference of this world in attempts to steal a
bit of truth and beauty. I hope that, some day, I might return with
something that truly delights.

Chao Li, Giulia Sacc\`a, Will Sawin, and Jason Starr: thank you for
agreeing to read a demanding draft of this manuscript, serving on my
defense committee, and sharing helpful comments and questions. Alongside
David Hansen, Eric Katz, Daniel Litt, and Alex Perry: you have taught me
so much and have immensely shaped my outlook on mathematics. I am
grateful for all your guidance.

My fellow travellers Remy van Dobben de Bruyn,
Shizhang Li, Qixiao Ma,
Carl Lian, Monica Marinescu, Dmitrii Pirozhkov,
Noah Olander,
Roy Magen, and
Morena Porzio:
I so cherish our fellowship, and hope that,
in spite of the mountains and oceans that may separate us,
that we might share time and space again.

My dear friends: you have sustained me with your companionship and have brought
me such joy through the years. Among many, I would like to
specially mention:
Daniel Cizin for our late night doorway conversations;
Clara Dolfen for fostering such a warm community;
Sam Mundy for coffee, beer, and much more;
Lauren Tang for sharing language and friendship;
Renatta Picciotto for Moldus III, casatiello, and more;
and finally, to you for everything---thank you.

\clearpage

%% file: dedications.tex

\vspace*{2cm}

\begin{center}
\begin{CJK*}{UTF8}{gbsn}
李庆云，
程在望，
欧淑贞，
献给你们。
\clearpage\end{CJK*}
\end{center}
\clearpage

%% file: intro.tex
\chapter*{Introduction}\label{chapter-intro}

Space over a field of positive characteristic is curved and arranged in
surprising ways. Imagine the projective plane. Over the complex numbers,
it is impossible to find a set of \(7\) points such that every triple is collinear.
Yet, over a field of characteristic \(2\), the Fano plane is just this: see
Figure \parref{figure-fano-plane}. More generally, for any power \(q\) of a
prime \(p\), the \(\mathbf{F}_q\)-rational points and lines in the projective
plane over a field of characteristic \(p\) provide a configuration with
more collinearities than permitted in Euclidean space.

\begin{figure}[h]
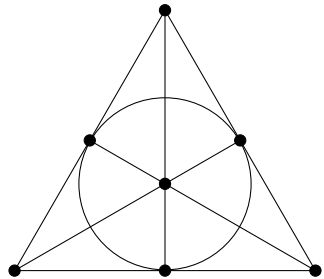

\centering
\FanoPlane[4cm]
\caption{This is a depiction of the \emph{Fano plane}, namely, the arrangement
of \(\mathbf{F}_2\)-rational points and lines in the projective plane \(\PP^2\)
over a field of characteristic \(2\). The arcs indicate the lines and the dots
their intersection points. Curiously, the three midpoints of the triangle are
collinear and their common line is disjoint from the three bisectors.}
\label{figure-fano-plane}
\end{figure}

Fundamentally, these and other exotic shapes arise because systems of
polynomial equations are, in a sense, easier to solve in positive
characteristic than any potential complex counterpart. This leads to a degree
of flexibility in geometric possibilities that is sometimes startling to
classical intuition. Notoriously, first-order differential invariants lose some
of their potency, leading to new possibilities such as inseparable morphisms,
fibrations with singular general fibre, and general type varieties which are
unirational.

Perhaps this is a glimmer to a glittering Atlantis. But in seeking treasure,
one is immediately confronted with two questions: where to look, and what to
look for? With charts of the seascape ahead partial at best, and instruments on
hand designed for other purposes, one navigates by star and by instinct. With
some luck, one returns a bit older but, hopefully, a bit wiser, with a map on
which Lilliput is freshly inked, and maybe---just maybe---with a tiny pearl in
hand.

This dissertation is the humble result of one such expedition.

\section*{Main results}

The purpose of this dissertation is to systematically study the geometry of a
class of hypersurfaces in positive characteristic which I call \emph{\(q\)-bic
hypersurfaces}. The contributions of this work come in two flavours:

First is a systematic development of an intrinsic algebraic and global
geometric theory with which to study \(q\)-bic hypersurfaces. As shall be
indicated below, \(q\)-bic hypersurfaces are of recurring interest, and I hope
the theory developed here brings some unity and clarity to the fascinating
discoveries made by authors before, and provides foundations for studying these
and related objects for authors to come.

Second is the beginnings of an analogy between the projective geometry of
\(q\)-bic hypersurfaces with that of quadric and cubic hypersurfaces. In
the past, the distinctive properties of \(q\)-bic hypersurfaces were often
villainized as ``pathological'', but I hope that this perspective offers a
new light with which their idiosyncrasies may be conceptualized and appreciated.
Looking forward, I hope this can inform the type of questions asked of and
answers sought from the geometry of \(q\)-bic hypersurfaces and other similar
objects.

The remainder of this Section details some concrete statements. But frankly,
the true thesis in this manuscript is that \(q\)-bic hypersurfaces deserve
their name: that they ought to be singled out and viewed as hypersurfaces
belonging to one family; that they are reminiscent of \(q\)uadrics; and that
the homophony is justified. To the extent I have proven this point is for you,
dear reader, to judge.

Throughout, \(q\) is a power of a prime \(p\), and \(\kk\) is a
field of characteristic \(p\).

\subsection*{\(q\)-bic forms}
The first contribution is an intrinsic algebraic theory of objects
I call \emph{\(q\)-bic forms}: over a field \(\kk\), these are bilinear forms
\[ \beta \colon \Fr^*(V) \otimes_\kk V \to \kk \]
between a finite-dimensional vector space \(V\) and its \(q\)-power Frobenius
twist \(\Fr^*(V) \coloneqq \kk \otimes_{\Fr,\kk} V\); equivalently, \(q\)-bic
forms are determined by sesquilinear forms on \(V\) which are \(q\)-linear in
the first variable.

These forms do not enjoy any simple symmetry property; after all, the pairing
relates distinct vector spaces. When the base field is \(\mathbf{F}_{q^2}\),
however, the \(q\)-power Frobenius is an involution and \(\beta\) may be asked
to be \emph{Hermitian} in the sense
\[
\beta(v^{(q)},w) = \beta(w^{(q)},v)^q \quad\text{for all}\; v, w \in V
\]
where \(v^{(q)} \coloneqq 1 \otimes v \in \Fr^*(V)\). For general \(\kk\), as
will be explained in the text, a \(q\)-bic form nevertheless determines a
Hermitian form on a \(\mathbf{F}_{q^2}\)-subvector space in \(V\). This serves
as a powerful arithmetic invariant of \(q\)-bic forms.

What does distinguish a \(q\)-bic form from a general bilinear form is
the canonical \(q\)-linear map \((-)^{(q)} \colon V \to \Fr^*(V)\). This
severely constrains the structure of \(q\)-bic forms, and makes it possible to
give a classification over algebraically closed fields:

\begin{IntroTheorem}
Let \((V,\beta)\) be a \(q\)-bic form over an algebraically closed field.
Then there exists a basis \(V = \langle e_0,\ldots,e_n \rangle\) such that
\[
\operatorname{Gram}(\beta;e_0,\ldots,e_n) =
\mathbf{N}_1^{\oplus a_1} \oplus
\cdots \oplus
\mathbf{N}_m^{\oplus a_m} \oplus
\mathbf{1}^{\oplus b}
\]
for some nonnegative integers \(m\), and \(a_1,\ldots,a_m,b\) satisfying
\(b + \sum_{k = 1}^m k a_k = n+1\).
\end{IntroTheorem}

See \parref{forms-classification-theorem}, and \parref{forms-gram-matrix} and
\parref{forms-standard} for the notation. In particular, over an algebraically
closed field, there are only finitely many isomorphism classes of \(q\)-bic
forms of a given dimension.

\subsection*{\(q\)-bic hypersurfaces}
The subscheme \(X\) of \(\PP V\) parameterizing isotropic vectors for a
nonzero \(q\)-bic form \(\beta \colon \Fr^*(V) \otimes_\kk V \to \kk\)
is a hypersurface of degree \(q+1\): this is the \emph{\(q\)-bic hypersurface}
associated with \(\beta\). An equation for \(X\) is obtained by pairing the
coordinates of \(\PP V\) according to \(\beta\). For example, if there were a
basis \(V = \langle e_0,\ldots,e_n \rangle\) such that
\[
\beta(e_i^{(q)},e_j) =
\begin{dcases*}
1 & if \(i = j\), and \\
0 & if \(i \neq j\),
\end{dcases*}
\]
then in the corresponding coordinates \((x_0:\cdots:x_n)\) of \(\PP V = \PP^n\),
the associated \(q\)-bic hypersurface is the Fermat hypersurface of degree \(q+1\):
\[
X =
\mathrm{V}(x_0^{q+1} + x_1^{q+1} + \cdots + x_n^{q+1}) \subset
\PP^n.
\]

Defined in this way, \(q\)-bic hypersurfaces are clearly similar to quadric
hypersurfaces. This perspective is not entirely new: authors have perennially
observed that \(q\)-bic hypersurfaces may be specified by something like a
bilinear form and that certain geometric features are encoded by the algebra
of the form. The novelty here, and what forms the second contribution of
this work, is a systematic development of this perspective, in which global
geometric properties of the hypersurface \(X\) are intrinsically expressed
through the underlying form \(\beta\). This brings scheme-theoretic methods
to bear and, I believe, clarifies some of the geometric curiosities previously
found about \(q\)-bics. For instance, these methods give an improvement and two
new geometric proofs of Shioda's old observation from \cite{Shioda:Unirational}
that smooth \(q\)-bic surfaces are unirational:

\begin{IntroTheorem}
A smooth \(q\)-bic hypersurface \(X\) of dimension at least \(2\) with a
\(\kk\)-rational line admits a purely inseparable unirational parameterization
of degree \(q\).
\end{IntroTheorem}

See
\parref{hypersurfaces-unirational-smooth}; see
\parref{hypersurfaces-unirational-shioda} for a geometric view on Shioda's
construction; and see \parref{hypersurfaces-unirational-tangent-morphism} for
the new construction.

\smallskip

From this perpsective, the Fano schemes \(\mathbf{F}_r(X)\) classifying
projective \(r\)-planes contained in \(X\) acquire an alternative moduli
interpretation: \(\mathbf{F}_r(X)\) parameterizes \((r+1)\)-dimensional
subspaces of \(V\) which are isotropic for \(\beta\). These Fano schemes
therefore find analogy with orthogonal Grassmannians. The third contribution of
this work is to describe some of the basic geometric features of these schemes,
a summary of which is as follows:

\begin{IntroTheorem}
Let \(X\) be any \(q\)-bic hypersurface of dimension \((n-1)\). Then for each
\(0 < r < \frac{n}{2}\), its Fano scheme \(\mathbf{F}_r(X)\) of \(r\)-planes
\begin{enumerate}
\item is nonempty,
\item has dimension at least \((r+1)(n-2r-1)\),
\item is connected whenever \(n \geq 2r+2\), and
\item is smooth of dimension \((r+1)(n-2r-1)\) at a point corresponding to an
\(r\)-plane \(\PP U \subset X\) contained in the smooth locus of \(X\).
\end{enumerate}
\end{IntroTheorem}

See \parref{hypersurfaces-equations-of-fano}, \parref{hypersurfaces-nonempty-fano},
\parref{hypersurfaces-smooth-point-fano}, and \parref{hypersurfaces-fano-connected}.
This is simplest when \(X\) is smooth:

\begin{IntroTheorem}
Let \(X\) be a smooth \(q\)-bic hypersurface of dimension \((n-1)\). Then its
Fano scheme \(\mathbf{F}_r(X)\) of \(r\)-planes is smooth, irreducible, of
dimension
\[
\dim\mathbf{F}_r(X) = (r+1)(n-2r-1)
\quad\text{whenever}\;0 < r < \tfrac{n}{2}
\]
and empty otherwise, and has canonical bundle
\[
\omega_{\mathbf{F}_r(X)} \cong
\sO_{\mathbf{F}_r(X)}\big((q+1)(r+1) - (n+1)\big).
\]
\end{IntroTheorem}

See \parref{hypersurfaces-smooth-fano}. Here, \(\sO_{\mathbf{F}_r(X)}(1)\)
is the Pl\"ucker line bundle restricted from the ambient Grassmannian. This
gives an interesting collection of smooth projective varieties in positive
characteristic depending on the three discrete parameters \(q\), \(r\), and
\(n\). Perhaps a particularly interesting sequence occurs upon fixing the prime
power \(q\) and positive integer \(r\), and taking \(n+1 = (q+1)(r+1)\),
yielding smooth projective varieties of dimension \((r+1)^2(q-1)\) with trivial
canonical bundle.

\subsection*{\(q\)-bic threefolds}
Whereas the coarse global geometry of \(q\)-bic hypersurfaces resembles that of
quadrics, finer aspects are more closely analogous to that of cubic
hypersurfaces. This is most striking in its geometry of lines. For example, a
smooth \(q\)-bic surface contains \((q+1)(q^3+1)\) lines; when \(q = 2\), a
\(q\)-bic surface is indeed a cubic surface, and this recovers the storied
\(27\) lines.

The final contribution of this work is to elaborate on this analogy in the case
\(X\) is a \(q\)-bic threefold. The main result to this effect is an
analogue of the theorem of Clemens and Griffiths from \cite[Theorem 11.19]{CG}
which says that, over the complex numbers, the intermediate Jacobian of a cubic
threefold is isomorphic to the Albanese of its Fano surface of lines. To
formulate the result, observe that, as in the case of a cubic threefold, the
Fano scheme \(S \coloneqq \mathbf{F}_1(X)\) of lines in \(X\) is a smooth
projective surface. By the work of Murre in \cite{Murre:Jacobian-CR,
Murre:Jacobian}, there is an abelian variety \(\mathbf{Ab}_X^2\) which
parameterizes algebraically trivial \(1\)-cycles in \(X\) and which captures
the algebraic part of the intermediate Jacobian. Pulling and pushing
algebraically trivial \(0\)-cycles along the incidence
correspondence \(S \leftarrow \mathbf{L} \rightarrow X\) induces a natural
morphism \(\mathbf{L}_* \colon \mathbf{Alb}_S \to \mathbf{Ab}_X^2\). The result
is:

\begin{IntroTheorem}
The incidence correspondence \(\mathbf{L}\) induces a purely inseparable isogeny
\[ \mathbf{L}_* \colon \mathbf{Alb}_S \to \mathbf{Ab}_X^2 \]
from the Albanese variety of \(S\) to the intermediate Jacobian of \(X\).
Furthermore, every choice of Hermitian line \(\ell \subset X\) induces a
purely inseparable isogeny
\[
\prod\nolimits_{i = 0}^{q^2} \mathbf{Jac}_{C_i} \xrightarrow{\nu_*}
\mathbf{Alb}_S \xrightarrow{\mathbf{L}_*}
\mathbf{Ab}_X^2
\]
from a product of Jacobians of smooth \(q\)-bic curves.
\end{IntroTheorem}

See \parref{threefolds-smooth-intermediate-jacobian-result}. A Hermitian
line a special line contained in \(X\). For example, in the case \(X\) is the
Fermat \(q\)-bic, a Hermitian line is the same as a
\(\mathbf{F}_{q^2}\)-rational line. I do not yet know whether
\(\mathbf{L}_*\) is actually an isomorphism.

\smallskip

The geometry of the surface \(S\) is rich. Some of its basic properties and
numerical invariants are as follows:

\begin{IntroTheorem}
The Fano surface \(S\) is of general type with \(\omega_S \cong \sO_S(2q-3)\).
Moreover:
\begin{enumerate}
\item If \(q > 2\), then \(S\) does not lift to \(\mathrm{W}_2(\kk)\).
\item The \(\ell\)-adic Betti numbers of \(S\) are given by
\[
\dim_{\mathbf{Q}_\ell}\mathrm{H}^i_{\mathrm{\acute{e}t}}(S,\mathbf{Q}_\ell) =
\begin{dcases*}
1 & if \(i = 0\) or \(i = 4\), \\
q(q-1)(q^2+1) & if \(i = 1\) or \(i = 3\), and \\
(q^4 - q^3 + 1)(q^2 + 1) & if \(i = 2\).
\end{dcases*}
\]
\item If \(q = p\), the coherent Betti numbers of \(S\) are given by
\[
\dim_\kk\mathrm{H}^i(S,\sO_S) =
\begin{dcases*}
1 & if \(i = 0\), \\
p(p-1)(p^2+1)/2 & if \(i = 1\), and \\
p(p-1)(5p^4-2p^2-5p-2)/12 & if \(i = 2\).
\end{dcases*}
\]
\end{enumerate}
\end{IntroTheorem}

See \parref{threefolds-smooth-generalities}, \parref{threefolds-smooth-fano-lift},
\parref{threefolds-smooth-betti-S}, and \parref{threefolds-cohomology-S-all}.
\'Etale cohomology is computed by carefully relating \(S\) to a
Deligne--Lusztig variety for the finite unitary group \(\mathrm{U}_5(q)\).
Coherent cohomology is computed by degenerating \(S\) to the Fano scheme
\(S_0\) of lines on a mildly singular \(q\)-bic threefold, computing
cohomology on the singular surface with help from the modular representation
theory of \(\mathrm{U}_3(q)\), and then relating the cohomology of \(S_0\)
with \(S\) along the specially chosen degeneration.

\section*{Comments on the literature}
Often in the guise of special equations, \(q\)-bic hypersurfaces have attracted
the attention of mathematicians time and time again. These hypersurfaces tend
to be distinguished via extremal geometric or arithmetic properties. The
following is a sampling of works, roughly organized by topic, featuring
\(q\)-bic hypersurfaces in some way. I have tried to be comprehensive, but
the collection is surely incomplete: I apologize for the inevitable omission.

\subsection*{Projective duality}
One classical source of interest stems from the theory of projective duality:
since Wallace observed in \cite[\S7]{Wallace:Duality} that every point of a
smooth \(q\)-bics curve is a flex, authors have characterized \(q\)-bic
hypersurfaces amongst all projective hypersurfaces via exceptional properties
of their duals: see \cite[Propostion 3.7]{Pardini:Curves},
\cite{Homma:FunnyCurves}, \cite[Corollary 2.5]{Homma:SoupedUp}, and
\cite[Corollary 7.20]{Hefez:Nonreflexive} for the case of curves, and
\cite[Theorem 14]{KP:Gauss} for the case of surfaces, and
\cite[Theorem]{Noma:Dual} for the case of general type hypersurfaces. See also
\cite{HK:Notes}. In a related vein, smooth \(q\)-bics are characterized amongst
all hypersurfaces as those whose moduli of smooth hyperplane sections is
constant: see \cite{Beauville:Moduli}.

\subsection*{Automorphisms and unitary groups}
The relationship between \(q\)-bic hypersurfaces and the finite unitary groups
has been another major source of interest. Perhaps most famously, Tate used the
large automorphism group to verify his conjecture on algebraic cycles for
even-dimensional smooth \(q\)-bics: see \cite{Tate:Conjecture} and also
\cite{HM:TT-Lemma, SK:Fermat}. Reciprocally, smooth \(q\)-bics give, via
\(\ell\)-adic cohomology, geometric realizations of representations of finite
unitary groups: see \cite[Lemma 30]{Lusztig:Green} for an explicit mention;
this has been systematized by Deligne--Lusztig theory, see \cite{DL}. See also
\cite[\S3]{Hansen:DL} for some specifics on the relationship between
Deligne--Lusztig theory, finite unitary groups, and smooth \(q\)-bics.

In the specific case of a smooth \(q\)-bic curve, its automorphism group, the
finite projective unitary group \(\mathrm{PU}_3(q)\), has order \(q^3(q^3+1)(q^2-1)\), and this much
exceeds the classical Hurwitz bound \(84(g-1)\). In fact, smooth \(q\)-bic
curves nonetheless have exceptionally large automorphism groups even amongst
curves in positive characteristic: Stichtenoth proves in
\cite[Hauptsatz]{Stichtenoth:AutI} that unless \(C\) is a smooth \(q\)-bic
curve, then
\[ \#\Aut(C) \leq 16 g(C)^4. \]
See also \cite{Singh:Auts, Henn, Stichtenoth:AutII, Nakajima:Auts} for
related results and refinements. For more on the automorphisms of smooth
\(q\)-bic curves, see \cite[Section 8]{Dummigan:MW}, \cite[Section
3.5]{GSY:CanonicalRep}, \cite{HJ:Crystalline}, and \cite{Bonnafe:SL2}.

\subsection*{Finite fields and Hermitian hypersurfaces}
Another source of classical interest stems from the geometry of Hermitian
\(q\)-bic hypersurfaces over finite fields. The classical references are
\cite{Segre:Hermitian, BC:Hermitian}; see \cite[Chapter 2]{HT:Geometries} for a
more modern one, and see \cite[Section 7.3]{Hirschfeld:Geometries}. From the
perspective of this thesis, these classical works study the configuration and
arrangement of the Hermitian subspaces of general \(q\)-bic hypersurfaces. This
leads to characterizations of \(q\)-bics in terms of combinatorial data: see
\cite{HSTV:Hermitian, Thas:Hermitian}, for instance.

The Hermitian \(q\)-bic hypersurfaces tend to be characterized by having the
maximal number of low-degree rational points amongst hypersurfaces in projective
space. For example, Hermitian \(q\)-bic curves have the maximal number of
\(\mathbf{F}_{q^2}\)-rational points as permitted by the Hasse--Weil bound, and
this property characterizes these curves: see \cite{RS:Hermitian}. Other curves
with many points are often related to the Hermitian \(q\)-bic curves:
see \cite{FGT:Maximal, GSX:HermitianQuotients, GT:Maximal,
CKT:HermitianCover}.
Likewise,
Hermitian \(q\)-bic surfaces are characterized amongst surfaces over
\(\mathbf{F}_q\) that do not contain a plane as those with the maximal possible
number of \(\mathbf{F}_q\)-rational points: see
\cite{HK:SurfaceBounds,HK:Surface}. See also
\cite{Aguglia:Threefolds,AP:Threefolds} for threefolds, and
\cite{HK:Bound,Tironi,HK:EvenDimensional} for higher dimensions.

After the work of Goppa \cite{Goppa:Codes}, varieties with many rational points
over finite fields are very interesting from the perspective of constructing
error-correcting codes. So, in view of the results above, Hermitian \(q\)-bic
hypersurfaces often feature in this field: see \cite{GS:Codes, Hansen:Codes,
Stichtenoth:Codes, TVN:Codes, BDH:Sorenson} for example.

From the perspective of finite projective geometries, special arrangements
of points and linear spaces in Hermitian \(q\)-bic hypersurfaces are often of
interest. For example, \(q\)-bic curves may be combinatorially characterized
via special configurations of points: see \cite{HSTV:Hermitian, Thas:Hermitian},
\cite[Section 7.3]{Hirschfeld:Geometries}, and \cite{HKT:Curves}. Special
configurations of points on \(q\)-bic surfaces are studied in
\cite{CK:Ovoids,GK:Ovoids}, for example; the configuration of lines is
studied in \cite[Chapter 19]{Hirschfeld:Three} and \cite{EH:Lines, BPRS}, for
example.

\subsection*{\(q\)-bics and bilinear forms}
This venerable connection between \(q\)-bic and Hermitian hypersurfaces means
that many authors have expressed the equation of \(q\)-bic hypersurfaces in
terms of a bilinear form: see
\cite[\S4]{BC:Hermitian}, \cite[p.69]{Hefez:Thesis},
\cite[p.125]{Beauville:Moduli}, \cite[\S2]{Shimada:Lattices}, \cite{HH:Fermat},
and \cite[\S5]{KKPSSW:F-Pure}. Of particular note are the works \cite{Shimada:Lattices}
and \cite{KKPSSW:F-Pure}: the former begins to develop methods akin to the
theory of \(q\)-bic forms, and the latter gives a classification of \(q\)-bic
forms. The latter work is also very interesting in that it distinguishes
\(q\)-bic hypersurfaces as being extremal from the point of view of
\(F\)-singularity theory: reduced \(q\)-bic hypersurfaces have the smallest
\(F\)-pure threshold amongst hypersurfaces of the same degree. This perspective
may perhaps help to explain the curious geometry of \(q\)-bics.

\subsection*{Unirationality and rational curves}
From the classification of algebraic varieties, \(q\)-bics are notorious for
defying traditional geometric expectations. For instance, Shioda showed in
\cite{Shioda:Unirational} that smooth \(q\)-bics hypersurfaces of dimension
at least \(2\) are unirational. In particular, such \(q\)-bics are always
rationally connected, and it has long been known that they, in fact, contain
many lines: see \cite[Example 1.27]{Collino}, \cite[Example 35]{Kollar:ST},
\cite[Example 4.6.3]{Kollar:RationalCurves}, \cite{DLR:Lines}, and \cite[\S2.15]{Debarre:HDAG};
this last source even observes that the Fano scheme of lines of a smooth
\(q\)-bic is smooth. Such properties indicate that the geometry of rational
curves lying on \(q\)-bics may be quite subtle. This is already the case with
\(q\)-bic surfaces: see \cite{Ojiro:Hermitian}. In another direction, it is not
clear when a Fano smooth \(q\)-bic hypersurface is separably rationally
connected: see \cite{Conduche, Shen:Fermat, BDENY:Fermat}; though contrast
with \cite{Zhu, CZ14, Tian, CR:RationalCurves} for the general results known.

\subsection*{Supersingular varieties}
Positive characteristic invariants of \(q\)-bic hypersurfaces tend to be
very special. For instance, the Jacobians of \(q\)-bic curves are supersingular
abelian varieties which, furthermore in the case \(q = p\), are even
superspecial, see \cite[p.172]{Ekedahl:Supersingular}, \cite{PW:EO}, and
\cite[\S5.3]{AP:Superspecial}. Using the term supersingular in a slightly
different manner, Shioda's unirationality construction implies that \(q\)-bic
hypersurfaces in general are supersingular; see also \cite{SK:Fermat}, and
\cite{Shioda:Supersingular, Shioda:Unirationality} for related results in the
case of surfaces.

\subsection*{Singular \(q\)-bic hypersurfaces}
Singular \(q\)-bics have not occurred with as much prominence as their
smooth counterparts, and most occurrences tend to feature cones or else the
singular \(q\)-bic curves which are not cones. Notably, the non-linear
components are rational unicuspidal curves, wherein the cusp is, in a sense,
more singular than usual. For instance, the \(F\)-threshold of this cuspidal
singularity is lower than it would have in other characteristics; see
\cite[Example 4.3]{MTW:BS} for the case \(q = 2\) and also
\cite{KKPSSW:F-Pure}. These curves appear as generalizations of quasi-elliptic
fibrations: compare \cite[\S1]{BM:III} and \cite{Shimada:Supercuspidal,
Takeda:Cuspidal, IIL:Rational}.

\section*{Outline}
What follows are four chapters and two appendices.
Chapter \parref{chapter-forms} develops the algebraic theory of \(q\)-bic
forms. Chapter \parref{chapter-hypersurfaces} applies the theory to study the geometry
of \(q\)-bic hypersurfaces. Chapter \parref{chapter-lowdim} illustrates this in low
dimensions, computing examples for \(q\)-bic points, curves, and surfaces.
Chapter \parref{chapter-threefolds}, the core of the work, studies the geometry
of \(q\)-bic threefolds in detail. Appendix \parref{chapter-linear} collects some
generalities on projective geometry. Appendix \parref{chapter-representations}
collects some representation theoretic facts and computations.

\section*{Notation and conventions}

Throughout this manuscript, \(q\) is a positive integer power of a prime
\(p\), and \(\kk\) is a field of characteristic \(p\). For a vector space \(V\)
over \(\kk\), let \(\PP V\) denote the projective space of lines in \(V\), so
that
\[
\mathrm{H}^0(\PP V, \sO_{\PP V}(1))
= V^\vee.
\]
Given a scheme \(X\) over \(\kk\), write \(V_X \coloneqq V \otimes_\kk \sO_X\).
Let
\[
\mathrm{eu}_{\PP V} \colon
\sO_{\PP V}(-1) \to
V_{\PP V}
\]
denote the \emph{Euler section} on \(\PP V\): the tautological section dual to
the evaluation map for global sections. The projective subspace corresponding
to a \((r+1)\)-dimensional linear subspace \(U \subseteq V\) will be written
\(\PP U \subseteq \PP V\) and referred to as the associated \emph{\(r\)-plane}.
For subspaces \(U_1,U_2 \subseteq V\), write \(\langle U_1,U_2 \rangle\) for
the subspace spanned by \(U_1\) and \(U_2\) in \(V\).

%% file: forms.tex
\chapter{\texorpdfstring{\(q\)}{q}-bic Forms}\label{chapter-forms}

A \emph{\(q\)-bic form} on a module \(M\) over a \(\mathbf{F}_{q^2}\)-algebra
\(R\) is an \(R\)-linear map
\[ \beta \colon \Fr^*(M) \otimes_R M \to R \]
where \(\Fr^*(M) \coloneqq R \otimes_{\Fr, R} M\) is the \(q\)-power Frobenius
twist of \(M\). This Chapter develops the theory of \(q\)-bic forms in the
spirit of the theory of bilinear forms. Since \(q\)-bic forms pair different
modules, there is no simple symmetry condition available. The feature that
distinguishes \(q\)-bic forms amongst general bilinear forms between two
modules is the canonical \(\Fr\)-linear map \(M \to \Fr^*(M)\) given by
\(m \mapsto 1 \otimes m\).

Basic definitions and constructions are given in \parref{section-forms-definitions}.
The relationship between \(q\)-bic forms and Hermitian forms over
\(\mathbf{F}_{q^2}\) is explained in \parref{section-hermitian}; in particular,
\parref{forms-hermitian-diagonal} shows that any nonsingular \(q\)-bic form
over a separably closed field admits an orthonormal basis. Automorphisms are
discussed in \parref{section-forms-automorphisms}. The Chapter closes
in \parref{section-forms-classification} with comments on the classification
and moduli of \(q\)-bic forms over an algebraically closed field.

\section{Basic notions}\label{section-forms-definitions}
This Section contains the basic definitions and properties of \(q\)-bic
forms. Throughout, let \(R\) be a \(\mathbf{F}_{q^2}\)-algebra and write
\(\Fr \colon R \to R\) for the \(q\)-power Frobenius morphism. For an
\(R\)-module \(M\), write \(M^\vee \coloneqq \Hom_R(M,R)\).
The functor \(M \mapsto M^\vee\) restricts to a duality on the category of finite
projective \(R\)-modules. In the case that \(M\) is finite projective,
\(M^{\vee,\vee}\) will be tacitly identified with \(M\) via the canonical
isomorphism \(\operatorname{can}_M \colon M \to M^{\vee,\vee}\) which sends
\(m \in M\) to the evaluation functional \(\mathrm{ev}_m \colon \Hom_R(M,R) \to R\).

\subsection{\texorpdfstring{\(q\)}{q}-linearization}\label{forms-linearization}
For any \(R\)-module \(M\), let \(\Fr^*(M) \coloneqq R \otimes_{\Fr,R} M\)
be the (left) \(R\)-module with the action of \(R\) twisted by the \(q\)-power
Frobenius. Elements of \(\Fr^*(M)\) are \(R\)-linear combinations elements of
the form \(r \cdot m^{(q)} \coloneqq r \otimes m\) where \(r \in R\) and
\(m \in M\). Moreover, if \(m = r'm'\) for some \(r' \in R\) and \(m' \in M\),
then
\[
r \cdot m^{(q)}
= r \cdot (r'm')^{(q)}
= r \otimes r'm'
= r r'^q \otimes m'
= (rr'^q) \cdot m'^{(q)}.
\]
The map \(m \mapsto m^{(q)}\) is the universal \(\Fr\)-linear additive map
\((-)^{(q)} \colon M \to \Fr^*(M)\) out of \(M\). As a piece of notation,
given a submodule \(N \subseteq \Fr^*(M)\), write
\[ \Fr^{-1}(N) \coloneqq \Set{m \in M | m^{(q)} \in N}. \]



\subsection{\(q\)-bic forms}\label{forms-definition}
A \emph{\(q\)-bic form} over \(R\) is a pair \((M,\beta)\) consisting of an
\(R\)-module \(M\) and an \(R\)-linear map
\(\beta \colon \Fr^*(M) \otimes_R M \to R\). A morphism \(\varphi \colon (M_1,\beta_1) \to (M_2,\beta_2)\)
between two \(q\)-bic forms is a morphism \(\varphi \colon M_1 \to M_2\) of
\(R\)-modules such that
\[
\beta_1(m',m) = \beta_2(\Fr^*(\varphi)(m'), \varphi(m))
\quad\text{for every}\; m' \in \Fr^*(M_1)\;\text{and}\;m \in M_1,
\]
where \(\Fr^*(\varphi) \colon \Fr^*(M_1) \to \Fr^*(M_2)\) is the \(\Fr\)-twist
of \(\varphi\). The morphism \(\varphi\) is an \emph{isomorphism} if the underlying
module map is an isomorphism.

Adjunction induces two \(R\)-linear maps which, by an abuse of notation, are
denoted
\[
\beta \colon M \to \Fr^*(M)^\vee
\quad\text{and}\quad
\beta^\vee \colon \Fr^*(M) \to M^\vee.
\]
The two adjoints are dual to one another via \(\Hom_R(-,R)\) in the sense that
\[
\beta = \Hom_R(\beta^\vee,R) \circ \operatorname{can}_M
\quad\text{and}\quad
\beta^\vee = \Hom_R(\beta,R) \circ \operatorname{can}_{\Fr^*(M)}.
\]
The form \((M,\beta)\) is said to be \emph{nondegenerate} if
the adjoint map \(\beta \colon M \to \Fr^*(M)\) is injective, and
\emph{nonsingular} if the adjoint map is an isomorphism.

\subsection{Isotropicity}\label{forms-isotropicity}
Let \((M,\beta)\) be a \(q\)-bic form over \(R\). An element \(m \in M\) is
said to be \emph{isotropic} if \(\beta(m^{(q)},m) = 0\). A submodule
\(M' \subseteq M\) is said to be \emph{isotropic} if every element of \(M'\) is
isotropic, and it is \emph{totally isotropic} if the restriction of \(\beta\)
to \(M'\) is the zero form. A morphism \(\varphi \colon N \to M\) of
\(R\)-modules is said to be \emph{isotropic}, respectively \emph{totally
isotropic}, if the image of \(\varphi\) is an isotropic submodule, respectively
totally isotropic  submodule.

Since \(R\) is an algebra over \(\mathbf{F}_{q^2}\), the two notions of
isotropic subspaces agree:

\begin{Lemma}\label{forms-notions-of-isotropicity}
Let \((M,\beta)\) be a \(q\)-bic form over \(R\). Then a submodule
\(M' \subseteq M\) is isotropic if and only if it is totally isotropic.
\end{Lemma}

\begin{proof}
It suffices to show that an isotropic submodule is totally isotropic. So consider
a pair \(m_1, m_2 \in M'\). Isotropicity gives, for any \(\lambda \in R^\times\),
\[
0
= \beta((\lambda m_1 + m_2)^{(q)}, \lambda m_1 + m_2)
= \lambda^q \beta(m_1^{(q)},m_2) + \lambda \beta(m_2^{(q)}, m_1).
\]
Thus \(\beta(m_2^{(q)},m_1) = - \lambda^{q-1} \beta(m_1^{(q)},m_2)\)
for every \(\lambda \in R^\times\). Since \(\mathbf{F}_{q^2}^\times \subseteq R^\times\),
the \(\lambda^{q-1}\) take on at least two distinct values, whence
\(\beta(m_2^{(q)},m_1) = \beta(m_1^{(q)},m_2) = 0\). Since every element of
\(\Fr^*(M')\) is an \(R\)-linear combination of elements of the form
\(m_1^{(q)}\), this shows that \(M'\) is totally isotropic.
\end{proof}

\begin{Example}
If \(R\) were simply a \(\mathbf{F}_q\)-algebra, say, then the two notions of
isotropic subspaces may diverge. For instance, a \(q\)-bic form over a field
\(\kk\) contained in \(\mathbf{F}_q\) is simply a bilinear form
\(\beta \colon V \otimes_\kk V \to \kk\). Take \(V = \kk^{\oplus 2}\) the
standard \(2\)-dimensional vector space and consider the bilinear form with
Gram matrix
\[ \Gram(\beta;e_1,e_2) = \begin{pmatrix} 0 & 1 \\ q - 1 & 0 \end{pmatrix}. \]
Then \(V\) is isotropic but not totally isotropic.
\end{Example}

\subsection{Orthogonals}\label{forms-orthogonals}
Given submodules \(N_1 \subseteq \Fr^*(M)\) and \(N_2 \subseteq M\), write
\begin{align*}
N_1^\perp
& \coloneqq \ker(M \xrightarrow{\beta} \Fr^*(M)^\vee \to N_1^\vee), \\
N_2^\perp
& \coloneqq \ker(\Fr^*(M) \xrightarrow{\beta^\vee} M^\vee \to N_2^\vee).
\end{align*}
These are the \emph{orthogonals}, with respect to \(\beta\), of \(N_1\) and
\(N_2\), respectively. The orthogonals \(\Fr^*(M)^\perp \subseteq M\) and
\(M^\perp \subseteq \Fr^*(M)\) are referred to as the \emph{kernels} of
\(\beta\) (and \(\beta^\vee\)). The \emph{radical} of \(\beta\) is the submodule
\[
\rad(\beta) \coloneqq
\Set{m \in M |
\beta(m^{(q)},n_1) = \beta(n_2,m) = 0
\;\text{for all}\;
n_1 \in M, n_2 \in \Fr^*(M)}.
\]
More succinctly, \(\rad(\beta) = \Fr^*(M)^\perp \cap \Fr^{-1}(M^\perp)\).

\begin{Lemma}\label{forms-orthogonal-inclusion-reversing}
Let \((M,\beta)\) be a \(q\)-bic form over \(R\). The orthogonals of nested
submodules \(M_1 \subseteq M_2\) of \(\Fr^*(M)\) or \(M\) satisfy
\(M_1^\perp \supseteq M_2^\perp\) as submodules of \(M\) or \(\Fr^*(M)\).
\end{Lemma}

\begin{proof}
This follows from the definitions in \parref{forms-orthogonals} since the
restriction map to \(M_1\) factors through \(M_2^\vee \to M_1^\vee\). For
instance, if given submodules of \(\Fr^*(M)\),
\[
M_2^\perp
= \ker(M \to \Fr^*(M)^\vee \to M_2^\vee)
\subseteq \ker(M \to \Fr^*(M)^\vee \to M_2^\vee \to M_1^\vee)
= M_1^\perp,
\]
and similarly if given submodules of \(M\).
\end{proof}

\begin{Lemma}\label{forms-orthogonal-sequence}
Let \((M,\beta)\) be a \(q\)-bic form over \(R\). If \(M\) and \(\coker(\beta)\)
are finite projective, then there is an exact sequence of \(R\)-modules
\[
  0 \to
  \Fr^*(M)^\perp \to
  M \xrightarrow{\beta}
  \Fr^*(M)^\vee \to
  M^{\perp,\vee} \to
  0.
\]
\end{Lemma}

\begin{proof}
It remains to identify \(\coker(\beta)\) with \(M^{\perp,\vee}\). The
assumptions imply that each of \(\Fr^*(M)^\perp\), \(M\), \(\Fr^*(M)^\vee\),
and \(\coker(\beta)\) are finite projective. Applying \(\Hom_R(-,R)\) and
making identifications using the canonical maps as in \parref{forms-definition}
gives an exact sequence
\[
0 \to
\coker(\beta)^\vee \to
\Fr^*(M) \xrightarrow{\beta^\vee}
M^\vee \to
\Fr^*(M)^{\perp,\vee} \to
0.
\]
This identifies \(\coker(\beta)^\vee = \ker(\beta^\vee)\) as \(M^\perp\).
Reflexivity shows \(\coker(\beta) = M^{\perp,\vee}\).
\end{proof}

\begin{Lemma}\label{forms-orthogonal-reflexive}
Let \((M,\beta)\) be a \(q\)-bic form over \(R\). If \(M\) and \(\coker(\beta)\)
are finite projective, then
\[
M_1^{\perp,\perp} = M^\perp + M_1
\quad\text{and}\quad
M_2^{\perp,\perp} = \Fr^*(M)^\perp + M_2
\]
for any submodules \(M_1 \subseteq \Fr^*(M)\) and
\(M_2 \subseteq M\) which are locally direct summands.
\end{Lemma}

\begin{proof}
Argue only for \(M_1\), the argument for \(M_2\) being completely analogous.
By the definitions of \parref{forms-orthogonals},
\[
M_1^{\perp,\perp}
= \ker(\Fr^*(M) \xrightarrow{\beta^\vee} M^\vee \to M_1^{\perp,\vee})
= M^\perp + \beta^{\vee,-1}(\ker(M^\vee \to M_1^{\perp,\vee})).
\]
Since \(M_1^\perp = \ker(M \to \Fr^*(M)^\vee \to M_1^\vee)\), that
\(M_1\) and \(\Fr^*(M)\) are reflexive \(R\)-modules together with the
exact sequence \parref{forms-orthogonal-sequence} implies
\(M_1^{\perp,\vee} = \coker(M_1 \to \Fr^*(M) \to M^\vee)\). Thus
the kernel of \(M^\vee \to M_1^{\perp,\vee}\) is the image of \(M_1\) under
\(\beta^\vee\), giving the result.
\end{proof}

\subsection{Orthogonal complements}\label{forms-orthogonal-complements}
Let \((M,\beta)\) be a \(q\)-bic form over \(R\). An
\emph{orthogonal complement} of a submodule \(M' \subseteq M\) is another
submodule \(M'' \subseteq M\) such that
\(\beta(m'^{(q)},m'') = \beta(m''^{(q)},m') = 0\)
for every \(m' \in M'\) and \(m'' \in M''\), and \(M = M' \oplus M''\) as
\(R\)-modules. This situation is signified by
\[ (M,\beta) = (M',\beta_{M'}) \perp (M'',\beta_{M''}) \]
where \(\beta_{M'}\) and \(\beta_{M''}\) denotes the restriction of \(\beta\) to
\(M'\) and \(M''\), respectively.

Orthogonal complements need not exist, and when they do, need not be unique.
For instance, it follows from definitions that an orthogonal complement
\(M''\) of \(M'\) must satisfy
\[ M'' \subseteq \Fr^*(M')^\perp \cap \Fr^{-1}(M'^\perp). \]
This may be reformulated into a numerical criterion over a field:

\begin{Lemma}\label{forms-orthogonal-complements-criterion}
Let \((V,\beta)\) be a \(q\)-bic form over a field \(\kk\) on a
finite-dimensional vector space \(V\). Then a subspace \(V' \subseteq V\) has
an orthogonal complement if and only if
\[
\dim_\kk V - \dim_\kk V' =
\dim_\kk W - \dim_\kk W \cap V'
\]
where \(W \coloneqq \Fr^*(V')^\perp \cap \Fr^{-1}(V'^\perp)\).
It is unique if and only if \(W \cap V' = \{0\}\).
\end{Lemma}

\begin{proof}
By \parref{forms-orthogonal-complements}, \(V'\) has an orthogonal complement
if and only if there is a linear subspace of \(W\) with complementary dimension
and which is linearly disjoint with \(W \cap V'\); taking dimensions gives the
criterion in the statement. The complementary subspace \(V'' \subseteq W\) is
unique if and only if \(W \cap V' = \{0\}\).
\end{proof}

\subsection{Gram matrix}\label{forms-gram-matrix}
Let \((M,\beta)\) be a \(q\)-bic form over \(R\). If \(M\) is a finite
free module, then a choice of basis renders \(\beta\) quite explicit.
Let \(\varphi \colon \bigoplus_{i = 0}^n R \cdot e_i \cong M\) be
a basis. The \emph{Gram matrix} of \(\beta\) with respect to the basis
\(\langle e_0,\ldots,e_n \rangle\) is the matrix
\[
\Gram(\beta;e_0,\ldots,e_n) \coloneqq
\big(\beta(e_i^{(q)},e_j)\big)_{i, j = 0}^n
\in \mathbf{Mat}_{(n+1) \times (n+1)}(R).
\]
In other words, this is the matrix of the adjoint map
\(\beta \colon M \to \Fr^*(M)^\vee\) with respect to the basis
\(\langle e_0,\ldots,e_n \rangle\) in the source, and
\(\langle e_0^{(q)},\ldots,e_n^{(q)} \rangle\) in the
target:
\[
\Fr^*(\varphi)^\vee \circ \beta \circ \varphi \colon
\bigoplus\nolimits_{i = 0}^n R \cdot e_i \cong
M \to
\Fr^*(M)^\vee \cong
\bigoplus\nolimits_{i = 0}^n R \cdot e_i^{\vee,(q)}.
\]
where \(e_i^\vee\) is the dual basis element to \(e_i\).

\subsection{Change of basis}\label{forms-gram-matrix-change-basis}
Suppose \(\varphi' \colon \bigoplus_{i = 0}^n R \cdot e_i' \cong M\) is another
choice of basis. Consider the change of basis isomorphism
\[
A \coloneqq
\varphi^{-1} \circ \varphi' \colon
\bigoplus\nolimits_{i = 0}^n R \cdot e_i' \xrightarrow{\cong}
\bigoplus\nolimits_{i = 0}^n R \cdot e_i
\]
viewed as an invertible matrix \(A \in \mathbf{GL}_{n+1}(R)\). Its \(\Fr\)-twist
\[
\Fr^*(A) \coloneqq \Fr^*(\varphi)^{-1} \circ \Fr^*(\varphi') \colon
\bigoplus\nolimits_{i = 0}^n R \cdot e_i'^{(q)} \xrightarrow{\cong}
\bigoplus\nolimits_{i = 0}^n R \cdot e_i^{(q)}
\]
is the matrix obtained of \(A\) by taking \(q\)-powers entrywise. Then
there is an identity
\[
\Gram(\beta; e_0',\ldots,e_n') =
\Fr^*(A)^\vee \cdot \Gram(\beta; e_0,\ldots,e_n) \cdot A.
\]
Nonsingularity of \(\beta\) takes a familiar meaning in terms of the Gram
matrix:

\begin{Lemma}\label{forms-gram-matrix-nondegenerate}
Let \((M,\beta)\) be a \(q\)-bic form over \(R\) on a finite free module \(M\).
The following are equivalent:
\begin{enumerate}
\item\label{forms-gram-matrix-nondegenerate.intrinsic}
\(\beta\) is nonsingular;
\item\label{forms-gram-matrix-nondegenerate.some-basis}
\(\Gram(\beta;e_0,\ldots,e_n)\) is invertible for some basis
\(\langle e_0,\ldots,e_n \rangle\) of \(M\); and
\item\label{forms-gram-matrix-nondegenerate.any-basis}
\(\Gram(\beta;e_0,\ldots,e_n)\) is invertible for any basis
\(\langle e_0,\ldots,e_n\rangle\) of \(M\).
\end{enumerate}
\end{Lemma}

\begin{proof}
That \ref{forms-gram-matrix-nondegenerate.some-basis} \(\Leftrightarrow\)
\ref{forms-gram-matrix-nondegenerate.any-basis} follows immediately from the
discussion of \parref{forms-gram-matrix-change-basis}. By the definitions
in \parref{forms-definition}, \(\beta\) is nonsingular if and only if the
adjoint map \(\beta \colon M \to \Fr^*(M)\) is an isomorphism. By
\parref{forms-gram-matrix}, the matrix of this map, with respect to some
basis \(\langle e_0,\ldots,e_n\rangle\), is the Gram matrix \(\Gram(\beta;e_0,\ldots,e_n)\),
thereby yielding \ref{forms-gram-matrix-nondegenerate.intrinsic}
\(\Leftrightarrow\) \ref{forms-gram-matrix-nondegenerate.some-basis}.
\end{proof}

Gram matrices make it easy to give examples of \(q\)-bic forms.
Given an \((n+1)\)-by-\((n+1)\) matrix \(B\) over \(R\), write
\((R^{\oplus n+1}, B)\) for the unique \(q\)-bic form with Gram matrix \(B\) in
the given basis. Particularly important examples are given in the following;
they appear prominently in the classification of \(q\)-bic forms, see
\parref{forms-classification-theorem}.

\subsection{Standard forms}\label{forms-standard}
For each positive integer \(k\), write
\[
\mathbf{N}_k \coloneqq
\left(
\begin{smallmatrix}
0 & 1 & \cdots &   &   \\
  & 0 & \cdots &   &   \\
  &   & \cdots &   &   \\
  &   & \cdots & 0 & 1 \\
  &   & \cdots &   & 0
\end{smallmatrix}
\right)
\]
for the \(k\)-by-\(k\) Jordan block with \(0\) on the diagonals, and write
\(\mathbf{1}\) for the \(1\)-by-\(1\) identity matrix;
\(\mathbf{N}_1\) is the \(1\)-by-\(1\) zero matrix, and is often
written \(\mathbf{0}\). Given two matrices \(B_1\) and \(B_2\),
write \(B_1 \oplus B_2\) for their block diagonal sum. Let \(m\) and
\(a_1,\ldots,a_m,b\) be non-negative integers such that
\(n+1 = b + \sum_{k = 1}^m k a_k\). The \(q\)-bic form
\[
(R^{\oplus n+1},
\mathbf{N}_1^{\oplus a_1} \oplus
\cdots \oplus
\mathbf{N}_m^{\oplus a_m} \oplus
\mathbf{1}^{\oplus b})
\]
is called a \emph{standard \(q\)-bic form} and the Gram matrix is called
its \emph{type}.

\medskip

Associated with a \(q\)-bic form \((M,\beta)\) over \(R\) are two natural
sequences of \(R\)-modules, both obtained by successively taking orthogonals
with respect to \(\beta\). The simpler of the two produces a filtration of
\(M\), and is described in \parref{forms-canonical-filtration}. The second
produces a sequence of filtrations on the Frobenius twists \(\Fr^{i,*}(M)\),
see \parref{forms-canonical-filtration-second}.

\subsection{\(\perp\)-filtration}\label{forms-canonical-filtration}
Construct an increasing filtration \(M_\bullet\) on \(M\) using \(\beta\) as
follows: set \(M_i = \{0\}\) for \(i < 0\),
\(M_0 \coloneqq \mathrm{rad}(\beta)\), and for each \(i \geq 1\), inductively define
\[
M_i \coloneqq \Fr^*(\Fr^*(M_{i-1})^\perp)^\perp.
\]
This is an increasing filtration in which each submodule \(M_i\) is isotropic,
as can be seen by induction and \parref{forms-orthogonal-inclusion-reversing}:
twisting by \(\Fr\) and taking orthogonals twice yields
\[
\begin{array}{lclcl}
M_{i-2} \subseteq M_{i-1}
& \Rightarrow
& \Fr^*(M_{i-2})^\perp \supseteq \Fr^*(M_{i-1})^\perp
& \Rightarrow
& M_{i-1} \subseteq M_i,
\\
M_{i-1} \subseteq \Fr^*(M_{i-1})^\perp
& \Rightarrow
& \Fr^*(M_{i-1})^\perp \supseteq M_i
& \Rightarrow
& M_{i\phantom{-1}} \subseteq \Fr^*(M_i)^\perp.
\end{array}
\]
Therefore the filtration \(M_\bullet\) may be extended to a filtration
\[
\mathrm{rad}(\beta) =
M_0 \subseteq
M_1 \subseteq
M_2 \subseteq
\cdots \subseteq
\Fr^*(M_2)^\perp \subseteq
\Fr^*(M_1)^\perp \subseteq
\Fr^*(M_0)^\perp =
M
\]
called the \emph{\(\perp\)-filtration} of \((M,\beta)\).

In good cases, \(M_i\) is the radical of \(\beta\) restricted to
\(\Fr^*(M_i)^\perp\); in fact, the \(\perp\)-filtration can be characterized as
the maximal filtration on \(M\) with this property. This implies that this
filtration may be constructed recursively upon replacing \(M\) by
\(\Fr^*(M)^\perp\).

\begin{Lemma}\label{forms-canonical-filtration-radical}
Let \((M,\beta)\) be a \(q\)-bic form over \(R\). If \(M\) and
\(\coker(\beta)\) are finite projective, then
\(M_i\) is the radical of \(\beta\) restricted to \(\Fr^*(M_i)^\perp\) for each
\(i \geq 0\).
\end{Lemma}

\begin{proof}
Indeed, by \parref{forms-orthogonal-reflexive},
\begin{align*}
\mathrm{rad}(\beta_{\Fr^*(M_i)^\perp})
& = \Fr^*(\Fr^*(M_i)^\perp)^\perp
  \cap \Fr^{-1}(\Fr^*(M_i)^{\perp,\perp})
  \cap \Fr^*(M_i)^\perp \\
& = M_{i+1} \cap M_i \cap \Fr^*(M_i)^\perp.
\end{align*}
Since
\(M_i \subseteq M_{i+1} \subseteq \Fr^*(M_i)^\perp\) by
\parref{forms-canonical-filtration}, this gives the result.
\end{proof}

\subsection{Frobenius-twisted form}\label{forms-fr-twist}
Given a \(q\)-bic form \((M,\beta)\) over \(R\), applying the functor \(\Fr^*\)
yields another \(q\)-bic form \((\Fr^*(M),\Fr^*(\beta))\) over \(R\), where
\[
\Fr^*(\beta) \colon \Fr^{2,*}(M) \otimes_R \Fr^*(M) \to R
\]
is the bilinear pairing between \(\Fr^{2,*}(M) = \Fr^*(M) \otimes_{\Fr,R} R\)
and \(\Fr^*(M) = M \otimes_{\Fr,R} R\) obtained as the \(R\)-linear extension
of
\[
\Fr^*(\beta)(m_1^{(q)}, m_2^{(q)}) = \beta(m_1,m_2)^q
\quad\text{for every}\;
m_1 \in \Fr^*(M)
\;\text{and}\;
m_2 \in M,
\]
notation as in \parref{forms-linearization}. In particular, if \(M\) is a
free \(R\)-module of rank \(n+1\), then
the Gram matrices of \(\Fr^*(\beta)\) and \(\beta\) are related by
\[
\mathrm{Gram}(\Fr^*(\beta); e_0^{(q)},\ldots,e_n^{(q)}) =
\mathrm{Gram}(\beta; e_0,\ldots,e_n)^{(q)}
\]
for any basis \(M \cong \bigoplus_{i = 0}^n R \cdot e_i\); here,
\((-)^{(q)}\) denotes taking \(q\)-powers entrywise.

\subsection{Frobenius-twisted orthogonal}\label{forms-fr-twist-orthogonal}
Let \((M,\beta)\) be a \(q\)-bic form over \(R\).
Given a submodule \(N \subseteq \Fr^*(M)\), let
\[
N^{\Fr^*(\perp)} \coloneqq
\ker\big(\Fr^*(\beta)^\vee \colon
\Fr^{2,*}(M) \to \Fr^*(M)^\vee \to N^\vee\big)
\]
be the submodule of \(\Fr^{2,*}(M)\) obtained by taking orthogonals with
respect to the Frobenius-twisted form \((\Fr^*(M),\Fr^*(\beta))\), see
\parref{forms-fr-twist}. This is referred to as the \emph{Frobenius-twisted
orthogonal of \(N\)} with respect to \(\beta\). Since this operation is but
taking orthogonals upon moving to another \(q\)-bic form, formal properties
such as \parref{forms-orthogonal-inclusion-reversing} and
\parref{forms-orthogonal-reflexive} continue to hold.

\subsection{\(\Fr^*(\perp)\)-filtration}\label{forms-canonical-filtration-second}
Construct submodules \(W_i \subseteq \Fr^{2i-1,*}(M)\) using \(M\) and \(\beta\)
as follows: set \(\Fr^{-1,*}(M) \coloneqq W_0 \coloneqq \rad(\beta)\),
and for each \(i \geq 1\), inductively define
\[
W_i \coloneqq
W_{i-1}^{\Fr^*(\perp), \Fr^*(\perp)} \subseteq
\Fr^{2i-1,*}(M).
\]
Expanding the definition of the Frobenius-twisted orthogonal from
\parref{forms-fr-twist-orthogonal}, this is
\begin{align*}
W_{i-1}^{\Fr^*(\perp)}
& \coloneqq
  \ker\big(\Fr^{2i-2,*}(\beta)^\vee \colon
  \Fr^{2i-2,*}(M) \to
  \Fr^{2i-3,*}(M)^\vee \to
  W_i^\vee\big), \\
W_i & \coloneqq
  \ker\big(\Fr^{2i-1,*}(\beta)^\vee \colon
  \Fr^{2i-1,*}(M) \to
  \Fr^{2i-2,*}(M)^\vee \to
  W_{i-1}^{\Fr^*(\perp),\vee}\big).
\end{align*}
For instance, \(W_1 = M^\perp\) as a submodule of \(\Fr^*(M)\).

As in \parref{forms-canonical-filtration}, these modules are, in a sense,
increasing and totally isotropic for \(\beta\). Taking Frobenius-twisted
orthogonals and applying \parref{forms-orthogonal-inclusion-reversing} twice
gives:
\[
\begin{array}{lclcl}
\Fr^{2,*}(W_{i-2}) \subseteq W_{i-1}
& \Rightarrow
& \Fr^{2,*}(W_{i-2})^{\Fr^*(\perp)} \supseteq W_{i-1}^{\Fr^*(\perp)}
& \Rightarrow
& \Fr^{2,*}(W_{i-1}) \subseteq W_i,
\\
\phantom{^2}\Fr^*(W_{i-1}) \subseteq W_{i-1}^{\Fr^*(\perp)}
& \Rightarrow
& \phantom{^2}\Fr^*(W_{i-1})^{\Fr^*(\perp)} \supseteq W_i
& \Rightarrow
& \phantom{^2}\Fr^*(W_i)_{\phantom{-1}}\subseteq W_i^{\Fr^*(\perp)}.
\end{array}
\]
This means that, for each \(i \geq 0\), this gives an increasing
filtration of \(\Fr^{2i-1,*}(M)\):
\[
\Fr^{2i-2,*}(W_1) \subseteq
\cdots
\subseteq
\Fr^{2,*}(W_{i-1}) \subseteq
W_i \subseteq
\Fr^*(W_{i-1}^{\Fr^*(\perp)}) \subseteq
\cdots \subseteq
\Fr^{2i-3,*}(W_1^{\Fr^*(\perp)}).
\]

\subsection{Rank and corank}\label{forms-rank-corank}
A sequence of numerical invariants may be extracted from the \(\perp\)- and
\(\Fr^*(\perp)\)-filtrations for a \(q\)-bic form \((V,\beta)\) on a
finite-dimensional vector space over a field \(\kk\). The simplest, determined
by the first step of either filtration, are its \emph{rank} and \emph{corank}:
\begin{align*}
\rank(V,\beta)
& \coloneqq \dim_\kk V/\Fr^*(V)^\perp
= \dim_\kk \Fr^*(V)/V^\perp, \\
\corank(V,\beta)
& \coloneqq \dim_\kk V - \rank(V,\beta)
= \dim_\kk \Fr^*(V)^\perp
= \dim_\kk V^\perp.
\end{align*}
The identifications come from the exact sequence of
\parref{forms-orthogonal-sequence}. The rank of a \(q\)-bic form \((V,\beta)\)
is the same as the rank of any Gram matrix as in \parref{forms-gram-matrix}.

The corank changes within a bounded range upon restriction to subspaces:

\begin{Lemma}\label{forms-rank-linear-subspace}
Let \((V,\beta)\) be a \(q\)-bic form over a field \(\kk\). If \((U,\beta_U)\)
is its restriction to a subspace \(U \subseteq V\), then
\[
\corank(V,\beta) - \codim(U \subseteq V) \leq
\corank(U,\beta_U) \leq
\corank(V,\beta) + \codim(U \subseteq V).
\]
\end{Lemma}

\begin{proof}
In fact, the orthogonal sequence \parref{forms-orthogonal-sequence}
induces a short exact sequence
\[
0 \to
\Fr^*(V)^{\perp_\beta} \cap U \to
\Fr^*(U)^{\perp_{\beta_U}} \to
\beta(U) \cap \Fr^*(V/U)^\vee \to
0.
\]
Taking dimensions and using the final comments of \parref{forms-rank-corank}
give the upper bound:
\begin{align*}
\corank(U,\beta_U)
& = \dim_\kk(\Fr^*(V)^{\perp_\beta} \cap U)
  + \dim_\kk(\beta(U) \cap \Fr^*(V/U)^\vee) \\
& \leq \dim_\kk \Fr^*(V)^{\perp_\beta}
     + \dim_\kk \Fr^*(V/U)^\vee \\
& = \corank(V,\beta) + \codim(U \subset V).
\end{align*}
The lower bound corank is obtained from:
\[
\corank(U,\beta_U)
\geq \dim_\kk(\Fr^*(V)^{\perp_\beta} \cap U)
\geq \dim_\kk \Fr^*(V)^{\perp_\beta} - \codim(U \subset V)
\]
from which the statement follows upon consulting \parref{forms-rank-corank}
again.
\end{proof}

\section{\texorpdfstring{\(q\)}{q}-bic and Hermitian forms}\label{section-hermitian}
By definition, \(q\)-bic forms linearize a biadditive map \(V \times V \to \kk\)
that is linear in the second variable, but only \(q\)-linear in the first.
In the case that \(\kk \subseteq \mathbf{F}_{q^2}\), such forms are
sesquilinear with respect the \(q\)-power Frobenius, and a notion of
symmetry is given by that of Hermitian forms; see \cite[Chapter XIII,
\S7]{Lang:Algebra} for generalities and definitions on sesquilinear linear
algebra. Over a general field \(\kk\), the Hermitian condition does not make
sense. Nevertheless, associated with any \(q\)-bic form is a natural Hermitian
form over \(\mathbf{F}_{q^2}\): see \parref{forms-hermitian-basics}. This
associated Hermitian form elucidates special features and yields a powerful
invariant of the \(q\)-bic form.

Throughout, \((V,\beta)\) is a \(q\)-bic form of dimension \(n+1\) over a field
\(\kk\).

\subsection{Hermitian subspaces}\label{forms-hermitian}
A vector \(v \in V\) is said to be \emph{Hermitian} if
\[
\beta(w^{(q)},v) = \beta(v^{(q)},w)^q
\quad\text{for all}\;w \in V.
\]
A subspace \(U \subseteq V\) is said to be \emph{Hermitian} if it is
spanned by Hermitian vectors. Let
\[
V_{\mathrm{Herm}} \coloneqq
\set{v \in V \;\text{a Hermitian vector}}.
\]
The following gives some basic properties of the subset of Hermitian vectors;
notably, it shows that the restriction of \(\beta\) therein gives rise to a
canonical Hermitian form for the quadratic field extension
\(\mathbf{F}_{q^2}/\mathbf{F}_q\):

\begin{Lemma}\label{forms-hermitian-basics}
The set \(V_{\mathrm{Herm}}\) is a vector space over \(\mathbf{F}_{q^2}\) and
satisfies
\[
\beta(v_1^{(q)}, v_2) \in \mathbf{F}_{q^2}
\quad\text{for every}\;
v_1,v_2 \in V_{\mathrm{Herm}}.
\]
Thus the form
\(
\beta_{\mathrm{Herm}} \colon
\Fr^*(V_{\mathrm{Herm}}) \otimes_{\mathbf{F}_{q^2}} V_{\mathrm{Herm}} \to
\mathbf{F}_{q^2}
\)
induced by \(\beta\) is a Hermitian form.
\end{Lemma}

\begin{proof}
Let \(u, v \in V_{\mathrm{Herm}}\). For any \(\lambda \in \mathbf{F}_{q^2}\)
and any \(w \in V\),
\begin{align*}
\beta(w^{(q)}, u + \lambda v)
& = \beta(u^{(q)},w)^q + \lambda \beta(v^{(q)}, w)^q \\
& = \big(\beta(u^{(q)},w) + \beta(\lambda^{1/q} v^{(q)}, w)\big)^q =
\beta((u + \lambda v)^{(q)}, w)^q
\end{align*}
showing that \(u + \lambda v \in V_{\mathrm{Herm}}\), whence the first
statement. Applying the defining property of Hermitian vectors twice shows
\[ \beta(u^{(q)},v) = \beta(v^{(q)},u)^q = \beta(u^{(q)},v)^{q^2} \]
proving the second statement. The final statement now follows.
\end{proof}

The last statement of \parref{forms-hermitian-basics} partially justifies
the nomenclature:

\begin{Corollary}\label{forms-hermitian-gram}
Suppose there is a basis \(V = \langle v_0,\ldots,v_n \rangle\) consisting of
Hermitian vectors. Then the associated Gram matrix is a Hermitian matrix
over \(\mathbf{F}_{q^2}\); that is,
\[
\pushQED{\qed}
\Gram(\beta;v_0,\ldots,v_n)^\vee =
\Gram(\beta;v_0,\ldots,v_n)^{(q)}.
\qedhere
\popQED
\]
\end{Corollary}

The Hermitian form \((V_{\mathrm{Herm}}, \beta_{\mathrm{Herm}})\) associated
is generally quite large; for instance, the next statement shows that
\(V_{\mathrm{Herm}}\) contains the radical of \(\beta\):

\begin{Lemma}\label{forms-hermitian-basics-kernels}
\(
\Fr^*(V)^\perp \cap \Fr^{-1}(V^\perp) =
V_{\mathrm{Herm}} \cap \Fr^*(V)^\perp =
V_{\mathrm{Herm}} \cap \Fr^{-1}(V^\perp)
\).
\end{Lemma}

\begin{proof}
It follows directly from definitions that the radical
\(\Fr^*(V)^\perp \cap \Fr^{-1}(V^\perp)\) of \(\beta\) is contained in
\(V_{\mathrm{Herm}}\), so it remains to show the second equality in the
statement. For \(v \in V_{\mathrm{Herm}}\), the defining property of Hermitian
vectors shows that, for any \(w \in V\),
\[ \beta(w^{(q)},v) = 0 \quad\Leftrightarrow\quad \beta(v^{(q)},w) = 0 \]
meaning \(v \in \Fr^*(V)^\perp\) if and only of \(v \in \Fr^{-1}(V^\perp)\),
giving the second equality.
\end{proof}

Taking Hermitian vectors is compatible with orthogonal decompositions:

\begin{Lemma}\label{forms-hermitian-basics-orthogonals}
An orthogonal decomposition
\((V,\beta) = (V',\beta') \perp (V'',\beta'')\)
induces an orthogonal decomposition of Hermitian spaces
\[
(V_{\mathrm{Herm}},\beta_{\mathrm{Herm}}) =
(V_{\mathrm{Herm}}',\beta_{\mathrm{Herm}}') \perp
(V_{\mathrm{Herm}}'',\beta_{\mathrm{Herm}}'').
\]
\end{Lemma}

\begin{proof}
Let \(v \in V_{\mathrm{Herm}}\), and let \(v = v' + v''\) be its unique
decomposition with \(v' \in V'\) and \(v'' \in V''\). Then it suffices to show
that
\[
v' \in
V_{\mathrm{Herm}}' \coloneqq
\Set{u \in V' |
\beta'(w^{(q)},u) =
\beta'(u^{(q)},w)^q
\;\text{for every}\;
w \in V'}
\]
and similarly for \(v''\) and \(V''\). Since \(V'\) and \(V''\) are orthogonal,
for every \(w \in V'\),
\[
\beta'(w^{(q)}, v')
= \beta(w^{(q)}, v' + v'')
= \beta((v' + v'')^{(q)}, w)^q
= \beta'(v'^{(q)},w)^q.
\]
Therefore \(v' \in V_{\mathrm{Herm}}'\). An analogous
argument shows \(v'' \in V_{\mathrm{Herm}}''\).
\end{proof}

The notion of an orthogonal to a Hermitian subspace is unambiguous:

\begin{Lemma}\label{forms-hermitian-subspace-basics}
If \(U \subseteq V\) is a Hermitian subspace, then
\(\Fr^*(U)^\perp = \Fr^{-1}(U^\perp)\).
\end{Lemma}

\begin{proof}
Choose a basis \(U = \langle u_1,\ldots,u_m \rangle\) of Hermitian vectors.
Then
\[
\Fr^*(U)^\perp = \bigcap\nolimits_{i = 1}^m \Fr^*(\langle u_i \rangle)^\perp
\quad\text{and}\quad
\Fr^{-1}(U^\perp) = \bigcap\nolimits_{i = 1}^m \Fr^{-1}(\langle u_i \rangle^\perp).
\]
Thus it suffices to consider the case when \(U = \langle u \rangle\) is spanned
by a single Hermitian vector. By the definition of a Hermitian vector,
\[
\Fr^*(U)^\perp =
\set{w \in V | \beta(u^{(q)},w) = 0} =
\set{w \in V | \beta(w^{(q)},u) = 0}
= \Fr^{-1}(U^\perp)
\]
since \(\beta\) takes values in the field \(\kk\).
\end{proof}

\begin{Corollary}\label{forms-orthogonal-complement-hermitian}
Let \(U \subseteq V\) be a Hermitian subspace such that \(\beta_U\) has no
radical. Then \(U\) has a unique orthogonal complement.
\end{Corollary}

\begin{proof}
By \parref{forms-hermitian-subspace-basics}, it makes sense to set
\(W \coloneqq \Fr^*(U)^\perp = \Fr^{-1}(U^\perp)\). Since \(\beta_U\) has
no radical, \(U \cap W = \{0\}\), yielding the first inequality in
\[
\dim_\kk V - \dim_\kk U \geq
\dim_\kk W \geq
\dim_\kk V - \dim_\kk U,
\]
and the second inequality is because \(W = \ker(V \to \Fr^*(U)^\vee)\). So
equality holds, and
\parref{forms-orthogonal-complements-criterion} implies \(W\) is the unique
orthogonal complement to \(U\).
\end{proof}

Symmetry of orthogonals in \parref{forms-hermitian-subspace-basics} has a
partial converse for \(1\)-dimensional spaces:

\begin{Lemma}\label{forms-hermitian-subspace-converse}
Assume \(\kk\) is separably closed. If \(L \subseteq V\) is a
\(1\)-dimensional subspace such that \(\Fr^*(L)^\perp = \Fr^{-1}(L^\perp)\),
then \(L\) is Hermitian.
\end{Lemma}

\begin{proof}
Let \(W \coloneqq \Fr^*(L)^\perp = \Fr^{-1}(L^\perp)\).
If \(W = V\), then \(L\) lies in the radical of \(\beta\) and so any nonzero
vector is Hermitian by \parref{forms-hermitian-basics-kernels}. So
suppose that \(W\) is of codimension \(1\) in \(V\). Let \(v \in L\) be any
nonzero vector. Then, for any \(\lambda \in \kk^\times\),
\[
\beta(-,\lambda v) - \beta((\lambda v)^{(q)},-)^q \colon
V \to \kk
\]
is a \(q\)-linear functional that is identically zero if and only if
\(\lambda v\) is Hermitian. In any case, it vanishes on \(W\), and hence passes
to the quotient \(V/W\); fix a nonzero \(\bar{w} \in V/W\) and observe that,
for any \(\mu \in \kk\),
\[
\beta((\mu \bar{w})^{(q)}, \lambda v) -
\beta((\lambda v)^{(q)}, \mu \bar{w})^q
=
\lambda \mu^q
\big(\beta(\bar{w}^{(q)}, v) -
\lambda^{q^2-1} \beta(v^{(q)}, \bar{w})\big).
\]
Since \(W\) the orthogonal of \(L\), neither \(\beta(\bar{w}^{(q)},v)\) nor
\(\beta(v^{(q)},\bar{w})\) vanish. Thus \(\lambda\) may be taken to be any
\((q^2-1)\)-st root of \(\beta(\bar{w}^{(q)},v)/\beta(v^{(q)},\bar{w})\).
Then \(\lambda v\) is Hermitian.
\end{proof}

\subsection{Equations for Hermitan vectors}\label{forms-hermitian-equations}
The subgroup \(V_{\mathrm{Herm}}\) of Hermitian vectors of \(V\) may be endowed
with the structure of a closed subgroup scheme of the affine space
\(\mathbf{A}V \coloneqq \Spec(\Sym(V^\vee))\) on \(V\), viewed as an additive
group scheme. Consider the two additive maps
\[
(v \mapsto \beta(-,v)) \colon V \to \Fr^*(V)^\vee
\quad\text{and}\quad
(v \mapsto \beta(v^{(q)},-)^q) \colon V \to \Fr^*(V)^\vee.
\]
The first map is linear over \(\kk\) and the second map is \(q^2\)-linear over
\(\kk\). Their duals induce ring homomorphisms
\[
\beta(-,v), \beta(v^{(q)},-)^q \colon
\Sym(\Fr^*(V)) \to \Sym(V^\vee)
\]
and hence morphisms of affine schemes \(\mathbf{A}V \to \mathbf{A}\Fr^*(V)^\vee\).
Their difference gives the equations for the subscheme of Hermitian points:
\[
\mathbf{A}V_{\mathrm{Herm}} =
\mathrm{V}\big(\beta(-,v) - \beta(v^{(q)},-)^q\big).
\]
Explicitly, choose a basis \(V = \langle v_0,\ldots,v_n \rangle\),
let \(B \coloneqq \Gram(\beta;v_0,\ldots,v_n)\) be the corresponding
Gram matrix, and
let \(\mathbf{x}^\vee \coloneqq (x_0,\ldots,x_n)\) be the corresponding coordinates for
\(\mathbf{A}V \cong \mathbf{A}^{n+1}\).
Then the equations for \(\mathbf{A}V_{\mathrm{Herm}}\)
may be expressed matricially as
\[ B \mathbf{x} - B^{(q),\vee} \mathbf{x}^{(q^2)} = 0. \]

\subsection{Examples}\label{forms-hermitian-examples}
Here are examples illustrating the structure of the scheme of Hermitian vectors.
Notation for \(q\)-bic forms is as in \parref{forms-standard}.
\begin{enumerate}
\item\label{forms-hermitian-examples.split}
Let \((V,\beta) = (\kk^{\oplus n+1},\mathbf{1}^{\oplus n+1})\) be the \(q\)-bic
form with Gram matrix given by the identity matrix. Then the group of Hermitian
vectors
\[
\mathbf{A}V_{\mathrm{Herm}} =
\Spec\big(\kk[x_0,\ldots,x_n]/(x_0 - x_0^{q^2}, \ldots, x_n - x_n^{q^2})\big)
\]
is isomorphic to the \'etale group scheme \(\mathbf{F}_{q^2}^{\oplus n+1}\).
\item\label{forms-hermitian-examples.twisted}
Let \(k\) be any field and let \(\kk \coloneqq k(t)\). Let
\(V = \langle v \rangle\) be a \(1\)-dimensional vector space over \(\kk\), and
let \(\beta \colon \Fr^*(V) \otimes_\kk V \to \kk\) be the \(q\)-bic form
determined by \(\beta(v^{(q)},v) = -t\). The subscheme of Hermitan vectors is
given by
\[
\mathbf{A}V_{\mathrm{Herm}} =
\Spec\big(\kk[x]/(t^{q-1}x - x^{q^2})\big).
\]
This is a form of \(\mathbf{F}_{q^2}\) which splits along the
extension \(k(t) \subset k(t^{1/(q+1)})\).
\item\label{forms-hermitian-examples.nilpotent}
Let \((V,\beta) = (\mathbf{k}^{\oplus n+1},\mathbf{N}_{n+1})\) be the \(q\)-bic
form of dimension \(n+1\) with Gram matrix given by a Jordan block of size \(n+1\)
with \(0\) on the diagonals. Then the subscheme of Hermitian points is given
by
\[
\mathbf{A}V_{\mathrm{Herm}} =
\Spec\big(\kk[x_0,\ldots,x_n]/(x_1,x_2 - x_0^{q^2},\ldots, x_n - x_{n-2}^{q^2},x_{n-1}^{q^2})\big).
\]
The structure of this scheme depends on the parity of \(n\):
\[
\frac{\kk[x_0,\ldots,x_n]}{(x_1,x_2 - x_0^{q^2},\ldots,x_n- x_{n-2}^{q^2},x_{n-1}^{q^2})}
\cong
\begin{dcases*}
\kk[x_0]/(x_0^{q^{2m}}) & if \(n = 2m-1\), and \\
\kk[x_0] & if \(n = 2m\).
\end{dcases*}
\]
Indeed, in both cases, the ideal gives the equations
\[
x_1 = x_3 = \cdots = x_{2m-1} = 0
\quad\text{and}\quad
x_{2m-2} = x_{2m-4}^{q^2} = \cdots = x_0^{q^{2m-2}}.
\]
The difference is the equation \(x_{n-1}^{q^2} = 0\): when \(n = 2m\), this is
implied by the vanishing of the odd-indexed variables; when \(n = 2m-1\), this
shows \(0 = x_{2m-2}^{q^2} = x_0^{q^{2m}}\).
\end{enumerate}

\subsection{Hermitian vectors of nonsingular forms}\label{forms-hermitian-nondegenerate-section}
The remainder of this Section is devoted to the study of Hermitian vectors
and subspaces for a nonsingular form \((V,\beta)\). In this setting, these
notions were already isolated by Shimada in
\cite[Definition 2.11]{Shimada:Lattices}; their identification as fixed points
of a map, as will be done in \parref{forms-hermitian-fixed}, was also observed
in \cite[Remark 2.13]{Shimada:Lattices}.

Examples \parref{forms-hermitian-examples}\ref{forms-hermitian-examples.split}
and \parref{forms-hermitian-examples}\ref{forms-hermitian-examples.twisted}
indicate that the structure of the subgroup of Hermitian vectors for a
nonsingular \(q\)-bic form is quite simple.

\begin{Lemma}\label{forms-hermitian-nondegenerate}
If \((V,\beta)\) is nonsingular, then
\(\mathbf{A}V_{\mathrm{Herm}}\) is an \'etale group scheme
of degree \(q^2(n+1)\) over \(\kk\), geometrically isomorphic to
\(\mathbf{F}_{q^2}^{\oplus n+1}\).
\end{Lemma}

\begin{proof}
Consider the equations given for \(\mathbf{A}V_{\mathrm{Herm}}\) in
\parref{forms-hermitian-equations}. Since \(\beta\) is nonsingular, its
Gram matrix is invertible and the equations may be expressed as
\[ \mathbf{x}^{(q^2)} = B^{(q),\vee,-1} B \mathbf{x}. \]
This is a system of \(n+1\) equations in \(n+1\) variables.
The Jacobian of this system of equations is given by \(B^{(q),\vee,-1} B\);
this is of full rank since \(\beta\) is nonsingular. Therefore
\(\mathbf{A}V_{\mathrm{Herm}}\) is \'etale of degree \(q^2(n+1)\) over \(\kk\).
By \parref{forms-hermitian-basics}, \(\mathbf{A}V_{\mathrm{Herm}}\) is a
vector space over \(\mathbf{F}_{q^2}\), so it must be geometrically isomorphic
to \(\mathbf{F}_{q^2}^{\oplus n+1}\).
\end{proof}

The following shows that if \((V,\beta)\) is nonsingular, then \(V\) is spanned
by its Hermitian vectors after a separable field extension; in other words,
after a base change, \(V\) itself is a Hermitian subspace. This is not true for
general forms:
\parref{forms-hermitian-examples}\ref{forms-hermitian-examples.nilpotent} shows
that a form of type \(\mathbf{N}_{2m-1}\) has no nonzero Hermitian vectors.

\begin{Proposition}\label{forms-hermitian-nondegenerate-span}
Assume that \(\kk\) is separably closed. If \((V,\beta)\) is nonsingular, then
the natural map \(V_{\mathrm{Herm}} \otimes_{\mathbf{F}_{q^2}} \kk \to V\) is
an isomorphism.
\end{Proposition}

\begin{proof}
The two \(\kk\)-vector spaces have the same dimension by
\parref{forms-hermitian-nondegenerate}, so it suffices to show that the map
\(V_{\mathrm{Herm}} \otimes_{\mathbf{F}_{q^2}} \kk \to V\) is injective.
Suppose for sake of contradiction that it is not injective. Then there is a
linear relation in \(V\) of the form
\[
v_{k+1} = a_0 v_0 + \cdots + a_k v_k
\quad\text{with}\; k \geq 0,\; v_i \in V_{\mathrm{Herm}},\;\text{and}\; a_i \in \kk.
\]
Choose such a relation with \(k\) minimal. Minimality implies that the
\(v_0,\ldots,v_k\) are linearly independent in \(V\). Since \(\beta\) is
nonsingular, there exists \(w \in V\) such that \(\beta(w^{(q)},v_i) = 0\) for
\(0 \leq i \leq k-1\), and \(\beta(w^{(q)},v_k) \neq 0\) which, up to scaling
\(v_k\), may be taken to be \(1\). Since each of the vectors
\(v_0,\ldots,v_{k+1}\) are Hermitian,
\begin{align*}
a_k
= \beta(w^{(q)}, v_{k+1})
& = \beta(v_{k+1}^{(q)},w)^q \\
& = \sum\nolimits_{i = 0}^k a_i^{q^2} \beta(v_i^{(q)},w)^q
= \sum\nolimits_{i = 0}^k a_i^{q^2} \beta(w^{(q)},v_i)
= a_k^{q^2}
\end{align*}
and so \(a_k \in \mathbf{F}_{q^2}\). But then
\(v_k' \coloneqq v_{k+1} - a_k v_k\) lies in \(V_{\mathrm{Herm}}\) by
\parref{forms-hermitian-basics} and there is a linear relation in \(V\) given by
\[ v_k' = a_0 v_0 + \cdots + a_{k-1} v_{k-1}. \]
This contradicts the minimality of the original linear relation, thereby showing
the injectivity of \(V_{\mathrm{Herm}} \otimes_{\mathbf{F}_{q^2}} \kk \to V\).
\end{proof}

This implies that there is always a basis for which the Gram matrix of a
nonsingular form is the identity matrix; compare with \parref{forms-hermitian-gram}.

\begin{Corollary}\label{forms-hermitian-diagonal}
A nonsingular \(q\)-bic form \((V,\beta)\) over a separably closed field has an
orthonormal basis: there exists a basis
\(V = \langle v_0,\ldots,v_n \rangle\) such that for \(0 \leq i, j \leq n\),
\[
\beta(v_i^{(q)},v_j) =
\begin{dcases*}
0 & if \(i \neq j\), and \\
1 & if \(i = j\).
\end{dcases*}
\]
\end{Corollary}

\begin{proof}
By \parref{forms-hermitian-nondegenerate-span}, it suffices to show that
the Hermitian form \((V_{\mathrm{Herm}},\beta_{\mathrm{Herm}})\) from
\parref{forms-hermitian-basics} admits such a basis. So
replace \(V\) by \(V_{\mathrm{Herm}}\) and work over \(\mathbf{F}_{q^2}\).

Choose a basis \(V = \langle u_0,\ldots,u_n \rangle\) such that
\(\beta(u_i^{(q)},u_j) \neq 0\) if and only if \(i = j\). This can be
achieved, for instance, by successively choosing a nonisotropic vector and then
taking its orthogonal complement, possible by
\parref{forms-orthogonal-complement-hermitian}. Since \(\beta\) is Hermitian,
\[
\beta_i \coloneqq
\beta(u_i^{(q)},u_i) =
\beta(u_i^{(q)},u_i)^q
\quad\text{for each}\;
0 \leq i \leq n.
\]
Therefore \(\beta_i \in \mathbf{F}_q\).
The image of the \((q+1)\)-power map
\(\mathbf{F}_{q^2}^\times \to \mathbf{F}_{q^2}^\times\) is the
cyclic subgroup of order \(q-1\) given by \(\mathbf{F}_q^\times\), so every
element of \(\mathbf{F}_q\) has a \((q+1)\)-st root in \(\mathbf{F}_{q^2}\).
Then taking
\[
v_i \coloneqq \beta_i^{-1/(q+1)} u_i
\quad\text{for each}\;
0 \leq i \leq n
\]
gives the desired basis.
\end{proof}

Questions about Hermitian subspaces in all of \(V\) may be reduced to questions
about subspaces of the associated Hermitian form \(V_{\mathrm{Herm}}\). For
instance, isotropic Hermitian subspaces are always contained in a maximal such:

\begin{Corollary}\label{forms-hermitian-maximal-isotropic}
Assume that \(\kk\) is separably closed. If \((V,\beta)\) is nonsingular, then
any isotropic Hermitian subspace \(U \subseteq V\) is contained in an isotropic
Hermitian subspace of dimension \(\lfloor \frac{n+1}{2} \rfloor\).
\end{Corollary}

\begin{proof}
Set \(U_{\mathrm{Herm}} \coloneqq U \cap V_{\mathrm{Herm}}\). Then
\parref{forms-hermitian-nondegenerate-span} implies that
\(U_{\mathrm{Herm}} \otimes_{\mathbf{F}_{q^2}} \kk \to U\) is an isomorphism,
and so it suffices to see that \(U_{\mathrm{Herm}}\) is contained in a maximal
isotropic subspace of \(V_{\mathrm{Herm}}\). This is standard:
\(U_{\mathrm{Herm}}\) can be completed to a sum of hyperbolic planes, at
which point \parref{forms-orthogonal-complement-hermitian} gives an orthogonal
complement; then induction applied to the zero subspace of the complement gives
the desired maximal isotropic subspace.
\end{proof}

\subsection{Hermitian self-map}\label{forms-hermitian-endomorphism}
The Hermitian vectors of a nonsingular \(q\)-bic form \((V,\beta)\) can be
understood in a different way via a canonical self-map of \(V\) induced by
\(\beta\). Let \((\Fr^*(V), \Fr^*(\beta))\) be the Frobenius twist of \((V,\beta)\),
see \parref{forms-fr-twist}. Consider the \(q^2\)-linear map
\[
\phi \coloneqq
(\beta^{-1} \circ \Fr^*(\beta)^\vee) \circ (-)^{(q^2)} \colon
V \to \Fr^{2,*}(V) \to V
\]
obtained by composing the \(q^2\)-linear map \((-)^{(q^2)} \colon V \to \Fr^{2,*}(V)\)
with the linear isomorphism \(\beta^{-1} \circ \Fr^*(\beta)^\vee \colon \Fr^{2,*}(V) \to V\).
This is a canonical self-map of \(V\) which expresses a certain symmetry property
of the form:

\begin{Lemma}\label{forms-endomorphism-V-identity}
The \(q^2\)-linear map \(\phi \colon V \to V\) satisfies
\[
\beta(w, \phi(v)) =
\Fr^*(\beta)(v^{(q^2)},w)
\quad\text{for all}\;
v \in V
\;\text{and}\;
w \in \Fr^*(V).
\]
\end{Lemma}

\begin{proof}
Indeed, compute using the definition of \(\phi\):
\begin{align*}
\beta(w,\phi(v))
& = w^\vee \circ \beta \circ (\beta^{-1} \circ \Fr^*(\beta)^\vee \circ v^{(q^2)}) \\
& = w^\vee \circ \Fr^*(\beta)^\vee \circ v^{(q^2)}
= \Fr^*(\beta)(v^{(q^2)},w).
\qedhere
\end{align*}
\end{proof}

\begin{Corollary}\label{forms-hermitian-phi-phi}
Let \((V,\beta)\) be a nonsingular \(q\)-bic form. Then for every \(v,w \in V\),
\[
\beta(\phi(v)^{(q)}, \phi(w)) =
\beta(v^{(q)},w)^{q^2}.
\]
In particular, \(v\) is isotropic if and only of \(\phi(v)\) is isotropic.
\end{Corollary}

\begin{proof}
Apply \parref{forms-endomorphism-V-identity} twice and use that
\(\Fr^*(\beta)(x^{(q)},y^{(q)}) = \beta(x,y)^q\) from \parref{forms-fr-twist}:
\[
\beta(\phi(v)^{(q)},\phi(w))
= \beta(w^{(q)}, \phi(v))^q
= \beta(v^{(q)}, w)^{q^2}
\qedhere
\]
\end{proof}

The following shows that linear subspaces fixed by the canonical self-map
\(\phi\) are precisely the Hermitian subspaces, as defined in
\parref{forms-hermitian-basics}:

\begin{Corollary}\label{forms-hermitian-fixed}
Let \((V,\beta)\) be a nonsingular \(q\)-bic form. Then
\begin{enumerate}
\item\label{forms-hermitian-fixed.vector}
a vector \(v \in V\) is Hermitian if and only if \(\phi(v) = v\); and
\item\label{forms-hermitian-fixed.subspace}
a subspace \(U \subseteq V\) is Hermitian if and only if \(\phi(U) = U\).
\end{enumerate}
\end{Corollary}

\begin{proof}
If \(v\) is Hermitian, then by its definition in \parref{forms-hermitian},
\[
\beta(w^{(q)},v)
= \beta(v^{(q)},w)^q
= \Fr^*(\beta)(v^{(q^2)}, w^{(q)})
= \beta(w^{(q)},\phi(v))
\]
for all \(w \in V\); here, the second equality is due to the definition of
\(\Fr^*(\beta)\) as in \parref{forms-fr-twist}, and
the third equality is due to \parref{forms-endomorphism-V-identity}. Since
\(\beta\) is nonsingular, this implies that \(\phi(v) = v\). Reversing
the argument shows that if \(\phi(v) = v\), then \(v\) is Hermitian.
This proves \ref{forms-hermitian-fixed.vector}.

Consider \ref{forms-hermitian-fixed.subspace}. If \(U \subseteq V\) is
Hermitian, then \(U\) admits a basis \(\langle v_1,\ldots,v_m \rangle\)
consisting of Hermitian vectors. Thus
\[
\phi(U) =
\langle \phi(v_1), \ldots, \phi(v_m) \rangle =
\langle v_1, \ldots, v_m \rangle =
U
\]
by \ref{forms-hermitian-fixed.vector}. Conversely, if \(\phi(U) = U\), then
\(\phi\) restricts to a map \(U \to U\). This is a bijective \(q^2\)-linear
map, so there is a basis of \(U\) consisting of fixed vectors, see
\cite[Expos\'e XXII, 1.1]{SGAVII}. Thus \(U\) is Hermitian by
\ref{forms-hermitian-fixed.vector}.
\end{proof}

This gives a way of deciding when an arbitrary vector is contained in a
Hermitian subspace of a particular dimension:

\begin{Corollary}\label{forms-hermitian-min-contain}
Let \((V,\beta)\) be a nonsingular \(q\)-bic form. Then for any \(v \in V\),
\[
\min\set{\dim_\kk U | U\;\text{a Hermitian subspace containing \(v\)}}
= \dim_\kk \langle \phi^i(v) \mid i \geq 0 \rangle.
\]
\end{Corollary}

\begin{proof}
If \(v\) is contained in a Hermitian subspace \(U\), then
\parref{forms-hermitian-fixed}\ref{forms-hermitian-fixed.subspace} implies
\(\phi^i(v) \in U\) for every \(i \geq 0\). This shows
the inequality ``\(\geq\)'' between the quantities in the statement. On the
other hand, the space \(\langle \phi^i(v) \mid i \geq 0 \rangle\) is fixed under
\(\phi\) and so it is Hermitian by
\parref{forms-hermitian-fixed}\ref{forms-hermitian-fixed.subspace}. This
proves the inequality ``\(\leq\)''.
\end{proof}

\section{Automorphisms}\label{section-forms-automorphisms}
An automorphism of a \(q\)-bic form \((M,\beta)\) over a
\(\mathbf{F}_{q^2}\)-algebra \(R\) is a self-isomorphism of \((M,\beta)\),
as defined in \parref{forms-definition}; the set of automorphisms forms a group
\(\Aut(M,\beta)\). When \(M\) is a finite projective \(R\)-module, the group
of automorphisms may be enriched to an \(R\)-group scheme
\(\AutSch(M,\beta)\), see \parref{forms-aut-schemes}. This Section discusses
a few basic properties of these group schemes, and provides a few simple
examples.

\subsection{Automorphism group schemes}\label{forms-aut-schemes}
Let \((M,\beta)\) be a \(q\)-bic form over a \(\mathbf{F}_{q^2}\)-algebra \(R\).
Consider the group-valued functor \(\mathrm{Alg}_R \to \mathrm{Grps}\) on the
category of \(R\)-algebras given by
\[
S \mapsto \Aut(M \otimes_R S, \beta \otimes \id_S).
\]
This is the subfunctor of the functor \(\mathbf{GL}(M)\) of linear
automorphisms of \(M\), specified as the stabilizer of the element
\(\beta \in \Hom_R(\Fr^*(M) \otimes_R M, R)\). Thus when \(M\) is a finite
projective \(R\)-module, this is representable by a closed subgroup scheme
of \(\GL(M)\), see \cite[II.1.2.4 and II.1.2.6]{DG}. The representing scheme
is denoted \(\AutSch(M,\beta)\) and is referred to as the \emph{automorphism
group scheme of \((M,\beta)\)}.

The automorphism group scheme is \emph{a priori} contained in a much smaller
closed subscheme of \(\GL(M)\):

\begin{Lemma}\label{forms-aut-canonical-filtration}
Let \((M,\beta)\) be a \(q\)-bic form over \(R\). Assume \(M\) is
finite projective. Then
\(\AutSch(M,\beta)\) stabilizes the associated \(\perp\)- and
\(\Fr^*(\perp)\)-filtrations.
\end{Lemma}

\begin{proof}
By their description in \parref{forms-canonical-filtration} and
\parref{forms-canonical-filtration-second}, the formation of the two
filtrations commutes with extension of scalars to any \(R\)-algebra \(S\).
Since the filtrations are constructed by twisting by \(\Fr\) and by taking
successive kernels with respect to \(\beta\), it follows that they are
preserved by the action of the automorphism groups
\(\Aut(M \otimes_R S, \beta \otimes \id_S)\). Thus the filtrations are
stabilized by the entire automorphism group scheme.
\end{proof}

Orthogonal decompositions give a simple construction of subgroup schemes:

\begin{Lemma}\label{forms-aut-orthogonal-sum}
Let \((M,\beta)\) be a \(q\)-bic form over \(R\). An orthogonal decomposition
\((M,\beta) = (M_1,\beta_1) \perp (M_2,\beta_2)\) induces an inclusion of group
schemes
\[
\AutSch(M_1,\beta_1) \times \AutSch(M_2,\beta_2) \subseteq
\AutSch(M,\beta).
\]
\end{Lemma}

\begin{proof}
The action of \(\AutSch(M_i,\beta_i)\) on \(M_i\) extends to one on \(M\)
via the trivial action on the complement.
\end{proof}

The remainder of this Section is concerned with \(q\)-bic forms \((V,\beta)\)
over a field \(\kk\). The next statement identifies the tangent space to the
identity of \(\AutSch(V,\beta)\):

\begin{Proposition}\label{forms-aut-tangent-space}
Let \((V,\beta)\) be a \(q\)-bic form over a field \(\kk\). Then
there is a canonical isomorphism of \(\kk\)-vector spaces
\[
\mathrm{Lie}(\AutSch(V,\beta)) \cong
\Hom_\kk(V,\Fr^*(V)^\perp).
\]
\end{Proposition}

\begin{proof}
The Lie algebra of \(\AutSch(V,\beta)\) is the \(\kk\)-vector space of
isomorphisms
\[
\varphi \colon
V \otimes_\kk \kk[\epsilon]/(\epsilon^2) \to
V \otimes_\kk \kk[\epsilon]/(\epsilon^2)
\]
of \(\kk[\epsilon]/(\epsilon^2)\)-modules which restrict to the identity upon
setting \(\epsilon = 0\), and which preserve \(\beta\). The first condition
means that \(\varphi\) is determined by the \(\kk\)-linear map
\[
\epsilon \bar\varphi \coloneqq
\varphi - \id \colon
V \otimes_\kk \kk[\epsilon]/(\epsilon) \to
V \otimes_\kk (\epsilon)/(\epsilon^2).
\]
Write \(\varphi = \id + \epsilon \bar\varphi\). Then since \(\epsilon^2 = 0\),
that \(\varphi\) preserves \(\beta\) means
\[
\beta(v,w) =
\beta(\Fr^*(\id + \epsilon\bar\varphi)(v), (\id + \epsilon\bar\varphi)(w)) =
\beta(v,w) + \beta(v,\epsilon \bar\varphi(\bar{w}))
\]
for every \(v \in \Fr^*(V \otimes_\kk \kk[\epsilon]/(\epsilon^2))\) and
\(w \in V \otimes_\kk \kk[\epsilon]/(\epsilon^2)\), and where
\(\bar{w} \in V \otimes_\kk \kk[\epsilon]/(\epsilon)\) is the image of \(w\)
under the quotient map. Viewing \(\bar\varphi\) as a linear map \(V \to V\),
this means that \(\beta(v,\bar\varphi(w)) = 0\) for all
\(v \in \Fr^*(V)\) and \(w \in V\). Thus \(\bar\varphi\) factors through
\(\Fr^*(V)^\perp\) and the map \(\varphi \mapsto \bar\varphi\) determines an
isomorphism from \(\mathrm{Lie}(\AutSch(V,\beta))\) to
\(\Hom_\kk(V,\Fr^*(V)^\perp)\).
\end{proof}

\subsection{Unitary groups}\label{forms-aut-unitary}
In the case that \((V,\beta)\) is a nonsingular \(q\)-bic form, it
follows from \parref{forms-aut-tangent-space} that its automorphism group
scheme is reduced and finite. Write
\[ \mathrm{U}(V,\beta) \coloneqq \AutSch(V,\beta) \]
and call it the \emph{unitary group} of \(\beta\).
For the standard nonsingular form
\((\kk^{\oplus n+1}, \mathbf{1}^{\oplus n+1})\),
this is the classical finite unitary group
\(\mathrm{U}_{n+1}(q)\), as in \cite[\S2.1]{ATLAS}.

To give another description of this group, observe that \(g \in \GL(V)\) lies
in \(\mathrm{U}(V,\beta)\) if and only if the following diagram commutes:
\[
\begin{tikzcd}[column sep=4em]
V \rar["g"'] \dar["\beta"'] & V \\
\Fr^*(V)^\vee \rar["\Fr^*(g)^{\vee,-1}"] & \Fr^*(V)^\vee \uar["\beta^{-1}"']
\end{tikzcd}
\]
In other words, letting \(F \colon \mathbf{GL}(V) \to \GL(V)\) be the morphism
of algebraic groups determined
by \(g \mapsto \beta^{-1} \circ \Fr^*(g)^{\vee,-1} \circ \beta\), the unitary
group is the subgroup of fixed points of \(F\):
\[ \mathrm{U}(V,\beta) = \GL(V)^F. \]
Compare this description with \cite[Lecture 11]{Steinberg:Lectures}.

\subsection{Type \texorpdfstring{\(\mathbf{N}_2^{\oplus a} \oplus \mathbf{1}^{\oplus b}\)}{1^a+N2^b}}
\label{forms-aut-1^a+N2^b}
Automorphism group schemes of singular \(q\)-bic forms are intricate in
rather different ways. For example, let \(a,b \geq 0\) be
integers and consider a \(q\)-bic form \((V,\beta)\) of type
\(\mathbf{N}_2^{\oplus a} \oplus \mathbf{1}^{\oplus b}\) over a
perfect field \(\kk\), see \parref{forms-standard} for notation. Let
\[
U_- \coloneqq \Fr^*(V)^\perp
\quad\text{and}\quad
U_+ \coloneqq \Fr^{-1}(V^\perp)
\]
be kernels of \(\beta\); since \(\kk\) is perfect
\((-)^{(q)} \colon V \to \Fr^*(V)\) is a bijection and so \(U_+\) is
\(a\)-dimensional. Set \(U \coloneqq U_- \oplus U_+\). Then the restricted
\(q\)-bic form \((U,\beta_U)\) is of type \(\mathbf{N}_2^{\oplus a}\), and so
induces an isomorphism \(\beta_U\rvert_{U_+} \colon U_+ \to \Fr^*(U_-)\). It
follows from \parref{forms-orthogonal-complements-criterion} that \(U\) fits
into a unique orthogonal decomposition
\[
(V,\beta) \cong
(U,\beta_U) \perp
(W,\beta_W)
\]
where \(W \coloneqq \Fr^*(U)^\perp \cap \Fr^{-1}(U^\perp)\) and \((W,\beta_W)\)
is of type \(\mathbf{1}^{\oplus b}\). The following computes the automorphism
group scheme in general, see also \parref{qbic-points-automorphisms.N2},
\parref{curves-1+N2.auts}, \parref{surfaces-1+1+N2.auts}, and
\parref{surfaces-N2+N2.auts} for explicit low dimension expressions.
Notation: given a group scheme \(\mathbf{G}\) over \(\kk\), write
\(\mathbf{G}[\Fr] \coloneqq \ker(\Fr \colon \mathbf{G} \to \mathbf{G})\)
for the subgroup scheme obtained as the kernel of the \(q\)-power absolute
Frobenius homomorphism.

\begin{Proposition}\label{forms-aut-1^a+N2^b.computation}
Let \((V,\beta)\) be a \(q\)-bic form of type
\(\mathbf{N}_2^{\oplus a} \oplus \mathbf{1}^{\oplus b}\) over a perfect field.
Then \(\AutSch(V,\beta)\) is isomorphic to the \(a^2\)-dimensional closed
subgroup scheme of \(\GL_{2a + b}\) consisting of
\[
\left(
\begin{array}{cc|@{}c}
A_- & B   & \;\;\mathbf{y}^\vee \\
0   & A_+ & 0 \\
\hline
0 & \mathbf{x} & C
\end{array}
\right)
\]
where \(A_\pm \in \GL(U_\pm)\),
\(B \in \HomSch(U_+,U_-)[\Fr]\),
\(C \in \mathrm{U}_a(q)\),
\(\mathbf{x} \in \HomSch(U_+,W)[\Fr]\), and
\(\mathbf{y} \in \HomSch(U_-,W)[\Fr^2]\), subject to the equations
\[
A_-^{(q),\vee} \cdot \beta_U\rvert_{U_-} \cdot A_+ = \beta_U\rvert_{U_-}
\quad\text{and}\quad
C^{(q),\vee} \cdot \beta_W \cdot \mathbf{x} +
\mathbf{y}^{(q)} \cdot \beta_U\rvert_{U_-} \cdot A_- = 0.
\]
\end{Proposition}

\begin{proof}
Let \((V,\beta) \cong (U,\beta_U) \perp (W,\beta_W)\) be the orthogonal
decomposition provided by \parref{forms-aut-1^a+N2^b}, using that the base
field is perfect. The \(\perp\)-filtration is given by \(U_- \subset U_- \oplus W\),
whereas the first step of the \(\Fr^*(\perp)\)-filtration is \(\Fr^*(U_+)\),
see \parref{forms-canonical-filtration} and
\parref{forms-canonical-filtration-second}. Thus by
\parref{forms-aut-canonical-filtration}, the automorphism group scheme is the
closed subgroup scheme of \(\GL(U_- \oplus U_+ \oplus W)\) consisting of block
matrices that satisfy
\[
\left(\begin{array}{cc|@{}c}
A_-^\vee & 0        & 0 \\
0        & A_+^\vee & 0 \\
\hline
\mathbf{y} & 0 & \;\;C^\vee
\end{array}\right)^{\!\!(q)}
\left(\begin{array}{cc|@{}c}
0 & \beta_U\rvert_{U_+} & 0 \\
0 & 0 & 0 \\
\hline
0 & 0 & \;\;\beta_W
\end{array}\right)
\left(\begin{array}{cc|@{}c}
A_- & B   & \;\;\mathbf{y}^\vee \\
0   & A_+ & 0 \\
\hline
0   & \mathbf{x}   & C
\end{array}\right) =
\left(\begin{array}{cc|@{}c}
0 & \beta_U\rvert_{U_+} & 0 \\
0 & 0 & 0 \\
\hline
0 & 0 & \;\;\beta_W
\end{array}\right)
\]
such that \(A_\pm \in \GL(U_\pm)\),
\(B \in \HomSch(U_+,U_-)[\Fr]\),
\(C \in \GL(W)\),
\(\mathbf{x} \in \HomSch(U_+,W)[\Fr]\), and
\(\mathbf{y} \in \HomSch(U_-,W)\).
Expanding shows that \(C \in \mathrm{U}_b(q)\) and gives the \(2\) equations
in the statement, completing the computation.
\end{proof}

Another simple cases is when the form has a radical:

\begin{Lemma}\label{forms-automorphisms-cones}
Let \((V,\beta)\) be a \(q\)-bic form. Let \((\bar{V},\beta_{\bar{V}})\) be the
\(q\)-bic form obtained by dividing out the radical \(L \coloneqq \rad(\beta)\).
Then there is a block matrix decomposition
\[
\AutSch(V,\beta) \cong
\begin{pmatrix}
\GL(L) & \HomSch(L,\bar{V}) \\
       & \AutSch(\bar{V},\beta_{\bar{V}})
\end{pmatrix}
\]
so in particular,
\(\dim\AutSch(V,\beta) = \dim\AutSch(\bar{V},\beta_{\bar{V}}) + \dim_\kk L \cdot \dim_\kk V\).
\end{Lemma}

\begin{proof}
This follows from a matrix calculation upon observing that \(\AutSch(V,\beta)\)
must preserve the radical.
\end{proof}

\section{Classification of \texorpdfstring{\(q\)}{q}-bic forms}\label{section-forms-classification}
Over an algebraically closed field, \(q\)-bic forms are classified up to
isomorphism by the finitely many numerical invariants arising from its
\(\perp\)- and \(\Fr^*(\perp)\)-filtrations; in short, this means that every
\(q\)-bic form is a standard in the sense of \parref{forms-standard}:

\begin{Theorem}\label{forms-classification-theorem}
Let \((V,\beta)\) be a \(q\)-bic form of dimension \(n+1\) over an
algebraically closed field \(\kk\). Then there exists a basis
\(V = \langle e_0,\ldots,e_n \rangle\) such that
\[
\Gram(\beta;e_0,\ldots,e_n) =
\mathbf{N}_1^{\oplus a_1} \oplus
\cdots \oplus
\mathbf{N}_m^{\oplus a_m} \oplus
\mathbf{1}^{\oplus b}
\]
for \(m,a_1,\ldots,a_m,b \in \mathbf{Z}_{\geq 0}\) such that
\(b + \sum_{k = 1}^m k a_k = n+1\).
\end{Theorem}

\subsection{Remarks on the classification theorem}\label{forms-classification-theorem-remarks}
A proof of the classification theorem using the abstract theory of \(q\)-bic
forms will appear elsewhere. The observation is that a \(q\)-bic form is
completely determined by the numerical characteristics of its \(\perp\)- and
\(\Fr^*(\perp)\)-filtrations, as defined in \parref{forms-canonical-filtration}
and \parref{forms-canonical-filtration-second}, and a suitable basis can be
constructed by examining how the two filtrations interact.

During the preparation of this work, another group has independently discovered
this classification and given a proof via explicit matrix methods: see
\cite[Theorem 7.1]{KKPSSW:F-Pure}. Various partial cases of
\parref{forms-classification-theorem} have been known for much longer. When
\(\beta\) is Hermitian, this classification was well-known to authors working
in finite projective geometry, see \cite{Segre:Hermitian, BC:Hermitian}. The
case when \(\beta\) is nonsingular has been rediscovered several times, see
\cite{Hefez:Thesis, Beauville:Moduli}. The case when \(\beta\) is of corank
\(1\) was established by Hoai Hoang in \cite{HH:Fermat} using explicit matrix
methods.

\subsection{Remark regarding general fields}\label{forms-classification-imperfect}
The classification of \(q\)-bic forms over non-closed fields is decidedly more
subtle. For instance, the two \(q\)-bic forms
\[
\Big(
  \mathbf{F}_{q^2}^{\oplus 3},
  \left(\begin{smallmatrix} 1 \\ & 1 \\ & & 1 \end{smallmatrix}\right)
\Big)
\quad\text{and}\quad
\Big(
  \mathbf{F}_{q^2}^{\oplus 3},
  \left(\begin{smallmatrix} 0 & 1 \\ -1 & 0 \\ & & 1 \end{smallmatrix}\right)
\Big)
\]
are not isomorphic over \(\mathbf{F}_{q^2}\), as can be seen by showing that
the schemes of Hermitian vectors, as from \parref{forms-hermitian-equations},
are not isomorphic over \(\mathbf{F}_{q^2}\); see also \cite[Exercise 2.4]{Bonnafe:SL2}.

Even for separably closed fields, the situation is quite complicated. For
instance, let \(k\) be any field and let \(\kk\) be the separable closure of
the function field \(k(t)\). Consider a \(2\)-dimensional \(q\)-bic form over
\(\kk\) given by
\[
\Big(\kk^{\oplus 2},
\left(
\begin{smallmatrix}
0 & f \\
0 & g
\end{smallmatrix}
\right)
\Big)
\quad\text{for some}\; f, g \in \kk.
\]
Whenever \(f \neq 0\), this isomorphic to a form of type
\(\mathbf{N}_2\) upon passing to the perfect closure of \(\kk\). However,
such an isomorphism is defined over \(\kk\) if and only if \(g\) is a
\(q\)-power: Indeed, a direct computation shows that the change of coordinates
bringing the above form to the standard \(\mathbf{N}_2\) form is given by
\[
\begin{pmatrix}
1 & -g^{1/q}/f \\
0 & 1/f
\end{pmatrix}
\]
where this is viewed as a matrix over the perfect closure of \(\kk\).

\subsection{Moduli of \(q\)-bic forms}\label{forms-classification-moduli}
A parameter space for the set of \(q\)-bic forms on \(V\) is given by the
\((n+1)^2\)-dimensional affine space
\[
\qbics(V) \coloneqq
\mathbf{A}(\Fr^*(V)^\vee \otimes V^\vee).
\]
This carries a universal \(q\)-bic form
\[
\beta_{\mathrm{univ}} \colon
\Fr^*(V) \otimes V \otimes \sO_{\qbics(V)} \to
\kk
\]
such that the fibre over a point \([\beta] \in \qbics(V)\) recovers \(\beta\).
A choice of basis \(V = \langle e_0,\ldots,e_n \rangle\) yields an isomorphism
\(\Gram \colon \qbics(V) \to \mathbf{Mat}_{(n+1) \times (n+1)}\) of affine spaces
over \(\kk\), given by
\[
[\beta] \mapsto \Gram(\beta;e_0,\ldots,e_n)
\]
taking a \(q\)-bic form to its Gram matrix.

\subsection{Rank stratification}\label{forms-classification-rank-stratification}
For each \(0 \leq r \leq n+1\), let
\[
\qbics(V)_{\leq r} \coloneqq \Set{[\beta] \in \qbics(V) | \rank(\beta) \leq r}
\]
be the closed subscheme consisting of \(q\)-bic forms whose rank, in the
sense of \parref{forms-rank-corank}, is at most \(r\); equivalently, this
is the locus of \(q\)-bic forms with corank at least \(n+1-r\). This gives
a filtration by closed subschemes
\[
\qbics(V) =
\qbics(V)_{\leq n+1} \supseteq
\qbics(V)_{\leq n} \supseteq
\cdots \supseteq
\qbics(V)_{\leq 1} \supseteq
\qbics(V)_{\leq 0}.
\]
The locally closed subschemes
\[
\qbics(V)_r \coloneqq \qbics(V)_{\leq r} \setminus \qbics(V)_{\leq r-1}
\]
together give the \emph{rank stratification} of \(\qbics(V)\).

Since ranks are compatible with taking Gram matrices, \(\qbics(V)_{\leq r}\) is
isomorphic to the locus of \((n+1)\)-by-\((n+1)\) matrices of rank at most
\(r\). Thus \(\qbics(V)_{\leq r}\) is irreducible and has codimension
\((n+1-r)^2\); see \cite[Chapter 14]{Fulton} for example. For instance,
\(\qbics(V)_{\leq n}\) is a hypersurface and \(\qbics(V)_{\leq 0}\) is a point.

\subsection{Type stratification}\label{forms-classification-type-stratification}
A much finer stratification of \(\qbics(V)\) is afforded by the
classification theorem \parref{forms-classification-theorem}. Namely, each
\(q\)-bic form is isomorphic to a standard form of type \(\lambda\), and the
sets
\[
\qbics(V)_\lambda \coloneqq
\Set{[\beta] \in \qbics(V) | \mathrm{type}(\beta) = \lambda}
\]
may be construed as locally closed subschemes of \(\qbics(V)\), which together
give a stratification \(\qbics(V)\)---the \emph{type stratification}---refining
the rank stratification. For instance,
\(\qbics(V)_{n+1} = \qbics(V)_{\mathbf{1}^{\oplus n+1}}\) and
\[
\qbics(V)_n =
\coprod\nolimits_{k = 1}^{n+1} \qbics(V)_{\mathbf{N}_k \oplus \mathbf{1}^{\oplus n+1-k}}.
\]

To describe the basic properties of these strata, let \(\AutSch(V,\lambda)\)
denote the automorphism group scheme of any \(q\)-bic form of type \(\lambda\).

\begin{Lemma}\label{forms-aut-strata-dimension}
\(\qbics(V)_\lambda\) is irreducible and of codimension \(\dim\AutSch(V,\lambda)\).
\end{Lemma}

\begin{proof}
Fix a \(q\)-bic form \(\beta\) on \(V\) of type \(\lambda\).
By \parref{forms-classification-theorem}, the map
\(\GL(V) \to \qbics(V)\) given by \(g \mapsto g \cdot \beta\) is a surjection
onto \(\qbics(V)_\lambda\). Since the fibres are isomorphic to
\(\AutSch(V,\beta)\), the result follows.
\end{proof}

\begin{Corollary}\label{forms-aut-general-corank-b}
A general \(q\)-bic form of corank \(a \leq \frac{n+1}{2}\) is
of type \(\mathbf{N}_2^{\oplus a} \oplus \mathbf{1}^{\oplus n+1-2a}\).
\end{Corollary}

\begin{proof}
Comparing \parref{forms-classification-rank-stratification},
\parref{forms-aut-strata-dimension}, and \parref{forms-aut-1^a+N2^b.computation}
shows that both \(\qbics(V)_{n+1-a}\) and
\(\qbics(V)_{\mathbf{N}_2^{\oplus a} \oplus \mathbf{1}^{\oplus n+1-a}}\) are
irreducible of codimension \(a^2\) in \(\qbics(V)\). Since the type stratum
is contained in the rank stratum, the result follows.
\end{proof}

%% file: hypersurfaces.tex
\chapter{\texorpdfstring{\(q\)}{q}-bic Hypersurfaces}\label{chapter-hypersurfaces}

Given a \(q\)-bic form \((V,\beta)\) over a field \(\kk\), the space
\[
X \coloneqq \Set{[v] \in \PP V | \beta(v^{(q)},v) = 0}
\]
parameterizing isotropic vectors is a hypersurface of degree \(q+1\) in the
projective space \(\PP V\): this is the \emph{\(q\)-bic hypersurface}
associated with \(\beta\). Defined in this way, \(q\)-bic hypersurfaces are
akin to quadrics; the aim of this Chapter is to substantiate this analogy by
systematically relating global geometric properties of \(q\)-bic hypersurfaces
with algebraic properties of the underlying \(q\)-bic form.

The basic definitions are given in \parref{section-hypersurfaces-setup}.
Differential invariants of \(X\) are expressed in terms of \(\beta\) in
\parref{section-hypersurfaces-differential}; for instance, \(X\) is smooth
smooth if and only if \(\beta\) is nonsingular, and the nonsmooth locus of
\(X\) is canonically the \(q\)-fold linear space given by the kernel of
\(\beta^\vee\): see \parref{hypersurfaces-smooth-and-nondegeneracy} and
\parref{hypersurfaces-nonsmooth-locus}. Automorphisms of \(X\) are related
with those of \(\beta\) in \parref{section-hypersurfaces-automorphisms}.
Section \parref{section-hypersurfaces-cones} contains a basic study of
cones over \(q\)-bic hypersurfaces: \parref{hypersurfaces-cones} shows that
\(X\) is a cone if and only if \(\beta\) has a radical, and that the vertex
of \(X\) is given by \(\rad(\beta)\). More interestingly, cone points
of \(q\)-bic hypersurfaces, defined in
\parref{hypersurfaces-cone-points-definition}, may be characterized and the set
of which carries a canonical scheme structure: see
\parref{hypersurfaces-cone-points-classify},
\parref{hypersurfaces-cone-points-smooth}, and
\parref{hypersurfaces-cone-points-equations-general}. Section
\parref{hypersurfaces-unirational} explains two unirationality constructions
for smooth \(q\)-bic hypersurfaces: see
see \parref{hypersurfaces-unirationality-shioda-parameterization} and
\parref{hypersurfaces-unirational-tangent-morphism}.
Some basic cohomological properties of \(q\)-bic hypersurfaces are collected in
\parref{section-hypersurfaces-cohomological}.

Section \parref{section-hypersurfaces-linear-spaces} initiates the study of
the Fano schemes of linear spaces associated with \(q\)-bic hypersurfaces.
Notably, the Fano schemes of a smooth \(q\)-bic are smooth, irreducible, and
has dimension independent of \(q\): see \parref{hypersurfaces-smooth-fano}. The
tautological incidence correspondence is studied in more detail in
\parref{section-hypersurfaces-fano-correspondences}. Finally,
\parref{section-hypersurfaces-hermitian} constructs a canonical filtration of
smooth \(q\)-bic hypersurfaces by closed subschemes induced by the Hermitian
self-map of \parref{forms-hermitian-endomorphism}.

Throughout this Chapter, \(\kk\) is a field containing \(\mathbf{F}_{q^2}\),
\(\Fr \colon \kk \to \kk\) denotes the \(q\)-power Frobenius homomorphism,
and \(V\) is a \(\kk\)-vector space of dimension \(n+1\).

\section{Setup and basic properties}\label{section-hypersurfaces-setup}

\subsection{\(q\)-bic equations and hypersurfaces}\label{hypersurfaces-qbic-equation}
Let \((V,\beta)\) be a \(q\)-bic form over a field \(\kk\), as defined in
\parref{forms-definition}. Unless otherwise stated, the form \(\beta\) will be
assumed to be nonzero. This induces a nonzero section \(f_\beta\) of
\(\sO_{\PP V}(q+1)\) via
\[
f_\beta \coloneqq \beta(\mathrm{eu}^{(q)},\mathrm{eu}) \colon
\sO_{\PP V}(-q-1) \to
\Fr^*(V)_{\PP V} \otimes V_{\PP V} \xrightarrow{\beta} \sO_{\PP V}
\]
where \(\mathrm{eu} \colon \sO_{\PP V}(-1) \to V_{\PP V}\) is the tautological
Euler section. In other words, \(f_\beta\) is the degree \(q+1\) polynomial
obtained by pairing the linear coordinates of \(\PP V\) with their \(q\)-powers
according to \(\beta\). This section is called the \emph{\(q\)-bic equation}
associated with \((V,\beta)\). The degree \(q+1\) hypersurface of \(\PP V\)
given by
\[
X = X_\beta \coloneqq \mathrm{V}(f_\beta) \subset \PP V
\]
is the \emph{\(q\)-bic hypersurface} associated with the \(q\)-bic form
\((V,\beta)\).

\subsection{}\label{hypersurfaces-qbic-equations-coordinates}
Explicitly, let \(V = \langle e_0,\ldots,e_n\rangle\) be a basis and let
\(\mathbf{x}^\vee \coloneqq (x_0:\cdots:x_n)\) be the corresponding projective
coordinates on \(\PP V = \PP^n\). For each \(0 \leq i, j \leq n\), let
\(a_{ij} \coloneqq \beta(e_i^{(q)},e_j)\) be the \((i,j)\)-entry of the
Gram matrix of \(\beta\) with respect to the chosen basis, see
\parref{forms-gram-matrix}. Then the \(q\)-bic equation \(f_\beta\) is the
polynomial
\[
f_\beta(x_0,\ldots,x_n)
= \mathbf{x}^{\vee,(q)} \cdot \Gram(\beta;e_0,\ldots,e_n) \cdot \mathbf{x}
= \sum\nolimits_{i,j = 0}^n a_{ij} x_i^q x_j
\]
and the associated \(q\)-bic hypersurface \(X\) is its vanishing locus in
\(\PP^n\).

\medskip

The simple yet fundamental observation is that \(q\)-bic hypersurfaces are
moduli spaces of isotropic vectors for the \(q\)-bic form \(\beta\). Thus they
are akin to quadric hypersurfaces, and their geometry may be accessed via
algebraic methods relating to the bilinear form \(\beta\).

\begin{Proposition}\label{hypersurfaces-moduli-of-isotropic-vectors}
The \(q\)-bic hypersurface \(X\) associated with a \(q\)-bic form \((V,\beta)\)
represents the functor
\(\mathrm{Sch}_{\mathbf{k}}^{\mathrm{opp}} \to \mathrm{Set}\) given by
\[
T \mapsto \Set{\mathcal{V}' \subset V_T |
\mathcal{V}'\;\text{rank \(1\) subbundle isotropic for \(\beta\)}
}.
\]
\end{Proposition}

\begin{proof}
This follows directly from the moduli description of projective space, see
\citeSP{01NE}, together with the construction of \(X\) in
\parref{hypersurfaces-qbic-equation}.
\end{proof}

A simple application of this observation is to show that a linear section of a
\(q\)-bic hypersurface is another \(q\)-bic hypersurface:

\begin{Lemma}\label{hypersurface-hyperplane-section}
Let \(X\) be the \(q\)-bic hypersurface associated with a \(q\)-bic form
\((V,\beta)\). Let \(U \subseteq V\) be any linear subspace and let
\[
\beta_U \colon \Fr^*(U) \otimes U \subset
\Fr^*(V) \otimes V \xrightarrow{\beta} \kk
\]
be the---possibly zero---\(q\)-bic form on \(U\) obtained by restricting
\(\beta\). Then \(X \cap \PP U\) is the \(q\)-bic hypersurface associated with
the \(q\)-bic form \((U,\beta_U)\).
\end{Lemma}

\begin{proof}
It follows from \parref{hypersurfaces-moduli-of-isotropic-vectors} that both
\(X \cap \PP U\) and \(X_{\beta_U}\) represent the functor of lines in \(U\)
isotropic for \(\beta_U\). Alternatively, this can be seen directly in coordinates
by examining \parref{hypersurfaces-qbic-equations-coordinates}.
\end{proof}

\subsection{Terminology}\label{hypersurfaces-terminology}
It is convenient to conflate certain notions regarding \(q\)-bic forms with
their counterparts in the setting of \(q\)-bic hypersurfaces. Let
\(X\) be the \(q\)-bic hypersurface associated with a \(q\)-bic form
\((V,\beta)\). If the base change of \((V,\beta)\) to an algebraic closure
of \(\kk\) is isomorphic to a standard form of type \(\lambda\), see
\parref{forms-standard} and \parref{forms-classification-theorem}, then
\(X\) is said to be \emph{of type \(\lambda\)}. The \emph{(co)rank} of \(X\)
is the (co)rank of \((V,\beta)\) in the sense of \parref{forms-rank-corank}.
Projective subspaces \(\PP U \subseteq X\) for which \(U\) is an isotropic
Hermitian subspace of \((V,\beta)\) in the sense of \parref{forms-hermitian}
are called \emph{Hermitian subspaces} of \(X\); in the case \(\PP U\) is a
point, this is also called a \emph{Hermitian point} of \(X\).

\section{Differential invariants}\label{section-hypersurfaces-differential}
Differential invariants of a \(q\)-bic hypersurface, such as the (co)normal,
(co)tangent, and embedded tangent sheaves, admit simple descriptions in terms
of an underlying \(q\)-bic form. As such, smoothness of a \(q\)-bic
hypersurface and even the schematic nonsmooth locus may be described in terms
of the underlying form \((V,\beta)\): see
\parref{hypersurfaces-smooth-and-nondegeneracy} and
\parref{hypersurfaces-nonsmooth-locus}. Most interesting are the properties of
the embedded tangent sheaf, see \parref{hypersurfaces-embedded-tangent-sheaf}.
This carries a \(q\)-bic form and its properties encapsulate many of the
idiosyncrasies found in the projective geometry of \(q\)-bic hypersurfaces: see
\parref{hypersurfaces-frobenius-euler}, \parref{hypersurfaces-sigma-section},
and \parref{hypersurfaces-tangent-kernel}.

\subsection{Conormal map}\label{hypersurfaces-conormal}
For any complete intersection \(Y \subset \PP V\) whose equations have degrees
coprime to the characteristic, its conormal sequence is a short exact sequence
\[
0 \to
\mathcal{C}_{Y/\PP V} \xrightarrow{\delta}
\Omega^1_{\PP V}\rvert_Y \to
\Omega_Y^1 \to
0.
\]
The conormal map \(\delta\) maps local equations of \(Y\) to their differential.
When \(Y\) is a hypersurface of degree \(d\), a choice of equation \(f\)
determines an isomorphism \(f \colon \sO_Y(-d) \to \mathcal{C}_{Y/\PP V}\)
sending a local generator of \(\sO_Y(-d)\) to \(f\).

The conormal map of a \(q\)-bic hypersurface admits a neat description in terms
of the \(q\)-bic form \(\beta\):

\begin{Lemma}\label{hypersurfaces-conormal-compute}
There exists a commutative diagram
\[
\begin{tikzcd}
\sO_X(-q-1) \ar[rr,"\mathrm{eu}^{(q)}"'] \dar["f_\beta"'] && \Fr^*(V)_X(-1) \dar["\beta^\vee"] \\
\mathcal{C}_{X/\PP V} \rar["\delta"] & \Omega_{\PP V}^1\rvert_X \rar[hook] & V_X^\vee(-1)
\end{tikzcd}
\]
\end{Lemma}

\begin{proof}
Observe that
\(\beta^\vee \circ \mathrm{eu}^{(q)} \colon \sO_X(-q-1) \to V_X^\vee(-1)\)
factors through \(\Omega_{\PP V}^1\rvert_X\): this sheaf is the kernel of
\(\mathrm{eu}^\vee \colon V_X^\vee(-1) \to \sO_X\) and
\[ \mathrm{eu}^\vee \circ \beta^\vee \circ \mathrm{eu}^{(q)} = f_\beta \]
vanishes on \(X\). Since \(q+1 = 1\) in \(\kk\), together
with Euler's formula, this shows that the diagram in question commutes.
\end{proof}

Smoothness of \(q\)-bic hypersurfaces is characterized in terms of
nonsingularity of the underlying \(q\)-bic form:

\begin{Lemma}\label{hypersurfaces-smooth-and-nondegeneracy}
Let \(X\) be the \(q\)-bic hypersurface associated with a \(q\)-bic form
\((V,\beta)\). Then the following are equivalent:
\begin{enumerate}
\item\label{hypersurfaces-smooth-and-nondegeneracy.smooth}
\(X\) is smooth;
\item\label{hypersurfaces-smooth-and-nondegeneracy.nondegenerate}
\(\beta\) is nonsingular; and
\item\label{hypersurfaces-smooth-and-nondegeneracy.gram}
\(\Gram(\beta;v_0,\ldots,v_n)\) is invertible for any basis
\(V = \langle v_0,\ldots,v_n \rangle\).
\end{enumerate}
\end{Lemma}

\begin{proof}
That
\ref{hypersurfaces-smooth-and-nondegeneracy.nondegenerate} \(\Leftrightarrow\)
\ref{hypersurfaces-smooth-and-nondegeneracy.gram} are equivalent follows directly
from \parref{forms-gram-matrix-nondegenerate}. To see that these conditions are
equivalent to \ref{hypersurfaces-smooth-and-nondegeneracy.smooth}, recall from
\cite[Theorem II.8.17]{Hartshorne:AG} that \(X\) is smooth if and only if the
conormal sequence
\[
0 \to
\mathcal{C}_{X/\PP V} \xrightarrow{\delta}
\Omega^1_{\PP V}\rvert_X \to
\Omega^1_X \to
0
\]
is exact and \(\Omega^1_X\) is locally free of rank \(n-1\). By
\parref{hypersurfaces-conormal-compute}, the conormal map \(\delta\)
coincides with \(\beta^\vee \circ \mathrm{eu}^{(q)}\). So if \(\beta\) is
nonsingular, then \(\beta^\vee \circ \mathrm{eu}^{(q)}\) is injective on each
fibre and \(X\) is smooth. Conversely, if \(\beta\) is singular, then up
to passing to a purely inseparable extension of \(\kk\), there is a
\(1\)-dimensional subspace \(L \subset V\) such that \(\Fr^*(L) \subset
\Fr^*(V)^\perp\). Then \(L\) is isotropic and
\(\beta^\vee \circ \mathrm{eu}^{(q)}\) is zero on the fibre over \(x \coloneqq \PP L \in X\),
showing that \(\Omega_X^1\) is not locally free at
\(x\) and so \(X\) is not smooth.
\end{proof}

Together with the classification of nonsingular \(q\)-bic forms, this implies
that over a separably closed field there is only one smooth \(q\)-bic
hypersurface of a given dimension up to projective equivalence. This statement
has been rediscovered many times: see, for instance, \cite[Proposition
3.7]{Pardini:Curves}, \cite[Corollary 9.11]{Hefez:Thesis}, and
\cite[Th\'eor\`eme]{Beauville:Moduli}.

\begin{Corollary}\label{hypersurfaces-smooth-projectively-equivalent}
Over a separably closed field, all smooth \(q\)-bic hypersurfaces in \(\PP V\)
are projectively equivalent.
\end{Corollary}

\begin{proof}
Let \(X\) and \(X'\) be smooth \(q\)-bic hypersurfaces in \(\PP V\), and
let \(\beta\) and \(\beta'\) be defining \(q\)-bic forms. By
\parref{hypersurfaces-smooth-and-nondegeneracy}, these forms are nonsingular.
By \parref{hypersurfaces-moduli-of-isotropic-vectors}, it suffices to show that
there is an isomorphism \((V,\beta) \cong (V,\beta')\) of \(q\)-bic forms.
By \parref{forms-hermitian-diagonal}, there exists diagonalizing bases
\(\langle v_0,\ldots,v_n \rangle\) for \(\beta\), and
\(\langle v_0',\ldots,v_n' \rangle\) for \(\beta'\). The isomorphism
\(V \to V\) determined by \(v_i \mapsto v_i'\) then yields an isomorphism
of the \(q\)-bic forms.
\end{proof}

Another way of stating this is that, over a separably closed field, the only
smooth \(q\)-bic hypersurface is the Fermat hypersurface of degree \(q+1\):

\begin{Corollary}\label{hypersurfaces-smooth-fermat-model}
Let \(X\) be a smooth \(q\)-bic hypersurface over a separably
closed field. Then there exists a choice of coordinates \((x_0:\cdots:x_n)\) of
\(\PP V = \PP^n\) such that
\[
\pushQED{\qed}
X =
\mathrm{V}(x_0^{q+1} + \cdots + x_n^{q+1}) \subset \PP^n.
\qedhere
\popQED
\]
\end{Corollary}

\subsection{Nonsmooth locus}
The proof of \parref{hypersurfaces-smooth-and-nondegeneracy} indicates that the
singular locus may become visible only after a purely inseparable extension of
\(\kk\). This means that the scheme-theoretic singular locus is defined by
\(q\)-power equations. To explain, recall that the \emph{nonsmooth locus} of a
morphism \(f \colon Y \to S\) which is flat, locally of finite presentation,
and such that the nonempty fibres of \(f\) are of dimension \(d\), is the
closed subscheme of \(Y\) defined by the \(d\)-th Fitting ideal
\[ \Sing(f) \coloneqq \mathrm{V}(\Fitt_d(\Omega_{Y/S})). \]
Importantly, this scheme is supported on the points where \(f\) is not smooth,
and its formation commutes with arbitrary base change; see
\cite[\hrefSP{0C3K} and \hrefSP{0C3I}]{stacks-project}. When \(S = \Spec(\kk)\),
write \(\Sing(Y) \coloneqq \Sing(f)\).

The following statement identifies the nonsmooth locus of a \(q\)-bic
hypersurface. The formula below means, in the case that
\(V^\perp = \Fr^*(U)\) for some linear subspace \(U \subseteq V\),
\(\Sing(X)\) is the closed subscheme cut out by \(q\)-powers of the
linear forms vanishing on \(U\); in particular, \(\Sing(X)\) would be supported
on \(\PP U\).

\begin{Lemma}\label{hypersurfaces-nonsmooth-locus}
For a \(q\)-bic hypersurface \(X\) associated with a \(q\)-bic form
\((V,\beta)\),
\[
\Sing(X) \coloneqq
\Fitt_{n-1}(\Omega_X^1) =
\mathrm{V}(
\sO_X(-q) \xrightarrow{\mathrm{eu}^{(q)}}
\Fr^*(V)_X \twoheadrightarrow
(\Fr^*(V)/V^\perp)_X).
\]
\end{Lemma}

\begin{proof}
The conormal sequence in \parref{hypersurfaces-conormal} yields a
presentation of \(\Omega_X^1\). Since \(\Omega_{\PP V}^1\rvert_X\) is of
rank \(n\), this gives the first equality in
\[
\Fitt_{n-1}(\Omega_X^1)
= \mathrm{V}(\delta \colon \mathcal{C}_{X/\PP V} \to \Omega^1_{\PP V}\rvert_X)
= \mathrm{V}(\beta^\vee \circ \mathrm{eu}^{(q)} \colon \sO_X(-q) \to \Fr^*(V)_X \to V_X^\vee).
\]
The second equality follows from \parref{hypersurfaces-conormal-compute}. This
implies the result upon noting that, by \parref{forms-orthogonal-sequence},
\(\beta^\vee\) factors through the surjection \(\Fr^*(V) \to
\Fr^*(V)/V^\perp\).
\end{proof}

\subsection{Embedded tangent sheaf}\label{hypersurfaces-embedded-tangent-sheaf}
Given a local complete intersection closed subscheme
\(Y \subseteq \PP V\), the \emph{embedded tangent sheaf} of \(Y\) in \(\PP V\)
is the sheaf \(\mathcal{T}_Y^{\mathrm{e}}\) defined via pullback in the
commutative diagram
\[
\begin{tikzcd}
0 \rar &
\mathcal{T}_Y(-1) \rar &
\mathcal{T}_{\PP V}(-1)\rvert_Y \rar["\delta^\vee"] &
\mathcal{N}_{Y/\PP V}(-1) \\
0 \rar &
\mathcal{T}_Y^{\mathrm{e}} \uar[two heads] \rar &
V_Y \rar \uar[two heads] &
\mathcal{N}_{Y/\PP V}(-1) \uar[equal] \\
&
\mathcal{O}_Y(-1) \uar[hook] \rar[equal] & \mathcal{O}_Y(-1) \uar[hook, "\mathrm{eu}"]
\end{tikzcd}
\]
obtained by juxtaposing the tangent sequence and the restriction to \(Y\) of
the Euler sequence for \(\PP V\). Then for every smooth point \(y \in Y\),
the fibre \(\mathcal{T}_{Y,y}^{\mathrm{e}} \subseteq V\) is the linear subspace
underlying the embedded tangent space \(\mathbf{T}_{Y,y}\) of \(Y\) at \(y\).

This property of the embedded tangent sheaf together with the computation
of the conormal map from \parref{hypersurfaces-conormal-compute}
yields a neat description of the linear subspaces underlying the embedded
tangent spaces of a \(q\)-bic hypersurface:

\begin{Lemma}\label{hypersurfaces-tangent-space-as-kernel}
Let \(x = \PP L\) be a smooth \(\kk\)-point of a \(q\)-bic hypersurface \(X\). Then
\[
\mathbf{T}_{X,x} = \PP\Fr^*(L)^\perp.
\]
\end{Lemma}

\begin{proof}
The identification of the conormal map in \parref{hypersurfaces-conormal-compute}
implies
\[
\mathcal{T}_X^{\mathrm{e}}
= \ker(
    \delta^\vee \colon
      V_X \to
      \mathcal{N}_{X/\PP V}(-1)
  )
= \ker(
    \mathrm{eu}^{(q),\vee} \circ \beta \colon
      V_X \to
      \Fr^*(V)_X \to \sO_X(q)
  ).
\]
Thus the fibre of \(\mathcal{T}_X^{\mathrm{e}}\) at a \(x = \PP L\) is, on the
one hand, the linear subspace of \(V\) underlying \(\mathbf{T}_{X,x}\) by
\parref{hypersurfaces-embedded-tangent-sheaf}, and on the other hand, the
subspace \(\Fr^*(L)^\perp\).
\end{proof}

The following identifies the embedded tangent bundle of a smooth \(q\)-bic
hypersurface. This was first observed in \cite[Equation (3)]{Shen:Fermat}.

\begin{Lemma}\label{hypersurfaces-frobenius-euler}
The morphism \(\beta \colon V_X \to \Fr^*(V)_X^\vee\) induces a morphism
\[
\mathcal{T}_X^{\mathrm{e}} \to
\Fr^*(\Omega_{\PP V}^1(1))\rvert_X.
\]
If \(\beta\) is nonsingular, then this is an isomorphism.
\end{Lemma}

\begin{proof}
The computation of \parref{hypersurfaces-conormal-compute} yields a commutative
diagram of solid arrows with exact rows, given by
\[
\begin{tikzcd}
0 \rar
& \mathcal{T}_X^{\mathrm{e}} \rar \dar[dashed]
& V_X \dar["\beta"] \rar
& \mathcal{N}_{X/\PP V}(-1) \dar["f_\beta"] \\
0 \rar
& \Fr^*(\Omega_{\PP V}^1(1))\rvert_X \rar
& \Fr^*(V)_X^\vee \rar["\mathrm{eu}^{(q),\vee}"]
& \sO_X(q) \rar
& 0\punct{.}
\end{tikzcd}
\]
This implies that \(\beta\) restricts to a morphism
\(\mathcal{T}_X^{\mathrm{e}} \to \Fr^*(\Omega_{\PP V}^1(1))\rvert_X\) filling
the dashed arrow in the diagram.

If \(\beta\) is nonsingular, then \(\beta \colon V_X \to \Fr^*(V)_X^\vee\)
is an isomorphism. Moreover, \(X\) is smooth in this case by
\parref{hypersurfaces-smooth-and-nondegeneracy} so
\(V_X \to \mathcal{N}_{X/\PP V}(-1)\) is surjective.
That the map
\(\mathcal{T}_X^{\mathrm{e}} \to \Fr^*(\Omega^1_{\PP V}(1))\rvert_X\)
is an isomorphism now follows from the Five Lemma.
\end{proof}

\subsection{}\label{hypersurfaces-sigma-section-sheaf}
Composing the Euler section with the map of
\parref{hypersurfaces-frobenius-euler} yields a morphism
\[
\beta \circ \mathrm{eu} \colon
\sO_X(-1) \hookrightarrow
\mathcal{T}_X^{\mathrm{e}}
\to \Fr^*(\Omega_{\PP V}^1(1))\rvert_X
\]
which is nonzero whenever \(\beta\) is, inducing a rational map
\(X \dashrightarrow \PP(\Omega^1_{\PP V}(1)\rvert_X)\). In general, this
is not linear over \(\kk\), and may be linearized as follows: form the
diagram
\[
\begin{tikzcd}
X \ar[r,"\Fr_{X/\kk}"] \ar[rr,bend left=42, "\Fr"] \ar[dr]
& X^{(q)} \rar["\pr_X"] \dar & X \dar\\
& \Spec(\kk) \rar["\Fr"] & \Spec(\kk)
\end{tikzcd}
\]
in which \(X^{(q)} \coloneqq X \otimes_{\kk,\Fr} \kk\) is the \(q\)-power
Frobenius twist of \(X\)---so that \(X'\) is the \(q\)-bic hypersurface
associated with \((\Fr^*(V),\Fr^*(\beta))\), see \parref{forms-fr-twist}---and
\(\Fr_{X/\kk} \colon X \to X^{(q)}\) is the \(\kk\)-linear \(q\)-power
Frobenius morphism. Then
\[
\Fr^*(\Omega^1_{\PP V}(1))\rvert_X \cong
\Fr^*_{X/\kk}(\pr_X^*\Omega^1_{\PP V}(1)\rvert_X) \cong
\Fr^*_{X/\kk}(\Omega^1_{\PP V^{(q)}}(1)\rvert_{X^{(q)}})
\]
where \(\PP V^{(q)} \coloneqq \PP\Fr^*(V) = \PP V \otimes_{\kk,\Fr} \kk\).
Then \(\beta \circ \mathrm{eu}\) induces a map of \(\sO_X\)-modules
\[
\sigma_X \colon
\sO_X(-1) \to
\Fr_{X/\kk}^*(\Omega^1_{\PP V^{(q)}}(1)\rvert_{X^{(q)}})
\]
which, in turn, induces a rational map
\(X \dashrightarrow \PP(\Omega^1_{\PP V^{(q)}}(1)\rvert_{X^{(q)}})\).
The Euler sequence shows that the \(\kk\)-points of the
projective bundle of \(\Omega^1_{\PP V^{(q)}}(1)\) are
\[
\PP(\Omega^1_{\PP V^{(q)}}(1)) =
\Set{(L \subset W \subset \Fr^*(V)) | \dim_\kk L = 1\;\text{and}\;\dim_\kk W = n}
\]
flags consisting of a line and a hyperplane in \(\Fr^*(V)\).

\begin{Lemma}\label{hypersurfaces-sigma-section}
There exists a commutative diagram of rational maps
\[
\begin{tikzcd}
& \PP(\Omega_{\PP V^{(q)}}^1(1)\rvert_{X^{(q)}}) \dar["\tau"] \\
X \rar["\Fr_{X/\kk}"] \ar[ur,dashed,"\sigma_X"] & X^{(q)}
\end{tikzcd}
\]
of schemes over \(\kk\) such that \(\sigma_X\) is defined away from
\(\PP\Fr^*(V)^\perp\), and
\[
\sigma_X(x) = (\Fr^*(L) \subset L^\perp \subset \Fr^*(V))
\quad\text{for any}\; x = \PP L \in (X \setminus \PP\Fr^*(V)^\perp)(\kk).
\]
If \(X\) is furthermore smooth, then \(\sigma_X\) is a morphism and
\(\sigma_X^*\sO_\tau(-1) = \sO_X(-1)\).
\end{Lemma}

\begin{proof}
That the diagram exists follows from the comments of
\parref{hypersurfaces-sigma-section-sheaf}. It is indeterminate at the points
\(x = \PP L\) in which
\[
\beta \circ \mathrm{eu} \colon
\sO_X(-1) \hookrightarrow
\mathcal{T}_X^{\mathrm{e}} \to
\Fr^*(\Omega^1_{\PP V}(1)\rvert_X)
\]
is not injective; this occurs precisely when
\(L \subseteq \ker(\beta \colon V \to \Fr^*(V)^\vee) = \Fr^*(V)^\perp\).
Away from such points, \(\sigma_X(x)\) determines the hyperplane
in \(\Fr^*(V)\) with equation \(\beta(-,L)\). The final statement now follows
from \parref{hypersurfaces-smooth-and-nondegeneracy}.
\end{proof}

\subsection{Tangent form}\label{hypersurfaces-tangent-kernel}
Since the embedded tangent sheaf \(\mathcal{T}_X^{\mathrm{e}}\) is a subsheaf
of \(V_X\), the \(q\)-bic form \(\beta\) defining \(X\) restricts to a form
\[
\beta_{\mathrm{tan}} \colon
\Fr^*(\mathcal{T}_X^{\mathrm{e}}) \otimes \mathcal{T}_X^{\mathrm{e}} \subset
\Fr^*(V)_X \otimes V_X \xrightarrow{\beta}
\sO_X,
\]
called the \emph{tangent \(q\)-bic form} of \(X\).
On the fibre at a smooth \(\kk\)-point \(x = \PP L\), this is the restriction
of \(\beta\) to the subspace \(\Fr^*(L)^\perp\) underlying the tangent
hyperplane \(\mathbf{T}_{X,x}\).

When \(X\) is smooth, the basic properties of \(\beta_{\mathrm{tan}}\) are
as follows:

\begin{Proposition}\label{hypersurfaces-tangent-form-properties}
Let \(X\) be the \(q\)-bic hypersurface associated with a \(q\)-bic form
\((V,\beta)\). If \(X\) is smooth, then its tangent form
\[
\beta_{\mathrm{tan}} \colon
\Fr^*(\mathcal{T}_X^{\mathrm{e}}) \otimes \mathcal{T}_X^{\mathrm{e}} \to \sO_X
\]
is everywhere of corank \(1\) and it induces an exact sequence
\[
0 \longrightarrow
\Fr^*(\mathcal{N}_{X/\PP V}(-1))^\vee \stackrel{\phi_X}{\longrightarrow}
\mathcal{T}_X^{\mathrm{e}} \stackrel{\beta_{\mathrm{tan}}}{\longrightarrow}
\Fr^*(\mathcal{T}_X^{\mathrm{e}})^\vee \stackrel{\mathrm{eu}^{(q),\vee}}{\longrightarrow}
\Fr^*(\sO_X(-1))^\vee \longrightarrow
0
\]
where the map \(\phi_X\) is induced by \(\beta^{-1} \circ \delta^{(q)}\).
\end{Proposition}

\begin{proof}
Note that \(\beta_{\mathrm{tan}}\) has corank at most \(1\) by
\parref{forms-rank-linear-subspace}; it has corank at least \(1\) by
\parref{hypersurfaces-smooth-and-nondegeneracy}, since the restriction of the
tangent form at a point \(x = \PP L \in X\) defines, by
\parref{hypersurface-hyperplane-section}, the \(q\)-bic hypersurface
\(X \cap \mathbf{T}_{X,x}\) and this is singular at \(x\). In fact, this
implies by \parref{hypersurfaces-nonsmooth-locus} that
\[
\ker(
\beta_{\mathrm{tan}}^\vee \colon
\Fr^*(\mathcal{T}_X^{\mathrm{e}}) \to
\mathcal{T}_X^{\mathrm{e},\vee})
= \mathrm{image}(\mathrm{eu}^{(q)} \colon
\Fr^*(\sO_X(-1)) \to
\Fr^*(\mathcal{T}_X^{\mathrm{e}})\big),
\]
thereby identifying the cokernel of \(\beta_{\mathrm{tan}}\) in the above
exact sequence. To identify the kernel, consider the commutative diagram
\[
\begin{tikzcd}
0 \rar
& \mathcal{T}_X^{\mathrm{e}} \rar \dar["\beta_{\mathrm{tan}}"]
& V_X \rar["\delta^\vee"']
& \mathcal{N}_{X/\PP V}(-1) \rar
& 0 \\
0
& \lar \Fr^*(\mathcal{T}_X^{\mathrm{e}})^\vee
& \lar \Fr^*(V_X)^\vee \uar["\beta^{-1}"']
& \lar["\delta^{(q)}"'] \Fr^*(\mathcal{N}_{X/\PP V}(-1))^\vee \uar
& \lar 0
\end{tikzcd}
\]
in which the rows are the exact sequences defining the embedded tangent
bundle from \parref{hypersurfaces-embedded-tangent-sheaf}. The
map \(\Fr^*(\mathcal{N}_{X/\PP V}(-1))^\vee \to \mathcal{N}_{X/\PP V}\)
since, by \parref{hypersurfaces-conormal-compute}, it is
\[
\delta^\vee \circ \beta^{-1} \circ \delta^{(q)} =
(\mathrm{eu}^{(q),\vee} \circ \beta) \circ \beta^{-1} \circ (\beta^\vee \circ \mathrm{eu}^{(q)})^{(q)}
= (\mathrm{eu}^{(q),\vee} \circ \beta \circ \mathrm{eu})^{(q),\vee}
\]
the \(q\)-power of the equation of \(X\). Thus there is a morphism
\[
\phi_X \coloneqq \beta^{-1} \circ \delta^{(q)} \colon
\Fr^*(\mathcal{N}_{X/\PP V}(-1))^\vee \to \mathcal{T}_X^{\mathrm{e}}
\]
and the diagram now shows that this is an isomorphism onto the kernel of
\(\beta_{\mathrm{tan}}\).
\end{proof}

\subsection{Residual point of tangency}\label{hypersurfaces-tangent-form-properties-residual-tangent}
Let \(x\) be a closed point of a smooth \(q\)-bic hypersurface \(X\). The
computation of \parref{hypersurfaces-tangent-form-properties} means that
the intersection \(X \cap \mathbf{T}_{X,x}\) of \(X\) with its embedded tangent
space at \(x\) is a \(q\)-bic hypersurface of corank \(1\) which is singular
at \(x\). Besides the singular point, each \(q\)-bic of corank \(1\) has
another distinguished point corresponding to
\(\Fr^*(\mathbf{T}_{X,x})^\perp\);
this is given by the isotropic line subbundle
\(\Fr^*(\mathcal{C}_{X/\PP V}(1))\)
above. This point is called the \emph{residual point of tangency} to \(x\)
and is denoted by \(\phi_X(x)\). See \parref{hypersurfaces-filtration-embedded-tangent}
for more.

\section{Automorphisms}\label{section-hypersurfaces-automorphisms}
Let \(X\) be the \(q\)-bic hypersurface associated with a \(q\)-bic form
\((V,\beta)\). Then \(X\) has an automorphism group scheme, denoted
\(\AutSch(X)\), which is a group scheme over \(\kk\), locally of finite type:
see \cite[Theorem 3.7]{MO:Automorphisms}. This brief Section discusses a few
properties of this scheme.

\subsection{Linear automorphisms}\label{hypersurfaces-automorphisms-linear}
Let \(\AutSch(V,\beta)\) be the automorphism group scheme of the \(q\)-bic
form \((V,\beta)\), as introduced in \parref{forms-aut-schemes}. It acts
linearly on \(\PP V\) preserving \(X\), thereby inducing a natural morphism
\(\AutSch(V,\beta) \to \AutSch(X)\) of group schemes. Its image is the subgroup
scheme consisting of linear automorphisms of \(X\). Its kernel is the
intersection
\[
\AutSch(V,\beta) \cap \mathbf{G}_m \cong \boldsymbol{\mu}_{q+1}
\]
of \(\AutSch(V,\beta)\) with the central torus \(\mathbf{G}_m \subset \GL(V)\).
Examining its action on any \(v \in \Fr^*(V)\) and \(w \in V\) such that
\(\beta(v,w) \neq 0\) shows it consists only of \((q+1)\)-st roots of unity.
Therefore there is an exact sequence of group schemes
\[
1 \to
\boldsymbol{\mu}_{q+1} \to
\AutSch(V,\beta) \to
\AutSch(X).
\]
This sequence is not right exact for general \((V,\beta)\), see
\parref{curves-1+N2.nonlinear-auts} for example. It is, however, usually
surjective when \((V,\beta)\) is nonsingular:

\begin{Lemma}\label{hypersurfaces-automorphisms-smooth}
If \((V,\beta)\) is nonsingular and \((n,q)\) is neither \((1,2)\) nor \((2,3)\),
then the restriction morphism \(\AutSch(V,\beta) \to \AutSch(X)\) is surjective
and \(\AutSch(X) \cong \mathrm{PU}(V,\beta)\).
\end{Lemma}

\begin{proof}
By \parref{hypersurfaces-smooth-and-nondegeneracy}, \(X\) is smooth. By
\parref{hypersurfaces-tangent-vectors} below, \(\AutSch(X)\) is reduced; see also
\cite[Example 5]{MO:Automorphisms}. Then by
\cite[Theorem 1]{Chang:Automorphisms} and
\cite[Theorem 2]{MM:Automorphisms}, all automorphisms of \(X\) are linear;
see also \cite[Theorem 1.1]{Poonen:Automorphisms}. This proves the statement.
This can also be verified by a direct computation: see
\cite{Shioda:Automorphisms}.
\end{proof}

The Lie algebra of \(\AutSch(X)\) can be determined as follows:

\begin{Lemma}\label{hypersurfaces-tangent-vectors}
There are canonical isomorphisms of \(\kk\)-vector spaces:
\[
\mathrm{Lie}(\AutSch(X)) \cong
\mathrm{H}^0(X,\mathcal{T}_X) \cong
\Fr^*(V)^\perp \otimes V^\vee.
\]
\end{Lemma}

\begin{proof}
Its defining exact sequence from \parref{hypersurfaces-embedded-tangent-sheaf}
together with the computation of the conormal map from \parref{hypersurfaces-conormal-compute}
give a canonical identification
\[
\mathrm{H}^0(X,\mathcal{T}_X^{\mathrm{e}}(1))
= \ker\big(
V \otimes V^\vee \xrightarrow{\beta \otimes \id}
\Fr^*(V)^\vee \otimes V^\vee \xrightarrow{\mathrm{mult}}
\Sym^{q+1}(V^\vee)\big).
\]
The kernel of the multiplication map
\(\Fr^*(V)^\vee \otimes V^\vee \to \Sym^{q+1}(V^\vee)\) consists of the
\(q\)-bic polynomials vanishing on \(X\). This is the
\(1\)-dimensional subspace spanned by the equation \(f_\beta\) of
\(X\) and is the image under \(\beta \otimes \id\) of the trace element of
\(V \otimes V^\vee\). Recalling that
\(\Fr^*(V)^\perp = \ker(\beta \colon V \to \Fr^*(V)^\vee)\),
this means there is a canonical short exact sequence
\[ 0 \to \kk \xrightarrow{\tr} \mathrm{H}^0(X,\mathcal{T}_X^{\mathrm{e}}(1)) \to \Fr^*(V)^\perp \otimes V^\vee \to 0. \]
The trace element is the image of the Euler section in
\[ 0 \to \sO_X \xrightarrow{\mathrm{eu}} \mathcal{T}_X^{\mathrm{e}}(1) \to \mathcal{T}_X \to 0 \]
so this gives the claimed identification of \(\mathrm{H}^0(X,\mathcal{T}_X)\).
\end{proof}

Although \(\AutSch(V,\beta)/\boldsymbol{\mu}_{q+1}\) is not always isomorphic
to \(\AutSch(X)\), their Lie algebras are nonetheless isomorphic:

\begin{Proposition}\label{hypersurfaces-automorphisms-identify-tangent-spaces}
Then canonical morphism
\(\AutSch(V,\beta) \to \AutSch(X)\) induces an isomorphism on Lie algebras.
\end{Proposition}

\begin{proof}
Taking Lie algebras of along the exact sequence of group schemes given in
\parref{hypersurfaces-automorphisms-linear} shows that the map
\(\mathrm{Lie}(\AutSch(V,\beta)) \to \mathrm{Lie}(\AutSch(X))\) is
injective. But these spaces have the same dimensions by \parref{forms-aut-tangent-space}
and \parref{hypersurfaces-tangent-vectors}, and so they are isomorphic.
\end{proof}

\section{Cones}\label{section-hypersurfaces-cones}
Cones over \(q\)-bic hypersurfaces often intervene in inductive arguments and
constructions, and the notion of cone points provides a manner with which to
endow the set of Hermitian points of a \(q\)-bic hypersurface with a scheme
structure. This Section develops some basic properties of cones, see especially
\parref{hypersurfaces-cone-points-classify}, \parref{hypersurfaces-cone-points-smooth},
and \parref{hypersurfaces-cone-points-equations-general}. An application of
this material is to count the number of Hermitian points in a smooth \(q\)-bic
hypersurface, and to show that maximal isotropic subspaces are Hermitian: see
\parref{hypersurfaces-smooth-cone-points-count} and
\parref{hypersurfaces-cones-even-maximal-isotropic}.

To fix terminology, consider first a projective variety \(Y \subset \PP V\)
over an algebraically closed field. Then \(Y\) is said to be a \emph{cone over
a closed point \(v \in Y\)} if for every closed point \(y \in Y\), the line
\(\langle v,y \rangle\) spanned in \(\PP V\) by \(v\) and \(y\) is contained in
\(Y\). The locus of such points \(v\) form a linear subspace
\(\mathrm{Vert}(Y)\) called the \emph{vertex of \(Y\)}.
A \emph{cone} is a projective variety with a nonempty vertex. A projective
variety over a general field is called a \emph{cone} if it is a cone upon base
extension to an algebraic closure.

From its definition in \parref{forms-orthogonals}, it follows that the
radical of a \(q\)-bic form can only grow along purely inseparable extensions.
This allows the vertex of a \(q\)-bic hypersurface to be described in terms of
its \(q\)-bic form when the base field is perfect:

\begin{Proposition}\label{hypersurfaces-cones}
Let \(X\) be the \(q\)-bic hypersurface associated with a \(q\)-bic form
\((V,\beta)\). If \(\kk\) is perfect, then the vertex of \(X\) is defined over
\(\kk\) and is given by
\[
\mathrm{Vert}(X) =
\PP(\rad(\beta)) =
\PP\Fr^*(V)^\perp \cap \PP\Fr^{-1}(V^\perp).
\]
In particular, \(X\) is a cone if and only if \(\beta\) has a radical.
\end{Proposition}

\begin{proof}
Since the radical
\(W \coloneqq \rad(\beta) = \Fr^*(V)^\perp \cap \Fr^{-1}(V^\perp)\) of \(\beta\) is
defined over \(\kk\), it suffices to check that this give the vertex of
\(X\) upon passage to the algebraic closure. For the remainder of the proof,
assume \(\kk\) is algebraically closed.

The inclusion \(\PP W \subseteq \mathrm{Vert}(X)\) follows easily from
the definitions. For the reverse inclusion \(\PP W \supseteq \mathrm{Vert}(X)\),
consider a closed point \(z = \PP\langle w\rangle\) of the vertex. The goal is
to show \(\beta(w^{(q)},v) = \beta(v^{(q)},w) = 0\) for all \(v \in V\). If
\(y = \PP\langle v \rangle\) were contained in \(X\), then
\(\langle y,z \rangle = \PP\langle v,w \rangle\) is contained in \(X\),
and the result follows from \parref{forms-notions-of-isotropicity}. If
\(y\) were not contained in \(X\), let \(\lambda \in \kk\) be arbitrary and
consider
\[
\beta((v + \lambda w)^{(q)}, v + \lambda w)
= \beta(v^{(q)}, v) + \lambda \beta(v^{(q)}, w) + \lambda^q \beta(w^{(q)}, v).
\]
That \(y \notin X\) means \(\beta(v^{(q)},v) \neq 0\). Thus if
\(\beta(v^{(q)}, w) \neq 0\) or \(\beta(w^{(q)}, v) \neq 0\), then there exists
nonzero \(\lambda\) such that \(v + \lambda w\) is isotropic. Then
\(y' \coloneqq \PP\langle v + \lambda w\rangle\) is a point of \(X\) and so
\[
X \supseteq
\langle y',z \rangle =
\PP\langle v + \lambda w, w\rangle =
\PP\langle v,w \rangle =
\langle y,z \rangle,
\]
contradicting \(y \notin X\). Thus
\(\beta(v^{(q)},w) = \beta(w^{(q)},v) = 0\), whence \(w \in W\).
\end{proof}

As a consequence, a \(q\)-bic hypersurface \(X\) with high corank must be a
cone:

\begin{Corollary}\label{hypersurfaces-cone-high-corank}
If \(2 \dim\Sing(X) \geq \dim X\), then \(X\) is either of type
\(\mathbf{N}_2^{\oplus m}\) or a cone.
\end{Corollary}

\begin{proof}
Consider the contrapositive: assume \(X\) is not of type
\(\mathbf{N}_2^{\oplus m}\) and does not have a vertex. By
\parref{hypersurfaces-cones}, this means the \(q\)-bic form \((V,\beta)\)
underlying \(X\) does not have a radical. By the classification
\parref{forms-classification-theorem}, this menas \(\beta\) is a sum of forms
of type \(\mathbf{1}\) and \(\mathbf{N}_k\) with \(k \geq 2\). Since the
corank of \(\beta\) is maximized by a sum of forms of type \(\mathbf{N}_2\),
\[
2 \corank(\beta)
< 2 \corank(\mathbf{N}_2^{\oplus \lfloor \dim_\kk V/2\rfloor})
\leq  \dim_\kk V.
\]
Then \parref{hypersurfaces-nonsmooth-locus} implies
\(2 \dim\Sing(X) < \dim X\), as required.
\end{proof}

Cones over \(q\)-bic hypersurfaces frequently arise upon taking special linear
sections of a given \(q\)-bic. The following characterizes when this happens:

\begin{Corollary}\label{hypersurfaces-cone-maximal}
Let \(X\) be the \(q\)-bic hypersurface associated with a \(q\)-bic form
\((V,\beta)\). Let \(\PP U \subseteq X\) be a linear subspace. Then a linear
section \(X \cap \PP W\) is a cone with vertex containing \(\PP U\) if and
only if \(W \subseteq \Fr^*(U)^\perp \cap \Fr^{-1}(U^\perp)\).
\end{Corollary}

\begin{proof}
By \parref{hypersurface-hyperplane-section}, \(X \cap \PP W\) is the \(q\)-bic
defined by the restriction \(\beta_W\) of \(\beta\) to \(W\). Then, by
\parref{hypersurfaces-cones}, \(X \cap \PP W\) is a cone with vertex containing
\(\PP U\) if and only if \(U\) lies in the radical of \(\beta_W\). And this
happens if and only if \(W \subseteq \Fr^*(U)^\perp \cap \Fr^{-1}(U^\perp)\).
\end{proof}

The following is a special case in which \(\PP W\) is furthermore contained
in \(X\). A particular form of this has been observed by Shimada in
\cite[Proposition 2.10]{Shimada:Lattices}. This refines the general fact
that given a projective variety \(Y \subseteq \PP V\) and a \(\kk\)-point
\(y \in Y\), any linear subspace \(\PP U \subseteq Y\) passing through \(y\)
must be contained in the embedded tangent space of \(Y\) at \(y\).

\begin{Corollary}\label{threefolds-fano-linear-flag}
Let \(\PP U \subseteq \PP W \subset X\) be a nested pair of linear subspaces.
Then
\[
\PP W \subseteq X \cap \PP\Fr^*(U)^\perp \cap \PP\Fr^{-1}(U^\perp).
\]
\end{Corollary}

\begin{proof}
This follows from \parref{hypersurfaces-cone-maximal}, where \(\PP W\) is
viewed as the \(q\)-bic in \(\PP W\) defined by the zero form, and noting
that any linear subspace is a vertex.
\end{proof}

\subsection{Cone points}\label{hypersurfaces-cone-points-definition}
A \(\kk\)-point \(y\) of a projective variety \(Y \subseteq \PP V\) is said to be a
\emph{cone point} of \(Y\) if there exists a hyperplane \(H \subset \PP V\)
such that \(Y \cap H\) is a cone over \(y\). A few remarks about this definition
when \(\kk\) is algebraically closed:
\begin{enumerate}
\item\label{hypersurfaces-cone-points-definition.quadrics}
If \(Y\) is a linear or a quadric hypersurface, then every \(\kk\)-point
is a cone point.
\item\label{hypersurfaces-cone-points-definition.cubics}
Cone points of a smooth cubic hypersurface are its
\emph{Eckardt points}.
\item\label{hypersurfaces-cone-points-definition.tangent}
If a cone point \(y \in Y\) is a smooth point, then the witnessing
hyperplane \(H\) must contain the embedded tangent space of \(Y\) at \(y\).
\item\label{hypersurfaces-cone-points-definition.star}
In the case that \(Y\) is a hypersurface and \(y\) is a smooth point, cone
points are sometimes referred
to as \emph{total inflection points} or \emph{star points}; see
\cite{CC:StarPoints, CC:StarPointsII}.
\end{enumerate}

Cone points of \(q\)-bic hypersurfaces are analogous to those of
quadrics as in
\parref{hypersurfaces-cone-points-definition}\ref{hypersurfaces-cone-points-definition.quadrics}
in that \(X \cap H\) is itself a \(q\)-bic, see
\parref{hypersurface-hyperplane-section}, which is a cone over a lower
dimensional \(q\)-bic. Unlike quadrics, however, \(q\)-bics generally only
have finitely many cone points, and are thus also analogous to the Eckardt points
of cubic hypersurfaces as in
\parref{hypersurfaces-cone-points-definition}\ref{hypersurfaces-cone-points-definition.cubics}.

Cone points of a \(q\)-bic hypersurface may be characterized as follows:

\begin{Lemma}\label{hypersurfaces-cone-points-criterion}
Let \(X\) be a \(q\)-bic hypersurface and let \(x = \PP L\) be
a \(\kk\)-point of \(X\).
\begin{enumerate}
\item\label{hypersurfaces-cone-points-criterion.linear-slice}
Let \(\PP W \subseteq \PP V\) be a linear subspace. Then \(X \cap \PP W\) is a
cone over \(x\) if and only if
\[
W \subseteq
\Fr^*(L)^\perp \cap \Fr^{-1}(L^\perp).
\]
\item\label{hypersurfaces-cone-points-criterion.cone-point}
The point \(x\) is a cone point of \(X\) if and only if
\[
\dim_\kk(\Fr^*(L)^\perp \cap \Fr^{-1}(L^\perp)) \geq
\dim_\kk V - 1.
\]
\end{enumerate}
\end{Lemma}

\begin{proof}
Item \ref{hypersurfaces-cone-points-criterion.linear-slice} follows directly
from \parref{hypersurfaces-cone-maximal}. Item
\ref{hypersurfaces-cone-points-criterion.cone-point} now follows since \(x\) is
a cone point of \(X\) if and only if there exists a hyperplane \(W\) of \(V\)
contained in \(\Fr^{-1}(L^\perp) \cap \Fr^*(L)^\perp\).
\end{proof}

This criterion can be reformulated into a more geometric classification of cone
points. The notion of Hermitian points of a \(q\)-bic hypersurface was
defined in \parref{hypersurfaces-terminology}.

\begin{Corollary}\label{hypersurfaces-cone-points-classify}
A \(\kk\)-point \(x\) of \(X\) is a cone point if and only if either
\begin{enumerate}
\item\label{hypersurfaces-cone-points-classify.singular}
\(x\) is a singular point, or
\item\label{hypersurfaces-cone-points-classify.intermediate}
\(x\) is a smooth point lying in \(\PP\Fr^*(V)^\perp \subseteq X\), or
\item\label{hypersurfaces-cone-points-classify.smooth}
\(x\) is a smooth Hermitian point.
\end{enumerate}
\end{Corollary}

\begin{proof}
Let \(x = \PP L\). Then by
\parref{hypersurfaces-cone-points-criterion}\ref{hypersurfaces-cone-points-criterion.linear-slice},
\(x\) is a cone point if and only there is a hyperplane lying in
\(\Fr^*(L)^\perp \cap \Fr^{-1}(L^\perp)\). This happens if and only if either
\begin{enumerate}
\item \(\Fr^*(L)^\perp = V\), or
\item \(\Fr^{-1}(L^\perp) = V\), or
\item \(\Fr^*(L) = \Fr^{-1}(L^\perp)\) are both hyperplanes.
\end{enumerate}
The first condition occurs when \(x\) is a singular point, see \parref{hypersurfaces-tangent-space-as-kernel};
if \(x\) is not a singular point, then the second condition occurs when
\(\beta^\vee \colon \Fr^*(V) \to V^\vee \twoheadrightarrow L^\vee\) vanishes;
dually, this means \(\beta \colon L \subset V \to \Fr^*(V)^\vee\) is zero, so
\(L \subseteq \Fr^*(V)^\perp\).
Finally, the third condition occurs precisely when \(x\) is a smooth Hermitian
point, see \parref{forms-hermitian-subspace-converse}.
\end{proof}

The following recasts cone points of smooth \(q\)-bic surfaces in several ways:

\begin{Proposition}\label{hypersurfaces-cone-points-smooth}
Let \(X\) be a smooth \(q\)-bic hypersurface. Then for a \(\kk\)-point
\(x\) of \(X\), the following are equivalent:
\begin{enumerate}
\item\label{hypersurfaces-cone-points-smooth.tangent}
\(X \cap \mathbf{T}_{X,x}\) is a cone over a smooth \(q\)-bic hypersurface
of dimension \(2\) less;
\item\label{hypersurfaces-cone-points-smooth.cone}
\(x\) is a cone point of \(X\);
\item\label{hypersurfaces-cone-points-smooth.hermitian}
\(x\) is a Hermitian point of \(X\).
\end{enumerate}
\end{Proposition}

\begin{proof}
That
\ref{hypersurfaces-cone-points-smooth.tangent} \(\Rightarrow\)
\ref{hypersurfaces-cone-points-smooth.cone} \(\Leftrightarrow\)
\ref{hypersurfaces-cone-points-smooth.hermitian} follows from definitions and
\parref{hypersurfaces-cone-points-classify}\ref{hypersurfaces-cone-points-classify.smooth}.
For
\ref{hypersurfaces-cone-points-smooth.hermitian} \(\Rightarrow\)
\ref{hypersurfaces-cone-points-smooth.tangent}, let \((V,\beta)\) be a \(q\)-bic
form underlying \(X\) and let \(x = \PP L\) with \(L\) a \(1\)-dimensional Hermitian
subspace of \(V\). Then \(\Fr^*(L)^\perp = \Fr^{-1}(L^\perp)\) by \parref{forms-hermitian-subspace-basics},
and this is the hyperplane underlying the embedded tangent space of
\(X\) at \(x\), see \parref{hypersurfaces-tangent-space-as-kernel}. The
restriction of \(\beta\) therein contains \(L\) in its radical and has corank
at most \(1\) by \parref{forms-rank-linear-subspace}. Thus by
\parref{hypersurfaces-cones}, \(\mathbf{T}_{X,x} \cap X\) is a cone over a
smooth \(q\)-bic of \(2\) dimensions less.
\end{proof}

\subsection{Scheme of cone points}\label{hypersurfaces-cone-points-equations-general}
The criterion of \parref{hypersurfaces-cone-points-criterion} may be used to
produce global equations for the set of cone points in a \(q\)-bic hypersurface
\(X\), moreover endowing it with a natural scheme structure. Consider the
morphism of \(\sO_X\)-modules:
\[
(\mathrm{eu}^\vee \circ \beta^\vee,
\mathrm{eu}^{(q^2),\vee} \circ \beta^{(q)}) \colon
\Fr^*(V)_X \to
\sO_X(1) \oplus \sO_X(q^2).
\]
Assume for the moment that \(\kk\) is algebraically closed. Let \(x = \PP L\)
be a closed point of \(X\) and let \(\kappa(x)\) be its residue field. Taking
fibres at \(x\) of the kernels of the components gives:
\begin{align*}
\ker(\mathrm{eu}^\vee \circ \beta^\vee \colon
\Fr^*(V)_X \to \sO_X(1)) \otimes_{\sO_X} \kappa(x)
& = L^\perp, \\
\ker(\mathrm{eu}^{(q^2),\vee} \circ \beta^{(q)} \colon
\Fr^*(V)_X \to \sO_X(q^2)) \otimes_{\sO_X} \kappa(x)
& = \Fr^*(\Fr^*(L)^\perp).
\end{align*}
Combined with
\parref{hypersurfaces-cone-points-criterion}\ref{hypersurfaces-cone-points-criterion.cone-point},
this implies that the degeneracy locus \(X_{\mathrm{cone}}\) given by
\[
\rank\big(
(\mathrm{eu}^\vee \circ \beta^\vee,
\mathrm{eu}^{(q^2),\vee} \circ \beta^{(q)}) \colon
\Fr^*(V)_X \to
\sO_X(1) \oplus \sO_X(q^2)
\big) \leq 1
\]
is supported on the set of cone points of \(X\).

This construction may be performed for any \(q\)-bic hypersurface
\(X\) over any field \(\kk\) to yield the \emph{scheme of cone points}
\(X_{\mathrm{cone}}\) of \(X\). Its equations may be simplified slightly:

\begin{Lemma}\label{hypersurfaces-cone-points-equations}
The scheme \(X_{\mathrm{cone}}\) of cone points is the top degeneracy locus of
\[
(\mathrm{eu}^\vee \circ \beta^\vee, \delta^{(q),\vee}) \colon
\Fr^*(\mathcal{T}_{\PP V}(-1))\rvert_X \to
\sO_X(1) \oplus \Fr^*(\mathcal{N}_{X/\PP V}(-1)).
\]
On the smooth locus \(X_{\mathrm{sm}} \subset X\), the scheme
\(X_{\mathrm{cone}}\) is defined by the vanishing of
\[
\mathrm{eu}^\vee \circ \beta^\vee \colon
\Fr^*(\mathcal{T}_X(-1))\rvert_{X_{\mathrm{sm}}} \to
\sO_X(1)\rvert_{X_{\mathrm{sm}}}.
\]
\end{Lemma}

\begin{proof}
By the Euler sequence,
\(\Fr^*(\mathcal{T}_{\PP V}(-1))\rvert_X\) is the quotient of \(\Fr^*(V)_X\)
by the section \(\mathrm{eu}^{(q)} \colon \sO_X(-q) \to \Fr^*(V)_X\). The
first statement follows from \parref{hypersurfaces-cone-points-equations-general}
upon observing that
\[
\mathrm{eu}^\vee \circ \beta^\vee \circ \mathrm{eu}^{(q)} = 0
\quad\text{and}\quad
\mathrm{eu}^{(q^2),\vee} \circ \beta^{(q)} \circ \mathrm{eu}^{(q)} = 0
\]
as morphisms of line bundles on \(X\). The identification of the second
component with the normal map \(\delta^{(q),\vee}\) follows from
\parref{hypersurfaces-conormal-compute}.

For the second statement, the normal map \(\delta^{(q),\vee}\) is surjective
on the smooth locus of \(X\), and thereon, the degeneracy locus coincides
with the vanishing locus of the first component \(\mathrm{eu}^\vee \circ \beta^\vee\)
restricted to
\(\ker(\delta^{(q),\vee}) = \Fr^*(\mathcal{T}_X(-1))\).
\end{proof}

The scheme \(X_{\mathrm{cone}}\) is therefore the rank \(1\)
degeneracy locus of a map of bundles of ranks \(n\) and \(2\). Thus it has
expected dimension \(0\), see \cite[Chapter 14]{Fulton}, for example. When
it does have expected dimension, its degree is as follows:

\begin{Lemma}\label{hypersurfaces-cone-points-scheme-degree}
If the scheme of cone points is of expected dimension \(0\), then
\[
\deg(X_{\mathrm{cone}}) =
\frac{(q^{n+1} - (-1)^{n+1})(q^n - (-1)^n)}{q^2-1}.
\]
\end{Lemma}

\begin{proof}
Since \(X_{\mathrm{cone}}\) is given as a degeneracy locus in
\parref{hypersurfaces-cone-points-equations}, this follows from the
Thom--Porteous formula, see \cite[Theorem 14.4 and Example 14.4.1]{Fulton}.
Namely, the class of \(X_{\mathrm{cone}}\) in the Chow ring of \(X\) is
given by the degree \(n-1\) component of
\[
\frac{c(\Fr^*(\Omega^1_{\PP V}(1))\rvert_X)}{c(\sO_X(-1)) c(\sO_X(-q^2))} =
\frac{1}{(q+1)^2}\Bigg(
q \cdot \frac{1}{1+qh} -
\frac{1}{q-1} \cdot \frac{1}{1-h} +
\frac{q^3}{q-1} \cdot \frac{1}{1-q^2 h}
\Bigg)
\]
where \(h \coloneqq c_1(\sO_X(1))\) and
\(c(\Fr^*(\Omega^1_{\PP V}(1))\rvert_X) = c(\sO_X(q))^{-1}\) by the Euler
sequence. Expanding the geometric series and taking the coefficient of
\(h^{n-1}\) gives
\[
[X_{\mathrm{cone}}] =
\frac{(-1)^{n-1} q^n(q-1) - 1 + q^{2n+1}}{(q+1)^2(q-1)} \cdot h^{n-1} =
\frac{(q^{n+1} - (-1)^{n+1})(q^n - (-1)^n)}{q^2-1} \cdot \frac{h^{n-1}}{q+1}
\]
in \(\mathrm{CH}_0(X)\). Since \(\deg(h) = q+1\), the result follows.
\end{proof}

For general \(X\), the schematic structure of \(X_{\mathrm{cone}}\) may be
quite intricate even when it is of expected dimension: see
\parref{curves-1+N2.cone-points}, \parref{surfaces-1+1+N2.cone-points-scheme},
and \parref{surfaces-1+1+N2.cone-points-scheme-addendum}. When \(X\) is smooth,
it is as simple as possible:

\begin{Lemma}\label{hypersurfaces-smooth-cone-points-etale}
If \(X\) is a smooth \(q\)-bic hypersurface, then \(X_{\mathrm{cone}}\) is
\'etale over \(\kk\).
\end{Lemma}

\begin{proof}
The result will follow upon showing that the conormal morphism
\[
\delta_{X_{\mathrm{cone}}/X} \colon
\mathcal{C}_{X_{\mathrm{cone}}/X} \to
\Omega^1_X\rvert_{X_{\mathrm{cone}}}
\]
is an isomorphism of sheaves on \(X_{\mathrm{cone}}\). By the second statement
of \parref{hypersurfaces-cone-points-equations}, \(X_{\mathrm{cone}}\) is the
vanishing locus of
\(\mathrm{eu}^\vee \circ \beta^\vee \colon \Fr^*(\mathcal{T}_X(-1)) \to \sO_X(1)\),
so restriction gives a map
\[
\delta_{X_{\mathrm{cone}}/X} \circ (\mathrm{eu}^\vee \circ \beta^\vee) \colon
\Fr^*(\mathcal{T}_X(-1))\rvert_{X_{\mathrm{cone}}} \to
\Omega^1_X\rvert_{X_{\mathrm{cone}}}
\]
Locally, \(\delta_{X_{\mathrm{cone}}/X}\) acts by differentiating the
coordinates \(\mathrm{eu}^\vee\). Thus this map is dual to the isomorphism
induced by
\(\mathcal{T}_X^{\mathrm{e}} \cong \Fr^*(\Omega^1_{\PP V}(1))\rvert_X\)
from \parref{hypersurfaces-frobenius-euler} upon passing to the quotient by
the Euler section.
\end{proof}

This gives a geometric method to determine the number of Hermitian points of
a smooth \(q\)-bic hypersurface. See also \cite[Theorem 8.1]{BC:Hermitian}
and \cite[n.32]{Segre:Hermitian}.

\begin{Corollary}\label{hypersurfaces-smooth-cone-points-count}
The number of Hermitian points of a smooth \(q\)-bic \((n-1)\)-fold
over a separably closed field is
\[
\# X_{\mathrm{Herm}} =
\frac{(q^{n+1} - (-1)^{n+1})(q^n - (-1)^n)}{q^2-1}.
\]
\end{Corollary}

\begin{proof}
By \parref{hypersurfaces-cone-points-smooth}, the Hermitian points of \(X\)
are the cone points of \(X\). Since the scheme of cone points is \'etale by
\parref{hypersurfaces-smooth-cone-points-etale}, this implies that
\(\# X_{\mathrm{Herm}} = \deg(X_{\mathrm{cones}})\) and the result follows from
\parref{hypersurfaces-cone-points-scheme-degree}.
\end{proof}

Cone points are often incident with linear spaces contained in \(X\), see
\parref{threefolds-fano-linear-flag}. However, linear spaces in \(X\) that
contain cone points are typically quite special. The remainder of this Section
presents a basic study of this. The following gives equations for the subscheme
of cone points contained in a given linear subspace:

\begin{Lemma}\label{hypersurfaces-cone-points-subspaces-equations}
Let \(\PP U \subset X\) be a linear subspace. Then
\(\PP U \cap X_{\mathrm{cone}}\) is the degeneracy locus of a map
\[
\Fr^*(\mathcal{N}_{\PP U/\PP V}(-1)) \to
\sO_{\PP U}(1) \oplus \sO_{\PP U}(q^2).
\]
If \(\PP U\) is contained in the smooth locus, then \(\PP U \cap X_{\mathrm{cone}}\)
is the vanishing locus of a map
\[
\Fr^*(\mathcal{N}_{\PP U/X}(-1)) \to
\sO_{\PP U}(1).
\]
\end{Lemma}

\begin{proof}
Restrict the morphism from \parref{hypersurfaces-cone-points-equations-general}
defining \(X_{\mathrm{cone}}\) to \(\PP U\). The kernel computation there shows
that there is an inclusion
\[
\Fr^*(U)_{\PP U} \subset
\ker\big(
(\mathrm{eu}^\vee \circ \beta^\vee,
\mathrm{eu}^{(q^2),\vee} \circ \beta^{(q)}) \colon
\Fr^*(V)_{\PP U} \to
\sO_{\PP U}(1) \oplus \sO_{\PP U}(q^2)\big).
\]
Passing to the quotient by this and the \(q\)-power Euler section, as in
the proof of \parref{hypersurfaces-cone-points-equations}, gives the first
statement. If \(\PP U \subset X\) is contained in the smooth locus, argue
as in the second part of \parref{hypersurfaces-cone-points-equations}
to see that the map
\(\Fr^*(\mathcal{N}_{\PP U/\PP V}(-1)) \to \sO_{\PP U}(q^2)\)
is the natural surjection to
\(\Fr^*(\mathcal{N}_{X/\PP V}(-1))\rvert_{\PP U}\). Passing to the kernel
gives the result.
\end{proof}

If \(X\) is of even dimension \(2m\) and \(\PP U\) is an \(m\)-plane contained
in \(X\), \parref{hypersurfaces-cone-points-subspaces-equations} expresses
\(\PP U \cap X_{\mathrm{cone}}\) as a rank \(1\) degeneracy locus between
bundles of ranks \(m+1\) and \(2\). Therefore, so long as the locus is
nonempty, it has dimension at least \(0\), see \cite[Chapter 14]{Fulton}.
The following verifies that the degeneracy locus is nonempty in this case
using a normal bundle computation performed later in
\parref{hypersurfaces-linear-subspace-normal-bundle}:

\begin{Corollary}\label{hypersurfaces-cone-points-subspaces-have}
Let \(X\) be a \(q\)-bic hypersurface of dimension \(2m\). Any
\(m\)-plane \(\PP U\) in \(X\) contains a cone point, and if
\(\dim(\PP U \cap X_{\mathrm{cone}}) = 0\), then
\[
\deg(\PP U \cap X_{\mathrm{cone}}) =
\frac{q^{2m+2} - 1}{q^2-1} =
q^{2m} + q^{2m-2} + \cdots + q^2 + 1.
\]
\end{Corollary}

\begin{proof}
If \(\PP U\) intersects the singular locus of \(X\), then it contains a
cone point, see \parref{hypersurfaces-cone-points-classify}. If \(\PP U\)
is contained in the smooth locus of \(X\), then the second statement of
\parref{hypersurfaces-cone-points-subspaces-equations} shows that
\(\PP U \cap X_{\mathrm{cone}}\) is the zero locus of a map
\[ \Fr^*(\mathcal{N}_{\PP U/X}(-1)) \to \sO_{\PP U}(1). \]
Since \(\PP U\) is maximal isotropic,
\parref{hypersurfaces-linear-subspace-normal-bundle} gives
\(\mathcal{N}_{\PP U/X}(-1) \cong \Fr^*(\Omega^1_{\PP U}(1))\). Thus
the section above vanishes on \(\PP U\), showing it contains a cone point.

If \(\dim(\PP U \cap X_{\mathrm{cone}}) = 0\), the Thom--Porteous
formula, see \cite[Theorem 14.4 and Example 14.4.1]{Fulton}, shows that
its class in the Chow ring is the degree \(m\) part of
\[
\frac{c(\sO_{\PP U} \otimes \Fr^*(V/U)^\vee)}{c(\sO_{\PP U}(-1)) c(\sO_{\PP U}(-q^2))}
= \frac{1}{1-q^2}\Bigg(\frac{1}{1-h} - \frac{q^2}{1-q^2 h}\Bigg) \in \mathrm{CH}(\PP U)
\]
where \(h \coloneqq c_1(\sO_{\PP U}(1))\). Expanding and taking degrees gives
the result.
\end{proof}

Consequently, any maximal linear space in an even dimensional \(q\)-bic
hypersurface must be Hermitian:

\begin{Corollary}\label{hypersurfaces-cones-even-maximal-isotropic}
Let \(X\) be a smooth \(q\)-bic hypersurface of dimension \(2m\) over a separably
closed field. Then any \(m\)-plane contained in \(X\) is Hermitian.
\end{Corollary}

\begin{proof}
Let \(\PP U\) be such an \(m\)-plane. Then it contains a cone point by
\parref{hypersurfaces-cone-points-subspaces-have}. Since \(X_{\mathrm{cone}}\)
is \'etale over \(\kk\) by \parref{hypersurfaces-smooth-cone-points-etale},
the intersection \(\PP U \cap X_{\mathrm{cone}}\) in fact consists of
\[
\deg(\PP U \cap X_{\mathrm{cone}})
= \frac{q^{2m+2} - 1}{q^2 - 1}
= \#\PP^m(\mathbf{F}_{q^2})
\]
reduced points. But cone points of \(X\) are also its Hermitian points, see
\parref{hypersurfaces-cone-points-smooth}, and its Hermitian points arise from
a \(\mathbf{F}_{q^2}\)-subspace that spans \(V\), see \parref{forms-hermitian-basics}
and \parref{forms-hermitian-nondegenerate-span}. Therefore \(\PP U\) must be
spanned by its Hermitian points.
\end{proof}

\section{Unirationality}\label{hypersurfaces-unirational}
Despite typically being of general type, a smooth \(q\)-bic hypersurface is
typically unirational. This was first discovered by Shioda via an explicit
coordinate computation, see \cite{Shioda:Unirational}. This Section describes
two geometric unirationality constructions: the first, summarized in
\parref{hypersurfaces-unirationality-shioda-parameterization}, refines and
generalizes Shioda's construction, and is based on projecting from a
\(\kk\)-rational line; the second, see
\parref{hypersurfaces-unirational-tangent-morphism}, gives a new construction
that is based on examining tangent lines to the hypersurface based at a fixed
\(\kk\)-rational line. Both constructions should be compared with the standard
unirationality constructions of cubic hypersurfaces: see
\cite[Appendix B]{CG}, \cite[\S\S2 and 5]{Murre:Prym}, and
\cite[Exemple 4.5.1]{Beauville:Prym}.

A summary of the results of
\parref{hypersurfaces-unirationality-shioda-parameterization} and
\parref{hypersurfaces-unirational-tangent-morphism} is:

\begin{Proposition}\label{hypersurfaces-unirational-smooth}
Let \(X\) be a smooth \(q\)-bic hypersurface of dimension at least \(2\). If
\(X\) contains a \(\kk\)-rational line, then \(X\) admits a purely inseparable
unirational parameterization of degree \(q\) over \(\kk\). \qed
\end{Proposition}

\subsection{Unirationality via projection from a line}\label{hypersurfaces-unirational-shioda}
Let \(X\) be the smooth \(q\)-bic hypersurface of dimension \(n - 1 \geq 2\),
and let \((V,\beta)\) be an underlying \(q\)-bic form. The first construction
takes a choice of \(\kk\)-rational line \(\ell = \PP U\) contained in \(X\) and
produces a purely inseparable parameterization \(\PP^{n-1} \dashrightarrow X\)
of degree \(q\) defined over \(\kk\). A special case of this construction was
first found by Shioda in \cite{Shioda:Unirational} via explicit coordinate
computations.

Set \(W \coloneqq V/U\) and write \(\psi \colon V \to W\) for the quotient map.
Consider linear projection \(\PP V \dashrightarrow \PP W\) with centre
\(\ell = \PP U\). By \parref{linear-projection-resolve}, this is resolved along
the incidence correspondence
\[
\PP\psi \coloneqq
\Set{([V'], [W']) \in \PP V \times \PP W | \psi(V') \subseteq W'}.
\]
The first projection \(\PP\psi \to \PP V\) is isomorphic to the blowup of \(\PP V\)
along \(\ell\), whereas the second projection \(\PP\psi \to \PP W\) is isomorphic
to the \(\PP^2\)-bundle associated with the locally free \(\sO_{\PP W}\)-module
\(\mathcal{V}\) of rank \(3\) fitting into the pullback diagram
\[
\begin{tikzcd}[row sep=1em]
0 \rar &
U_{\PP W} \rar &
V_{\PP W} \rar &
W_{\PP W} \rar & 0 \\
0 \rar &
U_{\PP W} \rar \uar[symbol={=}] &
\mathcal{V} \rar \uar[symbol={\subset}] &
\sO_{\PP W}(-1) \rar \uar["\mathrm{eu}"] & 0\punct{.}
\end{tikzcd}
\]
Let
\(\beta_{\mathcal{V}} \colon \Fr^*(\mathcal{V}) \otimes \mathcal{V} \to \sO_{\PP W}\)
be the \(q\)-bic form induced by \(\beta\) on \(\mathcal{V}\). Let \(X'\)
be the inverse image of \(X\) along the blowup \(\PP\psi \to \PP V\).
Then \(X' \to \PP W\) is the family of \(q\)-bic curves defined by the
\(q\)-bic form \((\mathcal{V},\beta_{\mathcal{V}})\) over \(\PP W\). Each fibre
contains the line \(\ell\) and the type of each fibre determines a sequence of
subschemes of \(\PP W\). In the following statement, let \(y\) be a point of
\(\PP W\) corresponding to a \(3\)-dimensional linear subspace \(V' \subset V\)
containing \(U\), so that the fibre \(X'_y\) is a \(q\)-bic curve in the plane
\(\PP V'\).

\begin{Lemma}\label{hypersurfaces-unirationality-shioda-fibres}
The family \(X' \to \PP W\) of \(q\)-bic curves determines a filtration
\[
\PP W \eqqcolon
D_0 \supseteq
D_1 \supseteq
D_2 \supseteq
D_3 \supseteq
D_4
\]
of \(\PP W\) by closed subschemes such that the type of \(X'_y\) over
\(D_i^\circ \coloneqq D_i \setminus D_{i-1}\) is given by
\[
\begin{array}{c|ccccc}
& D_0^\circ & D_1^\circ & D_2^\circ & D_3^\circ & D_4^\circ \\
\hline
\mathrm{type}(X'_y) &
\mathbf{N}_3 &
\mathbf{0} \oplus \mathbf{1}^{\oplus 2} &
\mathbf{0} \oplus \mathbf{N}_2 &
\mathbf{0}^{\oplus 2} \oplus \mathbf{1} &
\mathbf{0}^{\oplus 3}
\end{array}
\]
and such that \(D_1\) is a \(q^2\)-bic hypersurface.
\end{Lemma}

\begin{proof}
Let \(\tilde{X}\) and \(E\) be the strict transform of \(X\) and the exceptional
divisor along the blowup \(\PP\psi \to \PP V\). Since \(X\) is irreducible
and \(\ell \subset X\) is an inclusion of smooth varieties, the decomposition
of \(X'\) into irreducible components is given by
\[ X' = \tilde{X} \cup E. \]
Therefore the fibre \(X'_y\) for a general point \(y \in \PP W\) is the union
of \(\ell\) along with an irreducible plane curve of degree \(q\), so \(X'_y\)
is of type \(\mathbf{N}_3\) by the classification of \(q\)-bic curves: see
\parref{forms-classification-theorem}, or \parref{qbic-curves-classification}
for the case of \(q\)-bic curves specifically. The filtration of \(\PP W\)
now arises by locally pulling back the type stratification from the parameter
space of \(q\)-bic forms on a \(4\)-dimensional vector space: see
\parref{forms-classification-moduli},
\parref{forms-classification-type-stratification} and
\parref{qbic-curves-moduli}.

It remains to show that \(D_1\) is a \(q^2\)-bic hypersurface. Consider the
\(q\)-bic form \((\mathcal{V},\beta_{\mathcal{V}})\) defining the family
\(X' \to \PP W\). Since \(U\) is totally isotropic for \(\beta\), there is a
commutative diagram of \(\sO_{\PP W}\)-modules
\[
\begin{tikzcd}
0 \rar
& U_{\PP W} \rar \dar
& \mathcal{V} \rar \dar["\beta_{\mathcal{V}}"]
& \sO_{\PP W}(-1) \rar \dar[hook]
& 0  \\
0 \rar
& \sO_{\PP W}(q) \rar
& \Fr^*(\mathcal{V})^\vee \rar
& \Fr^*(U)^\vee_{\PP W} \rar
& 0
\end{tikzcd}
\]
in which the rows are exact. Since \(\beta_{\mathcal{V}}\) is generically
of rank \(2\), the diagram implies that the vertical map
\(\sO_{\PP W}(-1) \to \Fr^*(U)_{\PP W}^\vee\) on the right is injective. This
shows that the kernels of \(\beta_{\mathcal{V}}\) factor through the subbundle
\(U_{\PP W} \subset \mathcal{V}\). Now \(D_1\) is the locus over which the
fibres of \(X' \to \PP W\) become cones; by \parref{hypersurfaces-cones},
this is the locus over which the kernels of \(\beta_{\mathcal{V}}\) coincide
and yield a radical. This is given by the degeneracy locus of
\[
\Fr^*(\beta_{\mathcal{V}}) \oplus
\beta_{\mathcal{V}}^\vee \colon
\Fr^*(U)_{\PP W} \to \sO_{\PP W}(q^2) \oplus \sO_{\PP W}(1).
\]
Since both the source and target are the same rank, taking determinants shows
that \(D_1\) is the vanishing locus of the morphism
\[
\det(\Fr^*(\beta_{\mathcal{V}}) \oplus
\beta_{\mathcal{V}}^\vee) \colon
\sO_{\PP W} \to \Fr^*(\det(U))^\vee \otimes \sO_{\PP W}(q^2+1).
\]
The construction of this section shows it is a product of linear and
\(q^2\)-power combinations of the coordinates of \(\PP W\), and so \(D_1\) is a
\(q^2\)-bic hypersurface.
\end{proof}

Let \(\eta \in \PP W\) be the generic point and consider the generic fibre
\(\tilde{X}_\eta\) of \(\tilde{X} \to \PP W\). On the one hand,
\parref{hypersurfaces-unirationality-shioda-fibres} implies that
\(\tilde{X}_\eta\) is geometrically isomorphic to the degree \(q\) component of
a \(q\)-bic curve of type \(\mathbf{N}_3\), so it is geometrically a
unicuspidal rational curve: see \parref{qbic-curves-components-are-rational}.
On the other hand, \(X\) is regular so \(\tilde{X}_\eta\) is regular and its
cusp becomes visible only upon passing to some purely inseparable extension of
the function field \(\kk(\PP W)\). This extension is found by considering the
exceptional divisor \(E_X\) of the blowup \(\tilde{X} \to X\), which more
globally gives:

\begin{Proposition}\label{hypersurfaces-unirationality-shioda-parameterization}
Let \(X\) be a smooth \(q\)-bic hypersurface of dimension at least \(2\). For
every \(\kk\)-rational line \(\ell \subset X\), there exists a diagram
\[
\begin{tikzcd}
\ell \dar[symbol={\subset}]
& E_X \lar \dar[symbol={\subset}]
& E_X \times_{\PP W} E_X \lar \dar[symbol={\subset}] \\[-1em]
X
& \tilde{X} \lar \dar
& \tilde{X} \times_{\PP W} E_X \dar \lar \\
& \PP W
& E_X \lar
\end{tikzcd}
\]
in which
\begin{enumerate}
\item\label{hypersurfaces-unirationality-shioda-paramterization.blowup}
\(\tilde{X} \to X\) is the blowup along \(\ell\) with
exceptional divisor \(E_X\),
\item\label{hypersurfaces-unirationality-shioda-paramterization.insep}
\(E_X \to \PP W\) is purely inseparable of degree \(q\) and finite away from
\(D_2 \subset \PP W\), and
\item\label{hypersurfaces-unirationality-shioda-paramterization.rational}
the fibre product \(\tilde{X} \times_{\PP W} E_X\) is \(\kk\)-rational.
\end{enumerate}
Thus \(X\) admits a purely inseparable unirational parameterization
of degree \(q\) over \(\kk\).
\end{Proposition}

\begin{proof}
Let \(y \in \PP W\) be a point corresponding to a \(2\)-plane
\(\PP V' \subset \PP V\) containing the line \(\ell\). Then the fibre
\(E_{X,y}\) of the exceptional divisor is the intersection between
\(\ell\) and the degree \(q\)-curve \(\tilde{X}_y\) in \(\PP V'\). If
\(y \notin D_2\), then \(X'_y\) has type either \(\mathbf{N}_3\) or
\(\mathbf{1}^{\oplus 2} \oplus \mathbf{0}\) by
\parref{hypersurfaces-unirationality-shioda-fibres}, and this is a scheme of
length \(q\) supported on the unique singular point of \(\tilde{X}_y\).
This proves \ref{hypersurfaces-unirationality-shioda-paramterization.insep}
and implies \ref{hypersurfaces-unirationality-shioda-paramterization.rational}:
\(E_X\) is rational, being a projective bundle over the line \(\ell\), and the
generic fibre of \(\tilde{X} \times_{\PP W} E_X\) is a unicuspidal rational
curve.
\end{proof}

\subsection{Discriminant divisor}\label{hypersurfaces-unirational-discriminant}
Before discussing the next unirationality construction, continue with the
setting of \parref{hypersurfaces-unirational-shioda} and consider the divisor
\(D_1 \subset \PP W\) from \parref{hypersurfaces-unirationality-shioda-fibres}
over which the fibre type of \(X' \to \PP W\) changes; in analogy with the
theory of conic bundles, call this the \emph{discriminant divisor} of the
family of \(q\)-bic curves. Since \(D_1\) itself is a \(q^2\)-bic hypersurface,
the classification theorem \parref{forms-classification-theorem} applies.
The following determines the type of \(D_1\) in properties of the line
\(\ell\); since this is a geometric statement, the base field will be assumed
to be algebraically closed.

\begin{Lemma}\label{hypersurfaces-unirationality-discriminant-type}
Assume \(\kk = \bar{\kk}\) and let
\[
r \coloneqq
\mathrm{rank}\big(
  \beta \oplus \Fr^*(\beta^\vee) \colon
  U \oplus \Fr^{2,*}(U) \to \Fr^*(W)^\vee
\big).
\]
\begin{enumerate}
\item\label{hypersurfaces-unirationality-discriminant-type.two-hermitian}
If \(\#\ell \cap X_{\mathrm{Herm}} = q^2+1\), then \(r = 2\) and
\(\mathrm{type}(D_1) = \mathbf{0}^{\oplus n-3} \oplus \mathbf{1}^{\oplus 2}\).
\item\label{hypersurfaces-unirationality-discriminant-type.one-hermitian}
If \(\#\ell \cap X_{\mathrm{Herm}} = 1\), then \(r = 3\) and
\(\mathrm{type}(D_1) = \mathbf{0}^{\oplus n-4} \oplus \mathbf{N}_3\).
\item\label{hypersurfaces-unirationality-discriminant-type.zero-hermitian}
If \(\# \ell \cap X_{\mathrm{Herm}} = 0\), then \(3 \leq r \leq 4\) and
\[
\mathrm{type}(D_1) =
\begin{dcases*}
\mathbf{0}^{\oplus n-4} \oplus \mathbf{N}_2 \oplus \mathbf{1} &
when \(r = 3\), and \\
\mathbf{0}^{\oplus n-4} \oplus \mathbf{1}^{\oplus 3} &
when \(r = 4\).
\end{dcases*}
\]
\end{enumerate}
\end{Lemma}

\begin{proof}
Both maps \(\beta \colon U \to \Fr^*(W)^\vee\) and
\(\Fr^*(\beta^\vee) \colon \Fr^{2,*}(U) \to \Fr^*(W)^\vee\) are
injective since \(\beta\) is nonsingular, and so \(r \geq 2\). In fact,
\(r = 2\) if and only if \(\ell\) is Hermitian as in
\ref{hypersurfaces-unirationality-discriminant-type.two-hermitian}:
If \(\ell\) is Hermitian, choose a Hermitian basis
\(U = \langle u_0, u_1 \rangle\) and observe that
\[
\ker(\beta \oplus \Fr^*(\beta^\vee))
= \big\langle
(u_0,-u_0^{(q^2)}),
(u_1,-u_1^{(q^2)})
\big\rangle.
\]
Conversely, if \(r = 2\), then the first and second projections of the
direct sum restrict to isomorphisms between
\(\ker(\beta \oplus \Fr^*(\beta^\vee))\) and \(U\) and \(\Fr^{2,*}(U)\),
respectively. Therefore the \(q^2\)-linear map
\[
U \xrightarrow{\Fr^2}
\Fr^{2,*}(U) \xrightarrow{\pr_2^{-1}}
\ker(\beta \oplus \Fr^*(\beta^\vee)) \xrightarrow{\pr_1}
U
\]
is a bijection and \cite[Expos\'e XXII, 1.1]{SGAVII} gives a basis
\(U = \langle u_0,u_1 \rangle\) consisting of vectors fixed for this map.
Comparing with the definition of Hermitian vectors from
\parref{forms-hermitian} shows that \(u_0\) and \(u_1\) are Hermitian.

To determine the type of \(D_1\), observe that a choice of
basis \(U = \langle u_0, u_1 \rangle\) expresses the equation
\(\det(\Fr^*(\beta_{\mathcal{V}}) \oplus \beta_{\mathcal{V}})\) of \(D_1\)
constructed in \parref{hypersurfaces-unirationality-shioda-fibres} as
\[
\det\begin{pmatrix}
\beta(-,u_0)^q & \beta(-,u_1)^q \\
\beta(u_0^{(q)}, -) & \beta(u_1^{(q)}, -)
\end{pmatrix}
=
\beta(-,u_0)^q \beta(u_1^{(q)}, -) -
 \beta(u_0^{(q)}, -)\beta(-,u_1)^q
\]
where, for \(u \in U\), the linear forms \(\beta(-,u) \colon \Fr^*(W) \to \kk\)
and \(\beta(u^{(q)},-) \colon W \to \kk\) are identified with the corresponding
sections in
\[
\beta(-,u) \in \Fr^*(W)^\vee \subset \mathrm{H}^0(\PP W,\sO_{\PP W}(q))
\;\;\text{and}\;\;
\beta(u^{(q)},-) \in W^\vee = \mathrm{H}^0(\PP W, \sO_{\PP W}(1)).
\]
Now consider each case in turn:
\begin{enumerate}
\item Choose \(u_0\) and \(u_1\) to be Hermitian. Setting
\(y_i \coloneqq \beta(u_i^{(q)},-)\) then gives
\[
\beta(-,u_i) = \beta(u_i^{(q)},-)^q = y_i^q
\quad\text{whence}\quad
D_1 = \mathrm{V}(y_0^{q^2}y_1 - y_0y_1^{q^2})
\]
showing that it is of type
\(\mathbf{0}^{\oplus n-3} \oplus \mathbf{1}^{\oplus 2}\).
\item Choose \(u_0\) to be Hermitian. Then \((u_0,-u_0^{(q^2)})\) lies in
the kernel of \(\beta \oplus \Fr^*(\beta^\vee)\), so \(r = 3\). Thus there
are coordinates \(y_0\), \(y_1\), and \(y_2\) such that
\[
\beta(-,u_0) = \beta(u_0^{(q)},-)^q = y_0^q,
\quad
\beta(u_1^{(q)},-) = y_1,
\quad
\beta(-,u_1) = y_2^q.
\]
Therefore \(D_1 = \mathrm{V}(y_0^{q^2} y_1 - y_0 y_2^{q^2})\) and it has type
\(\mathbf{0}^{\oplus n-4} \oplus \mathbf{N}_3\).
\item If \(r = 3\), then there is a basis \(U = \langle u_0,u_1 \rangle\) such that
the vector \((u_0,-u_1^{(q^2)})\) lies in the kernel of
\(\beta \oplus \Fr^*(\beta^\vee)\). Thus there are coordinates
\(y_0\), \(y_1\), and \(y_2\) with
\[
\beta(-,u_0) =
\beta(u_1^{(q)},-) = y_0^q,
\quad
\beta(u_0^{(q)},-) = y_1,
\quad
\beta(-,u_1) = y_2^q.
\]
and so \(D_1 = \mathrm{V}(y_0^{q^2+1} - y_1 y_2^{q^2})\) has type
\(\mathbf{0}^{\oplus n-4} \oplus \mathbf{N}_2 \oplus \mathbf{1}\). Otherwise,
\(r = 4\) and there are coordinates \(y_0\), \(y_1\), \(y_2\), and \(y_3\)
so that
\[
\beta(-,u_0) = y_0^q,
\quad
\beta(u_1^{(q)},-) = y_1,
\quad
\beta(u_0^{(q)},-) = y_2,
\quad
\beta(-,u_1) = y_3^q.
\]
Then \(D_1 = \mathrm{V}(y_0^{q^2} y_1 - y_2 y_3^{q^2})\) and it is of type
\(\mathbf{0}^{\oplus n-4} \oplus \mathbf{1}^{\oplus 3}\). \qedhere
\end{enumerate}
\end{proof}

\subsection{}\label{hypersurfaces-unirational-discriminant-incidence}
The discriminant divisor \(D_1\) is related to the \((n-3)\)-dimensional
\(D_{1,\ell}\) scheme which parameterizes lines in \(X\) incident with
\(\ell\); see \parref{hypersurfaces-fano-correspondences-incidence} below.
There is a natural rational map
\[
D_{1,\ell} \dashrightarrow D_1,
\qquad
[\ell'] \mapsto \langle \ell,\ell' \rangle
\]
which sends a point corresponding to a line \(\ell' \subset X\) incident with
\(\ell\) to the plane spanned by \(\ell\) and \(\ell'\), viewed as a point of
\(\PP W = \PP(V/U)\); this factors through the discriminant divisor \(D_1\)
since such a plane intersects \(X\) at a \(q\)-bic curve containing at least
\(2\) lines, so it has type at most \(\mathbf{0} \oplus \mathbf{1}^{\oplus 2}\).
In the case that \(X\) is a threefold and \(\ell\) is a Hermitian, this extends
to a morphism; see \parref{threefolds-smooth-intermediate-jacobian-result}
and the comments that follow.

\subsection{Unirationality via tangents to a line}\label{hypersurfaces-unirational-tangent}
As in \parref{hypersurfaces-unirational-shioda}, let \(X\) be a smooth
\(q\)-bic hypersurface and \(\ell \subset X\) a \(\kk\)-rational line. The
following will produce another unirational parameterization of \(X\) which is
in a sense dual to that of
\parref{hypersurfaces-unirationality-shioda-parameterization}.

The projective bundle associated with the tangent bundle of \(X\) restricted to
\(\ell\) is canonically identified as the incidence correspondence
\[
Y \coloneqq
\PP(\mathcal{T}_X(-1)\rvert_\ell) =
\Set{(x,[\ell']) \in \ell \times \mathbf{G}(2,V) |
x \in \ell'\;\text{and}\;
\mathrm{mult}_x(X \cap \ell') \geq 2}
\]
between points \(x \in \ell\) and lines \(\ell' \subset \PP V\) that are
tangent to \(X\) at \(x\). Writing \(\pi \colon Y \to \ell\) for the structure
morphism, this description endows \(Y\) with tautological bundles fitting into
a short exact sequence
\[ 0 \to \pi^*\sO_\ell(-1) \to \mathcal{S} \to \sO_\pi(-1) \to 0, \]
where \(\mathcal{S}\) is the rank \(2\) subbundle of \(V_Y\) whose fibre over
a point \((x,[\ell'])\) is the \(2\)-dimensional linear subspace of \(V\)
underlying \(\ell'\). Then \(\beta\) induces a \(q\)-bic form
\[
\beta_{\mathcal{S}} \colon
\Fr^*(\mathcal{S}) \otimes \mathcal{S} \subset
\Fr^*(V)_Y \otimes V_Y \to
\sO_Y
\]
for which \(\pi^*\sO_\ell(-1)\) is an isotropic subbundle.

\begin{Lemma}\label{hypersurfaces-unirational-tangent-diagram}
There exists an exact commutative diagram of \(\sO_Y\)-modules
\[
\begin{tikzcd}
0 \rar
& \pi^*\sO_\ell(-1) \rar \dar[hook]
& \mathcal{S} \rar \dar["\beta_\mathcal{S}"]
& \sO_\pi(-1) \rar \dar["0"]
& 0 \\
0 \rar
& \sO_\pi(q) \rar
& \Fr^*(\mathcal{S})^\vee \rar
& \pi^*\sO_\ell(q) \rar
& 0
\end{tikzcd}
\]
where the left map is injective and the right map is zero. This induces an
exact sequence
\[
0 \to
\Fr^*(\mathcal{S})^\perp \to
\sO_\pi(-1) \to
\sO_Z \to
\coker(\beta_{\mathcal{S}}) \to
\pi^*\sO_\ell(q) \to
0
\]
where \(Z\) is the effective Cartier divisor defined by the map
\(\pi^*\sO_\ell(-1) \to \sO_\pi(q)\).
\end{Lemma}

\begin{proof}
The diagram exists because \(\pi^*\sO_\ell(-1)\) is isotropic for
\(\beta_{\mathcal{S}}\). Let \(y = (x,[\ell'])\) be a point of \(Y\).
Then \(X \cap \ell'\) is the scheme of \(q\)-bic points in \(\ell' = \PP\mathcal{S}_y\)
defined by \(\beta_{\mathcal{S},y}\), see \parref{hypersurface-hyperplane-section}.
Since \(x\) is a point of tangency between \(\ell'\) and \(X\) by
\parref{hypersurfaces-unirational-tangent}, it is the singular point of
\(X \cap \ell'\), so \(\pi^*\sO_\ell(-q) \subseteq \mathcal{S}^\perp\) by
\parref{hypersurfaces-nonsmooth-locus}. This implies that
the right vertical map vanishes.

For the left vertical map, observe that \(X \cap \ell'\) is of type
\(\mathbf{N}_2\) for general \(y\): The plane slice
\(X \cap \langle \ell, \ell' \rangle\) is of type \(\mathbf{N}_3\) by
\parref{hypersurfaces-unirationality-shioda-fibres} with singular point \(x\)
and linear component \(\ell\); since \(x\) has multiplicity \(q\),
\(\ell'\) will intersect the degree \(q\) component at one point besides
\(x\). Therefore \(\Fr^*(\mathcal{S})^\perp\) is generically disjoint from
\(\pi^*\sO_\ell(-1)\), giving the desired injectivity. The exact sequence now
arises from the Snake Lemma.
\end{proof}

By \parref{hypersurfaces-unirational-tangent-diagram},
\(\Fr^*(\mathcal{S})^\perp\) is torsion-free of rank \(1\). The exact sequence
\[
0 \to
\Fr^*(\mathcal{S})^\perp \to
\mathcal{S} \xrightarrow{\beta_{\mathcal{S}}}
\Fr^*(\mathcal{S})^\vee \to
\coker(\beta_{\mathcal{S}}) \to
0
\]
means that it is a subbundle of \(\mathcal{S}\) away from the vanishing locus
of \(\beta_{\mathcal{S}}\): this is the scheme
\(Y \times_{\mathbf{G}(2,V)} \mathbf{F}_1(X)\) whose points are those
\((x,[\ell'])\) such that \(\ell' \subset X\). Let \(Y^\circ\) be the open
complement in \(Y\), so that \(\Fr^*(\mathcal{S})^\perp\) induces a morphism
\(Y^\circ \to \PP V\). Since \(\Fr^*(\mathcal{S})^\perp\) is isotropic for
\(\beta\), this morphism factors through \(X \subset \PP V\).

\begin{Lemma}\label{hypersurfaces-unirational-tangent-dominant}
\(Y^\circ \to X\) is dominant, purely inseparable, and finite of degree \(q\).
\end{Lemma}

\begin{proof}
By \parref{hypersurfaces-unirationality-shioda-fibres}, the union of lines
in \(X\) incident with \(\ell\) forms a divisor. Let \(z\) be a point of
the complement. Then the plane \(\langle \ell, z \rangle\) intersects \(X\) at
a \(q\)-bic curve of type \(\mathbf{N}_3\) and \(z\) lies on the residual curve
\(C\) of degree \(q\). Let \(x\) be the point supporting \(C \cap \ell\),
and let \(\ell' \coloneqq \langle x,z \rangle\). Then \((x,[\ell'])\) is
a preimage of \(z\) under \(Y^\circ \to X\), showing that it is dominant. In
fact, the scheme-theoretic fibre over \(z\) is the scheme
\((C \cap \ell, [\ell'])\), showing that \(Y^\circ \to X\) is finite and purely
inseparable of degree \(q\).
\end{proof}

The following summarizes the construction. Since a projective bundle over a
line is rational, this provides another unirational parameterization of \(X\).

\begin{Proposition}\label{hypersurfaces-unirational-tangent-morphism}
Let \(X\) be a smooth \(q\)-bic hypersurface of dimension at least
\(2\). For every \(\kk\)-rational line \(\ell \subset X\), the rational map
\[
\PP(\mathcal{T}_X(-1)\rvert_\ell) \dashrightarrow X,
\qquad
(x,[\ell']) \mapsto X \cap \ell' - qx
\]
is dominant, purely inseparable, and is defined and finite of degree \(q\)
away from
\[
\pushQED{\qed}
\PP(\mathcal{T}_X(-1)\rvert_\ell) \times_{\mathbf{G}(2,V)} \mathbf{F}_1(X)
= \Set{(x,[\ell']) \in \PP(\mathcal{T}_X(-1)\rvert_\ell) | \ell' \subset X}.
\qedhere
\popQED
\]
\end{Proposition}

\section{Cohomological properties}\label{section-hypersurfaces-cohomological}
This Section collects some cohomological facts about \(q\)-bic hypersurfaces
\(X\) in the projective \(n\)-space \(\PP V\).

\subsection{Frobenius action on cohomology}\label{hypersurfaces-cohomology-frobenius-hypersurfaces}
Let \(Y\) be a scheme over \(\kk\) and let \(\Fr\) be the \(q\)-power absolute
Frobenius morphism. Pullback induces \(q\)-linear maps  on cohomology:
\[
\Fr^* \colon \mathrm{H}^i(Y,\sO_Y) \to \mathrm{H}^i(Y,\sO_Y)
\quad\text{for every}\;i \geq 0.
\]
This can be made explicit in the case \(Y\) is a hypersurface in a projective
space \(\PP V\): Let \(g \in \mathrm{H}^0(\PP V, \sO_{\PP V}(d))\) be an
equation for \(Y\). Then there is a commutative diagram of abelian sheaves on
\(\PP V\) given by;
\[
\begin{tikzcd}
0 \rar
& \sO_{\PP V}(-d) \rar["g"'] \dar["g^{q-1} \Fr"]
& \sO_{\PP V} \rar \dar["\Fr"]
& \sO_Y \rar \dar["\Fr"]
& 0 \\
0 \rar
& \Fr_*(\sO_{\PP V}(-d)) \rar["g^q"]
& \Fr_*\sO_{\PP V} \rar
& \Fr_*\sO_Y \rar
& 0
\end{tikzcd}
\]
where the sheaf map \(\Fr\) acts on local sections by \(s \mapsto s^q\). Taking
cohomology then yields a commutative diagram
\[
\begin{tikzcd}
\mathrm{H}^{n-1}(Y,\sO_Y) \rar["\cong"'] \dar["\Fr"']
& \mathrm{H}^n(\PP V,\sO_{\PP V}(-d)) \dar["g^{q-1} \Fr"] \\
\mathrm{H}^{n-1}(Y,\sO_Y) \rar["\cong"]
& \mathrm{H}^n(\PP V,\sO_{\PP V}(-d))\punct{.}
\end{tikzcd}
\]
The following shows that the \(q\)-power Frobenius always acts trivially on the
coherent cohomology of a \(q\)-bic hypersurface:

\begin{Lemma}\label{hypersurfaces-cohomology-zero-frobenius}
Let \(X\) be a \(q\)-bic hypersurface of dimension at least \(1\). Then the map
\(
\Fr \colon
\mathrm{H}^{n-1}(X,\sO_X) \to
\mathrm{H}^{n-1}(X,\sO_X)
\)
induced by \(q\)-power Frobenius is zero.
\end{Lemma}

\begin{proof}
Let \(f\) be an equation for \(X\). Then the discussion of
\parref{hypersurfaces-cohomology-frobenius-hypersurfaces} shows that the
Frobenius map on \(\mathrm{H}^{n-1}(X,\sO_X)\) is computed as \(f^{q-1}\Fr\)
applied to \(\mathrm{H}^n(\PP V,\sO_{\PP V}(-q-1))\). Upon choosing coordinates
\((x_0:\cdots:x_n)\) for \(\PP V = \PP^n\), a basis of the latter space is
given by elements
\[
\xi \coloneqq (x_0^{i_0} \cdots x_n^{i_n})^{-1}
\quad\text{where}\;
i_0 + \cdots + i_n = q+1\;\text{and each}\;
i_j \geq 1.
\]
Monomials appearing in \(f^{q-1}\) are of the form \(a^q b\), where \(a\) and
\(b\) are themselves monomials of degree \(q-1\). Writing
\(a = x_0^{a_0} \cdots x_n^{a_n}\) with \(a_0 + \cdots + a_n = q-1\), it follows
that \(f^{q-1}\Fr(\xi)\) is a sum of terms
\[
a^q b \cdot (x_0^{i_0} \cdots x_n^{i_n})^{-q} =
b \cdot (x_0^{i_0 - a_0} \cdots x_n^{i_n - a_n})^{-q}.
\]
But now \((i_0 - a_0) + \cdots + (i_n - a_n) = 2 \leq n\). Therefore there
is some index \(j\) such that \(i_j - a_j \leq 0\) so the above term represents
\(0\) in \(\mathrm{H}^n(\PP^n,\sO_{\PP^n}(-q-1))\). Thus each potential
contribution to \(f^{q-1}\Fr(\xi)\) vanishes, and so \(f^{q-1}\Fr(\xi) = 0\).
\end{proof}

\subsection{\'Etale cohomology}\label{hypersurfaces-smooth-etale-setting}
Assume the base field \(\kk\) is separably closed, let \(X\) be a smooth
\(q\)-bic hypersurface, and fix a prime \(\ell \neq p\). For each
\(0 \leq i \leq 2n-2\), let
\[
\mathrm{H}^i_{\mathrm{\acute{e}t}}(X,\mathbf{Z}_\ell)_{\mathrm{prim}}
\coloneqq
\ker\big(
  c_1(\sO_X(1))^{n+1-i} \colon
  \mathrm{H}^i_{\mathrm{\acute{e}t}}(X,\mathbf{Z}_\ell) \to
  \mathrm{H}^{2n+2-i}_{\mathrm{\acute{e}t}}(X,\mathbf{Z}_\ell)
\big)
\]
be the primitive \'etale cohomology with respect to the action of the natural
hyperplane class. Since \(X\) is a hypersurface, the Lefschetz hyperplane
theorem shows that \(\mathrm{H}^i_{\mathrm{\acute{e}t}}(X,\mathbf{Z}_\ell) = 0\)
for all \(i \neq n-1\), and the universal coefficient theorem shows that the
middle cohomology is torsion-free. Each of the primitive cohomology groups are
representations for \(\mathrm{U}_{n+1}(q) = \AutSch(V,\beta)\) since the group
acts on \(X\) through a linear action on \(\PP V\), see
\parref{hypersurfaces-automorphisms-linear}. The middle cohomology has the
following properties:

\begin{Theorem}\label{hypersurfaces-smooth-etale}
Let \(X\) be a smooth \(q\)-bic \((n-1)\)-fold over a separably closed field.
Then its middle \'etale cohomology group satisfies:
\begin{enumerate}
\item\label{hypersurfaces-smooth-etale.betti}
\(\rank_{\mathbf{Z}_\ell}
\mathrm{H}^{n-1}_{\mathrm{\acute{e}t}}(X,\mathbf{Z}_\ell)_{\mathrm{prim}}
= q(q^n - (-1)^n)/(q+1)\);
\item\label{hypersurfaces-smooth-etale.irrep}
\(\mathrm{H}_{\mathrm{\acute{e}t}}^{n-1}(X,\mathbf{Q}_\ell)_{\mathrm{prim}}\)
is an irreducible representation of \(\mathrm{U}_{n+1}(q)\); and
\item\label{hypersurfaces-smooth-etale.cycles}
if \(n-1\) is even, then \(\mathrm{H}^{n-1}_{\mathrm{\acute{e}t}}(X,\mathbf{Q}_\ell)\)
is spanned by algebraic cycles.
\end{enumerate}
\end{Theorem}

\begin{proof}
Item \ref{hypersurfaces-smooth-etale.betti} is standard and may be obtained
via an Euler characteristic computation, see \cite[p.246]{Milne:EC};
\ref{hypersurfaces-smooth-etale.irrep} was first observed by
\cite{Tate:Conjecture}, but see also \cite[Theorem 1]{HM:TT-Lemma}; this
implies \ref{hypersurfaces-smooth-etale.cycles} upon noting that the cycle
classes of maximal linear spaces in \(X\) are not restricted from \(\PP V\),
but see also \cite[Theorem II]{SK:Fermat}.
\end{proof}

Since all smooth \(q\)-bic hypersurfaces are isomorphic to the Fermat
hypersurface by \parref{hypersurfaces-smooth-fermat-model}, the rank of the
middle cohomology may be determined by counting points:

\begin{Theorem}\label{hypersurfaces-smooth-zeta}
The zeta function of the Fermat \(q\)-bic \((n-1)\)-fold over
\(\mathbf{F}_{q^2}\) is
\[
Z(X;t) =
\frac{(1-q^{n-1}t)^{(-1)^n b_{n-1, \mathrm{prim}}}}{(1-t) \cdots (1 - q^{2n-2}t)}
\]
where \(b_{n-1, \mathrm{prim}} \coloneqq q(q^n-(-1)^n)/(q+1)\).
\end{Theorem}

\begin{proof}
This is essentially due to Weil in \cite{Weil:Fermat,Weil:Grossenchar};
see also \cite[\S3]{SK:Fermat}.
\end{proof}

In particular, this means that the Jacobian of a smooth \(q\)-bic curve
is supersingular; see also \cite[Proposition 3.10]{SK:Fermat}.

\section{Linear spaces}\label{section-hypersurfaces-linear-spaces}
Recognizing a \(q\)-bic hypersurface \(X\) the space of isotropic vectors for a
\(q\)-bic form \(\beta\), as in
\parref{hypersurfaces-moduli-of-isotropic-vectors}, endows the Fano schemes
\(\mathbf{F}_r(X)\) parameterizing \(r\)-planes contained in \(X\) with the
alternative moduli interpretation as the space of \((r+1)\)-dimensional
isotropic subspaces of \(V\). This likens the Fano schemes to orthogonal
Grassmannian, bringing a perspective from which their special properties
may be better understood. Throughout this Section, \(X\) is the \(q\)-bic
hypersurface associated with a \(q\)-bic form \((V,\beta)\) of
dimension \(n+1\) over a field \(\kk\).

\subsection{Fano schemes}\label{hypersurfaces-linear-spaces-fano-schemes}
For the generalities of the next two paragraphs, see, for example,
\cite{AK:Fano} or \cite[Section V.4]{Kollar:RationalCurves}
for details. Given a projective scheme \(Y \subseteq \PP V\) and
an integer \(0 \leq r \leq n\), the \emph{Fano scheme of \(r\)-planes in \(Y\)}
is the closed subscheme \(\mathbf{F}_r(Y) \subseteq \mathbf{G}(r+1,V)\)
representing the functor
\(\mathrm{Sch}_\kk^{\mathrm{opp}} \to \mathrm{Set}\) given by
\[
T \mapsto
\Set{P \subseteq Y \times_\kk T |
\begin{array}{c}
\text{\(P\) flat over \(T\) such that, for all \(t \in T\),}\\
\text{\(P_t \subseteq \PP V \otimes_\kk \kappa(t)\) is an \(r\)-plane}
\end{array}}.
\]
Since \(r\)-planes in a projective space are precisely projectivizations of
linear subspaces of dimension \(r+1\), there is a canonical identification
\(\mathbf{F}_r(\PP V) = \mathbf{G}(r+1,V)\) and the tautological short exact
sequence of the Grassmannian
\[ 0 \to \mathcal{S} \to V_{\mathbf{G}(r+1,V)} \to \mathcal{Q} \to 0 \]
is such that the fibre of the universal subbundle \(\mathcal{S}\) at a point
\([P] \in \mathbf{F}_r(\PP V)\) is the linear subspace of \(V\) underlying
\(P \subset \PP V\).

\subsection{}\label{hypersurfaces-linear-spaces-fano-schemes-hypersurfaces}
Let \(Y \subset \PP V\) be a hypersurface of degree \(d\), say defined by a
section
\[ g \in \mathrm{H}^0(\PP V,\sO_{\PP V}(d)) = \Sym^d(V^\vee). \]
Then \(\mathbf{F}_r(Y)\) is the closed subscheme of \(\mathbf{G}(r+1,V)\)
consisting of \(r\)-planes on which the restriction of \(g\) vanishes.
Thus \(\mathbf{F}_r(Y)\) is the zero locus of the section
\[
\sO_{\mathbf{G}(r+1,V)} \xrightarrow{g}
\sO_{\mathbf{G}(r+1,V)} \otimes \Sym^d(V^\vee) \to
\Sym^d(\mathcal{S}^\vee).
\]
where the latter map comes from the \(d\)-th symmetric power of the dual
tautological sequence above. In particular, if \(\mathbf{F}_r(Y)\) is nonempty,
then
\[
\dim\mathbf{F}_r(Y) \geq
\dim\mathbf{G}(r+1,V) - \rank_{\sO_{\mathbf{G}(r+1,V)}}\Sym^d(\mathcal{S}^\vee)
= (n-r)(r+1) - \binom{d+r}{r}.
\]
In fact, equality holds for general \(Y\), see \cite[Th\'eor\`eme 2.1]{DM:Fano}.

In contrast, it has long been observed that \(q\)-bic hypersurfaces contain
many more linear subspaces and that this dimension bound is typically a gross
underestimate: see \cite[Example 1.27]{Collino} and
\cite[pp.51--52]{Debarre:HDAG}. One way to resolve this conundrum is to
recognize \(q\)-bic hypersurfaces as the moduli spaces of isotropic
vectors for \(q\)-bic forms, as in
\parref{hypersurfaces-moduli-of-isotropic-vectors}. Then the associated Fano
schemes take on the alternative interpretation as moduli spaces of
isotropic subspaces in \(V\):

\begin{Lemma}\label{hypersurfaces-equations-of-fano}
Let \(X\) be the \(q\)-bic hypersurface associated with a \(q\)-bic form \((V,\beta)\).
Then its Fano scheme \(\mathbf{F}_r(X)\) represents the functor
\(\mathrm{Sch}^{\mathrm{opp}}_\kk \to \mathrm{Set}\) given by
\[
T \mapsto
\Set{\mathcal{V}' \subset V_T \;\text{a \(\beta\)-isotropic subbundle of rank \(r+1\)}}.
\]
Thus \(\mathbf{F}_r(X)\) is the vanishing locus in
\(\mathbf{G}(r+1,V)\) of the \(q\)-bic form
\[
\beta_{\mathcal{S}} \colon \Fr^*(\mathcal{S}) \otimes \mathcal{S} \subset
\Fr^*(V)_{\mathbf{G}(r+1,V)} \otimes V_{\mathbf{G}(r+1,V)} \xrightarrow{\beta}
\sO_{\mathbf{G}(r+1,V)}.
\]
and \(\dim\mathbf{F}_r(X) \geq (r+1)(n-2r-1)\) whenever it is nonempty.
\end{Lemma}

\begin{proof}
By its definition in \parref{hypersurfaces-linear-spaces-fano-schemes},
\(\mathbf{F}_r(X)\) represents the presheaf on \(\mathrm{Sch}_\kk\)
sending a \(\kk\)-scheme \(T\) to the set of rank \(r+1\) subbundles
\(\mathcal{V}' \subset V_T\) such that
\(\PP\mathcal{V}' \subset X_T \subset \PP V_T\). But
\parref{hypersurfaces-moduli-of-isotropic-vectors} and
\parref{forms-notions-of-isotropicity} together imply that \(\mathcal{V}'\) is
a totally isotropic subbundle for \(\beta_T\), yielding the first statement.
That \(\mathbf{F}_r(X)\) is the zero locus of \(\beta_{\mathcal{S}}\) follows
from universal property of the Grassmannian, see
\parref{bundles-grassmannian-subbundles}. Finally,
\begin{align*}
\dim\mathbf{F}_r(X)
& \geq
\dim\mathbf{G}(r+1,V) -
\rank_{\sO_{\mathbf{G}(r+1,V)}}(\Fr^*(\mathcal{S}) \otimes \mathcal{S}) \\
& = (r+1)(n-r) - (r+1)^2
= (r+1)(n-2r-1).
\qedhere
\end{align*}
\end{proof}

The following verifies that the Fano schemes are nonempty in a certain range.

\begin{Lemma}\label{hypersurfaces-nonempty-fano}
The Fano scheme \(\mathbf{F}_r(X)\) is nonempty for each \(0 < r < \frac{n}{2}\).
\end{Lemma}

\begin{proof}
This is a geometric question, so assume \(\kk\) is algebraically closed.
Fix \(0 < r < \frac{n}{2}\). Let \(\PP(\Fr^*(V)^\vee \otimes V^\vee)\) denote the
parameter space of \(q\)-bic hypersurfaces in \(\PP V\), and consider the
incidence correspondence
\[
\Phi \coloneqq
\Set{([\PP U], [X])
\in \mathbf{G}(r+1,V) \times \PP(\Fr^*(V)^\vee \otimes V^\vee) |
\PP U \subseteq X}.
\]
The fibre of the second projection
\(\pr_2 \colon \Phi \to \PP(\Fr^*(V)^\vee \otimes V^\vee)\) over a point
\([X]\) is \(\mathbf{F}_r(X)\), so the goal is to show \(\pr_2\)
is surjective. Note \(\Phi\) is proper since \(\Gr(r+1,V)\) is
proper, and the fibre of the first projection
\(\pr_1 \colon \Phi \to \Gr(r+1,V)\) over a point \([\PP U]\) is the projective
space
\[
\pr_1^{-1}([\PP U]) =
\PP\big(\ker(\Fr^*(V)^\vee \otimes V^\vee \to \Fr^*(U)^\vee \otimes U^\vee)\big)
\]
of \(q\)-bic equations vanishing on \(U\). Thus it suffices to show that
\(\pr_2\) is dominant, and this will follow if \(\mathbf{F}_r(X)\)
is nonempty for every smooth \(X\). A smooth \(X\) is defined by a nonsingular
\(q\)-bic form by \parref{hypersurfaces-smooth-and-nondegeneracy}, so this
follows from \parref{forms-hermitian-maximal-isotropic}, which implies that a
smooth \(X\) contains a Hermitian linear subspace of dimension \(\lfloor
\frac{n-1}{2} \rfloor\).
\end{proof}

Since \(\mathbf{F}_r(X)\) is a Hilbert scheme, its tangent space to a point
parameterizing a linear subspace \(\PP U \subset X\) is canonically identified
as \(\mathrm{H}^0(\PP U, \mathcal{N}_{\PP U/X})\): see, for example,
\cite[Proposition 6.5.2]{FGAExplained}. The next Proposition
explicitly identifies the normal bundle of a linear subspace contained in the
smooth locus of \(X\). This will show, in particular, that such linear
spaces give smooth points of \(\mathbf{F}_r(X)\) around which the Fano scheme
has the expected dimension. Toward this, the next, rather technical, Lemma
reformulates the geometric property that a linear subspace \(\PP U \subset \PP
V\) is disjoint from the singular locus of \(X\) in terms of linear algebraic
notions. For part of its statement, set \(W \coloneqq V/U\) and note that
\(\beta\) induces bilinear pairings
\[
\beta_W \colon U^\perp \otimes W \to \kk
\quad\text{and}\quad
\beta_{\Fr^*(W)} \colon \Fr^*(W) \otimes \Fr^*(U)^\perp \to \kk
\]
computed by taking any lift along \(V \to W\) and \(\Fr^*(V) \to \Fr^*(W)\),
respectively.

\begin{Lemma}\label{hypersurfaces-linear-subspaces-smooth-locus}
Let \(\PP U \subset \PP V\) be a linear subspace. The following are equivalent:
\begin{enumerate}
\item\label{hypersurfaces-linear-subspaces-smooth-locus.disjoint-X}
\(\PP U\) is disjoint from the nonsmooth locus of \(X\).
\item\label{hypersurfaces-linear-subspaces-smooth-locus.disjoint-space}
\(\Fr^*(U)\) is linearly disjoint from \(V^\perp\).
\item\label{hypersurfaces-linear-subspaces-smooth-locus.injective}
The map \(\beta^\vee \colon \Fr^*(U) \to V^\vee\) is injective.
\item\label{hypersurfaces-linear-subspaces-smooth-locus.surjective}
The map \((-)\rvert_{\Fr^*(U)} \circ \beta \colon V \to \Fr^*(V)^\vee \to \Fr^*(U)^\vee\)
is surjective.
\item\label{hypersurfaces-linear-subspaces-smooth-locus.dim}
\(\dim_\kk\Fr^*(U)^\perp = \dim_\kk V - \dim_\kk U\).
\end{enumerate}
Suppose \(U\) is furthermore isotropic for \(\beta\). Set \(W \coloneqq V/U\).
Then these are equivalent to:
\begin{enumerate}
\setcounter{enumi}{5}
\item\label{hypersurfaces-linear-subspaces-smooth-locus.isotropic-surjective}
The map
\((-)\rvert_{\Fr^*(U)} \circ \beta_W \colon W \to U^{\perp,\vee} \to \Fr^*(U)^\vee\)
is surjective.
\item\label{hypersurfaces-linear-subspaces-smooth-locus.isotropic-dim}
\(\dim_\kk(\Fr^*(U)^\perp/U) = \dim_\kk V - 2 \dim_\kk U\).
\end{enumerate}
\end{Lemma}

\begin{proof}
That \ref{hypersurfaces-linear-subspaces-smooth-locus.disjoint-X} \(\Leftrightarrow\)
\ref{hypersurfaces-linear-subspaces-smooth-locus.disjoint-space} follows
from \parref{hypersurfaces-nonsmooth-locus}; their equivalence with
\ref{hypersurfaces-linear-subspaces-smooth-locus.injective} then follows
from definitions. That
\ref{hypersurfaces-linear-subspaces-smooth-locus.injective} \(\Leftrightarrow\)
\ref{hypersurfaces-linear-subspaces-smooth-locus.surjective}
follows by linear duality. Finally,
\ref{hypersurfaces-linear-subspaces-smooth-locus.surjective} \(\Leftrightarrow\)
\ref{hypersurfaces-linear-subspaces-smooth-locus.dim} since
\(\Fr^*(U)^\perp\) is precisely the kernel of the map \(\beta \colon V \to \Fr^*(U)^\vee\).

Now assume that \(U\) is isotropic for \(\beta\). Then there is a commutative
diagram
\[
\begin{tikzcd}
V \dar[two heads] \rar["\beta"']
& \Fr^*(V)^\vee \dar[two heads] \rar[two heads]
& \Fr^*(U)^\vee \dar[equal] \\
W \rar["\beta_W"]
& U^{\perp,\vee} \rar[two heads]
& \Fr^*(U)^\vee
\end{tikzcd}
\]
where the maps in the right square are the restriction maps. The top composite
is the map of \ref{hypersurfaces-linear-subspaces-smooth-locus.surjective}, and
the bottom composite that of
\ref{hypersurfaces-linear-subspaces-smooth-locus.isotropic-surjective}.
Since \(V \to W\) is surjective, this shows
\ref{hypersurfaces-linear-subspaces-smooth-locus.surjective}
\(\Leftrightarrow\)
\ref{hypersurfaces-linear-subspaces-smooth-locus.isotropic-surjective}.
The kernel of the lower map is \(\Fr^*(U)^\perp/U\), so
\ref{hypersurfaces-linear-subspaces-smooth-locus.isotropic-surjective} \(\Leftrightarrow\)
\ref{hypersurfaces-linear-subspaces-smooth-locus.isotropic-dim}.
\end{proof}

This gives the normal bundle of a linear subspace wholly contained in the
smooth locus of a \(q\)-bic hypersurface.

\begin{Proposition}\label{hypersurfaces-linear-subspace-normal-bundle}
Let \(\PP^r \cong \PP U \subset \PP V\) be a linear subspace contained
in the smooth locus of \(X\). Then there is a canonical, split short exact sequence
\[
0 \to (\Fr^*(U)^\perp/U)_{\PP U} \to
\mathcal{N}_{\PP U/X}(-1) \to
\Fr^*(\Omega_{\PP U}^1(1)) \to 0.
\]
This yields a canonical identification
\[
\mathcal{T}_{\mathbf{F}_r(X)} \otimes_{\sO_{\mathbf{F}_r(X)}} \kappa([\PP U]) \cong
\mathrm{H}^0(\PP U,\mathcal{N}_{\PP U/X}) \cong U^\vee \otimes (\Fr^*(U)^\perp/U)
\]
and this has dimension \((r+1)(n-2r-1)\).
\end{Proposition}

\begin{proof}
Since \(\PP U\) is contained in the smooth locus of \(X\), there is a short
exact sequence of normal bundles
\[
0 \to
\mathcal{N}_{\PP U/X}(-1) \to
\mathcal{N}_{\PP U/\PP V}(-1) \xrightarrow{\delta}
\mathcal{N}_{X/\PP V}(-1)\rvert_{\PP U} \to
0.
\]
The rightmost sheaf is \(\sO_{\PP U}(q)\);
the Euler sequence together with \parref{hypersurfaces-conormal-compute} shows that
the map \(\delta\) fits into a commutative diagram
\[
\begin{tikzcd}
V_{\PP U} \dar[two heads] \rar["\beta"'] &
\Fr^*(V)^\vee_{\PP U} \dar["\mathrm{eu}^{(q),\vee}", two heads] \\
\mathcal{N}_{\PP U/\PP V}(-1) \rar["\delta"] &
\mathcal{N}_{X/\PP V}(-1)\rvert_{\PP U}\punct{.}
\end{tikzcd}
\]
Since \(\mathcal{N}_{\PP U/\PP V}(-1) \cong W_{\PP U}\) and the Euler section
\(\mathrm{eu}^{(q),\vee}\) factors through the
quotient \(\Fr^*(V)_{\PP U}^\vee \twoheadrightarrow \Fr^*(U)_{\PP U}^\vee\),
the map \(\delta\) may be identified as
\[
\delta
\colon W_{\PP U} \xrightarrow{\beta_W}
\Fr^*(U)^\vee_{\PP U} \xrightarrow{\mathrm{eu}^{(q),\vee}}
\sO_{\PP U}(q)
\]
where \(\beta_W\) is as in
\parref{hypersurfaces-linear-subspaces-smooth-locus}\ref{hypersurfaces-linear-subspaces-smooth-locus.surjective};
since \(\PP U\) is contained in the smooth locus of \(X\), the map \(\beta_W\) is
surjective with kernel \(\Fr^*(U)^\perp/U\). Thus the normal bundle sequence
above fits into an exact commutative diagram
\[
\begin{tikzcd}
& (\Fr^*(U)^\perp/U)_{\PP U} \rar[equal] \dar[hook]
& (\Fr^*(U)^\perp/U)_{\PP U} \dar[hook] \\
0 \rar
& \mathcal{N}_{\PP U/X}(-1) \rar \dar[two heads]
& W_{\PP U} \rar["\delta"] \dar[two heads, "\beta_W"]
& \sO_{\PP U}(q) \rar \dar[equal]
& 0 \\
0 \rar
& \Fr^*(\Omega^1_{\PP U}(1)) \rar
& \Fr^*(U)^\vee_{\PP U} \rar["\mathrm{eu}^{(q),\vee}"]
& \Fr^*(\sO_{\PP U}(1)) \rar
& 0
\end{tikzcd}
\]
where the lower row is identified as the Frobenius pullback of the Euler
sequence on \(\PP U\). The left column is the sequenced promised by the Lemma.
That the left column splits is because the middle column is splits, being a
sequence of free \(\sO_{\PP U}\)-modules.

To compute global sections of \(\mathcal{N}_{\PP U/X}\), note that the Euler
sequence implies that
\(\mathrm{H}^0(\PP U,\Fr^*(\Omega_{\PP U}^1(1))(1)) = 0\). Thus the sequence
of the Lemma shows
\begin{align*}
\mathrm{H}^0(\PP U,\mathcal{N}_{\PP U/X}) \cong
\mathrm{H}^0(\PP U, \sO_{\PP U}(1) \otimes (\Fr^*(U)^\perp/U)) \cong
U^\vee \otimes (\Fr^*(U)^\perp/U).
\end{align*}
That this has dimension
\((r+1)(n - 2r - 1)\) follows from
\parref{hypersurfaces-linear-subspaces-smooth-locus}\ref{hypersurfaces-linear-subspaces-smooth-locus.isotropic-dim}.
\end{proof}

The identification of the tangent space to \(\mathbf{F}_r(X)\) in
\parref{hypersurfaces-linear-subspace-normal-bundle} can also be done
using the alternative functor underlying the Fano scheme given in
\parref{hypersurfaces-equations-of-fano}. The computation given here, however,
has the advantage of determining the normal bundle of a linear space in
the smooth locus of \(X\).

\begin{Corollary}\label{hypersurfaces-smooth-point-fano}
Let \(\PP^r \cong \PP U \subset \PP V\) be a linear subspace
contained in the smooth locus of \(X\). Then \(\mathbf{F}_r(X)\) is smooth
at the point \([\PP U]\) and has local dimension
\[ \dim_{[\PP U]}(\mathbf{F}_r(X)) = (r+1)(n-2r-1). \]
In particular, \(\Sing(\mathbf{F}_r(X)) \subseteq \Set{[\PP U] \in \mathbf{F}_r(X) | \PP U \cap \Sing(X) \neq \varnothing}\).
\end{Corollary}

\begin{proof}
Let \(\kappa \coloneqq \kappa([\PP U])\) be the residue field at
\([\PP U] \in \mathbf{F}_r(X)\). Then
\begin{align*}
(r + 1)(n-2r-1)
& \leq \dim_{[\PP U]}(\mathbf{F}_r(X)) \\
& \leq \dim_\kappa(\mathcal{T}_{\mathbf{F}_r(X)} \otimes_{\sO_{\mathbf{F}_r(X)}} \kappa)
= (r + 1)(n-2r-1)
\end{align*}
where the lower bound follows from \parref{hypersurfaces-equations-of-fano}
together with \parref{hypersurfaces-nonempty-fano}, and the
right equality follows from \parref{hypersurfaces-linear-subspace-normal-bundle}.
Thus equality holds throughout and \(\mathbf{F}_r(X)\) is smooth at \([\PP U]\).
\end{proof}

\subsection{Koszul resolution}\label{hypersurfaces-fano-koszul}
As explained in \parref{hypersurfaces-equations-of-fano},
the Fano scheme \(\mathbf{F}_r(X)\) is cut out by the morphism of locally free
\(\sO_{\mathbf{G}(r+1,V)}\)-modules
\[
\beta_{\mathcal{S}} \colon
\Fr^*(\mathcal{S}) \otimes \mathcal{S} \to
\sO_{\mathbf{G}(r+1,V)}.
\]
The associated Koszul complex
\[
\Kosz_\bullet(\beta_{\mathcal{S}})
\coloneqq \big[
\wedge^{(r+1)^2} (\Fr^*(\mathcal{S}) \otimes \mathcal{S}) \to \cdots \to
\Fr^*(\mathcal{S}) \otimes \mathcal{S} \xrightarrow{\beta_{\mathcal{S}}}
\sO_{\mathbf{G}(r+1,V)}\big]
\]
is exact in positive homological degrees away from \(\mathbf{F}_r(X)\).
Since \(\mathbf{G}(r+1,V)\) is regular, it is, in particular, Cohen--Macaulay
and the complex will be furthermore form a resolution of
\(\sO_{\mathbf{F}_r(X)}\) as an \(\sO_{\mathbf{G}(r+1,V)}\)-module when
\(\mathbf{F}_r(X)\) is of expected dimension \((r+1)(n-2r-1)\). See, for
example, \cite[p.320]{Lazarsfeld:PositivityI}.

A convenient sufficient condition for when the Fano scheme is of expected
dimension can be given in terms of the map
\[
\beta^\vee \colon
\Fr^*(\mathcal{S}_{\mathbf{F}_r(X)}) \to
\mathcal{Q}_{\mathbf{F}_r(X)}^\vee
\]
where \(\beta^\vee\) factors through the subbundle
\(\mathcal{Q}_{\mathbf{F}_r(X)}^\vee \subset V_{\mathbf{F}_r(X)}\) given by
the restriction of the tautological quotient bundle because \(\mathcal{S}_{\mathbf{F}_r(X)}\)
is isotropic for \(\beta\). The following identifies the conormal map of
\(\mathbf{F}_r(X)\) in \(\mathbf{G}(r+1,V)\), generalizing
\parref{hypersurfaces-conormal-compute}:

\begin{Proposition}\label{hypersurfaces-fano-cotangent}
There is a commutative diagram of \(\sO_{\mathbf{F}_r(X)}\)-modules
\[
\begin{tikzcd}
\Fr^*(\mathcal{S}_{\mathbf{F}_r(X)}) \otimes \mathcal{S}_{\mathbf{F}_r(X)}
\rar["\beta^\vee \otimes \id"'] \dar["\beta_{\mathcal{S},\mathbf{F}_r(X)}"']
& \mathcal{Q}_{\mathbf{F}_r(X)}^\vee \otimes \mathcal{S}_{\mathbf{F}_r(X)} \dar["\cong"] \\
\mathcal{C}_{\mathbf{F}_r(X)/\mathbf{G}(r+1,V)} \rar["\delta"]
& \Omega^1_{\mathbf{G}(r+1,V)}\rvert_{\mathbf{F}_r(X)}
\end{tikzcd}
\]
in which \(\delta\) is the conormal map of \(\mathbf{F}_r(X)\) in \(\mathbf{G}(r+1,V)\).
\end{Proposition}

\begin{proof}
The image of the morphism \(\beta_{\mathcal{S}}\) from
\parref{hypersurfaces-fano-koszul} defining \(\mathbf{F}_r(X)\) is the
ideal sheaf of the Fano scheme in the Grassmannian, whence its restriction
\(\beta_{\mathcal{S},\mathbf{F}_r(X)}\) surjects onto the conormal sheaf as on
the left of the diagram. The right of the diagram arises from the usual
identification of the cotangent bundle of the Grassmannian. Since \(\delta\)
acts by differentiating local equations, it is linear over \(q\)-powers,
whence commutativity of the square.
\end{proof}

Injectivity of the map \(\beta^\vee\) then gives
sufficient condition for, amongst other things, when \(\mathbf{F}_r(X)\) is
of expected dimension:

\begin{Corollary}\label{hypersurfaces-fano-expdim-sufficient}
If the map
\(\beta^\vee \colon \Fr^*(\mathcal{S}_{\mathbf{F}_r(X)}) \to \mathcal{Q}^\vee_{\mathbf{F}_r(X)}\)
is injective, then
\begin{enumerate}
\item\label{hypersurfaces-fano-expdim-sufficient.conormal}
\(\beta_{\mathcal{S},\mathbf{F}_r(X)} \colon \Fr^*(\mathcal{S}_{\mathbf{F}_r(X)}) \otimes \mathcal{S}_{\mathbf{F}_r(X)} \to \mathcal{C}_{\mathbf{F}_r(X)/\mathbf{G}(r+1,V)}\)
is an isomorphism,
\item\label{hypersurfaces-fano-expdim-sufficient.exact}
there is a short exact sequence
\[
0 \to
\mathcal{C}_{\mathbf{F}_r(X)/\mathbf{G}(r+1,V)} \xrightarrow{\delta}
\Omega^1_{\mathbf{G}(r+1,V)}\rvert_{\mathbf{F}_r(X)} \to
\Omega^1_{\mathbf{F}_r(X)} \to
0,
\]
\item\label{hypersurfaces-fano-expdim-sufficient.kahler}
there is an isomorphism
\(
\Omega^1_{\mathbf{F}_r(X)} \cong
(\mathcal{Q}^\vee_{\mathbf{F}_r(X)}/\Fr^*(\mathcal{S}_{\mathbf{F}_r(X)})) \otimes \mathcal{S}_{\mathbf{F}_r(X)}
\), and
\item\label{hypersurfaces-fano-expdim-sufficient.dim}
\(\mathbf{F}_r(X)\) is generically smooth, Cohen--Macaulay, and of dimension \((r+1)(n-2r-1)\).
\end{enumerate}
\end{Corollary}

\begin{proof}
If the map \(\beta^\vee\) in question is injective, then the commutative
diagram of \parref{hypersurfaces-fano-cotangent} implies that
\(\beta_{\mathcal{S},\mathbf{F}_r(X)}\) is injective. Since it is always
surjective, it is an isomorphism, giving
\ref{hypersurfaces-fano-expdim-sufficient.conormal}. The diagram of \parref{hypersurfaces-fano-cotangent}
now implies that
\(\delta\) is injective, whence exactness of the sequence
\ref{hypersurfaces-fano-expdim-sufficient.exact} and the identification of
\ref{hypersurfaces-fano-expdim-sufficient.kahler}. Since the conormal sequence
is split exact at generic points, so \(\mathbf{F}_r(X)\) is generically smooth
of the expected dimension, yielding \ref{hypersurfaces-fano-expdim-sufficient.dim}
in particular.
\end{proof}

In this setting, \(\mathbf{F}_r(X)\) is a local complete intersection closed
subscheme of the Grassmannian and so its dualizing bundle may be computed
by taking determinants of its sheaf of K\"ahler differentials; see
\cite[Theorem III.7.11]{Hartshorne:AG}. Write \(\sO_{\mathbf{F}_r(X)}(1)\)
for the restriction of the Pl\"ucker line bundle to the Fano scheme. The
following additionally keeps track of a determinant twist:

\begin{Corollary}\label{hypersurfaces-fano-canonical}
If
\(\beta^\vee \colon \Fr^*(\mathcal{S}_{\mathbf{F}_r(X)}) \to \mathcal{Q}^\vee_{\mathbf{F}_r(X)}\)
is injective, then
\[
\omega_{\mathbf{F}_r(X)}
\cong \sO_{\mathbf{F}_r(X)}\big((q+1)(r+1) - (n+1)\big) \otimes \det(V^\vee)^{\otimes r + 1}.
\]
\end{Corollary}

\begin{proof}
Taking determinants of the sequence in
\parref{hypersurfaces-fano-expdim-sufficient}\ref{hypersurfaces-fano-expdim-sufficient.exact}
yields
\[
\omega_{\mathbf{F}_r(X)} \cong
\det(\Omega^1_{\mathbf{F}_r(X)}) \cong
\det(\Omega^1_{\mathbf{G}(r+1,V)}) \otimes \det(\Fr^*(\mathcal{S}) \otimes \mathcal{S})^\vee \rvert_{\mathbf{F}_r(X)}.
\]
The result follows upon using the identifications
\[
\det(\mathcal{S}) \cong \sO_{\mathbf{G}(r+1,V)}(-1),
\quad
\det(\mathcal{Q}) \cong \sO_{\mathbf{G}(r+1,V)}(1) \otimes \det(V)
\]
and that
\(\det(\Omega^1_{\mathbf{G}(r+1,V)}) \cong \sO_{\mathbf{G}(r+1,V)}(-n-1) \otimes \det(V^\vee)^{\otimes r+1}\).
\end{proof}

By \parref{hypersurfaces-smooth-point-fano}, the singular locus of
\(\mathbf{F}_r(X)\) is contained in the locus parameterizing \(r\)-planes that
intersect the singular locus of \(X\). When \(\beta^\vee\) is injective, all
such \(r\)-planes give singular points of the Fano scheme:

\begin{Corollary}\label{hypersurfaces-fano-singular-locus}
If \(\beta^\vee \colon \Fr^*(\mathcal{S}_{\mathbf{F}_r(X)}) \to \mathcal{Q}^\vee_{\mathbf{F}_r(X)}\)
is injective, then the singular locus of \(\mathbf{F}_r(X)\) is the locus
parameterizing \(\PP U \subset X\) that intersect the singular locus of \(X\):
\[
\Sing(\mathbf{F}_r(X)) = \Set{ [\PP U] \in \mathbf{F}_r(X) | \PP U \cap \Sing(X) \neq \varnothing}.
\]
\end{Corollary}

\begin{proof}
Since the conormal sequence is exact by
\parref{hypersurfaces-fano-expdim-sufficient}\ref{hypersurfaces-fano-expdim-sufficient.exact},
the smoothness criterion from \cite[Theorem II.8.17]{Hartshorne:AG} means that
\(\mathbf{F}_r(X)\) is singular where \(\Omega^1_{\mathbf{F}_r(X)}\) has torsion.
The computation of \(\delta\) in \parref{hypersurfaces-fano-cotangent} shows
that this occurs along the degeneracy locus of
\[
\beta^\vee \colon
\Fr^*(\mathcal{S}_{\mathbf{F}_r(X)}) \to
\mathcal{Q}^\vee_{\mathbf{F}_r(X)} \subset
V^\vee_{\mathbf{F}_r(X)}.
\]
This drops rank over the points where the fibre of \(\Fr^*(\mathcal{S}_{\mathbf{F}_r(X)})\)
intersects \(V^\perp \subset \Fr^*(V)\). By \parref{hypersurfaces-nonsmooth-locus},
these points are given by \(\PP U \subset X\) which intersects the singular
locus.
\end{proof}

It is clear from \parref{hypersurfaces-fano-expdim-sufficient} that
injectivity of \(\beta^\vee\) should generally be a condition stronger than
expected dimensionality of the Fano scheme. Indeed, the scheme of lines in a
\(q\)-bic surface of type \(\mathbf{N}_3 \oplus \mathbf{1}\) is of expected
dimension \(0\), but \(\beta^\vee\) is not injective: by
\parref{surfaces-1+N3.lines}, \(\mathbf{F}_1(X)\) is a single nonreduced point.
It may also be the case that \(\beta^\vee\) is injective on some irreducible
components of \(\mathbf{F}_r(X)\) and not on others. This occurs on the scheme
of lines on a \(q\)-bic of type \(\mathbf{1}^{\oplus 2} \oplus \mathbf{N}_2\):
see \parref{surfaces-1+1+N2.lines}. In fact, the following gives a simple
characterization of these two conditions:

\begin{Lemma}\label{hypersurfaces-fano-expdim-criterion}
The morphism
\(\beta^\vee \colon \Fr^*(\mathcal{S}_{\mathbf{F}_r(X)}) \to \mathcal{Q}_{\mathbf{F}_r(X)}^\vee\)
is injective if and only if
\[
\dim\set{[\PP U] \in \mathbf{F}_r(X) | \PP U \cap \Sing(X) \neq \varnothing}
< (r+1)(n-2r-1)
\]
and \(\mathbf{F}_r(X)\) is of expected dimension \((r+1)(n-2r-1)\) if
and only if
\[
\dim\set{[\PP U] \in \mathbf{F}_r(X) | \PP U \cap \Sing(X) \neq \varnothing}
\leq (r+1)(n-2r-1).
\]
\end{Lemma}

\begin{proof}
Consider the first statement.
If \(\beta^\vee\) is injective, then by \parref{hypersurfaces-fano-expdim-sufficient},
\(\mathbf{F}_r(X)\) has expected dimension \((r+1)(n-2r-1)\) and is generically
smooth. Its singular locus is therefore a proper closed subscheme, and the
conclusion follows from \parref{hypersurfaces-fano-singular-locus}. For the
converse, consider the diagram
\[
\begin{tikzcd}[row sep=1.5em]
&& \Fr^*(\mathcal{S}_{\mathbf{F}_r(X)}) \rar["\beta^\vee"'] \dar[hook]
& \mathcal{Q}_{\mathbf{F}_r(X)}^\vee  \dar[hook] \\
0 \rar
& V^\perp_{\mathbf{F}_r(X)} \rar
& \Fr^*(V)_{\mathbf{F}_r(X)} \rar["\beta^\vee"]
& V^\vee_{\mathbf{F}_r(X)} \rar
& \Fr^*(V)_{\mathbf{F}_r(X)}^{\perp,\vee} \rar
& 0
\end{tikzcd}
\]
defining \(\beta^\vee \colon \Fr^*(\mathcal{S}_{\mathbf{F}_r(X)}) \to \mathcal{Q}_{\mathbf{F}_r(X)}^\vee\).
Since
\[
\dim\set{[\PP U] \in \mathbf{F}_r(X) | \PP U \cap \Sing(X) \neq \varnothing}
< (r+1)(n-2r-1)
\leq \dim\mathbf{F}_r(X),
\]
the fibres of \(\Fr^*(\mathcal{S}_{\mathbf{F}_r(X)})\)
and \(V^\perp_{\mathbf{F}_r(X)}\) are disjoint at the generic point of
\(\mathbf{F}_r(X)\). Since \(\Fr^*(\mathcal{S}_{\mathbf{F}_r(X)})\) is torsion-free,
this shows that \(\beta^\vee\) is injective.

For the second statement, by \parref{hypersurfaces-smooth-point-fano}, the
singular locus is contained in the set in question. Since the singular locus
certainly has dimension less than the entirety of \(\mathbf{F}_r(X)\), if the
Fano scheme has expected dimension, then the stated inequality holds.
Conversely, since the open complement of smooth points in \(\mathbf{F}_r(X)\)
has dimension \((r+1)(n-2r-1)\) by \parref{hypersurfaces-smooth-point-fano}, so
as long as the inequality holds, the entire Fano scheme has expected dimension.
\end{proof}

As a simple application of this criterion, the following shows that
when \((V,\beta)\) has corank at most \(1\) and has no radical, then
\(\mathbf{F}_r(X)\) achieves its expected dimension whenever it is nonnegative, and
\(\beta^\vee \colon \Fr^*(\mathcal{S}_{\mathbf{F}_r(X)}) \to \mathcal{Q}_{\mathbf{F}_r(X)}^\vee\)
is injective as soon as the expected dimension of \(\mathbf{F}_r(X)\) is positive:

\begin{Lemma}\label{hypersurfaces-fano-corank-1}
Suppose \((V,\beta)\) is a \(q\)-bic form of corank at most \(1\) with no radical.
\begin{enumerate}
\item\label{hypersurfaces-fano-corank-1.expdim}
If \(n \geq 2r+1\), then \(\mathbf{F}_r(X)\) is of expected dimension
\((r+1)(n-2r-1)\).
\item\label{hypersurfaces-fano-corank-1.beta}
If \(n \geq 2r+2\), then
\(\beta^\vee \colon \Fr^*(\mathcal{S}_{\mathbf{F}_r(X)}) \to \mathcal{Q}_{\mathbf{F}_r(X)}^\vee\)
is injective.
\end{enumerate}
\end{Lemma}

\begin{proof}
Both statements are clear when \(\beta\) is nonsingular, so assume that
\(\beta\) has corank \(1\) and no radical. Let \(x = \PP L\) be the unique
singular point of \(X\). Then by \parref{hypersurfaces-fano-expdim-criterion},
it suffices to show that the locus
\[ Z \coloneqq \Set{[\PP U] \in \mathbf{F}_r(X) | x \in \PP U} \]
has dimension at most \((r+1)(n-2r-1)\) for \ref{hypersurfaces-fano-corank-1.expdim},
and strictly less for \ref{hypersurfaces-fano-corank-1.beta}. Proceed by
induction on \(r\). When \(r = 0\), \(Z = \{x\}\) and the required inequalities
are clear. Assume \(r \geq 1\). The intersection
\(X \cap \PP\Fr^{-1}(L^\perp)\) is a cone over a \(q\)-bic \((n-3)\)-fold
\(X'\) of corank \(1\) with no radical, and by
\parref{threefolds-fano-linear-flag},
any \(r\)-plane through \(x\) is the cone over an \((r-1)\)-plane in \(X'\).
Applying induction to \(\mathbf{F}_{r-1}(X')\) now shows that
\[
\dim Z =
\dim\mathbf{F}_{r-1}(X') =
r(n-2r-1).
\]
If \(n \geq 2r+1\), this is at most \((r+1)(n-2r-1)\), proving
\ref{hypersurfaces-fano-corank-1.expdim}; if \(n \geq 2r+2\), this is strictly
less than \((r+1)(n-2r-1)\), proving \ref{hypersurfaces-fano-corank-1.beta}.
\end{proof}

Adapting the argument of \cite[Theorem 6]{BVV:Fano}, see also
\cite[p.544]{DM:Fano}, shows that the Fano schemes are connected whenever they
are positive dimensional:

\begin{Proposition}\label{hypersurfaces-fano-connected}
The scheme \(\mathbf{F}_r(X)\) is connected whenever \(n \geq 2r+2\).
\end{Proposition}

\begin{proof}
Let \(\Phi \subset \mathbf{G}(r+1,V) \times \PP(\Fr^*(V)^\vee \otimes V^\vee)\)
be the incidence correspondence of pairs \(([\PP U], [X])\) of \(r\)-planes and
\(q\)-bic hypersurfaces with \(\PP U \subseteq X\), as in the proof of
\parref{hypersurfaces-nonempty-fano}. Let \(Z \subset \Phi\) be the locus
where \(\pr_2 \colon \Phi \to \PP(\Fr^*(V)^\vee \otimes V^\vee)\) is
not smooth. Then the codimension of \(Z\) in \(\Phi\) is at least \(n-2r\).
Indeed, it follows from \parref{hypersurfaces-fano-singular-locus} that \(\pr_2\)
is smooth over the open subset of \(\PP(\Fr^*(V)^\vee \otimes V^\vee)\) parameterizing
smooth \(q\)-bic hypersurfaces. A general point of the codimension \(1\) complement
corresponds to a \(q\)-bic hypersurface \(X\) of corank \(1\) without a radical, see
\parref{forms-classification-rank-stratification} and \parref{forms-aut-general-corank-b}. By
\parref{hypersurfaces-fano-corank-1}, \(\mathbf{F}_r(X)\) is of expected dimension
\((r+1)(n-2r-1)\), and the proof of \parref{hypersurfaces-fano-corank-1}
shows that \(\Sing\mathbf{F}_r(X)\) has dimension \(r(n-2r-1)\). Therefore
\[
\codim(Z \subset \Phi) \geq 1 +
\codim(\Sing\mathbf{F}_r(X) \subset \mathbf{F}_r(X)) = n-2r.
\]

Now consider the Stein factorization
\[
\pr_2 \colon
\Phi \to \Phi' \to \PP(\Fr^*(V)^\vee \otimes V^\vee)
\]
of the second projection. Then \(Z\) contains the inverse image of
the branch locus of \(\Phi' \to \PP(\Fr^*(V)^\vee \otimes V^\vee)\). So
if \(n \geq 2r+2\), the codimension estimate implies that the branch locus has
codimension at least \(2\) in \(\Phi'\). By the Purity of the Branch Locus,
\citeSP{0BMB}, \(\Phi' \to \PP(\Fr^*(V)^\vee \otimes V^\vee)\) is
\'etale, and hence an isomorphism since projective space is simply connected.
The properties of the Stein factorization means that \(\pr_2\) has connected
fibres, and this means that each \(\mathbf{F}_r(X)\) is connected.
\end{proof}

The results so far are simplest when \(X\) itself is smooth:

\begin{Corollary}\label{hypersurfaces-smooth-fano}
If \(X \subset \PP V\) is a smooth \(q\)-bic \((n-1)\)-fold, then
its Fano scheme \(\mathbf{F}_r(X)\) of \(r\)-planes is smooth, of dimension
\[
\dim\mathbf{F}_r(X) = (r+1)(n - 2r - 1)
\quad\text{whenever}\; 0 < r < \tfrac{n}{2},
\]
and empty otherwise, and irreducible when \(n > 2r+1\). Moreover,
\[
\Omega_{\mathbf{F}_r(X)}^1 \cong
(\Fr^*(\mathcal{S}_{\mathbf{F}_r(X)})^\perp/\mathcal{S}_{\mathbf{F}_r(X)})^\vee \otimes \mathcal{S}_{\mathbf{F}_r(X)}.
\]
\end{Corollary}

\begin{proof}
The smoothness and dimension statement follow from \parref{hypersurfaces-smooth-point-fano}.
The statement about the sheaf of K\"ahler differentials follows from
\parref{hypersurfaces-fano-expdim-sufficient}\ref{hypersurfaces-fano-expdim-sufficient.kahler}
upon noting that there is a short exact sequence
\[
0 \to
\Fr^*(\mathcal{S}_{\mathbf{F}_r(X)})^\perp/\mathcal{S}_{\mathbf{F}_r(X)} \to
\mathcal{Q}_{\mathbf{F}_r(X)} \xrightarrow{\beta}
\Fr^*(\mathcal{S}_{\mathbf{F}_r(X)})^\vee \to
0
\]
since \(\beta\) is nondegenerate by \parref{hypersurfaces-smooth-and-nondegeneracy}.
Irreducibility when \(n > 2r+1\) follows from its smoothness and the
connectedness provided by \parref{hypersurfaces-fano-connected}.
\end{proof}

\begin{Remark}
Combining \parref{hypersurfaces-fano-canonical} and \parref{hypersurfaces-smooth-fano}
lead to an interesting collection of smooth projective varieties with a given
canonical bundle. For instance, let \(r > 0\) and let \(n = (q+1)(r+1) - 1\).
If \(X \subset \PP V = \PP^n\) is any a smooth \(q\)-bic hypersurface, then
\(\mathbf{F}_r(X)\) is a smooth projective variety of dimension
\((r+1)^2(q-1)\) with trivial canonical bundle.
\end{Remark}

\subsection{Numerical invariants}\label{hypersurfaces-fano-numerical}
Being the zero locus of a of
\(\Fr^*(\mathcal{S}^\vee) \otimes \mathcal{S}^\vee\)
by \parref{hypersurfaces-equations-of-fano}, when \(\mathbf{F}_r(X)\) has
expected codimension \((r+1)^2\), its class in the Chow ring of
\(\mathbf{G}(r+1,V)\) is given by the top Chern class
\[
[\mathbf{F}_r(X)] = c_{(r+1)^2}(\Fr^*(\mathcal{S}^\vee) \otimes \mathcal{S}^\vee)
\in \mathrm{CH}^{(r+1)^2}(\mathbf{G}(r+1,V)).
\]
When \(r = 1\), this can be reasonably used to compute numerical invariants
of the Fano schemes \(\mathbf{F}_1(X)\) of lines. Notation as in
\cite[p.271]{Fulton}, write
\(\sigma_\lambda = \sigma_{\lambda_0,\ldots,\lambda_r}\) for the Schubert cycle
associated with the partition \(\lambda = (\lambda_0,\ldots,\lambda_r)\) with
\(n-r \geq \lambda_0 \geq \cdots \geq \lambda_r \geq 0\). Then the total
Chern class of the dual of the tautological subbundle is
\[ c(\mathcal{S}^\vee) = 1 + \sigma_1 + \sigma_{1^2} + \cdots + \sigma_{1^{r+1}}, \]
where \(1^{d+1}\) denotes the partition whose nonzero parts are
\(\lambda_0 = \cdots = \lambda_d = 1\).

\begin{Proposition}\label{hypersurfaces-fano-lines-degree}
In the Chow ring of \(\mathbf{G}(2,V)\),
\begin{align*}
c_4(\Fr^*(\mathcal{S}^\vee) \otimes \mathcal{S}^\vee)
& = (q+1)^2 \sigma_{1,1} \big((q^2-2q+1) \sigma_{1,1} + q \sigma_1^2\big) \\
& = (q+1)(q^3+1)\sigma_{2,2}  + q(q+1)^2 \sigma_{3,1}.
\end{align*}
If \(\mathbf{F}_1(X)\) is of expected dimension \(2n-6\), then
\[
\deg(\sO_{\mathbf{F}_1(X)}(1))
= \frac{(2n-6)!}{(n-1)!(n-3)!} (q+1)^2\big((n-1)q^2 + (2n-8)q + (n-1)\big).
\]
\end{Proposition}

\begin{proof}
Use the splitting principle to
write \(c(\mathcal{S}^\vee) = (1 + a)(1 + b)\) in terms of its Chern roots
\(a\) and \(b\), so that \(c_1(\mathcal{S}^\vee) = \sigma_1 = a + b\) and
\(c_2(\mathcal{S}^\vee) = \sigma_{1,1} = ab\). The Chern roots of \(\Fr^*(\mathcal{S}^\vee)\)
are given by \(qa\) and \(qb\), so
\(c_4(\Fr^*(\mathcal{S}^\vee) \otimes \mathcal{S}^\vee)\) is given by
\begin{align*}
(a + qa)(b + qb)(a + qb)(b + qa)
& = (q+1)^2 a b \big((q^2 + 1) ab + q(a^2 + b^2)\big)  \\
& = (q+1)^2 \sigma_{1,1}\big( (q^2 - 2q + 1) \sigma_{1,1} + q \sigma_1^2\big),
\end{align*}
giving the first formula; the second formula now comes upon applying Pieri's
rule, see \cite[p.271]{Fulton}, to compute the products
\(\sigma_{1,1}^2 = \sigma_{2,2}\)
and \(\sigma_{1,1} \sigma_1^2 = \sigma_{2,2} + \sigma_{3,1}\).

The degree of the Pl\"ucker line bundle on \(\mathbf{F}_1(X)\)
is now obtained as
\begin{align*}
\deg(\sO_{\mathbf{F}_1(X)}(1))
& = \int_{\mathbf{G}(2,V)} \sigma_1^{2n-6} c_4(\Fr^*(\mathcal{S}^\vee) \otimes \mathcal{S}^\vee) \\
& =
(q+1)^2\Bigg((q^2-2q+1) \int_{\mathbf{G}(2,V)} \sigma_1^{2n-6}\sigma_{1,1}^2 +
q \int_{\mathbf{G}(2,V)} \sigma_1^{2n-4}\sigma_{1,1}\Bigg).
\end{align*}
Since \(\sigma_{1,1} = c_2(\mathcal{S}^\vee)\), it represents the cycle of
the Grassmannian of subspaces contained in a hyperplane of \(V\). Thus
\[
\deg(\sO_{\mathbf{F}_1(X)}(1)) =
(q+1)^2\big((q^2-2q+1) \deg(\sO_{\mathbf{G}(2,n-1)}(1)) + q \deg(\sO_{\mathbf{G}(2,n)}(1))\big).
\]
As computed in \cite[Example 14.7.11]{Fulton},
\(\deg(\sO_{\mathbf{G}(2,n)}(1)) = \frac{1}{n-2}\binom{2n-4}{n-1}\). Putting this
into the above gives the formula in the statement.
\end{proof}

\section{Fano correspondence}\label{section-hypersurfaces-fano-correspondences}
The universal family
\(\mathbf{L} \coloneqq \mathbf{L}_r(X) \coloneqq \PP\mathcal{S}\) of \(r\)-planes
over \(\mathbf{F} \coloneqq \mathbf{F}_r(X)\), given by the projective bundle
associated with the universal rank \(r+1\) subbundle \(\mathcal{S}\) on
\(\mathbf{F}\), defines an incidence correspondence
\[
\begin{tikzcd}
& \mathbf{L} \ar[dl,"\pr_{\mathbf{F}}"'] \ar[dr,"\pr_X"] \\
\mathbf{F} && X
\end{tikzcd}
\]
between the Fano scheme \(\mathbf{F}\) and the \(q\)-bic hypersurface \(X\).
This is referred to as the \emph{Fano correspondence}. Its basic geometric
properties are collected in this Section, and for that purpose, assume
throughout that the base field \(\kk\) is algebraically closed.

The first property says that \(X\) is swept out by \(r\)-planes whenever
\(\dim \mathbf{F}_r(X) > 0\):

\begin{Lemma}\label{hypersurfaces-fano-correspondences-surjective}
If \(n \geq 2r+2\), then \(\pr_X \colon \mathbf{L} \to X\) is surjective.
\end{Lemma}

\begin{proof}
It suffices to show that for a general closed point \(x = \PP L\) of \(X\), there
exists an \(r\)-plane contained in \(X\) which passes through \(x\). By
\parref{threefolds-fano-linear-flag}, any such \(r\)-plane must be contained in
the intersection
\[ X \cap \PP\Fr^*(L)^\perp \cap \PP\Fr^{-1}(L^\perp). \]
For general \(x\), this is a cone over a \(q\)-bic hypersurface \(X'\) of
dimension \(n-4\), and an \(r\)-plane through \(x\) as above is a cone over an
\((r-1)\)-plane in \(X'\). Since
\[ \dim\mathbf{F}_{r-1}(X') \geq r(n-2r-2) \geq 0 \]
by \parref{hypersurfaces-equations-of-fano}, the result follows from
\parref{hypersurfaces-nonempty-fano}.
\end{proof}

The morphism \(\pr_X \colon \mathbf{L} \to X\) never flat as the fibre
dimension jumps over special points. When \(X\) is smooth, the locus over which
the fibre dimension jumps is precisely the set of Hermitian points; compare
with the comments in \parref{hypersurfaces-cone-points-definition}.

\begin{Lemma}\label{hypersurfaces-fano-correspondences-smooth-fibres}
Let \(X\) be a smooth \(q\)-bic \((n-1)\)-fold. Then for a closed point
\(x \in X\),
\[
\dim \pr_X^{-1}(x) =
\begin{dcases*}
r(n-2r-1) & if \(n \geq 2r+1\) and \(x\) is Hermitian, and \\
r(n-2r-2) & if \(n \geq 2r+2\) and \(x\) is not Hermitian.
\end{dcases*}
\]
\end{Lemma}

\begin{proof}
As in the proof of \parref{hypersurfaces-fano-correspondences-surjective},
the reduced subscheme underlying \(\pr_X^{-1}(x)\) may be identified with the
space of \((r-1)\)-planes in the base \(X'\) of the cone
\[ X \cap \PP\Fr^*(L)^\perp \cap \PP\Fr^{-1}(L^\perp). \]
When \(x\) is a Hermitian point, then \(X'\) is a smooth \(q\)-bic \((n-3)\)-fold,
see \parref{hypersurfaces-cone-points-smooth}, and so by
\parref{hypersurfaces-equations-of-fano}, the fibre \(\pr^{-1}_X(x)\) is
nonempty of dimension \(r(n-2r-1)\) whenever \(n \geq 2r+1\). When \(x\) is not
a Hermitian point, then \(X'\) is a \(q\)-bic \((n-4)\)-fold of corank at most
\(1\) which is not a cone, so by
\parref{hypersurfaces-fano-corank-1}\ref{hypersurfaces-fano-corank-1.expdim},
\(\mathbf{F}_{r-1}(X')\) is of expected dimension \(r(n-2r-2)\).
\end{proof}

\subsection{}\label{hypersurfaces-fano-correspondences-flag}
In particular, \parref{hypersurfaces-fano-correspondences-smooth-fibres} shows
that when \(n \geq 2r+2\), the morphism \(\pr_X \colon \mathbf{L} \to X\) is
surjective and its generic fibre is of dimension \(r(n-2r-2)\). To describe
some further geometric properties of the fibres of \(\pr_X\), view
\(\mathbf{L}\) as a closed subscheme of the flag variety
\[
\mathbf{G}(1,r+1,V)_X \coloneqq
\mathbf{G}(1,r+1,V) \times_{\PP V} X =
\Set{(x,[\PP U]) \in X \times \mathbf{G}(r+1,V) | x \in \PP U}
\]
restricted to \(X\). Let \(\pi \colon \mathbf{G}(1,r+1,V)_X \to X\) be the
projection, and let
\[ 0 \to \pi^*\sO_X(-1) \to \mathcal{S} \to \mathcal{S}' \to 0 \]
be the tautological sequence, with \(\mathcal{S}\) the rank \(r+1\)
tautological subbundle of \(V_{\mathbf{G}(1,r+1,V)_X}\). The fibre of \(\pi\)
over a point \(x = \PP L\) of \(X\) is canonically isomorphic to the
Grassmannian \(\mathbf{G}(r,V/L)\) in a way compatible with Pl\"ucker
polarizations. The scheme \(\mathbf{L}\) is defined by the \(q\)-bic form
\(\beta_{\mathcal{S}}\) on \(\mathcal{S}\) obtained by restricting \(\beta\).
The following analyzes the basic structure of the equations of \(\mathbf{L}\)
relative to \(X\), and in particular computes the degree of the general fibre
with respect to the Pl\"ucker polarization
\(\sO_{\pr_X}(1) \coloneqq \pr_{\mathbf{F}}^*(\sO_{\mathbf{F}}(1))\) over \(X\):

\begin{Lemma}\label{hypersurfaces-fano-correspondences-degree}
Let \(X\) be a smooth \(q\)-bic \((n-1)\)-fold with \(n \geq 2r+2\). Then
each closed fibre of \(\pr_X \colon \mathbf{L} \to X\) has multiplicity
\(q^r\), and
\[
\deg(\sO_{\pr_X}(1)\rvert_{\pr_X^{-1}(x)}) =
q^r \deg(\mathbf{F}_{r-1}(X'))
\]
for every non-Hermitian closed point \(x \in X\), where \(X'\) is a smooth
\(q\)-bic \((n-4)\)-fold.
\end{Lemma}

\begin{proof}
Since \(\pi^*\sO_X(-1)\) is an isotropic subbundle of \(\mathcal{S}\),
\(\beta_{\mathcal{S}}\) induces two morphisms
\[
\Fr^*(\pi^*\sO_X(-1)) \otimes \mathcal{S}' \to \sO_{\mathbf{G}(1,r+1,V)_X}
\quad\text{and}\quad
\Fr^*(\mathcal{S}') \otimes \pi^*\sO_X(-1) \to \sO_{\mathbf{G}(1,r+1,V)_X}
\]
and their vanishing locus \(\mathbf{L}'\) is a closed subscheme of
\(\mathbf{G}(1,r+1,V)_X\) containing \(\mathbf{L}\). Taken together, the
equations mean that the reduced fibre of \(\mathbf{L}'\) over a non-Hermitian
point \(x = \PP L\) of \(X\) is the Grassmannian of \(r\)-spaces in
\(V' \coloneqq (\Fr^*(L)^\perp \cap \Fr^{-1}(L^\perp))/L\) viewed as a
subscheme of \(\mathbf{G}(r,V/L)\). Since the first set of equations are linear
in the fibre coordinates over \(X\), whereas the second set are \(q\)-powers,
the closed fibres of \(\mathbf{L}' \to X\) have multiplicity \(q^r\). Finally,
the scheme \(\mathbf{L}\) is the closed subscheme of \(\mathbf{L}'\) cut out by
a morphism
\[ \Fr^*(\mathcal{S}') \otimes \mathcal{S}' \to \sO_{\mathbf{L}'} \]
induced by \(\beta_{\mathcal{S}}\). Restricted to the reduced fibre
\(\mathbf{G}(r,V')\) over \(x\), this gives the equations to the Fano
scheme \(\mathbf{F}_{r-1}(X')\) of \((r-1)\)-planes to a \(q\)-bic hypersurface
\(X'\) of corank at most \(1\) and which is not a cone in \(\PP V'\). The
argument of \parref{hypersurfaces-fano-correspondences-smooth-fibres}
shows that this Fano scheme is of expected dimension, and so its Pl\"ucker
degree coincides with the corresponding Fano scheme of a smooth \(q\)-bic.
\end{proof}

\subsection{Action of the correspondence}\label{hypersurfaces-fano-correspondences-action}
The incidence correspondence \(\mathbf{L}\) is a closed subscheme of
\(\mathbf{F} \times X\), and may be viewed as a correspondence of degree
\(-r\) from \(\mathbf{F}\) to \(X\); see \cite[Chapter 16]{Fulton} for
generalities on correspondences. This induces morphisms
\[
\mathbf{L}_* \colon \mathrm{CH}_*(\mathbf{F}) \to \mathrm{CH}_{* + r}(X)
\quad\text{and}\quad
\mathbf{L}^* \colon \mathrm{CH}^*(X) \to \mathrm{CH}^{*-r}(\mathbf{F})
\]
given by \(\mathbf{L}_*(\alpha) \coloneqq \pr_{X,*}(\pr_{\mathbf{F}}^*(\alpha) \cdot \mathbf{L})\)
and \(\mathbf{L}^*(\beta) \coloneqq \pr_{\mathbf{F},*}(\pr_X^*(\beta) \cdot \mathbf{L})\),
where the dot denotes the intersection product of cycles on \(\mathbf{F} \times X\).

The next statement shows that the correspondence \(\mathbf{L}\) relates the
two polarizations
\[
h \coloneqq c_1(\sO_X(1)) \in \mathrm{CH}^1(X)
\quad\text{and}\quad
g \coloneqq c_1(\sO_{\mathbf{F}}(1)) \in \mathrm{CH}^1(\mathbf{F})
\]
given by the hyperplane class \(h\) of \(X\), and the Pl\"ucker polarization
\(g\) of \(\mathbf{F}\).

\begin{Lemma}\label{hypersurfaces-fano-correspondences-polarization}
If \(n \geq 2r+2\), then for every \(r+1 \leq k \leq n-1\),
\[
0 \neq \mathbf{L}^*(h^k) = c_{k-r}(\mathcal{Q}) \in
\mathrm{CH}^{k-r}(\mathbf{F}).
\]
In particular, \(\mathbf{L}^*(h^{r+1}) = g\) in \(\mathrm{CH}^1(\mathbf{F})\).
\end{Lemma}

\begin{proof}
View \(h^k\) as the restriction to \(X\) of the class of a general
\((n-k)\)-plane in \(\PP^n\). Then \(\mathbf{L}^*(h^k)\) represents the class
of the subscheme of \(r\)-planes in \(X\) that are incident with \(h^k\); in
other words, this is represented by the intersection of \(\mathbf{F}_r(X)\) with
the Schubert variety \(\Sigma_k\) corresponding to incidence with a general
\((n-k)\)-plane. Since there is an \(r\)-plane through every point of \(X\) by
\parref{hypersurfaces-fano-correspondences-surjective}, it follows that
\(\Sigma_k \cap \mathbf{F}_r(X)\) is always nonempty. Therefore
\[
\mathbf{L}^*(h^k)
= [\Sigma_k]
= c_{k-r}(\mathcal{Q}) \in \mathrm{CH}^{k-r}(\mathbf{F})
\]
and each class is nonzero; see
\cite[Example 14.7.3]{Fulton} for the second equality.
\end{proof}

Using the cycle class map and the Poincar\'e duality pairing,
\(\mathbf{L}\) acts on any Weil cohomology theory. In particular, it
yields a map on \(\ell\)-adic cohomology
\[
\mathbf{L}^* \colon
\mathrm{H}^*_{\mathrm{\acute{e}t}}(X,\mathbf{Q}_\ell) \to
\mathrm{H}^{*-2r}_{\mathrm{\acute{e}t}}(\mathbf{F},\mathbf{Q}_\ell).
\]
The following shows that, in the case \(r = 1\), the action is injective on the
middle cohomology of \(X\). The injectivity statement is analogous to one
involving schemes of lines to smooth Fano hypersurfaces over the complex
number, see \cite{Shimada:Cylinder}. The second statement is analogous to
one for cubic hypersurfaces, see \cite{Huybrechts:Cubics}.

\begin{Lemma}\label{hypersurfaces-fano-correspondences-injective}
The Fano correspondence defines an injective map
\[
\mathbf{L}^* \colon
\mathrm{H}_{\mathrm{\acute{e}t}}^{n-1}(X,\mathbf{Z}_\ell) \to
\mathrm{H}_{\mathrm{\acute{e}t}}^{n-3}(\mathbf{F}_1(X),\mathbf{Z}_\ell)
\]
and satisfies
\[
(\alpha \cdot \beta) =
-\frac{1}{q(q+1)} \int_{\mathbf{F}_1(X)} \mathbf{L}^*(\alpha) \cdot \mathbf{L}^*(\beta) \cdot g^{n-2}
\]
for all primitive classes
\(\alpha,\beta \in \mathrm{H}_{\mathrm{\acute{e}t}}^{n-1}(X,\mathbf{Q}_\ell)_{\mathrm{prim}}\).
\end{Lemma}

\begin{proof}
Since \(\mathrm{H}^{n-1}_{\mathrm{\acute{e}t}}(X,\mathbf{Z}_\ell)\) is torsion-free,
see \parref{hypersurfaces-smooth-etale}, it suffices to prove that \(\mathbf{L}^*\)
is injective upon inverting \(\ell\). When \(n-1\) is even, \parref{hypersurfaces-fano-correspondences-polarization}
shows that the class of \(h^{(n-1)/2}\) is mapped to a nonzero class in \(\mathrm{H}^{n-3}_{\mathrm{\acute{e}t}}(\mathbf{F}_1(X),\mathbf{Q}_\ell)\).
Since \(\mathbf{L}^*\) is equivariant for the action of \(\mathrm{U}_{n+1}(q)\)
and the hyperplane class and the primitive subspace span distinct irreducible
representations by
\parref{hypersurfaces-smooth-etale}\ref{hypersurfaces-smooth-etale.irrep}, it
remains to show that \(\mathbf{L}^*\) is injective on
\(\mathrm{H}^{n-1}_{\mathrm{\acute{e}t}}(X,\mathbf{Q}_\ell)_{\mathrm{prim}}\);
by Poincar\'e duality, this will be implied by the second statement.

To relate the pairings, since \(\mathbf{L}\) is a \(\PP^1\)-bundle over
\(\mathbf{F}\), there is an isomorphism
\[
\mathrm{H}^{n-1}_{\mathrm{\acute{e}t}}(\mathbf{L},\mathbf{Q}_\ell) \cong
\pr_{\mathbf{F}}^*\mathrm{H}^{n-1}_{\mathrm{\acute{e}t}}(\mathbf{F},\mathbf{Q}_\ell) \oplus
\pr_{\mathbf{F}}^*\mathrm{H}^{n-3}_{\mathrm{\acute{e}t}}(\mathbf{F},\mathbf{Q}_\ell) \cdot \xi
\]
where \(\xi \coloneqq c_1(\sO_{\pr_{\mathbf{F}}}(1))\). Then for any class
\(\alpha \in \mathrm{H}^{n-1}_{\mathrm{\acute{e}t}}(X,\mathbf{Q}_\ell)_{\mathrm{prim}}\),
\[
\pr_X^*(\alpha) =
\pr_{\mathbf{F}}^*(\alpha') +
\pr_{\mathbf{F}}^*(\mathbf{L}^*(\alpha)) \cdot \xi
\quad\text{for some}\;
\alpha' \in \mathrm{H}^{n-1}_{\mathrm{\acute{e}t}}(\mathbf{F},\mathbf{Q}_\ell).
\]
Using the relation
\(
\xi^2 +
\pr_{\mathbf{F}}^*(c_1(\mathcal{S})) \cdot \xi +
\pr_{\mathbf{F}}^*(c_2(\mathcal{S})) = 0
\)
in \(\mathrm{CH}^*(\mathbf{L})\) given by Grothendieck's construction of the
Chern classes, see \cite[p.144]{Grothendieck:Chern} or \cite[Remark
3.2.4]{Fulton}, and the fact that \(\alpha\) is primitive, it follows that
\[
0 =
\pr_X^*(\alpha \cdot h) =
\pr_X^*(\alpha) \cdot \xi =
-\pr^*_{\mathbf{F}}(\mathbf{L}^*(\alpha) \cdot c_2(\mathcal{S}))
+ \pr_X^*(\mathbf{L}^*(\alpha) \cdot c_1(\mathcal{S}) - \alpha') \cdot \xi
\]
since \(\pr_X^*(h) = \xi\) by the
discussion of \parref{hypersurfaces-fano-correspondences-flag}. Therefore
\(\mathbf{L}^*(\alpha) \cdot c_2(\mathcal{S}) = 0\) and
\[
\pr_X^*(\alpha) =
\pr_{\mathbf{F}}^*(\mathbf{L}^*(\alpha)) \cdot
(\pr_{\mathbf{F}}^*(c_1(\mathcal{S})) + \xi)
\quad\text{for all}\;
\alpha \in \mathrm{H}^{n-1}_{\mathrm{\acute{e}t}}(X,\mathbf{Q}_\ell)_{\mathrm{prim}}.
\]
Therefore, for any pair \(\alpha\) and \(\beta\) of primitive classes, since
\(\pr_{\mathbf{F},*}\) extracts the coefficient of \(\xi\), the Chern class
relation above implies that
\begin{align*}
\pr_{\mathbf{F},*}\pr_X^*(\alpha \cdot \beta)
=
\pr_{\mathbf{F},*}\pr^*_{\mathbf{F}}(\mathbf{L}^*(\alpha) \cdot \mathbf{L}^*(\beta)) \cdot
(\pr_{\mathbf{F}}^*(c_1(\mathcal{S})) + \xi)^2
= -\mathbf{L}^*(\alpha) \cdot \mathbf{L}^*(\beta) \cdot g
\end{align*}
in \(\mathrm{H}^{2n-4}_{\mathrm{\acute{e}t}}(\mathbf{F},\mathbf{Q}_\ell)\).
Write \(\alpha \cdot \beta = \deg(\alpha \cdot \beta) [x]\) in terms
of the class of a point in \(\mathrm{H}^{2n-2}_{\mathrm{\acute{e}t}}(X,\mathrm{Q}_\ell)\)
so that \(\pr_X^*(\alpha \cdot \beta) = (\alpha \cdot \beta) [\mathbf{L}_x]\)
is a multiple of the class of a fibre of \(\pr_X \colon \mathbf{L} \to X\).
Multiplying by \(g^{n-4}\) on both sides and taking degrees now shows that
\[
(\alpha \cdot \beta)
\int_{\mathbf{L}} \mathbf{L}_x \cdot g^{n-4}
=
-\int_{\mathbf{F}} \mathbf{L}^*(\alpha) \cdot \mathbf{L}^*(\beta) \cdot g^{n-3}
\]
and the integral on the left is the Pl\"ucker degree of a general fibre of
\(\pr_X\). By \parref{hypersurfaces-fano-correspondences-degree}, this is
\(q(q+1)\), yielding the result.
\end{proof}

\subsection{Incidence schemes}\label{hypersurfaces-fano-correspondences-incidence}
Let \(X\) be a smooth \(q\)-bic hypersurface. Suppose that \(n \geq 2r+2\)
so that \(\mathbf{F}_k(X)\) is of positive dimension and
\(\pr_X \colon \mathbf{L}_k(X) \to X\) is surjective for every \(0 \leq k \leq r\),
see \parref{hypersurfaces-fano-correspondences-surjective}. For
every \(r\)-plane \(P \subset X\), let \(D_{k,P}\) be the closed subscheme of
\(\mathbf{F}_k(X)\) obtained by taking the closure of the locus
\[
\Set{P' \in \mathbf{F}_k(X) | P' \neq P \;\text{and}\; P' \cap P \neq \varnothing}
\]
of \(k\)-planes in \(X\) which are incident with \(P\). This is equivalently
the support of the \((k(n-2k-2)+r)\)-cycle obtained by applying the
correspondence
\[
\mathbf{L}_k(X)^* \colon
\mathrm{CH}^{n-r-1}(X) \to \mathrm{CH}^{n-r-k-1}(\mathbf{F}_k(X))
\]
to the cycle \([P]\). The following gives a basic property of the cycles
\([D_{k,P}]\):

\begin{Lemma}\label{hypersurfaces-fano-correspondences-D}
The \(D_{k,P}\) are algebraically equivalent for varying \(r\)-planes
\(P\), and
\[ (q+1)[D_{k,P}] \sim_{\mathrm{alg}} c_{n-r-k-1}(\mathcal{Q}). \]
\end{Lemma}

\begin{proof}
The classes \([D_{k,P}]\) are images of the classes \([P]\) under the
correspondence \(\mathbf{L}_k(X)^*\). All \(r\)-planes in \(X\) are parameterized
by the Fano scheme \(\mathbf{F}_r(X)\) which, since \(n \geq 2r+2\), is connected
by \parref{hypersurfaces-fano-connected}. So the classes \([P]\) are algebraically
equivalent for varying \(P\), whence by \cite[Proposition 10.3]{Fulton},
the \([D_{k,P}]\) are also algebraically equivalent for varying \(P\).

For the second statement, by the first part and
\parref{hypersurfaces-fano-correspondences-polarization}, it suffices to show
that there exists a \((r+1)\)-plane section of \(X\) which is a union of
\(r\)-planes. By taking general hyperplane sections, it suffices to consider
the case \(n = 2r+2\), in which case, such a \((r+1)\)-plane is obtained by
taking successively taking the embedded tangent space at a Hermitian point and
applying induction using \parref{hypersurfaces-cone-points-smooth}.
\end{proof}

\section{Hermitian structures}\label{section-hypersurfaces-hermitian}
Given a nonsingular \(q\)-bic form \((V,\beta)\), the constructions of
\parref{forms-hermitian-endomorphism} give a canonical \(q^2\)-linear
map \(\phi \colon V \to V\). This may be viewed as a
\(\mathbf{F}_{q^2}\)-rational structure of \(V\) with which \(\beta\) is
related to a Hermitian form. This section will describe some of the geometric
content of this structure. Paragraphs
\parref{hypersurfaces-endomorphism-V}--\parref{hypersurfaces-filtration-graph-resolution}
describe the structure provided by \(\phi\) on all of \(\PP V\).

The structure induced on the associated \(q\)-bic hypersurface \(X\) is
described starting from \parref{hypersurfaces-endomorphism}. In particular,
this restricts to a morphism \(\phi_X \colon X \to X\) of schemes over
\(\kk\), which then induces a canonical filtration \(X^\bullet\) of \(X\) by
complete intersections: see \parref{hypersurfaces-filtration-embedded-tangent}
and \parref{hypersurfaces-filtration}. This filtration is related to the Fano
schemes in \parref{hypersurfaces-filtration-rational-map}.

The endomorphism \(\phi\) is preserved by the action of finite subgroups of
\(\GL(V)\); namely, \(\mathrm{GL}_{n+1}(q^2)\) and, upon taking \(\beta\) into
consideration, \(\mathrm{U}_{n+1}(q)\). Thus the varieties appearing in this
section are intimately related to Deligne--Lusztig varieties, as introduced in
\cite{DL}. The objects in the first half are related to Drinfel's upper
half plane over finite fields: compare with \cite{Langer:Drinfeld} and
\cite{Ekedahl:CY}. The objects in the second half are related to Deligne--Lusztig
varieties of type \(\mathrm{A}_n^2\), and this relation will be used in
\parref{section-threefolds-smooth} to determine the \(\ell\)-adic cohomology
of the surface of lines associated with a smooth \(q\)-bic threefold.

\subsection{Endomorphism}\label{hypersurfaces-endomorphism-V}
The canonical \(q^2\)-linear map \(\phi \colon V \to V\) defined in
\parref{forms-hermitian-endomorphism} induces, for each \(0 \leq r \leq n\),
endomorphisms of schemes over \(\kk\)
\[
\phi_{\mathbf{G}(r+1,V)} \coloneqq
(\beta^{-1} \circ \beta^{(q),\vee}) \circ \Fr^2_{\mathbf{G}(r+1,V)/\kk} \colon
\mathbf{G}(r+1,V) \to
\mathbf{G}(r+1,V)^{(q^2)} \to
\mathbf{G}(r+1,V).
\]
This is defined by the subbundle
\(\Fr^{2,*}(\mathcal{S}_{\mathbf{G}(r+1,V)}) \hookrightarrow V_{\mathbf{G}(r+1,V)}\)
obtained by composing the \(q^2\)-Frobenius pullback of the
tautological subbundle with the linearization of \(\phi\), as in
\parref{forms-linearization}. Therefore \(\phi_{\mathbf{G}(r+1,V)}\) is the
composition of the \(\kk\)-linear \(q^2\)-power Frobenius with the
isomorphism induced by the linear isomorphism
\(\beta^{-1} \circ \beta^{(q),\vee} \colon \Fr^{2,*}(V) \to V\).
In particular, this shows that \(\phi_{\mathbf{G}(r+1,V)}\) is purely inseparable
of degree \(q^{2(r+1)(n-r)}\).

\subsection{Cyclic subspaces}
The endomorphisms \(\phi_{\mathbf{G}(r+1,V)}\) identify various structures on
the Grassmannians. For instance, it follows from \parref{forms-hermitian-fixed}
that the fixed points of \(\phi_{\mathbf{G}(r+1,V)}\) correspond to the
Hermitian subspaces in \(V\) of dimension \(r+1\). As another simple construction,
iterating \(\phi_{\PP V}\) yields a rational map from \(\PP V\) to the Grassmannians:

\begin{Proposition}\label{hypersurfaces-filtration-cyclic}
For each \(0 \leq r \leq n\), the injection of \(\sO_{\PP V}\)-modules
\[
\bigoplus\nolimits_{i = 0}^r \phi_{\PP V}^{i,*}(\mathrm{eu}) \colon
\bigoplus\nolimits_{i = 0}^r \sO_{\PP V}(-q^{2i}) \to
V_{\PP V}
\]
is locally split away from the Hermitian \((r-1)\)-planes in \(\PP V\). The
yields a rational map
\[
\operatorname{cyc}^r_\phi \colon \PP V \dashrightarrow \mathbf{G}(r+1,V),
\quad
x \mapsto \langle x, \phi_{\PP V}(x), \ldots, \phi_{\PP V}^r(x) \rangle
\]
defined away from the Hermitian \((r-1)\)-planes of \(\PP V\), and injective
away from the Hermitian \(r\)-planes of \(\PP V\).
\end{Proposition}

\begin{proof}
The morphism \(\phi_{\PP V}\) is constructed in \parref{hypersurfaces-endomorphism-V}
as the linearization of \(\phi \colon V \to V\). Therefore the sheaf map in the
statement degenerates along points corresponding to vectors \(v \in V\) such
that the cyclic subspace \(U \coloneqq \langle v, \phi(v),\ldots, \phi^r(v) \rangle\)
has dimension at most \(r\). By \parref{forms-hermitian-min-contain}, this
means that \(v\) is contained in a Hermitian subspace of dimension at most
\(r\), yielding the first two statements. If \(U\) does not contain a Hermitian
subspace of dimension \(r+1\), then \(U \cap \phi^r(U) = \langle \phi^r(v)
\rangle\). Since \(\phi\) is injective, this determines \(v\). This gives the
injectivity statement.
\end{proof}

The rational map \(\operatorname{cyc}_\phi^r\) factors through the locally
closed subscheme of the Grassmannian parameterizing cyclic subspaces for
\(\phi\). Let \(\mathbf{G}(r+1,V)_{\mathrm{cyc}}\) denote the closure of this
locus, so that \(\operatorname{cyc}_\phi^r\) may be viewed as a dominant
rational map \(\PP V \dashrightarrow \mathbf{G}(r+1,V)_{\mathrm{cyc}}\).
There is a map in the other direction:

\begin{Proposition}\label{hypersurfaces-filtration-Gr-cyc}
For each \(0 \leq r \leq n\), there is a dominant rational map
\[
[\id \cap \phi^r] \colon \mathbf{G}(r+1,V)_{\mathrm{cyc}} \dashrightarrow \PP V
\]
defined away from the locus of subspaces containing a \(2\)-dimensional
Hermitian subspace of \(V\), and injective away from the locus of subspaces
containing a Hermitian vector of \(V\).
\end{Proposition}

\begin{proof}
The scheme \(\mathbf{G}(r+1,V)_{\mathrm{cyc}}\) is an irreducible component of
\[
\operatorname{Degen}_{r+2}\big(
\mathcal{S} \oplus \phi^*_{\mathbf{G}(r+1,V)}(\mathcal{S}) \to V_{\mathbf{G}(r+1,V)}
\big)
\]
the subscheme of the Grassmannian where the fibre of the tautological subbundle
and its pullback by \(\phi\) intersect in a space of dimension at least \(r\).
Therefore, for each \(0 \leq i \leq r\), the
intersection \(\mathcal{S} \cap \phi_{\mathbf{G}(r+1,V)}^{i,*}(\mathcal{S})\)
is torsion-free subsheaf of \(V_{\mathbf{G}(r+1,V)_{\mathrm{cyc}}}\) of rank
\(r + 1 - i\), which by \parref{forms-hermitian-min-contain}, is a subbundle
away from the locus of subspaces containing a \((r+2-i)\)-dimensional
Hermitian subspace of \(V\). Taking \(i = r\) gives the first statement, and
injectivity follows by the same argument as in
\parref{hypersurfaces-filtration-cyclic}.
\end{proof}

It follows from the functorial description given in
\parref{hypersurfaces-filtration-Gr-cyc} that
\(\mathbf{G}(r+1,V)_{\mathrm{cyc}}\) is stable under the endomorphism
\(\phi_{\mathbf{G}(r+1,V)}\) of \parref{hypersurfaces-endomorphism-V}. Write
\(\phi_{\mathbf{G}(r+1,V)_{\mathrm{cyc}}}\) for the induced endomorphism.

\begin{Corollary}\label{hypersurfaces-filtration-cyclic-composition}
For each \(0 \leq r \leq n\),
\[
[\id \cap \phi^r] \circ \operatorname{cyc}_\phi^r = \phi_{\PP V}^r
\quad\text{and}\quad
\operatorname{cyc}_\phi^r \circ\; [\id \cap \phi^r] = \phi_{\mathbf{G}(r+1,V)_{\mathrm{cyc}}}^r.
\]
Therefore \(\operatorname{cyc}_\phi^r\) and \([\id \cap \phi^r]\) are
dominant, generically finite, and purely inseparable.
\end{Corollary}

\begin{proof}
This follows from the functorial descriptions given in
\parref{hypersurfaces-filtration-cyclic} and
\parref{hypersurfaces-filtration-Gr-cyc}.
\end{proof}

\subsection{Resolution}\label{hypersurfaces-filtration-resolution}
The graph closures of the rational maps \(\mathrm{cyc}_\phi^r\) and
\([\id \cap \phi^r]\) constructed in \parref{hypersurfaces-filtration-cyclic}
and \parref{hypersurfaces-filtration-Gr-cyc} admit moduli descriptions which
then provide resolutions of the two rational maps. Indeed, these
are the closed subschemes of the incidence correspondence between \(\PP V\)
and \(\mathbf{G}(r+1,V)_{\mathrm{cyc}}\) with points
\begin{align*}
\Gamma_{\mathrm{cyc}_\phi^r} & =
\Set{(L \subset U) | \langle L, \phi(L), \ldots, \phi^r(L) \rangle \subseteq U}, \\
\Gamma_{[\id \cap \phi^r]} & =
\Set{(L \subset U) | L \subseteq U \cap \phi^r(U)},
\end{align*}
where the points of the incidence correspondence are written as flags
\((L \subset U)\) with \([L] \in \PP V\) and
\([U] \in \mathbf{G}(r+1,V)_{\mathrm{cyc}}\). The map \(\phi\) induces morphisms:
\begin{align*}
\phi^r \times \id & \colon \Gamma_{\operatorname{cyc}^r_\phi} \to \Gamma_{[\id \cap \phi^r]},
&
(L \subset U) & \mapsto (\phi^r(L) \subset U), \\
\id \times \phi^r & \colon
\Gamma_{[\id \cap \phi^r]} \to \Gamma_{\operatorname{cyc}^r_\phi},
&
(L \subset U) & \mapsto (L \subset \phi^r(U)).
\end{align*}
These morphisms resolve the rational maps in the following sense:

\begin{Proposition}\label{hypersurfaces-filtration-graph-resolution}
There is a commutative diagrams of schemes over \(\kk\) given by
\[
\begin{tikzcd}[column sep=3em]
\Gamma_{\operatorname{cyc}^r_\phi} \rar["\phi^r \times \id"'] \dar["\pr_1"']
& \Gamma_{[\id \cap \phi^r]}
\rar["\id \times \phi^r"'] \dar["\pr_2"']
& \Gamma_{\operatorname{cyc}^r_\phi} \dar["\pr_1"] \\
\PP V \rar[dashed,"\operatorname{cyc}^r_\phi"]
& \mathbf{G}(r+1,V)_{\mathrm{cyc}} \rar[dashed,"{[\id \cap \phi^r]}"]
& \PP V
\end{tikzcd}
\]
such that
\begin{enumerate}
\item\label{hypersurfaces-filtration-graph-resolution.cyc}
\(\pr_1 \colon \Gamma_{\operatorname{cyc}_\phi^r} \to \PP V\) is an
isomorphism away from the union of the Hermitian \((r-1)\)-planes of \(\PP V\);
\item\label{hypersurfaces-filtration-graph-resolution.int}
\(\pr_2 \colon \Gamma_{[\id \cap \phi^r]} \to \mathbf{G}(r+1,V)_{\mathrm{cyc}}\)
is an isomorphism away from the locus of subspaces
containing a \(2\)-dimensional Hermitian subspace of \(V\); and
\item\label{hypersurfaces-filtration-graph-resolution.insep}
\(\phi^r \times \id\) and \(\id \times \phi^r\)
are finite purely inseparable of degrees \(q^{r(r+1)}\) and
\(q^{r(2n-r-1)}\).
\end{enumerate}
\end{Proposition}

\begin{proof}
Since \(\Gamma_{\operatorname{cyc}_\phi^r}\) and \(\Gamma_{[\id \cap \phi^r]}\)
are the graphs of the rational maps from
\parref{hypersurfaces-filtration-cyclic} and
\parref{hypersurfaces-filtration-Gr-cyc}, items
\ref{hypersurfaces-filtration-graph-resolution.cyc} and
\ref{hypersurfaces-filtration-graph-resolution.int} follow from the
identification of the indetereminacy loci. For
\ref{hypersurfaces-filtration-graph-resolution.insep}, note that the
morphisms in question are induced by the finite purely inseparable morphisms
\(\phi_{\PP V} \times \id\) and \(\id \times \phi_{\mathbf{G}(r+1,V)}\)
on the product \(\PP V \times \mathbf{G}(r+1,V)\), so they too share these
properties.

Compute the degree of \(\phi^r \times \id\) by considering the fibre of a
general geometric point of \(\Gamma_{[\id \cap \phi^r]}\).
By \parref{forms-hermitian-min-contain}, a general point may be taken to be
\((L_r \subset U)\) with
\[
L_i = \langle \phi^i(v) \rangle,
\quad
U = \langle v, \phi(v), \ldots, \phi^r(v) \rangle,
\quad\text{and}\quad
V = \langle v, \phi(v), \ldots, \phi^n(v) \rangle.
\]
The fibre of \(\phi^r \times \id\) over \((L_r \subset U)\) is supported on
the point \((L_0 \subset U)\), and represents the deformation functor
on the category of Artinian local \(\kk\)-algebras given by
\[
A \mapsto
\Set{
(\tilde{L} \subset U \otimes_\kk A) \in \Gamma_{\operatorname{cyc}^r_\phi}(A) |
\tilde{L} \otimes_A \kk = L_0
\;\text{and}\;
\phi^r(\tilde{L}) = L_r \otimes_\kk A
}.
\]
Since \(A\) is local, \(\tilde{L}\) is free. Since \(\tilde{L}\) is a
deformation of \(L_0\), a basis is of the form
\[
\tilde{v} =
v + a_1 \phi(v) + \cdots + a_r \phi^r(v)
\quad\text{for some}\; a_1,\ldots,a_r \in \mathfrak{m}_A.
\]
The condition that \((\tilde{L} \subset U \otimes_\kk A)\) is an \(A\)-point
of \(\Gamma_{\mathrm{cyc}_\phi^r}\) means that
\[
\phi^i(\tilde{v}) =
\phi^i(v) + a_1^{q^{2i}} \phi^{i+1}(v) + \cdots + a_r^{q^{2i}} \phi^{i+r}(v)
\in U \otimes_\kk A
\]
for each \(0 \leq i \leq r\). Since \(\phi^j(v) \notin U\) for \(j > r\),
this implies that
\[ a_{r - i}^{q^{2(i+1)}} = 0 \quad\text{for}\; 0 \leq i \leq r-1. \]
This also implies the condition that
\(\phi^r(\tilde{L}) = L_r \otimes_\kk A\). Therefore
\[
(\phi^r \times \id)^{-1}(L_r \subset U) \cong
\Spec\big(
\kk[\epsilon_1,\ldots,\epsilon_r]/(\epsilon_1^{q^{2r}}, \ldots, \epsilon_r^{q^2})
\big)
\]
and so \(\deg(\phi^r \times \id) = q^{r(r+1)}\). Since the composite
\((\phi^r \times \id) \circ (\id \times \phi^r)\) has degree \(q^{2rn}\) by
\parref{hypersurfaces-filtration-cyclic-composition}, this implies that
\(\deg(\id \times \phi^r) = q^{r(2n-r-1)}\).
\end{proof}

\subsection{Restriction to \(\mathbf{F}_r(X)\)}\label{hypersurfaces-endomorphism}
The constructions made so far on all of \(\mathbf{G}(r+1,V)\) restrict
well to the Fano schemes of the smooth \(q\)-bic hypersurface
\(X\). Indeed, since the \(\mathbf{F}_r(X)\) are moduli of isotropic subspaces
by \parref{hypersurfaces-equations-of-fano} and since \(\phi\) preserves
isotropic vectors by \parref{forms-hermitian-phi-phi}, the maps
\(\phi_{\mathbf{G}(r+1,V)}\) restrict to endomorphisms
\[
\phi_{\mathbf{F}_r(X)} \colon \mathbf{F}_r(X) \to \mathbf{F}_r(X)
\]
of schemes over \(\kk\), for each \(0 \leq r < n/2\).

The \(r = 0\) case, corresponding to \(\mathbf{F}_0(X) = X\), has a
particularly simple geometric interpretation. Recall from
\parref{hypersurfaces-tangent-form-properties-residual-tangent} that the
\emph{residual point of tangency} to a point \(x \in X\) is the other special
point of the \(q\)-bic hypersurface \(X \cap \mathbf{T}_{X,x}\) of corank
\(1\); for a smooth \(q\)-bic curve, this is the residual point of intersection
with the tangent line, see \parref{curve-residual-intersection}.

\begin{Lemma}\label{hypersurfaces-filtration-embedded-tangent}
\(\phi_X \colon X \to X\) sends a point to its residual point of tangency.
\end{Lemma}

\begin{proof}
The discussion of \parref{hypersurfaces-conormal-compute} gives a commutative
diagram of \(\sO_X\)-modules
\[
\begin{tikzcd}
\sO_X(-q^2) \dar["f_\beta^q"'] \ar[r,"\mathrm{eu}^{(q^2)}"'] &
\Fr^{2,*}(V)_X \rar["\beta^{(q),\vee}"'] & \Fr^*(V)_X^\vee \dar["\beta^{-1}"] \\
\Fr^*(\mathcal{N}_{X/\PP V}(-1))^\vee \rar["\phi_X"] &
\mathcal{T}_X^{\mathrm{e}} \rar[hook] & V_X
\end{tikzcd}
\]
where the \(\sO_X\)-module morphism \(\phi_X\) is as constructed in
\parref{hypersurfaces-tangent-form-properties} and which gives the residual
point of tangency, and the top composition is the restriction of the map
defining \(\phi_{\PP V}\) from \parref{hypersurfaces-endomorphism-V}. This
gives the result.
\end{proof}

\subsection{Geometric filtration}\label{hypersurfaces-filtration}
The endomorphism \(\phi_X\) determines a filtration \(X^\bullet\) of \(X\) by
closed subschemes as follows: Viewing \(X\) as the moduli space of isotropic
vectors for the \(q\)-bic form \((V,\beta)\) as in
\parref{hypersurfaces-moduli-of-isotropic-vectors}, consider the closed
subfunctors
\[
X^\bullet \colon
X
\eqqcolon X^0
\supseteq X^1
\supseteq X^2
\supseteq \cdots
\supseteq X^{\lfloor n/2 \rfloor}
\]
where the value of \(X^k \colon \mathrm{Sch}_\kk^{\mathrm{opp}} \to \mathrm{Set}\)
on a \(\kk\)-scheme \(T\) is given by
\[
T \mapsto
\Set{\iota \colon \mathcal{V}' \hookrightarrow V_T |
\bigoplus\nolimits_{i = 0}^k \phi_T^i \circ \Fr^{2i,*}(\iota) \colon
\bigoplus\nolimits_{i = 0}^k \Fr^{2i,*}(\mathcal{V}') \to V_T
\;\text{is isotropic for \(\beta\)}}
\]
where \(\iota \colon \mathcal{V}' \hookrightarrow V_T\) is an isotropic
subbundle of rank \(1\), \(\phi_T^i \colon \Fr^{2i,*}(V)_T \to V_T\) is the
linearization of the \(q^{2i}\)-linear map induced by \(\phi^i \colon V \to V\),
and a morphism \(\varphi \colon \mathcal{E} \to V_T\) is said to be
\emph{isotropic} if the composition
\[
\Fr^*(\varphi)^\vee \circ \beta \circ \varphi \colon
\mathcal{E} \to
V_T \to
\Fr^*(V)_T^\vee \to
\Fr^*(\mathcal{E})^\vee
\]
vanishes. The pieces of the filtration are representable by complete
intersections:

\begin{Proposition}\label{hypersurfaces-filtration-properties}
Let \(0 \leq k \leq n/2\). Then \(X^k\) is represented by the complete
intersection in \(\PP V\) with equations
\[
\beta(\phi_{\PP V}^{*,i}(\mathrm{eu})^{(q)}, \mathrm{eu}) \colon
\sO_{\PP V} \to
\sO_{\PP V}(q^{2i+1}+1)
\quad\text{for}\;0 \leq i \leq k.
\]
The singular locus of \(X^k\) is supported on the union of the Hermitian
\((k-1)\)-planes in \(X\).
\end{Proposition}

\begin{proof}
The case \(k = 0\) is \parref{hypersurfaces-moduli-of-isotropic-vectors}.
So consider \(0 < k \leq n/2\). View \(X\) as the moduli space of isotropic
vectors for \((V,\beta)\) and view \(X^k\) as a closed subfunctor of \(X\).
Then to recognize \(X^k\) as the stated complete intersection, it suffices to
show that for every isotropic \(v \in V\),
\[
\langle v, \phi(v), \ldots, \phi^k(v) \rangle
\;\text{is isotropic}
\quad\Leftrightarrow\quad
\beta(\phi^i(v)^{(q)},v) = 0
\;\text{for}\; 0 \leq i \leq k.
\]
The linear space \(\langle v, \phi(v), \ldots, \phi^k(v) \rangle\) is isotropic
if and only if
\[ \beta(\phi^i(v)^{(q)}, \phi^j(v)) = 0 \quad\text{for each}\; 0 \leq i, j \leq k. \]
By successively applying the identities \parref{forms-hermitian-phi-phi}
and \parref{forms-endomorphism-V-identity},
\[
\beta(\phi^i(v)^{(q)},\phi^j(v)) =
\begin{dcases*}
\beta(\phi^{i-j}(v)^{(q)},v)^{q^{2j}} & if \(i \geq j\), and \\
\beta(\phi^{j-i-1}(v)^{(q)},v)^{q^{2i+1}} & if \(i < j\).
\end{dcases*}
\]
This shows that
\[
\beta(\phi^i(v)^{(q)}, \phi^j(v)) = 0
\;\text{for}\; 0 \leq i, j \leq r
\quad\Leftrightarrow\quad
\beta(\phi^i(v)^{(q)},v) = 0
\;\text{for}\; 0 \leq i \leq r
\]
which shows that \(X^k\) is the claimed complete intersection.

Observe that, as in \parref{hypersurfaces-conormal-compute}, the conormal map
of \(X^k\) fits into a commutative diagram
\[
\begin{tikzcd}
\bigoplus_{i = 0}^k \sO_{X^k}(-q^{2i+1})
\ar[rr,"\bigoplus_{i=0}^k \phi_{\PP V}^{i,*}(\mathrm{eu})^{(q)}"'] \dar["\cong"']
&& \Fr^*(V)_{X^k} \dar["\beta^\vee"] \\
\mathcal{C}_{X^k/\PP V}(1) \rar["\delta"]
& \Omega_{\PP V}^1(1)\rvert_{X^k} \rar[hook]
& V^\vee_{X^k}\punct{.}
\end{tikzcd}
\]
The singular locus of \(X^k\) is supported on the degeneracy locus of
\(\mathcal{C}_{X^k/\PP V}(1) \to V^\vee\). Since \(\beta^\vee\) is an
isomorphism, this is the degeneracy locus of the top map. But this is a
Frobenius pullback of the sheaf map from \parref{hypersurfaces-filtration-cyclic}
defining \(\operatorname{cyc}_\phi^r \colon \PP V \dashrightarrow \mathbf{G}(r+1,V)\).
Since Frobenius pullbacks preserve the support of
degeneracy loci, the discussion there implies that the singular locus of \(X\)
is supported on the union of the Hermitian \((k-1)\)-planes in \(X\).
\end{proof}

\subsection{Filtration in Hermitian coordinates}\label{hypersurfaces-filtration-hermitian}
The equations of the \(X^k\) are particularly simple upon choosing a
basis \(V = \langle v_0,\ldots,v_n\rangle\) consisting of Hermitian vectors,
see \parref{forms-hermitian-nondegenerate-span}. Fix such a basis and let
\((x_0:\cdots:x_n)\) be the corresponding coordinates on \(\PP V = \PP^n\).
Then the Gram matrix
\(A \coloneqq (a_{ij})_{i,j=0}^n \coloneqq \Gram(\beta;v_0,\ldots,v_n)\)
of \(\beta\) with respect to such a basis is Hermitian
by \parref{forms-hermitian-gram}, meaning \(A^\vee = A^{(q)}\).
This shows that the endomorphism \(\phi_{\PP V}\) from
\parref{hypersurfaces-endomorphism-V} is given in these coordinates simply by
\[
\phi_{\PP^n} \colon \PP^n \to \PP^n
\qquad (x_0:\cdots:x_n) \mapsto (x_0^{q^2}: \cdots: x_n^{q^2}).
\]
In particular, the Hermitian subspaces of \(X\) coincide with its
\(\mathbf{F}_{q^2}\)-rational ones. The equations of
\(X^r\) from \parref{hypersurfaces-filtration-properties} are now given by
\[
X^r
= \bigcap\nolimits_{k = 0}^r
\mathrm{V}\Big(\sum\nolimits_{i,j = 0}^n a_{ij} x_i x_j^{q^{2k+1} + 1}\Big)
\subset \PP^n.
\]
When \(X\) is given by the Fermat equation, this filtration is seen to
coincide with that defined by Lusztig in \cite[Definition 2]{Lusztig:Green},
see also \cite[\S6.2]{Rodier:DL}.

\medskip

The final piece \(X^{\lfloor n/2 \rfloor}\) of the filtration from
\parref{hypersurfaces-filtration} essentially consists of the maximal Hermitian
isotropic subspaces of \((V,\beta)\):

\begin{Lemma}\label{hypersurfaces-filtration-maximal-isotropic}
The scheme \(X^{\lfloor n/2 \rfloor}\) is supported on the union of the maximal
Hermitian subspaces contained in \(X\). Furthermore, the sections
\[
\beta(\phi^{*,k}_{\PP V}(\mathrm{eu})^{(q)}, \mathrm{eu}) \colon
\sO_{\PP V} \to \sO_{\PP V}(q^{2k+1}+1)
\]
vanish on \(X^{\lfloor n/2 \rfloor}\) for all integers \(k \geq 0\).
\end{Lemma}

\begin{proof}
Write \(\dim X = n - 1\) as \(2m\) or \(2m+1\) so that a maximal isotropic
contained in \(X\) is an \(m\)-plane. It follows directly from the functorial
description of \(X^{\lfloor n/2 \rfloor}\) in \parref{hypersurfaces-filtration}
and \parref{forms-hermitian-min-contain} that it contains the Hermitian
\(m\)-planes of \(X\). Thus to prove the first statement, it remains to see
that a point \(x \in X^{\lfloor n/2 \rfloor}\) is contained in a Hermitian
\(m\)-plane of \(X\). Consider the even- and odd-dimensional cases separately:
\begin{itemize}
\item When \(n - 1 = 2m\), the definition
of \(X^{\lfloor n/2 \rfloor} = X^m\) together with \parref{forms-hermitian-min-contain}
imply that \(x\) is contained in an \(m\)-plane in \(X\), any of which is
Hermitian by \parref{hypersurfaces-cones-even-maximal-isotropic}.
\item When \(n - 1 = 2m+1\), the definition of
\(X^{\lfloor n/2 \rfloor} = X^{m+1}\) together with the fact that any linear
subvariety of \(X\) has dimension at most \(m\) imply that
\[
\langle x, \phi_X(x), \ldots, \phi_X^m(x), \phi_X^{m+1}(x) \rangle =
\langle x, \phi_X(x), \ldots, \phi_X^m(x) \rangle.
\]
Thus, by \parref{forms-hermitian-min-contain}, \(x\) lies in a Hermitian
\(m\)-plane of \(X\).
\end{itemize}

For the second statement, the sections
\(\beta(\phi_{\PP V}^{*,k}(\mathrm{eu})^{(q)}, \mathrm{eu})\) vanish on
\(X^{\lfloor n/2 \rfloor}\) for \(0 \leq k \leq n/2\) by
\parref{hypersurfaces-filtration-properties}, so it remains to show vanishing
when \(k > n/2\). So consider any \(v \in V\) contained in a maximal
isotropic Hermitian subspace. Then by \parref{forms-hermitian-min-contain},
\[
\phi^k(v) = a_0 v + a_1 \phi(v) + \cdots + a_m \phi^m(v)
\quad\text{for some}\; a_0, a_1,\ldots, a_m \in \kk.
\]
Thus \(\beta(\phi^k(v)^{(q)},v) = \sum_{i = 0}^m a_i^q \beta(\phi^i(v),v) = 0\)
by \parref{hypersurfaces-filtration-properties}. The vanishing of the
sections \(\beta(\phi^k_{\PP V}(\mathrm{eu})^{(q)}, \mathrm{eu})\) now follow
from the first part together with the functorial description of
\(X^{\lfloor m/2 \rfloor}\) from \parref{hypersurfaces-filtration}.
\end{proof}

A finer analysis of the structure of \(X^{\lfloor n/2 \rfloor}\) gives a
geometric method to count the number of maximal isotropic Hermitian subspaces
contained in a smooth \(q\)-bic hypersurface \(X\). The following count is
classical: see \cite[n.32]{Segre:Hermitian} and \cite[Theorem
9.2]{BC:Hermitian}; see also \cite[Corollary 2.22]{Shimada:Lattices}.

\begin{Corollary}\label{hypersurfaces-maximal-hermitian-subspaces}
The number of maximal isotropic Hermitian subspaces in a smooth \(q\)-bic
\((n-1)\)-fold \(X\) is
\[
\#\mathbf{F}_m(X)_{\mathrm{Herm}} =
\begin{dcases*}
\prod\nolimits_{i = 0}^m (q^{2i+1} + 1) & if \(n - 1 = 2m\) is even, and \\
\prod\nolimits_{i = 0}^m (q^{2i+3} + 1) & if \(n - 1 = 2m+1\) is odd.
\end{dcases*}
\]
\end{Corollary}

\begin{proof}
In the even case, \parref{hypersurfaces-filtration-maximal-isotropic} implies
that \(X^m\) is supported on the union of the maximal Hermitian subspaces in
\(X\). By \parref{hypersurfaces-filtration-properties}, the singular locus of
\(X^m\) is supported on the union of the Hermitian \((m-1)\)-planes in \(X\);
in particular, \(X^m\) is reduced. Therefore
\[
\#\mathbf{F}_m(X)_{\mathrm{Herm}} =
\deg(X_m) =
\prod\nolimits_{i = 0}^m (q^{2i+1}+1).
\]

In the odd case, it follows from
\parref{hypersurfaces-filtration-maximal-isotropic} and
\parref{hypersurfaces-filtration-properties} that \(X^{m+1}\) is the union of
the isotropic Hermitian \(m\)-planes and, by symmetry, each \(m\)-plane appears
with the same multiplicity \(a > 1\). Proceed by induction on \(m\) to show
that \(a = q+1\), at which point the result follows by using
\parref{hypersurfaces-filtration-properties} and counting degrees.
When \(m = 0\), so \(X\) is a smooth \(q\)-bic curve and
\(X^1\) is supported on its Hermitian points, it follows from
\parref{hypersurfaces-smooth-cone-points-count} that
\[ a = \deg(X^1)/\# X_{\mathrm{Herm}} = q+1. \]
Suppose \(m > 1\). For each \(i \geq 0\), set
\[
X_i \coloneqq
\mathrm{V}(\beta(\phi_{\PP V}^{*,i}(\mathrm{eu})^{(q)},\mathrm{eu})) \subseteq
\PP V
\]
so that each \(X_i\) is a \(q^{2i+1}\)-bic hypersurface in \(\PP V\) and
\(X^k = \bigcap_{i = 0}^k X_i\) for \(0 \leq k \leq m+1\). The computations of
\parref{hypersurfaces-filtration-hermitian} imply that Hermitian points \(x\)
of \(X\) are also Hermitian points of \(X_i\) for all \(i \geq 0\) and, by
\parref{hypersurfaces-tangent-space-as-kernel},
\(\mathbf{T}_{X,x} = \mathbf{T}_{X_i,x}\) as hyperplanes in \(\PP V\).
This, together with \parref{hypersurfaces-cone-points-smooth},
implies that \(X \cap \mathbf{T}_{X,x}\) is a cone over a \(q\)-bic
\((n-3)\)-fold \(X'\), and \(X^{m+1} \cap \mathbf{T}_{X,x}\) is a cone
over
\[ X'^{m} \cap X'_{m+1} = X'^{m} \]
where the notation is as above, and where the second equality follows from
\parref{hypersurfaces-filtration-maximal-isotropic}. Induction now gives that
any irreducible component \(P \subseteq X^{m+1}\) through \(x\) has
multiplicity at least \(q+1\).

To conclude \(a = q+1\), it remains to show that \(P\) is scheme-theoretically
contained in \(\mathbf{T}_{X,x}\). For this, let \(y \in P\) be a closed point
not lying on any Hermitian \((m-1)\)-plane of \(X\). Then \(y\) is a smooth
point of \(X^m\) by \parref{hypersurfaces-filtration-properties}. Therefore the
nonreduced structure of \(P\) at \(y\) is scheme-theoretically contained in
\(\mathbf{T}_{X^m,y} \cap X_{m+1}\). Thus it suffices to show that
\(\mathbf{T}_{X^m,y} \subset \mathbf{T}_{X,x}\) for all such \(y\). By the
computation of \parref{hypersurfaces-tangent-space-as-kernel}, it follows that
\[
\mathbf{T}_{X^m,y} =
\bigcap\nolimits_{i = 0}^m \Set{z \in \PP V | \beta(\phi^i(y)^{(q)}, z) = 0}.
\]
Since \(P = \langle y,\phi(y),\ldots,\phi^m(y) \rangle\) by
\parref{forms-hermitian-min-contain} and since \(x \in P\), the
equation \(z \mapsto \beta(x^{(q)},z)\) defining \(\mathbf{T}_{X,x}\) vanishes
on \(\mathbf{T}_{X^m,y}\), as desired.
\end{proof}

For each \(0 \leq r < n/2\), set
\(\mathbf{F}_r(X)_{\mathrm{cyc}} \coloneqq \mathbf{F}_r(X) \cap \mathbf{G}(r+1,V)_{\mathrm{cyc}}\).
The functorial description of \(X^r\) from \parref{hypersurfaces-filtration}
shows that the rational map from \parref{hypersurfaces-filtration-cyclic} fits
into a commutative diagram
\[
\begin{tikzcd}
X^r \rar[dashed,"\operatorname{cyc}_\phi^r"'] \dar[hook] & \mathbf{F}_r(X)_{\mathrm{cyc}} \dar[hook] \\
\PP V \rar[dashed,"\operatorname{cyc}_\phi^r"] & \mathbf{G}(r+1,V)_{\mathrm{cyc}}
\end{tikzcd}
\]
and that the diagram is Cartesian upon restricting to the complement of
the Hermitian \((r-1)\)-planes on the left. Together with
\parref{hypersurfaces-filtration-cyclic-composition}, this shows the
first two statements of:

\begin{Corollary}\label{hypersurfaces-filtration-rational-map}
For each \(0 \leq r < n/2\), the map
\(\operatorname{cyc}_\phi^r \colon X^r \dashrightarrow \mathbf{F}_r(X)_{\mathrm{cyc}}\)
is dominant and injective away from the Hermitian \(r\)-planes in \(X^r\), and
\[
\dim \mathbf{F}_r(X)_{\mathrm{cyc}} =
\begin{dcases*}
n-r-1 & if \(r < (n-1)/2\), and \\
0 & if \(n - 1 = 2m\) and \(r = m\).
\end{dcases*}
\]
In particular, if \(n - 1 = 2m\) or \(n - 1 = 2m+1\), then
\(\mathbf{F}_m(X)_{\mathrm{cyc}} = \mathbf{F}_m(X)\).
\end{Corollary}

\begin{proof}
Consider the dimension of \(\mathbf{F}_r(X)_{\mathrm{cyc}}\). When \(r < (n-1)/2\),
then the Hermitian \(r\)-planes contained in \(X^r\) is a proper closed subscheme.
Therefore \(X^r \dashrightarrow \mathbf{F}_r(X)_{\mathrm{cyc}}\) is dominant
and generically finite by the first statement, so
\[ \dim \mathbf{F}_r(X)_{\mathrm{cyc}} = \dim X^r = n-r-1 \]
by \parref{hypersurfaces-filtration-properties}. The remaining case is when
\(n - 1 = 2m\) and \(r = m\). Then, as in \parref{hypersurfaces-filtration-maximal-isotropic},
\(X^m\) is the union of the isotropic Hermitian \(m\)-planes of \(\PP V\)
and the map \(X^m \dashrightarrow \mathbf{F}_m(X)_{\mathrm{cyc}}\) collapses
each \(m\)-plane to a point. In this case,
\(\dim \mathbf{F}_m(X)_{\mathrm{cyc}} = 0\). Finally, when \(r = m\) with
\(n - 1 = 2m+1\), a comparison with
\parref{hypersurfaces-smooth-fano} shows that
\[ \dim \mathbf{F}_m(X) = \dim \mathbf{F}_m(X)_{\mathrm{cyc}}. \]
Therefore \(\mathbf{F}_m(X)_{\mathrm{cyc}}\) is an irreducible component of
\(\mathbf{F}_m(X)\). Since the Fano scheme is irreducible,
the two schemes coincide.
\end{proof}

The proof of \parref{hypersurfaces-filtration-rational-map} also shows that
when \(0 \leq r < (n-1)/2\), the rational map from \parref{hypersurfaces-filtration-Gr-cyc}
fits into a commutative diagram
\[
\begin{tikzcd}[column sep=4em]
\mathbf{F}_r(X)_{\mathrm{cyc}} \rar[dashed,"{[\id \cap \phi^r]}"'] \dar[hook] &
X^r \dar[hook] \\
\mathbf{G}(r+1,V)_{\mathrm{cyc}} \rar[dashed,"{[\id \cap \phi^r]}"] &
\PP V
\end{tikzcd}
\]
such that the horizontal maps are defined away from the locus parameterizing
subspaces which contain a \(2\)-dimensional Hermitian subspace of \(V\).
Furthermore, as in \parref{hypersurfaces-filtration-cyclic-composition},
its composite with \(\operatorname{cyc}_\phi^r\) from above are the endomorphisms
\[
[\id \cap \phi^r] \circ \operatorname{cyc}^r_\phi
= \phi_{X^r}^r
\quad\text{and}\quad
\operatorname{cyc}^r_\phi \circ [\id \cap \phi^r]
= \phi_{\mathbf{F}_r(X)_{\mathrm{cyc}}}^r
\]
of \(X^r\) and \(\mathbf{F}_r(X)_{\mathrm{cyc}}\) induced by \(\phi\), as in
\parref{hypersurfaces-endomorphism}. To describe a resolution of these rational
maps, with the notation as from \parref{hypersurfaces-filtration-resolution},
set
\begin{align*}
\tilde{X}^r
& \coloneqq
  X^r \times_{\PP V}
  \Gamma_{\operatorname{cyc}^r_\phi} \times_{\mathbf{G}(r+1,V)}
  \mathbf{F}_r(X)_{\mathrm{cyc}}
= \set{(x \in P) | \langle x, \phi_X(x),\ldots,\phi_X^r(x) \rangle \subset P},\\
\tilde{\mathbf{F}}_r(X)
& \coloneqq
  X^r \times_{\PP V}
  \Gamma_{[\id \cap \phi^r]} \times_{\mathbf{G}(r+1,V)}
  \mathbf{F}_r(X)_{\mathrm{cyc}}
= \set{(x \in P) | x \in P \cap \phi_X^r(P)}.
\end{align*}
Then \(\phi_X\) and \(\phi_{\mathbf{F}_r(X)_{\mathrm{cyc}}}\) induce morphisms
\begin{align*}
\phi^r \times \id & \colon
\tilde{X}^r \to \tilde{\mathbf{F}}_r(X),
& (x \in P) & \mapsto (\phi^r(x) \in P), \\
\id \times \phi^r & \colon \tilde{\mathbf{F}}_r(X) \to \tilde{X}^r,
& (x \in P) & \mapsto (x \in \phi^r(P)).
\end{align*}
As before, these give resolutions of the rational maps in the following sense:

\begin{Proposition}\label{hypersurfaces-filtration-rational-map-resolution}
There is a commutative diagram of schemes over \(\kk\) given by
\[
\begin{tikzcd}[column sep=3em]
\tilde{X}^r
  \rar["\phi^r \times \id"']
  \dar["\pr_1"'] &
\tilde{\mathbf{F}}_r(X)_{\mathrm{cyc}}
  \rar["\id \times \phi^r"']
  \dar["\pr_2"'] &
\tilde{X}^r
  \dar["\pr_1"] \\
X^r
  \rar[dashed,"\operatorname{cyc}^r_\phi"] &
\mathbf{F}_r(X)_{\mathrm{cyc}}
  \rar[dashed,"{[\id \cap \phi^r]}"] &
X^r
\end{tikzcd}
\]
such that
\begin{enumerate}
\item\label{hypersurfaces-filtration-rational-map-filtration.X}
\(\pr_1 \colon \tilde{X}^r \to X^r\) is an isomorphism away from the union of
the Hermitian \((r-1)\)-planes contained in \(X\);
\item\label{hypersurfaces-filtration-rational-map-filtration.F}
\(\pr_2 \colon \tilde{\mathbf{F}}_r(X) \to \mathbf{F}_r(X)\) is an isomorphism
away from the locus parameterizing isotropic subspaces of \(V\) which contain a
\(2\)-dimensional Hermitian subspace;
\item\label{hypersurfaces-filtration-rational-map-filtration.degree}
\(\phi^r \times \id\) and \(\id \times \phi^r\) are finite purely inseparable
of degree \(q^{r(r+1)}\) and \(q^{r(2n-3r-3)}\).
\end{enumerate}
\end{Proposition}

\begin{proof}
Items \ref{hypersurfaces-filtration-rational-map-filtration.X} and
\ref{hypersurfaces-filtration-rational-map-filtration.F} follow from their
counterparts in \parref{hypersurfaces-filtration-graph-resolution}. For
\ref{hypersurfaces-filtration-rational-map-filtration.degree}, that they are
finite purely inseparable is because they are restrictions of finite purely
inseparable morphisms, see
\parref{hypersurfaces-filtration-graph-resolution}\ref{hypersurfaces-filtration-graph-resolution.insep}.
Since the preimage of \((x \in P)\) under \(\phi^r \times \id\) parameterizes
deformations of the point \(x\) in \(P\); since \(P\) is totally isotropic,
the computation given in
\parref{hypersurfaces-filtration-graph-resolution}\ref{hypersurfaces-filtration-graph-resolution.insep}
still applies to show that \(\deg(\phi^r \times \id) = q^{r(r+1)}\). Since
the degree of \((\id \times \phi^r) \circ (\phi^r \times \id)\) is
\(q^{2r(n-1-r)}\), this implies \(\deg(\id \times \phi^r) = q^{r(2n-3r-3)}\).
\end{proof}

%% file: lowdim.tex
\chapter{\texorpdfstring{\(q\)}{q}-bic Points, Curves, and Surfaces}
\label{chapter-lowdim}

Points, curves, and surfaces amongst \(q\)-bic hypersurfaces offer a variety of
simple examples with which to illustrate the general theory. The Sections that
follow step through the projective equivalence classes of low-dimensional
\(q\)-bics which are not cones, and discuss particular features of their
geometry. In most cases, automorphism group schemes are explicitly presented;
schemes of cone points are often discussed, and are used to illustrate how
various types deform into one another; and, for \(q\)-bic surfaces, schemes of
lines are described.

Throughout this Chapter, \(\kk\) is an algebraically closed field
of characteristic \(p > 0\).

\section{\texorpdfstring{\(q\)}{q}-bic points}\label{lowdim-qbic-points}
Much of the special projective geometry of \(q\)-bic hypersurfaces is related to the
fact that \(q\)-bic hypersurfaces of dimension \(0\), or \emph{\(q\)-bic points},
come in only three shapes. Indeed, the classification of \(q\)-bic forms
\parref{forms-classification-theorem} shows that:

\begin{Proposition}\label{qbic-points-classification}
Let \((V,\beta)\) be a nonzero \(q\)-bic form of dimension \(2\) and let
\(X \subset \PP V\) be the associated \(q\)-bic hypersurface of dimension \(0\).
Then either
\[
\mathrm{type}(\beta) =
\begin{dcases*}
\mathbf{1}^{\oplus 2} \\
\mathbf{N}_2 \\
\mathbf{1} \oplus \mathbf{0}
\end{dcases*}
\quad\text{and}\quad
X \cong
\begin{dcases*}
\mathrm{V}(x_0^{q+1} + x_1^{q+1}) \\
\mathrm{V}(x_0^q x_1) \\
\mathrm{V}(x_0^{q+1})
\end{dcases*}
\]
for some choice of coordinates \((x_0:x_1)\) on \(\PP V = \PP^1\).
\qed
\end{Proposition}

Notably, this implies that any multiple point in a scheme of \(q\)-bic
points must appear with multiplicity at least \(q\). The significance of this
observation arises in conjunction with the fact
\parref{hypersurface-hyperplane-section} that linear sections of \(q\)-bic
hypersurfaces are \(q\)-bics: it implies that tangent lines to \(q\)-bic
hypersurfaces always have contact order at least \(q\) at the point of
tangency. Consequently, singular points of \(q\)-bics are points of
multiplicity at least \(q\).

The shapes of the \(q\)-bic points make it clear how they may specialize to
one another in families. To make this precise, consider the parameter space
\[
\qbics(V)
\coloneqq \mathbf{A}(\Fr^*(V)^\vee \otimes V^\vee)
\cong \mathbf{A}^4
\]
of \(q\)-bic forms on the \(2\)-dimensional vector space \(V\), as in
\parref{forms-classification-moduli}. The closure relations and strata
dimensions of the type stratification from
\parref{forms-classification-type-stratification} are as follows:

\begin{Proposition}\label{qbic-points-moduli}
The Hasse diagram for the type stratification of \(\qbics(V)\) is
\[
\begin{tikzcd}[row sep=0.1em, column sep=0.8em]
\mathbf{1}^{\oplus 2} \rar[symbol={\rightsquigarrow}]
& \mathbf{N}_2 \rar[symbol={\rightsquigarrow}]
& \mathbf{0} \oplus \mathbf{1} \rar[symbol={\rightsquigarrow}]
& \mathbf{0}^{\oplus 2}
\end{tikzcd}
\]
and the strata dimension are given by
\[
\begin{array}{c|cccc}
\lambda
& \mathbf{1}^{\oplus 2}
& \mathbf{N}_2
& \mathbf{0} \oplus \mathbf{1}
& \mathbf{0}^{\oplus 2} \\
\hline
\dim\qbics(V)_\lambda
& 4
& 3
& 2
& 0
\end{array}
\]
\end{Proposition}

\begin{proof}
The first two closure relations are because \(\mathbf{N}_2\) is the generic type
in corank \(1\), see \parref{forms-aut-general-corank-b}. Since everything
specializes to the zero form, the remaining closure relation is clear. The
strata dimensions are computed using \parref{forms-aut-strata-dimension}; see
\parref{qbic-points-automorphisms.1+1} and
\parref{qbic-points-automorphisms.N2} for concrete computations of the
automorphism group schemes in the first two types, and apply
\parref{forms-automorphisms-cones} for type \(\mathbf{0} \oplus \mathbf{1}\).
\end{proof}

\subsection{Special families of \texorpdfstring{\(q\)}{q}-bic points}\label{qbic-points-families}
Flat families of \(q\)-bic points for which the general member has type
\(\mathbf{1}^{\oplus 2}\) and special members have type \(\mathbf{N}_2\)
are thought of as degenerations. It is clear from \parref{qbic-points-moduli}
that there are many such families, even over \(1\)-dimensional bases. A
geometrically simple class of such degenerations involves fixing one point and
letting the remaining \(q\) come together. For example, some degenerations of
this form over \(\mathbf{A}^1 = \Spec(\kk[t])\) can be given by a \(q\)-bic
form over \(\kk[t]\) with Gram matrix
\[
\begin{pmatrix} 0 & 1 \\ t & f \end{pmatrix} \colon
\Fr^*(V[t]) \otimes_{\kk[t]} V[t] \to \kk[t]
\]
where \(V[t] \coloneqq V \otimes_\kk  \kk[t]\), \(f \in \kk[t]\), and
the fixed point corresponds to the subspace spanned by
\(\left(\begin{smallmatrix} 1 \\ 0 \end{smallmatrix}\right)\).

The family in which \(f = 0\) is particularly special, and arises
geometrically by fixing an additional point amongst the original set of
\(q\)-bic points. Algebraically, this is the following:

\begin{Proposition}\label{qbic-points-basic-algebra-family}
Let \((V,\beta)\) be a \(q\)-bic form of type \(\mathbf{1}^{\oplus 2}\). For
every decomposition \(V = L_- \oplus L_+\) into two isotropic
subspaces of dimension \(1\), there exists a unique \(q\)-bic form
\[ \beta^{L_\pm} \colon \Fr^*(V[t]) \otimes_{\kk[t]} V[t] \to \kk[t] \]
over \(\kk[t]\) such that
\begin{enumerate}
\item\label{qbic-points-basic-algebra-family.isotropic}
\(L_-[t]\) and \(L_+[t]\) are isotropic for \(\beta^{L_\pm}\);
\item\label{qbic-points-basic-algebra-family.isomorphism}
the induced pairing \(\Fr^*(L_-[t]) \otimes_{\kk[t]}  L_+[t] \to \kk[t]\) is perfect;
\item\label{qbic-points-basic-algebra-family.degenerate}
the induced pairing \(\Fr^*(L_+[t]) \otimes_{\kk[t]} L_-[t] \to \kk[t]\) has
image the ideal \((t)\); and
\item\label{qbic-points-basic-algebra-family.restrict}
\(\beta^{L_\pm}\rvert_{t = 0}\) is of type \(\mathbf{N}_2\) and
\(\beta^{L_\pm}\rvert_{t = 1} = \beta\).
\end{enumerate}
\end{Proposition}

\begin{proof}
Choose a basis \(L_\pm = \langle v_\pm \rangle\) so that \(\beta\) has Gram
matrix
\(\left(\begin{smallmatrix} 0 & 1 \\ 1 & 0 \end{smallmatrix}\right)\): see
\parref{qbic-points-1+1.basis} below.
Then the \(q\)-bic form on \(V[t]\) with Gram matrix
\[
\Gram(\beta^{L_\pm}; v_- \otimes 1, v_+ \otimes 1) =
\begin{pmatrix} 0 & 1 \\ t & 0 \end{pmatrix}
\]
is the unique form satisfying the stated properties.
\end{proof}

\subsection{\(\mathbf{G}_m\)-action}\label{qbic-points-family-torus-action}
The \(q\)-bic \((V[t],\beta^{L_\pm})\) over \(\kk[t]\) in
\parref{qbic-points-basic-algebra-family} defines a scheme
\[
\mathcal{X} \subset
\PP^1 \times \mathbf{A}^1 \coloneqq
\PP V \times \Spec(\kk[t])
\]
which is viewed as a family of \(q\)-bic points over \(\mathbf{A}^1\).
In the coordinates \((x_-:x_+)\) of \(\PP V = \PP^1\) dual to the basis chosen
above, \(\mathcal{X} = \mathrm{V}(x_-^q x_+ + t x_- x_+^q)\).
This degeneration is special because the total space \(\mathcal{X}\) admits a
\(\mathbf{G}_m\)-action that is compatible with a \(\mathbf{G}_m\)-action on
the base \(\mathbf{A}^1\). Namely, let \(\mathbf{G}_m\) act linearly on \(\PP^1
\times \mathbf{A}^1\) with
\[
\mathrm{wt}(L_-) = -1,
\quad
\mathrm{wt}(L_+) = q,
\quad
\mathrm{wt}(t) = q^2-1.
\]
In terms of the coordinates \((x_-:x_+)\), \(\lambda \in \mathbf{G}_m\) acts by
\[
\lambda \cdot \big((x_-:x_+), t\big) =
\big((\lambda x_-: \lambda^{-q} x_+), \lambda^{q^2-1} t\big).
\]
Then \(\mathcal{X}\) is stable for the \(\mathbf{G}_m\)-action. In fact,
slightly more is true:

\begin{Lemma}\label{qbic-points-family-invariant-form}
The \(q\)-bic form \(\beta^{L_\pm}\) is invariant for the
\(\mathbf{G}_m\)-action of \parref{qbic-points-family-torus-action}.
\end{Lemma}

\begin{proof}
In terms of the coordinates chosen in \parref{qbic-points-basic-algebra-family},
the action of \(\lambda \in \mathbf{G}_m\) on \(\beta^{L_\pm}\) is
\[
\lambda \cdot
\begin{pmatrix} 0 & 1 \\ t & 0 \end{pmatrix} =
\begin{pmatrix} \lambda^q & 0 \\ 0 & \lambda^{-q^2} \end{pmatrix}
\begin{pmatrix} 0 & 1 \\ \lambda^{q^2-1}t & 0 \end{pmatrix}
\begin{pmatrix} \lambda & 0 \\ 0 & \lambda^{-q} \end{pmatrix} =
\begin{pmatrix} 0 & 1 \\ t & 0 \end{pmatrix}.
\qedhere
\]
\end{proof}

\subsection{}\label{qbic-points-family-automorphisms}
The automorphism group scheme \(\AutSch(V[t], \beta^{L_\pm})\), see
\parref{forms-aut-schemes}, of the \(q\)-bic form over \(\kk[t]\) from
\parref{qbic-points-basic-algebra-family} has two \(1\)-dimensional irreducible
components. Let
\[
\mathcal{G} \cong
\Set{
\begin{pmatrix} \lambda & \epsilon \\ 0 & \lambda^{-q} \end{pmatrix}
\in \GL_{2,\mathbf{A}_1}
|
\lambda \in \boldsymbol{\mu}_{q^2-1},
\epsilon^q + t\lambda^{q-1}\epsilon = 0
}
\]
be the component that dominates \(\mathbf{A}^1\); compare with
\parref{qbic-points-automorphisms-parabolic},
\parref{qbic-points-automorphisms.N2}, and
\parref{forms-aut-canonical-filtration}. The linear action of \(\mathbf{G}_m\)
on \(\PP^1 \times \mathbf{A}^1\) induces one on \(\mathcal{G}\) given by
\(\lambda \cdot \left(\begin{smallmatrix} a & b \\ 0 & a^{-q} \end{smallmatrix}\right) =
\left(\begin{smallmatrix} a & \lambda^{q+1}b \\ 0 & a^{-q} \end{smallmatrix}\right)\).
The action of \(\mathcal{G}\) is compatible with this \(\mathbf{G}_m\)-action:

\begin{Lemma}\label{qbic-points-family-equivariant-action}
The action map \(\mathcal{G} \times_{\mathbf{A}^1} \mathcal{X} \to \mathcal{X}\)
is \(\mathbf{G}_m\)-equivariant over \(\mathbf{A}^1\).
\end{Lemma}

\begin{proof}
Since \(\mathbf{G}_m\) acts linearly on \(\PP^1 \times \mathbf{A}^1\)
the action map
\[
\mathbf{GL}_{2,\mathbf{A}^1} \times_{\mathbf{A}^1} \PP^1_{\mathbf{A}^1} \to \PP^1_{\mathbf{A}^1}
\]
is \(\mathbf{G}_m\)-equivariant. The claim now follows as
the action map \(\mathcal{G} \times_{\mathbf{A}^1} \mathcal{X} \to \mathcal{X}\)
is simply the restriction of linear action map from the ambient projective space.
\end{proof}

\section{Type \texorpdfstring{\(\mathbf{1}^{\oplus 2}\)}{1+1}}\label{qbic-points-1+1}
Let \(X\) be \(q\)-bic points associated with a \(q\)-bic form \((V,\beta)\) of
type \(\mathbf{1}^{\oplus 2}\). Such \(X\) is as simple as possible: it is a
set of \(q+1\) reduced points on the projective line. Of course, not any set of
\(q+1\) points on the line determines a scheme of \(q\)-bic points; the points
must be arranged in a particularly symmetric way. One way to make sense of
this is the following, which dictates how, upon choosing two points to serve
as \(0\) and \(\infty\) for \(\PP V = \PP^1\), the remaining \(q-1\) points
must be distributed:

\begin{Lemma}\label{qbic-points-1+1.basis}
For any decomposition \(V = L_- \oplus L_+\) into a pair of isotropic
\(1\)-dimensional subspaces and any \(a,b \in \kk^\times\), there exists a
basis
\(L_\pm = \langle v_\pm\rangle\) such that
\[
\Gram(\beta; v_-,v_+) = \begin{pmatrix} 0 & a \\ b & 0 \end{pmatrix}.
\]
\end{Lemma}

\begin{proof}
Begin with any basis \(L_\pm = \langle v_\pm' \rangle\). Then the associated
Gram matrix is
\(\left(\begin{smallmatrix} 0 & a' \\ b' & 0 \end{smallmatrix}\right)\) for
some \(a',b' \in \kk^\times\). Scaling by
\(\lambda_\pm \in \kk^\times\), affects the Gram matrix as
\[
\Gram(\beta; \lambda_-v_-', \lambda_+v_+') =
\begin{pmatrix}
0
& \lambda_+^q \lambda_- a' \\
\lambda_+ \lambda_-^q b'
& 0
\end{pmatrix}.
\]
Take any solution to \(\lambda_+^{q^2-1} = ab'/a'b\), let
\(\lambda_- \coloneqq a/a'\lambda_+^q\), and set
\(v_\pm \coloneqq \lambda_\pm v_\pm'\).
\end{proof}

Taking \(a = -b = 1\) gives coordinates \((x_0:x_1)\)
so that \(X = \mathrm{V}(x_0^q x_1 - x_0 x_1^q)\) is the set of
\(\mathbf{F}_q\) points of \(\PP V = \PP^1\). This gives a pleasant
presentation for its group of linear automorphisms of \(X\):

\begin{Proposition}\label{qbic-points-automorphisms.1+1}
Let \((V,\beta)\) be a \(q\)-bic form of type \(\mathbf{1}^{\oplus 2}\). Then
its automorphism group scheme admits the presentation:
\[
\AutSch(V,\beta) \cong
\Set{
\begin{pmatrix}
\lambda a & \lambda^{-q} b \\
\lambda c & \lambda^{-q} d
\end{pmatrix}
|
\lambda \in \boldsymbol{\mu}_{q^2-1},
a,b,c,d \in \mathbf{F}_q,
ad - bc = 1}.
\]
\end{Proposition}

\begin{proof}
Choose a basis of \(V\) as in \parref{qbic-points-1+1.basis} so that
\(a = -b = 1\). Then \(\AutSch(V,\beta)\) is isomorphic to the closed
subgroup scheme of \(\mathbf{GL}_2\) consisting of matrices satisfying
\[
\begin{pmatrix} 0 & 1 \\ -1 & 0 \end{pmatrix} =
\begin{pmatrix} a_0^q & c_0^q \\ b_0^q & d_0^q \end{pmatrix}
\begin{pmatrix} 0 & 1 \\ -1 & 0 \end{pmatrix}
\begin{pmatrix} a_0 & b_0 \\ c_0 & d_0 \end{pmatrix} =
\begin{pmatrix}
a_0^q c_0 - a_0 c_0^q & a_0^q d_0 - b_0 c_0^q \\
b_0^q c_0 - a_0 d_0^q & b_0^q d_0 - b_0 d_0^q
\end{pmatrix}.
\]
Since the columns of the matrices are linearly independent,
\(a_0^q c_0 - a_0 c_0^q = 0\) implies \((a_0:c_0) \in \PP^1(\mathbf{F}_q)\).
Similarly, \((b_0:d_0) \in \PP^1(\mathbf{F}_q)\). Therefore
\[
\begin{pmatrix} a_0 & b_0 \\ c_0 & d_0 \end{pmatrix} =
\begin{pmatrix} \lambda a & \mu b' \\ \lambda c & \mu d' \end{pmatrix}
\quad\text{for some}\; a, b', c, d' \in \mathbf{F}_q
\;\text{and}\; \lambda, \mu \in \mathbf{G}_m.
\]
Since elements of \(\mathbf{F}_q\) are fixed under \(q\)-powers, the remaining
equations may be written \(\lambda^q \mu \Delta = \lambda \mu^q \Delta = 1\)
where \(\Delta \coloneqq ad' - b'c\). Thus \(\mu = 1/\lambda^q\Delta\) and,
since \(\Delta^q = \Delta\), \(\lambda^{q^2-1} = 1\). Setting
\(b \coloneqq b'/\Delta\) and \(d \coloneqq d'/\Delta\) gives the
presentation.
\end{proof}

The presentation can be phrased more invariantly as follows:

\begin{Corollary}\label{qbic-points-automorphisms.1+1.presentation}
There is a short exact sequence of groups
\[
1 \to
\boldsymbol{\mu}_{q-1} \to
\boldsymbol{\mu}_{q^2-1} \ltimes \SL_2(\mathbf{F}_q) \to
\AutSch(V,\beta) \to 1
\]
where the action of \(\boldsymbol{\mu}_{q^2-1}\) on \(\SL_2(\mathbf{F}_q)\) in
the product and the inclusion of \(\boldsymbol{\mu}_{q-1}\) are
\[
\lambda \cdot
\begin{pmatrix}
a & b \\
c & d
\end{pmatrix}
\coloneqq
\begin{pmatrix}
a & \lambda^{q+1} b \\
\lambda^{-q-1} c & d
\end{pmatrix}
\quad\text{and}\quad
\zeta \mapsto
\left(\zeta,
\begin{pmatrix}
\zeta^{-1} & 0 \\
0 & \zeta
\end{pmatrix}\right).
\]
\end{Corollary}

\begin{proof}
The presentation of \parref{qbic-points-automorphisms.1+1} means that the map
of sets
\[
\varphi \colon
\boldsymbol{\mu}_{q^2-1} \ltimes \SL_2(\mathbf{F}_q) \to \AutSch(V,\beta),
\quad
\left(\lambda, \left(\begin{smallmatrix} a & b \\ c & d \end{smallmatrix}\right)\right) \mapsto
\left(\begin{smallmatrix} \lambda a & \lambda^{-q} b \\ \lambda c & \lambda^{-q} d \end{smallmatrix}\right)
\]
is a surjection. It is a homomorphism: write
\[
\left(
\begin{smallmatrix}
\lambda_1 a_1 & \lambda_1^{-q} b_1 \\
\lambda_1 c_1 & \lambda_1^{-q} d_1
\end{smallmatrix}
\right)
\left(
\begin{smallmatrix}
\lambda_2 a_2 & \lambda_2^{-q} b_2 \\
\lambda_2 c_2 & \lambda_2^{-q} d_2
\end{smallmatrix}
\right)
=
\varphi\left(
\lambda_1\lambda_2,
\left(
\begin{smallmatrix}
a_1a_2 + \lambda_1^{-q-1}b_1c_2 &
\lambda_1^{q+1} a_1b_2 + b_1d_2 \\
c_1a_2 + \lambda_1^{-q-1}d_1c_2 &
\lambda_1^{q+1} c_1a_2 + d_1d_2
\end{smallmatrix}
\right)
\right)
\]
and observe that the matrix on the right is the product
\[
\begin{pmatrix}
a_1 & b_1 \\
c_1 & d_1
\end{pmatrix}
\begin{pmatrix}
a_2 & \lambda_1^{q+1}b_2 \\
\lambda_1^{-q-1}c_2 & d_2
\end{pmatrix}
=
\begin{pmatrix}
a_1 & b_1 \\
c_1 & d_1
\end{pmatrix}
\left(
\lambda_1 \cdot
\begin{pmatrix}
a_2 & b_2 \\ c_2 & d_2
\end{pmatrix}
\right).
\]
The kernel of \(\varphi\) consists of pairs
\(\left(\lambda, \left(\begin{smallmatrix} a & 0 \\ 0 & d \end{smallmatrix}\right)\right)\)
such that \(\lambda a = \lambda^{-q}d = 1\). Thus \(\lambda = a^{-1}\),
so \(\lambda \in \mathbf{F}_q^\times\) and thus
\(d = \lambda = a^{-1}\). Identifying \(\mathbf{F}_q^\times\) with
\(\boldsymbol{\mu}_{q-1}\) gives the statement.
\end{proof}

An interesting subgroup scheme of \(\AutSch(V,\beta)\) is the stabilizer of
a given isotropic line; in other words, these are the linear automorphisms
fixing a chosen point of \(X\).

\begin{Lemma}\label{qbic-points-automorphisms-parabolic}
Let \(L \subset V\) be a \(1\)-dimensional isotropic subspace. Then
\[
\AutSch(L \subset V, \beta) \coloneqq
\set{g \in \AutSch(V,\beta) | g \cdot L = L} \cong
\boldsymbol{\mu}_{q^2-1} \ltimes \mathbf{F}_q.
\]
\end{Lemma}

\begin{proof}
Choose a basis as in the computation of \parref{qbic-points-automorphisms.1+1}
such that, in addition, \(L\) is the span of
\(\left(\begin{smallmatrix} 1 \\ 0 \end{smallmatrix}\right)\). Then by inspection,
\[
\AutSch(L \subset V, \beta) =
\Set{
\left(
\begin{smallmatrix}
\lambda a & \lambda^{-q}b \\
0 & \lambda^{-q}a^{-1}
\end{smallmatrix}
\right)
|
\lambda \in \boldsymbol{\mu}_{q^2-1},
a \in \mathbf{F}_q^\times,
b \in \mathbf{F}_q}.
\]
Since \(a^q = a\), setting \(\lambda \mapsto a\lambda\)
gives the isomorphism with \(\boldsymbol{\mu}_{q^2-1} \ltimes \mathbf{F}_q\).
\end{proof}

\section{Type \texorpdfstring{\(\mathbf{N}_2\)}{N2}}\label{qbic-points-N2}
When \(q\) of the points of a reduced scheme of \(q\)-bic points collide,
the result is one of type \(\mathbf{N}_2\). Such a scheme \(X\) is supported
on two points, corresponding to the kernels of the underlying
\(q\)-bic form \((V,\beta)\):
\[
L_- \coloneqq \Fr^*(V)^\perp
\quad\text{and}\quad
L_+ \coloneqq \Fr^{-1}(V^\perp).
\]
Since \(\kk\) is perfect, these spaces span \(V\): see
\parref{forms-aut-1^a+N2^b} and also \parref{forms-classification-imperfect}.
Then by \parref{hypersurfaces-nonsmooth-locus},
\begin{itemize}
\item \(L_-\) underlies the reduced point of \(X\), and
\item \(L_+\) underlies the \(q\)-fold point of \(X\).
\end{itemize}
Concretely, choose basis vectors \(L_\pm = \langle v_\pm \rangle\) such that
\(\beta(v_-^{(q)}, v_+) = 1\) and let \(x_\pm \in L_\pm^\vee\) be
the dual coordinate. This choice of coordinates on \((x_-:x_+)\) on
\(\PP V = \PP^1\) gives the equation \(X = \mathrm{V}(x_-^q x_+)\) and
identifies the points above directly. In these coordinates, the automorphism
group scheme of \((V,\beta)\) admits the following simple presentation:

\begin{Proposition}\label{qbic-points-automorphisms.N2}
Let \((V,\beta)\) be a \(q\)-bic form of type \(\mathbf{N}_2\). Then its
automorphism scheme is isomorphic to the \(1\)-dimensional closed subscheme of
\(\GL_2\) given by
\[
\AutSch(V,\beta) \cong
\Set{
\begin{pmatrix}
\lambda & \epsilon \\
0 & \lambda^{-q}
\end{pmatrix}
| \lambda \in \mathbf{G}_m, \epsilon \in \boldsymbol{\alpha}_q}.
\]
\end{Proposition}

\begin{proof}
This is a special case of \parref{forms-aut-1^a+N2^b.computation}.
For a direct computation, use the coordinates as above, and note
that any automorphism preserves \(L_-\) and \(\Fr^*(L_+)\), so
\(\AutSch(V,\beta)\) is isomorphic to the closed subgroup scheme of
\(\mathbf{GL}_2\) consisting of matrices satisfying
\[
\begin{pmatrix} 0 & 1 \\ 0 & 0 \end{pmatrix} =
\begin{pmatrix} \lambda^q & 0 \\ 0 & \mu^q \end{pmatrix}
\begin{pmatrix} 0 & 1 \\ 0 & 0 \end{pmatrix}
\begin{pmatrix} \lambda & \epsilon \\ 0 & \mu \end{pmatrix} =
\begin{pmatrix} 0 & \lambda^q \mu \\ 0 & 0 \end{pmatrix}
\]
with \(\epsilon^q = 0\). Therefore \(\mu = \lambda^{-q}\) as required.
\end{proof}

\section{\texorpdfstring{\(q\)}{q}-bic curves}\label{section-lowdim-curves}
Plane curves that are \(q\)-bic hypersurfaces are \emph{\(q\)-bic curves}. The
classification of \(q\)-bic forms in \parref{forms-classification-theorem}
shows that, there are three isomorphism classes of \(q\)-bic curves which are
not cones; see also \cite[p.213]{HH:Fermat} for a direct computation.

\begin{Proposition}\label{qbic-curves-classification}
Let \(X\) be the \(q\)-bic curve associated with a nonzero \(q\)-bic form
\((V,\beta)\) of dimension \(3\). Then either \(X\) is a cone over \(q\)-bic
points or
\[
\mathrm{type}(\beta) =
\begin{dcases*}
\mathbf{1}^{\oplus 3} \\
\mathbf{N}_2 \oplus \mathbf{1} \\
\mathbf{N}_3
\end{dcases*}
\quad\text{and}\quad
X \cong
\begin{dcases*}
\mathrm{V}(x_0^{q+1} + x_1^{q+1} + x_2^{q+1}), \;\text{or} \\
\mathrm{V}(x_0^q x_1 + x_2^{q+1}), \;\text{or} \\
\mathrm{V}(x_0^q x_1 + x_1^q x_2),
\end{dcases*}
\]
for some choice of coordinates \((x_0:x_1:x_2)\) on \(\PP V = \PP^2\). \qed
\end{Proposition}

Either by direct inspection, or else by considering linear projection in view
of the shapes of \(q\)-bic points from \parref{qbic-points-classification}, it
follows that:

\begin{Lemma}\label{qbic-curves-components-are-rational}
Reduced irreducible components of singular \(q\)-bic curves are rational.
\end{Lemma}

\begin{proof}
Looking at the classification given in \parref{qbic-curves-classification}, it
suffices to observe that
\[
\PP^1 \to \mathrm{V}(x_0^d + x_1^{d-1} x_2),
\quad\quad
(t_0:t_1) \mapsto (t_0^{d-1} t_1: t_0^d: -t_1^d)
\]
is the normalization onto the irreducible degree \(d=q+1\) curve of type
\(\mathbf{N}_2 \oplus \mathbf{1}\), and that onto the degree \(d = q\)
irreducible component of the curve of type \(\mathbf{N}_3\).
\end{proof}

As with \(q\)-bic points in \parref{qbic-points-moduli}, the type
stratification on the \(9\)-dimensional affine parameter space \(\qbics(V)\),
see \parref{forms-classification-moduli} and
\parref{forms-classification-type-stratification}, has a totally ordered Hasse
diagram:

\begin{Proposition}\label{qbic-curves-moduli}
The Hasse diagram for the type stratification of \(\qbics(V)\) is
\[
\begin{tikzcd}[row sep=0.1em, column sep=0.8em]
\mathbf{1}^{\oplus 3} \rar[symbol={\rightsquigarrow}]
& \mathbf{N}_2 \oplus \mathbf{1} \rar[symbol={\rightsquigarrow}]
& \mathbf{N}_3 \rar[symbol={\rightsquigarrow}]
& \mathbf{0} \oplus \mathbf{1}^{\oplus 2} \rar[symbol={\rightsquigarrow}]
& \mathbf{0} \oplus \mathbf{N}_2 \rar[symbol={\rightsquigarrow}]
& \mathbf{0}^{\oplus 2} \oplus \mathbf{1} \rar[symbol={\rightsquigarrow}]
& \mathbf{0}^{\oplus 3}
\end{tikzcd}
\]
and the strata dimension are given by
\[
\begin{array}{c|cccccccc}
\lambda
& \mathbf{1}^{\oplus 3}
& \mathbf{N}_2 \oplus \mathbf{1}
& \mathbf{N}_3
& \mathbf{0} \oplus \mathbf{1}^{\oplus 2}
& \mathbf{0} \oplus \mathbf{N}_2
& \mathbf{0}^{\oplus 2} \oplus \mathbf{1}
& \mathbf{0}^{\oplus 3} \\
\hline
\dim\qbics(V)_\lambda
& 9
& 8
& 7
& 6
& 5
& 3
& 0
\end{array}
\]
\end{Proposition}

\begin{proof}
The relations
\(\mathbf{1}^{\oplus 3} \rightsquigarrow
  \mathbf{N}_2 \oplus \mathbf{1} \rightsquigarrow
  \mathbf{N}_3\)
follow because \(\mathbf{N}_2 \oplus \mathbf{1}\) is the general corank \(1\)
form, see \parref{forms-aut-general-corank-b}. The relation
\(\mathbf{N}_3 \rightsquigarrow \mathbf{0} \oplus \mathbf{1}^{\oplus 2}\)
can be seen geometrically via plane sections of a smooth \(q\)-bic surface
\(X\): Fix one of the finitely many lines \(\ell \subset X\), see
\parref{hypersurfaces-smooth-fano}. Then a general plane section of \(X\)
containing \(\ell\) must be a \(q\)-bic curve of type \(\mathbf{N}_3\).
But by \parref{hypersurfaces-cone-points-smooth}, the tangent plane section of
\(X\) at a Hermitian point contained in \(\ell\) is a \(q\)-bic curve of type
\(\mathbf{0} \oplus \mathbf{1}^{\oplus 2}\), as required.
The remaining specialization relations are now deduced from those of \(q\)-bic
points, see \parref{qbic-points-moduli}.

The strata dimension are computed via \parref{forms-aut-strata-dimension};
see \parref{curves-1+N2.auts} and \parref{curves-N3.auts} for the automorphisms
in the second two types, and apply \parref{forms-automorphisms-cones} and
\parref{qbic-points-moduli} for the cone types.
\end{proof}

\section{Type \texorpdfstring{\(\mathbf{1}^{\oplus 3}\)}{1+1+1}}
\label{curve-1+1+1}
Smooth \(q\)-bic curves \(X\) are well-studied for many reasons: see the
bibliographic comments in the Introduction. Some familiar equations for such a
curve include:
\[
X \cong
\begin{dcases*}
\mathrm{V}(x_0^{q+1} + x_1^{q+1} + x_2^{q+1}) & the Fermat curve, \\
\mathrm{V}(x_0^q x_1 + x_0 x_1^q - x_2^{q+1}) & the Hermitian curve, and \\
\mathrm{V}(x_0^q x_1 - x_0 x_1^q - x_2^{q+1}) & the Drinfeld--Deligne--Lusztig curve.
\end{dcases*}
\]
Potential source of linguistic confusion: \emph{Hermitian curve} refers to the
specific equation given here, whereas \emph{Hermitian \(q\)-bic curve} refers
to any \(q\)-bic curve given with a choice of Hermitian coordinates as in
\parref{forms-hermitian-nondegenerate-section} and
\parref{hypersurfaces-filtration-hermitian}. Thankfully, the Hermitian curve is
a Hermitian \(q\)-bic curve; so is the Fermat curve, but \emph{not}
the Drinfeld--Deligne--Lusztig curve.

\subsection{Tangents}\label{curve-1+1+1.tangents}
Classical algebraic geometers noticed smooth \(q\)-bic curves because the
Gauss map of such a curve is purely inseparable of degree \(q\). This means
that tangent lines to \(X\) are exactly \(q\)-fold tangent at the general point
of \(X\). From the perspective here, this is a feature arising from the
shapes of \(q\)-bic points, see \parref{qbic-points-classification} and the
comments that follow.

One question that remains, however, is where the order of tangency jumps to
\(q+1\). The following shows that these are precisely the Hermitian points of
\(X\), as defined in \parref{hypersurfaces-terminology}. Moreover, it gives a
geometric description of the canonical endomorphism \(\phi_X \colon X \to X\)
constructed in \parref{hypersurfaces-endomorphism}.

\begin{Lemma}\label{curve-residual-intersection}
Let \(x \in X\) be a closed point. Then
\[
\mathrm{mult}_x(X \cap \mathbf{T}_{X,x}) =
\begin{dcases*}
q & if \(x\) is not a Hermitian point, and \\
q+1 & if \(x\) is a Hermitian point.
\end{dcases*}
\]
The endomorphism \(\phi_X \colon X \to X\) sends a point \(x\) to its residual
intersection point \(X \cap \mathbf{T}_{X,x} - qx\) of \(X\) and the tangent
line at \(x\).
\end{Lemma}

\begin{proof}
By \parref{hypersurfaces-tangent-kernel}, \(X \cap \mathbf{T}_{X,x}\) is the
scheme of \(q\)-bic points defined by the fibre at \(x\) of the tangent form
\(
\beta_{\mathrm{tan}} \colon
\Fr^*(\mathcal{T}_X^{\mathrm{e}}) \otimes \mathcal{T}_X^{\mathrm{e}} \to
\sO_X
\).
This form is everywhere of corank \(1\) by
\parref{hypersurfaces-tangent-form-properties}, so the classification of
\(q\)-bic points \parref{qbic-points-classification} implies
\(\operatorname{mult}_x(X \cap \mathbf{T}_{X,x})\) is either \(q\) or \(q+1\),
and the latter occurs
when \(\beta_{\mathrm{tan},x}\) is of type \(\mathbf{0} \oplus \mathbf{1}\).
By the exact sequence
\[
0 \to
\Fr^*(\mathcal{N}_{X/\PP V}(-1))^\vee \xrightarrow{\phi_X}
\mathcal{T}_X^{\mathrm{e}} \xrightarrow{\beta_{\mathrm{tan}}}
\Fr^*(\mathcal{T}_X^{\mathrm{e}})^\vee \xrightarrow{\mathrm{eu}^{(q),\vee}}
\Fr^*(\sO_X(-1))^\vee \to
0
\]
from \parref{hypersurfaces-tangent-form-properties}, \(\beta_{\mathrm{tan}}\)
has type \(\mathbf{0} \oplus \mathbf{1}\) along the locus where the image of
the sheaf map \(\phi_X\) coincides with
\(\mathrm{eu} \colon \sO_X(-1) \to \mathcal{T}^e_X\). Exactness means this is
the vanishing locus of
\[
\beta \circ \mathrm{eu} \colon
\sO_X(-1) \to
\Fr^*(\mathcal{T}_X(-1))^\vee =
\ker\big(
\mathrm{eu}^{(q),\vee} \colon
\Fr^*(\mathcal{T}_X^{\mathrm{e}})^\vee \to
\Fr^*(\sO_X(-1))^\vee
\big).
\]
This is dual to the section defining the scheme of Hermitian points, see
\parref{hypersurfaces-cone-points-equations} and
\parref{hypersurfaces-cone-points-smooth}, thereby yielding the first
statement. The second statement follows from the discussion of
\parref{hypersurfaces-tangent-form-properties-residual-tangent} and
\parref{hypersurfaces-filtration-embedded-tangent}.
\end{proof}

\subsection{Hermitian points}\label{curves-1+1+1-hermitian-points}
By \parref{hypersurfaces-smooth-cone-points-count}, \(X\) has \(q^3 + 1\)
Hermitian points, the divisor of which is defined by the vanishing of the
morphism
\[
\mathrm{eu}^\vee \circ \beta^\vee \colon
\Fr^*(\mathcal{T}_X(-1)) \to
\sO_X(1)
\]
see \parref{hypersurfaces-cone-points-equations}.
Since \(\mathcal{T}_X(-1) \cong \sO_X(-q+1)\), the divisor of Hermitian points
lies in the linear system \(\abs{\sO_X(q^2-q+1)}\); in particular, the scheme
of Hermitian points of \(X\) is a complete intersection in \(\PP V\). Compare
this to the fact that the set of \(\mathbf{F}_q\)-rational points of \(\PP V\)
is a complete intersection. The following gives an explicit lift
\[ \tilde\phi_X \in \mathrm{H}^0(\PP V,\sO_{\PP V}(q^2-q+1)) \]
of \(\phi_X\), providing explicit equations for the scheme of Hermitian points
of \(X\) in \(\PP V\).

\begin{Lemma}\label{curves-1+1+1-lift-of-phi}
Let \((x_0:x_1:x_2)\) be coordinates so that
\(X = \mathrm{V}(x_0^q x_1 + x_0 x_1^q - x_2^{q+1})\).
Then a lift of \(\phi_X\) is given by
\[
\tilde\phi_X \coloneqq
\frac{x_0^{q^2} x_1 - x_0 x_1^{q^2}}{x_0^q x_1 + x_0 x_1^q}x_2 =
\big(x_0^{q(q-1)} - x_0^{(q-1)(q-1)} x_1^{q-1} + \cdots - x_1^{q(q-1)}\big) x_2.
\]
\end{Lemma}

\begin{proof}
As explained in \parref{hypersurfaces-filtration-hermitian}, in these coordinates,
the Hermitian points of \(X\) are its \(\mathbf{F}_{q^2}\)-rational points.
Thus it suffices to show that each point of intersection between
\(X\) with the vanishing locus \(D \coloneqq \mathrm{V}(\tilde\phi_X)\) of the
putative lift is a \(\mathbf{F}_{q^2}\)-point of \(X\). The irreducible
components of \(D\) are lines of the form \(\mathrm{V}(x_2)\) or else
\(\mathrm{V}(x_0 - \alpha x_1)\) where
\(\alpha \in \mathbf{F}_{q^2}\) is such that \(\alpha^q + \alpha \neq 0\). The
first line intersects \(X\) at the \(q+1\) points given by \((x_0:x_1:0)\)
where \(x_0,x_1 \in \mathbf{F}_q\). A line in the second set intersect \(X\) at
\[
X \cap \mathrm{V}(x_0 - \alpha x_1) =
\set{(\alpha x_1 : x_1 : x_2) | (x_2/x_1)^{q+1} = \alpha^q + \alpha}.
\]
Setting \(z \coloneqq x_2/x_1\) for such an intersection point, it follows that
\[
z^{q^2}
= z \cdot z^{(q+1)(q-1)}
= z \cdot \frac{(\alpha^q + \alpha)^q}{\alpha^q + \alpha}
= z \cdot \frac{\alpha + \alpha^q}{\alpha^q + \alpha}
= z
\]
so that \(z \in \mathbf{F}_{q^2}\).
\end{proof}

\section{Type \texorpdfstring{\(\mathbf{N}_2 \oplus \mathbf{1}\)}{N2+1}}
\label{curve-1+N2}
A general singular \(q\)-bic curve \(X\) is one of type
\(\mathbf{N}_2 \oplus \mathbf{1}\). It is irreducible, has a unique unibranch
singularity, and, as shown in \parref{qbic-curves-components-are-rational}, has
normalization given by the projective line. Its group of linear automorphisms
is computed as a special case of \parref{forms-aut-1^a+N2^b.computation}, and
can be made explicit as follows:

\begin{Proposition}\label{curves-1+N2.auts}
Let \((V,\beta)\) be a \(q\)-bic form of type \(\mathbf{N}_2 \oplus \mathbf{1}\).
Then \(\AutSch(V,\beta)\) is isomorphic to the \(1\)-dimensional closed subgroup
scheme of \(\GL_3\) given by
\[
\begin{pmatrix}
\lambda & \epsilon_1 & \epsilon_2 \\
0 & \lambda^{-q} & 0 \\
0 & -\zeta \epsilon_2^q/\lambda^q & \zeta
\end{pmatrix}
\]
where \(\lambda \in \mathbf{G}_m\), \(\epsilon_1 \in \boldsymbol{\alpha}_q\),
\(\epsilon_2 \in \boldsymbol{\alpha}_{q^2}\), and \(\zeta \in \boldsymbol{\mu}_{q+1}\).
\end{Proposition}

\begin{proof}
A direct computation here is short. Choose a basis
\(V = \langle e_0,e_1,e_2 \rangle\) so that
\(\Gram(\beta;e_0,e_1,e_2) = \mathbf{N}_2 \oplus \mathbf{1}\). The
\(\perp\)-filtration of \parref{forms-canonical-filtration}
is \(\langle e_0 \rangle \subset \langle e_0,e_2 \rangle\), and the first
piece of the \(\Fr^*(\perp)\)-filtration of \parref{forms-canonical-filtration-second}
\(\langle e_1^{(q)} \rangle\). Since the automorphism group scheme
preserves these filtrations, see \parref{forms-aut-canonical-filtration}, it is
isomorphic to the closed subgroup scheme of \(\GL_3\) consisting of matrices
which satisfy
\[
\begin{pmatrix}
a_{00}^q & 0 & 0 \\
0 & a_{11}^q & 0 \\
a_{02}^q & 0 & a_{22}^q
\end{pmatrix}
\begin{pmatrix}
0 & 1 & 0 \\
0 & 0 & 0 \\
0 & 0 & 1
\end{pmatrix}
\begin{pmatrix}
a_{00} & a_{01} & a_{02} \\
0      & a_{11} & 0      \\
0      & a_{21} & a_{22}
\end{pmatrix}
=
\begin{pmatrix}
0 & 1 & 0 \\
0 & 0 & 0 \\
0 & 0 & 1
\end{pmatrix}
\]
and \(a_{01}^q = a_{21}^q = 0\). Expanding gives the three equations
\[
a_{00}^q a_{11} = a_{22}^{q+1} = 1
\quad\text{and}\quad
a_{02}^q a_{11} + a_{22}^q a_{21} = 0
\]
The first two equations determine the diagonal entries.
Setting \(\lambda \coloneqq a_{00}\) and \(\zeta \coloneqq a_{22}\),
the third equation implies \(a_{21} = -\zeta a_{02}^q/\lambda^q\).
Combined with the fact that \(a_{21}^q = 0\), this implies
\(a_{02}^{q^2} = 0\). Setting \(\epsilon_1 \coloneqq a_{01}\) and
\(\epsilon_2 \coloneqq a_{02}\) finishes the computation.
\end{proof}

\begin{Remark}\label{curves-1+N2.nonlinear-auts}
The morphism \(\AutSch(V,\beta) \to \AutSch(X)\) from
\parref{hypersurfaces-automorphisms-linear} giving the linear automorphisms of
\(X\) is not surjective here: since \(X\) is normalized by \(\PP^1\), any
automorphism of \(\PP^1\) infinitesimally preserving the preimage \(\infty\) of
the cusp will descend to an automorphism of \(X\). In particular,
\(\AutSch(X)\) contains the \(2\)-dimensional group of automorphisms of
\(\mathbf{A}^1 = \PP^1 \setminus \infty\), whereas \(\AutSch(V,\beta)\) is only
\(1\)-dimensional by \parref{curves-1+N2.auts}. An explicit computation of
\(\AutSch(X)\) in the case \(q = 2\) can be found \cite[Proposition 6]{BM:III};
the method may be adapted to the case of general \(q\).
\end{Remark}

\subsection{Cone points}\label{curves-1+N2.cone-points}
Consider the cone points of \(X\) as defined in \parref{hypersurfaces-cone-points-definition}.
By \parref{forms-hermitian-basics-orthogonals} and
\parref{forms-hermitian-examples}\ref{forms-hermitian-examples.nilpotent}, the
underlying \(q\)-bic form does not have any nontrivial Hermitian points.
Therefore, by \parref{hypersurfaces-cone-points-classify}, the cone points
of \(X\) are precisely its singular point
\(x_+ \coloneqq \PP\Fr^{-1}(V^\perp)\) and the special point
\(x_- \coloneqq \PP\Fr^*(V)^\perp\).

The scheme \(X_{\mathrm{cone}}\) of cone points from
\parref{hypersurfaces-cone-points-equations-general} is more interesting.
Either by direct computation or else by analyzing the argument
of \parref{hypersurfaces-smooth-cone-points-etale}, it follows that \(x_-\) is
a reduced point of \(X_{\mathrm{cone}}\). Then by the degree computation of
\parref{hypersurfaces-cone-points-scheme-degree}, \(x_+\) appears with
multiplicity \(q^3\). This
gives some indication as to how a smooth \(q\)-bic curve degenerates to one of
type \(\mathbf{N}_2 \oplus \mathbf{1}\): fix one of the \(q^3+1\) Hermitian
points, and let the remaining \(q^3\) come together. Compare with the analysis
of \parref{qbic-points-families} for \(q\)-bic points.

\section{Type \texorpdfstring{\(\mathbf{N}_3\)}{N3}}\label{curve-N3}
A \(q\)-bic curve of type \(\mathbf{N}_3\), though not yet a
cone, is reducible with a linear component and a component of degree \(q\).
The latter is a rational curve with a unibranch singularity which is
analogous to the standard cusp in characteristic \(3\).

The two components intersect at the unique singular point \(x_+ =
\PP\Fr^{-1}(V^\perp)\). The special point \(x_- \coloneqq \PP\Fr^*(V)^\perp\)
lies on the linear component, and together with \(x_+\), spans the line.
The linear subspace underlying this component is therefore the second step of
the \(\perp\)-filtration of \((V,\beta)\), see \parref{forms-canonical-filtration}. This linear
component also comprise of the Hermitian points and cone points of
\(X\): see
\parref{forms-hermitian-examples}\ref{forms-hermitian-examples.nilpotent}
and \parref{hypersurfaces-cone-points-classify}.

Finally, the scheme of linear automorphisms of \(X\) is determined as follows:

\begin{Proposition}\label{curves-N3.auts}
Let \((V,\beta)\) be a \(q\)-bic form of type \(\mathbf{N}_3\). Then
\(\AutSch(V,\beta)\) is isomorphic to the \(2\)-dimensional closed subgroup
scheme of \(\GL_3\) consisting of
\[
\begin{pmatrix}
\lambda & t                     & \epsilon \\
0       & \lambda^{-q}          & 0 \\
0       & -\lambda^{q(q-1)} t^q & \lambda^{q^2}
\end{pmatrix}
\]
where \(\lambda \in \mathbf{G}_m\), \(t \in \mathbf{G}_a\), and
\(\epsilon \in \boldsymbol{\alpha}_q\).
\end{Proposition}

\begin{proof}
Choose a basis \(V = \langle e_0,e_1,e_2 \rangle\) so that
\(\Gram(\beta;e_0,e_1,e_2) = \mathbf{N}_3\). Then the \(\perp\)-filtration is
\(\langle e_0 \rangle \subset \langle e_0,e_2\rangle\), and the first piece
of the \(\Fr^*(\perp)\)-filtration is \(\langle e_2^{(q)} \rangle\),
see \parref{forms-canonical-filtration} and
\parref{forms-canonical-filtration-second}. Since automorphisms preserve these
filtrations by \parref{forms-aut-canonical-filtration}, \(\AutSch(V,\beta)\) is
isomorphic to the closed subscheme of \(\GL_3\) consisting of matrices which
satisfy
\[
\begin{pmatrix}
a_{00}^q & 0        & 0 \\
a_{01}^q & a_{11}^q & a_{21}^q \\
0        & 0        & a_{22}^q
\end{pmatrix}
\begin{pmatrix}
0 & 1 & 0 \\
0 & 0 & 1 \\
0 & 0 & 0
\end{pmatrix}
\begin{pmatrix}
a_{00} & a_{01} & a_{02} \\
0      & a_{11} & 0 \\
0      & a_{21} & a_{22}
\end{pmatrix}
=
\begin{pmatrix}
0 & 1 & 0 \\
0 & 0 & 1 \\
0 & 0 & 0
\end{pmatrix}
\]
and \(a_{02}^q = 0\). Expanding gives
\(a_{00}^q a_{11} = a_{11}^q a_{22} = 1\) and
\(a_{01}^q a_{11} + a_{11}^q a_{21} = 0\).
The first two equations give the diagonal entries; set
\(\lambda \coloneqq a_{00}\). The third equation implies
\(a_{21} = -a_{01}^q\lambda^{q(q-1)}\). Setting \(t \coloneqq a_{01}\)
and \(\epsilon \coloneqq a_{02}\) completes the computation.
\end{proof}

\section{\texorpdfstring{\(q\)}{q}-bic surfaces}\label{section-lowdim-surfaces}
The remainder of this Chapter is devoted to the study of \(q\)-bic surfaces:
the \(q\)-bic hypersurfaces in projective \(3\)-space. The classification of
\(q\)-bic forms from \parref{forms-classification-theorem} gives:

\begin{Proposition}\label{surfaces-classification}
Let \(X\) be the \(q\)-bic surface associated with a nonzero \(q\)-bic form
\((V,\beta)\) of dimension \(4\). Then either \(X\) is a cone over \(q\)-bic
curve or
\[
\mathrm{type}(\beta) =
\begin{dcases*}
\mathbf{1}^{\oplus 4} \\
\mathbf{N}_2 \oplus \mathbf{1}^{\oplus 2} \\
\mathbf{N}_4 \\
\mathbf{N}_3 \oplus \mathbf{1} \\
\mathbf{N}_2^{\oplus 2}
\end{dcases*}
\quad\text{and}\quad
X \cong
\begin{dcases*}
\mathrm{V}(x_0^{q+1} + x_1^{q+1} + x_2^{q+1} + x_3^{q+1}), \;\text{or} \\
\mathrm{V}(x_0^q x_1 + x_2^{q+1} + x_3^{q+1}), \;\text{or} \\
\mathrm{V}(x_0^q x_1 + x_1^q x_2 + x_2^q x_3), \;\text{or} \\
\mathrm{V}(x_0^q x_1 + x_1^q x_2 + x_3^{q+1}), \;\text{or} \\
\mathrm{V}(x_0^q x_1 + x_2^q x_3),
\end{dcases*}
\]
for some choice of coordinates \((x_0:x_1:x_2:x_3)\) on \(\PP V \cong \PP^3\). \qed
\end{Proposition}

The type stratification of the affine space \(\qbics(V)\), see
\parref{forms-classification-moduli} and
\parref{forms-classification-type-stratification}, is much more interesting
than in the case of points and curves. In the following,
temporarily drop the ``\(\oplus\)'' symbol for notation on types.

\begin{Proposition}\label{surfaces-specializations}
The Hasse diagram for the type stratification of \(\qbics(V)\) is
\[
\begin{tikzcd}[row sep=0.1em, column sep=0.8em]
\mathbf{1}^{4} \rar[symbol={\rightsquigarrow}]
& \mathbf{N}_2 \mathbf{1}^{2} \rar[symbol={\rightsquigarrow}]
& \mathbf{N}_4 \ar[dr,symbol={\rightsquigarrow}]
  \ar[rrr,phantom,"\mathbf{N}_2^2","\rightsquigarrow"{xshift=-2.25em},"\rightsquigarrow"{xshift=2.25em}]
&
&
& \mathbf{0}\mathbf{N}_2\mathbf{1} \rar[symbol={\rightsquigarrow}]
& \mathbf{0}\mathbf{N}_3 \rar[symbol={\rightsquigarrow}]
& \mathbf{0}^2\mathbf{1}^2 \rar[symbol={\rightsquigarrow}]
& \mathbf{0}^2\mathbf{N}_2   \rar[symbol={\rightsquigarrow}]
& \mathbf{0}^3\mathbf{1} \\
&
&
& \mathbf{N}_3\mathbf{1} \rar[symbol={\rightsquigarrow}]
& \mathbf{0}\mathbf{1}^3 \ar[ur,symbol={\rightsquigarrow}] \\
\end{tikzcd}
\]
and the strata dimension are given by
\[
\begin{array}{c|*{11}c}
\lambda
& \mathbf{1}^{4}
& \mathbf{N}_2\mathbf{1}^2
& \mathbf{N}_4
& \mathbf{N}_3\mathbf{1}
& \mathbf{N}_2^{2}
& \mathbf{0}\mathbf{1}^{3}
& \mathbf{0}\mathbf{N}_2\mathbf{1}
& \mathbf{0}\mathbf{N}_3
& \mathbf{0}^{2}\mathbf{1}^2
& \mathbf{0}^{2}\mathbf{N}_2
& \mathbf{0}^3\mathbf{1} \\
\hline
\dim\qbics(V)_\lambda
& 16
& 15
& 14
& 13
& 12
& 12
& 11
& 10
& 8
& 7
& 4
\end{array}
\]
\end{Proposition}

\begin{proof}
The relations
\(\mathbf{1}^{\oplus 4} \rightsquigarrow
  \mathbf{N}_2 \oplus \mathbf{1}^{\oplus 2} \rightsquigarrow
  \mathbf{N}_4\)
are because the general corank \(1\) form is of type
\(\mathbf{N}_2 \oplus \mathbf{1}^{\oplus 2}\), see
\parref{forms-aut-general-corank-b}. Analogously,
\(\mathbf{N}_2^{\oplus 2} \rightsquigarrow \mathbf{0} \oplus \mathbf{N}_2 \oplus \mathbf{1}\)
because the former is the general corank \(2\) form. That
\(\mathbf{N}_3 \oplus \mathbf{1} \rightsquigarrow \mathbf{0} \oplus \mathbf{1}^{\oplus 3}\)
can be seen by specializing the subform \(\mathbf{N}_3 \rightsquigarrow \mathbf{0} \oplus \mathbf{1}^{\oplus 2}\),
possible by \parref{qbic-curves-moduli}. The specialization relations to the
right of \(\mathbf{N}_2^{\oplus 2}\) and \(\mathbf{0} \oplus \mathbf{1}^{\oplus 3}\)
now follow from specializing \(2\)- or \(3\)-dimensional subforms and using
the relations from \parref{qbic-points-moduli} and \parref{qbic-curves-moduli}.

To see that \(\mathbf{N}_4 \rightsquigarrow \mathbf{N}_3 \oplus \mathbf{1}\),
consider the family \(\beta_t\) of \(q\)-bic forms on
\(V = \langle e_0,e_1,e_2,e_3 \rangle\) over \(\mathbf{A}^1 = \Spec(\kk[t])\)
given by
\[
\Gram(\beta_t; e_0,e_1,e_2,e_3) =
\left(
\begin{smallmatrix}
0 & 1 & t & 0 \\
0 & 0 & 1 & t \\
0 & 0 & 0 & 1 \\
0 & 0 & 0 & 0
\end{smallmatrix}
\right).
\]
Over a \(\kk\)-point of \(\mathbf{A}^1\), the
\(\perp\)-filtration and \(\Fr^*(\perp)\)-filtration of this form are
\begin{align*}
\langle e_0 \rangle & \subset
\langle e_0, t^q e_3 + (1-t^{q+1})( te_1 + e_2) \rangle, \;\;\text{and} \\
\langle e_3 \rangle & \subset
\langle e_3, t^{1/q} e_1 + (1-t^{(q+1)/q})(e_1 - t e_2) \rangle.
\end{align*}
For general \(t\), the second steps differ, whence
\(\mathrm{type}(\beta_t) = \mathbf{N}_4\); for \(t = 1\), the second step
coincides, so \(\mathrm{type}(\beta_t) = \mathbf{N}_3 \oplus \mathbf{1}\).
See \parref{forms-classification-theorem} and the comments of
\parref{forms-classification-theorem-remarks}.

It remains to show that \(\mathbf{N}_3 \oplus \mathbf{1}\) does \emph{not}
specialize to \(\mathbf{N}_2^{\oplus 2}\). Suppose there were such a
specialization. Then there would exist a discrete valuation ring \(R\) over \(\kk\)
and a flat family \(\mathcal{X} \subset \PP V \otimes_\kk R\) of \(q\)-bic
surfaces with generic fibre of type \(\mathbf{N}_3 \oplus \mathbf{1}\) and
special fibre of type \(\mathbf{N}_2^{\oplus 2}\). Consider the relative
Fano scheme of lines \(\mathbf{F}_1(\mathcal{X}/R)\) of this family. As in
\parref{hypersurfaces-equations-of-fano}, this is defined in
\(\mathbf{G}(2,V) \otimes_\kk R\) by a section of a rank \(4\) bundle.
In particular, every component of \(\mathbf{F}_1(\mathcal{X}/R)\) has codimension
at most \(4\) in \(\mathbf{G}(2,V)\). On the other hand, by \parref{surfaces-1+N3.lines}
the generic fibre of \(\mathbf{F}_1(\mathcal{X}/R)\) is a single point, whereas
by \parref{surfaces-N2+N2.lines}, the special fibre has a one-dimensional
component and two isolated points. In any case, this implies that
\(\mathbf{F}_1(\mathcal{X}/R)\) must have an isolated
\(0\)-dimensional component in its special fibre,
which would be codimension \(5\) in \(\mathbf{G}(2,V) \otimes_\kk R\)---a
contradiction! Thus such a \(\mathcal{X}\) cannot exist.

The dimensions of the strata are now determined by
\parref{forms-aut-strata-dimension}; for explicit computations of the
automorphism group schemes, see \parref{surfaces-1+1+N2.auts},
\parref{surfaces-N4.auts}, \parref{surfaces-1+N3.auts}, and
\parref{surfaces-N2+N2.auts} for the singular surfaces which are not cones, and
for the cones, see \parref{forms-automorphisms-cones} together with the
computations \parref{qbic-points-automorphisms.N2}, \parref{curves-1+N2.auts},
and \parref{curves-N3.auts}.
\end{proof}

\subsection{Lines}\label{qbic-surfaces-lines}
Lines in \(X\) correspond to maximal isotropic subspaces for \((V,\beta)\).
So by \parref{hypersurfaces-cone-points-subspaces-have}, every line in \(X\) contains a cone
point. The general results of \parref{section-hypersurfaces-linear-spaces} show
that the scheme \(\mathbf{F}_1(X)\) of lines in \(X\) has expected dimension
\(0\), in which case it has degree \((q+1)(q^3+1)\): see
\parref{hypersurfaces-equations-of-fano} and \parref{hypersurfaces-fano-lines-degree}.
By \parref{hypersurfaces-fano-corank-1}\ref{hypersurfaces-fano-corank-1.expdim},
\(\mathbf{F}_1(X)\) is \(0\)-dimensional whenever \(\beta\) has
corank at most \(1\) and no radical; in fact, the converse holds:

\begin{Lemma}\label{qbic-surfaces-lines-expected-dimension}
The Fano scheme \(\mathbf{F}_1(X)\) has expected dimension \(0\) if and only
if \((V,\beta)\) has corank at most \(1\) and no radical.
\end{Lemma}

\begin{proof}
If \(\beta\) has a radical, then \(X\) is a cone by
\parref{hypersurfaces-cones}, and \(\mathbf{F}_1(X)\) contains a curve. If
\(\beta\) does not have a radical, then
\(\mathrm{type}(\beta) = \mathbf{N}_2^{\oplus 2}\) by
\parref{hypersurfaces-cone-high-corank}. Then \(X\) contains the
\(1\)-dimensional family of lines given by the \(\PP \Fr^{-1}(L^\perp)\), where
\(\PP L\) is a singular point of \(X\), and again, \(\mathbf{F}_1(X)\) is at
least \(1\)-dimensional.
\end{proof}

The remainder of the Section is devoted to a general analysis of surfaces
of corank \(1\) and \(2\). Write \(L_- \coloneqq \Fr^*(V)^\perp\) and
\(L_+ \coloneqq \Fr^{-1}(V^\perp)\) for the two kernels of \(\beta\), so that
by \parref{forms-orthogonal-sequence}, they fit into an exact sequence
\[ 0 \to L_- \to V \xrightarrow{\beta} \Fr^*(V) \to \Fr^*(L_+) \to 0. \]
By \parref{hypersurfaces-nonsmooth-locus}, the singular locus of \(X\) is
supported on the linear space \(\PP L_+\).

\subsection{Corank \(1\)}\label{surfaces-corank-1}
Assume that \(X\) has corank \(1\), so that its type is amongst
\[
\mathbf{N}_2 \oplus \mathbf{1}^{\oplus 2},\;\;
\mathbf{N}_4,\;\;
\mathbf{N}_3 \oplus \mathbf{1},\;\;
\mathbf{0} \oplus \mathbf{1}^{\oplus 3}.
\]
Write \(x_\pm \coloneqq \PP L_\pm\) for the points corresponding to the
kernels of \(\beta\), so \(x_+\) is the unique singular point of \(X\).
By \parref{qbic-surfaces-lines-expected-dimension}, if \(X\) is not the cone,
then \(\mathbf{F}_1(X)\) is \(0\)-dimensional and by
\parref{hypersurfaces-smooth-point-fano}, points corresponding to lines
not passing through \(x_+\) are reduced.

\subsection{Projection from \(x_+\)}\label{surfaces-corank-1.projection}
Let \(\tilde{X} \to X\) be the blowup at the singular point \(x_+\), let
\(W \coloneqq V/L_+\), and let \(\pi \colon \tilde{X} \to \PP W\) be the
morphism resolving linear projection centred at \(x_+\). The structure of
\(\pi\) splits into two cases:
\begin{itemize}
\item if \(X\) is not a cone, then \(\pi\) is an isomorphism away from the
finitely many points of \(\PP W\) corresponding to the lines in \(X\) through
\(x_+\); and
\item if \(X\) is a cone, then \(\pi\) factors through the smooth
\(q\)-bic curve \(C \subset \PP W\) induced by \(\beta\) on the quotient, and
there is an isomorphism \(\tilde{X} \cong \PP(\sO_C(-1) \oplus L_{+,C})\).
\end{itemize}
The first case follows from the structure of \(q\)-bic points from
\parref{qbic-points-classification}: any line in \(\PP V\) through the singular
point \(x_+\) must contain it to multiplicity at least \(q\), and when \(x_+\)
is not a vertex point, the general line contains it to multiplicity at most \(q\).
The second case follows by general projective geometry; see also
\parref{linear-projection-resolve}.

The blowup \(\tilde{X}\) is not always smooth. A direct computation shows that
it is smooth for \(X\) of type \(\mathbf{N}_2 \oplus \mathbf{1}^{\oplus 2}\)
and \(\mathbf{0} \oplus \mathbf{1}^{\oplus 3}\), but not in the remaining
cases \(\mathbf{N}_4\) and \(\mathbf{N}_3 \oplus \mathbf{1}\).

\subsection{Corank \(2\)}\label{surfaces-corank-2}
Assume that \(X\) has corank \(2\), so that its type is amongst
\[
\mathbf{N}_2^{\oplus 2},\;\;
\mathbf{0} \oplus \mathbf{N}_3,\;\;
\mathbf{0} \oplus \mathbf{N}_2 \oplus \mathbf{1},\;\;
\mathbf{0}^{\oplus 2} \oplus \mathbf{1}^{\oplus 2}.
\]
Write \(\ell_\pm \coloneqq \PP L_\pm\) for the lines corresponding to the
kernels of \(\beta\), so that \(\ell_+\) is the singular locus of \(X\).
The analysis of corank \(2\) surfaces relies on the following simple consequence
of the scheme structure of the nonsmooth locus:

\begin{Lemma}\label{surfaces-corank-2-singular}
Any \(\PP^2\)-section of \(X\) containing \(\ell_+\) is a \(q\)-bic curve
containing \(\ell_+\) to multiplicity at least \(q\).
\end{Lemma}

\begin{proof}
Let \(\PP U \cong \PP^2\) be a plane containing \(\ell_+\). Then
\(L_+ \subseteq U\) so the restriction of \(\beta\) to \(U\) satisfies
\(\Fr^*(L_+) \subseteq \ker(\beta_U \colon \Fr^*(U) \to U^\vee)\). Since the
\(q\)-bic curve \(X \cap \PP U\) is that given by \((U,\beta_U)\) by
\parref{hypersurface-hyperplane-section}, the result follows from
\parref{hypersurfaces-nonsmooth-locus}.
\end{proof}

\subsection{Projection from \(\ell_+\)}\label{surfaces-corank-2-projection}
Let \(W \coloneqq V/L_+\) and let \(\psi \colon \PP V \dashrightarrow \PP W\)
be the linear projection from \(\ell_+ = \PP L_+\). As in \parref{linear-projection},
let
\[
\PP\psi \coloneqq
\Set{([V'],[W']) \in \PP V \times \PP W | \psi(V') \subseteq W'}
\]
be the incidence correspondence for \(1\)-dimensional subspaces under the
quotient map \(V \to W\). Then there is a commutative diagram
\[
\begin{tikzcd}
\tilde{X} \rar[hook] \dar & \PP\psi \dar \rar & \PP W \\
X \rar[hook] & \PP V \ar[ur,dashed,"\psi"']
\end{tikzcd}
\]
where, by \parref{linear-projection-resolve}, the vertical maps are blowups
along \(\ell_+\), and \(\PP\psi \to \PP W\) is the projective bundle associated
with the bundle \(\mathcal{V}\) defined by the pullback diagram
\[
\begin{tikzcd}
0 \rar
& L_{+,\PP W} \rar \dar[equal]
& \mathcal{V} \rar \dar[hook]
& \sO_{\PP W}(-1) \rar \dar[hook,"\mathrm{eu}_{\PP W}"]
& 0 \\
0 \rar
& L_{+,\PP W} \rar
& V_{\PP W} \rar
& W_{\PP W} \rar
& 0\punct{.}
\end{tikzcd}
\]
The inverse image \(X_{\PP\psi} \coloneqq X \times_{\PP V} \PP\psi\) of \(X\)
along the blowup is the bundle of \(q\)-bic curves over \(\PP W\) defined by
the \(q\)-bic form
\[
\beta_{\mathcal{V}} \colon
\Fr^*(\mathcal{V}) \otimes \mathcal{V} \subset
\Fr^*(V)_{\PP W} \otimes V_{\PP W} \xrightarrow{\beta}
\sO_{\PP W}.
\]
The general \(\PP^2\)-section of \(X\) containing \(\ell_+\) is not completely
contained in \(X\). Thus \(\mathcal{V}^\perp = \Fr^*(L_{+,\PP W})\) and so
the equation of \(\tilde{X}\) is obtained by pulling the morphism
\[
\beta_{\mathcal{V}} \colon
\mathcal{V} \to
\Fr^*(\mathcal{V}/L_{+,\PP W})^\vee \cong
\sO_{\PP W}(q)
\]
up to \(\PP\psi\), pre-composing and post-composing by the
\(\mathrm{eu}_{\PP\psi/\PP W}\) and \(\mathrm{eu}_{\PP\psi/\PP W}^{(q),\vee}\),
respectively. From this, it follows that \(X_{\PP\psi}\) contains the
exceptional subbundle \(\PP L_{+,\PP W}\) to multiplicity at least \(q\),
globalizing \parref{surfaces-corank-2-singular}.

\subsection{Strict transform}\label{surfaces-corank-2-strict-transform}
The blowup \(\tilde{X} \to X\) along \(\ell_+\), being the strict transform
of \(X\) along \(\PP\psi \to \PP V\), is obtained by factoring out the equation
the exceptional subbundle \(\PP L_{+,\PP W}\). This can be expressed in terms
of the morphism \(\beta_{\mathcal{V}} \colon \mathcal{V} \to \sO_{\PP W}(q)\)
constructed at the end of \parref{surfaces-corank-2-projection}; the construction
breaks up into two cases:

If \(X\) is a cone over \(\ell_+\), then \(X_{\PP\psi}\) contains
\(\PP L_{+,\PP W}\) to multiplicity \(q+1\), and the equation of \(\tilde{X}\)
is the pullback of \(\beta_{\mathcal{V}}\) to \(\PP\psi\). Thus
\(\tilde{X}\) consists of the fibres of \(\PP\psi \to \PP W\) over the smooth
\(q\)-bic points determined by the \(q\)-bic form induced by \(\beta\) on \(W\).

If \(X\) is not a cone over \(\ell_+\), then \(X_{\PP\psi}\) contains the
\(\PP L_{+,\PP W}\) to multiplicity \(q\), and the equation of \(\tilde{X}\)
is given by \(\mathrm{eu}_{\PP\psi/\PP W} \circ \beta_{\mathcal{V}}\).
Thus, if \(\beta_{\mathcal{V}} \colon \mathcal{V} \to \sO_{\PP W}(q)\) is
furthermore surjective or, equivalently, if \(X\) does not contain a \(\PP^2\)
as an irreducible component, then \(\tilde{X}\) is the projective subbundle of
\(\PP\psi\) associated with
\[
\Fr^*(\mathcal{V})^\perp
= \ker(\beta_{\mathcal{V}} \colon \mathcal{V} \to \sO_{\PP W}(q)).
\]
Even if \(\beta_{\mathcal{V}}\) is not surjective, it is a nonzero map; since
any subsheaf of \(\sO_{\PP W}(q) = \sO_{\PP^1}(q)\) is a line bundle,
\(\Fr^*(\mathcal{V})^\perp\) is a rank \(2\) isotropic subbundle of
\(\mathcal{V}\), and hence yields a family over \(\PP W\) of lines in \(X\). In
summary, this shows that:

\begin{Proposition}\label{surfaces-corank-2-no-planes}
A \(q\)-bic surface \(X\) of corank \(2\) is irreducible if and only if it
does not contain a plane. If \(X\) is irreducible, then the blowup
\(\tilde{X} \to X\) along its singular line \(\ell_+\) is the ruled
surface \(\PP\Fr^*(\mathcal{V})^\perp \to \PP W\).
\end{Proposition}

\begin{proof}
If \(X\) contains a plane, then it must be an irreducible component as \(X\)
is also \(2\)-dimensional. If \(X\) does not contain a plane, then
\parref{surfaces-corank-2-strict-transform} shows that the blowup
\(\tilde{X} \to X\) along the singular line \(\ell_+\)
is isomorphic to the projective bundle \(\PP\Fr^*(\mathcal{V})^\perp \to \PP W\).
In particular, \(\tilde{X}\) is irreducible, and thus so is \(X\).
\end{proof}

The \(q\)-bic surfaces of corank that contain a plane are classified as follows:

\begin{Lemma}\label{surfaces-corank-2-planes}
Let \(X\) be a \(q\)-bic surface of corank \(2\).
\begin{itemize}
\item If \(X\) is of type either \(\mathbf{N}_2^{\oplus 2}\) or
\(\mathbf{0} \oplus \mathbf{N}_2 \oplus \mathbf{1}\), then \(X\) contains no planes.
\item If \(X\) is of type \(\mathbf{0} \oplus \mathbf{N}_3\), then \(X\) contains
the unique plane \(\PP(L_- + L_+)\).
\item If \(X\) is of type \(\mathbf{0}^{\oplus 2} \oplus \mathbf{1}^{\oplus 2}\),
then \(X\) consists of \(q+1\) planes intersecting at \(\PP L_+\).
\end{itemize}
\end{Lemma}

\begin{proof}
The cases in which \(X\) is a cone reduce, upon by passing to the quotient by
the radical of \((V,\beta)\), to the classification of \(q\)-bic curves which
contain lines: see \parref{qbic-curves-classification}. It remains to show that
that a \(q\)-bic surface \(X\) of type \(\mathbf{N}_2^{\oplus 2}\)
does not contain a plane. Since \(X\) is connected, any plane in such \(X\)
must contain the singular line \(\ell_+\) in its intersection with other
irreducible components of \(X\). But, as in \parref{forms-aut-1^a+N2^b},
\(V = L_- \oplus L_+\) and \(\beta\) induces an isomorphism
\(L_+ \to \Fr^*(L_-)^\vee\). Thus if \(U \subset V\) is any \(3\)-dimensional
subspace containing \(L_+\), \(U \subset W\) containing \(L_+\), the composite
\(L_+ \subset U \to \Fr^*(U)^\vee\) has rank \(1\) and thus \(U\) cannot be
isotropic. In other words, a \(q\)-bic surface of type
\(\mathbf{N}_2^{\oplus 2}\) cannot contain a plane.
\end{proof}

The family \(\tilde{X} \to \PP W\) of lines in \(X\) at the end of
\parref{surfaces-corank-2-strict-transform} sweep out \(X\):

\begin{Lemma}\label{surfaces-corank-2-lines}
Let \(X\) be a \(q\)-bic surface of corank \(2\), not of type
\(\mathbf{0}^{\oplus 2}\oplus\mathbf{1}^{\oplus 2}\). Then the morphism
\(\PP W \to \mathbf{F}_1(X)\)
classifying the family of lines given by \(\tilde{X} \to \PP W\) is injective.
\end{Lemma}

\begin{proof}
By construction, the fibre of \(\tilde{X} \to \PP W\) over a point \(y \in \PP
W\) is the line \(\ell\) residual to \(q\ell_+\) in the intersection
\(X \cap \PP\mathcal{V}_y\). If \(\mathcal{V}_y \neq \Fr^{-1}(L_+^\perp) \subset V\),
which can only occur when \(X\) has a vertex, then \(\ell = \ell_-\); otherwise,
\(\ell \cap \ell_-\) is a single point \(x = \PP L\) and, by
\parref{threefolds-fano-linear-flag}, \(\mathcal{V}_y = \Fr^{-1}(L^\perp)\).
Thus all but at most one fibre of \(\tilde{X} \to \PP W\) are lines which
intersect \(\ell_-\) at a distinct point, with the remaining fibre being
\(\ell_-\). It follows that the classifying map \(\PP W \to \mathbf{F}_1(X)\)
is injective.
\end{proof}

This implies that, in the cases given by \parref{surfaces-corank-2-lines},
there is a component of \(\mathbf{F}_1(X)\) whose reduction is a geometrically
rational curve. The scheme structure of the component, however, is generally
quite complicated: see \parref{surfaces-N2+N2.lines} and \parref{surfaces-0+1+N2.lines}.

\section{Type \texorpdfstring{\(\mathbf{1}^{\oplus 4}\)}{1+1+1+1}}
\label{surface-type-1^4}
Let \(X\) be the smooth \(q\)-bic surface associated with a
\(q\)-bic form \((V,\beta)\) of type \(\mathbf{1}^{\oplus 4}\). Convenient
equations for \(X\) include;
\[
X = \begin{dcases*}
\mathrm{V}(x_0^{q+1} + x_1^{q+1} + x_2^{q+1} + x_3^{q+1})
& the Fermat surface, and \\
\mathrm{V}(x_0^q x_1 + x_0 x_1^q + x_2^q x_3 + x_2 x_3^q)
& the Hermitian surface.
\end{dcases*}
\]

\subsection{Hermitian points and lines}\label{surfaces-1+1+1+1.lines}
The surface \(X\) contains \((q^2+1)(q^3+1)\) Hermitian points,
and through each such point passes \(q+1\) Hermitian lines: see
\parref{hypersurfaces-smooth-cone-points-count} and
\parref{hypersurfaces-cone-points-smooth}. Since each line in \(X\) is
Hermitian by \parref{hypersurfaces-cones-even-maximal-isotropic}, each line in
\(X\) contains \(q^2+1\) Hermitian points, giving another way to verify that
there are exactly \((q+1)(q^3+1)\) lines in \(X\).
The union of the lines in \(X\) is the complete intersection
\(X \cap X^1\), see \parref{hypersurfaces-filtration} and
\parref{hypersurfaces-filtration-maximal-isotropic}. For example, for
the Fermat equation, \parref{hypersurfaces-filtration-hermitian} gives
\[
X \cap X^1 \cong
\mathrm{V}(x_0^{q+1} + x_1^{q+1} + x_2^{q+1} + x_3^{q+1}) \cap
\mathrm{V}(x_0^{q^3+1} + x_1^{q^3+1} + x_2^{q^3+1} + x_3^{q^3+1}).
\]

\subsection{Unirational parameterization}\label{surfaces-1+1+1+1.unirational}
By \parref{hypersurfaces-unirational-smooth}, the surface \(X\) admits
a purely inseparable unirational parameterization of degree \(q\).
Shioda first constructed an explicit coordinate parameterization in
\cite{Shioda:Unirational}; see also \cite[Corollary 5.3]{Katsura:Lefschetz}
for a coordinate computation akin to the method presented in
\parref{hypersurfaces-unirational-shioda}. The construction of
\parref{hypersurfaces-unirational-tangent} identifies an explicit blowup of a
particular Hirzebruch surface for this parameterization, and will be described
in the following. That the Fermat \(q\)-bic surface is purely inseparably
covered by the specific Hirzebruch surface below was described by Hirokado in
\cite[Proposition 3.6]{Hirokado:Hirzebruch} in the case \(q = p\) via explicit
computations using the theory of \(1\)-foliations.

\subsection{}
Fix a line \(\ell \subset X\) and let
\[
Y \coloneqq
\PP(\mathcal{T}_X(-1)\rvert_\ell) \cong
\PP(\sO_\ell(1) \oplus \sO_\ell(-q))
\]
be projective bundle over \(\ell\) associated with the restricted tangent
bundle of \(X\); see \parref{hypersurfaces-linear-subspace-normal-bundle} for
the second identification. Let \(\tilde{Y} \to Y\) be the blowup along the
points
\[
\Set{(x,[\ell']) \in Y | \ell' \neq \ell \;\text{and}\; \ell' \subset X}
\]
with notation as in \parref{hypersurfaces-unirational-tangent}; this
consists of \(q(q^2+1)\) points with \(q\) points lying over each of the
\(q^2+1\) Hermitian points of \(X\) contained in \(\ell\). Let
\(\tilde{X} \to X\) be the blowup along the Hermitian points contained in
\(\ell\). Then there exists a commutative diagram
\[
\begin{tikzcd}
\tilde{Y} \dar \rar["\varphi"'] & \tilde{X} \dar \\
Y \rar[dashed] & X
\end{tikzcd}
\]
where \(Y \dashrightarrow X\) is the rational map constructed in
\parref{hypersurfaces-unirational-tangent-morphism} and
\(\varphi \colon \tilde{Y} \to \tilde{X}\)
is a finite purely inseparable morphism of degree \(q\).

\subsection{Divisors}
The Picard rank of the surfaces \(\tilde{Y}\) and \(\tilde{X}\) are
\[
\rank_{\mathbf{Z}} \mathrm{NS}(\tilde{Y})
= \rank_{\mathbf{Z}} \mathrm{NS}(\tilde{X})
= q(q^2+1) + 2.
\]
This matches the \'etale Betti number of \(\tilde{X}\), which can be computed
from \parref{hypersurfaces-smooth-etale}\ref{hypersurfaces-smooth-etale.betti}.
The pullback map
\(\varphi^* \colon \mathrm{NS}(\tilde{X}) \to \mathrm{NS}(\tilde{Y})\) acts on
certain special classes as follows:
\[
\varphi^* E_x = q\tilde{Y}_x,
\quad
\varphi^*\ell' = q E_{(x,\ell')} + \tilde{Y}_x,
\quad
\varphi^*\tilde{\ell} = \sigma.
\]
Here, \(x \in \ell\) is a Hermitian point and \(\ell'\) is a line in \(X\)
intersecting \(\ell\) exactly at \(x\). The divisors on \(\tilde{X}\) are:
\(E_x\) is the exceptional divisor of \(\tilde{X} \to X\) over \(x\);
\(\tilde{\ell}\) is the strict transform of \(\ell\); and \(\sigma\) is the
pullback of the class of the negative section \(\PP\sO_\ell(1)\). The divisors
on \(\tilde{Y}\) are: \(E_{(x,\ell')}\) is the exceptional divisor of
\(\tilde{Y} \to Y\) above \((x,[\ell'])\); and \(\tilde{Y}_x\) is the strict
transform of the fibre of \(Y\) over \(x\).

\section{Type \texorpdfstring{\(\mathbf{N}_2 \oplus \mathbf{1}^{\oplus 2}\)}{N2+1+1}}
\label{surface-1+1+N2}
General \(q\)-bic surfaces of corank \(1\) are of type
\(\mathbf{N}_2 \oplus \mathbf{1}^{\oplus 2}\), see
\parref{surfaces-specializations}. By \parref{forms-aut-1^a+N2^b}, the
underlying \(q\)-bic form admits a canonical orthogonal decomposition
\[
(V,\beta) = (U,\beta_U) \perp (W,\beta_W)
\]
where, notation as in \parref{surfaces-corank-1}, the restriction of
\(\beta\) to \(U \coloneqq L_- \oplus L_+\) is of type \(\mathbf{N}_2\),
and the complement \(W\) is of type \(\mathbf{1}^{\oplus 2}\).

\subsection{Cone points}\label{surfaces-1+1+N2.cone-points}
Applying \parref{hypersurfaces-cone-points-classify} shows that \(X\) has
\(q+3\) cone points:
\begin{enumerate}
\item\label{surfaces-1+1+N2.cone-points.x+}
the singular point \(x_+\),
\item\label{surfaces-1+1+N2.cone-points.x-}
the special point \(x_-\), and
\item\label{surfaces-1+1+N2.cone-points.W}
the \(q+1\) Hermitian points of the type \(\mathbf{1}^{\oplus 2}\) subform
\((W,\beta_W)\).
\end{enumerate}
The cone points of the third type are determined with the help of
\parref{forms-hermitian-basics-orthogonals} and the computation of
\parref{forms-hermitian-examples}\ref{forms-hermitian-examples.nilpotent}.
In each case, there is a unique plane witnessing the cone point property:
this is \(\PP \Fr^{-1}(L_+^\perp)\) in
\ref{surfaces-1+1+N2.cone-points.x+}, and \(\mathbf{T}_{X,x}\) in
\ref{surfaces-1+1+N2.cone-points.x-} and \ref{surfaces-1+1+N2.cone-points.W}.

Consider the cones obtained in \ref{surfaces-1+1+N2.cone-points.x+} and
\ref{surfaces-1+1+N2.cone-points.x-}. Since
\[
\Fr^{-1}(L_+^\perp) = \langle L_+, W \rangle
\quad\text{and}\quad
\Fr^*(L_-)^\perp = \langle L_-, W \rangle
\]
the associated cones have base given by the \(q\)-bic points defined by
\((W,\beta_W)\). Consider a cone point \(x = \PP L\) as in
\ref{surfaces-1+1+N2.cone-points.W}. Its embedded tangent space has underlying
linear space \(\Fr^*(L)^\perp = \langle U, L \rangle\) and so the associated
cone is
\[
X \cap \PP\Fr^*(L)^\perp =
 \langle x,x_- \rangle \cup
q\langle x,x_+ \rangle.
\]
Since every line contains a cone point by
\parref{hypersurfaces-cone-points-subspaces-have}, this gives the first
statement of:

\begin{Proposition}\label{surfaces-1+1+N2.lines}
Every line in \(X\) passes through exactly one of \(x_-\) or \(x_+\). The
subscheme of \(\mathbf{F}_1(X)\) corresponding to lines through
\begin{itemize}
\item \(x_-\) consists of \(q+1\) reduced points, and
\item \(x_+\) consists of \(q+1\) points of multiplicity \(q^3\).
\end{itemize}
\end{Proposition}

\begin{proof}
This follows from the facts collected in \parref{qbic-surfaces-lines} and
\parref{surfaces-corank-1}: the lines through \(x_-\) are contained in the
smooth locus and hence give reduced points in \(\mathbf{F}_1(X)\); since
the Fano scheme has degree \((q+1)(q^3+1)\), the lines through \(x_+\) must
appear with multiplicity \(q^3\) each, by symmetry.
\end{proof}

The scheme of cone points \(X_{\mathrm{cone}}\), as defined in
\parref{hypersurfaces-cone-points-equations-general}, can now be described
by examining how cone points of a smooth \(q\)-bic surface must come together
in a degeneration to \(X\). The analysis relies on the observation that each
cone point in a smooth \(q\)-bic surface is the intersection of two distinct
lines contained in the surface.

\begin{Proposition}\label{surfaces-1+1+N2.cone-points-scheme}
The multiplicity of the scheme of cone points at a point \(x\) is
\[
\mathrm{mult}_x(X_{\mathrm{cone}}) =
\begin{dcases*}
q^5 &
if \(x\) is the singular point \(x_+\) as in
\parref{surfaces-1+1+N2.cone-points}\ref{surfaces-1+1+N2.cone-points.x-}, \\
1   & if \(x\) is the special point \(x_-\) as in
\parref{surfaces-1+1+N2.cone-points}\ref{surfaces-1+1+N2.cone-points.x+}, and \\
q^2 & if \(x\) is a Hermitian point as in
\parref{surfaces-1+1+N2.cone-points}\ref{surfaces-1+1+N2.cone-points.W}.
\end{dcases*}
\]
\end{Proposition}

\begin{proof}
Consider a flat family of \(q\)-bic surfaces containing \(X\) whose general
fibre is smooth and such that the special point \(x_-\) extends to a section.
A specific such family may be obtained via the construction of
\parref{qbic-points-basic-algebra-family}: consider, say,
\[
\mathcal{X} = \mathrm{V}(x_0^q x_1 + t x_0 x_1^q + x_2^{q+1} + x_3^{q+1})
\subset \PP^4 \times \mathbf{A}^1
\]
and abusively conflate the point \(x_- = (1:0:0:0)\) with the corresponding
constant section over \(\mathbf{A}^1\). Observe that the relative schemes
\(\mathcal{X}_{\mathrm{cone}}\) of cone points and
\(\mathbf{F}_1(\mathcal{X}/\mathbf{A}^1)\) of lines are flat over
\(\mathbf{A}^1\).

Consider how lines in a smooth general fibre \(X_t\) of \(\mathcal{X}\)
limit to lines in \(X\). Since the tangent space to \(x_-\) remains constant
along the family \(\mathcal{X}\), a line through \(x_-\) in \(X_t\) remains
a line through \(x_-\) in \(X\); let \(\ell^0,\ell^1,\ldots\,\ell^q\)
denote the \(q+1\) such lines, viewed interchangely as lying in \(X_t\) or as
the corresponding limit in \(X\). Since \(\ell^0 + \ell^1 + \cdots + \ell^q\)
is a hyperplane section, any other line \(\ell_t \subset X_t\) intersects
exactly one of the \(\ell^i\). Its limit \(\ell_0 \subset X\) must also
intersect \(\ell^i\). Since
each of the points \([\ell^i]\) is reduced in \(\mathbf{F}_1(X)\) by
\parref{surfaces-1+1+N2.lines}, \(\ell_0\) cannot coincide with \(\ell^i\)
and thus their point \(y_i\) of intersection one of the Hermitian points as in
\parref{surfaces-1+1+N2.cone-points}\ref{surfaces-1+1+N2.cone-points.W}. Thus
the \(q^2\) cone points of \(\ell^i \subset X_t\) distinct from \(x_-\) limit
to \(y_i\), implying it has multiplicity \(q^2\) in \(X_{\mathrm{cone}}\).
Since \(\deg(X_{\mathrm{cone}}) = (q^2+1)(q^3+1)\) by
\parref{hypersurfaces-cone-points-scheme-degree}, it follows that \(x_+\) has
multiplicity \(q^5\).
\end{proof}

\begin{Remark}\label{surfaces-1+1+N2.cone-points-scheme-addendum}
In fact, a slight modification of the arguments of
\parref{surfaces-1+1+N2.cone-points} and
\parref{surfaces-1+1+N2.cone-points-scheme} shows that for a \(q\)-bic
hypersurface of type \(\mathbf{N}_2 \oplus \mathbf{1}^{\oplus n-1}\),
its cone points consist of its singular point, special point, and the Hermitian
points coming from the subform \(\mathbf{1}^{\oplus n-1}\), and that the
multiplicities are \(q^{2n-1}\), \(1\), and \(q^2\), respectively.
\end{Remark}

Linear automorphisms of \(X\) can be obtained by specializing
\parref{forms-aut-1^a+N2^b.computation}:

\begin{Proposition}\label{surfaces-1+1+N2.auts}
Let \((V,\beta)\) be a \(q\)-bic form of type \(\mathbf{N}_2 \oplus \mathbf{1}^{\oplus 2}\).
Then \(\AutSch(V,\beta)\) is isomorphic to the \(1\)-dimensional subgroup scheme
of \(\GL_4\) consisting of
\[
\left(
\begin{array}{@{}c|c@{}}
\begin{matrix} \lambda & \epsilon \\ 0 & \lambda^{-q} \end{matrix} &
\begin{matrix} y_1 & y_2 \\ 0 & 0 \\ \end{matrix} \\
\hline
\begin{matrix} 0 & x_1 \\ 0 & x_2 \end{matrix} &
A
\end{array}
\right)
\]
with \(A \in \mathrm{U}_2(q)\), \(\lambda \in \mathbf{G}_m\), \(\epsilon, x_1,x_2 \in \boldsymbol{\alpha}_q\),
and \(y_1,y_2 \in \boldsymbol{\alpha}_{q^2}\), subject to the equation
\[
\pushQED{\qed}
\left(\begin{smallmatrix} x_1 \\ x_2 \end{smallmatrix}\right) =
\lambda^{-q} (A^{\vee,(q)} \beta_W)^{-1}
\left(\begin{smallmatrix} y_1 \\ y_2 \end{smallmatrix}\right).
\qedhere
\popQED
\]
\end{Proposition}

\section{Type \texorpdfstring{\(\mathbf{N}_4\)}{N4}}
\label{surface-N4}
The next \(q\)-bic surfaces of corank \(1\) are those of type \(\mathbf{N}_4\).
By \parref{hypersurfaces-cone-points-classify}, the cone points such a
surface \(X\) are its singular point \(x_+\) and its special point \(x_-\): the
computation of
\parref{forms-hermitian-examples}\ref{forms-hermitian-examples.nilpotent}
implies \(X\) does not have any Hermitian points. The associated cones are of
type \(\mathbf{0} \oplus \mathbf{N}_2\), and their components are
\[
X \cap \PP\Fr^{-1}(L_+^\perp) = \ell_0 \cup q\ell_+
\quad\text{and}\quad
X \cap \PP\Fr^*(L_-)^\perp = q \ell_0 \cup \ell_-
\]
where \(\ell_0 \coloneqq \langle x_-,x_+ \rangle\),
\(\ell_+ \coloneqq \PP\Fr^{-1}(\Fr^{-1}(L_+^\perp)^\perp)\),
and \(\ell_- \coloneqq \PP\Fr^*(\Fr^*(L_-)^\perp)^\perp\). That is, \(\ell_0\)
is the line between \(x_\pm\), and \(\ell_\pm\) are the first pieces of
the \(\perp\)- and \(\Fr^*(\perp)\)-filtrations of \((V,\beta)\), see
\parref{forms-canonical-filtration} and \parref{forms-canonical-filtration-second}.

The scheme of cone points is now easily determined:

\begin{Proposition}\label{surfaces-N4.cone-points-scheme}
\(\mathrm{mult}_{x_+}(X_{\mathrm{cone}}) = q^3 + 1\) and
\(\mathrm{mult}_{x_-}(X_{\mathrm{cone}}) = q^2(q^3 + 1)\).
\end{Proposition}

\begin{proof}
Consider a flat family of \(q\)-bic surfaces with general fibre of type
\(\mathbf{N}_2 \oplus \mathbf{1}^{\oplus 2}\) and special fibre \(X\).
This is possible by \parref{surfaces-specializations}.
Then the scheme of cone points also fit into a flat
family over the base. By the classification of cone points in the two types,
each of the Hermitian cone points
\parref{surfaces-1+1+N2.cone-points}\ref{surfaces-1+1+N2.cone-points.W}
in the general fibre specializes to either \(x_+\) or \(x_-\). Since the
line \(\ell_-\) is contained in the smooth locus of \(X\), considering
the family of lines implies that exactly one of the Hermitian cone points
specializes to \(x_-\), whereas the remaining \(q\) must specialize to \(x_+\).
The multiplicities are now determined from \parref{surfaces-1+1+N2.cone-points-scheme}.
\end{proof}

The scheme structure of the Fano scheme can now be similarly determined:

\begin{Proposition}\label{surfaces-N4.lines}
Let \(\ell\) be a line in \(X\). Then
\[
\mathrm{mult}_{[\ell]}(\mathbf{F}_1(X)) =
\begin{dcases*}
q^4 & if \(\ell = \ell_+\), \\
q(q^2+1) & if \(\ell = \ell_0\), and \\
1 & if \(\ell = \ell_-\).
\end{dcases*}
\]
\end{Proposition}

\begin{proof}
Since \(\ell_-\) is contained in the smooth locus, it is a reduced point in
\(\mathbf{F}_1(X)\), see \parref{qbic-surfaces-lines}. For the remaining
points, consider as in the proof of \parref{surfaces-N4.cone-points-scheme},
a flat family of \(q\)-bic surfaces with general fibre of type
\(\mathbf{N}_2 \oplus \mathbf{1}^{\oplus 2}\) and special fibre \(X\). Then the
relative Fano scheme of lines is flat over the base. The lines that limit to
\(\ell_+\) are those that join the singular point and one of the \(q\)
Hermitian cone points limiting to \(x_+\). By \parref{surfaces-1+1+N2.lines},
each such line has multiplicity \(q^3\) in the Fano scheme of the general
fibre. Thus \([\ell_+]\) has multiplicity \(q^4\) in \(\mathbf{F}_1(X)\). The
remaining multiplicity is deduced from the fact that
\(\deg(\mathbf{F}_1(X)) = (q+1)(q^3+1)\).
\end{proof}

The linear automorphisms of \(X\) are given by:

\begin{Proposition}\label{surfaces-N4.auts}
Let \((V,\beta)\) be a \(q\)-bic form of type \(\mathbf{N}_4\). Then
\(\AutSch(V,\beta)\) is isomorphic to the \(2\)-dimensional closed subgroup
scheme of \(\GL_4\) consisting of
\[
\begin{pmatrix}
\lambda & \epsilon_3                & t             & \epsilon_1 \\
0       & \lambda^{-q}              & 0             & 0 \\
0       & \epsilon_2                & \lambda^{q^2} & -\lambda^{-q^2(q-1)}\epsilon_2^q \\
0       & -\lambda^{-q^2(q+1)}t^q   & 0             & \lambda^{-q^3}
\end{pmatrix}
\]
where
\(\lambda \in \mathbf{G}_m\),
\(t \in \mathbf{G}_a\), and
\(\epsilon_i \in \boldsymbol{\alpha}_{q^i}\) for \(i = 1,2,3\), and subject to
the equation
\[
\epsilon_2^q t^q -
\lambda^{q(q^2-q+1)} \epsilon_2 -
\lambda^{q^3} \epsilon_3^q = 0.
\]
\end{Proposition}

\begin{proof}
Choose a basis \(V = \langle e_0,e_1,e_2,e_3 \rangle\) such that
\(\Gram(\beta;e_0,e_1,e_2,e_3) = \mathbf{N}_4\). Then the \(\perp\)-filtration
of \(V\) is given by
\(\langle e_0 \rangle \subset \langle e_0, e_2 \rangle \subset \langle e_0,e_2,e_3 \rangle\),
and the first step of the \(\Fr^*(\perp)\)-filtration is given by \(\langle
e_3^{(q)} \rangle\), see \parref{forms-canonical-filtration} and
\parref{forms-canonical-filtration-second}. Thus by
\parref{forms-aut-canonical-filtration}, \(\AutSch(V,\beta)\) is isomorphic
to the closed subgroup scheme of \(\GL_4\) consisting of matrices satisfying
\[
\begin{pmatrix}
a_{00}^q & 0        & 0        & 0 \\
a_{01}^q & a_{11}^q & a_{21}^q & a_{31}^q \\
a_{02}^q & 0        & a_{22}^q & 0 \\
0        & 0        & 0        & a_{33}^q
\end{pmatrix}
\begin{pmatrix}
0 & 1 & 0 & 0 \\
0 & 0 & 1 & 0 \\
0 & 0 & 0 & 1 \\
0 & 0 & 0 & 0
\end{pmatrix}
\begin{pmatrix}
a_{00} & a_{01} & a_{02} & a_{03} \\
0      & a_{11} & 0      & 0      \\
0      & a_{21} & a_{22} & a_{23} \\
0      & a_{31} & 0      & a_{33}
\end{pmatrix}
=
\begin{pmatrix}
0 & 1 & 0 & 0 \\
0 & 0 & 1 & 0 \\
0 & 0 & 0 & 1 \\
0 & 0 & 0 & 0
\end{pmatrix}
\]
and \(a_{03}^q = a_{23}^q = 0\). Expanding gives \(7\)
equations: \(a_{00}^q a_{11} = a_{11}^q a_{22} = a_{22}^q a_{33} = 0\) and
\[
a_{01}^q a_{11} + a_{11}^q a_{21} + a_{21}^q a_{31} =
a_{11}^q a_{23} + a_{21}^q a_{33} =
a_{02}^q a_{11} + a_{22}^q a_{31} =
0.
\]
The first three equations give the diagonal entries upon setting
\(\lambda \coloneqq a_{00}\). The second and third displayed equations may be
rearranged to give
\[
a_{23} = -a_{21}^q a_{33}/a_{11}^q = -\lambda^{-q^2(q+1)} a_{21}^q
\quad\text{and}\quad
a_{31} = -a_{02}^q a_{11}/a_{22}^q = -\lambda^{-q^2(q-1)} a_{12}^q.
\]
Substituting this into the remaining displayed equation and rearranging gives
\[
\lambda^{-q(q^2+1)}
(\lambda^{q^2} a_{01}^q + \lambda^{q(q^2-q+1)} a_{21} - a_{02}^q a_{21}^q) = 0.
\]
Setting \(t \coloneqq a_{02}\),
\(\epsilon_1 \coloneqq a_{03}\),
\(\epsilon_2 \coloneqq a_{21}\), and
\(\epsilon_3 \coloneqq a_{01}\) now gives the result.
\end{proof}

\section{Type \texorpdfstring{\(\mathbf{N}_3 \oplus \mathbf{1}\)}{N3+1}}
\label{surface-1+N3}
The next specialization is a \(q\)-bic surface \(X\) of type
\(\mathbf{N}_3 \oplus \mathbf{1}\). It's singular point \(x_+\) and special
point \(x_-\) span a line \(\ell \coloneqq \langle x_+, x_- \rangle\) which is
contained in \(X\). By \parref{hypersurfaces-cone-points-classify} together
with \parref{forms-hermitian-basics-orthogonals} and the computation of
\parref{forms-hermitian-examples}\ref{forms-hermitian-examples.nilpotent}, it
follows that the locus of cone points of \(X\) is supported on this line
\(\ell\). Moreover, the associated cone at each point is of type
\(\mathbf{0}^{\oplus 2} \oplus \mathbf{1}\) and consists of \(\ell\) with
multiplicity \(q+1\). It follows that \(\ell\) is the only line in \(X\);
this determines the Fano scheme:

\begin{Proposition}\label{surfaces-1+N3.lines}
\(\mathbf{F}_1(X)\) is supported on \([\ell_0]\) with multiplicity
\((q+1)(q^3+1)\). \qed
\end{Proposition}

The linear automorphisms of \(X\) are given by:

\begin{Proposition}\label{surfaces-1+N3.auts}
Let \((V,\beta)\) be a \(q\)-bic form of type \(\mathbf{N}_3 \oplus \mathbf{1}\).
Then \(\AutSch(V,\beta)\) is isomorphic to the \(3\)-dimensional closed subgroup
scheme of \(\GL_4\) consisting of
\[
\begin{pmatrix}
\lambda & t_1 & \epsilon & t_2 \\
0       & \lambda^{-q} & 0 & 0 \\
0 & -\lambda^{q(q-1)} t_1^q - \lambda^{-q} t_2^{q(q+1)} & \lambda^{q^2} & t_2^{q^2} \\
0 & -\zeta \lambda^{-q} t_2^q & 0 & \zeta
\end{pmatrix}
\]
where
\(\lambda \in \mathbf{G}_m\),
\(t_1,t_2 \in \mathbf{G}_a\),
\(\zeta \in \boldsymbol{\mu}_{q+1}\), and
\(\epsilon \in \boldsymbol{\alpha}_q\).
\end{Proposition}

\begin{proof}
Choose a basis \(V = \langle e_0,e_1,e_2,e_3 \rangle\) such that
\(\Gram(\beta;e_0,e_1,e_2,e_3) = \mathbf{N}_3 \oplus \mathbf{1}\). Then
the \(\perp\)-filtration of \(V\) is
\(\langle e_0 \rangle
\subset \langle e_0,e_2 \rangle
\subset \langle e_0,e_2,e_3 \rangle\) and the first piece of the
\(\Fr^*(\perp)\)-filtration is \(\langle e_2^{(q)} \rangle\), see
\parref{forms-canonical-filtration} and
\parref{forms-canonical-filtration-second}. Thus by
\parref{forms-aut-canonical-filtration}, \(\AutSch(V,\beta)\) is isomorphic to
the closed subgroup scheme of \(\GL_4\) consisting of matrices satisfying
\[
\begin{pmatrix}
a_{00}^q & 0        & 0        & 0        \\
a_{01}^q & a_{11}^q & a_{21}^q & a_{31}^q \\
0        & 0        & a_{22}^q & 0        \\
a_{03}^q & 0        & a_{23}^q & a_{33}^q
\end{pmatrix}
\begin{pmatrix}
0 & 1 & 0 & 0 \\
0 & 0 & 1 & 0 \\
0 & 0 & 0 & 0 \\
0 & 0 & 0 & 1
\end{pmatrix}
\begin{pmatrix}
a_{00} & a_{01} & a_{02} & a_{03} \\
0      & a_{11} & 0      & 0      \\
0      & a_{21} & a_{22} & a_{23} \\
0      & a_{31} & 0      & a_{33}
\end{pmatrix}
=
\begin{pmatrix}
0 & 1 & 0 & 0 \\
0 & 0 & 1 & 0 \\
0 & 0 & 0 & 0 \\
0 & 0 & 0 & 1
\end{pmatrix}
\]
and \(a_{02}^q = 0\). Expanding gives \(6\) more equations:
\(a_{00}^q a_{11} = a_{11}^q a_{22} = a_{33}^{q+1} = 1\), and
\[
a_{03}^q a_{11} + a_{33}^q a_{31} =
a_{11}^q a_{23} + a_{31}^q a_{33} =
a_{01}^q a_{11} + a_{11}^q a_{21} + a_{31}^{q+1} = 0.
\]
The first three equations yield the diagonal entries upon setting
\(\zeta \coloneqq a_{33}\) and \(\lambda \coloneqq a_{00}\). Then rearranging
the first two displayed equations gives
\[
a_{31} = -a_{03}^q a_{11}/a_{33}^q = -\zeta \lambda^{-q} a_{03}^q
\quad\text{and}\quad
a_{23} = -a_{31}^q a_{33}/a_{11}^q = a_{03}^{q^2}.
\]
The final equation then gives
\(a_{21} = -\lambda^{q(q-1)} a_{01}^q- \lambda^{-q} a_{03}^{q(q+1)}  \).
Setting \(t_1 \coloneqq a_{01}\), \(t_2 \coloneqq a_{03}\), and
\(\epsilon \coloneqq a_{02}\) finishes the computation.
\end{proof}

\section{Type \texorpdfstring{\(\mathbf{0} \oplus \mathbf{1}^{\oplus 3}\)}{0+1+1+1}}
\label{surface-0+1+1+1}
The final type of corank \(1\) surface is a cone over a smooth \(q\)-bic curve.
Let \(L \coloneqq L_+ = L_-\) denote the radical of the underlying \(q\)-bic
form, and let \(x \coloneqq x_+ = x_-\) denote the vertex of the surface \(X\).
As in \parref{surfaces-corank-1.projection},
write \(W \coloneqq V/L\) and let \(C \subset \PP W\) be the smooth \(q\)-bic
curve induced from \((V,\beta)\);
this forms the base of the cone.

Since \(C\) contains no lines, the lines in \(X\) must pass through the vertex.
Thus the support of \(\mathbf{F}_1(X)\) is canonically identified with \(C\).
Its schematic structure, however, is quite interesting. In the following, embed
\(\PP W\) as the closed subscheme of \(\mathbf{G}(2,V)\) parameterizing lines
in \(\PP V\) through \(x\).

\begin{Proposition}\label{surface-0+1+1+1.lines}
The Fano scheme \(\mathbf{F}_1(X)\) satisfies
\[
\mathbf{F}_1(X)_{\mathrm{red}} =
\mathbf{F}_1(X) \cap \PP W =
C
\]
as subschemes of \(\mathbf{G}(2,V)\), is generically a \(q\)-order
neighbourhood of \(C\), and has embedded points of multiplicity \(q+1\) at the
Hermitian points of \(C\).
\end{Proposition}

\begin{proof}
The first statement follows from the preceding comments. For the latter
statements, choose coordinates
\((x_0:x_1:x_2:x_3)\) on \(\PP V = \PP^3\) so that
\[
X =
\mathrm{V}(x_1^q x_2 + x_1 x_2^q + x_3^{q+1}) \subset \PP^3.
\]
Then \(\mathbf{F}_1(X)\) is covered by the Pl\"ucker charts
\(\mathrm{D}(p_{01})\) and \(\mathrm{D}(p_{02})\); by symmetry, it suffices to
consider \(\mathrm{D}(p_{01})\), wherein \(\mathbf{F}_1(X)\) has coordinates
and equations
\[
\begin{pmatrix} 1 & 0 & a_2 & a_3 \\ 0 & 1 & b_2 & b_3 \end{pmatrix}
\quad
\begin{array}{rr}
a_2 + a_2^q + a_3^{q+1} = 0, &
b_2 + a_3^q b_3 = 0, \\
b_2^q + a_3 b_3^q = 0, &
b_3^{q+1} = 0.
\end{array}
\]
Setting \(b_2 = -a_3^q b_3\) and substituting it into the third equation gives
\[ (a_3 - a_3^{q^2}) b_3^q = 0. \]
Therefore either \(a_3 - a_3^{q^2} = b_3^{q+1} = 0\), yielding the embedded
points of multiplicity \(q+1\) along the Hermitian points of \(C\), or
\(b_3^q = 0\), yielding the generic nonreduced structure of multiplicity \(q\).
\end{proof}

\section{Type \texorpdfstring{\(\mathbf{N}_2^{\oplus 2}\)}{N2+N2}}
\label{surface-N2+N2}
General \(q\)-bic surfaces of corank \(2\) are of type \(\mathbf{N}_2^{\oplus 2}\),
see \parref{surfaces-specializations}. Let \(X\) be such a surface, and,
as in \parref{surfaces-corank-2}, write \(\ell_+\) and \(\ell_-\) for its
singular and special lines.

Consider in detail the blowup \(\tilde{X} \to X\) along \(\ell_+\), continuing
the analysis started in \parref{surfaces-corank-2-projection}.
By \parref{surfaces-corank-2-planes}, \(X\) does not contain a plane, so the
discussion at the end of \parref{surfaces-corank-2-strict-transform} gives
a short exact sequence
\[
0 \to
\Fr^*(\mathcal{V})^\perp \to
\mathcal{V} \xrightarrow{\beta_{\mathcal{V}}}
\sO_{\PP W}(q) \to
0
\]
of vector bundles on \(\PP W\); the kernel is furthermore isotropic for
\(\beta_{\mathcal{V}}\). By \parref{surfaces-corank-2-no-planes},
\(\tilde{X} \to \PP W\) is the ruled surface associated with
\(\Fr^*(\mathcal{V})^\perp\). The kernel may be explicitly identified; the
description arises by observing that the projection induces an isomorphism
\(\ell_- \cong \PP W\).

\begin{Proposition}\label{surface-N2+N2.resolution}
The blowup \(\tilde{X} \to X\) along \(\ell_+\) is isomorphic to the
projective bundle \(\PP\Fr^*(\mathcal{V})^\perp\) over \(\PP W\). There
is a canonical split short exact sequence
\[
0 \to
\Fr^*(\Omega^1_{\PP W}(1)) \to
\Fr^*(\mathcal{V})^\perp \to
\sO_{\PP W}(-1) \to
0
\]
and the exceptional divisor of \(\tilde{X} \to X\) is given by the
subbundle \(\PP\Fr^*(\Omega^1_{\PP W}(1))\).
\end{Proposition}

\begin{proof}
The short exact sequence arises from the exact commutative diagram
\[
\begin{tikzcd}
& \sO_{\PP W}(-1) \rar[equal]
& \sO_{\PP W}(-1) \\
0 \rar
& \Fr^*(\mathcal{V})^\perp \rar \uar[two heads]
& \mathcal{V} \rar["\beta_{\mathcal{V}}"] \uar[two heads]
& \sO_{\PP W}(q) \rar
& 0  \\
0 \rar
& \Fr^*(\Omega^1_{\PP W}(1)) \uar[hook] \rar
& L_{+,\PP W} \rar \uar[hook]
& \Fr^*(\sO_{\PP W}(-1))^\vee \rar \uar[equal]
& 0
\end{tikzcd}
\]
in which the bottom row is identified with the \(\Fr^*\)-twisted Euler sequence
by:
\[
\Fr^*(\Omega_{\PP W}^1(1)) \cong
\ker\big(\mathrm{eu}^{\vee,(q)}_{\PP W} \circ \beta \colon
L_{+,\PP W} \xrightarrow{\cong}
\Fr^*(L_{-,\PP W})^\vee \twoheadrightarrow
\Fr^*(\sO_{\PP W}(-1))^\vee
\big).
\]
Since the vertical sequence for \(\mathcal{V}\) is split, so is that for
\(\Fr^*(\mathcal{V})^\perp\). The exceptional divisor of \(\tilde{X} \to X\)
correspond to points through \(\ell_+\), and from the diagram above, these
lie in the subbundle \(\PP\Fr^*(\Omega^1_{\PP W}(1))\).
\end{proof}

\subsection{Cone points}\label{surfaces-N2+N2.cone-points}
By \parref{hypersurfaces-cone-points-classify} together with
\parref{forms-hermitian-basics-orthogonals} and
\parref{forms-hermitian-examples}\ref{forms-hermitian-examples.nilpotent}, the
set of cone points of \(X\) is given by \(\ell_- \cup \ell_+\). The cone
associated with a cone point \(x = \PP L\) is given by
\[
\begin{dcases*}
X \cap \PP \Fr^{-1}(L^\perp) = \ell \cup q \ell_+ & if \(x \in \ell_+\), and \\
X \cap \PP \Fr^*(L)^\perp = \ell_- \cup q \ell    & if \(x \in \ell_-\),
\end{dcases*}
\]
where the line \(\ell\) contains \(x\) and intersects \(\ell_\mp\) at the
point corresponding to the \(1\)-dimensional subspaces
\(\Fr^{-1}(L^\perp) \cap L_-\) and \(\Fr^*(L)^\perp \cap L_+\), respectively.

Since every line passes through a cone point, this shows that the scheme
\(\mathbf{F}_1(X)\) has three components: two \(0\)-dimensional components
supported on the points \([\ell_\pm]\), and a \(1\)-dimensional component
parameterizing the lines \(\ell\) whose support may be identified with the
projective line \(\PP W\) by \parref{surfaces-corank-2-lines}.
The scheme structure is as follows:

\begin{Proposition}\label{surfaces-N2+N2.lines}
The Fano scheme \(\mathbf{F}_1(X)\) has multiplicity
\begin{enumerate}
\item\label{surfaces-N2+N2.lines.l-}
\(1\) at the isolated point \([\ell_-]\),
\item\label{surfaces-N2+N2.lines.l+}
\(q^4\) at the isolated point \([\ell_+]\), and
\item\label{surfaces-N2+N2.lines.1-dim}
\(q\) along the \(1\)-dimensional component \(\PP W\).
\end{enumerate}
\end{Proposition}

\begin{proof}
Since \(\ell_-\) is contained in the smooth locus of \(X\), the corresponding
point is reduced in \(\mathbf{F}_1(X)\). For the remainder, choose coordinates
\((x_0:x_1:x_2:x_3)\) on \(\PP V = \PP^3\) so that
\[
X =
\mathrm{V}(x_0^q x_2 + x_1^q x_3) \subset
\PP^3.
\]
Then \(\ell_+ = (0:0:*:*)\). The corresponding point is supported on the
origin in the Pl\"ucker chart \(\mathrm{D}(p_{23})\) in \(\mathbf{G}(2,4)\),
with coordinates and equations:
\[
\begin{pmatrix}
a_0 & a_1 & 1 & 0 \\
b_0 & b_1 & 0 & 1
\end{pmatrix}
\quad\quad
a_0^q = a_1^q = b_0^q = b_1^q = 0.
\]
The one-dimensional component is covered by \(\mathrm{D}(p_{03})\)
and \(\mathrm{D}(p_{12})\), and by symmetry, it suffices to consider the
former. The coordinates and equations of \(\mathbf{F}_1(X)\) there are
\[
\begin{pmatrix}
1 & a_1 & a_2 & 0 \\
0 & b_1 & b_2 & 1
\end{pmatrix}
\quad\quad
a_2 = b_2 + a_1^q = b_1^q = 0
\]
verifying that it is a  \(q\)-th order neighbourhood of an affine line.
\end{proof}

The automorphism group scheme of the underlying \(q\)-bic form is obtained
directly from the general computation of \parref{forms-aut-1^a+N2^b.computation}
with \(a = 2\) and \(b = 0\):

\begin{Proposition}\label{surfaces-N2+N2.auts}
Let \((V,\beta)\) be a \(q\)-bic form of type \(\mathbf{N}_2^{\oplus 2}\). Then
\(\AutSch(V,\beta)\) is isomorphic to the \(4\)-dimensional closed subgroup
scheme of \(\GL_4\) consisting of matrices
\[
\begin{pmatrix}
a_{00} & a_{01} & \epsilon_{02} & \epsilon_{03} \\
a_{10} & a_{11} & \epsilon_{12} & \epsilon_{13} \\
0      & 0      & a_{11}^q/\Delta^q & -a_{01}^q/\Delta^q \\
0      & 0      & -a_{10}^q/\Delta^q & a_{00}^q/\Delta^q
\end{pmatrix}
\]
where
\(A \coloneqq
\left(\begin{smallmatrix}
a_{00} & a_{01} \\
a_{10} & a_{11}
\end{smallmatrix}\right) \in \GL_2\),
\(\Delta \coloneqq \det(A)\), and
\(\epsilon_{02}, \epsilon_{03}, \epsilon_{12}, \epsilon_{13} \in \boldsymbol{\alpha}_q\).
\end{Proposition}

\begin{proof}
For a direct computation, choose
a basis \(\Fr^*(V)^\perp = \langle e_0,e_1 \rangle\), and let
\(V^\perp = \langle e_2^{(q)},e_3^{(q)} \rangle\) be the dual basis with respect
to \(\beta\). Since automorphisms must preserve the two kernels, \(\AutSch(V,\beta)\)
is isomorphic to the closed subgroup scheme of \(\GL_4\) consisting of block
matrices that satisfy
\[
\begin{pmatrix}
A_-^{\vee,(q)} & 0 \\
0 & A_+^{\vee,(q)}
\end{pmatrix}
\begin{pmatrix}
0 & I_2 \\
0 & 0
\end{pmatrix}
\begin{pmatrix}
A_- & B \\
0 & A_+
\end{pmatrix}
=
\begin{pmatrix}
0 & I_2 \\ 0 & 0
\end{pmatrix}
\]
where \(A_\pm \in \GL_2\), \(B \in \mathbf{Mat}_{2 \times 2}\) satisfies
\(B^{(q)} = 0\), and \(I_2 \coloneqq \left(\begin{smallmatrix} 1 & 0 \\ 0 & 1 \end{smallmatrix}\right)\). Expanding
shows that \(A_-^{\vee,(q)} A_+ = I_2\), so rearranging and setting \(A \coloneqq A_-\)
finishes the computation.
\end{proof}

\section{Type \texorpdfstring{\(\mathbf{0} \oplus \mathbf{N}_2 \oplus \mathbf{1}\)}{0+N2+1}}
\label{surface-0+1+N2}
The last reduced and irreducible \(q\)-bic surfaces are those of type
\(\mathbf{0} \oplus \mathbf{N}_2 \oplus \mathbf{1}\). The singular and special
lines \(\ell_\pm\) of such a surface \(X\) intersect
at the vertex \(x_0 \coloneqq \PP L_0\), with \(L_0 \coloneqq \rad(\beta)\).
To continue the analysis of the blowup \(\tilde{X} \to X\) along \(\ell_+\)
from \parref{surfaces-corank-2-projection}, note that \(X\) contains no planes
by \parref{surfaces-corank-2-no-planes}, so there is a short exact sequence
\[
0 \to
\Fr^*(\mathcal{V})^\perp \to
\mathcal{V} \xrightarrow{\beta_{\mathcal{V}}}
\sO_{\PP W}(q) \to
0
\]
of vector bundles on \(\PP W\), and \(\tilde{X}\) is isomorphic to
\(\PP\Fr^*(\mathcal{V})^\perp\) over \(\PP W\). There is a distinguished
closed point \(\infty \in \PP W\): Let \(\bar{V} \coloneqq V/L_0\) and write
\(\bar{L}_\pm \coloneqq L_\pm/L_0\) for the images of the kernels. Then \(\beta\)
passes to the quotient to yield a \(q\)-bic form \((\bar{V},\beta_{\bar{V}})\)
of type \(\mathbf{N}_2 \oplus \mathbf{1}^{\oplus 1}\). By
\parref{forms-aut-1^a+N2^b}, there is a canonical orthogonal decomposition
\[
(\bar{V},\beta_{\bar{V}}) =
(\bar{U},\beta_{\bar{U}}) \perp
(\bar{W},\beta_{\bar{W}})
\]
where \(\bar{U} \coloneqq \bar{L}_- \oplus \bar{L}_+\) has type \(\mathbf{N}_2\)
and its complement \(\bar{W}\) has type \(\mathbf{1}^{\oplus 2}\). Its
image along the quotient map \(\bar{V} \to W\) is \(1\)-dimensional so
it underlies a point \(\infty \in \PP W\).

\begin{Proposition}\label{surface-0+1+N2.resolution}
The blowup \(\tilde{X} \to X\) along \(\ell_+\) is isomorphic to the projective
bundle \(\PP\Fr^*(\mathcal{V})^\perp\) over \(\PP W\). There is a canonical
split short exact sequence
\[
0 \to
L_{0,\PP W} \to
\Fr^*\mathcal{V}^\perp \to
\sO_{\PP W}(-q\infty) \otimes \sO_{\PP W}(-1) \to
0
\]
and the exceptional divisor of \(\tilde{X} \to X\) is the subbundle
\(\PP L_{0,\PP W}\).
\end{Proposition}

\begin{proof}
The first statements follow from the preceding discussion. The short exact
sequence arises from the exact commutative diagram
\[
\begin{tikzcd}[row sep=1.5em]
0 \rar
& L_{0,\PP W} \rar \dar[hook]
& L_{+,\PP W} \rar["\beta"'] \dar[hook]
& \Fr^*(\bar{L}_-)^\vee_{\PP W} \rar \dar[hook]
& 0 \\
0 \rar
& \Fr^*(\mathcal{V})^\perp \rar \dar[two heads]
& \mathcal{V} \rar["\beta_{\mathcal{V}}"] \dar[two heads]
& \sO_{\PP W}(q) \rar \dar[two heads]
& 0 \\
0 \rar
& \sO_{\PP W}(-q\infty) \otimes \sO_{\PP W}(-1) \rar
& \sO_{\PP W}(-1) \rar
& \sO_{q\infty} \rar
& 0
\end{tikzcd}
\]
in which the right column arises by noting \(W \cong \bar{W} \oplus \bar{L}_-\)
and that the map to \(\sO_{\PP W}(q)\) is the \(q\)-th power of the
coordinate vanishing on \(\bar{L}_-\). Since the exceptional divisor of
\(\tilde{X} \to X\) corresponds to points intersecting \(\ell_+\),
it is the subbundle \(\PP L_{0,\PP W}\).
\end{proof}

That the point \(\infty \in \PP W\) is distinguished in regards to
\(\tilde{X} \to \PP W\) can be alternatively seen upon considering the
family of \(q\)-bic curves \(X_{\PP \psi} \to \PP W\) from
\parref{surfaces-corank-2-projection}.

\begin{Proposition}\label{surfaces-0+1+N2.fibration}
The fibre \(X_{\PP\psi, y}\) over \(y \in \PP W\) is of type
\[
\mathrm{type}(X_{\PP\psi,y}) =
\begin{dcases*}
\mathbf{0} \oplus \mathbf{N}_2 & if \(y \neq \infty\), and \\
\mathbf{0} \oplus \mathbf{1}^{\oplus 2} & if \(y = \infty\).
\end{dcases*}
\]
\end{Proposition}

\begin{proof}
The quotient
map \(\bar{V} \to W\) induces linear projection \(\PP\bar{V} \dashrightarrow
\PP W\) at the point \(\PP \bar{L}_+\). Since this is the singular point of the
\(q\)-bic curve \(C\) at the base of the cone, it induces an isomorphism
\(\pi \colon C \setminus \PP\bar{L}_+ \to \PP W \setminus \infty\). Given a point
\(y \in \PP W \setminus \infty\), write \(\ell_y\) for the line in \(X\)
corresponding to the point \(\pi^{-1}(y)\); write \(\ell_\infty \coloneqq \ell_+\).
Then since planes through \(\ell_+\) in \(\PP V\) correspond to lines through
\(\PP\bar{L}_+\) in \(\PP\bar{V}\), it follows that
\[
X_{\PP\psi,y} = \ell_y \cup q \ell_+
\quad\text{for all}\; y \in \PP W.
\qedhere
\]
\end{proof}

\subsection{Cone points}\label{surfaces-0+1+N2.cone-points}
As in \parref{surfaces-N2+N2.cone-points}, the set of cone points of \(X\) is
given by \(\ell_- \cup \ell_+\). Let \(x \coloneqq \PP L\) be a cone point.
If \(x\) is not the vertex of \(X\), then there is a unique associated cone
and it is given by
\[
\begin{dcases*}
X \cap \PP\langle L_-, \bar{W} \rangle = (q+1)\ell_- & if \(x \in \ell_- \setminus x_0\), and \\
X \cap \PP\langle L_+, \bar{W} \rangle = (q+1)\ell_+ & if \(x \in \ell_+ \setminus x_0\).
\end{dcases*}
\]
When \(x\) is the vertex of \(X\), then any plane through \(x\) will intersect
\(X\) at a \(q\)-bic curve which is a cone, thereby witnessing it as a cone point.

Since the \(q\)-bic curve \(C\) at the base of the cone contains no lines,
every line in \(X\) passes through the vertex. Embedding \(\PP \bar{V}\) as
the subscheme of \(\mathbf{G}(2,V)\) parameterizing lines through \(x_0\),
this gives the first statement of:

\begin{Proposition}\label{surfaces-0+1+N2.lines}
The Fano scheme \(\mathbf{F}_1(X)\) satisfies, as subschemes of \(\mathbf{G}(2,V)\),
\[
\mathbf{F}_1(X)_{\mathrm{red}} =
\mathbf{F}_1(X) \cap \PP\bar{V} = C,
\]
is generically a \(q\)-order neighbourhood
of \(C\), and has embedded points of multiplicities \(q+1\) and \(q^3(q+1)\)
supported on \([\ell_-]\) and \([\ell_+]\), respectively.
\end{Proposition}

\begin{proof}
Choose coordinates \((x_0:x_1:x_2:x_3)\) of \(\PP V = \PP^3\) so that
\[
X = \mathrm{V}(x_1^q x_2 + x_3^{q+1}) \subset \PP^3.
\]
Then \(\mathbf{F}_1(X)\) is covered by the Pl\"ucker charts \(\mathrm{D}(p_{01})\)
and \(\mathrm{D}(p_{02})\). Its coordinates and equations in the former are
\[
\begin{pmatrix}
1 & 0 & a_2 & a_3 \\
0 & 1 & b_2 & b_3
\end{pmatrix}
\quad\quad
a_3^{q+1} =
a_3^q b_3 =
a_2 + a_3 b_3^q =
b_2 + b_3^{q+1} =
0.
\]
Then \(\PP(V/L)\) is defined by \(a_2 = a_3 = 0\) and
\([\ell_-]\) is the origin. The last two
equations eliminate \(a_2\) and \(b_2\). Equation two implies either
\(a_3^q = 0\) or \(b_3 = 0\). The first case shows \(\mathbf{F}_1(X)\) is a
\(q\)-order neighbourhood of \(C\). The second case gives
\(a_2 = b_2 = b_3 = a_3^{q+1} = 0\), yielding the embedded point of
multiplicity \(q+1\) on \([\ell_-]\).

The coordinates and equations of \(\mathbf{F}_1(X)\) in \(\mathrm{D}(p_{02})\)
are given by
\[
\begin{pmatrix}
1 & a_1 & 0 & a_3 \\
0 & b_1 & 1 & b_3
\end{pmatrix}
\quad\quad
a_3^{q+1} =
a_1^q + a_3^q b_3 =
a_3 b_3^q =
b_1^q + b_3^{q+1} = 0.
\]
Then \(\PP(V/L)\) is defined by \(a_1 = a_3 = 0\) and \([\ell_+]\)
is the origin. The third equation implies either \(a_3 = 0\) or \(b_3^q = 0\).
The first case implies \(a_3 = a_1^q = 0\), again witnessing the \(q\)-order
thickening of \(C\).
The second case implies \(b_1^q = b_3^q = a_3^{q+1} = a_1^q + a_3^q b_3 = 0\),
giving the embedded point of multiplicity \(q^3(q+1)\) supported on \([\ell_+]\).
\end{proof}

%% file: threefolds.tex
\chapter{\texorpdfstring{\(q\)}{q}-bic Threefolds}\label{chapter-threefolds}

Whereas general geometric features of \(q\)-bic hypersurfaces described in
Chapter \parref{chapter-hypersurfaces} largely take on a quality akin to that
of quadrics, finer geometric properties begin to resemble those of
cubic hypersurfaces. Hints of this appeared in Chapter \parref{chapter-lowdim}:
there are three types of \(q\)-bic points, see \parref{qbic-points-moduli};
smooth \(q\)-bic curves have a collection of \(q^3+1\) special points analogous
to the flex points of smooth plane cubics, see
\parref{curve-residual-intersection} and
\parref{curves-1+1+1-hermitian-points}; and smooth \(q\)-bic surfaces
contain exactly \((q+1)(q^3+1)\) lines and are purely inseparably unirational,
see \parref{surfaces-1+1+1+1.lines} and \parref{surfaces-1+1+1+1.unirational}.

Lines on threefolds lend further support to this analogy: as with cubic
threefolds, \(q\)-bic threefolds have a smooth surface of lines and a certain
intermediate Jacobian of the hypersurface is closely related to the Albanese
variety of the surface of lines. This result and a basic geometric study
of the surface of lines on a smooth \(q\)-bic threefold are the main objects of
this Chapter: see Section \parref{section-threefolds-smooth}, especially
\parref{threefolds-smooth-intermediate-jacobian-result},
\parref{threefolds-smooth-generalities}, \parref{threefolds-smooth-betti-S},
and \parref{threefolds-cohomology-S-result}.

Smooth \(q\)-bic threefolds, however, are quite complicated and invariants of
its surface of lines are difficult to access directly. Instead, computations
are made indirectly via a degeneration method: study the family of Fano surfaces
obtained via a carefully chosen degeneration of a smooth \(q\)-bic
threefold to one of type \(\mathbf{N}_2 \oplus \mathbf{1}^{\oplus 3}\); this is
made possible using the global methods developed in Chapter
\parref{chapter-hypersurfaces}. Thus the bulk of this Chapter is devoted to
setting up this degeneration method, see Sections
\parref{section-threefolds-cone-situation} and
\parref{section-threefolds-smooth-cone-situation}, and studying the geometry of
\(q\)-bic threefolds of type \(\mathbf{N}_2 \oplus \mathbf{1}^{\oplus 3}\), see
Sections \parref{threefolds-N2+1+1+1}, \parref{section-D}, and
\parref{section-cohomology-F}.

Throughout this Chapter, \(\kk\) is an algebraically closed field of
characteristic \(p > 0\), \(X\) is the \(q\)-bic threefold associated with
a \(q\)-bic form \((V,\beta)\) of dimension \(5\), and
\(S \coloneqq \mathbf{F}_1(X)\) denotes the Fano scheme of lines in \(X\).

\section{Generalities on \texorpdfstring{\(q\)}{q}-bic threefolds}
\label{section-threefolds-generalities}
This Section begins with general comments on the geometry of \(q\)-bic
threefolds, especially in relation to their scheme of lines. Further
features in types \(\mathbf{N}_2 \oplus \mathbf{1}^{\oplus 3}\) and
\(\mathbf{1}^{\oplus 5}\) will be given in Sections
\parref{threefolds-N2+1+1+1} and \parref{section-threefolds-smooth},
respectively.

\subsection{Scheme of lines}\label{threefolds-lines}
The Fano scheme \(S\) of lines in \(X\) is a closed subscheme of the
Grassmannian \(\mathbf{G}(2,V)\) of expected dimension \(2\), see
\parref{hypersurfaces-equations-of-fano}. When \(\dim S = 2\),
it is Cohen--Macaulay and its dualizing bundle is computed in
\parref{hypersurfaces-fano-canonical} to be
\[ \omega_S \cong \sO_S(2q - 3) \otimes \det(V^\vee)^{\otimes 2} \]
where \(\sO_S(1)\) is the Pl\"ucker line bundle. The degree of the Pl\"ucker
bundle is obtained by taking \(n = 4\) in the formula given in
\parref{hypersurfaces-fano-lines-degree}, yielding
\[ \deg(\sO_S(1)) = (q+1)^2(q^2+1). \]

The next statement computes the Chern numbers of \(S\) in the case \(X\) is
smooth. It can be used to show that, if \(q > 2\), then \(S\) does not lift to
characteristic \(0\) as a lift would violate the Bogomolov--Miyaoka--Yau
inequality of \cite[Theorem 4]{Miyaoka:Inequality}. See, however,
\parref{threefolds-smooth-fano-lift} below for a stronger result.

\begin{Proposition}\label{threefolds-lines-smooth-chern-numbers}
Let \(X\) be a smooth \(q\)-bic threefold and let \(S\) be its Fano scheme
of lines. Then \(\mathcal{T}_S \cong \mathcal{S}^\vee \otimes \sO_S(1-q)\) and
the Chern numbers of \(S\) are
\begin{align*}
\int_S c_1(\mathcal{T}_S)^2 & = (q+1)^2(q^2+1)(2q-3)^2, \;\text{and}\\
\int_S c_2(\mathcal{T}_S)\phantom{^2} & = (q+1)^2(q^4 - 3q^3 + 4q^2 - 4q + 3).
\end{align*}
\end{Proposition}

\begin{proof}
By \parref{hypersurfaces-smooth-fano},
\(\mathcal{T}_S \cong (\Fr^*(\mathcal{S})^\perp/\mathcal{S}) \otimes \mathcal{S}^\vee\)
where the first tensor factor is the line bundle characterized by the short
exact sequence
\[
0 \to
\Fr^*(\mathcal{S})^\perp/\mathcal{S} \to
\mathcal{Q} \xrightarrow{\beta}
\Fr^*(\mathcal{S})^\vee \to
0.
\]
Taking determinants gives
\(\Fr^*(\mathcal{S})^\perp/\mathcal{S} \cong \sO_S(1-q)\), and this gives
the identification of the tangent bundle of \(S\). The first Chern number is
now easily obtained upon using the computation of \(\deg(\sO_S(1))\) recalled
above in \parref{threefolds-lines}:
\[
\int_S c_1(\mathcal{T}_S)^2 =
\int_S c_1(\sO_S(3-2q))^2 =
(2q-3)^2\deg(\sO_S(1)) =
(q+1)^2(q^2+1)(2q-3)^2.
\]
For the second Chern number, observe
\begin{align*}
c_2(\mathcal{T}_S)
& = c_2(\mathcal{S}^\vee) + c_1(\mathcal{S}^\vee) c_1(\sO_S(1-q)) + c_1(\sO_S(1-q))^2 \\
& = c_2(\mathcal{S}^\vee) + (q-2)(q-1) c_1(\sO_S(1))^2.
\end{align*}
A general section of \(\mathcal{S}^\vee\) cuts out the subscheme consisting
of lines contained in a general hyperplane section of \(X\). But such a
slice is a smooth \(q\)-bic surface by \parref{hypersurface-hyperplane-section}.
Thus the degree of \(c_2(\mathcal{S}^\vee)\) on \(S\) is \((q+1)(q^3+1)\) by
\parref{qbic-surfaces-lines}. Therefore
\begin{align*}
\int_S c_2(\mathcal{T}_S)
& = (q+1)(q^3+1) + (q-2)(q-1)(q+1)^2(q^2+1) \\
& = (q+1)^2(q^4 - 3q^3 + 4q^2 - 4q + 3)
\end{align*}
upon once again using the computation of the degree of \(\sO_S(1)\).
\end{proof}

The Chern numbers can now be used to compute the Euler characteristic of the
structure sheaf \(\sO_S\) whenever \(S\) is smooth; this holds more
generally, whenever \(S\) has expected dimension, by comparing with the
Koszul resolution given in \parref{hypersurfaces-fano-koszul}.

\begin{Proposition}\label{threefolds-lines-chi-OS}
Let \(X\) be a \(q\)-bic threefold and assume that its Fano scheme \(S\) of lines
is of expected dimension \(2\). Then
\[ \chi(S,\sO_S) = \frac{1}{12}(q+1)^2(5q^4 - 15q^3 + 17q^2 - 16q + 12). \]
\end{Proposition}

\begin{proof}
When \(X\) is a smooth \(q\)-bic threefold, \(S\) is a smooth surface by
\parref{hypersurfaces-smooth-fano}, so Noether's formula, see
\cite[Example 15.2.2]{Fulton}, gives the first equality in
\[
\chi(S,\sO_S)
= \frac{1}{12} \int_S c_1^2(\mathcal{T}_S) + c_2(\mathcal{T}_S)
= \sum\nolimits_{i = 0}^4 (-1)^i
\chi(\mathbf{G}(2,V), \wedge^i \Fr^*(\mathcal{S}) \otimes \mathcal{S}).
\]
Substituting the Chern number computations from
\parref{threefolds-lines-smooth-chern-numbers} gives the formula in the
statement. The second equality arises from the Koszul resolution
from \parref{hypersurfaces-fano-koszul}. Since the Koszul resolution persists
whenever \(\dim S = 2\), this gives the general case.
\end{proof}

\section{Cone Situation for threefolds}\label{section-threefolds-cone-situation}
The goal of this Section is to identify and study a geometric situation in
which there is a canonical rational map \(S \dashrightarrow C\) from the
Fano scheme \(S\) to a certain \(q\)-bic curve \(C\). A canonical resolution of
this rational map is constructed in \parref{threefolds-cone-situation-blowup},
and equations for the intervening spaces are constructed
in \parref{threefolds-cone-situation-equations-of-T} and
\parref{threefolds-cone-situation-vanishing-on-tilde-S}. The most important
cases of the general situation are further studied in
\parref{section-threefolds-smooth-cone-situation}.

\subsection{Cone Situation}\label{threefolds-cone-situation}
Let \((X,\infty,\PP W)\) be a triple consisting of a \(q\)-bic threefold
\(X\), a point \(\infty = \PP L\) of \(X\), and a hyperplane
\(\PP W\) such that \(X \cap \PP W\) is a cone with vertex
\(\infty\) over a reduced \(q\)-bic curve \(C \subset \PP(W/L)\).
This may sometimes be considered with some of the following additional
assumptions:
\begin{enumerate}
\item\label{threefolds-cone-situation.plane}
there does not exist a \(\PP^2\) contained in \(X\) which passes through
\(\infty\); or
\item\label{threefolds-cone-situation.vertex}
\(\infty\) is a smooth point of \(X\); or
\item\label{threefolds-cone-situation.curve}
\(C\) is a smooth \(q\)-bic curve.
\end{enumerate}
Conditions \ref{threefolds-cone-situation.vertex} and \ref{threefolds-cone-situation.curve}
together imply \ref{threefolds-cone-situation.plane}.
Condition \ref{threefolds-cone-situation.vertex} is easy to appreciate:

\begin{Lemma}\label{threefolds-cone-situation-tangent-space}
Let \((X, \infty, \PP W)\) be a Cone Situation \parref{threefolds-cone-situation}.
Then
\[ W \subseteq \Fr^{-1}(L^\perp) \cap \Fr^*(L)^\perp. \]
In particular, \(\PP W \subseteq \mathbf{T}_{X,\infty}\) with equality if and
only if \((X,\infty,\PP W)\) satisfies
\parref{threefolds-cone-situation}\ref{threefolds-cone-situation.vertex}.
\end{Lemma}

\begin{proof}
The containment follows from \parref{hypersurfaces-cone-maximal}.
Then \parref{hypersurfaces-tangent-space-as-kernel} implies \(\PP W\)
is contained in the embedded tangent space of \(X\) at \(\infty\). Equality
holds if and only if \(\mathbf{T}_{X,\infty}\) is a hyperplane, and this is
precisely when \(\infty\) is a smooth point of \(X\).
\end{proof}

The Cone Situation identifies the locus \(C_{\infty,\PP W}\) in \(S\)
parameterizing the lines in \(\PP W\) which pass through \(\infty\) with the
curve \(C\). With some of the further conditions on the Cone Situation, this
identification can be made even more precise, as follows:

\begin{Lemma}\label{threefolds-cone-situation-C}
Let \((X,\infty,\PP W)\) be a Cone Situation over the \(q\)-bic curve
\(C\). Then there exists a canonical closed immersion
\(C \hookrightarrow S\) identifying \(C\) with the closed subscheme
\[ C_{\infty,\PP W} \coloneqq \Set{[\ell] \in S | \infty \in \ell \subset \PP W} \]
of lines in \(X\) containing \(\infty\) and contained in \(\PP W\).
Moreover, the Pl\"ucker line bundle \(\sO_S(1)\) restricts to the given polarization
\(\sO_C(1)\). If \((X,\infty,\PP W)\) furthermore satisfies
\begin{enumerate}
\item\label{threefolds-cone-situation-C.plane}
\parref{threefolds-cone-situation}\ref{threefolds-cone-situation.plane}, then,
as closed subsets of \(S\),
\[
C_{\infty,\PP W} = \Set{[\ell] \in S | \infty \in \ell} = \Set{[\ell] \in S | \ell \subset \PP W};
\]
\item\label{threefolds-cone-situation-C.vertex}
\parref{threefolds-cone-situation}\ref{threefolds-cone-situation.vertex}, then
\(C_{\infty,\PP W}\) coincides with the subscheme \(C_\infty\) of lines in \(X\) through \(\infty\);
\item\label{threefolds-cone-situation-C.normal}
\parref{threefolds-cone-situation}\ref{threefolds-cone-situation.vertex} and
\parref{threefolds-cone-situation}\ref{threefolds-cone-situation.curve}, then
\(\mathcal{N}_{C/S} \cong \sO_C(-q+1)\).
\end{enumerate}
\end{Lemma}

\begin{proof}
The locus of lines in \(\mathbf{G}(2,V)\) contained in \(\PP W\) and passing
through \(\infty\) is given by a linearly embedded \(\PP(W/L)\); in terms of
the Pl\"ucker embedding, this is given by
\(- \wedge L \colon \PP(W/L) \hookrightarrow \PP(\wedge^2 V)\). The restriction
of the tautological subbundle of \(\mathbf{G}(2,V)\) to \(\PP(W/L)\) fits into
a short exact sequence
\[
0 \to
L_{\PP(W/L)} \to
\mathcal{S}\rvert_{\PP(W/L)} \to
\sO_{\PP(W/L)}(-1) \to 0.
\]
Moreover, \(\mathcal{S}\rvert_{\PP(W/L)}\) is a subbundle of
\(W_{\PP(W/L)}\). In particular, the restriction of the \(q\)-bic form \(\beta\)
contains \(L\) in its kernel. Thus the restriction to \(\PP(W/L)\) of the
equations of \(S\) in \(\mathbf{G}(2,V)\) factors through the quotient by \(L\)
and yields the equations of \(C\) in \(\PP(W/L)\). This establishes the first
statements.

Statement \ref{threefolds-cone-situation-C.plane} will be established
through its contrapositive. If there were a line \(\ell \subset X \cap \PP W\)
that did not contain \(\infty\), then since \(X \cap \PP W\) is a cone over
\(\infty\), the plane \(\langle \infty, \ell \rangle\) spanned by \(\infty\)
and \(\ell\) is contained in \(X\). Likewise, if there were a line \(\ell\)
which contains \(\infty\) but not contained in \(\PP W\), then \parref{threefolds-cone-situation-tangent-space}
implies that
\[ \Fr^{-1}(L^\perp) = \Fr^*(L)^\perp = V. \]
The characterization \parref{hypersurfaces-cones} of \(q\)-bics which are cones
implies that \(X\) is a cone over a \(q\)-bic surface with vertex \(\infty\).
But, by \parref{hypersurfaces-equations-of-fano}, every \(q\)-bic surface
contains a line, so taking the span of any such line with \(\infty\) yields a
plane in \(X\) through \(\infty\).

For \ref{threefolds-cone-situation-C.vertex}, consider the \(\PP(V/L)\) embedded
linearly in \(\mathbf{G}(2,V)\) parameterizing lines through \(\infty\). Then
\(L_{\PP(V/L)}\) is a subbundle of \(\mathcal{S}\rvert_{\PP(V/L)}\). View
\(S\) as the vanishing locus of the morphism
\(\mathcal{S} \to \Fr^*\mathcal{S}^\vee\) induced by \(\beta\). Upon restriction
to \(\PP(V/L)\), it follows that the morphism
\[ \mathcal{S} \to \Fr^*\mathcal{S}^\vee \to \Fr^*L^\vee \]
must vanish; in other words, \(\mathcal{S}\rvert_{\PP(V/L)}\) is contained in
the subspace \(\Fr^*(L)^\perp \subset V\). But
\parref{threefolds-cone-situation-tangent-space} implies that \(W = \Fr^*(L)^\perp\),
and so the first statement of the Lemma implies \(\PP(V/L) \cap S = C\) as
schemes.

For \ref{threefolds-cone-situation-C.normal}, consider the normal bundle sequence
\[ 0 \to \mathcal{T}_C \to \mathcal{T}_S\rvert_C \to \mathcal{N}_{C/S} \to 0. \]
This is exact since the assumptions together with
\parref{hypersurfaces-smooth-point-fano} imply that \(S\) is smooth along
\(C\). Taking determinants and applying \parref{hypersurfaces-fano-canonical}
then gives the normal bundle.
\end{proof}

\subsection{Examples}\label{threefolds-cone-situation-examples}
The following are some examples which illustrate the import of the various
additional assumptions.
\begin{enumerate}
\item\label{threefolds-cone-situation-examples.smooth}
Let \(X\) be a smooth \(q\)-bic threefold and let \(\infty \in X\) be a
cone point as in \parref{hypersurfaces-cone-points-smooth}. Then
\((X,\infty,\mathbf{T}_{X,\infty})\) is a Cone Situation, and all Cone Situations
for smooth \(X\) are obtained this way. All the conditions
\ref{threefolds-cone-situation.plane}, \ref{threefolds-cone-situation.vertex},
and \ref{threefolds-cone-situation.curve} are satisfied.
\end{enumerate}
In the next three examples, let \(X\) be a \(q\)-bic threefold of type
\(\mathbf{N}_2 \oplus \mathbf{1}^{\oplus 3}\). For concreteness, choose
coordinates so that
\[ X = \mathrm{V}(x_0^q x_1 + x_2^{q+1} + x_3^q x_4 + x_3 x_4^q) \subset \PP^4. \]
Set \(C \coloneqq \mathrm{V}(x_2^{q+1} + x_3^q x_4 + x_3 x_4^q)\) in the \(\PP^2\)
in which \(x_0\) and \(x_1\) are projected out.
\begin{enumerate}
\setcounter{enumi}{1}
\item\label{threefolds-cone-situation-examples.N2-}
Let \(x_- \coloneqq (1:0:0:0:0)\). Then
\((X,x_-,\mathbf{T}_{X,x_-})\) is a Cone Situation over the curve \(C\)
satisfying each of the conditions \ref{threefolds-cone-situation.plane},
\ref{threefolds-cone-situation.vertex}, and \ref{threefolds-cone-situation.curve}.
\item\label{threefolds-cone-situation-examples.N2+}
Let \(x_+ \coloneqq (0:1:0:0:0)\). Then \((X,x_+,\mathrm{V}(x_0))\) is a
Cone Situation over the curve \(C\) satisfying
\ref{threefolds-cone-situation.plane} and \ref{threefolds-cone-situation.curve},
but not \ref{threefolds-cone-situation.vertex}.
\item\label{threefolds-cone-situation-examples.N2infty}
Let \(\infty \coloneqq (0:0:0:0:1)\). Then \((X,\infty,\mathbf{T}_{X,\infty})\)
is a Cone Situation over the curve \(\mathrm{V}(x_0^q x_1 + x_2^{q+1})\) satisfying
\ref{threefolds-cone-situation.plane} and \ref{threefolds-cone-situation.vertex},
but not \ref{threefolds-cone-situation.curve}.
\end{enumerate}
The remaining examples illustrate the generality of the Cone Situation.
\begin{enumerate}
\setcounter{enumi}{4}
\item\label{threefolds-cone-situation-examples.N4}
Let \(X \coloneqq \mathrm{V}(x_0^q x_1 + x_1^q x_2 + x_3^q x_4)\) and
\(\infty \coloneqq (0:0:0:1:0)\). Then \((X,\infty,\mathbf{T}_{X,\infty})\)
is a Cone Situation over \(\mathrm{V}(x_0^q x_1 + x_1^q x_2)\) satisfying
\ref{threefolds-cone-situation.vertex}, but not \ref{threefolds-cone-situation.plane}
nor \ref{threefolds-cone-situation.curve}.
\item\label{threefolds-cone-situation-examples.1+1+1+1+0}
Let \(X\) be of type \(\mathbf{0} \oplus \mathbf{1}^{\oplus 4}\),
let \(\infty\) be the singular point of \(X\), and let \(\PP W\) be a general
hyperplane through \(\infty\). Then \((X,\infty,\PP W)\) is a Cone
Situation satisfying \ref{threefolds-cone-situation.curve}, but not
\ref{threefolds-cone-situation.plane} nor \ref{threefolds-cone-situation.vertex}.
\item\label{threefolds-cone-situation-examples.1+N2+0+0}
Let \(X\) be of type \(\mathbf{0}^{\oplus 2} \oplus \mathbf{N}_2 \oplus \mathbf{1}\),
let \(\infty\) be any point of the vertex of \(X\), and let \(\PP W\) be any hyperplane
intersecting the vertex of \(X\) exactly at \(\infty\). Then \((X,\infty,\PP W)\) is a Cone Situation
satisfying none of \ref{threefolds-cone-situation.plane}, \ref{threefolds-cone-situation.vertex},
nor \ref{threefolds-cone-situation.curve}.
\end{enumerate}

A Cone Situation \((X,\infty,\PP W)\) gives a canonical way to transform lines
contained in \(X\) to points of \(C\). Namely,
a line \(\ell \subset X\) neither contained in \(\PP W\) nor containing
\(\infty\) determines a point of \(C\) in the following two equivalent ways:
\[ \proj_\infty(\ell \cap \PP W) = \proj_\infty(\ell) \cap \PP(W/L) \in C. \]
Here \(\proj_\infty \colon \PP V \to \PP(V/L)\) is the linear projection
away from \(\infty = \PP L\). Using projective geometry of the ambient
Grassmannains, this construction may be refined to a rational map
\(S \dashrightarrow C\) from the Fano scheme \(S\). The
constructions of \parref{subquotient} then provide geometric methods to resolve
this map, essentially by tracking the intermediate points \(\ell \cap \PP W\)
and lines \(\proj_\infty(\ell)\) created in the process.

\subsection{Cone Situation to Subquotient Situation}\label{threefolds-cone-situation-resolve}
Let \((X,\infty,\PP W)\) be a Cone Situation \parref{threefolds-cone-situation}
over a \(q\)-bic curve \(C\). The data produces a Subquotient Situation
\parref{subquotient-diagram} for the ambient Grassmannian:
\[
\begin{tikzcd}
V \rar[two heads] & V/L \\
W \rar[two heads] \uar[hook] & W/L \uar[hook]
\end{tikzcd}
\quad\text{yielding maps}\quad
\begin{tikzcd}
S \rar[symbol={\subset}] &[-2em]
\mathbf{G}(2,V) \rar[dashed] \dar[dashed] & \mathbf{G}(2,V/L) \dar[dashed] \\
& \PP W \rar[dashed] & \mathbf{P}(W/L) &[-2em] \lar[symbol={\subset}] C
\end{tikzcd}
\]
where the horizontal maps are given by linear projection, and the
vertical maps are given by intersection with a hyperplane; compare with
\parref{subquotient.rational-maps}. Let
\(\PP W^\circ \coloneqq \PP W \setminus \{\infty\}\),
\[
\mathbf{G}(2,V)^\circ
  \coloneqq \Set{[\ell] | \infty \notin \ell\;\text{and}\; \ell \not\subset \PP W},
  \;\text{and}\;\;
\mathbf{G}(2,V/L)^\circ
  \coloneqq \Set{[\ell_0] | \ell_0 \not\subset \PP(W/L)},
\]
so that the rational maps restrict to morphisms
\[ \mathbf{G}(2,V)^\circ \to \PP W^\circ \times_{\PP(W/L)} \mathbf{G}(2,V/L)^\circ \to \PP(W/L). \]
Since \(X \cap \PP W\) is a cone over \(\infty\) with base \(C\),
restricting this to \(S^\circ \coloneqq S \cap \mathbf{G}(2,V)^\circ\) yields a
morphism \(S^\circ \to C\). In many cases, this gives a rational map \(S \dashrightarrow C\):

\begin{Lemma}\label{threefolds-cone-situation-rational-map-S}
Assume the Cone Situation \((X,\infty,\PP W)\) satisfies
\parref{threefolds-cone-situation}\ref{threefolds-cone-situation.plane}.
Then there is a rational map \(\varphi \colon S \dashrightarrow C\) given on
points \([\ell] \in S^\circ\) by
\[
\varphi([\ell]) =
\proj_\infty(\ell \cap \PP W) =
\proj_\infty(\ell) \cap \PP(W/L)
\]
where \(\proj_\infty \colon \PP V \dashrightarrow \PP(V/L)\) is linear projection
away from \(\infty\).
\end{Lemma}

\begin{proof}
That \(\varphi\) exists follows from
\parref{threefolds-cone-situation-C}\ref{threefolds-cone-situation-C.plane},
as it implies that the general line in \(X\) neither contains \(\infty\) nor is
contained in \(\PP W\).
\end{proof}

Some hypothesis on \((X,\infty,\PP W)\) is necessary to obtain a rational map
\(S \dashrightarrow C\). Indeed, \(S^\circ \to C\) does not induce a rational
map \(S \dashrightarrow C\) in example
\parref{threefolds-cone-situation-examples}\ref{threefolds-cone-situation-examples.1+1+1+1+0},
since no open subset of the irreducible component of \(S\) parameterizing lines
through \(\infty\) is contained in \(S^\circ\).

\subsection{}\label{threefolds-cone-situation-T-circ}
Towards a resolution of this map, let
\[ \PP^\circ \coloneqq (\PP W^\circ \times_{\PP(W/L)} \mathbf{G}(2,V/L)^\circ)\rvert_C \]
and view its points as triples \((y \mapsto y_0 \in \ell_0)\) consisting of
points \(y \in \PP W^\circ\), \(y_0 \in C\), and
\([\ell_0] \in \mathbf{G}(2,V/L)^\circ\) such that \(\proj_\infty(y) = y_0\)
and \(y_0 \in \ell_0\). Let
\[ T^\circ \coloneqq \image(S^\circ \to \PP^\circ). \]
Given a line \(\ell_0 \subset \PP(V/L)\), write
\(P_{\ell_0} \coloneqq \langle \ell_0, \infty \rangle \coloneqq \proj_\infty^{-1}(\ell_0)\)
for the plane in \(\PP V\) spanned by \(\ell_0\) and \(\infty\). Then
\(X_{\ell_0} \coloneqq X \cap P_{\ell_0}\) is a \(q\)-bic curve, possibly
defined by the zero form. The points of \(T^\circ\) may be characterized in
terms of the geometry of \(X_{\ell_0}\):

\begin{Lemma}\label{threefolds-points-of-T-geometrically}
\(T^\circ = \Set{(y \mapsto y_0 \in \ell_0) \in \PP^\circ | X_{\ell_0}\;\text{is a cone over}\;y}\).
\end{Lemma}

\begin{proof}
Consider a point \((y \mapsto y_0 \in \ell_0)\). If \(P_{\ell_0}
\subset X\), then any line \(\ell \subset P_{\ell_0} \setminus \{\infty\}\)
through \(y \in X\) witnesses the inclusion \((y \mapsto y_0 \in \ell_0) \in T^\circ\).
In the case \(P_{\ell_0} \not\subset X\), the intersection \(X_{\ell_0}\) is a
\(q\)-bic curve that contains the line \(\langle y,\infty \rangle\) as an
irreducible component. But
\[
\langle y,\infty \rangle =
\proj_\infty^{-1}(y_0) =
\proj_\infty^{-1}(\ell_0 \cap \PP(W/L)) =
P_{\ell_0} \cap \PP W
\]
so \((y \mapsto y_0 \in \ell_0)\) is contained in
\(T^\circ\) if and only if the residual curve \(X_{\ell_0} - \langle y,\infty \rangle\)
contains a line passing through \(y\). By the classification of \(q\)-bic
curves in \parref{qbic-curves-classification}, this happens if and only if
\(X_{\ell_0}\) is a cone and \(y\) is the vertex of the cone.
\end{proof}

This analysis also gives a geometric criterion for when \(S^\circ \to T^\circ\)
has finite fibres:

\begin{Lemma}\label{threefolds-cone-situation-quasi-finite-circ}
Assume the Cone Situation \((X,\infty,\PP W)\) satisfies
\parref{threefolds-cone-situation}\ref{threefolds-cone-situation.plane}.
Then the morphism \(S^\circ \to T^\circ\) is quasi-finite of degree \(q\).
\end{Lemma}

\begin{proof}
The proof of \parref{threefolds-points-of-T-geometrically} shows that the
fibre of \(S^\circ \to T^\circ\) over a point \((y \mapsto y_0 \in \ell_0)\) is
the scheme parameterizing lines in \(X_{\ell_0} = X \cap P_{\ell_0}\) which are
distinct from \(\langle y,\infty \rangle\). If \(X\) does not contain any
planes passing through \(\infty\), then \(X_{\ell_0}\) is a cone with vertex
over \(q\)-bic points and the scheme of lines in question is canonically the
scheme of \(q\)-bic points with a single point removed.
\end{proof}

\subsection{The scheme \texorpdfstring{\(\PP\)}{PP}}\label{threefolds-cone-situation-PP}
The rational maps of \parref{threefolds-cone-situation-resolve} can be resolved
following the methods of \parref{subquotient}. To begin, by
\parref{subquotient-G-bundle}, the restriction of the rational maps from
\(\PP W\) and \(\mathbf{G}(2,V/L)\) to \(C\) are resolved on the product of
projective bundles
\(\PP \coloneqq \PP\mathcal{V}_1 \times_C \PP\mathcal{V}_2\) over \(C\),
where \(\mathcal{V}_1\) and \(\mathcal{V}_2\) may be defined via the diagram
\[
\begin{tikzcd}[column sep=1em, row sep=1em]
0 \ar[rr]
&& \mathcal{V}_1 \ar[rr] \ar[dr,hook] \ar[dd, two heads]
&& V_C \ar[dr, two heads] \ar[rr]
&& \mathcal{V}_2 \ar[rr]
&& 0 \\
&&& W_C \ar[ur,hook] \ar[dr, two heads]
&& (V/L)_C \ar[ur,two heads] \\
0 \ar[rr]
&& \sO_C(-1) \ar[rr]
&& (W/L)_C \ar[rr] \ar[ur,hook]
&& \mathcal{T}_{\PP(W/L)}(-1)\rvert_C \ar[rr] \ar[uu,hook]
&& 0
\end{tikzcd}
\]
as the pullback of the left square and pushout of the right square, respectively.
Thus they fit into split short exact sequences
\[
0 \to L_C \to \mathcal{V}_1 \to \sO_C(-1) \to 0
\quad\text{and}\quad
0 \to \mathcal{T}_{\PP(W/L)}(-1)\rvert_C \to \mathcal{V}_2 \to (V/W)_C \to 0.
\]

Comparing the description of the functor represented by \(\PP\) given in
\parref{subquotient-scheme-G} with the definition of \(\PP^\circ\) in
\parref{threefolds-cone-situation-T-circ} shows that \(\PP^\circ\) is a dense
open affine subbundle of \(\PP\), and its complement is the projective subbundle
\[
\PP \setminus \PP^\circ \coloneqq
\Set{(\infty \mapsto y_0 \in \ell_0) | \ell_0 \subset \PP(W/L)}
= \PP L_C \times_C \PP(\mathcal{T}_{\PP(W/L)}(-1)\rvert_C).
\]
To give equations for this closed complement,
let \(\pi_i \colon \PP\mathcal{V}_i \to C\) and \(\pi \colon \PP \to C\) be
the structure morphisms, and for any \(\sO_\PP\)-module \(\mathcal{F}\) and
\(a,b \in \mathbf{Z}\), write
\[
\mathcal{F}(a,b) \coloneqq
\mathcal{F} \otimes
\pr_1^*\sO_{\pi_1}(a) \otimes
\pr_2^*\sO_{\pi_2}(b).
\]

\begin{Lemma}\label{threefolds-cone-situation-PPcirc-equations}
Consider the morphisms of line bundles on \(\PP\)
\begin{align*}
u_1 & \colon
\sO_\PP(-1,0) \hookrightarrow
\pi^*\mathcal{V}_1 \twoheadrightarrow
\pi^*\sO_C(-1) \\
u_2 & \colon
\sO_\PP(0,-1) \hookrightarrow
\pi^*\mathcal{V}_2 \twoheadrightarrow
(V/W)_\PP
\end{align*}
obtained by composing the Euler sections and the quotient map from
the construction of \(\mathcal{V}_1\) and \(\mathcal{V}_2\). Let
\[
u \coloneqq
\left(\begin{smallmatrix} u_1 \\ u_2 \end{smallmatrix}\right) \colon
\sO_\PP \to
\sO_\PP(1,0) \otimes \pi^*\sO_C(-1) \oplus \sO_\PP(0,1) \otimes (V/W)_\PP \eqqcolon
\mathcal{E}_2.
\]
Then \(\PP \setminus \PP^\circ = \mathrm{V}(u)\).
\end{Lemma}

\begin{proof}
Consider a point \((y \mapsto y_0 \in \ell_0) \in \PP\). Then \(u_1\) vanishes
on this point if and only if \(y\) is the point corresponding to the subbundle
\(L \subset \mathcal{V}_1\), which is precisely the point \(\infty\) from
\parref{threefolds-cone-situation}. Similarly, \(u_2\) vanishes on this
point if and only if \(\ell_0\) is the line in \(\PP(V/L)\) spanned by \(y_0\)
and a direction at \(y_0\) within \(\PP(W/L)\), so that \(\ell_0 \subset \PP(W/L)\).
\end{proof}

\subsection{Closure of \(T^\circ\)}\label{threefolds-cone-situation-plane-bundle}
The next few statements aim to find equations for the Zariski closure of
\(T^\circ\) in \(\PP\). A first approximation, using the geometric description
of \(T^\circ\) given in \parref{threefolds-points-of-T-geometrically},
is given in the next statement. Consider the locally free \(\sO_\PP\)-module
\(\mathcal{P}\) of rank \(3\) fitting into the diagram
\[
\begin{tikzcd}
0 \rar
& \pi^*\mathcal{V}_1 \rar
& V_\PP \rar
& \pi^*\mathcal{V}_2 \rar
& 0 \\
0 \rar
& \pi^*\mathcal{V}_1 \rar \uar[equal]
& \mathcal{P} \rar \uar[hook]
& \sO_\PP(0,-1) \uar[hook,"\mathrm{eu}_{\pi_2}"] \rar
& 0
\end{tikzcd}
\]
where the upper sequence is the pullback of that from \parref{threefolds-cone-situation-PP}.
With the notation of \parref{threefolds-cone-situation-T-circ},
the fibre at a point \((y \mapsto y_0 \in \ell_0) \in \PP\) of
\begin{itemize}
\item the bundle \(\mathcal{P}\) is the subspace of \(V\) underlying the plane
\(P_{\ell_0} \coloneqq \langle \infty,\ell_0 \rangle\),
\item the subbundle \(\pi^*\mathcal{V}_1\) is the line \(\ell_{0,\infty} \coloneqq \langle \infty,y_0 \rangle\), and
\item the tautological subbundle \(\sO_\PP(-1,0) \hookrightarrow \pi^*\mathcal{V}_1\)
is the point \(y\).
\end{itemize}
Let
\(\beta_{\mathcal{P}} \colon \Fr^*(\mathcal{P}) \otimes \mathcal{P} \to \sO_{\PP}\)
be the restriction of the \(q\)-bic form \(\beta\) to \(\mathcal{P}\). Since
\(\mathcal{V}_1\) parameterizes lines contained in \(X\), it is isotropic for
\(\beta\) and the restriction of the adjoints
of \(\beta_{\mathcal{P}}\) to \(\mathcal{V}_1\) may be viewed as maps
\[
\beta_{\mathcal{P}} \rvert
\pi^*\mathcal{V}_1 \xrightarrow{\beta\rvert_{\mathcal{V}_1}}
\pi^*\Fr^*\mathcal{V}_2^\vee \xrightarrow{\mathrm{eu}_{\pi_2}^{(q)}}
\sO_\PP(0,q)
\quad\text{and}\quad
\beta^\vee_{\mathcal{P}} \rvert
\pi^*\Fr^*\mathcal{V}_1 \xrightarrow{\beta^\vee\rvert_{\mathcal{V}_1}}
\pi^*\mathcal{V}_2 \xrightarrow{\mathrm{eu}_{\pi_2}}
\sO_\PP(0,1).
\]

\begin{Lemma}\label{threefolds-cone-situation-T'}
The closed subscheme of \(\PP\) given by
\[ T' \coloneqq \Set{(y \mapsto y_0 \in \ell_0) | X_{\ell_0}\; \text{is a cone over}\; y} \]
satisfies the following:
\begin{enumerate}
\item\label{threefolds-cone-situation-T'.Tcirc}
\(T^\circ = T' \cap \PP^\circ\),
\item\label{threefolds-cone-situation-T'.PPcirc}
\(\PP \setminus \PP^\circ \subset T'\), and
\item\label{threefolds-cone-situation-T'.equations}
\(T'\) is the zero locus of a section \(v \colon \sO_\PP \to \sO_\PP(q,1) \oplus \sO_\PP(1,q)\).
\end{enumerate}
\end{Lemma}

\begin{proof}
Item \ref{threefolds-cone-situation-T'.Tcirc} follows directly from the
geometric characterization of the points of \(T^\circ\) from
\parref{threefolds-points-of-T-geometrically}. To see
\ref{threefolds-cone-situation-T'.PPcirc}, by
\parref{threefolds-cone-situation-PPcirc-equations}, a point of
\(\PP \setminus \PP^\circ\) is of the form \((\infty \mapsto y_0 \in \ell_0)\)
where \(\infty \in P_{\ell_0} \subset \PP W\).
Thus \(X_{\ell_0} = X \cap P_{\ell_0}\) is a cone over \(\infty\), so
\(\PP \setminus \PP^\circ \subset T'\). Finally, as for the equations of \(T'\)
promised in \ref{threefolds-cone-situation-T'.equations}, the characterization
of \(q\)-bic hypersurfaces which are cones, \parref{hypersurfaces-cones},
shows that \(T'\) is the locus in \(\PP\) on which
\(\sO_\PP(-1,0) \hookrightarrow \pi^*\mathcal{V}_1 \hookrightarrow \mathcal{P}\)
is contained in the kernel of \(\beta_{\mathcal{P}}\). The discussion in
\parref{threefolds-cone-situation-plane-bundle} shows that this happens
precisely when the following two morphisms vanish:
\begin{align*}
v_1 \coloneqq &
\beta^\vee_{\mathcal{P}} \circ \mathrm{eu}_{\pi_1}^{(q)} \colon
\sO_\PP(-q,0) \hookrightarrow
\Fr^*\pi^*\mathcal{V}_1 \twoheadrightarrow
\sO_\PP(0,1) \\
v_2 \coloneqq &
\beta_{\mathcal{P}} \circ \mathrm{eu}_{\pi_1} \colon
\sO_\PP(-1,0) \hookrightarrow
\pi^*\mathcal{V}_1 \twoheadrightarrow
\sO_\PP(0,q).
\end{align*}
Then \(v = \left(\begin{smallmatrix} v_1 \\ v_2 \end{smallmatrix}\right)\)
is the sought after section.
\end{proof}

Thus \(T'\) contains the closure of \(T^\circ\) in \(\PP\). However,
\parref{threefolds-cone-situation-T'}\ref{threefolds-cone-situation-T'.PPcirc}
shows that \(T'\) contains \(\PP \setminus \PP^\circ\), and this generally
will be an irreducible component of \(T'\) not contained in the closure of
\(T^\circ\) in \(\PP\). The following factors the equations of  \(\PP \setminus \PP^\circ\)
from those of \(T'\):

\begin{Lemma}\label{threefolds-cone-situation-T'-and-PPcirc}
The section \(v \colon \sO_\PP \to \sO_\PP(q,1) \oplus \sO_\PP(1,q)\)
factors through \(u\) as
\[
v \colon
\sO_\PP \xrightarrow{u}
\mathcal{E}_2 \xrightarrow{v'}
\sO_\PP(q,1) \oplus \sO_\PP(1,q)
\]
for some morphism \(v'\) of rank \(2\) vector bundles.
\end{Lemma}

\begin{proof}
On the one hand,
\parref{threefolds-cone-situation-T'}\ref{threefolds-cone-situation-T'.PPcirc}
shows that \(\PP \setminus \PP^\circ\) is a closed subscheme of \(T'\); on
the other hand,
\parref{threefolds-cone-situation-PPcirc-equations} and
\parref{threefolds-cone-situation-T'}\ref{threefolds-cone-situation-T'.equations}
show that both schemes are vanishing loci of a section of a vector bundle.
Thus the factorization \(v = v' \circ u\) exists by expressing local generators
of the ideal of \(T'\) in terms of those of \(\PP \setminus \PP^\circ\).
\end{proof}

The following gives a candidate for the Zariski closure of \(T^\circ\) in \(\PP\):

\begin{Proposition}\label{threefolds-cone-situation-equations-of-T}
Let \(T \coloneqq T' \cap \mathrm{V}(\det(v'))\). Then \(T\)
\begin{enumerate}
\item\label{threefolds-cone-situation-equations-of-T.closure}
contains the Zariski closure of \(T^\circ\) in \(\PP\), and
\item\label{threefolds-cone-situation-equations-of-T.degeneracy}
is the degeneracy locus of the map
\[
\phi \coloneqq \left(\begin{smallmatrix} v' \\ \wedge u \end{smallmatrix}\right) \colon
\mathcal{E}_2 \to
\mathcal{E}_1 \coloneqq
\sO_\PP(q,1) \oplus \sO_\PP(1,q) \oplus \det(\mathcal{E}_2)
\]
where \(\mathcal{E}_2 \coloneqq \sO_\PP(1,0) \otimes \pi^*\sO_C(-1) \oplus
\sO_\PP(0,1) \otimes (V/W)\).
\item\label{threefolds-cone-situation-equations-of-T.resolution}
If furthermore \(\dim S^\circ = 2\), then \(\dim T = 2\), \(T\) is connected,
Cohen--Macaulay, and there is an exact complex of sheaves on \(\PP\) given by
\[
0 \longrightarrow
\mathcal{E}_2(-q-1,-q-1) \stackrel{\phi}{\longrightarrow}
\mathcal{E}_1(-q-1,-q-1)  \stackrel{\wedge^2\phi^\vee}{\longrightarrow}
\sO_\PP \longrightarrow
\sO_T \longrightarrow
0.
\]
\end{enumerate}
\end{Proposition}

\begin{proof}
For \ref{threefolds-cone-situation-equations-of-T.closure}, by
\parref{threefolds-cone-situation-T'}\ref{threefolds-cone-situation-T'.Tcirc},
it suffices to see that \(\det(v')\) vanishes on \(T^\circ\). Since
\(u\rvert_{T^\circ} \neq 0\) By \parref{threefolds-cone-situation-PPcirc-equations}
whereas \(v\rvert_{T^\circ} = 0\), the factorization
\(v = v' \circ u\) of \parref{threefolds-cone-situation-T'-and-PPcirc} implies
\(v'\rvert_{T^\circ}\) has rank at most \(1\). Thus
\(\det(v')\rvert_{T^\circ} = 0\).

To express \(T\) as the degeneracy locus in
\ref{threefolds-cone-situation-equations-of-T.degeneracy}, observe first that
the vanishing locus of \(\det(v')\) is precisely the top degeneracy locus of
\(v' \colon \mathcal{E}_2 \to \sO_\PP(q,1) \oplus \sO_\PP(1,q)\). Next, observe
that there is a short exact sequence
\[ 0 \to \sO_\PP \xrightarrow{u} \mathcal{E}_2 \xrightarrow{\wedge u} \det(\mathcal{E}_2) \to 0. \]
The factorization \(v = v' \circ u\) from
\parref{threefolds-cone-situation-T'-and-PPcirc} means \(v\) vanishes on
\(\mathrm{V}(\det(v'))\) if and only if \(\image(u) \subseteq \ker(v')\);
by the short exact sequence, this is equivalent to \(\ker(\wedge u)
\subseteq \ker(v')\); finally, this is equivalent to the
degeneracy of the map
\(\phi \coloneqq \left(\begin{smallmatrix} v' \\ \wedge u \end{smallmatrix}\right)\).

Now to establish \ref{threefolds-cone-situation-equations-of-T.resolution}.
To see that \(\dim T = 2\), note that its manifestation as a degeneracy locus
already shows \(\dim T \geq 2\); since \(\PP \setminus \PP^\circ\) is of
dimension \(2\), it suffices to see that \(\dim T^\circ \leq 2\). But
\(T^\circ\) is an image of \(S^\circ\) by \parref{threefolds-cone-situation-T-circ},
so this certainly is the case when \(\dim S^\circ = 2\). Since \(C\) is assumed
to be reduced, see \parref{threefolds-cone-situation}, it is Cohen--Macaulay,
and thus so is \(T\) by \cite[Theorem 1]{HE:Determinantal}; see
also \cite[Section 3, p.9]{Kempf:Schubert} and \cite[Theorem 14.3(c)]{Fulton}.
The complex in question is the Eagon--Northcott complex associated with
\(\phi\); see \cite[Appendix B.2, sequence \((\mathrm{EN}_0)\)]{Lazarsfeld:PositivityI}.
Exactness of the complex follows from
\cite[Theorem B.2.2(ii)]{Lazarsfeld:PositivityI} since \(T\) is of expected
dimension \(2\). Finally, connectedness of \(T\) follows from that of \(S\), see
\parref{hypersurfaces-fano-connected}.
\end{proof}

The following gives some sufficient conditions in terms of the Cone Situation
for the conclusion \(\dim S^\circ = 2\), as in the hypothesis of
\parref{threefolds-cone-situation-equations-of-T}\ref{threefolds-cone-situation-equations-of-T.resolution}.
It is not a complete characterization, however: a cone over a smooth \(q\)-bic
surface has a \(2\)-dimensional Fano scheme but does not carry a Cone Situation
satisfying the hypotheses below.

\begin{Lemma}\label{threefolds-cone-situation-S-expected-dimension}
Assume the Cone Situation \((X,\infty,\PP W)\) satisfies
\parref{threefolds-cone-situation}\ref{threefolds-cone-situation.plane} and
\parref{threefolds-cone-situation}\ref{threefolds-cone-situation.vertex}.
Then \(\dim S^\circ = \dim S = 2\).
\end{Lemma}

\begin{proof}
The assumptions imply that \(X\) does not have a vertex, that is, a point \(x
\in X\) such that \(\langle x, y \rangle \subseteq X\) for all \(y \in X\).
Indeed, if \(x\) were such a point, then \(x \neq \infty\) since \(\infty\) is
a smooth point. Choose a line \(\ell \subset X \cap \PP W\) through \(\infty\)
and which does not pass through \(x\). Since \(x\) is a vertex, \(\langle x,
\ell \rangle\) would be a plane contained in \(X\) passing through \(\infty\).
Thus the vertex of \(X\) is empty. This also implies, via
\parref{hypersurfaces-cone-high-corank}, that \(\dim \Sing(X) = \corank(X) - 1 \leq 1\).

Consider any point \(x = \PP L \in \Sing(X)\). The lines in \(X\)
through \(x\) are contained in
\[ X \cap \PP\Fr^*(L)^\perp \cap \PP\Fr^{-1}(L^\perp) \]
see \parref{threefolds-fano-linear-flag}. Since \(x\) is a singular point,
\(\Fr^*(L)^\perp = V\) by \parref{hypersurfaces-tangent-space-as-kernel}. Since
\(x\) is not a vertex, however, \(\Fr^{-1}(L^\perp)\) is a hyperplane by
\parref{hypersurfaces-cones}. Thus lines in \(X\) through \(x\) are contained
in the surface \(X \cap \PP\Fr^{-1}(L^\perp)\). By
\parref{hypersurfaces-cone-points-classify}\ref{hypersurfaces-cone-points-classify.singular},
this is a cone with vertex \(x\) over a \(q\)-bic curve. Therefore the
subscheme of \(S\) parameterizing lines through \(x\) is of dimension \(1\).
Since the singular locus of \(X\) is of dimension at most \(1\), this implies
that the closed subscheme of \(S\) parameterizing lines in \(X\) through
\(\Sing(X)\) is dimension at most \(2\). Then
\parref{hypersurfaces-fano-expdim-criterion} implies \(S\) has expected
dimension \(2\).
\end{proof}

\subsection{The morphism \(v'\)}\label{threefolds-cone-situation-factor-v}
Describing further properties of \(T\) will depend on understanding the
morphism
\[
v' \coloneqq
\left(\begin{smallmatrix} v_{11}' & v_{12}' \\ v_{21}' & v_{22}' \end{smallmatrix}\right) \colon
\sO_\PP(1,0) \otimes \pi^*\sO_C(-1) \oplus \sO_\PP(0,1) \otimes (V/W)_\PP
\to \sO_\PP(q,1) \oplus \sO_\PP(1,q)
\]
constructed in \parref{threefolds-cone-situation-T'-and-PPcirc}. First, the
components of \(v\) may be written as
\[
v_1 \coloneqq
\beta_{\mathcal{P}}^\vee \circ \mathrm{eu}_{\pi_1}^{(q)} =
\mathrm{eu}_{\pi_2}^\vee \circ \beta^\vee \circ \mathrm{eu}_{\pi_1}^{(q)}
\quad\text{and}\quad
v_2 \coloneqq
\beta_{\mathcal{P}} \circ \mathrm{eu}_{\pi_1} =
\mathrm{eu}_{\pi_2}^{\vee,(q)} \circ \beta \circ \mathrm{eu}_{\pi_1}
\]
thanks to their construction in \parref{threefolds-cone-situation-T'} together
with the discussion of \parref{threefolds-cone-situation-plane-bundle}.
Thus \(v_1\) fits into a commutative diagram with exact columns given by
\[
\begin{tikzcd}[column sep=2em, row sep=1.5em]
& \pi^*\sO_C(-q) \rar
& \pi^*(\mathcal{T}_{\PP(W/L)}(-1)\rvert_C^\vee) \\
\sO_\PP(-q,0) \rar["\mathrm{eu}_{\pi_1}^{(q)}"] \ar[ur,"u_1^q"]
& \pi^*\Fr^*\mathcal{V}_1 \uar[two heads] \rar["\beta^\vee"]
& \pi^*\mathcal{V}_2^\vee \rar["\mathrm{eu}_{\pi_2}^\vee"] \uar[two heads]
& \sO_\PP(0,1) \\
& \Fr^*L_\PP \uar[hook] \rar
& (V/W)_\PP^\vee \uar[hook] \ar[ur,"u_2"']
\end{tikzcd}
\]
in which \(u_1\) and \(u_2\) are the components of \(u\) from
\parref{threefolds-cone-situation-PPcirc-equations}; that \(\beta_L^\vee\)
factors through \((V/L)_\PP^\vee\) is because the fibres of \(\mathcal{T}_{\PP(W/L)}(-1)\)
are quotients of \(W \subseteq \Fr^*(L)^\perp\). Similarly, \(v_2\) fits into
\[
\begin{tikzcd}[column sep=2em, row sep=1.5em]
& \pi^*\sO_C(-1) \rar["\sigma_C"]
& \pi^*\Fr^*(\mathcal{T}_{\PP(W/L)}(-1)\rvert_C^\vee) \\
\sO_\PP(-1,0) \rar["\mathrm{eu}_{\pi_1}"] \ar[ur,"u_1"]
& \pi^*\mathcal{V}_1 \uar[two heads] \rar["\beta"]
& \pi^*\Fr^*\mathcal{V}_2^\vee \rar["\mathrm{eu}_{\pi_2}^{\vee,(q)}"] \uar[two heads]
& \sO_\PP(0,q) \\
& L_\PP \uar[hook] \rar
& \Fr^*(V/W)_\PP^\vee \uar[hook] \ar[ur,"u_2^q"']
\end{tikzcd}
\]
in which the morphism on top induced by \(\beta_C\) is the morphism \(\sigma_C\)
from \parref{hypersurfaces-sigma-section}: indeed, the map \(\beta\) is
restricted from \(V\), and, comparing with the
diagram of \parref{threefolds-cone-situation-PP}, the quotient \(\sO_C(-1)\)
includes via the Euler section into \(V\).

Each column in the diagram is split exact, see \parref{threefolds-cone-situation-PP}.
Fix a splitting and write
\[
\mathrm{eu}_{\pi_1}\rvert_L \colon \sO_\PP(-1,0) \to L_\PP
\quad\text{and}\quad
\mathrm{eu}_{\pi_2}\rvert_{\mathcal{T}_{\PP(W/L)}(-1)} \colon \sO_\PP(0,-1)
\to \pi^*(\mathcal{T}_{\PP(W/L)}(-1))
\]
for the projection of the Euler sections to the subbundles of \(\mathcal{V}_1\)
and \(\mathcal{V}_2\). The components of \(v'\) may now be described in
terms of the morphisms appearing in the diagrams above:

\begin{Lemma}\label{threefolds-cone-situation-v'-components}
The components of \(v' \colon \mathcal{E} \to \sO_\PP(q,1) \oplus \sO_\PP(1,q)\)
are given by:
\begin{align*}
v_{11}' =
(\mathrm{eu}_{\pi_2}\rvert_{\mathcal{T}_{\PP(W/L)}(-1)})^\vee \circ \beta^\vee \circ u_1^{q-1}
& \colon
\sO_\PP(1,0) \otimes \pi^*\sO_C(-1) \to \sO_\PP(q,1), \\
v_{12}' = \beta^\vee \circ (\mathrm{eu}_{\pi_1}\rvert_L)^{(q)}
& \colon
\sO_\PP(0,1) \otimes (V/W)_\PP \to \sO_\PP(q,1), \\
v_{21}'  =
(\mathrm{eu}_{\pi_2}\rvert_{\mathcal{T}_{\PP(W/L)}(-1)})^{\vee,(q)} \circ \sigma_C
& \colon
\sO_\PP(1,0) \otimes \pi^*\sO_C(-1) \to \sO_\PP(1,q), \\
v_{22}' =
u_2^{q-1} \circ \beta \circ (\mathrm{eu}_{\pi_1}\rvert_L)
& \colon
\sO_\PP(0,1) \otimes (V/W)_\PP \to \sO_\PP(q,1).
\end{align*}
\end{Lemma}

\begin{proof}
This follows from the factorization \parref{threefolds-cone-situation-T'-and-PPcirc}
and the diagrams of \parref{threefolds-cone-situation-factor-v}.
\end{proof}

The proof of
\parref{threefolds-cone-situation-equations-of-T}\ref{threefolds-cone-situation-equations-of-T.closure}
shows that \(T^\circ = T \cap \PP^\circ\). Thus a natural closed
complement of \(T^\circ\) in \(T\) is given by
\[
T \setminus T^\circ \coloneqq
T \cap (\PP \setminus \PP^\circ) =
\mathrm{V}(\det(v')) \cap (\PP \setminus \PP^\circ).
\]
The computation of the components of \(v'\) shows that, at least when
\(\infty \in X\) is a smooth point, \(T \setminus T^\circ\) is generically a
purely inseparable multisection of \(T \to C\):

\begin{Lemma}\label{threefolds-cone-situation-boundary}
Assume the Cone Situation \((X,\infty,\PP W)\) satisfies
\parref{threefolds-cone-situation}\ref{threefolds-cone-situation.vertex}.
Then \(T \setminus T^\circ\) is connected, of pure dimension
\(1\), and there is a commutative diagram
\[
\begin{tikzcd}
& T \setminus T^\circ \dar["\pi"] \\
C \ar[ur,"\sigma_C"] \rar["\Fr"] & C
\end{tikzcd}
\]
in which \(\sigma_C\) is a closed immersion and an isomorphism away
from \(C \cap \PP\Fr^*(W/L)^\perp\).
\end{Lemma}

\begin{proof}
The scheme \(\PP \setminus \PP^\circ\) is the vanishing locus of \(u_1\) and \(u_2\)
by \parref{threefolds-cone-situation-PPcirc-equations}. Thus the computation of
\(v'\) from \parref{threefolds-cone-situation-v'-components} shows that
\(\det(v')\rvert_{\PP\setminus\PP^\circ}\) is, up to sign, the section
\[
(\mathrm{eu}_{\pi_2}\rvert_{\mathcal{T}_{\PP(W/L)}(-1)})^{\vee,(q)} \circ \sigma_C
\circ \beta^\vee \circ (\mathrm{eu}_{\pi_1}\rvert_L)^{(q)} \colon
\sO_{\PP \setminus \PP^\circ} \to (\pi^*\sO_C(-1) \otimes V/W)(q,q)\rvert_{\PP \setminus \PP^\circ}.
\]
Since \(\infty\) is a smooth point, \parref{hypersurfaces-nonsmooth-locus}
implies that \(\beta^\vee\rvert_L \colon \Fr^* L_\PP \to (V/W)_\PP^\vee\) is an isomorphism.
Since \(\sigma_C\) is also nonzero by \parref{hypersurfaces-sigma-section},
\(\det(v')\rvert_{\PP \setminus \PP^\circ}\) is a nonzero section
of an ample line bundle so
\(T \setminus T^\circ = \mathrm{V}(\det(v')) \cap (\PP \setminus \PP^\circ)\)
is of pure dimension \(1\) and, by \citeSP{0FD9}, is connected. The commutative
diagram now comes from \parref{hypersurfaces-sigma-section}, together with the
canonical isomorphism
\[ \mathcal{T}_{\PP(W/L)}(-1) \cong \Omega_{\PP(W/L)}^1(2) \]
of rank \(2\) vector bundles arising from the wedge product pairing.
\end{proof}

\begin{Corollary}\label{threefolds-cone-situation-T-is-closure}
Assume the Cone Situation \((X,\infty,\PP W)\) satisfies
\parref{threefolds-cone-situation}\ref{threefolds-cone-situation.vertex} and
that \(\dim S^\circ = 2\). Then \(T\) is the Zariski closure of \(T^\circ\) in
\(\PP\).
\end{Corollary}

\begin{proof}
By \parref{threefolds-cone-situation-equations-of-T}\ref{threefolds-cone-situation-equations-of-T.closure},
if \(T\) were not the Zariski closure in \(\PP\) of \(T^\circ\), then some
irreducible component of \(T \setminus T^\circ\) is an irreducible component
of \(T\). This is impossible: on the one hand, \(T \setminus T^\circ\) is of
pure dimension \(1\) by \parref{threefolds-cone-situation-boundary}; on the
other hand, \(T\) is connected and Cohen--Macaulay by
\parref{threefolds-cone-situation-equations-of-T}\ref{threefolds-cone-situation-equations-of-T.resolution},
and is therefore equidimension \(2\) by \citeSP{00OV}.
\end{proof}

\subsection{}\label{threefolds-cone-situation-blowup}
Assume the Cone Situation \((X,\infty,\PP W)\) satisfies
\parref{threefolds-cone-situation}\ref{threefolds-cone-situation.plane}.
Then, as in \parref{threefolds-cone-situation-rational-map-S}, the morphism
\(S^\circ \to T^\circ\) from \parref{threefolds-cone-situation-T-circ} induces
a rational map \(S \dashrightarrow T\). This may be resolved as follows: Write
\(\mathcal{T}_{\pi_i}\) for the pullback to \(\PP\) of the relative tangent
bundle of \(\pi_i \colon \PP\mathcal{V}_i \to C\), and let \(\mathcal{P}\) be
as in \parref{threefolds-cone-situation-plane-bundle}. Set
\[
\mathcal{V} \coloneqq
\mathcal{H}\big(
\sO_\PP(-1,0) \hookrightarrow
V_\PP \twoheadrightarrow
\mathcal{T}_{\pi_2}(0,-1)\big)
\cong
\coker\big(
\sO_\PP(-1,0) \hookrightarrow
\pi^*\mathcal{V}_1 \hookrightarrow
\mathcal{P}\big).
\]
This is a bundle of rank \(2\) and points of its projective bundle are
described as
\[
\mathbf{P}\mathcal{V} =
\Set{\big((y \in \ell) \mapsto (y_0 \in \ell_0)\big) |
(y \mapsto y_0 \in \ell_0) \in \PP\;\text{and}\;
\ell\;\text{a line in}\; P_{\ell_0}}.
\]
There exists a morphism \(\mathbf{P}\mathcal{V} \to \mathbf{G}(2,V)\) given
by projection onto the line \(\ell\) which, by \parref{subquotient-result}, is
an isomorphism onto its image away from the locus
\[
\Set{[\ell] \in \mathbf{G}(2,V)
| \infty \in \ell \;\text{or}\; \ell \subset \PP W}.
\]
By \parref{threefolds-cone-situation-rational-map-S},
the Fano scheme \(S\) is contained in the image of \(\PP\mathcal{V}\) and is
not completely contained in the non-isomorphism locus. Thus its strict transform
\(\tilde S\) along \(\PP\mathcal{V} \to \mathbf{G}(2,V)\) is well-defined,
and the resulting map \(\tilde S \to T\) resolves \(S \dashrightarrow T\).
In summary, there is a commutative diagram
\[
\begin{tikzcd}
S \ar[dr,dashed]
& \lar \tilde S \rar[hook] \dar
& \PP\mathcal{V} \dar \\
& T \rar[hook] \ar[dr]
& \PP \dar["\pi"] \\
&& C\punct{.}
\end{tikzcd}
\]

The next few paragraphs construct equations for \(\tilde{S}\) in some cases
by realizing it as a bundle of \(q\)-bic points over \(T\).
Let \(\mathcal{V}_T\) be the restriction of \(\mathcal{V}\) to \(T\), so that
\(\PP\mathcal{V} \times_\PP T = \PP\mathcal{V}_T\), and write
\(\rho \colon \PP\mathcal{V}_T \to T\) for the projection. Then
\parref{threefolds-cone-situation-quasi-finite-circ} implies that
\(\tilde{S}\) is a codimension \(1\) closed subscheme of \(\PP\mathcal{V}_T\).
Let \((\mathcal{P}_T,\beta_{\mathcal{P}_T})\) be the restriction to \(T\) of
the \(q\)-bic form \((\mathcal{P},\beta_{\mathcal{P}})\) from \parref{threefolds-cone-situation-plane-bundle}.

\begin{Lemma}\label{threefolds-cone-situation-section-over-T}
The \(q\)-bic form \(\beta_{\mathcal{P}_T}\) induces a \(q\)-bic form
\(\beta_{\mathcal{V}_T} \colon \Fr^*(\mathcal{V}_T) \otimes \mathcal{V}_T \to \sO_T\).
The \(q\)-bic equation induced by \(\beta_{\mathcal{V}_T}\), that is
the map of line bundles on \(\PP\mathcal{V}_T\) given by
\[
\beta_{\mathcal{V}_T}(\mathrm{eu}_\rho^{(q)}, \mathrm{eu}_\rho) \colon
\sO_\rho(-q-1)
\hookrightarrow \Fr^*\rho^*\mathcal{V}_T \otimes \rho^*\mathcal{V}_T
\rightarrow \sO_{\PP\mathcal{V}_T},
\]
vanishes at the point \(((y \in \ell) \mapsto (y_0 \in \ell_0)) \in \PP\mathcal{V}_T\)
if and only if \(\ell\) is an isotropic line for \(\beta\).
In particular, this section vanishes on the strict transform \(\tilde{S}\) of
\(S\).
\end{Lemma}

\begin{proof}
By \parref{threefolds-points-of-T-geometrically}, for every point
\((y \mapsto y_0 \in \ell_0) \in T^\circ\), the plane section \(X_{\ell_0} = X
\cap P_{\ell_0}\) is a cone over \(y\); since being a cone is a closed
condition, this holds for all points of \(T\). Comparing the description of the
tautological bundles given in \parref{threefolds-cone-situation-plane-bundle}
with the characterization of \(q\)-bic hypersurfaces which are cones given in
\parref{hypersurfaces-cones}, it follows that the tautological subbundle
\(\sO_T(-1,0) \hookrightarrow \mathcal{P}_T\) lies in the kernel of the form
\(\beta_{\mathcal{P}_T}\). Since
\(\mathcal{V}_T \cong \mathcal{P}_T/\sO_T(-1,0)\), as noted in
\parref{threefolds-cone-situation-blowup}, it passes to the quotient and
induces the desired \(q\)-bic form \(\beta_{\mathcal{V}_T}\).
The statement about the zero locus of the resulting section comes from unwinding
the construction.
\end{proof}

\subsection{}\label{threefolds-cone-situation-subbundle-over-T}
By \parref{subquotient-result}\ref{subquotient-result.grassmannian}, there is
a canonical short exact sequence
\[
0 \to
\mathcal{T}_{\pi_1}(-1,0) \to
\mathcal{V} \to
\sO_\PP(0,-1) \to
0
\]
in which the subbundle may be identified as
\begin{align*}
\mathcal{T}_{\pi_1}(-1,0)
& = \coker(\mathrm{eu}_{\pi_1} \colon \sO_\PP(-1,0) \to \pi^*\mathcal{V}_1) \\
& \cong \det(\pi^*\mathcal{V}_1)(1,0)
\cong \sO_\PP(1,0) \otimes \pi^*\sO_C(-1) \otimes L.
\end{align*}
Its inverse image under the quotient map \(\mathcal{P} \to \mathcal{V}\)
is the subbundle \(\pi^*\mathcal{V}_1 \subset \mathcal{P}\), so its points
are those of \(\PP\mathcal{V}\) in which the line \(\ell\) is that spanned
by \(y_0\) and \(\infty\):
\[
\PP(\mathcal{T}_{\pi_1}(-1,0)) =
\Set{\big((y \in \ell) \mapsto (y_0 \in \ell_0)\big) \in \PP\mathcal{V} |
\ell = \langle y_0, \infty\rangle}.
\]
Since \(\ell = \langle y_0,\infty \rangle \subset \PP W\) whenever
\((y \mapsto y_0 \in \ell_0) \in T\), this
subbundle is isotropic for \(\beta_{\mathcal{V}_T}\) by
\parref{threefolds-cone-situation-section-over-T}, yielding the following
observation:

\begin{Lemma}\label{threefolds-cone-situation-section-over-T-subbundle}
The \(q\)-bic equation
\(\beta_{\mathcal{V}_T}(\mathrm{eu}_\rho^{(q)}, \mathrm{eu}_\rho)\) from
\parref{threefolds-cone-situation-section-over-T} vanishes on the subbundle
\(\PP(\mathcal{T}_{\pi_1}(-1,0)\rvert_T) \subset \PP\mathcal{V}_T\)
and so it factors through the section
\[
v_3 \coloneqq
u_3^{-1}\beta_{\mathcal{V}_T}(\mathrm{eu}^{(q)}_\rho, \mathrm{eu}_\rho) \colon
\sO_{\PP\mathcal{V}_T} \to \sO_\rho(q) \otimes \rho^*\sO_T(0,1)
\]
where
\(u_3 \colon \sO_\rho(-1) \to \rho^*\mathcal{V}_T \to \rho^*\sO_T(0,-1)\)
is the equation of the subbundle. \qed
\end{Lemma}

This also allows for a fine description of the type of \(\beta_{\mathcal{V}_T}\)
upon restriction to fibres:

\begin{Lemma}\label{threefolds-cone-situation-VT-types}
Assume the Cone Situation \((X,\infty,\PP W)\) satisfies
\parref{threefolds-cone-situation}\ref{threefolds-cone-situation.plane}.
The restriction of \(\beta_{\mathcal{V}_T}\) to the fibre over
\(t = (y \mapsto y_0 \in \ell_0) \in T\) is of type
\begin{enumerate}
\item\label{threefolds-cone-situation-VT-types.smooth}
\(\mathbf{1}^{\oplus 2}\) if
\(P_{\ell_0} \not\subset \PP\Fr^*(L)^\perp\) and
\(P_{\ell_0} \not\subset \PP\Fr^{-1}(L^\perp)\);
\item\label{threefolds-cone-situation-VT-types.N2-reduced}
\(\mathbf{N}_2\) with
\(\PP(\mathcal{T}_{\pi_1}(-1,0)\rvert_t)\) reduced if
\(P_{\ell_0} \not\subset \PP\Fr^*(L)^\perp\) and
\(P_{\ell_0} \subset \PP\Fr^{-1}(L^\perp)\);
\item\label{threefolds-cone-situation-VT-types.N2-multiple}
\(\mathbf{N}_2\) with
\(\PP(\mathcal{T}_{\pi_1}(-1,0)\rvert_t)\) multiple if
\(P_{\ell_0} \subset \PP\Fr^*(L)^\perp\) and
\(P_{\ell_0} \not\subset \PP\Fr^{-1}(L^\perp)\); and
\item\label{threefolds-cone-situation-VT-types.cone}
\(\mathbf{0} \oplus \mathbf{1}\) if
\(P_{\ell_0} \subset \PP\Fr^*(L)^\perp\) and
\(P_{\ell_0} \subset \PP\Fr^{-1}(L^\perp)\).
\end{enumerate}
\end{Lemma}

\begin{proof}
The proof of \parref{threefolds-cone-situation-section-over-T} implies that
the restriction \(\beta_{\mathcal{V}_t}\) of \(\beta_{\mathcal{V}_T}\) to the
fibre over \(t\) is the \(q\)-bic form underlying the \(q\)-bic points obtained
by projecting \(X_{\ell_0} = X \cap P_{\ell_0}\) from its cone point \(y\);
note that the assumption on the Situation ensures that \(X_{\ell_0}\) is
of dimension \(1\). The subbundle
\(\bar{L}_t \coloneqq \mathcal{T}_{\pi_1}(-1,0)\rvert_t\) corresponds to the
image of the line \(\langle y_0,\infty \rangle \subset X_{\ell_0}\);
equivalently, this is the image of the subspace \(L\) in the fibre
\(\mathcal{V}_t\).

The type of \(\beta_{\mathcal{V}_t}\) may now be identified using
classification of \(q\)-bic points \parref{qbic-points-classification}
by examining the orthogonals of \(\bar{L}_t\). In case
\ref{threefolds-cone-situation-VT-types.smooth}, both its orthogonals of
\(\bar{L}_t\) are nontrivial and so \(\beta_{\mathcal{V}_t}\) is of type
\(\mathbf{1}^{\oplus 2}\). In cases
\ref{threefolds-cone-situation-VT-types.N2-reduced} and
\ref{threefolds-cone-situation-VT-types.N2-multiple}, exactly one of the
orthogonals of \(\bar{L}_t\) is nontrivial so \(\beta_{\mathcal{V}_t}\) is of
type \(\mathbf{N}_2\); whether \(\bar{L}_t\) underlies the smooth point
is now determined by the identification of tangent spaces given in
\parref{hypersurfaces-tangent-space-as-kernel}. Finally, in case
\ref{threefolds-cone-situation-VT-types.cone}, \(\bar{L}_t\) lies in the kernel
of \(\beta_{\mathcal{V}_t}\) and so the form is of type
\(\mathbf{0} \oplus \mathbf{1}\).
\end{proof}

\begin{Lemma}\label{threefolds-cone-situation-vanishing-on-tilde-S}
Assume the Cone Situation \((X,\infty,\PP W)\) satisfies
\parref{threefolds-cone-situation}\ref{threefolds-cone-situation.plane}. Then
\[
\tilde{S} \subseteq
\mathrm{V}(v_3) \subseteq
\PP\mathcal{V}_T
\]
and \(\tilde{S} \to T\) is finite flat of degree \(q\) onto its image.
\end{Lemma}

\begin{proof}
The discussion of \parref{threefolds-cone-situation-blowup} and
\parref{threefolds-cone-situation-subbundle-over-T} together with
\parref{threefolds-cone-situation-C}\ref{threefolds-cone-situation-C.plane} imply that the
intersection of \(\tilde S\) with the exceptional locus of
\(\PP\mathcal{V} \to \mathbf{G}(2,V)\) is contained in
\(\PP(\mathcal{T}_{\pi_1}(-1,0)\rvert_T)\). Therefore
\(v_3\) vanishes on \(\tilde{S}\) if and only if it vanishes on
\(S^\circ \cong \tilde S \setminus \PP(\mathcal{T}_{\pi_1}(-1,0)\rvert_T)\).
Since \(u_3\) is invertible on the latter open subscheme, \(v_3\)
vanishes on \(S^\circ\) by \parref{threefolds-cone-situation-section-over-T}.

Since \(S^\circ \to T\) is quasi-finite of degree \(q\) by
\parref{threefolds-cone-situation-quasi-finite-circ}, the final statement
follows upon showing that \(\mathrm{V}(v_3) \to T\) is finite flat of degree
\(q\). Since \(v_3\) is degree \(q\) on each fibre of \(\PP\mathcal{V}_T \to T\)
by \parref{threefolds-cone-situation-section-over-T-subbundle},
it suffices to see that \(v_3\) does not vanish on an entire fibre.
But if \(v_3\) did vanish on the fibre over
\((y \mapsto y_0 \in \ell_0) \in T\),
\parref{threefolds-cone-situation-section-over-T} would imply that all lines
\(\ell \subset P_{\ell_0}\) passing through \(y\) are isotropic for \(\beta\),
and hence \(P_{\ell_0}\) would be contained in \(X\). This is impossible with
condition \parref{threefolds-cone-situation}\ref{threefolds-cone-situation.plane}.
\end{proof}

\begin{Corollary}\label{threefolds-cone-situation-S-T}
Assume the Cone Situation \((X,\infty,\PP W)\) satisfies
\parref{threefolds-cone-situation}\ref{threefolds-cone-situation.plane} and
\parref{threefolds-cone-situation}\ref{threefolds-cone-situation.vertex}.
Then
\begin{enumerate}
\item\label{threefolds-cone-situation-S-T.finite}
\(\tilde{S} \to T\) is surjective and finite flat of degree \(q\),
\item\label{threefolds-cone-situation-S-T.equation}
\(\tilde{S} = \mathrm{V}(v_3)\), and
\item\label{threefolds-cone-situation-S-T.sequence}
there is a short exact sequence of bundles on \(T\) given by
\[
0 \to
\sO_T \to
\rho_*\sO_{\tilde{S}} \to
\Div^{q-2}(\mathcal{V}_T)(1,-2) \otimes \pi^*\sO_C(-1) \otimes L \to
0
\]
and so \(\rho_*\sO_{\tilde{S}}\) has an increasing filtration whose
graded pieces are
\[
\gr_i(\rho_*\sO_{\tilde{S}}) =
\begin{dcases*}
\sO_T & if \(i = 0\), \\
(\pi^*\sO_C(-q+i) \otimes L^{\otimes q-i})(q-i,-i-1) & if \(1 \leq i \leq q - 1\).
\end{dcases*}
\]
\end{enumerate}
\end{Corollary}

\begin{proof}
The assumptions on the Cone Situation imply, by
\parref{threefolds-cone-situation-S-expected-dimension}, that \(\dim S = 2\).
So \parref{threefolds-cone-situation-T-is-closure} applies to show that \(T\)
is the Zariski closure of \(T^\circ\) in \(\PP\). Thus \(\tilde{S} \to T\) is
proper and dominant, whence surjective. That it is finite flat of degree \(q\)
then follows from \parref{threefolds-cone-situation-vanishing-on-tilde-S},
establishing \ref{threefolds-cone-situation-S-T.finite}. This implies
\ref{threefolds-cone-situation-S-T.equation} and, with
\parref{threefolds-cone-situation-subbundle-over-T}, gives a short exact
sequence of sheaves on \(\PP\mathcal{V}_T\):
\[
0 \to
\sO_\rho(-q) \otimes \rho^*\sO_T(0,-1) \xrightarrow{v_3}
\sO_{\PP\mathcal{V}_T} \to
\sO_{\tilde{S}} \to
0.
\]
Pushing this along \(\rho\) gives a short exact sequence of \(\sO_T\)-modules
\[
0 \to
\sO_T \to
\rho_*\sO_{\tilde{S}} \to
\mathbf{R}^1\rho_*\sO_\rho(-q) \otimes \sO_T(0,-1) \to
0.
\]
The Euler sequence gives \(\omega_\rho \cong \rho^*\det(\mathcal{V}_T^\vee) \otimes \sO_\rho(-2)\)
and so by Grothendieck duality
\begin{align*}
\mathbf{R}^1\rho_*\sO_\rho(-q)
& \cong \mathbf{R}\rho_*
    \mathbf{R}\mathcal{H}\!\mathit{om}_{\sO_{\PP\mathcal{V}_T}}(
    \sO_\rho(q) \otimes \omega_\rho, \omega_\rho) \\
& \cong \mathbf{R}\mathcal{H}\!\mathit{om}_{\sO_T}(
    \mathbf{R}\rho_*\sO_\rho(q-2) \otimes \det(\mathcal{V}_T^\vee), \sO_T) \\
& \cong \Div^{q-2}(\mathcal{V}_T) \otimes \det(\mathcal{V}_T).
\end{align*}
The discussion of \parref{threefolds-cone-situation-subbundle-over-T} shows
\(\det(\mathcal{V}_T) \cong (\pi^*\sO_C(-1) \otimes L)(1,-1)\), and this gives
the exact sequence of \ref{threefolds-cone-situation-S-T.sequence}. The filtration
comes from applying divided powers to the short exact sequence for \(\mathcal{V}_T\)
in \parref{threefolds-cone-situation-subbundle-over-T}.
\end{proof}

\subsection{Equivariance}\label{threefolds-cone-situation-equivariance}
The Subquotient Situation in \parref{threefolds-cone-situation-resolve} admits
an action by the subgroup
\[ \AutSch(L \subset W \subset V) \subset \mathbf{GL}(V) \]
of linear automorphisms of \(V\) which preserve the flag \(L \subset W \subset V\).
Let
\[
\AutSch(L \subset W \subset V, \beta) \coloneqq
\AutSch(L \subset W \subset V) \cap \AutSch(V,\beta)
\]
be the subgroup of automorphisms that further preserve the \(q\)-bic form
\(\beta\), see \parref{forms-aut-schemes}. Its linear action on the Subquotient
Situation induces an action on each of the schemes \(X\), \(S\), \(\tilde{S}\),
\(T\), and \(C\).

In good situations, such as that considered in the following,
the morphism \(\tilde{S} \to T\) is a quotient map by the action by a finite
subgroup scheme of \(\AutSch(L \subset W \subset V, \beta)\). Namely, let
\[
\AutSch_{\mathrm{uni}}(W \subset V)
\subset \AutSch(L \subset W \subset V)
\subset \mathbf{GL}(V)
\]
be the unipotent subgroup which induces the identity on \(W\) and \(V/W\), and
let
\[
\AutSch_{\mathrm{uni}}(W \subset V, \beta) \coloneqq
\AutSch_{\mathrm{uni}}(W \subset V) \cap \AutSch(V,\beta)
\]
be the subgroup that preserves \(\beta\). This is a subgroup scheme
of \(\AutSch(L \subset W \subset V, \beta)\). Since this subgroup acts trivially
on \(W\), it acts trivially on \(C\).

\begin{Lemma}\label{threefolds-cone-situation-unipotent-auts}
Assume the Cone Situation \((X,\infty, \PP W)\) satisfies
\parref{threefolds-cone-situation}\ref{threefolds-cone-situation.plane} and
\parref{threefolds-cone-situation}\ref{threefolds-cone-situation.vertex}. Then
\[
\AutSch_{\mathrm{uni}}(W \subset V, \beta) \cong
\begin{dcases*}
\boldsymbol{\alpha}_q & if \(L^\perp = \Fr^*(V)\), and \\
\mathbf{F}_q & if \(L^\perp \neq \Fr^*(V)\).
\end{dcases*}
\]
\end{Lemma}

\begin{proof}
Consider a point \(g\) of  \(\AutSch_{\mathrm{uni}}(W \subset V)\).
Since \(g\) induces the identity on \(W\), the endomorphism
\(g - \id_V\) factors through the quotient \(V/W\); since
\(g\) also induces the identity on \(V/W\), the image of \(g - \id_V\) is
contained in \(W\). Thus \(g - \id_V\) induces a linear map
\(\delta_g \colon V/W \to W\) and the map \(g \mapsto \delta_g\) yields an
isomorphism of linear algebraic groups
\[ \delta \colon \AutSch_{\mathrm{uni}}(W \subset V) \to \HomSch(V/W,W) \]
where \(\HomSch(V/W,W)\) is viewed as a vector group.

Observe that \(\delta\) maps
\(\AutSch_{\mathrm{uni}}(W \subset V, \beta)\) into \(\HomSch(V/W,L)\). Indeed,
let \(A\) be any \(\kk\)-algebra. Then for any
\(g \in \AutSch_{\mathrm{uni}}(W \subset V, \beta)(A)\),
\[
\beta(v^{(q)}, w) = \beta((g \cdot v)^{(q)}, w)
\quad\text{and}\quad
\beta(w^{(q)}, v) = \beta(w^{(q)}, g \cdot v)
\]
for all \(w \in W \otimes_\kk A\) and \(v \in V \otimes_\kk A\). Rearranging
the equations shows that
\[ \delta_g(v) \in (\Fr^*(W)^\perp \cap \Fr^{-1}(W^\perp)) \otimes_\kk A. \]
The assumptions on the Situation imply that the intersection
is \(L \otimes_\kk A\), as required.

To construct equations of \(\AutSch_{\mathrm{uni}}(W \subset V, \beta)\) in
\(\HomSch(V/W,L)\), fix a nonzero \(w \in L\), and choose \(v \in V\) such
that its image \(\bar{v} \in V/W\) is nonzero. This induces an isomorphism
\[
\HomSch(V/W,L) \to \mathbf{G}_a\colon
(\bar{v} \mapsto t w) \mapsto t.
\]
Let \(g \in \AutSch_{\mathrm{uni}}(W \subset V, \beta)(A)\) be mapped to
\(t \in \mathbf{G}_a(A)\). Then
\[
0 =
\beta((g \cdot v)^{(q)}, g \cdot v) - \beta(v^{(q)},v) =
\beta(w^{(q)},v) t^q + \beta(v^{(q)}, w) t.
\]
Since \(\infty\) is a smooth point, \(\Fr^*(L)^\perp = W\) by
\parref{hypersurfaces-tangent-space-as-kernel} and so \(\beta(w^{(q)},v)\) is a
nonzero scalar. Now if \(L^\perp = \Fr^*(V)\), then
\(\beta(v^{(q)},w) = 0\) and
\(\AutSch_{\mathrm{uni}}(W \subset V, \beta) \cong \boldsymbol{\alpha}_q\).
If \(L^\perp \neq \Fr^*(V)\), however, then
\(L^\perp = \Fr^*(W)\) by \parref{threefolds-cone-situation-tangent-space}, and
so \(\beta(v^{(q)},w)\) is a nonzero scalar. This implies
\(\AutSch_{\mathrm{uni}}(W \subset V, \beta) \cong \mathbf{F}_q\).
\end{proof}

The proof of \parref{threefolds-cone-situation-unipotent-auts} shows a bit
more:

\begin{Corollary}\label{threefolds-cone-situation-unipotent-preserve}
Assume the Cone Situation \((X,\infty, \PP W)\) satisfies
\parref{threefolds-cone-situation}\ref{threefolds-cone-situation.plane} and
\parref{threefolds-cone-situation}\ref{threefolds-cone-situation.vertex}. Then
the action of \(\AutSch_{\mathrm{uni}}(W \subset V,\beta)\)
\begin{enumerate}
\item\label{threefolds-cone-situation-unipotent-preserve.T}
is trivial on \(T\), and
\item\label{threefolds-cone-situation-unipotent-preserve.V}
preserves the short exact sequence of \(\sO_T\)-modules
\[ 0 \to \mathcal{T}_{\pi_1}(-1,0)\rvert_T \to \mathcal{V}_T \to \sO_T(0,-1) \to 0, \]
and acts via the identity on the sub and quotient line bundles.
\end{enumerate}
\end{Corollary}

\begin{proof}
In fact, \(G \coloneqq \AutSch_{\mathrm{uni}}(W \subset V, \beta)\) acts trivially on the
bundle \(\PP = \PP\mathcal{V}_1 \times_C \PP\mathcal{V}_2\) containing \(T\).
Indeed, \(G\) acts diagonally on the product, so it suffices to see that
\(G\) acts trivially on each of \(\PP\mathcal{V}_1\) and \(\PP\mathcal{V}_2\).
On the one hand, the proof of
\parref{threefolds-cone-situation-unipotent-auts} shows that the nontrivial
component of the action of \(G\) maps \(V/W\) into \(L\); on the other hand,
the short exact sequences of \parref{threefolds-cone-situation-PP} show that
\(L\) is only contained in \(\mathcal{V}_1\) and \(V/W\) is only a quotient of
\(\mathcal{V}_2\). Thus \(G\) acts trivially on both \(\mathcal{V}_1\) and
\(\mathcal{V}_2\). This gives \ref{threefolds-cone-situation-unipotent-preserve.T}.
Item \ref{threefolds-cone-situation-unipotent-preserve.V} now follows from
the same considerations upon examining the construction of \(\mathcal{V}\) from
\parref{threefolds-cone-situation-blowup}.
\end{proof}

\section{Smooth Cone Situation}\label{section-threefolds-smooth-cone-situation}
This Section is devoted to studying a Cone Situation
\((X, \infty, \PP W)\) which satisfies the conditions
\parref{threefolds-cone-situation}\ref{threefolds-cone-situation.vertex} and
\parref{threefolds-cone-situation}\ref{threefolds-cone-situation.curve}; that
is, when \(\infty \in X\) is a smooth point, and the base \(C\) of the cone
\(X \cap \PP W\) is smooth \(q\)-bic curve. These conditions turn out to be quite
restrictive and the possibilities for \((X,\infty,\PP W)\) are classified in
\parref{threefolds-smooth-cone-situation-classification}. The properties of the
constructions made in \parref{section-threefolds-cone-situation} in this
situation are summarized in \parref{threefolds-smooth-cone-situation-summary};
the additional properties that hold are further summarized in
\parref{threefolds-smooth-cone-situation-output}. This Section ends by
constructing in \parref{threefolds-smooth-cone-situation-family}
a distinguished degeneration of the Smooth Cone Situation and applying the
constructions of the previous Section in families to yield
\parref{threefolds-smooth-cone-situation-family-fano}.

\subsection{}
Let \((X,\infty,\PP W)\) be a Cone Situation satisfying
\parref{threefolds-cone-situation}\ref{threefolds-cone-situation.vertex} and
\parref{threefolds-cone-situation}\ref{threefolds-cone-situation.curve}. Note
that these conditions together imply
\parref{threefolds-cone-situation}\ref{threefolds-cone-situation.plane}.
Further, since \(\infty\) is a smooth point,
\parref{threefolds-cone-situation-tangent-space} implies that the hyperplane
\(\PP W\) must be the embedded tangent space \(\mathbf{T}_{X,\infty}\)
and thus will be dropped from the notation. To reiterate and fix notation, a
\emph{Smooth Cone Situation} is a pair \((X,\infty)\) consisting of a \(q\)-bic
threefold \(X \subset \PP V\) together with a smooth point \(\infty\) such that
\(X \cap \mathbf{T}_{X,\infty}\) is a cone over a smooth \(q\)-bic curve \(C\).

The conditions of the Smooth Cone Situation turn
out to be quite restrictive and all possibilities have appeared amongst the
examples of \parref{threefolds-cone-situation-examples}:

\begin{Proposition}\label{threefolds-smooth-cone-situation-classification}
A Smooth Cone Situation \((X,\infty)\) is projectively equivalent to either
\parref{threefolds-cone-situation-examples}\ref{threefolds-cone-situation-examples.smooth}
or \parref{threefolds-cone-situation-examples}\ref{threefolds-cone-situation-examples.N2-}.
In particular, \(X\) is of type \(\mathbf{1}^{\oplus 5}\) or
\(\mathbf{N}_2 \oplus \mathbf{1}^{\oplus 3}\).
\end{Proposition}

\begin{proof}
If \(X\) is smooth, then \parref{hypersurfaces-cone-points-smooth} shows that
\((X,\infty)\) must be as in
\parref{threefolds-cone-situation-examples}\ref{threefolds-cone-situation-examples.smooth}.
Thus it remains to show that when \(X\) is singular, then \((X,\infty)\) is as in
\parref{threefolds-cone-situation-examples}\ref{threefolds-cone-situation-examples.N2-}:
\(X\) is of type \(\mathbf{N}_2 \oplus \mathbf{1}^{\oplus 3}\), and \(\infty\)
is the smooth point from the subform of type \(\mathbf{N}_2\).

Since \(\infty = \PP L\) is a smooth point and the base of the cone
\(X \cap \PP W\) is smooth, with \(W = \Fr^*(L)^\perp\) the linear subspace
underlying \(\mathbf{T}_{X,\infty}\), \(\PP W\) is disjoint from
the singular locus \(\mathrm{Sing}(X)\) of \(X\). By
\parref{hypersurfaces-nonsmooth-locus}, \(\mathrm{Sing}(X) = \PP L'\) with
\(L' \coloneqq \Fr^{-1}(V^\perp)\), so
\[
W \cap L' = \{0\}
\quad\text{so}\quad
\corank(X) = \dim(L') \leq 1.
\]
Since \(X\) is assumed not to be smooth, it must have corank \(1\). Since
\(L \subset W\), this also shows that
\(L' \neq L\); moreover, the natural map \(L' \to \Fr^*(L)^\vee\) is an
isomorphism, as it is a nonzero map between \(1\)-dimensional spaces. Let
\(\beta\) be the \(q\)-bic form underlying \(X\). By
\parref{forms-orthogonal-sequence}, there is an exact sequence
\[
0 \to
\Fr^*(V)^\perp \to
V \xrightarrow{\beta}
\Fr^*(V)^\vee \to
\Fr^*(L')^\vee \to
0.
\]
Then \(\Fr^*(V)^\perp = L\): write \(V = W \oplus L'\) and
use the above facts to see
\[
\Fr^*(V)^\perp =
\ker(W \oplus L' \to \Fr^*(W)^\vee) =
\ker(W \to \Fr^*(W/L)^\vee) =
L.
\]
Thus the restriction \(\beta_U\) of \(\beta\) to
\(U \coloneqq L \oplus L'\) is of type \(\mathbf{N}_2\) and
\begin{align*}
\Fr^*(U)^\perp & = \Fr^*(L)^\perp \cap \Fr^*(L')^\perp = \Fr^*(L)^\perp, \\
\Fr^{-1}(U^\perp) & = \Fr^{-1}(L^\perp) \cap \Fr^{-1}(L'^\perp) = \Fr^{-1}(L'^\perp),
\end{align*}
are distinct hyperplanes in \(V\) and so their intersection \(V_0\) is an
orthogonal complement of \(U\) in \(V\). This shows that \(X\) is of type
\(\mathbf{N}_2 \oplus \mathbf{1}^{\oplus 3}\) and \(\infty = \PP L\) is the
cone point arising from the smooth point of the subform \(\mathbf{N}_2\).
\end{proof}

\subsection{Summary of general properties}\label{threefolds-smooth-cone-situation-summary}
The constructions made in \parref{section-threefolds-cone-situation}
have very good properties in the Smooth Cone Situation. So far:
\begin{itemize}
\item the scheme \(T \subset \PP\) constructed in
\parref{threefolds-cone-situation-equations-of-T} is a Cohen--Macaulay surface
whose structure sheaf admits a short resolution by vector bundles on \(\PP\);
\item \(T\) is the Zariski closure of \(T^\circ\) in \(\PP\) by
\parref{threefolds-cone-situation-T-is-closure}, and the complement
\(T \setminus T^\circ\) maps to \(C\) via a purely inseparable morphism of
degree \(q\) by \parref{threefolds-cone-situation-boundary};
\item by \parref{threefolds-cone-situation-S-T} the dominant rational map
\(S \dashrightarrow T\) is resolved to a finite flat morphism \(\tilde S \to T\)
of degree \(q\) in which \(\tilde{S}\) is a hypersurface in a \(\PP^1\)-bundle
over \(T\).
\end{itemize}

In the Smooth Cone Situation, more can be said about the scheme
\(\tilde{S}\) resolving the rational map \(S \dashrightarrow T\). To begin,
the birational morphism \(\tilde{S} \to S\) may be explicitly identified.
In the following statement, \(C_\infty \subset S\) is the divisor---abstractly
isomorphic to the \(q\)-bic curve \(C\)---parameterizing lines in \(X\) through
\(\infty\), see
\parref{threefolds-cone-situation-C}\ref{threefolds-cone-situation-C.vertex}.

\begin{Lemma}\label{threefolds-smooth-cone-situation-S-tilde}
Let \((X,\infty)\) be a Smooth Cone Situation.
Then \(\tilde{S} \to S\) is the blowup along the Hermitian points of
\(C_\infty \subset S\).
\end{Lemma}

\begin{proof}
Consider in detail the construction of \(\tilde S\) as the strict transform of
\(S\) along the resolution of the rational map
\(\mathbf{G}(2,V) \dashrightarrow \PP(W/L)\) given by the subquotient situation
in \parref{threefolds-cone-situation-resolve}. By
\parref{subquotient-blowup-factor}, this is resolved on a blowup
\(\mathbf{H} \to \mathbf{G}(2,V)\) that factors as
\[
\mathbf{H} \to
\mathbf{H}_1 \times_{\mathbf{G}(2,V)} \mathbf{H}_2 \to
\mathbf{G}(2,V)
\]
in which
\begin{itemize}
\item \(\mathbf{H}_1 \coloneqq \Set{(y \in \ell) \in \PP W \times \mathbf{G}(2,V)}\)
is the scheme constructed in \parref{intersection} resolving the
intersection map \(\mathbf{G}(2,V) \dashrightarrow \PP W\), and
by \parref{intersection-resolve}\ref{intersection-resolve.blowup},
\(\mathbf{H}_1 \to \mathbf{G}(2,V)\) is the blowup along the subscheme of lines
contained in \(\PP W\); and
\item
\(\mathbf{H}_2 \coloneqq \Set{(\ell \mapsto \ell_0) \in \mathbf{G}(2,V) \times \mathbf{G}(2,V/L)}\)
is the scheme constructed in \parref{linear-projection}
resolving linear projection \(\mathbf{G}(2,V) \dashrightarrow \mathbf{G}(2,V/L)\),
and by \parref{linear-projection-resolve}\ref{linear-projection-resolve.blowup},
\(\mathbf{H}_2 \to \mathbf{G}(2,V)\) is the blowup along the subscheme of lines
containing \(\infty\).
\end{itemize}

Let \(S_i\) be the strict transform of \(S\) along
\(\mathbf{H}_i \to \mathbf{G}(2,V)\), and let
\(\tilde{S}' \coloneqq S_1 \times_S S_2\). Then there is a factorization
\(\tilde{S} \to \tilde{S}' \to S\). Now \(S_2 \to S\) is the blowup along the
subscheme of lines in \(X\) containing \(\infty\), which by
\parref{threefolds-cone-situation-C}\ref{threefolds-cone-situation-C.vertex}, is
the effective Cartier divisor \(C_\infty \subset S\) parameterizing the lines in the
cone \(X \cap \PP W\); thus \(S_2 \to S\) is an isomorphism. On the other hand,
\(S_1 \to S\) is the blowup along the subscheme of lines in \(X\) contained in
\(\PP W\), which by \parref{surface-0+1+1+1.lines} is a \(1\)-dimensional subscheme
of \(S\) with divisorial part \(qC_\infty\), but also embedded points along the
Hermitian points of \(C_\infty\); therefore \(S_1 \to S\) is the blowup of \(S\) along
the Hermitian points of \(C_\infty\).

This shows that \(\tilde{S}' \to S\) is the blowup along the Hermitian
points of \(C_\infty\); in particular, the fibre product is smooth.
Consider the morphism \(\tilde{S} \to \tilde{S}'\). In general, by the
description of points of \(\mathbf{H}\) from \parref{subquotient-scheme-H},
\[
\tilde{S} \subset
\Set{\big((y \in \ell) \mapsto (y_0 \in \ell_0)\big) | \ell \subset X, \ell_0 \subset \PP(V/L)}
\]
and the map to \(\tilde{S}'\) is that which omits \(y_0\).
By \parref{threefolds-cone-situation-rational-map-S}, \(\tilde{S}\) lies in the
closed subscheme in which \(y_0 \in C\), whence \(\tilde{S} \to \tilde{S}'\) is
finite. By \parref{subquotient-blowup-intermediate}, this is an isomorphism
away from the exceptional divisor of \(\tilde{S}' \to S\). But \(\tilde{S}'\)
is smooth therein, so Zariski's Main Theorem, as in \cite[Corollary
III.11.4]{Hartshorne:AG}, implies \(\tilde{S} \to \tilde{S}'\) is an
isomorphism.
\end{proof}

The next two statements describe the boundary divisors of
\(T\) and \(\tilde{S}\), respectively.

\begin{Lemma}\label{threefolds-smooth-cone-situation-divisors-T}
Let \((X,\infty)\) be a Smooth Cone Situation. Then \(T\) is smooth along
the closed subscheme \(T \setminus T^\circ\), and its points are given by
\[
\Set{(\infty \mapsto y_0 \in \ell_0) \in \PP |
y_0 = \PP L_0 \in C\;\text{and}\;\ell_0 = \PP\Fr^{-1}(L_0^\perp)},
\]
\end{Lemma}

\begin{proof}
The description of the points of \(T \setminus T^\circ\) follows directly from
\parref{threefolds-cone-situation-boundary} and \parref{hypersurfaces-sigma-section}.
For smoothness of \(T\) therein, \parref{threefolds-cone-situation-S-T}
together with \parref{threefolds-smooth-cone-situation-S-tilde} imply that
\(T \setminus T^\circ\) is the image of the strict transform
\(\tilde{C}_\infty\) of \(C_\infty\) along \(\tilde{S} \to S\). Since
\(\tilde{S} \to T\) is flat by
\parref{threefolds-cone-situation-S-T}\ref{threefolds-cone-situation-S-T.finite}
and smoothness descends along flat morphisms, see \citeSP{05AW}, it suffices
to show that \(\tilde{S}\) is smooth along \(\tilde{C}_\infty\). For this, note
that the hypotheses of the Smooth Cone Situation together with
\parref{hypersurfaces-smooth-point-fano} imply that \(S\) is smooth along
\(C_\infty\). Since \(\tilde{S} \to S\) is a blowup along smooth points by
\parref{threefolds-smooth-cone-situation-S-tilde}, \(\tilde{S}\) is smooth
along \(\tilde{C}_\infty\).
\end{proof}

\begin{Lemma}\label{threefolds-smooth-cone-situation-divisors-S}
Let \((X,\infty)\) be a Smooth Cone Situation.
The strict transform of \(C_\infty \subset S\) in \(\tilde{S}\) is the closed subscheme
\[
\tilde{C}_\infty \coloneqq
\Set{\big((\infty \in \ell) \mapsto (y_0 \in \ell_0)\big) |
y_0 = \PP L_0 \in C, \ell_0 = \PP\Fr^{-1}(L_0^\perp), \ell = \langle y_0,\infty \rangle}.
\]
The exceptional divisor of \(\tilde{S} \to S\) is the closed subscheme
\[
\Set{\big((y \in \ell) \mapsto (y_0 \in \ell_0)\big) |
y_0 = \PP L_0 \in C\;\text{Hermitian},
\ell_0 = \PP\Fr^{-1}(L_0^\perp),
\ell = \langle y_0, \infty \rangle}.
\]
\end{Lemma}

\begin{proof}
As noted in the proof of \parref{threefolds-smooth-cone-situation-divisors-T},
\(\tilde{C}_\infty\) maps to \(T \setminus T^\circ\) along the map \(\tilde{S} \to T\).
Thus the description of the points of \(\tilde{C}_\infty\) follows from that
of \(T \setminus T^\circ\) together with the discussion of
\parref{threefolds-cone-situation-subbundle-over-T}. The description of the
points of the exceptional divisor follows from
\parref{threefolds-smooth-cone-situation-S-tilde}.
\end{proof}

\begin{Corollary}\label{threefolds-smooth-cone-situation-rational-maps}
Let \((X,\infty)\) be a Smooth Cone Situation. Then
\begin{enumerate}
\item\label{threefolds-smooth-cone-situation-rational-maps.T}
the rational map \(S \dashrightarrow T\) is defined away from the
Hermitian points of \(C_\infty\); and
\item\label{threefolds-smooth-cone-situation-rational-maps.C}
the rational map \(S \dashrightarrow C\) extends to a morphism.
\end{enumerate}
\end{Corollary}

\begin{proof}
For \ref{threefolds-smooth-cone-situation-rational-maps.T}, this is because
\(\tilde{S} \to S\) is an isomorphism away from the Hermitian points of \(C_\infty\)
by \parref{threefolds-smooth-cone-situation-S-tilde}. Since the rational map
to \(C\) factors as \(S \dashrightarrow T \to C\), this implies that \(S \dashrightarrow C\)
is also defined away from the Hermitian points of \(C_\infty\). Then
\ref{threefolds-smooth-cone-situation-rational-maps.C} follows from the
description given in \parref{threefolds-smooth-cone-situation-divisors-S} of
the points of the exceptional divisors of \(\tilde{S} \to S\), as it implies
they get contracted along \(\tilde{S} \to C\).
\end{proof}

\begin{Lemma}\label{threefolds-smooth-cone-situation-quotient}
Let \((X,\infty)\) be a Smooth Cone Situation. Then
the morphism \(\tilde{S} \to T\) is a quotient map for
\(\AutSch_{\mathrm{uni}}(W \subset V, \beta)\).
\end{Lemma}

\begin{proof}
The scheme \(S^\circ\) may be identified as the open subscheme of \(\tilde{S}\)
parameterizing lines that neither contain \(\infty\) nor are contained in \(\PP W\).
To conclude, it suffices to show that \(S^\circ \to T^\circ\) is a torsor for
\(G \coloneqq \AutSch_{\mathrm{uni}}(W \subset V, \beta)\).
Indeed, this would imply that the canonical morphism \(\tilde{S}/G \to T\),
coming from \(G\)-invariance of \(T\) established in
\parref{threefolds-cone-situation-unipotent-preserve}\ref{threefolds-cone-situation-unipotent-preserve.T},
is an isomorphism over \(T^\circ\). But
\parref{threefolds-cone-situation-S-T}\ref{threefolds-cone-situation-S-T.finite}
implies that \(\tilde{S}/G \to T\) is finite; since \(T\) is regular along
\(T \setminus T^\circ\) by \parref{threefolds-smooth-cone-situation-divisors-T},
Zariski's Main Theorem as in \citeSP{02LR} implies that \(\tilde{S}/G \to T\)
is also an isomorphism away from \(T^\circ\).

So consider \(S^\circ \to T^\circ\). The characterization given in
\parref{threefolds-cone-situation-VT-types} of the type of
\(\beta_{\mathcal{V}_T}\) on fibres implies that, for every \(t \in T^\circ\),
\[
\mathrm{type}(\beta_{\mathcal{V}_t}) =
\begin{dcases*}
\mathbf{1}^{\oplus 2} &
if \((X,\infty)\) is as in \parref{threefolds-cone-situation-examples}\ref{threefolds-cone-situation-examples.smooth}
so that \(\Fr^{-1}(L^\perp) \neq V\), \\
\mathbf{N}_2 &
if \((X,\infty)\) is as in \parref{threefolds-cone-situation-examples}\ref{threefolds-cone-situation-examples.N2-}
so that \(\Fr^{-1}(L^\perp) = V\),
\end{dcases*}
\]
such that the subspace \(\bar{L}_t \coloneqq \mathcal{T}_{\pi_1}(-1,0)\rvert_t\)
of \(\mathcal{V}_t\) underlies a smooth point. The fibre of \(S^\circ \to T^\circ\)
is then the degree \(q\) scheme obtained by removing the point corresponding to
\(\bar{L}_t\). The action of \(G\) on \((\mathcal{V}_t,\beta_{\mathcal{V}_t})\)
factors through
\(\AutSch_{\mathrm{uni}}(\bar{L}_t \subset \mathcal{V}_t, \beta_{\mathcal{V}_t})\),
and the computations of \parref{qbic-points-automorphisms-parabolic} and
\parref{qbic-points-automorphisms.N2} show that this action is simply transitive
on the degree \(q\) scheme complementary to the point given by \(\bar{L}_t\).
\end{proof}

The following summarizes the additional properties of the Smooth Cone Situation.

\begin{Proposition}\label{threefolds-smooth-cone-situation-output}
Let \((X,\infty)\) be a Smooth Cone Situation. Then there is a canonical
commutative diagram of morphisms
\[
\begin{tikzcd}[column sep=1em, row sep=1em]
& \tilde{S} \ar[dr,"\rho"] \ar[dl,"b"'] \\
S \ar[dr,"\varphi"'] && T \ar[dl,"\pi"] \\
& C
\end{tikzcd}
\]
of schemes over \(\kk\) such that
\begin{enumerate}
\item\label{threefolds-smooth-cone-situation-output.blowup}
\(b \colon \tilde{S} \to S\) is the blowup along the Hermitian points of
\(C_\infty \subset S\),
\item\label{threefolds-smooth-cone-situation-output.quotient}
\(\rho\colon \tilde{S} \to T\) is a quotient by \(\mathbf{F}_q\) if \(X\) is
smooth and \(\boldsymbol{\alpha}_q\) if \(X\) is singular, and
\item\label{threefolds-smooth-cone-situation-output.flat}
\(\varphi \colon S \to C\) and \(\pi \colon T \to C\) are proper, surjective,
and flat of relative dimension \(1\).
\end{enumerate}
\end{Proposition}

\begin{proof}
That the commutative diagram exists follows from the constructions of
\parref{section-threefolds-cone-situation}; see also
\parref{threefolds-cone-situation-blowup}. Note that, \emph{a priori}, the
construction of \parref{threefolds-cone-situation-rational-map-S} yields
only a rational map \(S \dashrightarrow C\), but this extends to a morphism here
by
\parref{threefolds-smooth-cone-situation-rational-maps}\ref{threefolds-smooth-cone-situation-rational-maps.C}.
Item \ref{threefolds-smooth-cone-situation-output.blowup} is
\parref{threefolds-smooth-cone-situation-S-tilde}, and
\ref{threefolds-smooth-cone-situation-output.quotient} is
\parref{threefolds-smooth-cone-situation-quotient} together with
\parref{threefolds-cone-situation-unipotent-auts}.
To see \ref{threefolds-smooth-cone-situation-output.flat}, note first that both
\(S\) and \(T\) are Cohen--Macaulay by \parref{threefolds-cone-situation-S-expected-dimension}
and \parref{threefolds-cone-situation-equations-of-T}\ref{threefolds-cone-situation-equations-of-T.resolution}.
Since \(C\) is regular, flatness follows from Miracle Flatness, see
\citeSP{00R4}.
\end{proof}

The following gives some basic structure to the fibres of \(S\),
\(\tilde{S}\), and \(T\) over \(C\):

\begin{Lemma}\label{threefolds-smooth-cone-situation-pushforward}
Let \((X,\infty)\) be a Smooth Cone Situation. Then
\begin{enumerate}
\item\label{threefolds-smooth-cone-situation-pushforward.fibres}
for every \(x \in C\), there are isomorphisms
\[
\kappa(x)
\cong \mathrm{H}^0(T_x,\sO_{T_x})
\cong \mathrm{H}^0(\tilde{S}_x, \sO_{\tilde{S}_x})
\cong \mathrm{H}^0(S_x,\sO_{S_x}),
\]
\item\label{threefolds-smooth-cone-situation-pushforward.O}
\(\varphi_*\sO_S \cong (\pi \circ \rho)_*\sO_{\tilde{S}} \cong \pi_*\sO_T \cong \sO_C\), and
\item\label{threefolds-smooth-cone-situation-pushforward.R1}
\(\mathbf{R}^1\varphi_*\sO_S\) is locally free and carries a filtration with
graded pieces
\[
\mathrm{gr}_i(\mathbf{R}^1\varphi_*\sO_S) \cong
\begin{dcases*}
\mathbf{R}^1\pi_*\sO_T & if \(i = 0\), and \\
\sO_C(-q+i) \otimes L^{\otimes q-i} \otimes \mathbf{R}^1\pi_*\sO_T(q-i,-i-1) &
if \(1 \leq i \leq q-1\).
\end{dcases*}
\]
\end{enumerate}
\end{Lemma}

\begin{proof}
By
\parref{threefolds-cone-situation-equations-of-T}\ref{threefolds-cone-situation-equations-of-T.degeneracy},
every fibre \(T_x\) is the degeneracy locus in
\(\PP_x \coloneqq \PP^1_{\kappa(x)} \times \PP^2_{\kappa(x)}\) of a morphism
\[
\phi_x \colon
\sO(1,0) \oplus \sO(0,1) \to
\sO(q,1) \oplus \sO(1,q) \oplus \sO(1,1)
\]
where \(\sO(a,b) \coloneqq \sO_{\PP_x}(a,b)\) for integers \(a\) and \(b\).
By \parref{threefolds-smooth-cone-situation-output}\ref{threefolds-smooth-cone-situation-output.flat},
\(T_x\) is of expected codimension \(2\) and so, as in
\parref{threefolds-cone-situation-equations-of-T}\ref{threefolds-cone-situation-equations-of-T.resolution},
\(\sO_{T_x}\) has a resolution by \(\sO_{\PP_x}\)-modules given by
\[
\sO(-q,-q-1) \oplus \sO(-q-1,-q) \to
\sO(-1,-q) \oplus \sO(-q,-1) \oplus \sO(-q,-q) \to
\sO.
\]
The associated spectral sequence
computing global sections of \(\sO_{T_x}\) has only one nonzero contribution
and gives
\[
\mathrm{H}^0(T_x,\sO_{T_x}) \cong \mathrm{H}^0(\PP_x, \sO) \cong \kappa(x).
\]

By \parref{threefolds-cone-situation-S-T}\ref{threefolds-cone-situation-S-T.sequence},
the sheaf \(\rho_{x,*}\sO_{\tilde{S}_x}\) has a filtration with graded pieces
\(\sO_{T_x}\) and \(\sO_{T_x}(q-i,-i-1)\) for \(1 \leq i \leq q-1\). The latter
pieces are resolved by
\begin{multline*}
\sO(-i,-q-i-2) \oplus \sO(-i-1,-q-i-1)\;\; \to\;\;
\sO(q-i-1,-q-i-1) \oplus
\\
 \sO(-i,-i-2) \oplus \sO(-i,-q-i-1)\;\; \to\;\;
\sO(q-i,-i-1).
\end{multline*}
Since the degree in the second factor is negative, cohomology only appears in
either \(\mathrm{H}^2\) or \(\mathrm{H}^3\); moreover, since the degree in the
first factor is also negative for the left-most sheaves, cohomology for them
only appears in \(\mathrm{H}^3\). Thus the spectral sequence computing
\(\mathrm{H}^0(T_x,\sO_{T_x}(q-i,-i-1))\) has no nonzero contributions, so
\[ \mathrm{H}^0(\tilde{S}_x, \sO_{\tilde{S}_x}) \cong \mathrm{H}^0(T_x, \sO_{T_x}) \cong \kappa(x). \]
Since \(b_{x,*}\sO_{\tilde{S}_x} \cong \sO_{S_x}\) by
\parref{threefolds-smooth-cone-situation-output}\ref{threefolds-smooth-cone-situation-output.blowup},
\(\mathrm{H}^0(S_x,\sO_{S_x}) \cong \mathrm{H}^0(\tilde{S}_x, \sO_{\tilde{S}_x}) \cong \kappa(x)\).
This shows \ref{threefolds-smooth-cone-situation-pushforward.fibres} and implies
\ref{threefolds-smooth-cone-situation-pushforward.O}. That
\(\mathbf{R}^1\varphi_*\sO_S\) is locally free now follows from Cohomology
and Base Change, see \cite[Theorem III.12.11]{Hartshorne:AG}; the filtration
arises the filtration of \(\rho_*\sO_{\tilde{S}}\) from
\parref{threefolds-cone-situation-S-T}\ref{threefolds-cone-situation-S-T.sequence}
via the isomorphism
\(\mathbf{R}^1\varphi_*\sO_S \cong \mathbf{R}^1\pi_*(\rho_*\sO_{\tilde{S}})\)
obtained from \parref{threefolds-smooth-cone-situation-output}.
\end{proof}

\subsection{Families of Smooth Cone Situations}\label{section-threefolds-smooth-cone-situation-family}
The remainder of this Section is devoted to constructing a special
\(1\)-parameter family of Smooth Cone Situations with smooth general fibre and
one singular special fibre. In general, a \emph{family of Smooth Cone
Situations} over a base scheme \(S\) is a pair \((\mathcal{X}, \sigma)\)
consisting of a \(q\)-bic threefold bundle \(\pi \colon \mathcal{X} \to S\) and
a section \(\sigma \colon S \to \mathcal{X}\)
such that \((\mathcal{X}_s, \sigma(s))\) is a Smooth Cone Situation
for every closed point \(s \in S\). The following constructs a special
family of Smooth Cone Situations starting from a smooth \(q\)-bic threefold
and two appropriately chosen cone points:

\begin{Lemma}\label{threefolds-smooth-cone-situation-family}
Let \(X\) be a smooth \(q\)-bic threefold and let
\(x_-, x_+ \in X\) be cone points with \(\langle x_-,x_+ \rangle \not\subset X\).
Then there exists a
\(q\)-bic threefold \(\mathcal{X} \subset \PP V \times \mathbf{A}^1\)
over \(\mathbf{A}^1\) such that
\begin{enumerate}
\item\label{threefolds-smooth-cone-situation-family.sections}
the constant sections \(x_\pm \colon \mathbf{A}^1 \to \PP V \times \mathbf{A}^1\)
factor through \(\mathcal{X}\);
\item\label{threefolds-smooth-cone-situation-family.family}
\((\mathcal{X},x_-)\) is a family of Smooth Cone Situations;
\item\label{threefolds-smooth-cone-situation-family.smooth}
the projection \(\pi \colon \mathcal{X} \to \mathbf{A}^1\) is smooth
away from \(0 \in \mathbf{A}^1\) and \(X = \pi^{-1}(1)\); and
\item\label{threefolds-smooth-cone-situation-family.central}
\(X_0 \coloneqq \pi^{-1}(0)\) is of type
\(\mathbf{N}_2 \oplus \mathbf{1}^{\oplus 3}\) with singular point \(x_+\).
\end{enumerate}
Moreover, there exists a choice of coordinates \((x_0:x_1:x_2:x_3:x_4)\)
such that
\[
\mathcal{X} =
\mathrm{V}(x_0^q x_1 + t x_0 x_1^q + x_2^{q+1} + x_3^{q+1} + x_4^{q+1}) \subset
\PP^4 \times \mathbf{A}^1,
\]
\(x_- = (1:0:0:0:0)\), and \(x_+ = (0:1:0:0:0)\).
\end{Lemma}

\begin{proof}
Let \(x_- = \PP L_-\), \(x_+ = \PP L_+\), and set \(U \coloneqq L_- \oplus L_+\).
Let \((V,\beta)\) be a \(q\)-bic form defining \(X\). Since \((X,x_-)\) is
a Smooth Cone Situation by \parref{threefolds-cone-situation-examples}\ref{threefolds-cone-situation-examples.smooth},
the assumption that \(\langle x_-, x_+ \rangle = \PP U \not\subset X\) is
equivalent to the fact that the restricted \(q\)-bic form \((U,\beta_U)\)
is of type \(\mathbf{1}^{\oplus 2}\). Since \(x_\pm\) are cone points,
\parref{hypersurfaces-cone-points-smooth} implies that \(U\) is a
Hermitian subspace of \(V\) and thus has an orthogonal complement \((W,\beta_W)\)
by \parref{forms-orthogonal-complement-hermitian}.

Set \(V[t] \coloneqq V \otimes_\kk \kk[t]\), and similarly for \(U[t]\) and
\(W[t]\). Let
\[
\beta_W[t] \coloneqq
\beta_W \otimes \id_{\kk[t]} \colon
\Fr^*(W[t]) \otimes_{\kk[t]} W[t] \to \kk[t]
\]
denote the \(q\)-bic form on \(W[t]\) given by the constant
extension of \(\beta_W\). Applying the construction of
\parref{qbic-points-basic-algebra-family} with the decomposition
\(U = L_- \oplus L_+\) yields a \(q\)-bic form
\[
\beta_U^{L_\pm} \colon \Fr^*(U[t]) \otimes_{\kk[t]} U[t] \to \kk[t]
\]
so that the restriction to \(t = 1\) is \(\beta_U\), and the restriction to
\(t = 0\) is of type \(\mathbf{N}_2\). Set
\[
(V[t],\beta^{L_\pm}) \coloneqq
(U[t], \beta_U^{L_\pm}) \perp
(W[t], \beta_W[t]).
\]

Let \(\mathcal{X} \subset \PP V \times \mathbf{A}^1\)
be the \(q\)-bic threefolds over \(\mathbf{A}^1 \coloneqq \Spec(\kk[t])\)
defined by \((V[t],\beta^{L_\pm})\). By
\parref{qbic-points-basic-algebra-family}\ref{qbic-points-basic-algebra-family.isotropic},
the constant sections \(x_\pm \colon \mathbf{A}^1 \to \PP V \times
\mathbf{A}^1\) factor through \(\mathcal{X}\), verifying
\ref{threefolds-smooth-cone-situation-family.sections}. By
\parref{qbic-points-basic-algebra-family}\ref{qbic-points-basic-algebra-family.isomorphism}
and \parref{hypersurfaces-nonsmooth-locus}, \(x_- \colon \mathbf{A}^1 \to
\mathcal{X}\) lands in the smooth locus of each fibre. Then, together with
\parref{qbic-points-basic-algebra-family}\ref{qbic-points-basic-algebra-family.degenerate}
and \parref{hypersurfaces-cone-points-criterion}, \(x_-\) is a cone point in
each fibre and hence \((\mathcal{X}, x_-)\) is a family of Smooth Cone
Situations over \(\mathbf{A}^1\), verifying
\ref{threefolds-smooth-cone-situation-family.family}.
These also imply that \(\pi \colon \mathcal{X} \to \mathbf{A}^1\) is smooth
away from \(0 \in \mathbf{A}^1\), that \(X = \pi^{-1}(1)\) from
\parref{qbic-points-basic-algebra-family}\ref{qbic-points-basic-algebra-family.restrict},
and that \(X_0 \coloneqq \pi^{-1}(0)\) is of type \(\mathbf{N}_2 \oplus
\mathbf{1}^{\oplus 3}\), showing properties
\ref{threefolds-smooth-cone-situation-family.smooth} and
\ref{threefolds-smooth-cone-situation-family.central}. Finally, the explicit
equation of \(\mathcal{X}\) can be realized by choosing standard coordinates
for the form \((W,\beta_W)\) of type \(\mathbf{1}^{\oplus 3}\), and choosing
coordinates \((U,\beta_U)\) adapted to the decomposition \(U = L_- \oplus L_+\)
as in \parref{qbic-points-1+1.basis}.
\end{proof}

\subsection{}\label{threefolds-smooth-cone-situation-family-actions}
Consider the special family of \(q\)-bic threefolds
\(\pi \colon \mathcal{X} \to \mathbf{A}^1\) and its underlying \(q\)-bic form
over \(\kk[t]\),
\[
(V[t], \beta^{L_\pm}) =
(U[t], \beta_{U}^{L_\pm}) \perp
(W[t], \beta_W[t])
\]
as constructed in
\parref{threefolds-smooth-cone-situation-family}. Then \(\mathcal{X}\) admits
actions from two group schemes:

First, consider the linear \(\mathbf{G}_m\)-action on
\(\PP V \times \mathbf{A}^1\) with weights
\[
\mathrm{wt}(W) = 0,
\quad
\mathrm{wt}(L_-) = -1,
\quad
\mathrm{wt}(L_+) = q,
\quad
\mathrm{wt}(t) = q^2-1.
\]
By \parref{qbic-points-family-invariant-form}, this \(\mathbf{G}_m\)-action
leaves \(\beta^{L_\pm} = \beta_U^{L_\pm} \perp \beta_W[t]\) invariant, so
\(\mathcal{X}\) is preserved.

Second, \(\mathcal{X}\) admits an action over \(\mathbf{A}^1\) by the
automorphism group scheme \(\AutSch(V[t],\beta^{L_\pm})\) of the \(q\)-bic form
over \(\kk[t]\) which is furthermore equivariant for the action of
\(\mathbf{G}_m\). By \parref{forms-aut-orthogonal-sum}, this contains the group
scheme
\[
\AutSch(U[t], \beta_{U}^{L_\pm})
\times_{\mathbf{A}^1}
\AutSch(W[t], \beta_W[t])
\subseteq \AutSch(V[t],\beta^{L_\pm}),
\]
which contains the finite flat group scheme
\(\mathcal{G} \coloneqq \mathcal{G}_U \times_\kk \mathrm{U}(V_0,\beta_0)\) over
\(\mathbf{A}^1\), where
\[
\mathcal{G}_U \cong
\Set{
\begin{pmatrix} \lambda & \epsilon \\ 0 & \lambda^{-q} \end{pmatrix} |
\lambda \in \boldsymbol{\mu}_{q^2-1},
\epsilon^q + t\lambda^{q-1} \epsilon = 0
}
\quad\text{and}\quad
\mathrm{U}(W,\beta_W) \cong \mathrm{U}_3(q)
\]
are the subgroup of automorphisms of
\((U[t],\beta_{U}^{L_\pm})\) identified in
\parref{qbic-points-family-automorphisms}, and the finite unitary group
associated with \((W,\beta_W)\) as in \parref{forms-aut-unitary},
Together with \parref{qbic-points-family-equivariant-action}, this gives:

\begin{Lemma}\label{threefolds-smooth-cone-situation-family-equivariant-action}
In the setting of \parref{threefolds-smooth-cone-situation-family-actions},
there exists a \(\mathbf{G}_m\)-equivariant diagram
\[
\begin{tikzcd}[row sep=1em]
\mathcal{G} \times_{\mathbf{A}^1} \mathcal{X} \ar[dr] \ar[rr,"\mathrm{act}"] && \mathcal{X} \ar[dl] \\
& \mathbf{A}^1
\end{tikzcd}
\]
The section \(x_- \colon \mathbf{A}^1 \to \mathcal{X}\) is invariant
under both the action of \(\mathcal{G}\) and \(\mathbf{G}_m\). \qed
\end{Lemma}

\subsection{}\label{threefolds-smooth-cone-situation-family-fano-setup}
Let \(\mathcal{S} \to \mathbf{A}^1\) be the relative Fano scheme of lines
of the family of \(q\)-bic threefolds \(\pi \colon \mathcal{X} \to \mathbf{A}^1\)
constructed in \parref{threefolds-smooth-cone-situation-family}. The projective
constructions of \parref{functoriality-grassmannian} and \parref{subquotient}
work in families, and the constructions of
\parref{section-threefolds-cone-situation} may be applied to the family of
Smooth Cone Situations \((\mathcal{X},x_-)\). Namely, let
\[
C_{\mathbf{A}^1} \coloneqq C \times_\kk \mathbf{A}^1,
\quad
\PP_{\mathbf{A}^1} \coloneqq \PP \times_\kk \mathbf{A}^1,
\quad
\PP\mathcal{V}_{\mathbf{A}^1} \coloneqq \PP\mathcal{V} \times_\kk \mathbf{A}^1
\]
where \(\PP\) is as from \parref{threefolds-cone-situation-PP}, and
\(\PP\mathcal{V}\) is as from \parref{threefolds-cone-situation-blowup}.
Let \(\mathcal{T} \subset \PP_{\mathbf{A}^1}\) be the degeneracy
locus as in \parref{threefolds-cone-situation-equations-of-T} formed using
\(\beta^{L_{\pm}}\); then let
\(\tilde{\mathcal{S}} \subset \PP\mathcal{V}_{\mathbf{A}^1}\rvert_{\mathcal{T}}\)
be the hypersurface defined by the section induced by \(\beta^{L_\pm}\) as
in \parref{threefolds-cone-situation-section-over-T-subbundle}. Then by
\parref{threefolds-cone-situation-S-T}, \(\tilde{\mathcal{S}} \to \mathcal{T}\)
is a finite flat cover of degree \(q\) and the direct image of \(\sO_{\tilde{\mathcal{S}}}\)
has an increasing filtration with graded pieces as described in
\parref{threefolds-cone-situation-S-T}\ref{threefolds-cone-situation-S-T.sequence}.
Further properties of these schemes are as follows:

\begin{Proposition}\label{threefolds-smooth-cone-situation-family-fano}
Let \(X\) be a smooth \(q\)-bic threefold. Then every choice of
cone points \(x_-,x_+ \in X\) such that
\(\langle x_-,x_+ \rangle \not\subset X\) induces a commutative diagram
\[
\begin{tikzcd}[column sep=.5em, row sep=.75em]
& \tilde{\mathcal{S}} \ar[dr] \ar[dl] \\
\mathcal{S} \ar[dr] && \mathcal{T} \ar[dl] \\
& C_{\mathbf{A}^1}
\end{tikzcd}
\]
of morphisms of schemes over \(\mathbf{A}^1\) satisfying:
\begin{enumerate}
\item\label{threefolds-smooth-cone-situation-family-fano.surfaces}
each of \(\mathcal{S}\), \(\tilde{\mathcal{S}}\), and \(\mathcal{T}\) are
flat projective surfaces over \(\mathbf{A}^1\), smooth away from \(0\);
\item\label{threefolds-smooth-cone-situation-family-fano.fibrations}
the morphisms \(\mathcal{S} \to C_{\mathbf{A}^1} \leftarrow \mathcal{T}\)
are flat, projective, and of relative dimension \(1\);
\item\label{threefolds-smooth-cone-situation-family-fano.blowup}
\(\tilde{\mathcal{S}} \to \mathcal{S}\) is a blowup along \(q^3 +1\)
sections of \(\mathcal{S} \to C_{\mathbf{A}^1}\);
\item\label{threefolds-smooth-cone-situation-family-fano.actions}
the diagram admits a linear action of \(\mathbf{G}_m\) and a
\(\mathbf{G}_m\)-equivariant action of \(\mathcal{G}\); and
\item\label{threefolds-smooth-cone-situation-family-fano.quotient}
\(\tilde{\mathcal{S}} \to \mathcal{T}\) is the quotient by the group scheme
\(\mathcal{G}_U\).
\end{enumerate}
\end{Proposition}

\begin{proof}
The existence of the diagram together with properties
\ref{threefolds-smooth-cone-situation-family-fano.surfaces}--\ref{threefolds-smooth-cone-situation-family-fano.blowup}
and
\ref{threefolds-smooth-cone-situation-family-fano.quotient}
follow from applying
\parref{threefolds-smooth-cone-situation-output} to the family of \(q\)-bic
threefolds from \parref{threefolds-smooth-cone-situation-family}.
The group scheme actions in \ref{threefolds-smooth-cone-situation-family-fano.actions}
come from the discussion of
\parref{threefolds-smooth-cone-situation-family-actions} and
\parref{threefolds-smooth-cone-situation-family-equivariant-action}.
\end{proof}

These special families relate the coherent cohomology of the smooth Fano surface
\(S \coloneqq \mathbf{F}_1(X)\) with that of the singular Fano surface
\(S_0 \coloneqq \mathbf{F}_1(X_0)\). To formulate the next statement, let
\(\varphi \colon S \to C\) and \(\varphi_0 \colon S_0 \to C\) be the morphisms
induced on the Fano surfaces by taking the fibre of \(\mathcal{S} \to
C_{\mathbf{A}^1}\) from \parref{threefolds-smooth-cone-situation-family-fano}
over the points \(1\) and \(0\) of \(\mathbf{A}^1\), respectively. The
sheaves \(\mathbf{R}^1\varphi_*\sO_S\) and
\(\mathbf{R}^1\varphi_{0,*}\sO_{S_0}\) are locally free \(\sO_C\)-modules that
carry \(q\)-step filtrations by subbundles, see
\parref{threefolds-smooth-cone-situation-pushforward}\ref{threefolds-smooth-cone-situation-pushforward.R1}.
Finally, it follows from \parref{threefolds-nodal-automorphism-group-scheme}
and \parref{threefolds-smooth-cone-situation-family-fano}\ref{threefolds-smooth-cone-situation-family-fano.actions}
that \(S_0\) admits a \(\mathbf{G}_m\)-action over \(C\), and so the
sheaf \(\mathbf{R}^1\varphi_{0,*}\sO_{S_0}\) carries a natural
\(\mathbf{Z}\)-grading which is compatible with its aforementioned filtration.

\begin{Proposition}\label{threefolds-smooth-cone-situation-rees}
Let \(X\) be a smooth \(q\)-bic threefold. Then every choice of
cone points \(x_-,x_+ \in X\) such that
\(\langle x_-,x_+ \rangle \not\subset X\) induces a weight decomposition
\[
\mathbf{R}^1\varphi_*\sO_S
= \bigoplus\nolimits_{\alpha \in \mathbf{Z}/(q^2-1)\mathbf{Z}}
(\mathbf{R}^1\varphi_*\sO_S)_\alpha
\]
by filtered subbundles. Each weight component carries a further filtration
\(\Fil_\bullet\) by filtered subbundles such that there are canonical
isomorphisms
\[
\gr_i^{\Fil}(\mathbf{R}^1\varphi_*\sO_S)_\alpha \cong
(\mathbf{R}^1\varphi_{0,*}\sO_{S_0})_{\alpha + i(q^2-1)}
\]
of filtered \(\sO_C\)-modules for each \(\alpha = 1,2, \ldots,q^2-1\) and
\(i \in \mathbf{Z}\).
\end{Proposition}

\begin{proof}
By
\parref{threefolds-smooth-cone-situation-family-fano}\ref{threefolds-smooth-cone-situation-family-fano.actions},
\(\varphi \colon S \to C\) is equivariant for the action of the group
scheme
\[
\mathcal{G}_1
\cong (\boldsymbol{\mu}_{q^2-1} \cdot \mathbf{F}_q) \times \mathrm{U}_3(q)
\cong \AutSch(L_- \subset U,\beta_U) \times \AutSch(W,\beta_W).
\]
The subgroup \(\boldsymbol{\mu}_{q^2-1}\) acts trivially on \(C\),
so its action on \(\mathbf{R}^1\varphi_*\sO_S\) induces the claimed weight
decomposition. This respects the filtration from
\parref{threefolds-smooth-cone-situation-pushforward}\ref{threefolds-smooth-cone-situation-pushforward.R1}
because the morphisms \(\tilde{S} \to T \to C\) are equivariant for
\(\boldsymbol{\mu}_{q^2-1}\) by
\parref{threefolds-smooth-cone-situation-family-fano}\ref{threefolds-smooth-cone-situation-family-fano.actions}.

Consider the second statement. Applying
\parref{threefolds-smooth-cone-situation-family-fano}\ref{threefolds-smooth-cone-situation-family-fano.blowup},
the commutative diagram of \parref{threefolds-smooth-cone-situation-family-fano}, and
\parref{threefolds-smooth-cone-situation-family-fano}\ref{threefolds-smooth-cone-situation-family-fano.quotient}
successively gives isomorphisms
\begin{align*}
\mathbf{R}(\mathcal{S} \to C_{\mathbf{A}^1})_*\sO_{\mathcal{S}}
& \cong
\mathbf{R}(\mathcal{S} \to C_{\mathbf{A}^1})_*
\mathbf{R}(\tilde{\mathcal{S}} \to \mathcal{S})_*\sO_{\mathcal{\tilde{S}}} \\
& \cong
\mathbf{R}(\mathcal{T} \to C_{\mathbf{A}^1})_*
\mathbf{R}(\tilde{\mathcal{S}} \to \mathcal{T})_*\sO_{\mathcal{\tilde{S}}}
\cong
\mathbf{R}(\mathcal{T} \to C_{\mathbf{A}^1})_*
(\tilde{\mathcal{S}} \to \mathcal{T})_*\sO_{\tilde{\mathcal{S}}}.
\end{align*}
Taking cohomology sheaves then gives an isomorphism
\[
\mathbf{R}^1(\mathcal{S} \to C_{\mathbf{A}^1})_*\sO_{\mathcal{S}} \cong
\mathbf{R}^1(\mathcal{T} \to C_{\mathbf{A}^1})_*
(\tilde{\mathcal{S}} \to \mathcal{T})_*\sO_{\tilde{\mathcal{S}}}
\]
between locally free \(\sO_{C_{\mathbf{A}^1}}\)-modules. As in
\parref{threefolds-smooth-cone-situation-family-fano-setup}, the
\(\sO_{\mathcal{T}}\)-module
\((\tilde{\mathcal{S}} \to \mathcal{T})_*\sO_{\tilde{\mathcal{S}}}\) carries a
\(q\)-step filtration by subbundles, and this isomorphism induces the
filtration on \(\mathbf{R}^1(\mathcal{S} \to C_{\mathbf{A}^1})_*\sO_{\mathcal{S}}\)
globalizing the filtration from
\parref{threefolds-smooth-cone-situation-pushforward}\ref{threefolds-smooth-cone-situation-pushforward.R1}.

It now follows from
\parref{threefolds-smooth-cone-situation-family-fano}\ref{threefolds-smooth-cone-situation-family-fano.actions}
that the \(\sO_{C_{\mathbf{A}^1}}\)-module
\(\mathbf{R}^1(\mathcal{S} \to C_{\mathbf{A}^1})_*\sO_{\mathcal{S}}\), together
with its filtration and its action by \(\boldsymbol{\mu}_{q^2-1}\), are all
equivariant for an action of \(\mathbf{G}_m\), in which \(\mathbf{G}_m\)
acts on \(\mathbf{A}^1\) with weight \(q^2-1\). The Rees construction, see
\cite[Lemma 19]{Simpson}, endows \(\mathbf{R}^1\varphi_*\sO_S\)
with a filtration \(\Fil_\bullet\) so that its graded pieces are the graded
pieces of \(\mathbf{R}^1\varphi_{0,*}\sO_{S_0}\), these being the base change
of \(\mathbf{R}^1(\mathcal{S} \to C_{\mathbf{A}^1})_*\sO_{\mathcal{S}}\) to
\(1\) and \(0\), respectively. Moreover, \(\Fil_\bullet\) is compatible with
the filtrations constructed above and the \(\boldsymbol{\mu}_{q^2-1}\)-action,
and so the second statement now follows.
\end{proof}

\section{Type \texorpdfstring{\(\mathbf{N}_2 \oplus \mathbf{1}^{\oplus 3}\)}{N2+1+1+1}}
\label{threefolds-N2+1+1+1}
This Section pertains to the geometry of \(q\)-bic threefolds \(X\)
of type \(\mathbf{N}_2 \oplus \mathbf{1}^{\oplus 3}\), with a particular emphasis
on the geometry of lines on \(X\). The variety \(S\) of lines on \(X\) is a
surface which is singular along the curve \(C_+\) of lines through the unique
singular point \(x_+ = \PP L_+\) of \(X\). Its normalization \(\nu \colon S^\nu
\to S\) is constructed in \parref{threefolds-nodal-normalization} by resolving
the rational map \(S \dashrightarrow C\) obtained by applying
\parref{threefolds-cone-situation-rational-map-S} to the Cone Situation
\((X,x_+, \PP\Fr^{-1}(L_+^\perp))\) from
\parref{threefolds-cone-situation-examples}\ref{threefolds-cone-situation-examples.N2+};
this constructs \(S^\nu\) as a ruled surface over the smooth \(q\)-bic curve \(C\).
The resulting morphism \(\tilde\varphi_+\colon S^\nu \to C\)
compatible with the morphism \(\varphi_- \colon S \to C\) coming from
\parref{threefolds-smooth-cone-situation-rational-maps}\ref{threefolds-smooth-cone-situation-rational-maps.C}
applied to the Smooth Cone Situation \((X,x_-)\)
from \parref{threefolds-cone-situation-examples}\ref{threefolds-cone-situation-examples.N2-}
in that there is a commutative square
\[
\begin{tikzcd}
S^\nu \rar["\nu"'] \dar["\tilde\varphi_+"'] & S \dar["\varphi_-"] \\
C \rar["\phi_C"] & C
\end{tikzcd}
\]
see \parref{nodal-varphi-rational-diagram}. This diagram is used
to relate the groups \(\mathrm{H}^i(S,\sO_S)\) to the cohomology of
a sheaf \(\mathcal{F}\) on \(C\) in \parref{threefolds-normalize-cohomology};
the structure of \(\mathcal{F}\) is made explicit in \parref{section-D}, and
its cohomology is computed in the case \(q = p\) in
\parref{section-cohomology-F}. See \parref{threefolds-nodal-fano-cohomology}
for a summary.

Throughout this Section, \(X\) is a \(q\)-bic threefold of type
\(\mathbf{N}_2 \oplus \mathbf{1}^{\oplus 3}\) associated with a \(q\)-bic form
\((V,\beta)\). Write
\[
L_- \coloneqq \Fr^*(V)^\perp
\quad\text{and}\quad
L_+ \coloneqq \Fr^{-1}(V^\perp)
\]
for the two kernels of \(\beta\), and let \(x_\pm \coloneqq \PP L_\pm\) be the
corresponding points of \(X\). By \parref{hypersurfaces-nonsmooth-locus},
\(x_+\) is the unique singular point of \(X\).

\subsection{Orthogonal decomposition}\label{nodal-splitting}
As in \parref{forms-aut-1^a+N2^b}, the \(q\)-bic form \((V,\beta)\) admits
a canonical orthogonal decomposition
\[
(V,\beta) =
(U,\beta_U) \perp
(W,\beta_W)
\]
where \(U \coloneqq L_- \oplus L_+\) is the span of the two kernels, and
\[
W \coloneqq
\Fr^*(U)^\perp \cap \Fr^{-1}(U^\perp) =
\Fr^*(L_-)^\perp \cap \Fr^{-1}(L_+^\perp)
\]
is its unique orthogonal complement. In particular,
\(C \coloneqq X \cap \PP W\) is a distinguished smooth \(q\)-bic curve
contained inside \(X\).

\subsection{Automorphisms}\label{nodal-automorphisms}
Consider the automorphism group scheme of the \(q\)-bic
form \((V,\beta)\) of type \(\mathbf{N}_2 \oplus \mathbf{1}^{\oplus 3}\).
The orthogonal decomposition
\[
(V,\beta)
\cong (U,\beta_U) \perp (W,\beta_W)
\cong (\kk^{\oplus 2}, \mathbf{N}_2) \perp (\kk^{\oplus 3}, \mathbf{1}^{\oplus 3})
\]
constructed in \parref{nodal-splitting} furnishes, via
\parref{forms-aut-orthogonal-sum}, a subgroup
\[
\AutSch(V,\beta)
\supseteq \AutSch(U,\beta_U) \times \AutSch(W,\beta_W)
\cong \AutSch(\kk^{\oplus 2}, \mathbf{N}_2) \times \mathrm{U}_3(q)
\]
where the former factor was computed in \parref{qbic-points-automorphisms.N2}.
Furthermore, \parref{hypersurfaces-tangent-vectors} implies that the Lie
algebra of \(\AutSch(V,\beta)\) is of dimension \(5\). The full automorphism
group scheme can be described upon specializing \parref{forms-aut-1^a+N2^b.computation}
with \(a = 1\) and \(b = 3\):

\begin{Proposition}\label{threefolds-nodal-automorphism-group-scheme}
With the notation above, \(\AutSch(V,\beta)\) is the \(1\)-dimensional closed
subgroup scheme of \(\GL(V) \cong \GL(\kk^{\oplus 2} \oplus W)\) consisting of
matrices of the form
\[
\left(
\begin{array}{cc|@{}c}
\lambda & \epsilon   & \;\;\mathbf{y}^\vee \\
0   & \lambda^{-q} & 0 \\
\hline
0 & \mathbf{x} & A
\end{array}
\right)
\]
where \(A \in \mathrm{U}_3(q)\), \(\lambda \in \mathbf{G}_m\),
\(\epsilon \in \boldsymbol{\alpha}_q\),
\(\mathbf{x} \in \boldsymbol{\alpha}_q^3\), and
\(\mathbf{y} \in \boldsymbol{\alpha}_{q^2}^3\), subject to the equation
\[
\pushQED{\qed}
\mathbf{x} = \lambda^{-q} (A^{\vee,(q)} \beta_W)^{-1} \mathbf{y}^{(q)}
\qedhere
\popQED
\]
\end{Proposition}

As in \parref{hypersurfaces-automorphisms-linear}, the linear action of
\(\AutSch(V,\beta)\) on \(\PP V\) restricts to an action on \(X\). In particular,
the torus \(\mathbf{G}_m\) acts on \(X\) and its fixed locus is seen to be the
following:

\begin{Corollary}\label{nodal-Gm-action}
\(\mathbf{G}_m\) acts on \(X\) with fixed locus
\(\set{x_+,x_-} \cup C\). \qed
\end{Corollary}

\subsection{Lines}\label{nodal-fano-scheme}
By \parref{hypersurfaces-fano-corank-1}\ref{hypersurfaces-fano-corank-1.beta},
\parref{hypersurfaces-fano-expdim-sufficient},
\parref{hypersurfaces-fano-singular-locus}, and
\parref{hypersurfaces-fano-connected}, the Fano scheme \(S\) of \(X\) is a
connected, Cohen--Macaulay surface which is singular along the
curve \(C_+\), where
\[ C_\pm \coloneqq \Set{[\ell] \in S | x_\pm \in \ell}. \]
Since the \(x_\pm\) are cone points over the smooth \(q\)-bic curve
\(C\), both \(C_\pm\) are isomorphic to \(C\). Furthermore, it follows from
\parref{nodal-Gm-action} that the action of the torus \(\mathbf{G}_m\) from the
action of \(\AutSch(V,\beta)\) on \(S\) has fixed locus consisting of the
curves \(C_\pm\).

\begin{Lemma}\label{nodal-Gm-action-S}
\(\mathbf{G}_m\) acts on \(S\) with fixed locus \(C_+ \cup C_-\).
\end{Lemma}

\begin{proof}
Consider first a line \(\ell\) through either \(x_+\) or \(x_-\). Then
\(\ell\) intersects \(C\) at a unique point. But all of \(C\), \(x_+\), and
\(x_-\) are fixed for the action of \(\mathbf{G}_m\) on \(X\) by
\parref{nodal-Gm-action}, so \(\ell\) contains at least two fixed points and
thus must itself be fixed. If \(\ell\) misses
both \(x_+\) and \(x_-\) on \(X\), then \(\ell\) is not contained in
\(X \cap \PP\Fr^{-1}(L_+^\perp)\) and they intersect at a unique point away
from \(x_+\). Since \(\mathbf{G}_m\) scales along the cone, \(\ell\) is not
fixed.
\end{proof}

\subsection{Cone points and situations}\label{nodal-cone-points}
The orthogonal decomposition of \((V,\beta)\) given in \parref{nodal-splitting}
together with the compatibility of Hermitian points with orthogonal
decompositions from \parref{forms-hermitian-basics-orthogonals} neatly
determines the cone points, in the sense of
\parref{hypersurfaces-cone-points-definition}, of \(X\). In terms of the types
given in \parref{hypersurfaces-cone-points-classify}, they are:
\begin{enumerate}
\item the singular point \(x_+ \coloneqq \PP L_+\);
\item the special point \(x_- \coloneqq \PP L_-\); and
\item the Hermitian points of \(C\).
\end{enumerate}
The points \(x_\pm\) give rise to two Cone Situations
\[
(X,x_+,\PP\Fr^{-1}(L_+^\perp))
\quad\text{and}\quad
(X,x_-,\PP\Fr^*(L_-)^\perp)
\]
as already seen in
\parref{threefolds-cone-situation-examples}\ref{threefolds-cone-situation-examples.N2-} and
\parref{threefolds-cone-situation-examples}\ref{threefolds-cone-situation-examples.N2+}.
Both Cone Situations satisfy the conditions
\parref{threefolds-cone-situation}\ref{threefolds-cone-situation.plane} and
\parref{threefolds-cone-situation}\ref{threefolds-cone-situation.curve} and so,
by \parref{threefolds-cone-situation-rational-map-S}, give rise to rational
maps from \(S\) to a smooth \(q\)-bic curve. The corresponding cones
\[
X_+ \coloneqq X \cap \PP \Fr^{-1}(L_+^\perp)
\quad\text{and}\quad
X_- \coloneqq X \cap \PP \Fr^*(L_-)^\perp
\]
may be canonically viewed as cones over the curve \(C \subset \PP W\)
determined by \(\beta_W\). Therefore the two Cone Situations give a
pair of rational maps
\[
\varphi_\pm \colon S \dashrightarrow C
\]
which are \emph{a priori} defined away from the curve \(C_\pm \subset S\)
parameterizing lines through \(x_\pm\), see \parref{threefolds-cone-situation-C}.
Since \(\varphi_-\) arises from a Smooth Cone Situation, see
\parref{threefolds-smooth-cone-situation-classification}, it extends to a morphism
\(\varphi_- \colon S \to C\) by
\parref{threefolds-smooth-cone-situation-rational-maps}\ref{threefolds-smooth-cone-situation-rational-maps.C}.

\subsection{Projection from \texorpdfstring{\(U\)}{U}}\label{nodal-projection-from-line}
The maps \(\varphi_\pm\) may be more directly described via the specific
geometry of \(X\). The orthogonal decomposition of \parref{nodal-splitting}
identifies linear projection of \(X\) away from \(\PP U\) as a rational map
\[ \proj_{\PP U} \colon X \dashrightarrow \PP W. \]
As in \parref{linear-projection-resolve}, this is resolved by a diagram
\[
\begin{tikzcd}
\tilde{X} \dar \rar[hook] & \PP\mathcal{U} \dar \\
X \rar[dashed,"\proj_{\PP U}"] & \PP W
\end{tikzcd}
\]
in which \(\mathcal{U} \cong U_{\PP W} \oplus \sO_{\PP W}(-1)\). The
blowup \(\tilde{X}\) of \(X\) along \(X \cap \PP U\) inherits the structure of
a bundle of \(q\)-bic curves over \(\PP W\) given by the \(q\)-bic form
\[
\beta_{\mathcal{U}} \colon
\Fr^*(\mathcal{U}) \otimes \mathcal{U} \subset
\Fr^*(V)_{\PP W} \otimes V_{\PP W} \to \sO_{\PP W}.
\]
The fibres of \(\tilde{X} \to \PP W\) are all singular \(q\)-bic curves. In the
following, view the cones \(X_\pm\) as having base \(C\):

\begin{Lemma}\label{threefolds-nodal-fibres-of-projection}
The fibre \(\tilde{X}_{y_0}\) over a closed point \(y_0\) of \(\PP W\) is a
\(q\)-bic curve of type
\[
\mathrm{type}(\tilde{X}_{y_0}) =
\begin{dcases*}
\mathbf{N}_2 \oplus \mathbf{1} & if \(y_0 \in \PP W \setminus C\), and \\
\mathbf{0} \oplus \mathbf{N}_2 & if \(y_0 \in C\).
\end{dcases*}
\]
In fact, if \(y_0 \in C\), then
\(\tilde{X}_{y_0} = \langle y_0,x_-\rangle \cup q\langle y_0,x_+\rangle\).
\end{Lemma}

\begin{proof}
The fibre \(\tilde{X}_{y_0}\) is the \(q\)-bic curve obtained by intersecting
\(X\) with the plane spanned by \(\PP U\) and \(y_0\). Since each such plane
intersects the singular point \(x_+\), \parref{hypersurfaces-nonsmooth-locus}
implies that each fibre is a singular curve. Since \(\PP U \cap X = \{x_+,x_-\}\),
the only lines contracted by \(\proj_{\PP U}\) are those in the cones \(X_\pm\)
over \(C\). So if \(y_0 \in \PP W \setminus C\), then
\(\tilde{X}_{y_0}\) is a singular \(q\)-bic curve that contains no lines,
and hence is of type \(\mathbf{N}_2\oplus\mathbf{1}\) by \parref{qbic-curves-classification}.
On the other hand, if \(y_0 \in C\), the fact that \(x_+\)
is a cone point of multiplicity \(q\) implies that \(\tilde{X}_{y_0}\) must
be the union of the lines \(\langle y_0,x_- \rangle\) and \(\langle y_0,x_+ \rangle\),
the latter appearing with multiplicity \(q\); in particular, \(\tilde{X}_{y_0}\)
is of type \(\mathbf{0} \oplus \mathbf{N}_2\).
\end{proof}

\begin{Corollary}\label{nodal-projection-from-line-tangent}
Let \(\ell \subset X\) be a line not passing through either \(x_\pm\). Then
\[ \ell_0 \coloneqq \proj_{\PP U}(\ell) \subset \PP W \]
is the tangent to \(C\) at \(\proj_{\PP U}(\ell \cap X_+)\) with residual
intersection \(\proj_{\PP U}(\ell \cap X_-)\). Thus
\begin{align*}
\varphi_+([\ell]) & = \text{point of tangency between \(\ell_0\) and \(C\)}, \\
\varphi_-([\ell]) & = \text{residual point of intersection between \(\ell_0\) and \(C\)}.
\end{align*}
\end{Corollary}

\begin{proof}
Since the only lines contracted by \(\proj_{\PP U}\) are those through
\(x_\pm\), for a line \(\ell\) as in the statement,
\(\ell_0 \coloneqq \proj_{\PP U}(\ell)\) is a line in \(\PP W\). Observe that
\[
C \cap \ell_0 =
\proj_{\PP U}\big(\ell \cap \proj_{\PP U}^{-1}(C)\big) =
\proj_{\PP U}\big(\ell \cap (X_- \cup q X_+)\big)
\]
with the second equality due to \parref{threefolds-nodal-fibres-of-projection}.
Thus \(\ell_0\) and \(C\) are tangent at \(\proj_{\PP U}(\ell \cap X_+)\) and
have residual point of intersection at \(\proj_{\PP U}(\ell \cap X_-)\).
On the other hand, the definition of \(\varphi_\pm\) from
\parref{threefolds-cone-situation-rational-map-S} gives
\[
\varphi_\pm([\ell])
= \proj_{x_\pm}(\ell \cap X_\pm)
= \proj_{\PP U}(\ell \cap X_\pm)
\]
with the second equality because \(x_\mp \notin X_\pm\). The result now follows.
\end{proof}

This description of \(\varphi_-\) gives an alternate proof that it
extends to a morphism:

\begin{Corollary}\label{nodal-residual-intersection-morphism}
The rational map \(\varphi_-\) extends to a morphism \(S \to C\).
\end{Corollary}

\begin{proof}
By \parref{nodal-projection-from-line-tangent},
\(\varphi_-\) takes a \(\ell \subset X\) not contained in either cone \(X_\pm\)
to the residual intersection point with \(C\) to its tangent line at
\(\proj_{\PP U}(\ell \cap X_+)\). This description makes sense for
\(\ell \subset X_-\), thereby extending \(\varphi_-\) over \(C_-\).
\end{proof}

The descriptions of \(\varphi_\pm\) from
\parref{nodal-projection-from-line-tangent} yield a direct relationship between
them:

\begin{Corollary}\label{nodal-varphi-rational-diagram}
There is a commutative square
\[
\begin{tikzcd}
S \rar["\id_S"'] \dar[dashed,"\varphi_+"'] & S \dar["\varphi_-"] \\
C \rar["\phi_C"] & C
\end{tikzcd}
\]
where \(\phi_C \colon C \to C\) is the endomorphism sending a point \(x \in C\)
to the residual intersection point between \(C\) and its tangent line at \(x\),
as described in
\parref{curve-residual-intersection}. \qed
\end{Corollary}

The next task is to resolve \(\varphi_+\). By \parref{nodal-projection-from-line-tangent},
this can be done by constructing a proper family of lines in \(X\) parameterized
by \(C\) with the property that the general line lying over \(y_0 \in C\)
projects via \(\proj_{\PP U}\) to the tangent line \(\mathbf{T}_{C,y_0}\). The
universal such family is obtained by taking the pencil of hyperplane sections
of \(X\) parameterized by \(C\) such that the fibre over \(y_0 \in C\) is the
intersection of \(X\) with the hyperplane
\[
\mathcal{W}_{y_0} \coloneqq
\langle \mathbf{T}_{C,y_0}, \PP U \rangle \subset \PP V.
\]
This gives a family of \(q\)-bic surfaces over \(C\) in which each fibre has a
ruling by lines. Taking this family of rulings appropriately yields the
sought-after resolution.

\subsection{Tangent pencil}\label{nodal-tangent-pencil-construction}
To proceed with the construction, let \(\mathcal{T}_C^{\mathrm{e}} \subset W_C\) be the
embedded tangent bundle of \(C\) in \(\PP W\), as in
\parref{hypersurfaces-embedded-tangent-sheaf}, and let
\[
\mathcal{W} \coloneqq
U_C \oplus \mathcal{T}_C^{\mathrm{e}} \subset
U_C \oplus W_C \subset
V_C.
\]
The projective bundle \(\PP\mathcal{W}\) admits a canonical morphism to
\(\PP V\) so that the image of the fibre \(\PP\mathcal{W}_{y_0}\) over
a point \(y_0 \in C\) is the hyperplane spanned by \(\PP U\) and
\(\mathbf{T}_{C,y_0} \subset \PP W \subset \PP V\). The corresponding family
\(X_{\PP\mathcal{W}} \coloneqq X \times_{\PP V} \PP\mathcal{W}\)
of hyperplane sections of \(X\) is then the family of \(q\)-bic surfaces over
\(C\) determined by the \(q\)-bic form
\[
\beta_{\mathcal{W}} \colon
\Fr^*(\mathcal{W}) \otimes \mathcal{W} \subset
\Fr^*(V)_C \otimes V_C \xrightarrow{\beta}
\sO_C.
\]
The splitting \((V,\beta) = (U,\beta_U) \perp (W,\beta_W)\) of
\parref{nodal-splitting} gives an orthogonal decomposition
\[
(\mathcal{W},\beta_{\mathcal{W}}) =
(U_C, \beta_U) \perp
(\mathcal{T}_C^{\mathrm{e}}, \beta_{W,\mathrm{tan}})
\]
where \(\beta_{W,\mathrm{tan}}\) is the tangent form associated with \(C\),
as in \parref{hypersurfaces-tangent-kernel}. This implies that the fibres of
\(X_{\PP\mathcal{W}} \to C\) are certain singular \(q\)-bic surfaces:

\begin{Lemma}\label{nodal-tangent-pencil-types}
The fibre \(X_{\PP\mathcal{W},y_0}\) over \(y_0 \in C\) is a \(q\)-bic surface
of type
\[
\mathrm{type}(X_{\PP\mathcal{W},y_0}) =
\begin{dcases}
\mathbf{N}_2^{\oplus 2} & \text{if \(y_0 \in C\) is not a Hermitian point}, \\
\mathbf{0} \oplus \mathbf{N}_2 \oplus \mathbf{1}  & \text{if \(y_0 \in C\) is a Hermitian point.}
\end{dcases}
\]
The singular locus of \(X_{\PP\mathcal{W}} \to C\) is supported on the subbundle
\[
\PP\mathcal{W}'
= \PP(L_{+,C} \oplus \sO_C(-1))
= \Set{(y_0,x) | y_0 \in C, x \in \langle y_0,x_+ \rangle \subset \PP\mathcal{W}_{y_0}}
\subset \PP\mathcal{W}.
\]
The subbundle \(\mathcal{W}'\) is characterized by the property that
\(\Fr^*(\mathcal{W}') = \mathcal{W}^\perp\).
\end{Lemma}

\begin{proof}
The identification of types follows from the orthogonal decomposition of
\(\beta_{\mathcal{W}}\): \(\beta_{U}\) is a constant \(q\)-bic form of type
\(\mathbf{N}_2\), whereas \parref{curve-residual-intersection} implies that
\(\beta_{W,\mathrm{tan}}\) is of type \(\mathbf{0} \oplus \mathbf{1}\) over
Hermitian points of \(C\) and \(\mathbf{N}_2\) otherwise. By \parref{hypersurfaces-nonsmooth-locus},
the singular locus of the morphism \(X_{\PP\mathcal{W}} \to C\) is given by
the image of
\[
\mathcal{W}^\perp
= U_C^{\perp_{\beta_U}} \oplus (\mathcal{T}_C^{\mathrm{e}})^{\perp_{\beta_{W,\mathrm{tan}}}}
= \Fr^*(L_+)_C \oplus \Fr^*(\sO_C(-1))
\]
by \parref{nodal-splitting} and \parref{hypersurfaces-tangent-form-properties}.
This gives the remaining statements.
\end{proof}

The analyses from \parref{surfaces-N2+N2.cone-points} and \parref{surfaces-0+1+N2.cone-points}
show that \(q\)-bic surfaces of types \(\mathbf{N}_2^{\oplus 2}\) and
\(\mathbf{0} \oplus \mathbf{N}_2 \oplus \mathbf{1}\) are ruled by a family of
lines transversal to the singular line, and that this family of lines can be
constructed by projecting away from the singular locus.
Since the singular lines in the fibres of \(X_{\PP\mathcal{W}} \to C\)
form a subbundle over \(C\), this construction can be performed in this
relative setting and it will yield the sought-after family
\(\tilde{X}_{\PP\mathcal{W}} \to S^\nu\) of lines in \(X\) projecting to
tangent lines to \(C\).

\subsection{Construction of \texorpdfstring{\(\tilde{X}_{\PP\mathcal{W}}\)}{~X_PPW}}
\label{threefolds-nodal-construction-of-blowup}
Let \(\mathcal{W}'' \coloneqq \mathcal{W}/\mathcal{W}'\). Then linear
projection over \(C\) of \(X_{\PP\tilde{\mathcal{W}}} \subset \PP\mathcal{W}\)
along the subbundle \(\PP\mathcal{W}'\) yields a rational map to the surface
\[
S^\nu \coloneqq
\PP\mathcal{W}'' \stackrel{\tilde\varphi_+}{\longrightarrow}
C.
\]
Blowing up along \(\PP\mathcal{W}'\) resolves this into a morphism
\(\tilde{X}_{\PP\mathcal{W}} \to S^\nu\), which can be identified as a
\(\PP^1\)-bundle as follows: The blowup of \(\PP\mathcal{W}\) along
\(\PP\mathcal{W}'\) is identified in \parref{linear-projection-resolve} as the
\(\PP^2\)-bundle over \(S^\nu\) associated with \(\tilde{\mathcal{W}}\) formed
in the pullback diagram
\[
\begin{tikzcd}
0 \rar
& \tilde\varphi_+^*\mathcal{W}' \rar \dar[equal]
& \tilde{\mathcal{W}} \rar \dar[hook]
& \sO_{\tilde\varphi_+}(-1) \rar \dar[hook, "\mathrm{eu}_{\tilde\varphi_+}"]
& 0 \\
0 \rar
& \tilde\varphi_+^*\mathcal{W}' \rar
& \tilde\varphi_+^*\mathcal{W} \rar
& \tilde\varphi_+^*\mathcal{W}'' \rar
& 0
\end{tikzcd}
\]
and so the exceptional divisor of \(\PP\tilde{\mathcal{W}} \to \PP\mathcal{W}\)
is the subbundle \(\PP(\tilde\varphi_+^*\mathcal{W}') \subset \PP\tilde{\mathcal{W}}\).
The inverse image
\(X_{\PP\tilde{\mathcal{W}}} \coloneqq X_{\PP\mathcal{W}} \times_{\PP\mathcal{W}} \PP\tilde{\mathcal{W}}\)
of \(X_{\PP\mathcal{W}}\) along this blowup is the bundle of
\(q\)-bic curves over \(S^\nu\) defined by the \(q\)-bic form
\[
\beta_{\tilde{\mathcal{W}}} \colon \Fr^*(\tilde{\mathcal{W}}) \otimes \tilde{\mathcal{W}}
\subset \Fr^*(\tilde\varphi_+^*\mathcal{W}) \otimes \tilde\varphi_+^*\mathcal{W}
\xrightarrow{\beta_{\mathcal{W}}} \sO_{S^\nu}.
\]
Since \(\Fr^*(\mathcal{W}') = \mathcal{W}^\perp\) by
\parref{nodal-tangent-pencil-types},
\(\tilde\varphi_+^*\Fr^*(\mathcal{W}') = \tilde{\mathcal{W}}^\perp\) and
there is an exact sequence
\[
0 \to
\mathcal{K} \to
\tilde{\mathcal{W}} \xrightarrow{\beta_{\tilde{\mathcal{W}}}}
\Fr^*\tilde{\mathcal{W}}^\vee \to
\Fr^*(\tilde\varphi_+^*\mathcal{W}')^\vee \to
0
\]
where \(\mathcal{K}\) is a rank \(2\) subundle such that
\(\tilde{\mathcal{W}}/\mathcal{K} \cong \sO_{\tilde{\varphi}_+}(q)\) via
\(\beta_{\tilde{\mathcal{W}}}\). The sequences implies that
\(X_{\PP\tilde{\mathcal{W}}} \to S^\nu\) is generically of type
\(\mathbf{0} \oplus \mathbf{N}_2\) with its two irreducible components given by
\(\PP\mathcal{K}\) and \(\PP(\tilde{\varphi}^*_+\mathcal{W}')\), appearing
with multiplicities \(1\) and \(q\), respectively. Since
\(\PP(\tilde{\varphi}^*_+\mathcal{W}')\) is the exceptional divisor,
the strict transform is \(\tilde{X}_{\PP W} = \PP\mathcal{K}\).

Therefore \(\tilde{X}_{\PP\mathcal{W}} \to S^\nu\) is the family of lines in
\(X\) given by the isotropic subbundle \(\mathcal{K} \subset V_{S^\nu}\). This
induces a morphism \(\nu \colon S^\nu \to S\). Let
\[
\mathcal{W}'' = L_{-,C} \oplus \mathcal{T}_C(-1),
\quad
C_-^\nu \coloneqq \PP L_{-,C},
\quad
C_+^\nu \coloneqq \PP(\mathcal{T}_C(-1))
\]
be the splitting induced by the orthogonal decomposition
\(\mathcal{W} = U_C \oplus \mathcal{T}^{\mathrm{e}}_C\) and
the corresponding subbundles of \(S^\nu\).

\subsection{Fibres of \texorpdfstring{\(\mathcal{K}\)}{K}}\label{threefolds-nodal-K}
To understand the morphism \(\nu\), consider the fibre of \(\mathcal{K}\) at a
point \(x \in S^\nu\). Set \(y \coloneqq \tilde\varphi_+(x) \in C\).
The plane \(\PP\tilde{\mathcal{W}}_x\) is that spanned in \(\PP V\) by the line
\(\ell_y \coloneqq \langle x_+, y \rangle = \PP\mathcal{W}'_y\)
and a point \(z \neq y\) on the cone \(X_-\). Then
\(X \cap \PP\tilde{\mathcal{W}}_x\) is a
\(q\)-bic curve containing \(\ell_y\) with multiplicity \(q\) together
with a residual line \(\ell = \PP\mathcal{K}_x\).
The behaviour of \(\mathcal{K}_x\) splits into two main cases:
\begin{itemize}
\item If \(x \in C_+^\nu\), then
\(\PP\tilde{\mathcal{W}}_x = \langle x_+, \mathbf{T}_{C,y} \rangle\) and
\(X \cap \PP\tilde{\mathcal{W}}_x = (q+1)\ell_y\), so
\(\mathcal{K}_x = \mathcal{W}'_y\).
\item If \(x \notin C_+^\nu\), then
\(\PP\tilde{\mathcal{W}}_x = \langle x_+,y,z \rangle\) for \(z \in X_- \setminus \PP W\),
so \(z \in \ell\) and \(\mathcal{K}_x \neq \mathcal{W}'_y\).
\end{itemize}
Since \(\nu(x) = [\ell]\), this shows that
\(\nu(C^\nu_\pm) = C_\pm\) and implies that \(\nu\) is bijective on points.
In fact, much more is true:

\begin{Proposition}\label{threefolds-nodal-normalization}
The morphism \(\nu \colon S^\nu \to S\) fits into a commutative diagram
\[
\begin{tikzcd}
S^\nu \ar[rr,"\nu"] \ar[dr,"\tilde\varphi_+"'] && S \ar[dl,dashed,"\varphi_+"] \\
& C\punct{,}
\end{tikzcd}
\]
is bijective, restricts to an isomorphism
\(S^\nu \setminus C_+^\nu \to S \setminus C_+\), and so is the normalization.
\end{Proposition}

That the diagram commutes on closed points can be seen by comparing
the description of \(\varphi_+\) from
\parref{nodal-projection-from-line-tangent} with the fact that the fibres
\(S^\nu_{y_0}\), for \(y_0 \in C\), parameterizes the lines contained in the
hyperplane section
\(X_{\PP\mathcal{W}_{y_0}} = X \cap \langle \mathbf{T}_{C,y_0}, \PP U \rangle\).
Bijectivity of \(\nu\) is as above or can then be seen fibrewise via
\parref{surfaces-corank-2-lines}. However, the crucial point
is separability of \(\nu\). The following argument will directly
construct an inverse to \(\nu \colon S^\nu \setminus C_+^\nu \to S \setminus C_+\).

\begin{proof}
Let \(S^{\nu,\circ} \coloneqq S^\nu \setminus C_+^\nu\) and
\(S^\circ \coloneqq S \setminus C_+\). The discussion of
\parref{threefolds-nodal-construction-of-blowup} together with
\parref{threefolds-nodal-K} gives an exact commutative
diagram of locally free modules on \(S^{\nu,\circ}\) given by
\[
\begin{tikzcd}
0 \rar
& \mathcal{K}' \rar \dar[hook]
& \mathcal{K}\rvert_{S^{\nu,\circ}} \rar \dar[hook]
& \sO_{\tilde\varphi_+}(-1)\rvert_{S^{\nu,\circ}} \dar["\mathrm{eu}_{\tilde\varphi_+}",hook] \rar
& 0 \\
0 \rar
& \tilde\varphi_+^*\mathcal{W}'\rvert_{S^{\nu,\circ}} \rar
& \tilde\varphi_+^*\mathcal{W}\rvert_{S^{\nu,\circ}} \rar
& \tilde\varphi_+^*\mathcal{W}''\rvert_{S^{\nu,\circ}} \rar
& 0\punct{.}
\end{tikzcd}
\]
Moreover, since \(\ell = \PP\mathcal{K}_x\) never passes through \(x_+\) for
\(x \in S^{\nu,\circ}\), the composition
\[
\mathcal{K}' \hookrightarrow
\tilde\varphi_+^*\mathcal{W}'\rvert_{S^{\nu,\circ}} =
L_{+,S^{\nu,\circ}} \oplus
\tilde\varphi_+^*\sO_C(-1)\rvert_{S^{\nu,\circ}}
\twoheadrightarrow
\tilde\varphi_+^*\sO_C(-1)\rvert_{S^{\nu,\circ}}
\]
is an isomorphism. Thus
\(\tilde\varphi_+ \colon S^{\nu,\circ} \to C\)
is determined by the map \(\mathcal{K}' \to W_{S^{\nu,\circ}}\).

On the other hand, the description \(\varphi_+([\ell]) = \proj_{x_+}(\ell \cap X_+)\)
from \parref{threefolds-cone-situation-rational-map-S} shows that
\(\varphi_+ \colon S^\circ \to C\) is determined by a map
\(\mathcal{S}' \to W_{S^\circ}\) from a line subbundle
\(\mathcal{S}' \subset \mathcal{S}\) fitting into the exact commutative diagram
\[
\begin{tikzcd}[row sep=1.5em]
0 \rar
& \mathcal{S}' \rar \dar[hook]
& \mathcal{S}\rvert_{S^\circ} \rar \dar[hook]
& L_{-,S^\circ} \rar \dar[equal]
& 0 \\
0 \rar
& (L_+ \oplus W)_{S^\circ} \rar
& V_{S^\circ} \rar
& L_{-,S^\circ} \rar
& 0
\end{tikzcd}
\]
where \(\mathcal{S}\rvert_{S^\circ} \to L_{-,S^\circ}\) is surjective
by definition of \(S^\circ\). Since the composite
\[
\sO_{\tilde\varphi_+}(-1)\rvert_{S^{\nu,\circ}} \xrightarrow{\mathrm{eu}_{\tilde\varphi_+}}
\tilde\varphi_+^*\mathcal{W}''\rvert_{S^\nu,\circ} =
L_{-,S^{\nu,\circ}} \oplus \tilde\varphi_+^*\mathcal{T}_C(-1)\rvert_{S^{\nu,\circ}} \twoheadrightarrow
L_{-,S^{\nu,\circ}}
\]
is an isomorphism on \(S^{\nu,\circ}\), comparing the two diagrams shows that
\[ \nu^*(\mathcal{S}' \to W_{S^\circ}) = (\mathcal{K}' \to W_{S^{\nu,\circ}}), \]
which proves that \(\tilde\varphi_+ = \varphi_+ \circ \nu\) on \(S^{\nu,\circ}\).

To construct the inverse to \(\nu \colon S^{\nu,\circ} \to S^\circ\),
observe that
\(\mathcal{S}\rvert_{S^\circ} \subset \varphi_+^*\mathcal{W}\)
by \parref{nodal-projection-from-line-tangent}, and via this inclusion,
\(\mathcal{S}'\) includes into \(\varphi_+^*\mathcal{W}'\).
Set
\[
\mathcal{S}'' \coloneqq
\image(
\mathcal{S}\rvert_{S^\circ} \to
\varphi_+^*\mathcal{W} \twoheadrightarrow
\varphi_+^*\mathcal{W}'').
\]
Then \(\mathcal{S}'' \cong \mathcal{S}\rvert_{S^\circ}/\mathcal{S}'\) and it
projects isomorphically onto the \(L_{-,S^\circ}\)
part of \(\varphi_+^*\mathcal{W}''\). The inclusion \(\mathcal{S}'' \subset \varphi_+^*\mathcal{W}''\)
gives a morphism \(S^\circ \to \PP\mathcal{W}'' = S^\nu\) of schemes over \(C\),
and comparing with the diagram for \(\mathcal{K}\rvert_{S^\nu,\circ}\) shows
that this is an inverse for \(\nu\).
\end{proof}

Putting \parref{nodal-residual-intersection-morphism},
\parref{nodal-varphi-rational-diagram}, and
\parref{threefolds-nodal-normalization}
together yield the commutative diagram in:

\begin{Corollary}\label{nodal-nu-and-F}
There is a commutative diagram of morphisms
\[
\begin{tikzcd}
S^\nu \rar["\nu"'] \dar["\tilde\varphi_+"'] & S \dar["\varphi_-"] \\
C \rar["\phi_C"] & C
\end{tikzcd}
\]
which is equivariant for the action
\[
\mathbf{G}_m \times
\mathrm{U}_3(q) \cong
(\AutSch(L_+ \subset U, \beta_U) \cap \AutSch(L_- \subset U, \beta_U)) \times
\AutSch(W,\beta_W).
\]
\end{Corollary}

\begin{proof}
That the diagram is equivariant is because \(\tilde\varphi_+\) and
\(\varphi_-\) both arise from Cone Situations, so that the discussion of
\parref{threefolds-cone-situation-equivariance} applies. Since the unipotent
part must act trivially to preserve \(L_+\), the computation of
\parref{qbic-points-automorphisms.N2} implies
\(\AutSch(L_+ \subset U,\beta_U) \cong \mathbf{G}_m\).
\end{proof}

\subsection{Conductors}\label{nodal-conductors}
Let
\[
\operatorname{cond}_{\nu,S} \coloneqq
\mathcal{A}\!\mathit{nn}_{\sO_S}(\nu_*\sO_{S^\nu}/\sO_S) \subset \sO_S
\quad\text{and}\quad
\operatorname{cond}_{\nu,S^\nu} \coloneqq
\nu^{-1}\operatorname{cond}_{\nu,S} \cdot \sO_{S^\nu} \subset
\sO_{S^\nu}
\]
be the conductor ideals associated with the normalization
\(\nu \colon S^\nu \to S\), and let
\[
D \coloneqq \mathrm{V}(\operatorname{cond}_{\nu,S}) \subset S
\quad\text{and}\quad
D^\nu \coloneqq \mathrm{V}(\operatorname{cond}_{\nu,S^\nu}) \subset S^\nu
\]
be the conductor subschemes of \(S\) and \(S^\nu\), respectively. The conductor
ideal of \(S\) is characterized as the largest ideal of \(\sO_S\) which is also
an ideal of \(\nu_*\sO_{S^\nu}\), so there is a commutative diagram of exact
sequences of sheaves on \(S\):
\[
\begin{tikzcd}[row sep=1.5em]
& \operatorname{cond}_{\nu,S} \dar[hook] \rar["\cong"']
& \nu_*\operatorname{cond}_{\nu,S^\nu} \dar[hook] \\
0 \rar
& \sO_S \rar["\nu^\#"] \dar[two heads]
& \nu_*\sO_{S^\nu} \rar \dar[two heads]
& \nu_*\sO_{S^\nu}/\sO_S \rar \dar["\cong"]
& 0 \\
0 \rar
& \sO_D \rar["\nu^\#"]
& \nu_*\sO_{D^\nu} \rar
& \nu_*\sO_{D^\nu}/\sO_D \rar
& 0\punct{.}
\end{tikzcd}
\]
The basic properties of the conductor subschemes are as follows:

\begin{Lemma}\label{nodal-conductors-basics}
Let \(D \subset S\) and \(D^\nu \subset S^\nu\) be the conductor subschemes
associated with the normalization morphism \(\nu \colon S^\nu \to S\). Then
\begin{enumerate}
\item\label{nodal-conductors-basics.diagram}
there is a \(\mathbf{G}_m \times \mathrm{U}_3(q)\)-equivariant commutative
diagram
\[
\begin{tikzcd}
D^\nu \rar["\nu"'] \dar["\tilde\varphi_+"'] & D \dar["\varphi_-"] \\
C \rar["\phi_C"] & C\punct{,}
\end{tikzcd}
\]
\item\label{nodal-conductors-basics.fixed-scheme}
\(D^{\mathbf{G}_m} = D_{\mathrm{red}} = C_+\) and
\(D^{\nu,\mathbf{G}_m} = D^\nu_{\mathrm{red}} = C_+^\nu\), and
\item\label{nodal-conductors-basics.finite}
the morphisms \(\varphi_- \colon D \to C\) and \(\tilde\varphi_+ \colon D^\nu \to C\)
are finite.
\end{enumerate}
\end{Lemma}

\begin{proof}
The diagram of \ref{nodal-conductors-basics.diagram} exists and commutes
because of \parref{nodal-nu-and-F}; the action of
\(\mathbf{G}_m \times \mathrm{U}_3(q)\) on the surfaces restricts to an action
on the conductors because \(\nu \colon S^\nu \to S\) is equivariant, and so
\(\mathcal{A}\!\mathit{nn}_{\sO_S}(\nu_*\sO_{S^\nu}/\sO_S)\) is an equivariant
ideal.

For \ref{nodal-conductors-basics.fixed-scheme}, note that
\(D_+\) and \(D_+^\nu\) are supported on \(C_+\) and \(C_+^\nu\), respectively,
since \(\nu \colon S^\nu \to S\) is an isomorphism away from these curves by
\parref{threefolds-nodal-normalization}. That
\(D_+^{\mathbf{G}_m} = C_+\) follows from the corresponding statement for \(S\)
from \parref{nodal-Gm-action-S}; likewise, that
\(D_+^{\nu,\mathbf{G}_m} = C_+^\nu\) is because \(S^\nu \to C\) is a
\(\PP^1\)-bundle and \(\mathbf{G}_m\) acts along fibres with fixed locus
\(C_+^\nu \cup C_-^\nu\). This now implies \ref{nodal-conductors-basics.finite}
using properness of \(S\) and \(S^\nu\) over \(C\).
\end{proof}

The conductor ideal is determined in
\parref{threefolds-nodal-compute-conductor} via the identification
\(\operatorname{cond}_{\nu,S^\nu} \cong \omega_{S^\nu/S}\) from duality theory.
This is aided by two auxiliary computations. The first involves the
rank \(2\) subbundle \(\mathcal{K} \subset V_{S^\nu}\) constructed in
\parref{threefolds-nodal-construction-of-blowup}:

\begin{Lemma}\label{nodal-K'-presentation}
There are short exact sequences on \(S^\nu\) given by
\begin{align*}
0 \to
\sO_{\tilde\varphi_+}(-2q) \otimes \wedge^3\tilde{\mathcal{W}} \to
\sO_{\tilde\varphi_+}(-q)  \otimes \wedge^2\tilde{\mathcal{W}} \to
\mathcal{K} & \to
0, \;\text{and}\\
0 \to
\mathcal{K} & \to
\tilde{\mathcal{W}} \to
\sO_{\tilde\varphi_+}(q) \to
0.
\end{align*}
In particular,
\(\det(\mathcal{K}) \cong \tilde\varphi^*_+\sO_C(-1) \otimes \sO_{\tilde\varphi_+}(-q-1) \otimes L_+\).
\end{Lemma}

\begin{proof}
The concatenation of the two sequences is the Koszul complex for the surjection
\(\beta_{\tilde{\mathcal{W}}} \colon \tilde{\mathcal{W}} \to \sO_{\tilde\varphi_+}(q)\)
from \parref{threefolds-nodal-construction-of-blowup}.
Compute the determinant by using, for example, the second exact sequence:
\begin{align*}
\det(\mathcal{K})
\cong \det(\tilde{\mathcal{W}}) \otimes \sO_{\tilde\varphi_+}(-q)
& \cong \tilde\varphi^*_+\det(\mathcal{W}') \otimes \sO_{\tilde\varphi_+}(-q-1) \\
& \cong \tilde\varphi^*_+\sO_C(-1) \otimes \sO_{\tilde\varphi_+}(-q-1) \otimes L_+
\end{align*}
upon using the identification \(\mathcal{W}' = L_{+,C} \oplus \sO_C(-1)\) from
\parref{nodal-tangent-pencil-types}.
\end{proof}

The second computation identifies the \(\nu\)-pullback of the Pl\"ucker line bundle:

\begin{Lemma}\label{nodal-nu-pullback-plucker}
\(\nu^*\sO_S(1) = \tilde\varphi_+^*\sO_C(1) \otimes \sO_{\tilde\varphi_+}(q+1) \otimes L_+^\vee\).
\end{Lemma}

\begin{proof}
By its construction in \parref{threefolds-nodal-construction-of-blowup},
the pullback of the tautological subbundle \(\mathcal{S}\) on \(S\) via
\(\nu\) is the bundle \(\mathcal{K}\) on \(S^\nu\).
Therefore,
\[
\nu^*\sO_S(1)
\cong \nu^*\det(\mathcal{S})^\vee
\cong \det(\mathcal{K})^\vee
\cong \tilde\varphi^*_+\sO_C(1) \otimes \sO_{\tilde\varphi_+}(q+1) \otimes L_+^\vee,
\]
upon using the determinant computation of \parref{nodal-K'-presentation}.
\end{proof}

\begin{Proposition}\label{threefolds-nodal-compute-conductor}
The conductor ideal of \(S^\nu\) is isomorphic to
\[
\operatorname{cond}_{\nu,S^\nu} \cong
\sO_{\tilde\varphi_+}(-\delta-1) \otimes (L_+^{\otimes 2q-1} \otimes L_-)
\quad\text{where}\;\delta \coloneqq 2q^2 - q - 2,
\]
and the conductor subscheme \(D^\nu\) is the \(\delta\)-order
neighbourhood of \(C_+^\nu\).
\end{Proposition}

\begin{proof}
Evaluation at \(1 \in \nu_*\sO_{S^\nu}\) yields the first isomorphism
in
\[
\nu_*\operatorname{cond}_{\nu,S^\nu} \cong
\mathcal{H}\!\mathit{om}_{\sO_S}(\nu_*\sO_{S^\nu},\sO_S)
\cong \nu_*\omega_{S^\nu/S}
\]
and duality theory for \(\nu \colon S^\nu \to S\) gives the second isomorphism;
see \citeSP{0FKW}, for instance.
As \(\nu\) is affine, it follows that \(\operatorname{cond}_{\nu,S^\nu}\) is
isomorphic to the relative dualizing sheaf
\(\omega_{S^\nu/S} \cong \omega_{S^\nu} \otimes \nu^*\omega_S^\vee\).
Since \(S^\nu\) the projective bundle over \(C\) associated with
\(\mathcal{W}'' = \mathcal{T}_C(-1) \oplus L_{-,C}\), the relative
Euler sequence gives
\[
\omega_{S^\nu} \cong
\omega_{S^\nu/C} \otimes \tilde\varphi_+^*\omega_C \cong
\sO_{\tilde\varphi_+}(-2) \otimes
\tilde\varphi^*_+(\omega_C^{\otimes 2} \otimes \sO_C(1)) \otimes L_-^\vee.
\]
By \parref{threefolds-lines},
\(\omega_S \cong \sO_S(2q - 3) \otimes (L_+\otimes L_-)^{\vee,\otimes 2}\),
so \parref{nodal-nu-pullback-plucker} gives
\begin{align*}
\nu^*\omega_S
& \cong \nu^*\sO_S(2q-3) \otimes (L_+ \otimes L_-)^{\vee, \otimes 2} \\
& \cong
\tilde\varphi_+^*\sO_C(2q-3) \otimes
\sO_{\tilde\varphi_+}\big((2q-3)(q+1)\big) \otimes
L_+^{\vee, \otimes 2q-1} \otimes L_-^{\vee,\otimes 2}.
\end{align*}
Since \(C\) is a plane curve of degree \(q+1\),
\(\omega_C^{\otimes 2} \otimes_{\sO_C} \sO_C(1) \cong \sO_C(2q-3)\).
Putting the computations together yields
\[
\omega_{S^\nu/S}
\cong \omega_{S^\nu} \otimes \nu^*\omega_S^\vee
\cong \sO_{\tilde\varphi_+}(-\delta-1) \otimes (L_+^{\otimes 2q-1} \otimes L_-).
\]
That \(D^\nu\) is the \(\delta\)-order neighbourhood of \(C_+^\nu\)
now follows from \parref{nodal-conductors-basics}\ref{nodal-conductors-basics.fixed-scheme}.
\end{proof}

\subsection{The sheaf \texorpdfstring{\(\mathcal{F}\)}{F}}\label{nodal-conductors-F}
By \parref{nodal-conductors-basics}, the \(\sO_C\)-modules given by
\[
\mathcal{D} \coloneqq \varphi_{-,*}\sO_D
\quad\text{and}\quad
\mathcal{D}^\nu \coloneqq \phi_{C,*} \tilde\varphi_{+,*} \sO_{D^\nu}
\]
are graded \(\sO_C\)-algebras whose spectra over \(C\) yield the finite morphisms
\(\varphi_- \colon D \to C\) and \(\phi_C \circ \tilde\varphi_+ \colon D^\nu \to C\).
There is a sequence of graded coherent \(\sO_C\)-modules
\[ 0 \to \mathcal{D} \xrightarrow{\nu^\#} \mathcal{D}^\nu \to \mathcal{F} \to 0 \]
where the morphism \(\nu^\#\) is induced by \(\nu \colon D^\nu \to D\).
The graded module \(\mathcal{F}\) will play an important role in the
computation of the invariants of \(S\) thanks to the following:

\begin{Lemma}\label{normalize-splitting}
The graded coherent \(\sO_C\)-module
\[
\mathcal{F}
\coloneqq \mathcal{D}^\nu/\mathcal{D}
\cong \varphi_{-,*}(\nu_*\sO_{D^\nu}/\sO_D)
\cong \varphi_{-,*}(\nu_*\sO_{S^\nu}/\sO_S)
\]
is locally free and fits into a commutative diagram
\[
\begin{tikzcd}[row sep=1.5em]
0 \rar
& \sO_C \rar \dar[hook]
& \phi_{C,*} \sO_C \rar \dar[hook]
& \mathcal{F} \rar \dar[equal]
& \mathbf{R}^1\varphi_{-,*}\sO_S \rar
& 0 \\
0 \rar
& \mathcal{D} \rar
& \mathcal{D}^\nu \rar
& \mathcal{F} \rar
& 0
\end{tikzcd}
\]
in which the rows are exact and the surjection
\(\mathcal{F} \to \mathbf{R}^1\varphi_{-,*}\sO_S\) splits.
\end{Lemma}

\begin{proof}
The identifications of \(\mathcal{F}\) come from affineness of the morphisms
\(D \to C\) and \(D^\nu \to C\) from \parref{nodal-conductors-basics} together
with the commutative diagram of \parref{nodal-conductors}. Pushing the rows
of that diagram down to \(C\) yields the diagram in the statement;
exactness of the first row comes from
\(\varphi_{-,*}\sO_S \cong \tilde\varphi_{+,*}\sO_{S^\nu} \cong \sO_C\) by
\parref{threefolds-smooth-cone-situation-pushforward}\ref{threefolds-smooth-cone-situation-pushforward.O},
the commutative diagram of \parref{nodal-nu-and-F}, and that
\(\mathbf{R}^1\tilde\varphi_{+,*}\sO_{S^\nu} = 0\) as it is a projective bundle
over \(C\). The top row of the diagram shows that \(\mathcal{F}\) is an extension
of the locally free sheaves \(\mathbf{R}^1\varphi_{-,*}\sO_S\), which is locally free by
\parref{threefolds-smooth-cone-situation-pushforward}\ref{threefolds-smooth-cone-situation-pushforward.R1},
by \(\phi_{C,*}\sO_C/\sO_C\), which is locally free since \(C\) is regular so
the morphism \(\phi_{C,*} \colon C \to C\), which up to an automorphism is the
\(q^2\)-Frobenius by \parref{hypersurfaces-endomorphism-V}, is flat by
\cite[Theorem 2.1]{Kunz:Flat} or \citeSP{0EC0}.

To show that \(\mathcal{F} \to \mathbf{R}^1\varphi_{-,*}\sO_S\) splits, observe
that \(\sO_C \hookrightarrow \mathcal{D}\) and \(\phi_{C,*}\sO_C \hookrightarrow \mathcal{D}^\nu\)
are the inclusion of the constant functions which, by
\parref{nodal-conductors-basics}\ref{nodal-conductors-basics.fixed-scheme},
make up the degree \(0\) components. Therefore the positively graded components
of \(\mathcal{F}\) map isomorphically to \(\mathbf{R}^1\varphi_{-,*}\sO_S\),
providing the desired splitting.
\end{proof}

\begin{Lemma}\label{threefolds-conductor-plucker}
\(\sO_S(1)\rvert_D = \varphi_-^*\sO_C(1)\rvert_D \otimes L_+^\vee\).
\end{Lemma}

\begin{proof}
By \parref{threefolds-cone-situation-rational-map-S},
\(\varphi_- \colon S \setminus C_- \to C\) is induced by
the line subbundle of \(\mathcal{S}\rvert_{S \setminus C_-}\) obtained by
intersecting with \((L_- \oplus W)_{S \setminus C_-}\). Thus there is a short
exact sequence
\[
0 \to
\varphi_-^*\sO_C(-1)\rvert_{S \setminus C_-} \to
\mathcal{S}\rvert_{S \setminus C_-} \to
L_{+,S \setminus C_-} \to
0.
\]
Taking determinants and restricting to \(D\) yields the desired identification.
\end{proof}

The following shows that sheaf \(\mathcal{F}\) is, in a sense,  dual to the
sheaf \(\mathcal{D}\):

\begin{Proposition}\label{threefolds-conductor-dual}
There are canonical isomorphisms
\begin{enumerate}
\item\label{threefolds-conductor-dual.S}
\(\nu_*\sO_{S^\nu}/\sO_S \cong
\mathbf{R}\mathcal{H}\!\mathit{om}_{\sO_S}(\sO_D,\sO_S)[1]\)
in the derived category of \(S\), and
\item\label{threefolds-conductor-dual.C}
\(\mathcal{F} \cong \mathcal{D}^\vee \otimes \sO_C(-q+1) \otimes L_+^{\otimes 2q-1} \otimes L_-^{\otimes 2}\)
as graded \(\sO_C\)-modules.
\end{enumerate}
\end{Proposition}

\begin{proof}
Consider the conductor subscheme exact sequence on \(S\):
\[
0 \to
\operatorname{cond}_{\nu,S} \to
\sO_S \to
\sO_D \to
0.
\]
Identifying \(\operatorname{cond}_{\nu,S}\) with
\(\nu_*\omega_{S^\nu/S}\) as in \parref{threefolds-nodal-compute-conductor} and
applying \(\mathbf{R}\mathcal{H}\!\mathit{om}_{\sO_S}(-,\sO_S)\) to the
corresponding triangle in the derived category yields a triangle
\[
\mathbf{R}\mathcal{H}\!\mathit{om}_{\sO_S}(\sO_D,\sO_S) \to
\sO_S \to
\nu_*\sO_{S^\nu} \xrightarrow{+1}
\]
since \(\mathbf{R}\mathcal{H}\!\mathit{om}_{\sO_S}(\sO_S,\sO_S) = \sO_S\),
and
\(\mathbf{R}\mathcal{H}\!\mathit{om}_{\sO_S}(\nu_*\omega_{S^\nu/S},\sO_S) = \nu_*\sO_{S^\nu}\)
by duality for \(\nu\). The map \(\sO_S \to \nu_*\sO_{S^\nu}\) is
dual to evaluation at \(1\), and hence is the \(\sO_S\)-module map determined
by \(1 \mapsto 1\); in other words, this is the map \(\nu^\#\), yielding the
first statement.

Pushing forward \ref{threefolds-conductor-dual.S} and applying relative duality
for \(\varphi_- \colon S \to C\) yields
\[
\mathcal{F}
\cong \mathbf{R}\varphi_{-,*}\mathbf{R}\mathcal{H}\!\mathit{om}_{\sO_S}(\sO_D,\sO_S)[1]
\cong \mathbf{R}\mathcal{H}\!\mathit{om}_{\sO_C}(\mathbf{R}\varphi_{-,*}(\sO_D \otimes \omega_{\varphi_-}), \sO_C)[1]
\]
in which
\(\omega_{\varphi_-} = (\omega_S \otimes \varphi_-^*\omega_C^\vee)[1]\). By
\parref{threefolds-lines} together with \parref{threefolds-conductor-plucker},
\[
\sO_D \otimes \omega_S \cong
\varphi_-^*\sO_C(2q-3)\rvert_D \otimes
L_+^{\vee,\otimes 2q-1} \otimes L_-^{\vee, \otimes 2}.
\]
Combining with \(\omega_C \cong \sO_C(q-2)\) gives
\[
\mathbf{R}\varphi_{-,*}(\sO_D \otimes \omega_{\varphi_-}) =
(\mathbf{R}\varphi_{-,*}\sO_D) \otimes \sO_C(q-1) \otimes L_+^{\vee,\otimes 2q-1} \otimes L_-^{\vee,\otimes 2}[1].
\]
Since \(D \to C\) is of relative dimension \(0\),
\(\mathbf{R}\varphi_{-,*}\sO_D = \varphi_{-,*}\sO_D = \mathcal{D}\), yielding
the isomorphism in \ref{threefolds-conductor-dual.C}. All identifications
are functorial, so the isomorphism respects the action of \(\mathbf{G}_m\)
from \parref{nodal-nu-and-F}, and so it is an isomorphism of graded modules.
\end{proof}

The following relates the cohomology of the structure sheaf of \(S\)
with the cohomology of the sheaf \(\mathcal{F}\) on \(C\):

\begin{Proposition}\label{threefolds-normalize-cohomology}
The cohomology of \(\sO_S\) is given by
\[
\mathrm{H}^i(S,\sO_S) \cong
\begin{dcases}
\mathrm{H}^0(C,\sO_C) & \text{if}\; i = 0, \\
\mathrm{H}^0(C,\mathcal{F}) & \text{if}\; i = 1,\;\text{and} \\
\mathrm{H}^1(C,\mathcal{F})/\mathrm{H}^1(C,\sO_C) & \text{if}\; i = 2.
\end{dcases}
\]
\end{Proposition}

\begin{proof}
Consider the cohomology sequence associated with the exact sequence
\[ 0 \to \sO_S \to \nu_*\sO_{S^\nu} \to \nu_*\sO_{S^\nu}/\sO_S \to 0. \]
Note that \(\varphi_{-,*}\sO_S \cong \sO_C\) by
\parref{threefolds-smooth-cone-situation-pushforward}\ref{threefolds-smooth-cone-situation-pushforward.O},
and, for each \(i = 0,1,2\),
\[
\mathrm{H}^i(S,\nu_*\sO_{S^\nu}/\sO_S) \cong
\mathrm{H}^i(C,\mathcal{F})
\quad\text{and}\quad
\mathrm{H}^i(S^\nu,\sO_{S^\nu}) \cong
\mathrm{H}^i(C,\sO_C)
\]
by \parref{normalize-splitting} and the fact that \(S^\nu \to C\) is a
projective bundle. Thus the long exact sequence yields
\(\mathrm{H}^0(S,\sO_S) \cong \mathrm{H}^0(S^\nu,\sO_{S^\nu}) \cong \mathrm{H}^0(C,\sO_C)\)
and an exact sequence
\[
0 \to
\mathrm{H}^0(C,\mathcal{F}) \xrightarrow{a}
\mathrm{H}^1(S,\sO_S) \xrightarrow{b}
\mathrm{H}^1(S,\nu_*\sO_{S^\nu}) \to
\mathrm{H}^1(C,\mathcal{F}) \to
\mathrm{H}^2(S,\sO_S) \to 0.
\]
The result will follow upon verifying that
\(b \colon \mathrm{H}^1(S,\sO_S) \to \mathrm{H}^1(S,\nu_*\sO_{S^\nu})\)
vanishes. Since \(\varphi_{-,*}\sO_S = \sO_C\), the Leray
spectral sequence gives a short exact sequence
\[
0 \to
\mathrm{H}^1(C,\sO_C) \xrightarrow{c}
\mathrm{H}^1(S,\sO_S) \xrightarrow{d}
\mathrm{H}^0(C,\mathbf{R}^1\varphi_{-,*}\sO_S) \to
0.
\]
Now \parref{normalize-splitting} implies that the composite
\(
d \circ a \colon
\mathrm{H}^0(C,\mathcal{F}) \to
\mathrm{H}^0(C,\mathbf{R}^1\varphi_{-,*}\sO_S)
\)
is a surjection. So exactness of the long sequence means it remains to show
that
\(
b \circ c \colon
\mathrm{H}^1(C,\sO_C) \xrightarrow{c}
\mathrm{H}^1(S,\nu_*\sO_{S^\nu})
\)
vanishes. Pushing down to \(C\) along \(\varphi_-\) and applying
\parref{nodal-nu-and-F} shows that this is
\[ \phi_C \colon \mathrm{H}^1(C,\sO_C) \to \mathrm{H}^1(C,\phi_{C,*}\sO_C) \]
which, by \parref{hypersurfaces-endomorphism}, is the map induced
by the \(q^2\)-power Frobenius up to an automorphism. Thus by
\parref{hypersurfaces-cohomology-zero-frobenius}, this is the zero map.
\end{proof}

\section{Structure of the algebra \texorpdfstring{\(\mathcal{D}\)}{D}}\label{section-D}
The computation of \parref{threefolds-normalize-cohomology} reduces the problem
of computing the cohomology of the structure sheaf of \(S\) to the problem of
computing cohomology of the graded \(\sO_C\)-module \(\mathcal{F}\)
from \parref{nodal-conductors-F}. This will be accomplished, at least in the
case \(q = p\) is a prime, in the following Section, see
\parref{threefolds-cohomology-theorem}. The structure of the module
\(\mathcal{F}\) is accessed via its duality with the graded \(\sO_C\)-algebra
\(\mathcal{D} = \varphi_{-,*}\sO_D\) from \parref{threefolds-conductor-dual}.
The object of this Section is to describe \(\mathcal{D}\), the main result
being \parref{threefolds-cohomology-D}; a summary of its consequences for
\(\mathcal{F}\) is given in \parref{threefolds-cohomology-duality}.

\subsection{Ambient affine bundles}\label{threefolds-D-ambient}
The coordinate ring \(\mathcal{D}\) of \(\varphi_- \colon D \to C\) can be
expressed as a quotient of coordinate rings of affine spaces over \(C\) via the
following commutative diagram of schemes:
\[
\begin{tikzcd}
D \rar[hook] \ar[dr]
& S^\circ \rar[hook] \dar
& \mathbf{B} \dar["\rho"] \ar[dd,bend left=45, "\varphi"] \\
& T^\circ \rar[hook] \ar[dr]
& \mathbf{A}\dar["\pi"] \\
&& C
\end{tikzcd}
\]
in which \(S^\circ \coloneqq S \setminus C_-\) and \(T^\circ\) are as in
\parref{threefolds-cone-situation-T-circ} associated with the Smooth Cone
Situation \((X, x_-)\); the scheme \(\mathbf{A}\) is the fibre
product \(\mathbf{A}_1 \times_C \mathbf{A}_2\) where
\[
\mathbf{A}_1 \coloneqq \PP\mathcal{V}_1 \setminus \PP L_{-,C}
\quad\text{and}\quad
\mathbf{A}_2 \coloneqq \PP\mathcal{V}_2 \setminus \PP(\mathcal{T}_{\PP W}(-1)\rvert_C)
\]
with \(\mathcal{V}_1\) and \(\mathcal{V}_2\) as in
\parref{threefolds-cone-situation-PP}; and, finally,
\[
\mathbf{B} \coloneqq \PP\mathcal{V} \setminus \PP(\mathcal{T}_{\pi_1}(-1,0))\rvert_{\mathbf{A}}
\]
where \(\mathcal{V}\) is as in \parref{threefolds-cone-situation-blowup}.
That \(T^\circ\) is contained in \(\mathbf{A} \subset \PP\) follows from the
description of the points of \(T \setminus T^\circ\) from
\parref{threefolds-smooth-cone-situation-divisors-T} together with the fact
from  \parref{threefolds-cone-situation-C} that the set of lines in \(X\)
through \(x_-\) coincide with the set of lines in \(X\) in
\(\PP\Fr^*(L_-)^\perp\); compare with \parref{threefolds-cone-situation-PP}
for a description of the points in the boundary of the projective bundles. That
\(S^\circ\) is contained in \(\mathbf{B}\) is as observed in
\parref{threefolds-cone-situation-section-over-T-subbundle}.

Consider the sheaf of \(\sO_C\)-algebras
\(\mathcal{A} \coloneqq \pi_*\sO_{\mathbf{A}}\) and
\(\mathcal{B} \coloneqq \varphi_*\sO_{\mathbf{B}}\).
The first observation is that \(\mathcal{A}\) and \(\mathcal{B}\) are
equivariant for a linear algebraic group:

\begin{Lemma}\label{threefolds-D-ambient-equivariant}
The morphisms \(\mathbf{B} \to \mathbf{A} \to C\) are equivariant for the
linear action of
\[
\AutSch(L_- \subset U) \times \mathrm{U}_3(q) =
\AutSch(L_- \subset U) \times \AutSch(W,\beta_W) \subset
\GL(V).
\]
The unipotent radical of \(\AutSch(L_- \subset U)\) acts trivially on
\(\mathbf{A}\).
\end{Lemma}

\begin{proof}
By their construction from the Subquotient Situation as in
\parref{threefolds-cone-situation-PP}, the morphisms
\(\PP\mathcal{V} \to \PP \to C\) are equivariant for the action of
\(\AutSch(L_- \subset U) \times \mathrm{U}_3(q)\) induced by its linear action
on \(V\). Since \(L_{-,C} \subset \mathcal{V}_1\),
\(\mathcal{V}_2 \twoheadrightarrow L_{+,C}\), and
\(\AutSch(L_- \subset U)\) acts via the diagonal action on
\(\PP = \PP\mathcal{V}_1 \times_C \PP\mathcal{V}_2\), the action of its
unipotent radical is trivial. By
\parref{threefolds-cone-situation-subbundle-over-T},
\[
\PP(\mathcal{T}_{\pi_1}(-1,0)) =
\Set{\big((y \in \ell) \mapsto (y_0 \in \ell_0)\big) | \ell = \langle y_0,x_- \rangle}
\subset \PP\mathcal{V}
\]
and so is stable under the action \(\AutSch(L_- \subset U)\). Thus the action
restricts to the complement
\(\PP\mathcal{V}^\circ \coloneqq \PP\mathcal{V} \setminus \PP(\mathcal{T}_{\pi_1}(-1,0))\).
Restricting to \(\mathbf{A}\) gives the result.
\end{proof}

To describe the equivariant structure on \(\mathcal{A}\) and \(\mathcal{B}\),
choose an isomorphism
\[
\AutSch(L_- \subset U) \cong
\Set{
\begin{pmatrix} \lambda^{-1}_- & \epsilon \\ 0 & \lambda_+ \end{pmatrix}
\in \mathbf{GL}(L_- \oplus L_+)}.
\]
The maximal torus acts on \(L_-^{\vee, \otimes a} \otimes L_+^{\otimes b}\)
with weight \((a,b) \in \mathbf{Z}_{\geq 0}^2\) and equips both \(\mathcal{A}\)
and \(\mathcal{B}\) with a bigrading, and the bigraded pieces are described in
the next statement; the action of the unipotent radical is described in
\parref{threefolds-D-unipotent}.

\begin{Lemma}\label{threefolds-cohomology-A-B}
The \(q\)-bic form \(\beta\) induces an equivariant isomorphism of bigraded
\(\sO_C\)-algebras
\[
\mathcal{A} \cong
\Sym^*(\sO_C(-1) \otimes L_-^\vee \oplus \Omega^1_{\PP W}(1)\rvert_C \otimes L_+).
\]
The \(\sO_C\)-algebra \(\mathcal{B}\) is a filtered \(\mathcal{A}\)-algebra
with increasing \(\mathbf{Z}_{\geq 0}\)-filtration
\(\Fil_\bullet\mathcal{B}\) whose associated graded pieces are
equivariantly identified as
\[
\gr_i\mathcal{B} \coloneqq
\Fil_i\mathcal{B}/\Fil_{i-1}\mathcal{B} \cong
\mathcal{A} \otimes (L_-^\vee \otimes L_+)^{\otimes i}
\quad\text{for all}\; i \in \mathbf{Z}_{\geq 0}.
\]
\end{Lemma}

\begin{proof}
By \parref{nodal-splitting}, \(\beta\) induces an orthogonal
decomposition \(V \cong U \oplus W\). This induces splittings
\(\mathcal{V}_1 \cong L_{-,C} \oplus \sO_C(-1)\) and
\(\mathcal{V}_2 \cong L_{+,C} \oplus \mathcal{T}_{\PP W}(-1)\rvert_C\)
of the short exact sequences from \parref{threefolds-cone-situation-PP}.
Then \parref{bundles-affine-subs-split} gives an isomorphism
\[
\mathbf{A}
\cong \mathbf{A}(L_- \otimes \sO_C(1)) \times_C
\mathbf{A}(\mathcal{T}_{\PP W}(-1)\rvert_C \otimes L_+^\vee)
\cong
\mathbf{A}(L_- \otimes \sO_C(1) \oplus \mathcal{T}_{\PP W}(-1)\rvert_C \otimes L_+^\vee),
\]
whence the identification of \(\pi_*\sO_{\mathbf{A}}\).

As for \(\mathbf{B}\), the computation \parref{bundles-affine-subs-algebra}
together with the exact sequence from
\parref{threefolds-cone-situation-subbundle-over-T} gives an isomorphism of
\(\sO_{\mathbf{A}}\)-algebras
\[
\rho_*\sO_{\mathbf{B}} \cong
\colim_n \Sym^n(\mathcal{V}^\vee(0,-1))\rvert_{\mathbf{A}}.
\]
The exact sequence for \(\mathcal{V}\) induces a two step filtration
\[
\sO_T =
\Fil_0(\mathcal{V}^\vee(0,-1)) \subset
\Fil_1(\mathcal{V}^\vee(0,-1)) =
\mathcal{V}^\vee(0,-1).
\]
This induces an \(n+1\) step filtration on \(\Sym^n(\mathcal{V}^\vee(0,-1))\).
The description of the transition maps from \parref{bundles-affine-subs-algebra}
of the colimit show that they are compatible with the filtrations, so
\(\rho_*\sO_{\mathbf{B}}\) inherits a filtration by \(\mathbf{Z}_{\geq 0}\) and
its graded pieces are
\[
\gr_i\rho_*\sO_{\mathbf{B}}
= \Omega^1_{\pi_1}(1,-1)\rvert_{\mathbf{A}}^{\otimes i}
\cong \sO_{\mathbf{A}} \otimes (L_-^\vee \otimes L_+)^{\otimes i}
\quad\text{for all \(i \in \mathbf{Z}_{\geq 0}\)}
\]
upon using that \(\Omega^1_{\pi_1}(1,0)\rvert_{\mathbf{A}} \cong L_{-,\mathbf{A}}^\vee\)
and \(\sO_\PP(0,-1)\rvert_{\mathbf{A}} \cong L_{+,\mathbf{A}}\), see
\parref{bundles-affine-subs-tautological}.
\end{proof}

The bigraded pieces of \(\mathcal{A}\) and \(\mathcal{B}\) can be described
in terms of the sheaves appearing in \parref{threefolds-cohomology-A-B}.
Throughout this Section, the sheaf \(W^\vee_C(-1)\) is viewed as a filtered
\(\sO_C\)-module via the Euler sequence
\[
0 \to
\Omega_{\PP W}^1\rvert_C \to
W_C^\vee(-1) \to
\sO_C \to
0.
\]
This induces a \(d+1\) step filtration on \(\Sym^d(W_C^\vee(-1))\) with
graded pieces
\[
\gr_i\Sym^d(W_C^\vee(-1))
\cong \Sym^{d-i}(\Omega_{\PP W}^1(1)\rvert_C)
\quad\text{for}\; 0 \leq i \leq d.
\]

\begin{Lemma}\label{threefolds-cohomology-bigrading}
The form \(\beta\) induces equivariant isomorphisms of filtered
bundles
\[
\mathcal{B}_{(a,0)} \cong
\mathcal{A}_{(a,0)} \cong
\sO_C(-a),
\;\;
\mathcal{B}_{(0,b)} \cong
\mathcal{A}_{(0,b)} \cong
\Sym^b(\Omega_{\PP W}^1(1)\rvert_C),
\;\;
\mathcal{B}_{(1,1)} \cong W_C^\vee(-1),
\]
and
\(\mathcal{B}_{(a,b)} \cong \Fil_a\Sym^b(W_C^\vee(-1)) \otimes \sO_C(b-a)\)
for all integers \(a,b \geq 0\).
\end{Lemma}

\begin{proof}
Apply \parref{ext-and-PP-one-dimensional} with
\(\mathcal{W}_1 = \sO_C(-1)\), \(\mathcal{W}_2 = \mathcal{T}_{\PP W}(-1)\rvert_C\),
\(\mathcal{W} = W_C^\vee(-1)\), \(L_1 = L_{-,C}\), and \(L_2 = L_{+,C}\).
\end{proof}

A basic, useful point about the algebra structure of \(\mathcal{B}\) is:

\begin{Lemma}\label{threefolds-cohomology-B-multiplication}
The multiplication map
\(\mathcal{B}_{(a,b)} \otimes \mathcal{B}_{(c,0)} \to \mathcal{B}_{(a+c,b)}\)
is injective for all integers \(a,b,c \geq 0\), and is an isomorphism if and
only if \(a \geq b\).
\end{Lemma}

\begin{proof}
Injectivity is because \(\mathcal{B}\) is locally a polynomial algebra and
\(\mathcal{B}_{(c,0)}\) is generated by a monomial.
Surjectivity if and only if \(a \geq b\) now follows from
\parref{threefolds-cohomology-bigrading}, since
\(\Fil_a\Sym^b(W_C^\vee(-1)) = \Sym^b(W_C^\vee(-1))\)
if and only if \(a \geq b\).
\end{proof}

\subsection{Unipotent automorphisms}\label{threefolds-D-unipotent}
The unipotent radical of \(\AutSch(L_- \subset U)\) from
\parref{threefolds-D-ambient-equivariant} gives \(\mathcal{B}\) an
\(\mathcal{A}\)-linear \(\mathbf{G}_a\)-equivariant structure. This is
equivalent to an \(\mathcal{A}\)-comodule structure
\[
\mathcal{B} \to \mathcal{B} \otimes_\kk \kk[\epsilon] \colon
z \mapsto \sum\nolimits_{j = 0}^\infty \partial_j(z) \otimes \epsilon^j
\]
where \(\mathbf{G}_a \coloneqq \Spec(\kk[\epsilon])\), and the \(\partial_j \colon \mathcal{B} \to \mathcal{B}\)
are \(\mathcal{A}\)-linear maps such that a given local section \(z\) lies in the
kernel of all but finitely many \(\partial_j\), \(\partial_0 = \id\), and
\[
\partial_j \circ \partial_k = \binom{j+k}{j}\, \partial_{j+k}
\quad\text{for integers}\;j,k \geq 0.
\]
See \cite[I.7.3, I.7.8, and I.7.12]{Jantzen:RAGS} for details.
The \(\partial_j\) in the setting at hand will now be described in several steps:

\textbf{Step 1}. The \(\mathbf{G}_a\)-equivariant structure is induced by a
representation on \(V\) given by
\[
\AutSch_{\mathrm{uni}}(L_- \subset U) \cong
\Set{\begin{pmatrix} 1 & -\epsilon \\ 0 & 1 \end{pmatrix}} \subset \GL(V).
\]
It will be convenient to describe the dual representation, so let
\(U = \langle u_-,u_+ \rangle\) be a choice of basis compatible with this
embedding into \(\GL(V)\), and let \(U^\vee = \langle u_-^\vee, u_+^\vee \rangle\)
be the dual basis. Then the corresponding comodule
\(V^\vee \to V^\vee \otimes_\kk \kk[\epsilon]\) is given by
\(\partial_1(u_-^\vee) = u_+^\vee\),
\(\partial_1(W^\vee \oplus L_+^\vee) = 0\), and \(\partial_j = 0\)
for all \(j \geq 2\). In particular, the operator \(\partial_1\) restricts
to an isomorphism \(\partial_1 \colon L_-^\vee \to L_+^\vee\).

\textbf{Step 2.}
Consider the filtration on \(V_C^\vee\) determined by the short exact sequence
from \parref{threefolds-cone-situation-PP}: this is the two step filtration
with \(\Fil_0 V_C^\vee = \mathcal{V}_2^\vee\) and \(\Fil_1 V_C^\vee =
V_C^\vee\).
Since \(L_{+,C}^\vee \subset \mathcal{V}_2^\vee\) and
\(\mathcal{V}_1^\vee \twoheadrightarrow L_{-,C}^\vee\), this implies that
\(\partial_1 \colon V_C^\vee \to V_C^\vee\) satisfies
\[
\partial_1(\Fil_i V_C^\vee) \subseteq \Fil_{i-1} V_C^\vee
\quad\text{for all}\;i \in \mathbf{Z}.
\]
Pulling this up along \(\pi \colon \PP \to C\), twisting by the
\(\mathbf{G}_a\)-invariant sheaf \(\sO_\PP(0,-1)\), and using
\parref{threefolds-cone-situation-blowup} to write
\[
\mathcal{V}^\vee(0,-1) =
\mathcal{H}(
\Omega^1_{\pi_2} \hookrightarrow
V_\PP^\vee(0,-1) \twoheadrightarrow
\sO_\PP(1,-1))
\]
shows that \(\partial_1\) passes through homology to yield an
\(\sO_\PP\)-linear map \(\partial_1 \colon \mathcal{V}^\vee(0,-1) \to \mathcal{V}^\vee(0,-1)\).
The filtration on \(V_C\) induces the two step filtration
\(\sO_\PP \subset \mathcal{V}^\vee(0,-1)\) corresponding to the short exact
sequence of \parref{threefolds-cone-situation-subbundle-over-T}, and so
\[
\partial_1(\Fil_i \mathcal{V}^\vee(0,-1)) \subseteq
\Fil_{i-1} \mathcal{V}^\vee(0,-1)
\quad\text{for all}\;i \in \mathbf{Z}.
\]

\textbf{Step 3.}
Symmetric powers induce a \(\mathbf{G}_a\)-equivariant structure on
\(\Sym^n(\mathcal{V}^\vee(0,-1))\) for every \(n \geq 0\). For each
\(0 \leq j \leq n\), the maps
\[
\partial_j \colon
\Sym^n(\mathcal{V}^\vee(0,-1)) \to
\Sym^n(\mathcal{V}^\vee(0,-1))
\]
giving the corresponding comodule structure are obtained by summing all
possible \(j\)-fold tensor powers of \(\partial_1\) on
\(\mathcal{V}^\vee(0,-1)^{\otimes n}\), and passing to the symmetric quotient.
This is compatible with the convolution filtration on the \(n\)-th tensor
product, so \textbf{Step 2} implies
\[
\partial_j(\Fil_i \Sym^n(\mathcal{V}^\vee(0,-1))) \subseteq
\Fil_{i-j}\Sym^n(\mathcal{V}^\vee(0,-1))
\]
for all integers \(i,j,n \geq 0\).

\textbf{Step 4.}
As in the proof of \parref{threefolds-cohomology-A-B}, write
\[ \rho_*\sO_{\PP\mathcal{V}^\circ} = \colim_n \Sym^n(\mathcal{V}^\vee(0,-1)). \]
By \parref{bundles-affine-subs-algebra}, the transition maps in the colimit
are multiplication by a generator of \(\Fil_0\mathcal{V}^\vee(0,-1)\).
By \textbf{Step 2}, this is annihilated by \(\partial_1\), so the description of
the \(\mathbf{G}_a\)-equivariant structure on \(\Sym^n(\mathcal{V}^\vee(0,-1))\)
from \textbf{Step 3} shows that they pass to the colimit and induce a
\(\mathbf{G}_a\)-equivariant structure on \(\rho_*\sO_{\PP\mathcal{V}^\circ}\).
The corresponding \(\sO_\PP\)-linear maps \(\partial_j \colon \rho_*\sO_{\PP\mathcal{V}^\circ} \to \rho_*\sO_{\PP\mathcal{V}^\circ}\)
also satisfy
\[
\partial_j(\Fil_i\rho_*\sO_{\PP\mathcal{V}^\circ}) \subseteq
\Fil_{i-j}\rho_*\sO_{\PP\mathcal{V}^\circ}
\]
for all integers \(i,j \geq 0\). Finally, since
\(\mathcal{B} = \pi_*(\rho_*\sO_{\PP\mathcal{V}^\circ}\rvert_{\mathbf{A}})\),
this gives:

\begin{Lemma}\label{threefolds-D-unipotent-filter}
For each integer \(j \geq 0\), the \(\mathcal{A}\)-module map
\(\partial_j \colon \mathcal{B} \to \mathcal{B}\) satisfies
\[
\partial_j(\Fil_i \mathcal{B}) \subseteq \Fil_{i-j}\mathcal{B}
\quad\text{for each integer}\;i \geq 0
\]
and is of bidegree \((-j,-j)\). \qed
\end{Lemma}

\begin{Lemma}\label{threefolds-D-unipotent-graded}
For each integer \(i \geq 0\), the \(i\)-th associated graded map
\[
\gr_i\partial \colon
\gr_i\mathcal{B} \to
\gr_{i-1}\mathcal{B}
\]
induced by \(\partial \coloneqq \partial_1\) is an isomorphism if \(p \nmid i\)
and is zero otherwise.
\end{Lemma}

\begin{proof}
As in the proof of \parref{threefolds-D-unipotent-filter}, it suffices to
show the analogous conclusion about the restriction to \(\mathbf{A}\) of the
\(i\)-th associated graded map of \(\partial\) acting on
\(\Sym^n(\mathcal{V}^\vee(0,-1))\) with \(n \geq i\). As in
\parref{threefolds-cohomology-A-B}, the relative Euler sequence gives a
canonical isomorphism
\[
\Omega_{\pi_1}(1,-1)\rvert_{\mathbf{A}}
\cong
(L_-^\vee \otimes L_+)_{\mathbf{A}}.
\]
Thus the \(i\)-th associated graded map is identified with a map
\[
\gr_i\partial \colon
(L_-^\vee \otimes L_+)_{\mathbf{A}}^{\otimes i} \to
(L_-^\vee \otimes L_+)_{\mathbf{A}}^{\otimes i-1}.
\]
By \textbf{Step 2} of \parref{threefolds-D-unipotent}, \(\gr_1\partial\)
is in this way identified with the map \(L_-^\vee \otimes L_+ \to \kk\)
adjoint to the isomorphism \(\partial \colon L_-^\vee \to L_+^\vee\) from
\textbf{Step 1}. For \(i \geq 1\), the description of the induced
\(\mathbf{G}_a\)-equivariant structure on symmetric powers from \textbf{Step
3} shows that \(\gr_i\partial\) may be identified with the map
\[
i \cdot \gr_1\partial \colon
(L_-^\vee \otimes L_+)^{\otimes i} \to
(L_-^\vee \otimes L_+)^{\otimes i-1}
\]
where \(\gr_1\partial\) acts on, say, the first tensor factor. This
is an isomorphism whenever \(p \nmid i\) and zero otherwise.
\end{proof}

\subsection{Coordinate ring of \(T^\circ\)}\label{threefolds-D-Tcirc}
Consider the coordinate ring of \(T^\circ\) over \(C\). First, by
\parref{threefolds-cone-situation-equivariance} and the computation of
\parref{qbic-points-automorphisms.N2}, \(\pi_*\sO_{T^\circ}\) remains
equivariant for the subgroup
\[
\Set{
\begin{pmatrix}
\lambda^{-1} & \epsilon \\
0 & \lambda^q
\end{pmatrix} \in \GL_2(L_- \oplus L_+) |
\lambda \in \mathbf{G}_m,
\epsilon \in \boldsymbol{\alpha}_q}
\cong \AutSch(L_- \subset U, \beta_U).
\]
The torus endows each of \(\mathcal{A}\), \(\mathcal{B}\), and
\(\pi_*\sO_{T^\circ}\) with a \(\mathbf{Z}_{\geq 0}\)-grading; the degree \(d\)
component of this grading is related to the previous bigrading by
\[
\mathcal{A}_d = \bigoplus\nolimits_{a + bq = d} \mathcal{A}_{(a,b)}
\quad\text{and}\quad
\mathcal{B}_d = \bigoplus\nolimits_{a + bq = d} \mathcal{B}_{(a,b)}.
\]

Next, by \parref{threefolds-cone-situation-T'}, \(T^\circ \subset \mathbf{A}\)
is the complete intersection cut out by the sections
\begin{align*}
v_1 & \coloneqq
\mathrm{eu}_{\pi_2}^\vee \circ \beta^\vee \circ \mathrm{eu}_{\pi_1}^{(q)}
\colon \sO_\PP(-q,-1) \to \sO_\PP \\
v_2 & \coloneqq
\mathrm{eu}_{\pi_2}^{(q),\vee} \circ \beta \circ \mathrm{eu}_{\pi_1}
\colon \sO_\PP(-1,-q) \to \sO_\PP
\end{align*}
restricted to \(\mathbf{A}\). By \parref{bundles-affine-subs-tautological},
\(\sO_\PP(-1,0)\rvert_{\mathbf{A}} \cong \pi^*\sO_C(-1)\rvert_{\mathbf{A}}\) and
\(\sO_\PP(0,-1)\rvert_{\mathbf{A}} \cong L_{+,\mathbf{A}}\), so pushing along
\(\pi\) gives a presentation of \(\pi_*\sO_{T^\circ}\) as a graded
\(\mathcal{A}\)-algebra:
\[
\pi_*\sO_{T^\circ} \cong \coker\big(
v_1 \oplus v_2 \colon
(\mathcal{A}(-q) \otimes L_+) \oplus
(\mathcal{A}(-1) \otimes L_+^{\otimes q}) \to
\mathcal{A}\big).
\]
The components \(v_1\) and \(v_2\) of the presentation are describe as follows:

\begin{Lemma}\label{threefolds-smooth-cone-present-Tcirc}
The maps \(v_1 \colon \mathcal{A}(-q) \otimes L_+ \to \mathcal{A}\)
and \(v_2 \colon \mathcal{A}(-1) \otimes L_+^{\otimes q} \to \mathcal{A}\) are
the maps of graded \(\mathcal{A}\)-modules induced by
\[
\begin{gathered}
v_1 = (\beta_U^\vee, \delta) \colon
\sO_C(-q) \otimes L_+ \to
\sO_C(-q) \otimes L_-^{\vee,\otimes q} \oplus
\Omega_{\PP W}^1(1)\rvert_C \otimes L_+ \\
v_2 = (\beta_W \circ \mathrm{eu}) \colon
\sO_C(-1) \otimes L_+^{\otimes q} \to
\Fr^*(\Omega^1_{\PP W}(1))\rvert_C \otimes L_+^{\otimes q}
\end{gathered}
\]
where \(\delta \colon \sO_C(-q) \to \Omega^1_{\PP W}(1)\rvert_C\) is the
conormal map of \(C \subset \PP W\).
\end{Lemma}

\begin{proof}
By \parref{bundles-affine-euler-split}, \(v_1\) is the \(\mathcal{A}\)-module
map induced by
\begin{align*}
\sO_C(-q) \otimes L_+
& \to
  \Fr^*\big(L_- \otimes (\sO_C(-1) \otimes L_-^\vee) \oplus \sO_C(-1)\big) \otimes L_+ \\
& \to
\big(L_+^\vee \otimes (\sO_C(-1) \otimes L_-^\vee) \oplus \Omega_{\PP W}(1)\rvert_C\big) \otimes L_+
\end{align*}
where the first map is \(\Fr^*(\tr_{L_-}^\vee, \id_{\sO_C(-1)}) \otimes \id_{L_+}\)
and the second is \(\beta_U^\vee \oplus \beta^\vee_W\). The action of the
second map uses the computation of \parref{hypersurfaces-frobenius-euler} in
the first component and that, since
\(\operatorname{type}(\beta_U) = \mathbf{N}_2\),
\(\beta_U^\vee \colon \Fr^*(L_-) \to L_+^\vee\) is an isomorphism in the second.

Likewise, \(v_2\) is the \(\mathcal{A}\)-module map induced by
\[
\sO_C(-1) \otimes L_+^{\otimes q}
\to \big(L_- \otimes (\sO_C(-1) \otimes L_-^\vee) \oplus \sO_C(-1)\big) \otimes L_+^{\otimes q}
\to \Fr^*(\Omega_{\PP W}(1))\rvert_C \otimes L_+^{\otimes q}
\]
where the first map is \((\tr_{L_-}^\vee, \id_{\sO_C(-1)})\) and the second
map is \(\beta_U \oplus \beta_W\). This time, the map
\(\beta_U \colon L_- \to \Fr^*(L_+)^\vee\) is zero, whence the second summand
vanishes.
\end{proof}

\subsection{Coordinate ring of \(S^\circ\)}\label{threefolds-D-Scirc}
By \parref{threefolds-cone-situation-S-T}, \(S^\circ\) is the
hypersurface in \(\mathbf{B} \times_{\mathbf{A}} T^\circ\)
cut out by the restriction of the section
\[
v_3 \coloneqq
u_3^{-1}\beta_{\mathcal{V}_T}(\mathrm{eu}_\rho^{(q),\vee},\mathrm{eu}_\rho) \colon
\sO_\rho(-q) \otimes \rho^*\sO_T(0,-1)\rvert_{\PP\mathcal{V}_T} \to
\sO_{\PP\mathcal{V}_T}.
\]
Set
\(\mathcal{B}_T \coloneqq
\varphi_*\sO_{\mathbf{B} \times_{\mathbf{A}} T^\circ} \cong
\mathcal{B} \otimes_{\mathcal{A}} \pi_*\sO_{T^\circ}\).
By \parref{bundles-affine-subs-tautological} together with the exact sequences
in \parref{threefolds-cone-situation-subbundle-over-T} and
\parref{threefolds-cone-situation-PP},
\(\sO_\rho(-1)\rvert_{\mathbf{B}} \cong \rho^*(\sO_\PP(0,-1)\rvert_{\mathbf{A}}) \cong L_{+,\mathbf{B}}\).
Thus the coordinate ring of \(S^\circ\) over \(C\) admits a presentation
as a graded \(\mathcal{B}_T\)-algebra
\[
\varphi_*\sO_{S^\circ} \cong
\coker(v_3 \colon \mathcal{B}_T \otimes L_+^{\otimes q+1} \to \mathcal{B}_T).
\]
By \parref{threefolds-cohomology-A-B}, \(\mathcal{B}_T\) carries an increasing
filtration in which \(\Fil_i\mathcal{B}_T\) consists of all local sections
with degree at most \(i\) in the affine fibre coordinate of
\(\PP\mathcal{V}_T^\circ \coloneqq \PP\mathcal{V}_T \setminus \PP(\mathcal{T}_{\pi_1}(-1,0)\rvert_T)\)
over \(T\). Since \(v_3\) is of degree \(q\) over \(T\), this gives the first
statement in the following:

\begin{Lemma}\label{threefolds-D-Scirc-equations}
For each \(i \in \mathbf{Z}_{\geq 0}\),
\(v_3(\Fil_i\mathcal{B}_T \otimes L_+^{\otimes q+1}) \subseteq \Fil_{i+q}\mathcal{B}_T\).
The induced map on associated graded pieces is the isomorphism
\[
\gr_i\mathcal{B}_T \otimes L_+^{\otimes q+1} \cong
\pi_*\sO_{T^\circ} \otimes (L_-^\vee \otimes L_+)^{\otimes i} \otimes L_+^{\otimes q+1}
\to
\pi_*\sO_{T^\circ} \otimes (L_-^\vee \otimes L_+)^{\otimes i+q} \cong
\gr_{i+q}\mathcal{B}_T
\]
given by
\(\beta_U^\vee \colon L_+^{\otimes q+1} \to (L_-^\vee \otimes L_+)^{\otimes q}\).
\end{Lemma}

\begin{proof}
Begin by considering the pushfoward along \(\rho\) of \(v_3\) restricted to
\(\PP\mathcal{V}_T^\circ\). Identifying
\(\sO_\rho(-1)\rvert_{\PP\mathcal{V}_T^\circ}\) with \(\rho^*\sO_T(0,-1)\),
this gives a map
\[
\rho_*\rho^*\sO_T(0,-q-1) \to
\rho_*\sO_{\PP\mathcal{V}_T^\circ} \cong
\colim_n \Sym^n(\mathcal{V}_T^\vee(0,-1)).
\]
By its definition and the description of the Euler section from
\parref{bundles-affine-euler}, this is the
\(\rho_*\sO_{\PP\mathcal{V}_T^\circ}\)-module map induced by
\[
\gamma \colon
\sO_T(0,-q-1) \xrightarrow{\tr^\vee_{\mathcal{V}_T}}
\mathcal{V}_T \otimes \mathcal{V}_T^\vee
\otimes \sO_T(0,-q-1) \xrightarrow{\beta_{\mathcal{V}_T}}
\Fr^*(\mathcal{V}_T^\vee(0,-1)) \otimes
\mathcal{V}_T^\vee(0,-1)
\]
followed by the multiplication map into \(\Sym^{q+1}(\mathcal{V}_T^\vee(0,-1))\).

Next, consider the restriction of \(\gamma\) to \(T^\circ\). As explained in
\parref{bundles-affine-subs-tautological}, there are canonical isomorphisms
\(\mathcal{T}_{\pi_1}(-1,0)\rvert_{T^\circ} \cong L_{-,T^\circ}\) and
\(\sO_T(0,-1)\rvert_{T^\circ} \cong L_{+,T^\circ}\). Via the short exact
sequence \parref{threefolds-cone-situation-section-over-T-subbundle}, this
gives a canonical isomorphism
\((\mathcal{V}_T,\beta_{\mathcal{V}_T})\rvert_{T^\circ} \cong (U_{T^\circ},\beta_U)\)
of \(q\)-bic forms over \(T^\circ\). So upon passing through these
isomorphisms, \(\gamma\) restricts to the map
\[
\gamma\rvert_{T^\circ} \colon
L_{+,T^\circ}^{\otimes q+1} \xrightarrow{\tr^\vee_U}
U_{T^\circ} \otimes U_{T^\circ}^\vee \otimes L_{+,T^\circ}^{\otimes q+1} \xrightarrow{\beta_U}
\Fr^*(U)^\vee_{T^\circ} \otimes U_{T^\circ}^\vee \otimes L_{+,T^\circ}^{\otimes q+1}
\]

Finally, to consider the map on graded pieces, recall that the filtration on
\(\rho_*\sO_{\PP\mathcal{V}_T^\circ}\) is induced by the filtration on
\(\mathcal{V}_T^\vee(0,-1)\) given by the short exact sequence
\parref{threefolds-cone-situation-section-over-T-subbundle}; in particular, the
graded pieces of \(\mathcal{V}_T^\vee(0,-1)\) are
\[
\gr_0\mathcal{V}_T^\vee(0,-1) =
\sO_T
\quad\text{and}\quad
\gr_1\mathcal{V}_T^\vee(0,-1) = \Omega^1_{\pi_1}(1,-1)\rvert_T.
\]
Tensor functors applied to filtered vector bundles also carry filtrations.
In the case at hand, the graded pieces are given by:
\begin{align*}
\gr_i(\mathcal{E})^\vee
& = \gr_{-i}(\mathcal{E}^\vee), \\
\Fr^*(\gr_i(\mathcal{E}))
& = \gr_{iq}(\Fr^*(\mathcal{E})), \\
\gr_i(\mathcal{E}_1 \otimes \mathcal{E}_2)
& = \bigoplus\nolimits_{a+b=i} \gr_a(\mathcal{E}_1) \otimes \gr_b(\mathcal{E}_2),
\end{align*}
for filtered vector bundles \(\mathcal{E}\), \(\mathcal{E}_1\), and \(\mathcal{E}_2\).
Since \(\tr_U^\vee\) preserves the filtration degree and \(\beta_U\) shifts
the filtration degree by \(q\), thanks to the first statement of the Lemma,
the map induced by \(\gamma\rvert_{T^\circ}\) on associated graded pieces is
\[
L_{+,T^\circ}^{\otimes q+1} \to
(L_{+,T^\circ} \otimes L_{+,T^\circ}^\vee \oplus
L_{-,T^\circ} \otimes L_{-,T^\circ}^\vee) \otimes L_{+,T^\circ}^{\otimes q+1} \to
L_{-,T^\circ}^{\vee,\otimes q} \otimes L_{+,T^\circ}^\vee \otimes L_{+,T^\circ}^{\otimes q+1}
\]
where the first map is \((\tr_{L_+}^\vee, \tr_{L_-}^\vee)\), and the second map
is the isomorphism \(\beta_U \colon L_+ \to L_-^{\vee,\otimes q}\) on the first
summand and the zero map on the second summand. Identifying the target as
\((L_-^\vee \otimes L_+)^{\otimes q}_{T^\circ}\) then implies the
second statement of the Lemma.
\end{proof}

Hence the map \(v_3\) becomes a strict morphism of filtered bundles upon
shifting one of the filtrations by \(q\). Recall that a morphism \(f \colon
\mathcal{E}_1 \to \mathcal{E}_2\) between two filtered modules is called
\emph{strict} if
\[
f(\Fil_i\mathcal{E}_1) =
f(\mathcal{E}_1) \cap \Fil_i\mathcal{E}_2
\quad\text{for all}\; i \in \mathbf{Z},
\]
see \citeSP{0123}. In the following statement, a \emph{strict short exact
sequence} of filtered modules is a short exact sequence in which all morphisms
are strict.

\begin{Corollary}\label{threefolds-D-Scirc-gr}
The sheaf \(\varphi_*\sO_{S^\circ}\) is a filtered graded
\(\pi_*\sO_{T^\circ}\)-algebra with
\[
\gr_i\varphi_*\sO_{S^\circ} \cong
\begin{dcases*}
\pi_*\sO_{T^\circ} \otimes (L_-^\vee \otimes L_+)^{\otimes i} & if \(0 \leq i \leq q-1\), and \\
0 & otherwise,
\end{dcases*}
\]
which fits into a strict short exact sequence of graded filtered
\(\pi_*\sO_{T^\circ}\)-modules
\[
0 \to
(\mathcal{B}_T \otimes L_+^{\otimes q+1},\Fil_{\bullet-q}) \xrightarrow{v_3}
(\mathcal{B}_T,\Fil_\bullet) \to
(\varphi_*\sO_{S^\circ},\Fil_\bullet) \to
0.
\]
\end{Corollary}

\begin{proof}
The first statement of \parref{threefolds-D-Scirc-equations} shows that
the short exact sequence in question is one of filtered modules; the
second statement there with \citeSP{0127} shows that it is strict. To identify
the graded pieces, since the sequence is strict, for each integer \(i \geq 0\),
there is a short exact sequence of graded pieces
\[
0 \to
\gr_{i-q}\mathcal{B}_T \otimes L_+ \to
\gr_i\mathcal{B}_T \to
\gr_i\varphi_*\sO_{S^\circ} \to
0.
\]
When \(0 \leq i \leq q-1\), the first term is \(0\) and so
\[
\gr_i\varphi_*\sO_{S^\circ} \cong
\gr_i\mathcal{B}_T \cong
\pi_*\sO_{T^\circ} \otimes (L_-^\vee \otimes L_+)^{\otimes i}.
\]
When \(i \geq q\), the second statement of
\parref{threefolds-D-Scirc-equations} means that the first arrow in the
sequence is an isomorphism, and so \(\gr_i\varphi_*\sO_{S^\circ} = 0\).
\end{proof}

\subsection{Action of \(\boldsymbol{\alpha}_q\)}\label{threefolds-D-unipotent-S}
By \parref{threefolds-cone-situation-equivariance} together with the comments
in \parref{threefolds-D-Tcirc}, the quotient \(\varphi_*\sO_{S^\circ}\) of
\(\mathcal{B}\) remains equivariant for the action of the subgroup
\(\boldsymbol{\alpha}_q \subset \mathbf{G}_a\). Thus the
\(\mathcal{A}\)-comodule structure on \(\mathcal{B}\) from
\parref{threefolds-D-unipotent-S} induces a \(\pi_*\sO_{T^\circ}\)-comodule
structure
\[
\varphi_*\sO_{S^\circ} \to
\varphi_*\sO_{S^\circ} \otimes_\kk \kk[\epsilon]/(\epsilon^q) \colon
z \mapsto \sum\nolimits_{j = 0}^{q-1} \partial_j(z) \otimes \epsilon^j
\]
where \(\partial_j \colon \varphi_*\sO_{S^\circ} \to \varphi_*\sO_{S^\circ}\)
are induced by the \(\partial_j\) on \(\mathcal{B}\). By
\parref{threefolds-D-unipotent-filter}, the \(\partial_j\) satisfy
\[
\partial_j(\Fil_i \varphi_*\sO_{S^\circ}) \subseteq
\varphi_*\sO_{S^\circ}
\quad\text{for each integer}\; i \geq 0,
\]
and are of degree \(-j(q+1)\). A neat summary of the structure provided by
\(\partial_1\) is:

\begin{Corollary}\label{threefolds-cohomology-alpha-p}
The map \(\partial \coloneqq \partial_1\) fits into a complex of
\(\pi_*\sO_{T^\circ}\)-modules
\[
0 \to
\Fil_0\varphi_*\sO_{S^\circ} \to
\varphi_*\sO_{S^\circ} \xrightarrow{\partial}
\varphi_*\sO_{S^\circ} \to
\gr_{q-1}\varphi_*\sO_{S^\circ} \to
.
\]
This is exact when \(q = p\).
\end{Corollary}

\begin{proof}
This follows from \parref{threefolds-D-unipotent-graded} together with
\parref{threefolds-D-Scirc-gr}.
\end{proof}

\subsection{The equation \(v_1\)}\label{threefolds-D-v_1}
The next goal is to describe some of the graded pieces of \(\mathcal{D}\) in
terms of the sheaves appearing in \parref{threefolds-cohomology-bigrading}.
Recall by \parref{threefolds-nodal-compute-conductor} that \(\mathcal{D}\) is
the quotient of \(\varphi_*\sO_{S^\circ}\) by the ideal of all sections of
weight greater than \(\delta = 2q^2 - q -2\). Thus the task is to
describe the low degree pieces of \(\varphi_*\sO_{S^\circ}\) and this demands a
careful study of the equations \(v_1\), \(v_2\), and \(v_3\) appearing in
\parref{threefolds-D-Tcirc} and \parref{threefolds-D-Scirc}. The equation
\(v_1\) is the simplest: by its description from \parref{threefolds-smooth-cone-present-Tcirc},
it is locally of the form \(t^q + \delta\), where \(t\) is a local fibre
coordinate of \(\mathbf{A}_1\) over \(C\) and \(\delta\) is a linear form on
\(\mathbf{A}_2\) over \(C\). Set
\[
\mathcal{B}' \coloneqq
\coker(v_1 \colon \mathcal{B} \otimes \sO_C(-q) \otimes L_+ \to \mathcal{B}).
\]
Using \(v_1\) to locally eliminate \(t^q\) makes it possible to identify
the components of \(\mathcal{B}'\) of weight at most \(q^2-1\):

\begin{Lemma}\label{threefolds-cohomology-grading-easy}
The form \(\beta\) induces equivariant isomorphisms of filtered sheaves
\[
\mathcal{B}_a' \cong \sO_C(-a), \quad
\mathcal{B}_q' \cong \Omega_{\PP W}^1(1)\rvert_C, \quad
\mathcal{B}_{q+1}' \cong W_C^\vee(-1),
\]
and
\(\mathcal{B}_{bq + a}' \cong
\Sym^b(W_C^\vee(-1)) \otimes \sO_C(b-a)\)
for all \(0 \leq b \leq a \leq q - 1\).
\end{Lemma}

\begin{proof}
Since \(v_1\) is of degree \(q\), \(\mathcal{B}_a' = \mathcal{B}_a \cong \sO_C(-a)\)
for all \(0 \leq a \leq q - 1\), the second identification is due to
\parref{threefolds-cohomology-bigrading}. To describe the higher degree
components, observe that by \parref{threefolds-smooth-cone-present-Tcirc}, the
composite
\[
\sO_C(-q) \otimes L_+ \xrightarrow{v_1}
\sO_C(-q) \otimes L_-^{\vee,\otimes q}
\oplus
\Omega^1_{\PP W}(1)\rvert_C \otimes L_+
\xrightarrow{\pr_1}
\sO_C(-q) \otimes L_-^{\vee,\otimes q}
\]
is the isomorphism induced by \(\beta_U\). Since the middle term is
\(\mathcal{B}_q\) by \parref{threefolds-cohomology-bigrading}, this implies
that there is an isomorphism of \(\sO_C\)-modules
\[
\mathcal{B}_q' = \coker(v_1 \colon \sO_C(-q) \otimes L_+ \to \mathcal{B}_q)
\cong \Omega_{\PP W}^1(1)\rvert_C.
\]
Let \(0 \leq b \leq a \leq q-1\). Consider the composition of
\[
\xi \colon
\mathcal{B}_{(b-1)q+a} \otimes \sO_C(-q) \otimes L_+ \xrightarrow{v_1}
\mathcal{B}_{bq+a} \twoheadrightarrow
\mathcal{B}_{(bq+a,0)} \oplus
\mathcal{B}_{((b-1)q+a,1)} \oplus \cdots \oplus
\mathcal{B}_{(q + a,b-1)}
\]
the map \(v_1\) together with the projection of
\(\mathcal{B}_{bq+a} = \bigoplus_{i = 0}^b \mathcal{B}_{((b-i)q+a,i)}\)
onto the first \(b\) pieces of its bigraded decomposition. I claim that
\(\xi\) is an isomorphism. Combined with
\parref{threefolds-cohomology-bigrading}, this will complete the proof as it
gives the isomorphism in
\[
\mathcal{B}'_{bq+a} =
\coker(v_1 \colon \mathcal{B}_{(b-1)q+a}\otimes \sO_C(-q) \otimes L_+ \to \mathcal{B}_{bq+a})
\cong \mathcal{B}_{(a,b)}.
\]
Since both the projection and \(v_1\) are compatible with the natural filtrations,
this is gives an isomorphism of filtered \(\sO_C\)-modules.

To see that \(\xi\) is an isomorphism, consider the bigraded decomposition
\[
\mathcal{B}_{(b-1)q+a} =
\mathcal{B}_{((b-1)q+a,0)} \oplus \mathcal{B}_{((b-2)q+a,1)} \oplus
\cdots \oplus \mathcal{B}_{(a,b-1)}.
\]
Let \(0 \leq i \leq b-1\).
Restricting \(\xi\) to the component
\(\mathcal{B}_{((b-1-i)q+a,i)} \otimes \sO_C(-q) \otimes L_+\)
gives a map
\[
\xi_i \colon
\mathcal{B}_{((b-1-i)q+a,i)} \otimes \sO_C(-q) \otimes L_+ \to
\mathcal{B}_{((b-i)q+a,i)} \oplus \mathcal{B}_{((b-1-i)q+a,i+1)}
\]
where, by the computation of \parref{threefolds-smooth-cone-present-Tcirc},
the map to the first summand is multiplication by a generator of
\(\mathcal{B}_{(q,0)}\), and the map to the second summand is multiplication by
the subbundle \(\delta \colon \sO_C(-q) \otimes L_+ \to \mathcal{B}_{(0,1)}\).
Since \((b-1-i)q+a \geq i\) for each \(0 \leq i \leq b-1\), \parref{threefolds-cohomology-B-multiplication}
shows that post-composing \(\xi_i\) with projection onto the second factor
yields an isomorphism. With the ordering of the bigraded decompositions of
\(\mathcal{B}_{(b-1)q+a}\) and \(\mathcal{B}_{bq+a}\) above, \(\xi\) is
upper triangular with isomorphisms on the diagonal, and so is itself an
isomorphism.
\end{proof}

The identification \(\mathcal{B}'_{q+1} \cong W_C^\vee(-1)\) of
\parref{threefolds-cohomology-grading-easy} shows that the subalgebra
generated in degree \(q+1\) is strikingly simple. Higher weight pieces of
\(\mathcal{B}'\) can be partially accessed by comparing with this subalgebra.
In particular, the pieces \(\mathcal{B}'_{dq+q-1}\) with \(q \leq d \leq 2q-2\)
turn out to be rather simple and useful. In the following, let
\[ \iota \colon \Sym^*(\mathcal{B}'_{q+1}) \to \mathcal{B}' \]
be the inclusion of the subalgebra generated in degree \(q+1\).

\begin{Lemma}\label{threefolds-cohomology-grading-hard}
For each \(q \leq d \leq 2q-2\), the multiplication map factors as
\[
\mathrm{mult} \colon
\mathcal{B}_{dq + q - 1}' \otimes \mathcal{B}_{d - q + 1}'
\xrightarrow{\mu_d'}
\Sym^d(\mathcal{B}'_{q+1}) \xrightarrow{\iota}
\mathcal{B}_{d(q+1)}'
\]
where \(\mu_d'\) is a strict morphism of filtered bundles, and
\[
\coker(\mu_d') \cong
\Sym^d(\mathcal{B}'_{q+1})/
(\Sym^{d-1}(\mathcal{B}'_{q+1})(-q-1) + \Fil_{q-1}\Sym^d(\mathcal{B}'_{q+1})).
\]
\end{Lemma}

\begin{proof}
To show that the multiplication map factors on \(\mathcal{B}'\) as claimed,
it suffices to show that the multiplication map up on \(\mathcal{B}\) factors as
\[
\mathrm{mult} \colon
\mathcal{B}_{dq+q-1} \otimes \mathcal{B}_{d-q+1} \xrightarrow{\mu_d}
\Sym^d(\mathcal{B}_{q+1}) \to
\mathcal{B}_{d(q+1)}.
\]
Multiplication respects the bigrading of \(\mathcal{B}\). Noting that
\(q \leq d \leq 2q-2\), it decomposes into a sum of maps
\begin{align*}
(0 \to \mathcal{B}_{(d-q,d+1)}) \oplus
\bigoplus\nolimits_{i=0}^d
\Big(\mathcal{B}_{(iq+q-1,d-i)} \otimes \mathcal{B}_{(d-q+1,0)} \to
\mathcal{B}_{(iq + d, d-i)}\Big).
\end{align*}
Since \(\mathcal{B}_{q+1} = \mathcal{B}_{(q+1,0)} \oplus \mathcal{B}_{(1,1)}\)
by \parref{threefolds-cohomology-bigrading} and \parref{threefolds-D-Tcirc},
the targets of the latter sum add up to \(\Sym^d(\mathcal{B}_{q+1})\). Set
\(\mu_d\) to be the latter sum of the nonzero maps. Since multiplication
is strict, so is \(\mu_d\). Since \(\mathcal{B}'\) carries the quotient
filtration induced from \(\mathcal{B}\), this implies that the induced
map \(\mu'_d\) is also strict, see \citeSP{0124}.

By the construction of \(\mu_d\), its cokernel is identified as
\begin{align*}
\coker(\mu_d)
& = \bigoplus\nolimits_{i = 0}^d
\coker\big(
\mathcal{B}_{(iq+q-1,d-i)} \otimes \mathcal{B}_{(d-q+1,0)} \to
\mathcal{B}_{(iq+d,d-i)}\big) \\
& = \coker(\mathcal{B}_{(q-1,d)} \otimes \mathcal{B}_{(d-q+1,0)} \to \mathcal{B}_{(d,d)}) \\
& \cong \Sym^d(\mathcal{B}'_{q+1})/\Fil_{q-1}\Sym^d(\mathcal{B}'_{q+1})
\end{align*}
where the second equality is due to
\parref{threefolds-cohomology-B-multiplication} since \(2q-1 > d\), and
the third isomorphism is due to \parref{threefolds-cohomology-bigrading}.
To relate this with \(\coker(\mu_d')\), since the equation \(v_1\) is of
degree \(q\),
\begin{align*}
\mathcal{B}_{dq+q-1}' \otimes \mathcal{B}_{d-q+1}'
& = \coker\big(
\mathcal{B}_{(d-1)q+q-1} \otimes \sO_C(-q) \otimes L_+
\xrightarrow{v_1}
\mathcal{B}_{dq+q-1}\big)
\otimes \mathcal{B}_{d-q+1}, \\
\Sym^d(\mathcal{B}'_{q+1})
& = \coker\big(
\Sym^{d-1}(\mathcal{B}_{q+1}) \otimes \mathcal{B}_1 \otimes \sO_C(-q) \otimes L_+
\xrightarrow{v_1}
\Sym^d(\mathcal{B}_{q+1})\big).
\end{align*}
Since \(v_1\) is a map of algebras, the presentations above fit into an exact
commutative diagram with the maps \(\mu_{d-1}\), \(\mu_d\), and \(\mu_d'\)
to yield a short exact sequence
\[
0 \to
\coker(\mu_{d-1}) \otimes \mathcal{B}_1 \otimes \sO_C(-q) \otimes L_+ \xrightarrow{v_1}
\coker(\mu_d) \to
\coker(\mu_d') \to
0.
\]
Since \(v_1\) respects filtrations, this gives the desired description of
\(\coker(\mu_d')\).
\end{proof}

The next goal is to descend this structure along the quotient
\(\mathcal{B}' \to \mathcal{D}\) to identify \(\mathcal{D}_{dq+q-1}\) for
\(q \leq d \leq 2q-3\). This is achieved in \parref{threefolds-cohomology-D}.
Begin with a simple linear algebra fact. In the following, for a
vector bundle \(\mathcal{E}\) and an integer \(d \geq q\), write
\begin{align*}
(\Sym^d \otimes \Fr^*)(\mathcal{E})
& \coloneqq \Sym^d(\mathcal{E}) \otimes \Fr^*(\mathcal{E}), \\
(\Sym^d/\Sym^{d-q} \otimes \Fr^*)(\mathcal{E})
& \coloneqq \Sym^d(\mathcal{E})/(\Sym^{d-q}(\mathcal{E}) \otimes \Fr^*(\mathcal{E}))
\end{align*}
where \(\Sym^{d-q} \otimes \Fr^*\) includes into \(\Sym^d\) via multiplication.

\begin{Lemma}\label{threefolds-cohomology-surjective-map-vector-space}
Let \(V\) be a finite dimensional vector space with a two step filtration
\(\Fil_0 V \subseteq \Fil_1 V = V\). If \(\gr_1 V\) is one-dimensional,
then, for all integers \(d \geq q\), the map
\[ (\Sym^{d-q} \otimes \Fr^*)(V) \to \Sym^d(V)/\Fil_{q-1}\Sym^d(V) \]
induced by multiplication is surjective, and it induces a canonical
isomorphism
\[
\Fil_{q-1}\Sym^d(V)/\Fil_{q-1}(\Sym^{d-q}(V) \otimes \Fr^*(V)) \cong
(\Sym^d/\Sym^{d-q} \otimes \Fr^*)(V).
\]
\end{Lemma}

\begin{proof}
Choose a basis \(\Fil_0V = \langle v_1,\ldots,v_n \rangle\) and extend it to a
basis of \(V\) by taking \(w \in V\) mapping to a basis of \(\gr_1 V\). Then
\(\Sym^d(V)\) has a basis given by the monomials of degree \(d\) in the
\(v_1,\ldots,v_n,w\), and the \((q-1)\)-st piece of the induced filtration is
\[
\Fil_{q-1}\Sym^d(V) =
\langle v_1^{i_1} \cdots v_n^{i_n} w^j \mid
i_1 + \cdots i_n + j = d\; \text{and}\; j \leq q-1 \rangle.
\]
The quotient therefore has a basis of the form
\[
\Sym^d(V)/\Fil_{q-1}\Sym^d(V) =
\langle v_1^{i_1} \cdots v_n^{i_n} w^j \mid
i_1 + \cdots + i_n + j = d\;\text{and}\; j \geq q \rangle.
\]
This shows that every element of \(\Sym^d(V)/\Fil_{q-1}\Sym^d(V)\) can be
written as a product of \(w^q\) with an element of \(\Sym^{d-q}(V)\), so the
multiplication map is surjective. Since
\[
\Fil_{q-1}(\Sym^{d-q} \otimes \Fr^*)(V) =
\ker\big((\Sym^{d-q} \otimes \Fr^*)(V) \to \Sym^d(V)/\Fil_{q-1}\Sym^d(V)\big)
\]
the second statement now follows from the Five Lemma.
\end{proof}

Since the ideal of \(\mathcal{D}\) in \(\mathcal{B}'\) contains the
generators \(v_2\) and \(v_3\) in degrees \(q^2\) and \(q(q+1)\), see
\parref{threefolds-D-Tcirc} and \parref{threefolds-D-Scirc}, the map
\(\Sym^*(\mathcal{B}'_{q+1})\) onto the subalgebra of \(\mathcal{D}\)
generated by \(\mathcal{D}_{q+1}\) is not injective. The following
uses the filtration to eliminate the equation \(v_3\) via
\parref{threefolds-D-Scirc-equations} to give a simple relationship between
\(\mathcal{D}_{dq+q-1}\) and \(\Sym^d(\mathcal{B}'_{q+1})\) for \(q \leq d \leq 2q-3\):

\begin{Lemma}\label{threefolds-cohomology-D-image}
Let \(q \leq d \leq 2q - 3\) and let
\[
\mathcal{D}_{q+1}^{d} \coloneqq
\image\big(\Sym^d(\mathcal{B}'_{q+1}) \to \mathcal{D}_{d(q+1)}\big).
\]
Then the following hold:
\begin{enumerate}
\item\label{threefolds-cohomology-D-image.isomorphism}
Multiplication
\(\mathcal{D}_{dq+q-1} \otimes \mathcal{D}_{d-q+1} \to \mathcal{D}_{d(q+1)}\)
is an isomorphism onto \(\mathcal{D}_{q+1}^{d}\).
\item\label{threefolds-cohomology-D-image.filter}
The natural map
\(\Fil_{q-1}\Sym^d(\mathcal{B}'_{q+1}) \to \mathcal{D}_{q+1}^{d}\) is strict
surjective and its kernel is
\[
\image\big(
v_2 \colon \Sym^{d-q}(\mathcal{B}'_{q+1}) \otimes
L_+^{\otimes q}(-q-1) \to
\Fil_{q-1}\Sym^d(\mathcal{B}'_{q+1})
\big).
\]
\item\label{threefolds-cohomology-D-image.surjection}
The surjection \(\Fil_{q-1}\Sym^d(\mathcal{B}'_{q+1}) \to \mathcal{D}_{q+1}^d\)
induces a strict surjection
\[
\mathcal{D}_{q+1}^d \to
(\Sym^d/\Sym^{d-q} \otimes \Fr^*)(\mathcal{B}'_{q+1}).
\]
\end{enumerate}
\end{Lemma}

\begin{proof}
The factorization of the multiplication map on \(\mathcal{B}'\) from
\parref{threefolds-cohomology-grading-hard} implies that the multiplication
map \(\mathcal{D}_{dq+q-1} \otimes \mathcal{D}_{d-q+1} \to \mathcal{D}_{d(q+1)}\)
factors through \(\mathcal{D}_{q+1}^d\). Thus there is a commutative diagram
\[
\begin{tikzcd}[column sep=4em]
\Fil_{q-1}(\mathcal{B}_{dq+q-1}' \otimes \mathcal{B}_{d-q+1}') \rar["\cong","\Fil_{q-1}\mu_d'"'] \dar[two heads]
& \Fil_{q-1}\Sym^d(\mathcal{B}'_{q+1}) \dar[two heads] \\
\mathcal{D}_{dq+q-1} \otimes \mathcal{D}_{d-q+1} \rar[hook,"\mathrm{mult}"]
& \mathcal{D}_{q+1}^d
\end{tikzcd}
\]
in which
\begin{itemize}
\item the vertical maps are induced by the quotient map
\(\mathcal{B}' \to \mathcal{D}\) and are strict surjective since the filtration
on \(\mathcal{D}\) has \(q\) steps by \parref{threefolds-D-Scirc-gr};
\item the top map is an isomorphism by the computation
of \(\coker(\mu_d')\) in \parref{threefolds-cohomology-grading-hard}; and
\item the multiplication map below is injective since \(S\) is integral: it is
irreducible by \parref{threefolds-nodal-normalization} and reduced as it is
Cohen--Macaulay and generically smooth, see \parref{nodal-fano-scheme}.
\end{itemize}
Commutativity of the diagram implies that the multiplication map on
the bottom is an isomorphism, establishing \ref{threefolds-cohomology-D-image.isomorphism}.
This also shows the first statement of \ref{threefolds-cohomology-D-image.filter}.

To compute the kernel in \ref{threefolds-cohomology-D-image.filter}, use the
horizontal isomorphisms and the identifications
\(\mathcal{B}_{d-q+1}' \cong \mathcal{D}_{d-q+1} \cong \sO_C(-d+q-1)\)
from \parref{threefolds-cohomology-grading-easy} to obtain
\[
\ker\big(\Fil_{q-1}\Sym^d(\mathcal{B}'_{q+1}) \to \mathcal{D}_{q+1}^d\big) \cong
\ker\big(
\Fil_{q-1}\mathcal{B}_{dq+q-1}' \to
\mathcal{D}_{dq+q-1}
\big) \otimes \sO_C(-d+q-1).
\]
Since the kernel of
\(\mathcal{B}' \to \mathcal{D}\) is, in low degrees, the ideal
generated by \(v_2\) and \(v_3\), and the latter lies strictly in
\(\Fil_q\mathcal{B}'\) by \parref{threefolds-D-Scirc-equations}, it follows that
\[
\ker\big(\Fil_{q-1}\mathcal{B}_{dq+q-1}' \to \mathcal{D}_{dq+q-1}\big) =
\image\big(\mathcal{B}_{(d-q)q+q-1}' \otimes L_{+,C}^{\otimes q}(-1) \xrightarrow{v_2}
\Fil_{q-1}\mathcal{B}_{dq+q-1}'\big)
\]
where \(\Fil_{q-1}\mathcal{B}_{(d-q)q+q-1}' = \mathcal{B}_{(d-q)q+q-1}'\)
since \((d-q)q+q-1 < q(q+1)\). Putting these together with the isomorphism
\[
\mathcal{B}_{(d-q)q+q-1} \cong
\Sym^{d-q}(\mathcal{B}'_{q+1}) \otimes \sO_C(d-2q+1)
\]
from \parref{threefolds-cohomology-grading-easy} now identifies the kernel in
\ref{threefolds-cohomology-D-image.filter}.

Finally, \parref{threefolds-smooth-cone-present-Tcirc} shows that \(v_2\) factors
through \(\Fr^*(\mathcal{B}_q') \subset \mathcal{B}_{q^2}'\), and so the morphism
giving the kernel in \ref{threefolds-cohomology-D-image.filter} factors as
\[
\begin{tikzcd}
\Sym^{d-q}(\mathcal{B}'_{q+1}) \otimes L_+^{\otimes q}(-q-1) \rar["v_2"'] \dar[hook,"\beta_W \circ \mathrm{eu}_{\PP W}"']  &
\Fil_{q-1}\Sym^d(\mathcal{B}'_{q+1}) \\
\Sym^{d-q}(\mathcal{B}'_{q+1}) \otimes \Fr^*(\mathcal{B}_q' \otimes \mathcal{B}_1') \rar[equal]
& \Fil_{q-1}(\Sym^{d-q} \otimes \Fr^*)(\mathcal{B}'_{q+1})\punct{.} \uar[hook]
\end{tikzcd}
\]
In other words, the kernel of the surjection
\(\Fil_{q-1}\Sym^d(\mathcal{B}'_{q+1}) \to \mathcal{D}_{q+1}^d\) is contained
in the subbundle \(\Fil_{q-1}(\Sym^{d-q} \otimes \Fr^*)(\mathcal{B}'_{q+1})\).
Thus there is an induced surjection
\[
\mathcal{D}_{q+1}^d \to
\Fil_{q-1}\Sym^d(\mathcal{B}'_{q+1})/\Fil_{q-1}(\Sym^{d-q} \otimes \Fr^*)(\mathcal{B}'_{q+1}).
\]
This is strict since both sheaves carry the quotient filtration from
\(\Fil_{q-1}\Sym^d(\mathcal{B}'_{q+1})\), see \citeSP{0124}. Since the
quotient on the right is canonically isomorphic to
\((\Sym^d/\Sym^{d-q} \otimes \Fr^*)(\mathcal{B}'_{q+1})\) by
\parref{threefolds-cohomology-surjective-map-vector-space},
this completes the proof of \ref{threefolds-cohomology-D-image.surjection}.
\end{proof}

The main result regarding the structure of the pieces of \(\mathcal{D}\)
is the following:

\begin{Proposition}\label{threefolds-cohomology-D}
There are equivariant isomorphisms of filtered bundles
\[
\mathcal{D}_a \cong \sO_C(-a),
\quad
\mathcal{D}_q \cong \Omega^1_{\PP W}(1)\rvert_C, \quad
\mathcal{D}_{q+1} \cong W_C^\vee(-1),
\]
and
\(\mathcal{D}_{bq + a} \cong
\Sym^b(W_C^\vee(-1)) \otimes \mathcal{D}_{a - b}\)
for each \(0 \leq b \leq a \leq q - 1\).
There are equivariant strict surjections of filtered bundles
\[
\bar{\mu}_d \colon
\mathcal{D}_{dq+q-1} \otimes \mathcal{D}_{d - q + 1} \to
(\Sym^d/\Sym^{d-q} \otimes \Fr^*)(W_C^\vee(-1))
\]
for each \(q \leq d \leq 2q-3\), with
\(\ker(\bar{\mu}_d) \cong
\Sym^{d-q}(W_C^\vee(-1)) \otimes \sO_C(-2q+1)\).
\end{Proposition}

\begin{proof}
By construction, \(\mathcal{D}\) is a quotient of \(\mathcal{B}'\) by an ideal
generated in degrees at least \(q^2\), so \(\mathcal{D}_d = \mathcal{B}'_d\)
for all \(0 \leq d \leq q^2-1\) and the first series of statements follows
from \parref{threefolds-cohomology-grading-easy}.

For the latter statements, let \(q \leq d \leq 2q-3\) and let \(\bar{\mu}_d\)
be the composition
\[
\bar{\mu}_d \colon
\mathcal{D}_{dq+q-1} \otimes \mathcal{D}_{d-q+1} \cong
\mathcal{D}_{q+1}^d \twoheadrightarrow
(\Sym^d/\Sym^{d-q} \otimes \Fr^*)(\mathcal{B}'_{q+1})
\]
of the isomorphism from
\parref{threefolds-cohomology-D-image}\ref{threefolds-cohomology-D-image.isomorphism}
with the surjection of
\parref{threefolds-cohomology-D-image}\ref{threefolds-cohomology-D-image.surjection}.
Thus this is strict and there is a commutative square
\[
\begin{tikzcd}
\Fil_{q-1}\Sym^d(\mathcal{B}'_{q+1}) \dar[equal] \rar[two heads]
& \mathcal{D}_{dq+q-1} \otimes \mathcal{D}_{d-q+1} \dar[two heads,"\bar{\mu}_d"] \\
\Fil_{q-1}\Sym^d(\mathcal{B}'_{q+1}) \rar[two heads]
& (\Sym^d/\Sym^{d-q} \otimes \Fr^*)(\mathcal{B}'_{q+1})\punct{.}
\end{tikzcd}
\]
The kernels of the maps in the diagram fit into a short exact sequence
\[
0 \to
\Sym^{d-q}(\mathcal{B}'_{q+1}) \otimes \sO_C(-q-1) \xrightarrow{v_2}
\Fil_{q-1}(\Sym^{d-q} \otimes \Fr^*)(\mathcal{B}'_{q+1}) \to
\ker(\bar{\mu}_d) \to
0
\]
where the first term is identified by
\parref{threefolds-cohomology-D-image}\ref{threefolds-cohomology-D-image.filter}.
Since
\[
\Fil_{q-1}(\Sym^{d-q} \otimes \Fr^*)(\mathcal{B}'_{q+1}) =
\Sym^{d-q}(\mathcal{B}'_{q+1}) \otimes \Fr^*\Omega_{\PP W}^1\rvert_C,
\]
as is explained in the proof of \parref{threefolds-D-Scirc-equations},
the map \(v_2\) can be identified using \parref{threefolds-smooth-cone-present-Tcirc}
as the map
\(\beta_W \circ \mathrm{eu}_{\PP W} \colon
\sO_C(-q-1) \to \Fr^*\Omega_{\PP W}^1\rvert_C\).
It follows from \parref{hypersurfaces-frobenius-euler} and the diagram of
\parref{hypersurfaces-embedded-tangent-sheaf} that the cokernel of this is
identified via \(\beta\) with \(\mathcal{T}_C(-q-1) \cong \sO_C(-2q+1)\).
This shows that
\(\ker(\bar{\mu}_d) \cong \Sym^{d-q}(\mathcal{B}'_{q+1}) \otimes \sO_C(-2q+1)\).
The statement follows upon identifying \(\mathcal{B}'_{q+1} \cong W_C^\vee(-1)\)
as in \parref{threefolds-cohomology-grading-easy}.
\end{proof}

\subsection{The algebra \(\mathcal{D}^\nu\)}
Consider the short exact sequence of graded \(\sO_C\)-modules from
\parref{nodal-conductors-F}:
\[
0 \to
\mathcal{D} \xrightarrow{\nu^\#}
\mathcal{D}^\nu \to
\mathcal{F} \to
0.
\]
The algebra \(\mathcal{D}^\nu\) and low degree pieces of the map \(\nu^\#\) are
simple to describe and are given in the next two statements. These descriptions
will be used to explicitly compute the global sections of low degree pieces
of \(\mathcal{F}\) in \parref{threefolds-cohomology-upper-bound}.

\begin{Lemma}\label{threefolds-cohomology-D-nu}
\(\mathcal{D}^\nu \cong
\phi_{C,*}\big(\bigoplus\nolimits_{i = 0}^\delta (\mathcal{T}_C(-1) \otimes L_-^\vee)^{\otimes i}\big)\)
as graded \(\sO_C\)-modules.
\end{Lemma}

\begin{proof}
Recall from \parref{threefolds-nodal-construction-of-blowup} that
\(S^\nu = \PP(L_{-,C} \oplus \mathcal{T}_C(-1))\) over \(C\), and that
\(C_+^\nu\) is the subbundle \(\PP(\mathcal{T}_C(-1))\).
Thus by \parref{threefolds-nodal-compute-conductor}, \(D^\nu\) is the \(\delta\)-order
neighbourhood of the zero section in the affine bundle over \(C\) given by
\[
S^\nu \setminus C_-^\nu
= \PP(L_{-,C} \oplus \mathcal{T}_C(-1)) \setminus \PP L_{-,C}
\cong \mathbf{A}(\Omega_C^1(1) \otimes L_-),
\]
using the identification of \parref{bundles-affine-subs-split}. Whence
\(\tilde\varphi_{+,*}\sO_{D^\nu} \cong \bigoplus_{i=0}^\delta (\mathcal{T}_C(-1) \otimes L_-^\vee)^{\otimes i}\)
and the result follows from \parref{nodal-conductors-basics}\ref{nodal-conductors-basics.diagram}.
\end{proof}

By \parref{threefolds-cohomology-D} and \parref{threefolds-cohomology-D-nu},
the degree \(1\) component of \(\nu^\#\) is an \(\sO_C\)-module map
\[ \nu^\#_1 \colon \sO_C(-1) \to \phi_{C,*}(\mathcal{T}_C(-1)). \]
Its adjoint is a map between line bundles on \(C\); the following identifies
that map:

\begin{Lemma}\label{threefolds-cohomology-nu}
The adjoint to the map \(\nu^\#_1 \colon \sO_C(-1) \to \phi_{C,*}(\mathcal{T}_C(-1))\)
fits into a commutative diagram
\[
\begin{tikzcd}
\phi_C^*(\sO_C(-1)) \rar \dar["\cong"]
& \mathcal{T}_C(-1) \\
\Fr^*(\mathcal{N}_{C/\PP W}(-1))^\vee \rar["\phi_C"]
& \mathcal{T}_C^{\mathrm{e}} \uar[two heads]
\end{tikzcd}
\]
where
\(\phi_C \colon \Fr^*(\mathcal{N}_{C/\PP W}(-1))^\vee \to \mathcal{T}_C^{\mathrm{e}}\)
is the map induced by \(\beta_W^{-1} \circ \delta^{(q)}\) as in \parref{hypersurfaces-tangent-form-properties}.
\end{Lemma}

\begin{proof}
Write \(S^{\nu,\circ} \coloneqq S^\nu \setminus C_-^\nu\) and consider the
diagram
\[
\begin{tikzcd}
S^{\nu,\circ} \rar["\nu"'] \dar["\tilde{\varphi}_+"']
& S^\circ \rar["\psi"'] \dar["\varphi_-"]
& \mathbf{A}_1 \ar[dl,"\pi_1"] \\
C \rar["\phi_C"]
& C
\end{tikzcd}
\]
where notation is as in \parref{threefolds-D-ambient} and the square
commutes by \parref{nodal-nu-and-F}. Since the degree \(1\)
generator of \(\mathcal{D}\) is the fibre coordinate of
\(\mathbf{A}_1 = \PP\mathcal{V}_1 \setminus \PP L_{-,C}\) over \(C\), the
desired map \(\phi_C^*(\sO_C(-1)) \to \mathcal{T}_C(-1)\) arises as a map
\[
\nu^*\psi^*(\sO_{\PP\mathcal{V}_1}(-1)\rvert_{\mathbf{A}_1}) \to
\sO_{\tilde\varphi_+}(-1)\rvert_{S^{\nu,\circ}}
\]
upon identifying
\(\sO_{\PP\mathcal{V}_1}(-1)\rvert_{\mathbf{A}_1} \cong \pi_1^*\sO_C(-1)\) and
\(\sO_{\tilde\varphi_+}(-1)\rvert_{S^{\nu,\circ}} \cong \tilde\varphi_+^*(\mathcal{T}_C(-1))\)
via the relative Euler sequences, as in
\parref{bundles-affine-subs-tautological}. Note both identifications are
induced by the linear projection \(V \to W\).

To identify this map, note that the discussion of
\parref{threefolds-cone-situation-plane-bundle} implies
\[
\PP\mathcal{V}_1 = \Set{(y \mapsto y_0) | y_0 \in C, y \in \langle y_0,x_- \rangle}
\]
and \parref{threefolds-cone-situation-rational-map-S} shows that
the map \(\psi([\ell]) = \ell \cap \PP\Fr^*(L_-)^\perp \eqqcolon y\) for a line
\(\ell \subset X\) not passing through \(x_-\). Since the pullback under
\(\nu\) of the tautological subbundle on \(S\) is the bundle \(\mathcal{K}\)
by its construction in \parref{threefolds-nodal-construction-of-blowup}, it
follows that
\[
\nu^*\psi^*(\sO_{\PP\mathcal{V}_1}(-1)\rvert_{\mathbf{A}_1}) =
\ker(
\mathcal{K} \subset
V_{S^\nu} \twoheadrightarrow
(V/\Fr^*(L_-)^\perp)_{S^\nu})\rvert_{S^{\nu,\circ}}.
\]
Post-composing with the restriction to \(S^{\nu,\circ}\) of the natural map
\(\mathcal{K} \subset \tilde{\mathcal{W}} \twoheadrightarrow \sO_{\tilde\varphi_+}(-1)\)
arising from its construction, see again \parref{threefolds-nodal-construction-of-blowup},
gives the desired map relating the fibre coordinates of the two affine bundles.
Since
\[
\mathcal{K} \subset
\tilde{\mathcal{W}} \subset
\tilde\varphi_+^*\mathcal{W} =
U_{S^\nu} \oplus
\tilde\varphi_+^*\mathcal{T}_C^{\mathrm{e}},
\]
linear projection \(V \to W\) induces the following commutative diagram on
\(S^\nu\):
\[
\begin{tikzcd}
\nu^*\psi^*(\sO_{\PP\mathcal{V}_1}(-1)\rvert_{\mathbf{A}_1}) \rar \dar["\cong"']
& \mathcal{K}\rvert_{S^{\nu,\circ}} \dar \rar
& \sO_{\tilde\varphi_+}(-1)\rvert_{S^{\nu,\circ}} \dar["\cong"] \\
\tilde\varphi_+^*\phi_C^*(\sO_C(-1))\rvert_{S^{\nu,\circ}} \rar["\phi_C"]
& \tilde\varphi_+^*\mathcal{T}_C^{\mathrm{e}}\rvert_{S^{\nu,\circ}} \rar[two heads]
& \tilde\varphi_+^*(\mathcal{T}_C(-1))\rvert_{S^{\nu,\circ}}
\end{tikzcd}
\]
where the left-most and right-most vertical maps are the identifications of the
tautological bundles arising from the Euler sequence, see \parref{bundles-affine-subs-tautological}.
The commutative diagram of schemes above implies that the bottom-left map
gives the line subbundle of
\(\mathcal{T}_C^{\mathrm{e}} \subset W_C\) inducing the map \(\phi_C \colon C
\to C\); by \parref{curve-residual-intersection} this is the sheaf map \(\phi_C\).
\end{proof}

The following summarizes the consequences on the structure of \(\mathcal{F}\).

\subsection{Duality between \(\mathcal{D}\) and \(\mathcal{F}\)}
\label{threefolds-cohomology-duality}
By \parref{threefolds-conductor-dual}, there is a canonical isomorphism
\[
\mathcal{F} \cong
\mathcal{D}^\vee \otimes
\sO_C(-q+1) \otimes
L_+^{\otimes 2q-1} \otimes L_-^{\otimes 2}.
\]
This identifies graded components as
\(\mathcal{F}_i \cong \mathcal{D}_{\delta - i}^\vee \otimes \sO_C(-q+1)\)
for each \(0 \leq i \leq \delta = 2q^2-q-2\). In particular,
\(\mathcal{F}_{dq+q-1} \cong \mathcal{D}_{(2q-3-d)q+q-1}^\vee \otimes \sO_C(-q+1)\)
for \(0 \leq d \leq 2q-3\) and \parref{threefolds-cohomology-D} gives canonical
short exact sequences
\[
0 \to
\Div^{2q-3-d}_{\mathrm{red}}(W) \otimes \sO_C \to
\mathcal{F}_{dq + q -1} \to
\Div^{q-3-d}(W) \otimes \sO_C(q-1) \to
0
\]
where for each \(j \in \mathbf{Z}\), set \(\Div^j(W) \coloneqq 0\) if
\(j < 0\), and
\[
\Div^j_{\mathrm{red}}(W) \coloneqq
\ker(\Div^j(W) \to \Div^{j-q}(W) \otimes \Fr^*(W))
\]
where the map is dual to the multiplication map on symmetric powers. Note
that \(\Div^j_{\mathrm{red}}(W) = \Div^j(W)\) for \(j < q\).

The \(q\) step increasing filtration of \(\mathcal{D}\) induces one on
\(\mathcal{F}\) with
\[
\Fil_i\mathcal{F}
\coloneqq (\mathcal{D}/\Fil_{q-1-i}\mathcal{D})^\vee \otimes
\sO_C(-q+1) \otimes L_+^{\otimes 2q-1} \otimes L_-^{\otimes 2}
\]
for \(0 \leq i \leq q-1\) so that
\(\gr_i(\mathcal{F})
\cong
\gr_{q-1-i}(\mathcal{D})^\vee \otimes \sO_C(-q+1) \otimes L_+^{\otimes 2q-1} \otimes L_-^{\otimes 2}\).
Finally, the maps \(\partial_j \colon \mathcal{D} \to \mathcal{D}\), with
\(0 \leq j \leq q-1\), from \parref{threefolds-D-unipotent-S} giving the
\(\boldsymbol{\alpha}_q\)-action dually yield maps
\(\partial_j \colon \mathcal{F} \to \mathcal{F}\). The diagram
\[
\begin{tikzcd}
0 \rar
& \Fil_{q-1-i+j}\mathcal{D} \rar \dar["\partial_j"]
& \mathcal{D} \rar \dar["\partial_j"]
& \mathcal{D}/\Fil_{q-1-i+j}\mathcal{D} \rar \dar["\partial_j"]
& 0 \\
0 \rar
& \Fil_{q-1-i}\mathcal{D} \rar
& \mathcal{D} \rar
& \mathcal{D}/\Fil_{q-1-i}\mathcal{D} \rar
& 0
\end{tikzcd}
\]
implies that
\(\partial_j(\Fil_i\mathcal{F}) \subseteq \Fil_{i-j}\mathcal{F}\) and
that \(\partial_j\) is of degree \(-j(q+1)\).

\section{Cohomology of \texorpdfstring{\(\mathcal{F}\)}{F}}\label{section-cohomology-F}
This Section computes the cohomology of the sheaf \(\mathcal{F}\) introduced in
\parref{nodal-conductors-F} in the case \(q = p\) is a prime; see
\parref{threefolds-cohomology-theorem} for the final result. This is done by
identifying each graded component \(\mathrm{H}^0(C,\mathcal{F}_i)\) as a
representation \(\mathrm{SU}_3(p)\) using the structure
results from \parref{threefolds-cohomology-duality}, and proceeds in three
steps: First, many global sections are constructed for the components
\(\mathcal{F}_{ap+p-1}\) by using their explicit structure from
\parref{threefolds-cohomology-duality}. Second, the action of
\(\boldsymbol{\alpha}_p\) on \(\mathcal{F}\) gives maps
\[
\mathcal{F}_{ap+p-1} \xrightarrow{\partial}
\mathcal{F}_{ap+p-1-(p+1)} \xrightarrow{\partial}
\cdots \xrightarrow{\partial}
\begin{dcases*}
\mathcal{F}_{p-1-a} & if \(0 \leq a \leq p-2\), and \\
\mathcal{F}_{(a-p+1)p} & if \(p-1 \leq a \leq 3p-3\).
\end{dcases*}
\]
Each of the maps in the sequence are injective on global sections, see
\parref{threefolds-cohomology-del-sections}, and so this gives a lower
bound on the space of sections of each component, see \parref{threefolds-cohomology-F-sequences}.
Third, a corresponding upper bound is determined for those rightmost sheaves,
see \parref{threefolds-cohomology-upper-bounds}. The matching bounds give
equality throughout, completing the computation.

\subsection{Notation}\label{threefolds-cohomology-notation}
Throughout this Section, \emph{assume that \(q = p\) is prime}. In particular,
\(\Fr\) denotes the \(p\)-power Frobenuis. Let
\(\mathrm{SU}_3(p) \coloneqq \mathrm{U}_3(p) \cap \SL(W)\) be the special
unitary group associated with the \(q\)-bic form \((W,\beta_W)\).
As in \parref{threefolds-cohomology-duality}, for each integer \(b\), write
\[
\Div^b_{\mathrm{red}}(W) \coloneqq
\ker(\Div^b(W) \to \Div^{b-p}(W) \otimes \Fr^*(W))
\]
with the convention that \(\Div^b(W) = 0\) when \(b < 0\). For small \(b\),
this is the simple \(\mathrm{SU}_3(p)\) representation obtained by restricting
simple representations of \(\SL(W) = \SL_3\):
\[
\Div^b_{\mathrm{red}}(W) =
\begin{dcases*}
L(0,b) & for \(0 \leq b \leq p-1\), and \\
L(b-p+1,2p-2-b) & for \(p \leq b \leq 2p-3\).
\end{dcases*}
\]
See \parref{representations-divs} for the notation on the
\(\SL_3\)-representations and \parref{representations-steinberg-restriction} in
regards to restriction of representations to \(\mathrm{SU}_3(p)\).

The following determines the global sections of symmetric powers of the tangent
bundle of \(\PP W\). The restriction \(q = p\) makes it possible to use
the Borel--Weil--Bott Theorem in the cases at hand.

\begin{Lemma}\label{threefolds-cohomology-BWB-vanishing}
Let \(0 \leq b \leq p - 1\) and let \(a \leq 0\) be integers. Then
\[
\mathrm{H}^0(C,\Sym^b(\mathcal{T}_{\PP W}(-1))(a)\rvert_C) =
\begin{dcases*}
\Sym^b(W) & if \(a = 0\), and \\
0 & if \(a < 0\),
\end{dcases*}
\]
as representations of \(\mathrm{SU}_3(p)\).
\end{Lemma}

\begin{proof}
The restriction sequence
\[
0 \to
\Sym^b(\mathcal{T}_{\PP W}(-1))(a - p - 1) \to
\Sym^b(\mathcal{T}_{\PP W}(-1))(a) \to
\Sym^b(\mathcal{T}_{\PP W}(-1))(a)\rvert_C \to
0
\]
implies it suffices to show that
\(\mathrm{H}^0(\PP W,\Sym^b(\mathcal{T}_{\PP W}(-1))) = \Sym^b(W)\)
and
\[
\mathrm{H}^0(\PP W,\Sym^b(\mathcal{T}_{\PP W}(-1))(a)) =
\mathrm{H}^1(\PP W,\Sym^b(\mathcal{T}_{\PP W}(-1))(a-p)) = 0\;
\;\;\text{when}\; a < 0.
\]
The identification of global sections follows the Euler sequence;
since \(0 \leq b \leq p - 1\), the vanishing follows from the Borel--Weil--Bott
Theorem \emph{\'a la} Griffith, see \parref{representations-BWB-PP2}.
\end{proof}

The sections coming from the \(0\)-th filtered piece of \(\mathcal{F}\)
can now be determined.

\begin{Lemma}\label{threefolds-cohomology-fil-F}
As a representation of
\(\AutSch(L_- \subset U) \times \mathrm{SU}_3(p)\),
\[
\mathrm{H}^0(C,\Fil_0\mathcal{F}) \cong
\bigoplus\nolimits_{i = 0}^{p-2} \Div^{p-2-i}(W) \otimes L_+^{\otimes i}.
\]
\end{Lemma}

\begin{proof}
The duality between \(\mathcal{D}\) and \(\mathcal{F}\) as in
\parref{threefolds-cohomology-duality} yields an identification
\[
\Fil_0\mathcal{F} \cong
\gr_{p-1}(\mathcal{D})^\vee \otimes \sO_C(-p+1)
\otimes L_+^{\otimes 2p-1} \otimes L_-^{\otimes 2}.
\]
By \parref{threefolds-D-Scirc-gr} and
\parref{threefolds-nodal-compute-conductor}, the right hand side is
\[
\big(
  (\pi_*\sO_{T^\circ})^\vee \otimes
  \sO_C(-p+1) \otimes
  L_+^{\otimes p} \otimes
  L_-^{\otimes p+1}\big)_{\geq 0}
= \bigoplus\nolimits_{d = 0}^{p^2-p-1}
  (\pi_*\sO_{T^\circ})_{p^2-p-1-d}^\vee \otimes \sO_C(-p+1).
\]
Since \(p^2 - p - 1\) is strictly less than the degree \(p^2\) of the
equation \(v_2\) from \parref{threefolds-smooth-cone-present-Tcirc}, the
graded pieces of \(\pi_*\sO_{T^\circ}\) appearing above coincide with the
graded pieces of \(\mathcal{B}'\) as given in \parref{threefolds-cohomology-grading-easy}.
Thus this gives a canonical isomorphism
\[
\Fil_0\mathcal{F}
\cong
\bigoplus\nolimits_{i = 0}^{p-2}
\bigoplus\nolimits_{j = 0}^{p-1}
\Div^{p-2-i}(\mathcal{T}_{\PP W}(-1))(-j)\rvert_C \otimes
L_+^{\otimes i} \otimes
L_-^{\vee, \otimes j}.
\]
Since all divided powers have exponent less than \(p\), they coincide with
symmetric powers. Then \parref{threefolds-cohomology-BWB-vanishing} applies to
show that the summands with \(j = 0\) have the required spaces of
sections, and all other summands have no sections.
\end{proof}

Let \(\partial \colon \mathcal{F} \to \mathcal{F}\) be the operator induced by
the action of \(\boldsymbol{\alpha}_p\), as explained in
\parref{threefolds-D-unipotent-S} and \parref{threefolds-cohomology-duality}.
Its kernel is easily determined:

\begin{Lemma}\label{threefolds-cohomology-kernel-del}
\(\ker(\partial \colon \mathcal{F} \to \mathcal{F}) =
\bigoplus\nolimits_{i = 0}^p \mathcal{F}_i \oplus
\bigoplus\nolimits_{i = p+1}^\delta \Fil_0\mathcal{F}_i
\).
\end{Lemma}

\begin{proof}
Since \(\partial\) is of degree \(-p-1\), each of \(\mathcal{F}_i\) for
\(0 \leq i \leq p\) must be annihilated by \(\partial\). For \(p+1 \leq i \leq \delta\),
the duality of \parref{threefolds-cohomology-duality} applied to
\parref{threefolds-cohomology-alpha-p} gives an exact sequence
\[
0 \to
\Fil_0\mathcal{F}_i \to
\mathcal{F}_i \xrightarrow{\partial}
\mathcal{F}_{i-p-1} \to
\gr_{p-1}(\mathcal{F}_{i-p-1}) \to
0
\]
and this proves the statement.
\end{proof}

\begin{Corollary}\label{threefolds-cohomology-del-sections}
\(\partial \colon \mathcal{F}_i \to \mathcal{F}_{i-p-1}\) is injective
on global sections except when
\begin{itemize}
\item \(\mathcal{F}_i\) is on top, so that \(0 \leq i \leq p - 1\), or
\item \(\mathcal{F}_i\) is on the left, so that \(i = jp\) for \(0 \leq j \leq p - 2\).
\end{itemize}
\end{Corollary}

\begin{proof}
This follows from \parref{threefolds-cohomology-kernel-del} and
\parref{threefolds-cohomology-fil-F}.
\end{proof}

The following gives a lower bound on the spaces of sections of the
\(\mathcal{F}_i\) by using the explicit description of the \(\mathcal{F}_{ap+p-1}\)
from \parref{threefolds-cohomology-duality}, and successively applying the
operators \(\partial\) to propagate these sections to other components:

\begin{Lemma}\label{threefolds-cohomology-F-sequences}
There are inclusions of \(\mathrm{SU}_3(p)\)-representations
\[
\Div^{2p-3-a}_{\mathrm{red}}(W) \subseteq
\mathrm{H}^0(C,\mathcal{F}_{ap+p-1-b(p+1)})
\]
for \(0 \leq a \leq 2p-3\) and \(0 \leq b \leq \min(a,p-1)\).
\end{Lemma}

\begin{proof}
The sheaves \(\mathcal{F}_{ap+p-1}\) are identified in
\parref{threefolds-cohomology-duality} and taking global sections gives the
stated inclusion for \(b = 0\). The remaining statements for \(b > 0\) now
following from the injectivity statements on global sections for \(\partial\)
from \parref{threefolds-cohomology-del-sections}.
\end{proof}

\subsection{Upper bounds}\label{threefolds-cohomology-upper-bounds}
The statement of \parref{threefolds-cohomology-F-sequences} might be
thought of as a lower bound on the groups \(\mathrm{H}^0(C,\mathcal{F}_i)\). It
remains to give a matching upper bound. The injectivity statements for
\(\partial^\vee\) from \parref{threefolds-cohomology-del-sections} means it
suffices to give a bound when
\begin{itemize}
\item \(\mathcal{F}_i\) is on top, so that \(0 \leq i \leq p-1\), and
\item \(\mathcal{F}_i\) is on the left, so that \(i = jp\) for \(0 \leq j \leq 2p-2\).
\end{itemize}
The cases \(0 \leq i \leq p\) are dealt with an explicit cohomology computation, see
\parref{threefolds-cohomology-nonzero-map} and
\parref{threefolds-cohomology-upper-bound}; the remaining cases then follow
from this explicit calculation by further analyzing the action of
\(\partial\) on global sections, see \parref{threefolds-cohomology-left}.

\subsection{}\label{threefolds-cohomology-phi}
Let \(0 \leq i \leq p\) and consider the defining short exact sequence from
\parref{nodal-conductors-F}:
\[
0 \to
\mathcal{D}_i \to
\mathcal{D}^\nu_i \to
\mathcal{F}_i \to
0.
\]
Using the identification of the low degree pieces of \(\mathcal{D}\) from
\parref{threefolds-cohomology-D}, of \(\mathcal{D}^\nu\) from
\parref{threefolds-cohomology-D-nu}, and taking the long exact sequence in
cohomology shows that, as \(\mathrm{SU}_3(p)\)-representations,
\[
\mathrm{H}^0(C,\mathcal{F}_i) \cong
\begin{dcases*}
\ker\big(
  \mathrm{H}^1(C,\sO_C(-i)) \to
  \mathrm{H}^1(C,\phi_{C,*}(\mathcal{T}_C(-1)^{\otimes i}))
\big) & if \(0 \leq i \leq p-1\), \\
\ker\big(
  \mathrm{H}^1(C,\Omega_{\PP W}^1(1)\rvert_C) \to
  \mathrm{H}^1(C,\phi_{C,*}(\mathcal{T}_C(-1)^{\otimes p}))
\big) & if \(i = p\).
\end{dcases*}
\]
Then \parref{threefolds-cohomology-nu} identifies the map
\(\mathcal{D}_i \to \phi_{C,*}\mathcal{D}_i^\nu\), when \(0 \leq i \leq p - 1\),
as
\[
\phi_{C,*}(\phi_C^i) \circ \phi_C^\# \colon
\sO_C(-i) \to
\phi_{C,*}(\sO_C(-iq^2)) \to
\phi_{C,*}(\mathcal{T}_C(-1)^{\otimes i}).
\]
In the case \(i = p\),
the composite
\(\mathcal{D}_1^{\otimes p} \hookrightarrow \mathcal{D}_p \to \phi_{C,*}\mathcal{D}^\nu_p\)
is also given by this map. Its action on \(\mathrm{H}^1(C,\sO_C(-i))\) may
be computed explicitly as follows: Let
\[
f \in \mathrm{H}^0(\PP W, \sO_{\PP W}(p+1))
\quad\text{and}\quad
\tilde\phi_C \in \mathrm{H}^0(\PP W,\sO_{\PP^2}(p^2-p+1))
\]
be an equation for \(C\) and any lift of \(\phi_C \in \mathrm{H}^0(C,\sO_C(p^2-p+1))\),
respectively. Then there is a commutative diagram sheaves on \(\PP W\) with exact rows
given by
\[
\begin{tikzcd}
\sO_{\PP W}(-i-p-1) \rar[hook,"f"'] \dar["\phi_C^\#"]
& \sO_{\PP W}(-i) \rar[two heads] \ar[dd,"\phi_C^\#"]
& \sO_C(-i) \ar[dd,"\phi_C^\#"] \\
\phi_{C,*}\sO_{\PP W}(-(i+p+1)p^2) \dar["f^{p^2-1}"] \\
\phi_{C,*}\sO_{\PP W}(-ip^2 - p - 1) \rar[hook, "\phi_{C,*}(f)"] \dar["\phi_{C,*}(\tilde\phi_C^i)"]
& \phi_{C,*}\sO_{\PP W}(-ip^2) \rar[two heads] \dar["\phi_{C,*}(\tilde\phi_C^i)"]
& \phi_{C,*}\sO_C(-ip^2) \dar["\phi_{C,*}(\phi_C^i)"] \\
\phi_{C,*}\sO_{\PP W}(-i(p-1)-p-1) \rar[hook]
& \phi_{C,*}\sO_{\PP W}(-i(p-1)) \rar[two heads]
& \phi_{C,*}(\mathcal{T}_C(-1)^{\otimes i})\punct{.}
\end{tikzcd}
\]

Consider the long exact sequence in cohomology associated with the top row
of the diagram. The following identifies the resulting spaces as
representations for \(\mathrm{SU}_3(p)\) using some of the notation and
computations from Appendix \parref{chapter-representations}. See,
in particular, \parref{representations-root-data} and
\parref{representations-homogeneous} for the notations \(L(a,b)\) and
\(\Delta(a,b)\).

\begin{Lemma}\label{threefolds-cohomology-twist}
The \(\mathrm{SU}_3(p)\) representations
\(\mathrm{H}^1(C,\sO_C(-i))\) for \(0 \leq i \leq p\) are:
\begin{enumerate}
\item\label{threefolds-cohomology-twist.low}
If \(0 \leq i \leq 1\), then
\(\mathrm{H}^1(C,\sO_C(-i)) \cong \Div^{p+i-2}(W)\) is simple.
\item\label{threefolds-cohomology-twist.sequence}
If \(2 \leq i \leq p\), then there is a short exact sequence
\[
0 \to
\Div^{p+i-2}_{\mathrm{red}}(W) \to
\mathrm{H}^1(C,\sO_C(-i)) \to
\Delta(1,i-2) \to
0.
\]
\item\label{threefolds-cohomology-twist.simple}
If \(i \neq p\), then the quotient \(\Delta(1,i-2) = L(1,i-2)\) is simple.
\item\label{threefolds-cohomology-twist.not-simple}
If \(i = p\), then
\(\Delta(1,p-2) \cong \mathrm{H}^0(C,\Omega^1_{\PP W}(1)\rvert_C)\)
and a short exact sequence
\[
0 \to
\Div^{p-3}(W) \to
\mathrm{H}^0(C,\Omega^1_{\PP W}(1)\rvert_C) \to
L(1,p-2) \to
0.
\]
\end{enumerate}
\end{Lemma}

\begin{proof}
The cohomology sequence associated with
\[ 0 \to \sO_{\PP W}(-i-p-1) \xrightarrow{f} \sO_{\PP W}(-i) \to \sO_C(-i) \to 0 \]
shows that
\(\mathrm{H}^1(C,\sO_C(-i)) \cong
\ker(\Div^{i+p-2}(W) \xrightarrow{f} \Div^{i-3}(W))\).
When \(0 \leq i \leq 1\),
\parref{representations-divs}\ref{representations-divs.simple}
shows that \(\Div^{p+i-2}(W)\) is the simple representation \(L(0,p+i-2)\),
proving \ref{threefolds-cohomology-twist.low}.
Suppose from now on that \(2 \leq i \leq p\). Then the diagram
\[
\begin{tikzcd}[row sep=1.5em]
0 \rar
& \Div^{i+p-2}_{\mathrm{red}}(W) \rar
& \Div^{i+p-2}(W) \rar \dar["f"']
& \Fr^*(W) \otimes \Div^{i-2}(W) \rar \dar["f"]
& 0  \\
&& \Div^{i-3}(W) \rar[equal]
& \Div^{i-3}(W)
\end{tikzcd}
\]
is commutative with exact top row, and \parref{representations-F-weyl} gives
the identification
\[
\ker\big(\Fr^*(W) \otimes \Div^{i-2}(W) \xrightarrow{f} \Div^{i-3}(W)\big)
\cong \Delta(1,i-2).
\]
Taking kernels of the vertical maps in the diagram yields the exact sequence
in \ref{threefolds-cohomology-twist.sequence}:
\[
0 \to
\Div^{i+p-2}_{\mathrm{red}}(W) \to
\mathrm{H}^1(C,\sO_C(-i)) \to
\Delta(1,i-2) \to
0.
\]
If \(2 \leq i \leq p-1\), then \(\Delta(1,i-2) = L(1,i-2)\) is simple by
\parref{representations-1-b}\ref{representations-1-b.simple}, proving
\ref{threefolds-cohomology-twist.simple}. If \(i = p\), then the Euler sequence
yields an \(\mathrm{SU}_3(p)\)-equivariant identification
\[
\mathrm{H}^1(C,\Omega^1_{\PP W}(1)\rvert_C) \cong
\ker\big(W^\vee \otimes \mathrm{H}^1(C,\sO_C) \to \mathrm{H}^1(C,\sO_C(-1))\big)
\cong \Delta(1,p-2)
\]
and \parref{representations-1-b}\ref{representations-1-b.not-simple} gives
the short exact sequence finishing the proof of
\ref{threefolds-cohomology-twist.not-simple}.
\end{proof}

The following shows that
\(\mathrm{H}^1(C,\mathcal{D}) \to \mathrm{H}^1(C,\mathcal{D}^\nu)\) is
nonzero in low degree components by using the explicit computation in the
cohomology of \(\PP W\) using the diagram of \parref{threefolds-cohomology-phi}
and the identification of \(\nu^\#\) from \parref{threefolds-cohomology-nu}.

\begin{Lemma}\label{threefolds-cohomology-nonzero-map}
The map
\(
\phi_{C,*}(\phi_C^i) \colon
\mathrm{H}^1(C,\sO_C(-i)) \to
\mathrm{H}^1(C,\phi_{C,*}(\mathcal{T}_C(-1)^{\otimes i}))
\)
is nonzero for each \(2 \leq i \leq p\).
\end{Lemma}

\begin{proof}
Choose coordinates \((x:y:z)\) on \(\PP W = \PP^2\) so that
\(f = x^p y + x y^p - z^{p+1}\).
A lift \(\tilde\phi_C\) of \(\phi_C\) is computed in \parref{curves-1+1+1-lift-of-phi}
and is given by
\[
\tilde\phi_C \coloneqq
\frac{x^{p^2} y - x y^{p^2}}{x^p y + x y^p}z =
\big(x^{p(p-1)} - x^{(p-1)(p-1)} y^{p-1} + \cdots - y^{p(p-1)}\big) z.
\]
View the cohomology groups of \(\PP^2\) as a module over its homogeneous
coordinate ring as explained in \citeSP{01XT}, and consider a class
\[
\xi \coloneqq \frac{1}{xyz} \frac{1}{x^{i + p - 2}} \in
\mathrm{H}^2(\PP^2,\sO_{\PP^2}(-i-p-1)).
\]
Such a class acts on homogeneous polynomials by contraction, see \citeSP{01XV}.
In particular, \(\xi \cdot f = 0\) since \(f\) does not contain a pure power of
\(x\), so \(\xi\) represents a class in \(\mathrm{H}^1(C,\sO_C(-i))\).
I claim that \(\phi_{C,*}(\phi_C^i)(\xi) \neq 0\). Indeed, since \(f\) is a Hermitian
\(q\)-bic equation, \(\phi_C = \Fr^2\) is the \(q^2\)-power Frobenius by
\parref{hypersurfaces-endomorphism}, and the diagram of
\parref{threefolds-cohomology-phi} shows that \(\phi_{C,*}(\phi_C^i)(\xi)\) is
represented by the product
\[
\xi^{p^2} \cdot (f^{p^2-1} \tilde\phi_C^i) =
\Big(\frac{1}{xyz} \frac{1}{x^{d+p-2}}\Big)^{p^2} \cdot
\Big((x^p y + x y^p - z^{p+1})^{p^2 - 1}
\Big(\frac{x^{p^2} y - x y^{p^2}}{x^p y + x y^p}\Big)^i z^i\Big).
\]
To see this is nonzero, consider the coefficient of \(z^{(p+1)(p-2) + i}\) in
\(f^{p^2-1} \tilde\phi_C^i\). Since \(0 < i < p+1\), this is the coefficient
of \(z^i\) in \(\tilde\phi_C^i\) multiplied by the coefficient of \(z^{(p+1)(p-2)}\)
in \(f^{p^2-1}\). The latter is \(-(x^p y + x y^p)^{p^2-p+1}\), as found by writing
\[
f^{p^2-1} = \big((x^p y + x y^p)^p - z^{p(p+1)}\big)^{p-1} \big((x^p y + x y^p) - z^{p+1}\big)^{p-1}.
\]
Therefore \(\xi^{p^2} \cdot (f^{p^2-1} \tilde\phi_C^i)\) has a summand given by
\begin{multline*}
\frac{1}{x^{(i+p-1)p^2} y^{p^2} z^{p+2-i}} \cdot
  \Big(-(x^p y + x y^p)^{p^2-p+1}
  \Big(\frac{x^{p^2} y - x y^{p^2}}{x^p y + x y^p}\Big)^i\Big) \\
= \frac{-1}{x^{(i+p-2)p^2 + p - 1} y^{p-1} z^{p+2-i}} \cdot
\Big((x^{p-1} + y^{p-1})^{p^2 - p + 1 - i} (x^{p^2 - 1} - y^{p^2-1})^i\Big).
\end{multline*}
Since all monomials in \(y\) involve at least \(y^{p-1}\), the only potentially
nonzero contribution is the pure power of \(x\), so this is equal to
\[
\frac{-1}{x^{(i+p-1)p^2 + p - 1} y^{p-1} z^{p+2-i}} \cdot x^{(p-1)(p^2-p+1-i) + (p^2-1)i}
= \frac{-1}{x^{(i-1)p} y^{p-1} z^{p+2-i}}.
\]
This is nonzero in \(\mathrm{H}^2(\PP^2,\sO_{\PP^2}(-(i+1)(p+1)))\)
if \(2 \leq i \leq p\), so \(\phi_{C,*}(\phi_C^i)(\xi) \neq 0\).
\end{proof}

The following completely identifies the global sections of the low degree
graded components of \(\mathcal{F}\):

\begin{Proposition}\label{threefolds-cohomology-upper-bound}
There are \(\mathrm{SU}_3(p)\)-equivariant isomorphisms
\[
\mathrm{H}^0(C,\mathcal{F}_i) \cong
\begin{dcases*}
\Div^{p+i-2}_{\mathrm{red}}(W) & if \(0 \leq i \leq p - 1\), and \\
\Div^{p-3}(W) & if \(i = p\).
\end{dcases*}
\]
\end{Proposition}

\begin{proof}
Fix \(0 \leq i \leq p\) and let \(L\) denote the
\(\mathrm{SU}_3(p)\)-representation appearing on the right hand side of the
statement. Consider the short exact sequence
\[
0 \to
\mathrm{H}^0(C,\mathcal{F}_i) \to
\mathrm{H}^1(C,\mathcal{D}_i) \xrightarrow{\phi}
\mathrm{H}^1(C,\phi_{C,*}(\mathcal{D}_i^\nu)) \to
0.
\]
Applying \parref{threefolds-cohomology-F-sequences} with \(a = b = p-i-1\) when
\(i \neq p\) and \(a = b+1 = p\) when \(i = p\) shows that
\(L \subseteq \mathrm{H}^0(C,\mathcal{F}_i) \subseteq \mathrm{H}^1(C,\mathcal{D}_i)\).
When \(0 \leq i \leq 1\), \parref{threefolds-cohomology-D} with
\parref{threefolds-cohomology-twist}\ref{threefolds-cohomology-twist.low}
already show that \(\mathrm{H}^1(C,\mathcal{D}_i) = L\), proving the Proposition
in that case. For \(2 \leq i \leq p\), I claim that the map \(\phi\) is nonzero
and that \(\mathrm{H}^1(C,\mathcal{D}_i)/L\) is a simple
\(\mathrm{SU}_3(p)\)-representation. This implies the Proposition upon
comparing with the short exact sequence above.

Consider the claim for \(2 \leq i \leq p-1\). Then
\(\mathcal{D}_i \cong \sO_C(-i)\) by \parref{threefolds-cohomology-D},
\(\phi = \phi_{C,*}(\phi_C^i)\) by \parref{threefolds-cohomology-nu} which is nonzero
by \parref{threefolds-cohomology-nonzero-map}, and
\(\mathrm{H}^1(C,\mathcal{D}_i)/L \cong L(1,i-2)\) is simple by
\parref{threefolds-cohomology-twist}\ref{threefolds-cohomology-twist.simple}.
When \(i = p\), then \(\mathcal{D}_p \cong \Omega^1_{\PP W}(1)\rvert_C\) by
\parref{threefolds-cohomology-D}. The
subbundle given by \(\mathcal{D}_1^{\otimes p} = \sO_C(-p)\) maps to
\(\phi_{C,*}\mathcal{D}_p^\nu\) via \(\phi_{C,*}(\phi_C^p)\) by \parref{threefolds-cohomology-nu}.
Thus there is a diagram in cohomology
\[
\begin{tikzcd}
\mathrm{H}^1(C,\mathcal{D}_p) \rar["\phi"] & \mathrm{H}^1(C,\phi_{C,*}\mathcal{D}_p^\nu) \\
\mathrm{H}^1(C,\sO_C(-p)) \uar[two heads] \ar[ur,"\phi_{C,*}(\phi_C^p)"']
\end{tikzcd}
\]
where the vertical arrow is surjective since
\(\mathcal{D}_p/\mathcal{D}_1^{\otimes p} \cong \sO_C(p-1)\). Since
\(\phi_{C,*}(\phi_C^p) \neq 0\) by \parref{threefolds-cohomology-nonzero-map},
\(\phi \neq 0\). Finally,
\parref{threefolds-cohomology-twist}\ref{threefolds-cohomology-twist.not-simple}
identifies the quotient \(\mathrm{H}^0(C,\mathcal{D}_p)/L\) with the
simple module \(L(1,p-2)\), completing the proof of the claim in the case
\(i = p\).
\end{proof}

\begin{Corollary}\label{threefolds-cohomology-right}
For each \(0 \leq i \leq 2p-3\), there are \(\mathrm{SU}_3(p)\)-equivariant
isomorphisms
\[
\mathrm{H}^0(C,\mathcal{F}_{ip+p-1}) \cong
\Div^{2p-3-i}_{\mathrm{red}}(W).
\]
\end{Corollary}

\begin{proof}
When \(p-2 \leq i \leq 2p-3\),
\(\mathcal{F}_{ip+p-1} \cong \Div^{2p-3-i}_{\mathrm{red}}(W) \otimes \sO_C\)
by the sequence in \parref{threefolds-cohomology-duality},
yielding the conclusion in this case. When \(0 \leq i \leq p-3\),
\parref{threefolds-cohomology-F-sequences}, \parref{threefolds-cohomology-del-sections},
and \parref{threefolds-cohomology-upper-bound} together give a sequence
of inclusions
\[
\Div^{2p-3-i}_{\mathrm{red}}(W) \subseteq
\mathrm{H}^0(C,\mathcal{F}_{ip+p-1}) \stackrel{\partial^i}{\hookrightarrow}
\mathrm{H}^0(C,\mathcal{F}_{p-1-i}) =
\Div^{2p-3-i}_{\mathrm{red}}(W).
\]
Therefore equality holds throughout.
\end{proof}

\begin{Proposition}\label{threefolds-cohomology-left}
There are \(\mathrm{SU}_3(p)\)-equivariant isomorphisms
\[
\mathrm{H}^0(C,\mathcal{F}_{ip}) \cong
\begin{dcases*}
\Div^{p-2-i}(W) & if \(0 \leq i \leq p-2\), and \\
0 & if \(p-1 \leq i \leq 2p-2\).
\end{dcases*}
\]
\end{Proposition}

\begin{proof}
The cases \(0 \leq i \leq 1\) are handled by
\parref{threefolds-cohomology-upper-bound}. So assume that
\(2 \leq i \leq 2p-2\). Dualizing as in
\parref{threefolds-conductor-dual} together with
\parref{threefolds-cohomology-alpha-p} yields an exact sequence
\[
0 \to
\Fil_0\mathcal{F}_{ip} \to
\mathcal{F}_{ip} \xrightarrow{\partial}
\mathcal{F}_{(i-2)p+p-1} \to
\gr_{p-1}\mathcal{F}_{(i-2)p+p-1} \to
0.
\]
Since \(\mathrm{H}^0(C,\Fil_0\mathcal{F}_{ip}) = \Div^{p-2-i}(W)\) by
\parref{threefolds-cohomology-fil-F}, with the convention that negative divided
powers are zero as in \parref{threefolds-cohomology-notation}, it suffices to
show that \(\partial\) vanishes on global sections. The exact sequence
means that this is equivalent to injectivity of
\(\mathcal{F}_{(i-2)p+p-1} \to \gr_{p-1}\mathcal{F}_{(i-2)p+p-1}\)
on global sections.

Consider the composite
\[
\Div^{2p-1-i}_{\mathrm{red}}(W) \otimes \sO_C \subset
\mathcal{F}_{(i-2)p+p-1} \twoheadrightarrow
\gr_{p-1}\mathcal{F}_{(i-2)p+p-1}
\]
where the first map is the inclusion of the subbundle from
\parref{threefolds-cohomology-duality}. The first map
is an isomorphism on global sections by \parref{threefolds-cohomology-right},
and \(\Div^{2p-1-i}_{\mathrm{red}}(W)\) is a simple \(\mathrm{SU}_3(p)\)
representation by \parref{threefolds-cohomology-notation}, so it
suffices to see that the composite is a nonzero map of sheaves. Applying the
duality of \parref{threefolds-cohomology-duality}, this is equivalent to the
assertion that
\[
\bar{\mu}_{2p-1-i} \colon
\Fil_0\mathcal{D}_{(2p-1-i)p+p-1} \to
(\Sym^{2p-1-i}/\Sym^{p-1-i} \otimes \Fr^*)(W^\vee_C(-1)) \otimes \sO_C(p-i)
\]
is nonzero. This follows from the final statement of
\parref{threefolds-cohomology-D}: \(\bar{\mu}_{2p-1-i}\) is a strict surjection
of filtered bundles on all of \(\mathcal{D}_{(2p-1-i)p+p-1}\), so
it gives a surjection
\[
\Fil_0\bar{\mu}_{2p-1-i} \colon
\Fil_0\mathcal{D}_{(2p-1-i)p+p-1} \to
(\Sym^{2p-1-i}(\Omega_{\PP W}^1)/\Sym^{p-1-i}(\Omega_{\PP W}^1))(p-i)\rvert_C
\]
between \(0\)-th filtered pieces; the target is nonzero, so
\(\bar{\mu}_{2p-1-i}\) is nonzero.
\end{proof}

\begin{Theorem}\label{threefolds-cohomology-theorem}
There is an isomorphism of
\(\AutSch(L_- \subset U,\beta_U) \times \mathrm{SU}_3(p)\)-representations
\(
\mathrm{H}^0(C,\mathcal{F}) \cong
\Lambda_1 \oplus \Lambda_2
\)
where
\begin{align*}
\Lambda_1 & \coloneqq
\bigoplus\nolimits_{i = 0}^{p-2}
\Div^{p-2-i}(W) \otimes
\Sym^{p-1}(U) \otimes
L_-^{\vee, \otimes p-1} \otimes
L_+^{\otimes i}, \;\;\text{and} \\
\Lambda_2 & \coloneqq
\bigoplus\nolimits_{i = 0}^{p-2}
\Div^{p+i-1}_{\mathrm{red}}(W) \otimes
\Sym^{p-2-i}(U) \otimes
L_-^{\vee,\otimes p-1}.
\end{align*}
\end{Theorem}

\begin{proof}
Begin by identifying each \(\mathrm{H}^0(C,\mathcal{F}_i)\) as a representation
for \(\mathrm{SU}_3(p)\). The claim is that the inclusions from
\parref{threefolds-cohomology-F-sequences} are equalities: that
\[
\mathrm{H}^0(C,\mathcal{F}_{ap+p-1-b(p+1)}) =
\Div^{2p-3-a}_{\mathrm{red}}(W)
\]
if \(0 \leq a \leq 2p-3\) and \(0 \leq b \leq \min(a,p-1)\), and that the
group vanishes otherwise.
Choose \(0 \leq a \leq 3p-3\). Starting from
\(\mathrm{H}^0(C,\mathcal{F}_{ap+p-1})\) and successively applying \(\partial\)
a total of \(\min(a,p-1)\) times produces, thanks to
\parref{threefolds-cohomology-del-sections}, a chain of inclusions
\[
\mathrm{H}^0(C,\mathcal{F}_{ap+p-1}) \subseteq
\cdots \subseteq
\begin{dcases*}
\mathrm{H}^0(C,\mathcal{F}_{p-1-a}) & if \(0 \leq a \leq p-2\), and \\
\mathrm{H}^0(C,\mathcal{F}_{(a-p+1)p}) & if \(p-1 \leq a \leq 3p-3\).
\end{dcases*}
\]
By convention, set \(\mathrm{H}^0(C,\mathcal{F}_i) = 0\) whenever
\(i > \delta\).
The spaces on the left are given by \parref{threefolds-cohomology-right}
whereas the spaces on the right are given by \parref{threefolds-cohomology-upper-bound}
and \parref{threefolds-cohomology-left}, and for each fixed \(a\), the lower
and upper bounds match. Therefore equality holds throughout. The
\(\AutSch(L_- \subset U,\beta_U)\) side of the representation is obtained
by matching weights and using \parref{threefolds-cohomology-del-sections}
to identify the action of the unipotent radical.
\end{proof}

\begin{Corollary}\label{threefolds-cohomology-dimension-result}
\(
\dim_{\mathbf{k}}\mathrm{H}^0(C,\mathcal{F}) =
(p^2 + 1)\binom{p}{2} + \binom{p}{3}
\).
\end{Corollary}

\begin{proof}
Consider summing the cohomology groups in \parref{threefolds-cohomology-theorem}
column-wise, summing over residue classes modulo \(p\):
\begin{align*}
\dim_\kk\mathrm{H}^0(C,\mathcal{F})
& =
\sum\nolimits_{b = 0}^{p-1}
\sum\nolimits_{a = 0}^{b+p-2}
\dim_\kk\mathrm{H}^0(C,\mathcal{F}_{(b+p-2-a)p + b}) \\
& =
\sum\nolimits_{b = 0}^{p-1}
\sum\nolimits_{a = 0}^{b+p-2} \dim_\kk \Div^a_{\mathrm{red}}(W) \\
& =
\sum\nolimits_{b = 0}^{p-1}
\sum\nolimits_{a = 0}^{b+p-2} \big(\dim_\kk \Div^a(W) - \dim_\kk (\Fr^*(W) \otimes \Div^{a-p}(W))\big) .
\end{align*}
Since \(W\) is a \(3\)-dimensional vector space, \(\Div^a(W)\) is \(\binom{a+2}{2}\)
dimensional for all \(a \geq 0\), so using standard binomial coefficient
identities gives
\begin{align*}
\dim_\kk\mathrm{H}^0(C,\mathcal{F})
& =
\sum\nolimits_{b = 0}^{p-1}
\sum\nolimits_{a = 0}^{b+p-2} \binom{a+2}{2}
-3
\sum\nolimits_{b = 2}^{p-1}
\sum\nolimits_{a = 0}^{b-2} \binom{a+2}{2} \\
& =
\sum\nolimits_{b = 0}^{p-1} \binom{b+p+1}{3}
-3
\sum\nolimits_{b = 2}^{p-1} \binom{b+1}{3}
= \binom{2p+1}{4} - 4\binom{p+1}{4}.
\end{align*}
It can now be directly verified that
\(\binom{2p+1}{4} - 4\binom{p+1}{4} = (p^2+1)\binom{p}{2} + \binom{p}{3}\).
\end{proof}

Together with the computations in \parref{threefolds-N2+1+1+1}, this gives the
coherent cohomology of \(\sO_S\) when \(q = p\):

\begin{Theorem}\label{threefolds-nodal-fano-cohomology}
If \(q = p\), the dimensions of the cohomology groups of \(\sO_S\) are
\[
\dim_\kk \mathrm{H}^i(S,\sO_S) =
\begin{dcases*}
1 & if \(i = 0\), \\
(p^2+1)\binom{p}{2} + \binom{p}{3} & if \(i = 1\), and \\
\frac{1}{2}(5p^3 - 5p^2 + 3p - 6)\binom{p+1}{3} & if \(i = 2\).
\end{dcases*}
\]
\end{Theorem}

\begin{proof}
By \parref{threefolds-normalize-cohomology},
\(\mathrm{H}^0(S,\sO_S) \cong \mathrm{H}^0(C,\sO_C)\) and
\(\mathrm{H}^1(S,\sO_S) \cong \mathrm{H}^0(C,\mathcal{F})\). Combined with
\parref{threefolds-cohomology-dimension-result}, this gives the first two
numbers. The dimension of \(\mathrm{H}^2(S,\sO_S)\) can now be obtained from
the Euler characteristic computation of \parref{threefolds-lines-chi-OS}.
\end{proof}

This Section closes with some remarks on extending this computation for
\(q\) not necessarily prime.

\begin{figure}[t]
\[
\begin{smallmatrix}
28&36&42&46&51&48&55&42 \\
24&36&39&42&56&72&60&46 \\
18&30&36&36&51&73&75&48 \\
10&18&24&28&36&51&56&48 \\
 6&10&15&21&28&36&42&46 \\
 3& 6&10&15&24&36&39&42 \\
 1& 3& 6&10&18&30&36&36 \\
 0& 1& 3& 6&10&18&24&28 \\
 0& 0& 1& 3& 6&10&15&21 \\
 0& 0& 0& 1& 3& 6&10&15 \\
 0& 0& 0& 0& 1& 3& 6&10 \\
 0& 0& 0& 0& 0& 1& 3& 6 \\
 0& 0& 0& 0& 0& 0& 1& 3 \\
 0& 0& 0& 0& 0& 0& 0& 1 \\
 0& 0& 0& 0& 0& 0& 0
\end{smallmatrix}
\quad\quad\quad\quad
\begin{smallmatrix}
36&45&52&58&60&61&71&60&52 \\
31&45&45&58&75&60&77&90&57 \\
21&31&36&45&58&57&63&71&60 \\
15&21&28&36&45&52&58&60&61 \\
11&18&21&31&45&45&58&75&60 \\
 6&11&15&21&31&36&45&58&57 \\
 3& 6&10&15&21&28&36&45&52 \\
 1& 3& 6&11&18&21&31&45&45 \\
 0& 1& 3& 6&11&15&21&31&36 \\
 0& 0& 1& 3& 6&10&15&21&28 \\
 0& 0& 0& 1& 3& 6&11&18&21 \\
 0& 0& 0& 0& 1& 3& 6&11&15 \\
 0& 0& 0& 0& 0& 1& 3& 6&10 \\
 0& 0& 0& 0& 0& 0& 1& 3& 6 \\
 0& 0& 0& 0& 0& 0& 0& 1& 3 \\
 0& 0& 0& 0& 0& 0& 0& 0& 1 \\
 0& 0& 0& 0& 0& 0& 0& 0
\end{smallmatrix}
\]
\caption{The dimensions of the \(\mathrm{H}^0(C,\mathcal{F}_i)\) are displayed
with \(q = 8\) on the left, and \(q = 9\) on the right. The numbers are
arranged so that the first row displays the dimensions of
\(\mathrm{H}^0(C,\mathcal{F}_i)\) for \(0 \leq i \leq q-1\). These were
obtained from a computer calculation.}
\label{threefolds-cohomology-F-restriction-remarks.figure}
\end{figure}

\subsection{Remarks toward general \(q\)}\label{threefolds-cohomology-F-restriction-remarks}
The assumption that \(q = p\) was used in at least three ways:
\begin{enumerate}
\item to apply the Borel--Weil--Bott Theorem in
\parref{threefolds-cohomology-BWB-vanishing} and to identify
divided powers with symmetric powers in \parref{threefolds-cohomology-fil-F};
\item to reduce the action of \(\boldsymbol{\alpha}_q\) to the action of the
single operator \(\partial \coloneqq \partial_1\) so that its action on
associated graded modules is understood via
\parref{threefolds-cohomology-alpha-p}; and
\item to show in \parref{threefolds-cohomology-twist} that the
\(\mathrm{SU}_3(q)\) representations appearing are either simple or have very
short composition series.
\end{enumerate}

In any case, part of the difficulty to extending the computation past the
prime case is that the formula \parref{threefolds-cohomology-dimension-result}
does not hold for general \(q\). A computer computation shows that
\[
\dim_\kk \mathrm{H}^0(C,\mathcal{F})
=
\begin{dcases*}
106  \\
2096  \\
3231  \\
\end{dcases*}
\quad\text{whereas}\quad
(q^2+1)\binom{q}{2} + \binom{q}{3} =
\begin{dcases*}
106 & if \(q = 4\), \\
1876 & if \(q = 8\), and \\
3036 & if \(q = 9\).
\end{dcases*}
\]
The dimensions of \(\mathrm{H}^0(C,\mathcal{F}_i)\) in the cases
\(q = 8\) and \(q = 9\) are given in Figure
\parref{threefolds-cohomology-F-restriction-remarks.figure}. Certain
general features from the prime case hold true---for instance, the
action of \(\boldsymbol{\alpha}_q\) still relates graded components
which differ in weight by \(q-1\)---but there seem to jumps in certain
entries, likely due to jumps in cohomology of homogeneous bundles on
\(\PP W\).

\section{Smooth \texorpdfstring{\(q\)}{q}-bic threefolds}\label{section-threefolds-smooth}
Smooth \(q\)-bic threefolds and their Fano schemes of lines have
extraordinarily rich geometry, and it is here that the analogy with cubics
starts to become quite striking. One of the main results of this Section is an
analogue of a theorem of Clemens and Griffiths from \cite[Theorem 11.19]{CG}:
the intermediate Jacobian of a \(q\)-bic threefold \(X\) is purely inseparably
isogenous to the Albanese variety of its Fano surface \(S\) of lines, see
\parref{threefolds-smooth-intermediate-jacobian-result}.

Much of the Section is devoted to the fascinating geometry of the surface
\(S\). To start, \parref{threefolds-smooth-fano-lift} shows that \(S\) does not
lift to \(\mathrm{W}_2(\kk)\) whenever \(q > 2\), and so this surface is truly
a positive characteristic phenomenon. In terms of invariants: its \(\ell\)-adic
Betti numbers are computed in \parref{threefolds-smooth-betti-S}, and its
structure sheaf cohomology is computed in when \(q = p\) in
\parref{threefolds-cohomology-S-result} and
\parref{threefolds-cohomology-S-all}. Finally, some special divisors on \(S\)
are constructed and studied through
\parref{threefolds-smooth-incidence-divisors}--\parref{threefolds-smooth-line-incident-divisors}.

\subsection{Generalities}\label{threefolds-smooth-generalities}
Specializing the general facts about smooth \(q\)-bic hypersurfaces to the case
of threefolds gives the following:
\begin{enumerate}
\item\label{threefolds-smooth-generalities.U5(q)}
by \parref{hypersurfaces-automorphisms-linear}, \(X\) admits a linear action by
the finite unitary group \(\mathrm{U}_5(q)\);
\item\label{threefolds-smooth-generalities.H3}
by \parref{hypersurfaces-smooth-etale}, the middle \(\ell\)-adic
cohomology group \(\mathrm{H}^3_{\mathrm{\acute{e}t}}(X,\mathbf{Q}_\ell)\)
of \(X\) is an irreducible representation for \(\mathrm{U}_5(q)\) of dimension
\(q(q-1)(q^2+1)\);
\item\label{threefolds-smooth-generalities.S}
by \parref{hypersurfaces-smooth-fano}, the Fano scheme \(S\) of lines
is a smooth irreducible surface of general type with \(\omega_S \cong \sO_S(2q-3)\),
and moreover, by \parref{threefolds-lines-smooth-chern-numbers},
\(\mathcal{T}_S \cong \mathcal{S}^\vee \otimes \sO_S(1-q)\); and
\item\label{threefolds-smooth-generalities.plucker}
by \parref{hypersurfaces-fano-lines-degree}, the Pl\"ucker degree of \(S\) is
\(\deg(\sO_S(1)) = (q+1)^2(q^2+1)\).
\end{enumerate}

\subsection{Liftings}\label{threefolds-smooth-liftings}
Let \(\mathrm{W}_2(\kk)\) denote the second Witt vectors of the ground field
\(\kk\). Given a scheme \(Y\) over \(\kk\), a \emph{lift to
\(\mathrm{W}_2(\kk)\)} is a scheme \(\mathcal{Y}\), flat over
\(\mathrm{W}_2(\kk)\), such that
\(\mathcal{Y} \otimes_{\mathrm{W}_2(\kk)} \kk \cong Y\).
The next result shows that when \(q > 2\), the surface \(S\) does not lift
to \(\mathrm{W}_2(\kk)\), let alone characteristic \(0\); compare with the comments
preceding \parref{threefolds-lines-smooth-chern-numbers}. Note, in contrast,
that when \(q = 2\), \(S\) lifts to characteristic \(0\) because the canonical
map
\[ \Fr^*(W) \otimes W \hookrightarrow \Sym^3(W) \]
is an isomorphism for every \(2\)-dimensional vector space \(W\), and this means
that the equations of \(S\) in \(\mathbf{G}(2,V)\) lift.

\begin{Proposition}\label{threefolds-smooth-fano-lift}
If \(q > 2\), then \(S\) does not lift to \(\mathrm{W}_2(\kk)\).
\end{Proposition}

\begin{proof}
The sheaf of differentials of \(S\) is identified by
\parref{threefolds-lines-smooth-chern-numbers} as
\[
\Omega^1_S \cong
\mathcal{S} \otimes \sO_S(q-1) \cong
\mathcal{S}^\vee \otimes \sO_S(q-2)
\]
where the second isomorphism arises since \(\mathcal{S}\) is of rank \(2\)
so \(\mathcal{S} \cong \mathcal{S}^\vee \otimes \wedge^2 \mathcal{S}\). A
nonzero global section of \(\mathcal{S}^\vee\) induces a nonzero morphism
\(\sO_S(q-2) \to \Omega^1_S\). Thus, when \(q > 2\), \(\Omega^1_S\) contains
the ample line bundle \(\sO_S(q-2)\), and this implies \(S\) cannot lift to
\(\mathrm{W}_2(\kk)\) by Langer's analogue of Nakano--Akizuki--Kodaira
Vanishing, see \cite[Proposition 4.1]{Langer:BMY}, arising from work on the
Bogomolov--Miyaoka--Yau Inequality in positive characteristic.
\end{proof}

\subsection{Topology of the Fano scheme}\label{threefolds-smooth-topology-S}
The next few paragraphs discuss the topology of the Fano surface \(S\).
Throughout, \(\ell \neq p\) is a prime different from the characteristic.
Its \(\ell\)-adic cohomology groups are computed in
\parref{threefolds-smooth-betti-S} by relating \(S\) with certain
Deligne--Lusztig varieties of type \(\mathrm{A}^2_5\). Then in
\parref{threefolds-smooth-fano-S-X-H3}, the first cohomology group
\(\mathrm{H}^1_{\mathrm{\acute{e}t}}(S,\mathbf{Q}_\ell)\) is related to the
middle cohomology \(\mathrm{H}^3_{\mathrm{\acute{e}t}}(X,\mathbf{Q}_\ell)\)
of the \(q\)-bic threefold via the Fano correspondence.

The following specializes the constructions from
\parref{hypersurfaces-endomorphism}--\parref{hypersurfaces-filtration-rational-map-resolution}
to the case of lines on \(q\)-bic threefolds:

\begin{Proposition}\label{threefolds-smooth-X1-S}
There exists a commutative diagram
\[
\begin{tikzcd}[column sep=3em]
\tilde{X}^1 \rar["\phi \times \id"'] \dar
& \tilde{S} \dar \rar["\id \times \phi"']
& \tilde{X}^1 \dar \\
X^1 \rar[dashed,"\operatorname{cyc}_\phi^1"]
& S \rar[dashed,"{[\id \cap \phi]}"]
& X^1
\end{tikzcd}
\]
in which
\begin{enumerate}
\item\label{threefolds-smooth-X1-S.X1}
\(\tilde{X}^1 \to X^1\) is the blowup along the Hermitian points of \(X^1\),
\item\label{threefolds-smooth-X1-S.S}
\(\tilde{S} \to S\) is the blowup at the points corresponding to Hermitian
lines in \(X\), and
\item\label{threefolds-smooth-X1-S.homeo}
\(\phi \times \id\) and \(\id \times \phi\) are
finite purely inseparable of degree \(q^2\).
\end{enumerate}
\end{Proposition}

\begin{proof}
This is \parref{hypersurfaces-filtration-rational-map-resolution} upon noting
that \(S = S_{\mathrm{cyc}}\) by \parref{hypersurfaces-filtration-rational-map}.
\end{proof}

When \(X\) is the Fermat \(q\)-bic threefold, this diagram directly relates
\(S\) with a compactified Deligne--Lusztig variety of type \(\mathrm{A}_5^2\).
This connection may be used to determine the zeta function of \(S\):

\begin{Proposition}\label{threefolds-smooth-fermat-zeta}
Let \(X\) be the Fermat \(q\)-bic threefold. Then the zeta function of its
Fano scheme \(S\) of lines over \(\mathbf{F}_{q^2}\) is given by
\[
Z(S;t) =
\frac{(1 + qt)^{b_1} (1 + q^3t)^{b_3}}{(1 - t)(1 - q^2t)^{b_2}(1 - q^4t)}
\]
where \(b_1 = b_3 = q(q-1)(q^2+1)\) and \(b_2 = (q^4 - q^3 + 1)(q^2 + 1)\).
\end{Proposition}

\begin{proof}
First, by \parref{threefolds-smooth-X1-S}\ref{threefolds-smooth-X1-S.S} and
\parref{hypersurfaces-maximal-hermitian-subspaces},
\(\tilde{S} \to S\) is a blowup along the \((q^3+1)(q^5+1)\) points corresponding
to Hermitian lines in \(X\); these are precisely the \(\mathbf{F}_{q^2}\)-points
of \(S\) by \parref{forms-hermitian-fixed}\ref{forms-hermitian-fixed.subspace}.
Since \(Z(\PP^1;t)^{-1} = (1-t)(1-q^2t)\), this gives the first equality in
\[
Z(S;t) =
(1 - q^2t)^{(q^3+1)(q^5+1)} Z(\tilde{S};t) =
(1 - q^2t)^{(q^3+1)(q^5+1)} Z(\tilde{X}^1;t).
\]
The second equality follows from
\parref{threefolds-smooth-X1-S}\ref{threefolds-smooth-X1-S.homeo}, which
shows that \(\tilde{S}\) and \(\tilde{X}^1\) are related by finite purely
inseparable morphisms. Finally, by \parref{hypersurfaces-filtration-hermitian},
\[
X^1 = \Set{\begin{array}{r}
x_0^{q+1} + x_1^{q+1} + x_2^{q+1} + x_3^{q+1} + x_4^{q+1} = 0 \\
x_0^{q^3+1} + x_1^{q^3+1} + x_2^{q^3+1} + x_3^{q^3+1} + x_4^{q^3+1} = 0
\end{array}}
\subset \PP^4
\]
and \(\tilde{X}^1 \to X^1\) is the blowup along the \(\mathbf{F}_{q^2}\)-points
of \(X^1\). Therefore \(\tilde{X}^1\) is the Deligne--Lusztig variety
for \(\mathrm{U}_5(q)\), denoted by \(\overline{X}(s_1s_2)\) in
\cite[Th\'eor\`eme 8.1]{Rodier:DL}, and the result follows from the
computation given in \cite[Th\'eor\`eme 8.2]{Rodier:DL}.
\end{proof}

\begin{Corollary}\label{threefolds-smooth-betti-S}
For any smooth \(q\)-bic threefold \(X\), the \(\ell\)-adic Betti numbers of
its surface \(S\) of lines are given by
\[
\dim_{\mathbf{Q}_\ell}\mathrm{H}^i_{\mathrm{\acute{e}t}}(S,\mathbf{Q}_\ell) =
\begin{dcases*}
1 & if \(i = 0\) or \(i = 4\), \\
q(q-1)(q^2+1) & if \(i = 1\) or \(i = 3\), and \\
(q^4 - q^3 + 1)(q^2 + 1) & if \(i = 2\).
\end{dcases*}
\]
\end{Corollary}

\begin{proof}
This follows from \parref{threefolds-smooth-fermat-zeta}
since \(X\) is isomorphic to the Fermat by
\parref{hypersurfaces-smooth-projectively-equivalent}.
\end{proof}

\subsection{}\label{threefolds-smooth-fano-correspondence}
Consider the Fano correspondence, as in \parref{section-hypersurfaces-fano-correspondences}:
let \(\mathbf{L} \coloneqq \PP\mathcal{S}\) be the projective bundle associated
with the universal subbundle on \(S\) and let
\[
\begin{tikzcd}
& \mathbf{L} \ar[dl,"\pr_S"'] \ar[dr,"\pr_X"] \\
S && X
\end{tikzcd}
\]
be the incidence correspondence. By
\parref{hypersurfaces-fano-correspondences-smooth-fibres}
and \parref{hypersurfaces-fano-correspondences-degree}, the morphism
\(\pr_X \colon \mathbf{L} \to X\) is generically finite of degree \(q(q+1)\),
inseparable of degree \(q\), and its positive dimensional fibres are smooth
\(q\)-bic curves lying over the Hermitian points of \(X\).

As in \parref{hypersurfaces-fano-correspondences-action}, view \(\mathbf{L}\)
as a degree \(-1\) correspondence from \(S\) to \(X\), so that \(\mathbf{L}\)
acts on Chow rings and \(\ell\)-adic cohomology. The results established so far
imply that the cohomological action identifies the middle cohomology of \(X\)
with the first cohomology of \(S\):

\begin{Proposition}\label{threefolds-smooth-fano-S-X-H3}
The morphism
\[
\mathbf{L}^* \colon
\mathrm{H}^3_{\mathrm{\acute{e}t}}(X,\mathbf{Q}_\ell) \to
\mathrm{H}^1_{\mathrm{\acute{e}t}}(S,\mathbf{Q}_\ell)
\]
is an isomorphism of \(\mathrm{U}_5(q)\)-representations.
\end{Proposition}

\begin{proof}
The action of \(\mathbf{L}\) is \(\mathrm{U}_5(q)\)-equivariant and injective
by \parref{hypersurfaces-fano-correspondences-injective}. Comparing
\parref{threefolds-smooth-generalities}\ref{threefolds-smooth-generalities.H3}
and \parref{threefolds-smooth-betti-S} shows that the source and target spaces
are of the same dimension, whence the map is an isomorphism.
\end{proof}

\subsection{Incidence divisors}\label{threefolds-smooth-incidence-divisors}
The next few paragraphs study two families of divisors on \(S\), each
defined by certain incidence conditions:
\begin{itemize}
\item the curve \(C_\ell\) of lines incident with a fixed line
\(\ell \subset X\), and
\item the curve \(C_x\) of lines incident with a fixed Hermitian point
\(x \in X\).
\end{itemize}
The former set of curves was constructed in
\parref{hypersurfaces-fano-correspondences-incidence}. The latter set of curves
was constructed in \parref{threefolds-cone-situation-C}, since Smooth Cone
Situations of \(X\) are precisely given by Hermitian points of \(X\), see
\parref{hypersurfaces-cone-points-smooth}. In particular, the \(C_x\) are
smooth \(q\)-bic curves.

When \(\ell \subset X\) is a Hermitian line, the curve \(C_\ell\) is made up
of the curves \(C_x\), with \(x\) ranging over the Hermitian points of \(X\)
contained in \(\ell\):

\begin{Lemma}\label{threefolds-smooth-hermitian-line-divisor}
Let \(\ell \subset X\) be a Hermitian line. Then
\[
C_\ell = \sum\nolimits_{x \in \ell \cap X_{\mathrm{Herm}}} C_x
\]
as divisors on \(S\). The sum ranges over the Hermitian points of \(X\)
contained in \(\ell\).
\end{Lemma}

\begin{proof}
Let \(D\) be the sum appearing on the right. It is clear from definitions and
symmetry that \(C_\ell\) contains \(D\) to some positive multiplicity \(a\).
Consider the Pl\"ucker degree of the two divisors. On the one hand, by
\parref{hypersurfaces-fano-correspondences-D} and
\parref{threefolds-smooth-generalities}\ref{threefolds-smooth-generalities.plucker},
\[
\deg(\sO_S(1)\rvert_{C_\ell}) = (q+1)^{-1}\deg(\sO_S(1)) = (q+1)(q^2+1)
\]
On the other hand,
\parref{threefolds-cone-situation-C} implies that \(\sO_S(1)\rvert_{C_x}\) is
the usual polarization of \(C_x\) as a plane curve. Since \(\ell\) contains
\(q^2+1\) Hermitian points, this shows that \(D\) is the same degree as \(C_\ell\)
and so the two divisors coincide.
\end{proof}

The curves parameterizing lines through Hermitian points only intersect if they
lie on a common Hermitian line, in which case they intersect transversely:

\begin{Lemma}\label{threefolds-smooth-hermitian-points-intersect}
Let \(x, y \in X\) be distinct Hermitian points. Then, as \(0\)-cycles on \(S\),
\[
[C_x] \cdot [C_y] =
\begin{dcases*}
[\ell] & if \(x\) and \(y\) lie on a Hermitian line \(\ell \subset X\), and \\
0 & otherwise.
\end{dcases*}
\]
\end{Lemma}

\begin{proof}
Let \(\ell \coloneqq \langle x, y \rangle\) be the unique line in \(\PP V\)
passing through both \(x\) and \(y\). Then the \(3\)-planes \(\mathbf{G}(2,V;x)\)
and \(\mathbf{G}(2,V;y)\) in \(\mathbf{G}(2,V)\) parameterizing lines in
\(\PP V\) through \(x\) and \(y\), respectively, intersect at the unique point
\([\ell]\). Since these \(3\)-planes are linearly embedded into \(\mathbf{G}(2,V)\),
as explained in the proof of \parref{threefolds-cone-situation-C}, they
intersect with multiplicity \(1\) at \([\ell]\), and so
\[
[C_x] \cdot [C_y] =
\big([\mathbf{G}(2,V;x)] \cdot [\mathbf{G}(2,V;y)]\big)\big\rvert_{S} =
\begin{dcases*}
[\ell] & if \(\ell \subset X\), and \\
0 & otherwise.
\end{dcases*}
\]
It remains to observe that, by definition, any line with two distinct Hermitian
points is a Hermitian line, see \parref{forms-hermitian}.
\end{proof}

The divisor \(C_\ell\), with \(\ell \subset X\) Hermitian, intersects
each of its irreducible components in a \(0\)-cycle of degree \(1\):

\begin{Lemma}\label{threefolds-smooth-hermitian-degree}
Let \(\ell \subset X\) be a Hermitian line and \(x \in \ell\) a Hermitian point. Then
\[
\sO_S(C_\ell)\rvert_{C_x} \cong \sO_{C_x}([\ell]).
\]
In particular, \(\deg([C_\ell] \cdot [C_x]) = 1\).
\end{Lemma}

\begin{proof}
By \parref{threefolds-smooth-hermitian-line-divisor}, \(C_\ell\) can be
written as
\[
C_\ell =
C_x +
\sum\nolimits_{y \in \ell \cap X_{\mathrm{Herm}} \setminus \{x\}} C_y
\]
where the second sum ranges over the \(q^2\) Hermitian points of \(X\) in
\(\ell\) different from \(x\).
Applying \parref{threefolds-smooth-hermitian-points-intersect} then gives
the first equality in
\[
\sO_S(C_\ell)\rvert_{C_x} \cong
\mathcal{N}_{C_x/S} \otimes \sO_{C_x}(q^2[\ell]) \cong
\sO_{C_x}(-q+1) \otimes \sO_{C_x}(q^2[\ell]) \cong
\sO_{C_x}([\ell]).
\]
By \parref{threefolds-cone-situation-C}\ref{threefolds-cone-situation-C.normal},
\(\mathcal{N}_{C_x/S} \cong \sO_{C_x}(-q+1)\). Since \([\ell]\) is a
Hermitian point of \(C_x\), it follows from
\parref{curve-residual-intersection} that
\(\sO_{C_x}(1) \cong \sO_{C_x}((q+1)[\ell])\). This gives the latter two
isomorphisms above, from which the result follows.
\end{proof}

In general, whenever \(\ell \subset X\) contains at least one Hermitian point
of \(X\), the curve \(C_\ell\) contains \(C_x\) as an irreducible component,
and the residual component intersects \(C_x\) only at the point \([\ell]\).
The following establishes this, and a little more:

\begin{Proposition}\label{threefolds-smooth-one-hermitian-point}
Let \(\ell \subset X\) be a line containing a Hermitian point \(x\).
Then there exists an effective divisor \(C_\ell'\) such that
\begin{enumerate}
\item\label{threefolds-smooth-one-hermitian-point.sum}
\(C_\ell = C_\ell' + C_x\) as divisors on \(S\),
\item\label{threefolds-smooth-one-hermitian-point.components}
\(C_\ell'\) does not contain \(C_x\) as an irreducible component, and
\item\label{threefolds-smooth-one-hermitian-point.intersections}
for every Hermitian point \(y\) of \(X\), as \(0\)-cycles on \(S\),
\[
[C_\ell'] \cdot [C_y] =
\begin{dcases*}
q^2 [\ell] & if \(x = y\), and \\
0 & if \(x \neq y\).
\end{dcases*}
\]
\end{enumerate}
\end{Proposition}

\begin{proof}
If \(\ell\) is itself Hermitian, then \ref{threefolds-smooth-one-hermitian-point.sum}
and \ref{threefolds-smooth-one-hermitian-point.components} follow from
\parref{threefolds-smooth-hermitian-line-divisor}, and
\ref{threefolds-smooth-one-hermitian-point.intersections} follows from
\parref{threefolds-smooth-hermitian-points-intersect} and the computation in
the proof of \parref{threefolds-smooth-hermitian-degree}. So for the remainder
of the proof, assume that \(x\) is the unique Hermitian point of \(X\)
contained in \(\ell\).

By its construction from \parref{hypersurfaces-fano-correspondences-incidence},
\(C_\ell\) is the locus in \(S\) of lines that are incident with \(\ell\).
Since \(x \in \ell\), it follows that \(C_x\) is an irreducible component
of \(C_\ell\); also, \(C_y\) is not contained in \(C_\ell\) for any Hermitian
point \(y \neq x\). This already shows that the divisor
\(C_\ell' \coloneqq C_\ell - C_x\) is effective, proving
\ref{threefolds-smooth-one-hermitian-point.sum}.

Choose a Hermitian point \(y\) different from \(x\) such that \(\langle x,y
\rangle \subset X\). On the one hand, \([C_\ell'] \cdot [C_y]\) is an effective
\(0\)-cycle on \(S\). On the other hand,
\begin{align*}
\deg([C_\ell'] \cdot [C_y])
& = \deg([C_\ell] \cdot [C_y]) - \deg([C_x] \cdot [C_y]) \\
& = \deg([C_{\langle x,y \rangle}] \cdot [C_y]) - \deg([C_x] \cdot [C_y]) = 0
\end{align*}
where the first equality follows from
\ref{threefolds-smooth-one-hermitian-point.sum},
the second from the fact
\parref{hypersurfaces-fano-correspondences-D} that the curves \(C_\ell\) and
\(C_{\langle x,y\rangle}\) are algebraically equivalent, and
the third from the computations of
\parref{threefolds-smooth-hermitian-points-intersect} and
\parref{threefolds-smooth-hermitian-degree}. Therefore
\([C_\ell'] \cdot [C_y] = 0\) and, by
\parref{threefolds-smooth-hermitian-points-intersect}, \(C_\ell'\) does not
contain \(C_x\) as an irreducible component, proving
\ref{threefolds-smooth-one-hermitian-point.components}. In fact, this argument
applied to any Hermitian point \(y \neq x\) proves the second case of
\ref{threefolds-smooth-one-hermitian-point.intersections}.

It remains to consider the \(0\)-cycle \([C_\ell'] \cdot [C_x]\). By
\ref{threefolds-smooth-one-hermitian-point.components}, this is an effective
\(0\)-cycle. Since the only line that is incident with \(\ell\) and passes
through \(x\) is \(\ell\) itself, it is supported on \([\ell]\).
In other words, there is a nonnegative integer \(a\) such that
\[ [C_\ell'] \cdot [C_x] = a[\ell] \in \mathrm{CH}_0(S). \]
The multiplicity is now determined as above: choosing a Hermitian line
\(\ell_0 \subset X\) containing \(x\), it follows that
\begin{align*}
a =
\deg([C_\ell'] \cdot [C_x]) =
\deg([C_{\ell_0}'] \cdot [C_x]) =
\sum\nolimits_{y \in \ell_0 \cap X_{\mathrm{Herm}} \setminus \{x\}} \deg([C_y] \cdot [C_x])
= q^2.
\end{align*}
by \ref{threefolds-smooth-one-hermitian-point.sum},
\parref{hypersurfaces-fano-correspondences-D},
\parref{threefolds-smooth-hermitian-line-divisor}, and
\parref{threefolds-smooth-hermitian-points-intersect}.
\end{proof}

\begin{Remark}\label{threefolds-smooth-line-incident-divisors}
When \(\ell\) contains a single Hermitian point, further properties of the
divisor \(C_\ell'\) might be understood by relating it to the discriminant
divisor \(D_1\) associated with the \(q\)-bic curve fibration \(X' \to \PP W\)
obtained by linear projection of \(\PP V\) away from \(\ell\): see
\parref{hypersurfaces-unirational-discriminant-incidence}. The same methods
may be used in the case \(\ell\) does not contain any Hermitian points to study
the divisor \(C_\ell\).
\end{Remark}

The next few paragraphs relate a certain abelian variety \(\mathbf{Ab}_X^2\)
attached to the smooth \(q\)-bic threefold \(X\), referred to as the
\emph{intermediate Jacobian of \(X\)}, with the Albanese variety of \(S\):
see \parref{threefolds-smooth-intermediate-jacobian-result}. This is done by
using the Fano correspondence from
\parref{threefolds-smooth-fano-correspondence} to relate codimension \(2\)
cycles of \(X\) with divisors of \(S\), and by using the computations of
\parref{threefolds-smooth-incidence-divisors}--\parref{threefolds-smooth-one-hermitian-point}.
The definitions below follow \cite[\S3.2]{Beauville:Prym} and
\cite[\S\S1.5--1.8]{Murre:Jacobian}.

\subsection{Algebraic representatives}\label{threefolds-smooth-algebraic-representative}
Let \(Y\) a smooth projective variety over an algebraically closed field
\(\kk\) and let \(A\) be an abelian variety over \(\kk\). Given an
integer \(0 \leq k \leq \dim Y\), a homomorphism of groups
\[
\phi \colon
\mathrm{CH}^k(Y)_{\mathrm{alg}} \to
A(\kk)
\]
is said to be \emph{regular} if for every pointed smooth projective variety
\((T,t_0)\) over \(\kk\), and every cycle class
\(Z \in \mathrm{CH}^k(T \times Y)\), the map
\[
T(\kk) \to \mathrm{CH}^k(Y)_{\mathrm{alg}} \to A(\kk),
\qquad
t \mapsto \phi(Z_t - Z_{t_0})
\]
is induced by a morphism \(T \to A\) of varieties over \(\kk\). A regular
homomorphism
\[
\phi^k_Y \colon
\mathrm{CH}^k(Y)_{\mathrm{alg}} \to
\mathbf{Ab}_Y^k(\kk)
\]
that is initial is called an \emph{algebraic representative} for codimension
\(k\) cycles in \(Y\).

\begin{Theorem}\label{threefolds-smooth-intermediate-jacobian-exists}
Let \(Y\) be a smooth projective variety of dimension \(d\) over \(\kk\). Then
an algebraic representative for codimension \(k\) cycles exists when
\begin{enumerate}
\item\label{threefolds-smooth-intermediate-jacobian-exists.alb}
\(k = d\), and \(\mathbf{Ab}_Y^0 = \mathbf{Alb}_Y\),
\item\label{threefolds-smooth-intermediate-jacobian-exists.pic}
\(k = 1\), and \(\mathbf{Ab}_Y^1 = \mathbf{Pic}_{Y,\mathrm{red}}^0\), and
\item\label{threefolds-smooth-intermediate-jacobian-exists.jac}
\(k = 2\), and
\(2\dim \mathbf{Ab}_Y^2 \leq \dim_{\mathbf{Q}_\ell} \mathrm{H}^3_{\mathrm{\acute{e}t}}(Y,\mathbf{Q}_\ell)\).
\end{enumerate}
\end{Theorem}

\begin{proof}
For \ref{threefolds-smooth-intermediate-jacobian-exists.alb}
and \ref{threefolds-smooth-intermediate-jacobian-exists.pic}, see
\cite[\S1.4]{Murre:Jacobian}.
For \ref{threefolds-smooth-intermediate-jacobian-exists.jac},
see \cite{Murre:Jacobian-CR} or \cite[Theorem 1.9]{Murre:Jacobian}, along with
a correction by \cite{Kahn:Jacobian}.
\end{proof}

Returning to the situation of a smooth \(q\)-bic threefold \(X\), its
algebraic representative \(\mathbf{Ab}_X^2\) in codimension \(2\) from
\parref{threefolds-smooth-intermediate-jacobian-exists}\ref{threefolds-smooth-intermediate-jacobian-exists.jac}
is referred to as its \emph{intermediate Jacobian}. The next results
relate the intermediate Jacobian of \(X\) with the Albanese variety of \(S\):

\begin{Lemma}\label{threefolds-smooth-intermediate-jacobian}
There exists a commutative diagram of abelian groups
\[
\begin{tikzcd}
\mathrm{CH}^2(S)_{\mathrm{alg}} \rar["\mathbf{L}_*"'] \dar["\phi_S^2"]
& \mathrm{CH}^2(X)_{\mathrm{alg}} \rar["\mathbf{L}^*"'] \dar["\phi_X^2"]
& \mathrm{CH}^1(S)_{\mathrm{alg}} \dar["\phi_S^1"] \\
\mathbf{Alb}_{S}(\kk) \rar
& \mathbf{Ab}^2_X(\kk) \rar
& \mathbf{Pic}^0_S(\kk)
\end{tikzcd}
\]
and hence morphisms of abelian varieties
\[
\mathbf{Alb}_S \xrightarrow{\mathbf{L}_*}
\mathbf{Ab}_X^2 \xrightarrow{\mathbf{L}^*}
\mathbf{Pic}^0_{S,\mathrm{red}}.
\]
\end{Lemma}

\begin{proof}
The action of the Fano correspondence gives the top row of maps, see
\parref{hypersurfaces-fano-correspondences-action}. The vertical maps are the
universal regular homomorphisms recognizing the Albanese variety of \(S\), the
intermediate Jacobian of \(X\), and the Picard scheme of \(S\) as,
respectively, the algebraic representatives for algebraically trivial
\(0\)-cycles of \(S\), \(1\)-cycles of \(X\), and \(1\)-cycles of \(S\); see
\parref{threefolds-smooth-intermediate-jacobian-exists}. The morphisms of the
group schemes arise from the corresponding universal property of each scheme.
\end{proof}

Fix a Hermitian line \(\ell_0 \subset X\) and consider the Albanese morphism
\(\mathrm{alb}_S \colon S \to \mathbf{Alb}_S\) centred at \([\ell_0]\).
Composing this with the morphisms of abelian schemes from
\parref{threefolds-smooth-intermediate-jacobian} yields a morphism
\(S \to \mathbf{Pic}_S^0\). Its action on \(\kk\)-points is easily understood:

\begin{Lemma}\label{threefolds-smooth-intermediate-jacobian-S-Pic}
The morphism \(S \to \mathbf{Pic}_S^0\) acts on \(\kk\)-points by
\[ [\ell] \mapsto \sO_S(C_\ell - C_{\ell_0}). \]
\end{Lemma}

\begin{proof}
The Albanese morphism on \(\kk\)-points factorizes as
\[
\phi_S^2 \circ \mathrm{alb}_S(\kk) \colon
S(\kk) \to
\mathrm{CH}^2(S)_{\mathrm{alg}} \to
\mathbf{Alb}_S(\kk)
\]
where the first map is \([\ell] \mapsto [\ell] - [\ell_0]\), and the second
map is the universal regular homomorphism from
\parref{threefolds-smooth-intermediate-jacobian-exists}\ref{threefolds-smooth-intermediate-jacobian-exists.alb}.
The commutative diagram of \parref{threefolds-smooth-intermediate-jacobian}
then shows that \(S \to \mathbf{Pic}_S^0\) factors through the map
\(S(\kk) \to \mathrm{CH}^1(S)_{\mathrm{alg}}\) given by
\[
[\ell] \mapsto \mathbf{L}^*\mathbf{L}_*([\ell] - [\ell_0]) = [C_\ell] - [C_{\ell_0}],
\]
where the action of the correspondence is as in
\parref{hypersurfaces-fano-correspondences-action} and
\parref{hypersurfaces-fano-correspondences-incidence}.
Composing with the map
\(\phi^1_S \colon \mathrm{CH}^1(S)_{\mathrm{alg}} \to \mathbf{Pic}^0_S(\kk)\)
gives the result.
\end{proof}

Let \(C_{\ell_0}\) be the divisor in \(S\) paramterizing lines in \(X\)
incident with the fixed Hermitian line \(\ell_0\), as in
\parref{threefolds-smooth-incidence-divisors}. By
\parref{threefolds-smooth-hermitian-line-divisor},
its normalization \(C \to C_{\ell_0}\) is given by
\[
C = \coprod\nolimits_{x \in \ell_0 \cap X_{\mathrm{Herm}}} C_x,
\]
the \(1\)-dimensional scheme obtained as the disjoint union over the smooth
\(q\)-bic curves \(C_x\) parameterizing lines through the Hermitian points
\(x\) of \(X\) contained in \(\ell_0\), and the morphism to \(C_{\ell_0}\) is
the disjoint union of the canonical inclusions. Let \(\nu \colon C \to
C_{\ell_0} \hookrightarrow S\) be the composite of the normalization morphism
\(C \to C_{\ell_0}\) with the closed immersion
\(C_{\ell_0} \hookrightarrow S\). Let
\[
\mathbf{Jac}_C \coloneqq
\prod\nolimits_{x \in \ell_0 \cap X_{\mathrm{Herm}}} \mathbf{Jac}_{C_x}
\]
denote the product of the Jacobian varieties of the connected components
of \(C\). On the one hand, viewing the Jacobian as the Picard variety of
gives a map \(\nu^* \colon \mathbf{Pic}_S^0 \to \mathbf{Jac}_C\). On the
other hand, viewing \([\ell_0] \in C_x\) as a base point for each connected
component of \(C\) and taking the product of the corresponding Albanese
morphisms gives a map \(\mathrm{alb}_C \colon C \to \mathbf{Jac}_C\).
The universal property of the Albanese then gives a commutative diagram
\[
\begin{tikzcd}
C \rar["\nu"'] \dar["\mathrm{alb}_C"'] & S \dar["\mathrm{alb}_S"] \\
\mathbf{Jac}_C \rar["\nu_*"] & \mathbf{Alb}_S
\end{tikzcd}
\]
of morphisms of schemes over \(\kk\). Composing these morphisms to and from
\(\mathbf{Jac}_C\) with those of
\parref{threefolds-smooth-intermediate-jacobian} gives an endomorphism of
\(\mathbf{Jac}_C\), and it is determined as follows:

\begin{Proposition}\label{threefolds-smooth-intermediate-jacobian-multiplication}
The composite morphism
\[
\Phi \colon
\mathbf{Jac}_C \xrightarrow{\nu_*}
\mathbf{Alb}_S \xrightarrow{\mathbf{L}_*}
\mathbf{Ab}_X^2 \xrightarrow{\mathbf{L}^*}
\mathbf{Pic}^0_{S,\mathrm{red}} \xrightarrow{\nu^*}
\mathbf{Jac}_C
\]
is given by multiplcation by \(q^2\).
\end{Proposition}

\begin{proof}
Consider a \(\kk\)-point of a connected component \(C_x\) of \(C\) and identify
it with a \(\kk\)-point \([\ell]\) of its image in \(S\). Then by
\parref{threefolds-smooth-intermediate-jacobian-S-Pic}, its image under
\(\Phi \circ \mathrm{alb}_C \colon C \to \mathbf{Jac}_C\) is
\[ \Phi(\mathrm{alb}_C([\ell])) = \nu^*\sO_S(C_\ell - C_{\ell_0}). \]
Consider the restriction of \(\sO_S(C_\ell - C_{\ell_0})\) to each component
\(C_y\) of \(C\), where \(y \in \ell_0 \cap X_{\mathrm{Herm}}\). Being a point
of \(C_x\), the line \(\ell\) contains the Hermitian point \(x\).
So by \parref{threefolds-smooth-one-hermitian-point}, there is some effective
divisor \(C_\ell'\) such that \(C_\ell = C_\ell' + C_x\) and,
\[
[C_\ell'] \cdot [C_y] =
\begin{dcases*}
q^2[\ell] & if \(x = y\), and \\
0 & if \(x \neq y\),
\end{dcases*}
\]
in \(\mathrm{CH}_0(S)\). Thus writing
\[
C_\ell - C_{\ell_0} =
C_\ell' - \sum\nolimits_{z \in \ell_0 \cap X_{\mathrm{Herm}} \setminus \{x\}} C_z
\]
it follows from \parref{threefolds-smooth-hermitian-points-intersect} and
\parref{threefolds-smooth-hermitian-degree} that
\[
\sO_S(C_\ell - C_{\ell_0})\rvert_{C_y} \cong
\begin{dcases*}
\sO_{C_x}(q^2([\ell] - [\ell_0])) & if \(x = y\), and \\
\sO_{C_y} & if \(x \neq y\).
\end{dcases*}
\]
Therefore \(\Phi(\mathrm{alb}_C([\ell])) = q^2\mathrm{alb}_C([\ell])\). Since
the image of \(C\) under \(\mathrm{alb}_C\) generates \(\mathbf{Jac}_C\),
\(\Phi\) is given by multiplication by \(q^2\) on all of \(\mathbf{Jac}_C\).
\end{proof}

Putting everything together shows that each of the abelian varieties in
question are related to one another via purely inseparable isogenies:

\begin{Theorem}\label{threefolds-smooth-intermediate-jacobian-result}
Each of the morphisms of abelian varieties
\[
\nu_* \colon \mathbf{Jac}_C \to \mathbf{Alb}_S, \quad
\mathbf{L}_* \colon \mathbf{Alb}_S \to \mathbf{Ab}_X^2,
\quad
\mathbf{L}^* \colon \mathbf{Ab}_X^2 \to \mathbf{Pic}_{S,\mathrm{red}}^0,
\quad
\nu^* \colon \mathbf{Pic}^0_{S,\mathrm{red}} \to \mathbf{Jac}_C
\]
is a purely inseparable \(p\)-power isogeny.
\end{Theorem}

\begin{proof}
The map \(\Phi\) in \parref{threefolds-smooth-intermediate-jacobian-multiplication}
is multiplication by \(q^2\), so its kernel is finite. This implies that
the image of the partial composites from \(\mathbf{Jac}_C\) to each of
\(\mathbf{Alb}_S\), \(\mathbf{Ab}_X^2\), and \(\mathbf{Pic}^0_{S,\mathrm{red}}\)
are abelian subvarieties of dimension
\[ \dim\mathbf{Jac}_C = q(q-1)(q^2+1)/2. \]
By
\parref{threefolds-smooth-intermediate-jacobian-exists}\ref{threefolds-smooth-intermediate-jacobian-exists.jac}
together with \parref{threefolds-smooth-generalities}\ref{threefolds-smooth-generalities.H3},
\[
2\dim \mathbf{Ab}_X^2 \leq
\dim_{\mathbf{Q}_\ell} \mathrm{H}^3_{\mathrm{\acute{e}t}}(X,\mathbf{Q}_\ell) =
q(q-1)(q^2+1)
\]
so equality holds throughout.
Since the Albanese and Picard varieties have dimension
\(\dim_{\mathbf{Q}_\ell} \mathrm{H}^1_{\mathrm{\acute{e}t}}(S,\mathbf{Q}_\ell)\),
comparing with \parref{threefolds-smooth-betti-S} shows that that each of
\(\mathbf{Ab}_X^2\), \(\mathbf{Alb}_S\), \(\mathbf{Pic}_{S,\mathrm{red}}^0\),
and \(\mathbf{Jac}_C\) are abelian varieties of dimension \(q(q-1)(q^2+1)/2\).
It now follows that all the maps factoring \(\Phi\) are \(p\)-power isogenies.

It remains to see that all the maps occurring are purely inseparable. For this,
recall from \parref{hypersurfaces-smooth-zeta} that smooth \(q\)-bic curves
are supersingular. Thus \(\mathbf{Jac}_C\), being the product of Jacobians of
smooth \(q\)-bic curves, is itself supersingular. Whence multiplication by
\(q^2\) is purely inseparable. The result now follows from the factorization
\parref{threefolds-smooth-intermediate-jacobian-multiplication}.
\end{proof}

That \(\mathbf{L}_* \colon \mathbf{Alb}_S \to \mathbf{Ab}_X^2\) is purely
inseparable is analogous to the fact that for a cubic threefold, its
intermediate Jacobian is isomorphic to the Albanese of its surface of lines,
see \cite[Theorem 11.19]{CG}. The abelian variety \(\mathbf{Jac}_C\) is a
degenerate analogue of a Prym variety for the covering \(C \to D\), where
\(D\) is the discriminant curve to projection of \(X\) from \(\ell_0\), see
\parref{hypersurfaces-unirational-discriminant-incidence} and
\parref{hypersurfaces-unirationality-discriminant-type}\ref{hypersurfaces-unirationality-discriminant-type.two-hermitian}.
Thus the statement that \(\nu_* \colon \mathbf{Jac}_C \to \mathbf{Alb}_S\)
is an analogue of Mumford's identification between the Albanese of the Fano
surface for a cubic with a Prym variety, see \cite[Appendix C]{CG}. The
statement regarding
\(\nu^* \circ \mathbf{L}^* \colon \mathbf{Ab}^2_X \to \mathbf{Jac}_C\) is
an analogue of Murre's identification between the group of algebraically
trivial \(1\)-cycles on a cubic and the Prym, see \cite[Theorem 10.8]{Murre:Prym}.

\subsection{Coherent cohomology of the Fano scheme}\label{threefolds-smooth-coherent}
The next goal is to compute the cohomology of \(\sO_S\), at least when
\(q = p\) is prime: see \parref{threefolds-cohomology-S-result} and
\parref{threefolds-cohomology-S-all} for the result. The computation
proceeds in three steps:

First, carefully degenerate \(S\) to the singular surface \(S_0\) of lines in
type \(\mathbf{N}_2 \oplus \mathbf{1}^{\oplus 3}\) and relate the cohomology of
\(S\) with that of \(S_0\): see \parref{threefolds-cohomology-S-setup} and
\parref{threefolds-cohomology-S-Gm-sequences}.

Second, explicitly show that certain cohomology classes in
\(\mathrm{H}^1(S_0,\sO_{S_0})\) do not lift to classes in
\(\mathrm{H}^1(S,\sO_S)\): see \parref{threefolds-cohomology-S-not-split}.

Third, apply uppersemicontinuity of cohomology and use the computation of
\(\mathrm{H}^1(S_0,\sO_{S_0})\) from \parref{threefolds-nodal-fano-cohomology}
to conclude.

\subsection{}\label{threefolds-cohomology-S-setup}
Choose cone points \(x_-, x_+ \in X\) such that
\(\langle x_-, x_+ \rangle \not\subset X\) and let
\(\pi \colon \mathcal{X} \to \mathbf{A}^1\) be the associated family of
\(q\)-bic threefolds in \(\PP V \times \mathbf{A}^1\) as constructed in
\parref{threefolds-smooth-cone-situation-family}, so that \(\pi\) is smooth
away from \(0 \in \mathbf{A}^1\) and \(X_0 \coloneqq \pi^{-1}(0)\) is of type
\(\mathbf{N}_2 \oplus\mathbf{1}^{\oplus 3}\). Let
\(\mathcal{S} \to \mathbf{A}^1\) be the relative Fano scheme of lines of
\(\pi\), and let
\[
\begin{tikzcd}[row sep=0.75em, column sep=.5em]
& \tilde{\mathcal{S}} \ar[dl] \ar[dr] \\
\mathcal{S} \ar[dr] && \mathcal{T} \ar[dl] \\
& C_{\mathbf{A}^1}
\end{tikzcd}
\]
be the diagram resulting from the family of Smooth Cone Situations
\((\mathcal{X},x_-)\) as in
\parref{threefolds-smooth-cone-situation-family-fano-setup} and
\parref{threefolds-smooth-cone-situation-family-fano}.
Let \(\varphi \colon S \to C\) and \(\varphi_0 \colon S_0 \to C\)
be the fibres of the morphism \(\mathcal{S} \to C_{\mathbf{A}^1}\)
over the points \(1\) and \(0\) of \(\mathbf{A}^1\), respectively.
Then the sheaves \(\mathbf{R}^1\varphi_*\sO_S\) and
\(\mathbf{R}^1\varphi_{0,*}\sO_{S_0}\) are locally free \(\sO_C\)-modules
which carry a \(q\)-step filtration by
\parref{threefolds-smooth-cone-situation-pushforward}\ref{threefolds-smooth-cone-situation-pushforward.R1};
they also admit gradings by \(\mathbf{Z}/(q^2-1)\mathbf{Z}\) and
\(\mathbf{Z}\), respectively, see \parref{threefolds-smooth-cone-situation-rees}.
Moreover, by the proof of \parref{normalize-splitting},
\(\mathbf{R}^1\varphi_{0,*}\sO_{S_0}\) consists of the positively graded parts
of the sheaf \(\mathcal{F}\) associated with \(\varphi_0\) as in
\parref{nodal-conductors-F}.

With this notation, the following refines the relationship between
\(\mathbf{R}^1\varphi_*\sO_S\) and \(\mathbf{R}^1\varphi_{0,*}\sO_{S_0}\)
given in \parref{threefolds-smooth-cone-situation-rees}:

\begin{Lemma}\label{threefolds-cohomology-S-Gm-sequences}
In the setting of \parref{threefolds-cohomology-S-setup},
there are short exact sequences of filtered \(\mathrm{U}_3(q)\)-equivariant
locally free \(\sO_C\)-modules
\[
0 \to
\mathcal{F}_\alpha \to
(\mathbf{R}^1\varphi_*\sO_S)_\alpha \to
\mathcal{F}_{\alpha+q^2-1} \to
0
\]
for each \(\alpha = 1,2,\ldots,q^2-1\).
\end{Lemma}

\begin{proof}
By \parref{threefolds-smooth-cone-situation-rees},
the sheaves \((\mathbf{R}^1\varphi_*\sO_{\mathcal{S}})_\alpha\) carry
a second filtration \(\Fil_\bullet\) such that
\[
\gr^{\Fil}_i(\mathbf{R}^1\varphi_*\sO_{S})_\alpha \cong
(\mathbf{R}^1\varphi_{0,*}\sO_{S_0})_{\alpha + i(q^2-1)}
\quad\text{for each}\; i \in \mathbf{Z}
\]
as graded \(\sO_C\)-modules. Identify \(\mathbf{R}^1\varphi_{0,*}\sO_{S_0}\)
with the positively graded components of \(\mathcal{F}\) via the proof of
\parref{normalize-splitting}. Since, by \parref{threefolds-cohomology-duality},
the weights appearing in \(\mathcal{F}_{>0}\) lie in \([1,2q^2-q-2]\), so the
filtration above has at most \(2\) steps, yielding the desired short exact
sequences; they are equivariant for \(\mathrm{U}_3(q)\) by
\parref{threefolds-smooth-cone-situation-family-fano}\ref{threefolds-smooth-cone-situation-family-fano.actions}.
\end{proof}

Consider the indices \(\alpha = iq\) with \(1 \leq i \leq q-2\). In this case,
\parref{threefolds-cohomology-duality} combined with
\parref{threefolds-cohomology-D} identifies the quotient bundle as
\[ \mathcal{F}_{q^2 + iq - 1} \cong \Div^{q-2-i}(W) \otimes \sO_C. \]
When \(q = p\), this makes it easy to show its corresponding exact sequence
is not split:

\begin{Lemma}\label{threefolds-cohomology-S-not-split}
If \(q = p\), then for each \(1 \leq i \leq p-2\), the sequence
\[
0 \to
\mathcal{F}_{ip} \to
(\mathbf{R}^1\varphi_*\sO_S)_{ip} \to
\Div^{p-2-i}(W) \otimes \sO_C \to 0
\]
is not split and the induced map on global sections yields an isomorphism
\[
\mathrm{H}^0(C,\mathcal{F}_{ip}) \cong
\mathrm{H}^0(C,(\mathbf{R}^1\varphi_*\sO_S)_{ip}).
\]
\end{Lemma}

\begin{proof}
By \parref{threefolds-cohomology-S-Gm-sequences}, the sequence
is of \(\mathrm{U}_3(q)\)-equivariant filtered \(\sO_C\)-modules, in
which the filtration is as from
\parref{threefolds-cone-situation-S-T}\ref{threefolds-cone-situation-S-T.sequence}.
The result will follow upon showing
\begin{enumerate}
\item\label{threefolds-cohomology-S-not-split.nonvanishing}
\(\mathrm{gr}_{p-1}(\Div^{p-2-i}(W) \otimes \sO_C) \neq 0\), and
\item\label{threefolds-cohomology-S-not-split.vanishing}
\(\mathrm{H}^0(C,
\mathrm{gr}_{p-1}(\mathbf{R}^1\varphi_*\sO_S)_{ip}) = 0\).
\end{enumerate}
Indeed, there is a commutative square
\[
\begin{tikzcd}
(\mathbf{R}^1\varphi_*\sO_S)_{ip} \rar \dar &
\Div^{p-2-i}(W) \otimes \sO_C \dar \\
\mathrm{gr}_{p-1}(\mathbf{R}^1\varphi_*\sO_S)_{ip} \rar &
\mathrm{gr}_{p-1}(\Div^{p-2-i}(W) \otimes \sO_C)
\end{tikzcd}
\]
in which the vertical maps are the quotient maps. The nonvanishing from
\ref{threefolds-cohomology-S-not-split.nonvanishing} means that the right map
is nonzero, so there a nonzero global section of the form
\[
s \colon
\sO_C \hookrightarrow
\Div^{p-2-i}(W) \otimes \sO_C \twoheadrightarrow
\mathrm{gr}_{p-1}(\Div^{p-2-i}(W) \otimes \sO_C).
\]
If the sequence of the statement were split, then \(s\) would lift to a global
section of \(\mathrm{gr}_{p-1}(\mathbf{R}^1\varphi_*\sO_S)_{ip}\); this cannot
happen given \ref{threefolds-cohomology-S-not-split.vanishing}. Thus given
\ref{threefolds-cohomology-S-not-split.nonvanishing} and
\ref{threefolds-cohomology-S-not-split.vanishing}, the sequence of the
statement is not split, and so the boundary map
\[
\delta \colon
\mathrm{H}^0(C,\Div^{p-2-i}(W) \otimes \sO_C) \to
\mathrm{H}^1(C,\mathcal{F}_{ip})
\]
is nonzero. But \(\delta\) is \(\mathrm{U}_3(q)\)-equivariant by
\parref{threefolds-cohomology-S-Gm-sequences} and
\(\mathrm{Div}^{p-2-i}(W)\) is simple by
\parref{representations-divs}\ref{representations-divs.simple}. Therefore
\(\delta\) is injective, giving the second statement of the Lemma.

It remains to verify
\ref{threefolds-cohomology-S-not-split.nonvanishing} and
\ref{threefolds-cohomology-S-not-split.vanishing}.
The nonvanishing of \ref{threefolds-cohomology-S-not-split.nonvanishing}
follows from \parref{threefolds-cohomology-duality}:
\begin{align*}
\gr_{p-1}(\Div^{p-2-i}(W) \otimes \sO_C )
& \cong \gr_{p-1}(\mathcal{F}_{p^2+ip-1})  \\
& \cong \gr_0(\mathcal{D}_{(p-i-2)p+p-1})^\vee \otimes \sO_C(-q+1) \neq 0.
\end{align*}
The vanishing \ref{threefolds-cohomology-S-not-split.vanishing} follows from
the identification of \(\gr_{p-1}(\mathbf{R}^1\varphi_*\sO_S)_{ip}\)
from \parref{threefolds-cohomology-S-final-graded-vanishing} below together with
the Borel--Weil--Bott computation of \parref{threefolds-cohomology-BWB-vanishing}.
\end{proof}

The following identifies the final graded pieces of the weight \(iq\) component
of \(\mathbf{R}^1\varphi_*\sO_S\) with respect to the filtration of
\parref{threefolds-smooth-cone-situation-pushforward}\ref{threefolds-smooth-cone-situation-pushforward.R1}:

\begin{Lemma}\label{threefolds-cohomology-S-final-graded-vanishing}
For each \(1 \leq i \leq q-2\), there is an isomorphism of \(\sO_C\)-modules
\[
\gr_{q-1}(\mathbf{R}^1\varphi_*\sO_S)_{iq} \cong
\Div^{2q-2-i}(\mathcal{T}_{\PP W}(-1)\rvert_C)(-1).
\]
\end{Lemma}

\begin{proof}
Taking \(i = q-1\) in
\parref{threefolds-smooth-cone-situation-pushforward}\ref{threefolds-smooth-cone-situation-pushforward.R1}
shows
\[
\gr_{q-1}(\mathbf{R}^1\varphi_*\sO_S) =
\mathbf{R}^1\pi_*(\gr_{q-1}(\rho_*\sO_{\tilde{S}})) =
\mathbf{R}^1\pi_*(\sO_T(1,-q) \otimes \pi^*\sO_C(-1) \otimes L_-).
\]
By
\parref{threefolds-cone-situation-equations-of-T}\ref{threefolds-cone-situation-equations-of-T.resolution},
\(\sO_T(1,-q)\) is resolved by a complex
\([\mathcal{E}^{-2} \to \mathcal{E}^{-1}]\) of \(\sO_\PP\)-modules with
\begin{align*}
\mathcal{E}^{-2}
& = \sO_\PP(-q+1,-2q-1) \otimes \pi^*\sO_C(-1) \oplus \sO_\PP(-q,-2q) \otimes L_+, \;\text{and}\\
\mathcal{E}^{-1}
& = \sO_\PP(0,-2q) \oplus \sO_\PP(-q+1,-q-1) \oplus \sO_\PP(-q+1,-2q) \otimes \pi^*\sO_C(-1) \otimes L_+.
\end{align*}
Recall from \parref{threefolds-cone-situation-PP} that
\(\PP = \PP\mathcal{V}_1 \times_C \PP\mathcal{V}_2\) with
\(\mathcal{V}_1 \cong L_{-,C} \oplus \sO_C(-1)\) and
\(\mathcal{V}_2 \cong \mathcal{T}_{\PP W}(-1)\rvert_C \oplus L_{+,C}\).
The relative dualizing sheaves of the individual factors are
\[
\omega_{\PP\mathcal{V}_1/C} \cong \sO_{\pi_1}(-2) \otimes \pi_1^*\sO_C(1) \otimes L_-^\vee
\quad\text{and}\quad
\omega_{\PP\mathcal{V}_2/C} \cong \sO_{\pi_2}(-3) \otimes \pi_2^*\sO_C(-1) \otimes L_+^\vee
\]
so \(\omega_{\PP/C} \cong \sO_\PP(-2,-3) \otimes L_-^\vee \otimes L_+^\vee\).
The resolution provides a spectral sequence computing
\(\mathbf{R}^1\pi_*\sO_T(1,-q)\) with \(E_1\) page given by
\[
\begin{tikzcd}[row sep=1em, column sep=1em]
E_1^{-2,3} \rar["d_1"] & E_1^{-1,3} \rar["d_1"] & E_1^{0,3} \\
E_1^{-2,2} \rar["d_1"] & E_1^{-1,2} \rar["d_1"] & E_1^{0,2}
\end{tikzcd}
\quad
=
\quad
\begin{tikzcd}[row sep=1em, column sep=1em]
\mathbf{R}^3\pi_*\mathcal{E}^{-2} \rar["\phi"] &
\mathbf{R}^3\pi_*\mathcal{E}^{-1} &
0 \\
0 &
\mathbf{R}^2\pi_*\mathcal{E}^{-1} \rar["\wedge^2\phi^\vee"] &
\mathbf{R}^2\pi_*\sO_\PP(1,-q)
\end{tikzcd}
\]
and with all other terms vanishing. Observe that since
\(\boldsymbol{\mu}_{q^2-1}\) acts through linear automorphisms of \(\PP\) over
\(C\), the differentials of the spectral sequence are compatible with the
induced \(\mathbf{Z}/(q^2-1)\mathbf{Z}\) gradings on each term.

Let \(1 \leq i \leq q - 2\). To identify the weight component
\[
(\mathbf{R}^1\varphi_*\sO_S)_{iq} \cong
(\mathbf{R}^1\pi_*(\sO_T(1,-q) \otimes \pi^*\sO_C(-1) \otimes L_-))_{iq},
\]
consider the corresponding weight component of the spectral sequence. Since
the weights of \(L_-\) and \(L_+\) are \(-1\) and \(q\), respectively, a direct
computation shows that
\begin{align*}
(\mathbf{R}^3\pi_*(\mathcal{E}^{-2} \otimes \pi^*\sO_C(-1) \otimes L_-))_{iq}
& \cong \Div^{2q-2-i}(\mathcal{T}_{\PP W}(-1)\rvert_C)(-1) \otimes L_-^{\otimes q} \otimes L_+^{\otimes i+1}, \\
(\mathbf{R}^3\pi_*(\mathcal{E}^{-1} \otimes \pi^*\sO_C(-1) \otimes L_-))_{iq}
& \cong 0, \\
(\mathbf{R}^2\pi_*(\mathcal{E}^{-1} \otimes \pi^*\sO_C(-1) \otimes L_-))_{iq}
& \cong \Div^{q-2-i}(\mathcal{T}_{\PP W}(-1)\rvert_C) \otimes L_- \otimes L_+^{\otimes q+i}, \\
(\mathbf{R}^2\pi_*(\sO_\PP(1,-q) \otimes \pi^*\sO_C(-1) \otimes L_-))_{iq}
& \cong \Div^{q-2-i}(\mathcal{T}_{\PP W}(-1)\rvert_C) \otimes L_+^{\otimes i}.
\end{align*}
The differential \(\wedge^2\phi^\vee\) between the latter two sheaves is given by
\(u_1 v_{21}' + u_2 v_{22}'\).
By the computations of the components of \(v'\) from
\parref{threefolds-cone-situation-v'-components}, \(v_{21}'\) involves a
\(q\)-power of the Euler section of \(\PP\mathcal{V}_2\). Since the divided power
of \(\mathcal{T}_{\PP W}(-1)\rvert_C\) is always strictly less than \(q\),
this acts by zero. The remaining component \(u_2 v_{22}'\) then acts as
the isomorphism which is identity on \(\Div^{q-2-i}(\mathcal{T}_{\PP W}(-1)\rvert_C)\)
and the isomorphism \(L_- \cong L_+^{\vee,\otimes q}\) provided by \(\beta\).
This shows that, at least for computing the weight \(iq\) component of
\(\mathbf{R}^2\varphi_*\sO_S\), the spectral sequence degenerates on this page
and that
\[
(\mathbf{R}^2\varphi_*\sO_S)_{iq}
\cong (\mathbf{R}^3\pi_*(\mathcal{E}^{-2} \otimes \pi^*\sO_C(-1) \otimes L_-))_{iq}
\cong \Div^{2q-2-i}(\mathcal{T}_{\PP W}(-1)\rvert_C)(-1)
\]
for each \(1 \leq i \leq q-2\).
\end{proof}

Putting everything together yields a computation of the
cohomology of \(\sO_S\):

\begin{Theorem}\label{threefolds-cohomology-S-result}
Let \(X\) be a smooth \(q\)-bic threefold and let \(S\) be its Fano surface of
lines. If \(p = q\), then
\[
\dim_\kk\mathrm{H}^1(S,\sO_S) =
\frac{1}{2}\dim_{\mathbf{Q}_\ell}\mathrm{H}^1_{\mathrm{\acute{e}t}}(S,\mathbf{Q}_\ell) =
\frac{1}{2}p(p-1)(p^2+1).
\]
In particular, the Picard scheme \(\mathbf{Pic}_S\) of \(S\) is smooth.
\end{Theorem}

\begin{proof}
Since \(\mathrm{H}^1(S,\sO_S)\) is canonically the tangent space to
\(\mathbf{Pic}_S\) at the identity, and, for any prime \(\ell \neq p\),
\(\mathrm{H}^1_{\mathrm{\acute{e}t}}(S,\mathbf{Z}_\ell)\) is
the \(\ell\)-adic Tate module of \(\mathbf{Pic}_S\), there is always an inequality
\[
\dim_\kk\mathrm{H}^1(S,\sO_S) =
\dim_\kk \mathcal{T}_{\mathbf{Pic}_S,[\sO_S]} \geq
\dim \mathbf{Pic}_S =
\frac{1}{2} \rank_{\mathbf{Z}_\ell} \mathrm{T}_\ell \mathbf{Pic}_S =
\frac{1}{2} \dim_{\mathbf{Q}_\ell}\mathrm{H}^1_{\mathrm{\acute{e}t}}(S,\mathbf{Q}_\ell).
\]
By the \'etale cohomology computation for \(S\) in
\parref{threefolds-smooth-betti-S}, the dimension of \(\mathrm{H}^1(S,\sO_S)\)
is always at least \(q(q-1)(q^2+1)/2\) with no assumption on \(q\).

Assume \(q = p\) is prime. The corresponding upper bound follows by
semicontinuity of cohomology, see \citeSP{0BDN}, for the flat family
\(\mathcal{S} \to \mathbf{A}^1\) produced in
\parref{threefolds-smooth-cone-situation-family-fano}, the cohomology
computation for the singular surface \(S_0\) in
\parref{threefolds-cohomology-theorem}, and the non-splitting result of
\parref{threefolds-cohomology-S-not-split}. In more detail, and slightly more
directly, the Leray spectral sequence for \(\varphi \colon S \to C\) yields
a short exact sequence
\[
0 \to
\mathrm{H}^1(C,\varphi_*\sO_S) \to
\mathrm{H}^1(S,\sO_S) \to
\mathrm{H}^0(C,\mathbf{R}^1\varphi_*\sO_S) \to
0.
\]
By
\parref{threefolds-smooth-cone-situation-pushforward}\ref{threefolds-smooth-cone-situation-pushforward.O},
\(\varphi_*\sO_S = \sO_C\) so the first term has dimension \(p(p-1)/2\). For
the second term, consider the \(\mathbf{Z}/(p^2-1)\mathbf{Z}\) weight
decomposition induced by the action of \(\boldsymbol{\mu}_{p^2-1}\). Taking
global sections of the short exact sequences
from \parref{threefolds-cohomology-S-Gm-sequences} yields inequalities
\[
\dim_\kk \mathrm{H}^0(C,(\mathbf{R}^1\varphi_*\sO_S)_\alpha) \leq
\dim_\kk\mathrm{H}^0(C,\mathcal{F}_\alpha) +
\dim_\kk\mathrm{H}^0(C,\mathcal{F}_{\alpha + p^2-1})
\]
for each \(\alpha = 1,2,\ldots,q^2-1\).
When \(\alpha = ip\) with \(1 \leq i \leq p-2\),
\parref{threefolds-cohomology-S-not-split} refines this to an equality
\[
\dim_\kk \mathrm{H}^0(C,(\mathbf{R}^1\varphi_*\sO_S)_\alpha)=
\dim_\kk\mathrm{H}^0(C,\mathcal{F}_{ip}).
\]
Summing the inequalities over \(\alpha\) gives the inequality
\[
\dim_\kk\mathrm{H}^0(C,\mathbf{R}^1\varphi_*\sO_S) \leq
\dim_\kk\mathrm{H}^0(C,\mathcal{F})
- \dim_\kk\mathrm{H}^0(C,\mathcal{F}_0)
- \sum_{i = 1}^{p-2} \dim_\kk\mathrm{H}^0(C,\mathcal{F}_{p^2+ip-1}).
\]
By \parref{threefolds-cohomology-duality} and \parref{threefolds-cohomology-D},
\(\mathcal{F}_{p^2+ip-1} \cong \Div^{p-2-i}(W) \otimes \sO_C\). By
\parref{normalize-splitting}, \(\mathcal{F}_0\) is isomorphic to the cokernel
the map \(\sO_C \to \Fr^2_*\sO_C\) induced by the \(q^2\)-power Frobenius
morphism. Since the \(q\)-power Frobenius already acts by zero on \(\mathrm{H}^1(C,\sO_C)\),
see \parref{hypersurfaces-cohomology-zero-frobenius}, the long exact
sequence in cohomology shows
\[
\mathrm{H}^0(C,\mathcal{F}_0)
\cong \mathrm{H}^1(C,\sO_C)
\cong \Div^{p-2}(W).
\]
Therefore the negative terms in the inequality sum up to
\begin{align*}
\dim_\kk\mathrm{H}^0(C,\mathcal{F}_0) +
\sum\nolimits_{i = 1}^{p-2} \dim_\kk\mathrm{H}^0(C,\mathcal{F}_{p^2+ip-1})
& = \sum\nolimits_{i = 0}^{p-2} \dim_\kk \Div^i(W) \\
& = \sum\nolimits_{i = 0}^{p-2} \binom{i+2}{2} \\
& = \binom{p+1}{3}
= \binom{p}{2} + \binom{p}{3}.
\end{align*}
By \parref{threefolds-cohomology-theorem},
\(\mathrm{H}^0(C,\mathcal{F})\) has dimension \((p^2+1)\binom{p}{2} + \binom{p}{3}\),
so
\[
\dim_\kk\mathrm{H}^0(C,\mathbf{R}^1\varphi_*\sO_S) \leq
p^2\binom{p}{2}.
\]
the short exact sequence for \(\mathrm{H}^1(S,\sO_S)\) then gives
\[
\dim_\kk\mathrm{H}^1(S,\sO_S) \leq
\binom{p}{2} + p^2\binom{p}{2} = \frac{1}{2}p(p-1)(p^2+1)
\]
and this proves the result.
\end{proof}

\begin{Corollary}\label{threefolds-cohomology-S-all}
Let \(X\) be a smooth \(q\)-bic threefold and let \(S\) be its Fano surface of
lines. If \(p = q\), then
\[
\dim_\kk\mathrm{H}^i(S,\sO_S) =
\begin{dcases*}
1 & if \(i = 0\), \\
\frac{1}{2} p(p-1)(p^2+1) & if \(i = 1\), and \\
\frac{1}{12}p(p-1)(5p^4 - 2p^2 - 5p - 2) & if \(i = 2\).
\end{dcases*}
\]
\end{Corollary}

\begin{proof}
The first number follows from smoothness together with irreducibility of \(S\),
see \parref{threefolds-lines}. The second number is
\parref{threefolds-cohomology-S-result}. The third number is now deduced from
the Euler characteristic computation \parref{threefolds-lines-chi-OS}.
\end{proof}

%% file: proj-geom.tex
\chapter{Generalities on Linear Projective Geometry}\label{chapter-linear}

Some of the constructions in Chapters \parref{chapter-hypersurfaces} and
\parref{chapter-threefolds} involve fine properties of linear projection
and various attendant structures. This Appendix collects these projective
constructions in slightly more generality than needed. Section \parref{bundles}
begins with some conventions on Grassmannian bundles. Section
\parref{functoriality-grassmannian} describes a sort of functoriality amongst
such bundles. In geometric terms, this Section presents a resolution of the
rational maps induced by linear projection and intersection by a linear space:
see \parref{linear-projection-resolve} and \parref{intersection-resolve}.
These situations are combined in Section \parref{subquotient}, wherein the
Subquotient Situation is introduced and the induced rational maps are resolved:
see \parref{subquotient-result}. Finally, Section \parref{ext-and-PP} works out
a relationship between extensions and projective bundles; this is crucial to
the computations in \parref{section-D}.

\section{Projective and Grassmannian bundles}\label{bundles}
This Section collects some conventions in projective and Grassmannian bundles
over a base scheme \(B\). Definitions are collected in
\parref{bundles-subs-quotients}--\parref{bundles-grassmannian-quotients}.
Duality is described in \parref{bundles-duality}. A construction of certain
rational maps to Grassmannians and a method of resolution is given in
\parref{bundles-rational-maps}--\parref{bundles-rational-maps-resolve}.
Affine subbundles to projective bundles and the behaviour of tautological
structures are discussed in
\parref{bundles-affine-subs}--\parref{bundles-affine-euler-split}.

\subsection{Subbundles and quotients}\label{bundles-subs-quotients}
Let \(\mathcal{V}\) be a finite locally free \(\sO_B\)-module, alternatively
referred to as a \emph{vector bundle} over \(B\). Given a field \(\kk\) and
a morphism \(x \colon \Spec(\kk) \to B\), the \emph{fibre} of \(\mathcal{V}\)
is the pullback \(\mathcal{V}_x \coloneqq x^*\mathcal{V} = \mathcal{V} \otimes_{\sO_B} \kk\).
An injective morphism of vector bundles \(\mathcal{V}' \hookrightarrow
\mathcal{V}\) is called \emph{locally split} if it admits, Zariski locally, a
retraction; this is equivalent to asking for the morphism to be injective on
each fibre, or for the dual morphism to be surjective. In particular, the
cokernel of a locally split injection is also locally free. In the special case
that \(\mathcal{V}' \subset \mathcal{V}\) is the inclusion of a locally free
subbmodule, \(\mathcal{V}'\) is referred to as a \emph{subbundle}.

\subsection{Grassmannian of subbundles}\label{bundles-grassmannian-subbundles}
For any \(\sO_B\)-module \(\mathcal{F}\) and any morphism of schemes
\(T \to B\), write \(\mathcal{F}_T\) for the \(\sO_T\)-module obtained by
pullback of \(\mathcal{F}\). Now let \(\mathcal{V}\) be a finite locally free
\(\sO_B\)-module, \(r\) a positive integer, and consider the functor
\begin{align*}
\mathbf{G}(r,\mathcal{V}) \colon \mathrm{Sch}_B^{\mathrm{opp}} & \to \mathrm{Set} \\
T & \mapsto \Set{\mathcal{V}' \subset \mathcal{V}_T \;\text{subbundle of rank \(r\)}}.
\end{align*}
Then \(\mathbf{G}(r,\mathcal{V})\) is representable by a smooth projective
scheme over \(B\), see \citeSP{089T}.
Moreover, there is a tautological short exact sequence
\[
0 \to
\mathcal{S}_{\mathbf{G}(r,\mathcal{V})} \to
\mathcal{V}_{\mathbf{G}(r,\mathcal{V})} \to
\mathcal{Q}_{\mathbf{G}(r,\mathcal{V})} \to
0
\]
where \(\mathcal{S}_{\mathbf{G}(r,\mathcal{V})}\) is the universal subbundle
of rank \(r\) and \(\mathcal{Q}_{\mathbf{G}(r,\mathcal{V})}\) is the
corresponding universal quotient bundle of corank \(r\).

\subsection{Grassmanian of quotients}\label{bundles-grassmannian-quotients}
It will be sometimes convenient to think of the Grassmannian as parameterizing
quotient bundles. To make this distinction clear, consider the following
alternative functor: Let \(\mathcal{V}\) be a finite locally free \(\sO_B\)-module
as above, let \(s\) be a positive integer, and let
\begin{align*}
\mathbf{G}(\mathcal{V},s) \colon \mathrm{Sch}_B^{\mathrm{opp}} & \to \mathrm{Set} \\
T & \mapsto \Set{\mathcal{V}_T \twoheadrightarrow \mathcal{V}'' \;\text{quotient bundle of rank \(s\)}}.
\end{align*}
Since the data of a quotient bundle of rank \(s\) is equivalent to the data
of a subbundle of rank \(r = \rank_{\sO_B}(\mathcal{V}) - s\), there is an
isomorphism of functors
\(\mathbf{G}(\mathcal{V},s) \cong \mathbf{G}(r,\mathcal{V})\),
whence the former is representable. Write
\[
0 \to
\mathcal{S}_{\mathbf{G}(\mathcal{V},s)} \to
\mathcal{V}_{\mathbf{G}(\mathcal{V},s)} \to
\mathcal{Q}_{\mathbf{G}(\mathcal{V},s)} \to
0
\]
for the corresponding tautological sequence of universal bundles.

\subsection{Duality of Grassmannians}\label{bundles-duality}
Since a morphism is a fibrewise injection if and only if its dual is a surjection,
there are natural duality identificaitons
\[
\mathbf{G}(r,\mathcal{V}) \cong \mathbf{G}(\mathcal{V}^\vee,r)
\quad\text{and}\quad
\mathbf{G}(\mathcal{V},s) \cong \mathbf{G}(s,\mathcal{V}^\vee)
\]
so that the universal bundles are identified as
\[
\mathcal{S}^\vee_{\mathbf{G}(r,\mathcal{V})} \cong \mathcal{Q}_{\mathbf{G}(\mathcal{V}^\vee,r)}
\quad\text{and}\quad
\mathcal{Q}_{\mathbf{G}(\mathcal{V},s)}^\vee \cong \mathcal{S}_{\mathbf{G}(s,\mathcal{V}^\vee)}.
\]

\subsection{Rational maps to Grassmannians}\label{bundles-rational-maps}
Consider the following two situations which give rational maps
\([\varphi] \colon T \dashrightarrow \mathbf{G}\) over \(B\)
to some Grassmannian \(\mathbf{G}\) of \(\mathcal{V}\):
\begin{enumerate}
\item\label{bundles-rational-maps.injection}
An injection \(\varphi \colon \mathcal{E} \to \mathcal{V}_T\) from a bundle of
rank \(r\) gives a map to \(\mathbf{G}(r,\mathcal{V})\).
\item\label{bundles-rational-maps.surjection}
A generic surjection \(\varphi \colon \mathcal{V}_T \to \mathcal{F}\) to a bundle
of rank \(s\) gives a map to \(\mathbf{G}(\mathcal{V},s)\).
\end{enumerate}
Set \(n \coloneqq \rank_{\sO_B}(\mathcal{V})\) and consider the scheme
\[
D \coloneqq
\begin{dcases*}
\mathrm{V}(\Fitt_{n-r}(\coker(\varphi \colon \mathcal{E} \to \mathcal{V}_T)))
& in situation \ref{bundles-rational-maps.injection}, and \\
\mathrm{V}(\Fitt_s(\image(\varphi \colon \mathcal{V}_T \to \mathcal{F})))
& in situation \ref{bundles-rational-maps.surjection},
\end{dcases*}
\]
defined by an appropriate Fitting ideal; this is the quasi-coherent ideal
locally generated by the maximal minors of \(\varphi\) in the first case, and
maximal minors of a locally free presentation of \(\image(\varphi)\)
in the second case; see \citeSP{0C3C}. The properties of Fitting
ideals imply that the cokernel of \(\varphi\) in the first case and the
image of \(\varphi\) in the second is locally free away from \(D\), hence
rational map \([\varphi]\) restricts to a morphism
\(T \setminus D \to \mathbf{G}\).

\subsection{}\label{bundles-rational-maps-strict-transform}
In fact, \([\varphi]\) can be resolved to a morphism on the blowup
\(\mathrm{bl} \colon \tilde{T} \to T\) along \(D\). To give a precise
description, recall that the \emph{strict transform} of a quasi-coherent
\(\mathcal{O}_T\)-module \(\mathcal{G}\) along the blowup \(\mathrm{bl}\)
is the quotient of \(\mathrm{bl}^*\mathcal{G}\) by the quasi-coherent submodule
of sections supported on the exceptional divisor \(\mathrm{bl}^{-1}(D)\), see
the discussion at the beginning of \citeSP{080C}. In the above situations: Let
\begin{enumerate}
\item\label{bundles-rational-maps-strict-transform.injection}
\(\tilde{\mathcal{E}}\) be the kernel of the map from \(\mathcal{V}_{\tilde{T}}\)
to the strict transform of \(\coker(\varphi)\), and let
\(\tilde\varphi \colon \tilde{\mathcal{E}} \hookrightarrow \mathcal{V}_{\tilde{T}}\)
be the natural injection.
\item\label{bundles-rational-maps-strict-transform.surjection}
\(\tilde{\mathcal{F}}\) be the strict transform of \(\image(\varphi)\)
and
\(\tilde\varphi \colon \mathcal{V}_{\tilde{T}} \twoheadrightarrow \tilde{\mathcal{F}}\)
the natural surjection.
\end{enumerate}

\subsection{Lemma}\label{bundles-rational-maps-resolve}
\emph{There exists a morphism
\([\tilde\varphi] \colon \tilde{T} \to \mathbf{G}\) resolving \([\varphi]\)
which is characterized in the above situations as follows:
\begin{enumerate}
\item\label{bundles-rational-maps-resolve.injection}
\({[\tilde\varphi]}^*(\mathcal{S}_{\mathbf{G}(r,\mathcal{V})} \to \mathcal{V}_{\mathbf{G}(r,\mathcal{V})})
= (\tilde\varphi \colon \tilde{\mathcal{E}} \to \mathcal{V}_{\tilde{T}})\), and
\item\label{bundles-rational-maps-resolve.surjection}
\({[\tilde\varphi]}^*(\mathcal{V}_{\mathbf{G}(\mathcal{V},s)} \to \mathcal{Q}_{\mathbf{G}(\mathcal{V},s)})
= (\tilde\varphi \colon \mathcal{V}_{\tilde{T}} \to \tilde{\mathcal{F}})\)
\end{enumerate}
}

\begin{proof}
Set \(\mathcal{G}\) to be \(\coker(\varphi)\) in
\parref{bundles-rational-maps}\ref{bundles-rational-maps.injection} and
\(\image(\varphi)\) in
\parref{bundles-rational-maps}\ref{bundles-rational-maps.surjection}. Then
by the assumptions, \(\mathcal{G}\rvert_{T \setminus D}\) is locally free of
rank \(n - r\) and \(s\), respectively. Thus \citeSP{0CZQ} shows that the
strict transform \(\tilde{\mathcal{G}}\) of \(\mathcal{G}\) along \(\mathrm{bl}\)
is locally free of the corresponding rank. The natural surjection
\(\mathcal{V}_{\tilde{T}} \to \mathrm{bl}^*\mathcal{G} \to \tilde{\mathcal{G}}\)
induces a morphism from \(\tilde{T}\) to \(\mathbf{G}(\mathcal{V},n-r)\)
in the first case, and \(\mathbf{G}(\mathcal{V},s)\) in the second case, such
that \(\mathcal{V}_{\tilde{T}} \to \tilde{\mathcal{G}}\) is the pullback of
the universal quotient.
Now \(\tilde{\mathcal{G}} = \tilde{\mathcal{F}}\) from
\parref{bundles-rational-maps-strict-transform}\ref{bundles-rational-maps-strict-transform.surjection},
so this gives the conclusion in \ref{bundles-rational-maps-resolve.surjection}.
In the other case,
\(\tilde{\mathcal{E}} = \ker(\mathcal{V}_{\tilde{T}} \to \tilde{\mathcal{G}})\)
from
\parref{bundles-rational-maps-strict-transform}\ref{bundles-rational-maps-strict-transform.injection},
so the conclusion in \ref{bundles-rational-maps-resolve.injection} follows
from the identification
\(\mathbf{G}(\mathcal{V},n-r) \cong \mathbf{G}(r,\mathcal{V})\) of \parref{bundles-grassmannian-quotients}.
\end{proof}

\subsection{Affine subbundles}\label{bundles-affine-subs}
Let \(\pi \colon \PP\mathcal{V} \to B\) be a projective bundle and suppose
that there is a short exact sequence
\[ 0 \to \mathcal{U} \to \mathcal{V} \xrightarrow{\alpha} \mathcal{L} \to 0 \]
where \(\mathcal{L}\) is an invertible \(\sO_B\)-module. Then
\(\PP\mathcal{U} \subset \PP\mathcal{V}\) is the hyperplane subbundle
cut out by the morphism
\[
u \coloneqq \pi^*(\alpha) \circ \mathrm{eu}_\pi \colon
\sO_{\PP\mathcal{V}}(-1) \to
\pi^*\mathcal{V} \twoheadrightarrow
\pi^*\mathcal{L}
\]
obtained by composing the Euler section of \(\pi\) with the pullback of the
quotient map \(\alpha\). Therefore its complement
\(\PP\mathcal{V}^\circ \coloneqq \PP\mathcal{V} \setminus \PP\mathcal{U}\)
is an affine space bundle over \(B\) whose underlying algebra is identified
as follows:

\begin{Lemma}\label{bundles-affine-subs-algebra}
There exists a canonical isomorphism of \(\sO_B\)-algebras
\[
\pi_*\sO_{\PP\mathcal{V}^\circ} \cong
\colim_n \Sym^n(\mathcal{V}^\vee \otimes \mathcal{L})
\]
in which the transition maps in the colimit induced by
\begin{align*}
\pi_*(u^\vee) \colon
\Sym^n(\mathcal{V}^\vee) \otimes \mathcal{L}^{\otimes n}
& \xrightarrow{\alpha^\vee}
\Sym^n(\mathcal{V}^\vee) \otimes \mathcal{V}^\vee \otimes \mathcal{L}^{\otimes n+1} \\
& \xrightarrow{\mathrm{mult}}
\Sym^{n+1}(\mathcal{V}^\vee) \otimes \mathcal{L}^{\otimes n+1}.
\end{align*}
\end{Lemma}

\begin{proof}
By the description \parref{bundles-affine-subs}, the ideal sheaf
\(\mathcal{I}\) of \(\PP\mathcal{U}\) in \(\PP\mathcal{V}\) is the image of
\[
u \colon \sO_{\PP\mathcal{V}}(-1) \otimes \pi^*\mathcal{L}^\vee \to \sO_{\PP\mathcal{V}}.
\]
Thus \cite[Exercise III.3.7]{Hartshorne:AG} gives an isomorphism between
\(\pi_*\sO_{\PP\mathcal{V}^\circ}\) and
\begin{align*}
\colim_n \pi_*\mathcal{H}\!\mathit{om}_{\sO_{\PP\mathcal{V}}}(\mathcal{I}^n, \sO_{\PP\mathcal{V}})
& \cong \colim_n \pi_*\sO_{\PP\mathcal{V}}(n) \otimes \mathcal{L}^{\otimes n} \\
& \cong \colim_n \Sym^n(\mathcal{V}^\vee) \otimes \mathcal{L}^{\otimes n}
\end{align*}
in which the transition maps are induced by \(u\).
\end{proof}

This is more explicit when the sequence in \parref{bundles-affine-subs} splits:

\begin{Lemma}\label{bundles-affine-subs-split}
A choice of splitting of the sequence in \parref{bundles-affine-subs} gives
an isomorphism
\[
\PP\mathcal{V}^\circ \cong
\mathbf{A}(\mathcal{U} \otimes \mathcal{L}^\vee) \coloneqq
\Spec\big(\Sym^*(\mathcal{U}^\vee \otimes \mathcal{L})\big).
\]
\end{Lemma}

\begin{proof}
Choose a section \(\sigma \colon \mathcal{L} \to \mathcal{V}\) to \(\alpha\)
and let \(\iota \colon \mathcal{U} \to \mathcal{V}\) be the inclusion. Let
\(\iota \oplus \sigma \colon \mathcal{U} \oplus \mathcal{L} \to \mathcal{V}\)
be the corresponding splitting. Then \(\Sym^n(\iota^\vee \oplus \sigma^\vee)\)
gives the first isomorphism in
\begin{align*}
\Sym^n(\mathcal{V}^\vee) \otimes \mathcal{L}^{\otimes n}
& \cong \Sym^n(\mathcal{U}^\vee \oplus \mathcal{L}^\vee) \otimes \mathcal{L}^{\otimes n} \\
& \cong \big(\bigoplus\nolimits_{i = 0}^n \Sym^i(\mathcal{U}^\vee) \otimes \mathcal{L}^{\vee, \otimes n-i}\big) \otimes \mathcal{L}^{\otimes n} \\
& \cong \bigoplus\nolimits_{i = 0}^n \Sym^i(\mathcal{U}^\vee \otimes \mathcal{L}).
\end{align*}
There is a commutative square
\[
\begin{tikzcd}[column sep=6em]
\Sym^n(\mathcal{V}^\vee) \otimes \mathcal{L}^{\otimes n} \rar["\Sym^n(\iota^\vee \oplus \sigma^\vee)"'] \dar["\alpha^\vee"']
& \Sym^n(\mathcal{U}^\vee \oplus \mathcal{L}^\vee) \otimes \mathcal{L}^{\otimes n}  \dar["(\iota^\vee \oplus \sigma^\vee) \circ \alpha^\vee"] \\
\Sym^{n+1}(\mathcal{V}^\vee) \otimes \mathcal{L}^{\otimes n+1} \rar["\Sym^{n+1}(\iota^\vee \oplus \sigma^\vee)"]
& \Sym^{n+1}(\mathcal{U}^\vee \oplus \mathcal{L}^\vee) \otimes \mathcal{L}^{\otimes n+1}
\end{tikzcd}
\]
where the vertical maps are induced by \(\alpha^\vee \colon \sO_B \to
\mathcal{V}^\vee \otimes \mathcal{L}\) and the multiplication maps on symmetric
powers. Since
\((\iota^\vee \oplus \sigma^\vee) \circ \alpha^\vee = 0 \oplus \id_{\mathcal{L}^\vee}\),
this shows that the transition map of \parref{bundles-affine-subs-algebra} are
mapped to the natural inclusion
\[
\bigoplus\nolimits_{i = 0}^n \Sym^i(\mathcal{U}^\vee \otimes \mathcal{L})
\subset \bigoplus\nolimits_{i = 0}^{n+1} \Sym^i(\mathcal{U}^\vee \otimes \mathcal{L}).
\]
Thus \(\Sym^*(\iota^\vee \oplus \sigma^\vee)\) induces an isomorphism of
\(\pi_*\sO_{\PP\mathcal{V}^\circ}\) with
\(\Sym^*(\mathcal{U}^\vee \otimes \mathcal{L})\).
\end{proof}

\subsection{Tautological bundles}\label{bundles-affine-subs-tautological}
The pullback of the exact sequence in \parref{bundles-affine-subs} to
\(\PP\mathcal{V}\) together with relative Euler sequence fit into a
commutative diagram
\[
\begin{tikzcd}
&& \pi^*\mathcal{U} \dar[hook] \ar[dr] \\
0 \rar &
\sO_\pi(-1) \rar["\mathrm{eu}_{\pi}"] \ar[dr,"u"'] &
\pi^*\mathcal{V} \rar \dar[two heads,"\pi^*\alpha"] &
\mathcal{T}_\pi(-1) \rar &
0 \\
&& \pi^*\mathcal{L}
\end{tikzcd}
\]
in which the row and column are exact. By definition, \(u\) does not vanish on
\(\PP\mathcal{V}^\circ\), and so the diagonal maps give isomorphisms
of \(\sO_{\PP\mathcal{V}^\circ}\)-modules
\[
\sO_{\PP\mathcal{V}}(-1)\rvert_{\PP\mathcal{V}^\circ} \cong
\pi^*\mathcal{L}\rvert_{\PP\mathcal{V}^\circ}
\quad\text{and}\quad
\mathcal{T}_\pi(-1)\rvert_{\PP\mathcal{V}^\circ} \cong
\pi^*\mathcal{U}\rvert_{\PP\mathcal{V}^\circ}.
\]

\subsection{Euler section}\label{bundles-affine-euler}
Still in the situation of \parref{bundles-affine-subs}, consider the morphism
of \(\sO_B\)-modules obtained by pushing forward the Euler section.
Recall that the Euler section \(\sO_{\PP\mathcal{V}} \to \pi^*\mathcal{V} \otimes \sO_\pi(1)\)
is adjoint to the dual of the trace section:
\[
\pi_*\mathrm{eu}_\pi \colon
\sO_B \xrightarrow{\tr_\mathcal{V}^\vee}
\mathcal{V} \otimes \mathcal{V}^\vee \cong
\pi_*(\pi^*\mathcal{V} \otimes \sO_\pi(1)).
\]
Using
\(\sO_{\PP\mathcal{V}}(-1)\rvert_{\PP\mathcal{V}^\circ} \cong \pi^*\mathcal{L}\rvert_{\PP\mathcal{V}^\circ}\)
from \parref{bundles-affine-subs-tautological} and taking adjoints shows that
\[
\pi_*\mathrm{eu}_\pi\rvert_{\PP\mathcal{V}^\circ} \colon
\mathcal{L} \otimes \pi_*\sO_{\PP\mathcal{V}^\circ} \to
\mathcal{V} \otimes \pi_*\sO_{\PP\mathcal{V}^\circ}
\]
is the map of \(\pi_*\sO_{\PP\mathcal{V}^\circ}\)-algebras induced
by multiplication by
\(\tr^\vee_{\mathcal{V}} \colon
\mathcal{L} \to \mathcal{V} \otimes \mathcal{V}^\vee \otimes \mathcal{L}\).

\subsection{}\label{bundles-affine-euler-split}
This can be made more explicit in the case the short exact sequence of
\parref{bundles-affine-subs} splits. Fix a splitting
\(\mathcal{V} \cong \mathcal{U} \oplus \mathcal{L}\) and identify
\(\PP\mathcal{V}^\circ \cong \mathbf{A}(\mathcal{U} \otimes \mathcal{L}^\vee)\)
as in \parref{bundles-affine-subs-split}. Then the trace section of \(\mathcal{V}\)
factors as
\[
\tr_{\mathcal{V}}^\vee \colon
\sO_B \xrightarrow{\tr_{\mathcal{U}}^\vee \oplus \tr_{\mathcal{L}}^\vee}
(\mathcal{U} \otimes \mathcal{U}^\vee) \oplus (\mathcal{L} \otimes \mathcal{L}^\vee) \subset
\mathcal{V} \otimes \mathcal{V}^\vee
\]
and so the pushforward of the Euler section is the map
\[
\pi_*\mathrm{eu}_\pi \colon
\mathcal{L} \otimes \Sym^*(\mathcal{U}^\vee \otimes \mathcal{L}) \to
\mathcal{U} \otimes \Sym^*(\mathcal{U}^\vee \otimes \mathcal{L})
\]
of algebras induced by multiplication by
\[
(\id_{\mathcal{L}}, \tr_{\mathcal{U}}^\vee) \colon \mathcal{L} \to
\mathcal{L} \oplus (\mathcal{U} \otimes \mathcal{U}^\vee \otimes \mathcal{L})
\]
followed by the inclusions of \(\mathcal{L}\) and \(\mathcal{U}\) into \(\mathcal{V}\)
to have target
\(\mathcal{V} \otimes \Sym^{\leq 1}(\mathcal{U}^\vee \otimes \mathcal{L})\).

\section{Functoriality of Grassmannians}\label{functoriality-grassmannian}
The formation of Grassmannians is, in a sense, functorial on the
category of finite locally free \(\sO_B\)-modules upon restricting to certain
classes of morphisms. Namely, consider a morphism
\(\mathcal{V} \to \mathcal{W}\) of finite locally free \(\sO_B\)-modules of
ranks \(m\) and \(n\), respectively. This Section is concerned with the
following two situations:
\begin{enumerate}
\item\label{functoriality-grassmannian.projection}
If \(\psi \colon \mathcal{V} \to \mathcal{W}\) is surjective, then there is a
rational map
\[
\psi_r \colon \mathbf{G}(r,\mathcal{V}) \dashrightarrow \mathbf{G}(r,\mathcal{W})
\quad\text{for each}\;1 \leq r \leq n,
\]
which sends a general subbundle \(\mathcal{V}' \subseteq \mathcal{V}\) of rank
\(r\) to its image \(\psi(\mathcal{V}') \subseteq \mathcal{W}\).
\item\label{functoriality-grassmannian.intersection}
If \(\varphi \colon \mathcal{V} \to \mathcal{W}\) is locally split injective, then
there is a rational map
\[
\varphi^s \colon \mathbf{G}(\mathcal{W},s) \dashrightarrow \mathbf{G}(\mathcal{V},s)
\quad\text{for each}\; 1 \leq s \leq m,
\]
which sends a general quotient \(\mathcal{W} \twoheadrightarrow \mathcal{W}''\)
of rank \(s\) to \(\mathcal{V} \xrightarrow{\varphi} \mathcal{W} \twoheadrightarrow \mathcal{W}''\).
\end{enumerate}
Both situations arise as in \parref{bundles-rational-maps} and the methods of
\parref{bundles-rational-maps-strict-transform} and
\parref{bundles-rational-maps-resolve} may be used to construct a resolution.
The goal of this Section is to give a geometric description of these resolutions:
see \parref{linear-projection-resolve} and \parref{intersection-resolve}.

\subsection{Linear projection}\label{linear-projection}
Let \(\psi \colon \mathcal{V} \to \mathcal{W}\) be a surjection of finite
locally free sheaves of ranks \(m\) and \(n\), respectively. Let \(1 \leq r \leq n\)
be an integer and consider the rational map
\[ \psi_r \colon \mathbf{G}(r,\mathcal{V}) \dashrightarrow \mathbf{G}(r,\mathcal{W}) \]
as in \parref{functoriality-grassmannian}\ref{functoriality-grassmannian.projection}:
this is defined by the map of locally free sheaves
\begin{equation}\label{linear-projection.rational-map-definition}
\mathcal{S}_{\mathbf{G}(r,\mathcal{V})} \subset
\mathcal{V}_{\mathbf{G}(r,\mathcal{V})} \xrightarrow{\psi}
\mathcal{W}_{\mathbf{G}(r,\mathcal{V})}
\end{equation}
restricted to the open subscheme \(\mathbf{G}(r,\mathcal{V})^\circ\) on which
this is injective.
Consider the subfunctor of \(\mathbf{G}(r,\mathcal{V}) \times_B \mathbf{G}(r,\mathcal{W})\)
given by
\begin{align*}
\mathbf{G}(r,\psi) \colon \mathrm{Sch}_B^{\mathrm{opp}} & \to \mathrm{Set} \\
T & \mapsto
\Set{(\mathcal{V}', \mathcal{W}') | \psi(\mathcal{V}') \subseteq \mathcal{W}'}.
\end{align*}
This is representable by the closed subscheme of
\(\mathbf{G}(r,\mathcal{V}) \times_B \mathbf{G}(r,\mathcal{W})\) defined by
the vanishing of the map
\[
\pr_{\mathcal{V}}^*\mathcal{S}_{\mathbf{G}(r,\mathcal{V})} \subset
\mathcal{V}_{\mathbf{G}(r,\mathcal{V}) \times_B \mathbf{G}(r,\mathcal{W})} \xrightarrow{\psi}
\mathcal{W}_{\mathbf{G}(r,\mathcal{V}) \times_B \mathbf{G}(r,\mathcal{W})} \twoheadrightarrow
\pr_{\mathcal{W}}^*\mathcal{Q}_{\mathbf{G}(r,\mathcal{W})}.
\]
In particular, the pullback of \parref{linear-projection.rational-map-definition}
to \(\mathbf{G}(r,\psi)\) factors through the tautological map
\begin{equation}\label{linear-projection.tautological-map}
\pr_{\mathcal{V}}^*\mathcal{S}_{\mathbf{G}(r,\mathcal{V})} \to
\pr_{\mathcal{W}}^*\mathcal{S}_{\mathbf{G}(r,\mathcal{W})}
\end{equation}
induced by \(\psi\). The following summarizes the structure
of the scheme \(\mathbf{G}(r,\psi)\).

\begin{Proposition}\label{linear-projection-resolve}
In the setting of \parref{linear-projection}, there is a commutative
diagram
\[
\begin{tikzcd}
& \mathbf{G}(r,\psi) \ar[dl,"\pr_{\mathcal{V}}"'] \ar[dr,"\pr_{\mathcal{W}}"] \\
\mathbf{G}(r,\mathcal{V}) \ar[rr,dashed,"\psi_r"] &&
\mathbf{G}(r,\mathcal{W})
\end{tikzcd}
\]
of schemes over \(B\). Furthermore:
\begin{enumerate}
\item\label{linear-projection-resolve.blowup}
\(\pr_{\mathcal{V}} \colon \mathbf{G}(r,\psi) \to \mathbf{G}(r,\mathcal{V})\)
is the blowup along the Fitting scheme of \parref{linear-projection.rational-map-definition}:
\[
D_\psi \coloneqq \mathrm{V}\big(
\mathrm{Fitt}_{n-r}\big(\coker(
\mathcal{S}_{\mathbf{G}(r,\mathcal{V})} \subset
\mathcal{V}_{\mathbf{G}(r,\mathcal{V})} \xrightarrow{\psi}
\mathcal{W}_{\mathbf{G}(r,\mathcal{V})})\big)\big),
\]
\item\label{linear-projection-resolve.grassmannian}
\(\pr_{\mathcal{W}} \colon \mathbf{G}(r,\psi) \to \mathbf{G}(r,\mathcal{W})\) is
the Grassmannian bundle \(\mathbf{G}(r,\tilde{\mathcal{S}}) \to \mathbf{G}(r,\mathcal{W})\)
where
\[
\begin{tikzcd}
\tilde{\mathcal{S}} \rar[hook] \dar[two heads]
& \mathcal{V}_{\mathbf{G}(r,\mathcal{W})} \dar[two heads,"\psi"] \\
\mathcal{S}_{\mathbf{G}(r,\mathcal{V})} \rar[hook]
& \mathcal{W}_{\mathbf{G}(r,\mathcal{W})}
\end{tikzcd}
\]
is a pullback square, and
\item\label{linear-projection-resolve.tautological}
writing \(\mathcal{S}_{\pr_{\mathcal{W}}} \subset \pr_{\mathcal{W}}^*\tilde{\mathcal{S}}\) for the
relative tautological bundle of the Grassmannian bundle from
\ref{linear-projection-resolve.grassmannian}, the tautological map
\parref{linear-projection.tautological-map} factors as
\[
\pr_{\mathcal{V}}^*\mathcal{S}_{\mathbf{G}(r,\mathcal{V})} =
\mathcal{S}_{\pr_2} \subset
\pr_{\mathcal{W}}^*\tilde{\mathcal{S}} \twoheadrightarrow
\pr_{\mathcal{W}}^*\mathcal{S}_{\mathbf{G}(r,\mathcal{W})}.
\]
\end{enumerate}
\end{Proposition}

\medskip

That \(\mathbf{G}(r,\psi)\) fits into the commutative diagram follows from the
description of its functor of points. The remaining statements will now be
proved in turn.

\begin{proof}[Proof of \parref{linear-projection-resolve}\ref{linear-projection-resolve.blowup}]
The functorial description of \(\mathbf{G}(r,\psi)\) shows that
\(\pr_{\mathcal{V}}\) restricts to an isomorphism between the open subschemes
\begin{align*}
\mathbf{G}(r,\psi)^\circ
& \coloneqq \set{(\mathcal{V}',\mathcal{W}') | \psi(\mathcal{V}') = \mathcal{W}'}, \\
\mathbf{G}(r,\mathcal{V})^\circ
& \coloneqq \set{\mathcal{V}' | \mathcal{V}' \subset \mathcal{V} \xrightarrow{\psi} \mathcal{W}\;\text{is injective}},
\end{align*}
where \parref{linear-projection.tautological-map} is an isomorphism, and
where \parref{linear-projection.rational-map-definition} is injective, respectively.
The complement of \(\mathbf{G}(r,\mathcal{V})^\circ\), being the locus over
which \parref{linear-projection.rational-map-definition} degenerates,
is canonically the closed subscheme \(D_\psi\) given by the \((n-r)\)-th Fitting
ideal of \parref{linear-projection.rational-map-definition}. Let
\(\mathrm{bl} \colon \tilde{\mathbf{G}}(r,\mathcal{V}) \to \mathbf{G}(r,\mathcal{V})\)
be the blowup along \(D_\psi\). Then \parref{bundles-rational-maps-resolve}
shows that the strict transform of \(\coker(\parref{linear-projection.rational-map-definition})\)
along \(\mathrm{bl}\) is a locally free quotient of \(\mathcal{W}_{\tilde{\mathbf{G}}(r,\mathcal{V})}\)
and hence defines a morphism
\(\tilde\psi_r \colon \tilde{\mathbf{G}}(r,\mathcal{V}) \to \mathbf{G}(r,\mathcal{W})\)
resolving \(\psi_r\); this induces a morphism
\begin{equation}\label{linear-projection-resolve.tautological-map}
\mathrm{bl}^*\mathcal{S}_{\mathbf{G}(r,\mathcal{V})} \to
\tilde\psi_r^*\mathcal{S}_{\mathbf{G}(r,\mathcal{W})}
\end{equation}
with the additional property that the vanishing of its determinant is the
exceptional divisor of \(\mathrm{bl}\). This data provides a morphism
\(\Psi \colon \tilde{\mathbf{G}}(r,\mathcal{V}) \to \mathbf{G}(r,\psi)\) of
schemes over \(\mathbf{G}(r,\mathcal{V})\) for which the tautological map
\parref{linear-projection.tautological-map} pulls back
to \parref{linear-projection-resolve.tautological-map}.

To construct the inverse to \(\Psi\), consider the preimage of
\(D_\psi\) in \(\mathbf{G}(r,\psi)\). Base change of Fitting ideals, as in
\citeSP{0C3D}, gives the first equality in
\[
\pr_{\mathcal{V}}^{-1}(D_\psi)
 = \mathrm{V}(\mathrm{Fitt}_{n-r}(
\pr_{\mathcal{V}}^*(\parref{linear-projection.rational-map-definition})))
= \mathrm{V}(\mathrm{Fitt}_0(\parref{linear-projection.tautological-map}))
= \mathrm{V}(\det(\parref{linear-projection.tautological-map})).
\]
The third equality is because \parref{linear-projection.tautological-map} is
a map between locally free sheaves of the same rank. It remains to explain the
second equality. Observe that there is a factorization
\[
\pr_{\mathcal{V}}^*(\parref{linear-projection.rational-map-definition})
= (\parref{linear-projection.tautological-map}) \circ
\pr_{\mathcal{W}}^*(\mathcal{S}_{\mathbf{G}(r,\mathcal{W})} \subset \mathcal{W}_{\mathbf{G}(r,\mathcal{W})})
\]
where the latter map is the tautological inclusion and hence is of maximal
rank. Thus the ideal generated by the maximal minors
\(\pr_{\mathcal{V}}^*(\parref{linear-projection.rational-map-definition})\)
coincides with that of \parref{linear-projection.tautological-map}, and this is
the second equality. This can also be deduced from properties of Fitting
ideals using the convolution property from part (2) of \citeSP{07ZA}, together
with the identification of Fitting ideals of locally free modules of
\citeSP{0C3G}. In conclusion, \(\pr^{-1}_{\mathcal{V}}(D_\psi)\)
is an effective Cartier divisor in \(\mathbf{G}(r,\psi)\), so the universal
property of blowing up gives a morphism
\(\Psi^{-1} \colon \mathbf{G}(r,\psi) \to \tilde{\mathbf{G}}(r,\mathcal{V})\) of
schemes over \(\mathbf{G}(r,\mathcal{V})\) such that the
exceptional divisor of \(\mathrm{bl}\) pulls back to
\(\mathrm{V}(\det(\parref{linear-projection.tautological-map}))\).
Comparing their constructions shows that
\parref{linear-projection-resolve.tautological-map} pulls back to the tautological
map \parref{linear-projection.tautological-map}.

It remains to see that the two morphisms are mutually inverse. That
\(\Psi \circ \Psi^{-1}\) is the identity of \(\mathbf{G}(r,\psi)\)
is because the tautological maps \parref{linear-projection.tautological-map}
and \parref{linear-projection-resolve.tautological-map} pull back to one another.
For the other composite, note both maps are morphisms over
\(\mathbf{G}(r,\mathcal{V})\) so
\[
(\Psi^{-1} \circ \Psi \circ \mathrm{bl})^{-1}(D_\psi)
= (\Psi^{-1} \circ \pr_{\mathcal{V}})^{-1}(D_\psi)
= \mathrm{bl}^{-1}(D_\psi)
\]
and now its universal property shows that \(\Psi^{-1} \circ \Psi\) is the
identity of \(\tilde{\mathbf{G}}(r,\mathcal{V})\).
\end{proof}

\begin{proof}[Proof of
\parref{linear-projection-resolve}\ref{linear-projection-resolve.grassmannian} and
\ref{linear-projection-resolve.tautological}]
The pullback square defining \(\tilde{\mathcal{S}}\) implies that the
map \parref{linear-projection.tautological-map} factors through \(\pr_{\mathcal{W}}^*\tilde{\mathcal{S}}\).
This gives a locally split injection
\(
\pr_{\mathcal{V}}^*\mathcal{S}_{\mathbf{G}(r,\mathcal{V})} \hookrightarrow
\pr_{\mathcal{W}}^*\tilde{\mathcal{S}}
\)
defining a morphism
\(\mathbf{G}(r,\psi) \to \mathbf{G}(r,\tilde{\mathcal{S}})\) of schemes over
\(\mathbf{G}(r,\mathcal{W})\).

Conversely, the relative
tautological bundle \(\mathcal{S}_\pi \subseteq \pi^*\tilde{\mathcal{S}}\) gives
a rank \(r\) subbundle of \(\mathcal{V}_{\mathbf{G}(r,\tilde{\mathcal{S}})}\)
and so gives a morphism \(\mathbf{G}(r,\tilde{\mathcal{S}}) \to \mathbf{G}(r,\mathcal{V})\)
identifying \(\mathcal{S}_\pi\) as the pullback of \(\mathcal{S}_{\mathbf{G}(r,\mathcal{V})}\).
But the map \(\mathcal{S}_\pi \subset \pi^*\tilde{\mathcal{S}} \to \pi^*\mathcal{S}_{\mathbf{G}(r,\mathcal{W})}\)
shows that the resulting morphism
\(\mathbf{G}(r,\tilde{\mathcal{S}}) \to \mathbf{G}(r,\mathcal{V}) \times_B \mathbf{G}(r,\mathcal{W})\)
factors through a morphism \(\mathbf{G}(r,\tilde{\mathcal{S}}) \to \mathbf{G}(r,\psi)\).

That these morphisms are mutually inverse and identify the tautological bundles
as stated in \ref{linear-projection-resolve.tautological} follows from their
functorial construction and the description of the points of the schemes involved.
\end{proof}

\subsection{Intersection with subbundles}\label{intersection}
Dually, let \(\varphi \colon \mathcal{V} \to \mathcal{W}\) be a locally split
injection of locally free sheaves of ranks \(m\) and \(n\), respectively.
Let \(1 \leq s \leq m\) be an integer and consider the rational map
\[ \varphi^s \colon \mathbf{G}(\mathcal{W},s) \dashrightarrow \mathbf{G}(\mathcal{V},s) \]
as in \parref{functoriality-grassmannian}\ref{functoriality-grassmannian.intersection}:
this is defined by the map of locally free sheaves
\begin{equation}\label{intersection.rational-map-definition}
\mathcal{V}_{\mathbf{G}(\mathcal{W},s)} \xrightarrow{\varphi}
\mathcal{W}_{\mathbf{G}(\mathcal{W},s)} \twoheadrightarrow
\mathcal{Q}_{\mathbf{G}(\mathcal{W},s)}
\end{equation}
restricted to the open subscheme \(\mathbf{G}(\mathcal{W},s)\) on which this is
surjective. Consider the subfunctor of
\(\mathbf{G}(\mathcal{V},s) \times_B \mathbf{G}(\mathcal{W},s)\) given by
\begin{align*}
\mathbf{G}(\varphi,s) \colon \mathrm{Sch}_B^{\mathrm{opp}} & \to \mathrm{Set} \\
T & \mapsto
\Set{(\mathcal{V}'', \mathcal{W}'') |
\varphi\;\text{induces a map}\;\mathcal{V}'' \to \mathcal{W}''}.
\end{align*}
The identification between Grassmanians of quotients and subbundles given
in \parref{bundles-grassmannian-quotients} shows that \(\mathbf{G}(\varphi,s)\)
is also the subfunctor of \(\mathbf{G}(m-s,\mathcal{V}) \times_B \mathbf{G}(n-s,\mathcal{W})\)
given by
\[ T \mapsto \Set{(\mathcal{V}', \mathcal{W}') | \varphi(\mathcal{V}') \subseteq \mathcal{W}'}. \]
With either view, it follows that \(\mathbf{G}(\varphi,s)\) is representable
by the closed subscheme of
\(\mathbf{G}(\mathcal{V},s) \times_B \mathbf{G}(\mathcal{W},s)\) again defined
by the vanishing of the map
\[
\pr_{\mathcal{V}}^*\mathcal{S}_{\mathbf{G}(\mathcal{V},s)} \subset
\mathcal{V}_{\mathbf{G}(\mathcal{V},s) \times_B \mathbf{G}(\mathcal{W},s)} \xrightarrow{\varphi}
\mathcal{W}_{\mathbf{G}(\mathcal{V},s) \times_B \mathbf{G}(\mathcal{W},s)} \twoheadrightarrow
\pr_{\mathcal{W}}^*\mathcal{Q}_{\mathbf{G}(\mathcal{W},s)}.
\]
Then the pullback of \parref{intersection.rational-map-definition} to \(\mathbf{G}(\varphi,s)\)
factors through the tautological map
\begin{equation}\label{intersection.tautological-map}
\pr_{\mathcal{V}}^*\mathcal{Q}_{\mathbf{G}(\mathcal{V},s)} \to
\pr_{\mathcal{W}}^*\mathcal{Q}_{\mathbf{G}(\mathcal{W},s)}
\end{equation}
induced by \(\varphi\). The following summarizes the structure of the scheme
\(\mathbf{G}(\varphi,s)\).

\begin{Proposition}\label{intersection-resolve}
In the setting of \parref{intersection}, there is a commutative diagram
\[
\begin{tikzcd}
& \mathbf{G}(\varphi,s) \ar[dl,"\pr_{\mathcal{W}}"'] \ar[dr,"\pr_{\mathcal{V}}"] \\
\mathbf{G}(\mathcal{W},s) \ar[rr,dashed,"\varphi^s"] &&
\mathbf{G}(\mathcal{V},s)
\end{tikzcd}
\]
of schemes over \(B\). Furthermore:
\begin{enumerate}
\item\label{intersection-resolve.blowup}
\(\pr_{\mathcal{W}} \colon \mathbf{G}(\varphi,s) \to \mathbf{G}(\mathcal{W},s)\) is the blowup
along the Fitting scheme of \parref{intersection.tautological-map}:
\[
D_\varphi \coloneqq
\mathrm{V}\big(
\mathrm{Fitt}_s\big(\image(
\mathcal{V}_{\mathbf{G}(\mathcal{W},s)} \xrightarrow{\varphi}
\mathcal{W}_{\mathbf{G}(\mathcal{W},s)} \twoheadrightarrow
\mathcal{Q}_{\mathbf{G}(\mathcal{W},s)})\big)\big),
\]
\item\label{intersection-resolve.grassmannian}
\(\pr_{\mathcal{V}} \colon \mathbf{G}(\varphi,s) \to \mathbf{G}(\mathcal{V},s)\) is
the Grassmannian bundle \(\mathbf{G}(\tilde{\mathcal{Q}},s) \to \mathbf{G}(\mathcal{V},s)\)
where
\[
\begin{tikzcd}
\mathcal{V}_{\mathbf{G}(\mathcal{V},s)} \dar[hookrightarrow,"\varphi"'] \rar[two heads]
& \mathcal{Q}_{\mathbf{G}(\mathcal{V},s)} \dar[hookrightarrow] \\
\mathcal{W}_{\mathbf{G}(\mathcal{W},s)} \rar[two heads]
& \tilde{\mathcal{Q}}
\end{tikzcd}
\]
is a pushout square, and
\item\label{intersection-resolve.tautological}
writing \(\pr_{\mathcal{V}}^*\tilde{\mathcal{Q}} \twoheadrightarrow \mathcal{Q}_{\pr_{\mathcal{V}}}\) for the
relative tautological bundle of the Grassmannian bundle from
\ref{intersection-resolve.grassmannian}, the tautological map \parref{intersection.tautological-map}
factors as
\[
\varphi \colon
\pr_{\mathcal{V}}^*\mathcal{Q}_{\mathbf{G}(\mathcal{V},s)} \hookrightarrow
\pr_{\mathcal{V}}^*\tilde{\mathcal{Q}} \twoheadrightarrow
\mathcal{Q}_{\pr_{\mathcal{V}}} =
\pr_{\mathcal{W}}^*\mathcal{Q}_{\mathbf{G}(\mathcal{W},s)}.
\]
\end{enumerate}
\end{Proposition}

\begin{proof}
This can be proved in the same way as \parref{linear-projection-resolve}.
More directly, this follows from \parref{linear-projection-resolve} via
the duality identifications \(\mathbf{G}(\mathcal{V},s) \cong \mathbf{G}(s,\mathcal{V}^\vee)\),
\(\mathbf{G}(\mathcal{W},s) \cong \mathbf{G}(s,\mathcal{W}^\vee)\), and
\(\mathbf{G}(\varphi,s) \cong \mathbf{G}(s,\varphi^\vee)\) from \parref{bundles-duality}.
\end{proof}

\section{Passage to a subquotient}\label{subquotient}
This Section pertains to a composite of the two situations of
\parref{functoriality-grassmannian}, in which \(\mathcal{V}\) is a subquotient
of \(\mathcal{W}\) that can be realized either as
\begin{itemize}
\item a quotient of a subbundle \(\mathcal{W}_1\) of \(\mathcal{W}\), or
\item a subbundle of a quotient \(\mathcal{W}_2\) of \(\mathcal{W}\).
\end{itemize}
The two realizations of \(\mathcal{V}\) are related by an exact commutative
diagram
\begin{equation}\label{subquotient-diagram}
\begin{tikzcd}
&&
\mathcal{U}_1 \rar[equal] &
\mathcal{U}_1 \\
0 \rar &
\mathcal{U}_2 \rar &
\mathcal{W} \rar["\psi^2"'] \uar[two heads] &
\mathcal{W}_2 \rar \uar[two heads] &
0 \\
0 \rar &
\mathcal{U}_2 \rar \uar[equal] &
\mathcal{W}_1 \rar["\psi^1"] \uar[hook, "\varphi_1"] &
\mathcal{V} \rar \uar[hook,"\varphi_2"'] &
0\punct{.}
\end{tikzcd}
\end{equation}
Assume all modules appearing are finite locally free over \(\sO_B\) and set
\[
n \coloneqq \rank_{\sO_B}(\mathcal{W}),
\quad
m \coloneqq \rank_{\sO_B}(\mathcal{W}_2),
\quad
c \coloneqq \rank_{\sO_B}(\mathcal{U}_1),
\]
and let \(r \leq m\) be a positive integer.
The constructions of \parref{linear-projection} and \parref{intersection}
give a commutative diagram of rational maps
\begin{equation}\label{subquotient.rational-maps}
\begin{tikzcd}[column sep=1em]
\mathbf{G}(\mathcal{W},n-r) \dar[dashed,"\varphi_1^{n-r}"'] \rar[symbol={\cong}]
& \mathbf{G}(r,\mathcal{W}) \rar[dashed, "\psi^2_r"']
&[2em] \mathbf{G}(r,\mathcal{W}_2) \rar[symbol={\cong}]
& \mathbf{G}(\mathcal{W}_2,m-r)\dar[dashed, "\varphi_2^{m-r}"] \\
\mathbf{G}(\mathcal{W}_1,n-r) \rar[symbol={\cong}]
& \mathbf{G}(r-c,\mathcal{W}_1) \rar[dashed,"\psi^1_{r-c}"]
& \mathbf{G}(r-c,\mathcal{V}) \rar[symbol={\cong}]
& \mathbf{G}(\mathcal{V},m-r)
\end{tikzcd}
\end{equation}
where the horizontal arrows come as pushforward along surjections, and
the vertical arrows come as pullback along locally split injections. Write
\[
\mathbf{H}_0 \coloneqq \mathbf{G}(\mathcal{W}, n-r) \cong \mathbf{G}(r,\mathcal{W})
\quad\text{and}\quad
\mathbf{G}_0 \coloneqq \mathbf{G}(r - c,\mathcal{V}) \cong \mathbf{G}(\mathcal{V},m-r)
\]
for the Grassmannians on the top left and bottom right of the diagram,
and let
\[
\mathbf{G}_1 \coloneqq \mathbf{G}(r-c,\psi^1)
\quad\text{and}\quad
\mathbf{G}_2 \coloneqq \mathbf{G}(\varphi_2, m-r)
\]
be the spaces constructed in \parref{linear-projection-resolve} and
\parref{intersection-resolve} on which \(\psi^1\) and \(\varphi_2\),
respectively, are resolved. Then the commutative diagram provides a rational map
\[
\varphi_1^{n-r} \times \psi^2_r \colon
\mathbf{H}_0 \dashrightarrow
\mathbf{G} \coloneqq \mathbf{G}_1 \times_{\mathbf{G}_0} \mathbf{G}_2.
\]
The goal of this Section is to provide a geometric resolution of this rational
map, and to describe its structure over the \(\mathbf{H}_0\) and \(\mathbf{G}\):
see \parref{subquotient-result}.

\subsection{Structure of \(\mathbf{G}\)}\label{subquotient-scheme-G}
Consider the fibre product
\(\mathbf{G} = \mathbf{G}_1 \times_{\mathbf{G}_0} \mathbf{G}_2\) of the schemes
on which the rational maps \(\psi^1_{r-c}\) and \(\varphi_2^{m-r}\) are
resolved. By \parref{linear-projection} and \parref{intersection}, this scheme
represents the functor whose value on a scheme \(T\) is the set of diagrams
\[
\begin{tikzcd}
& \mathcal{W}_2' \\
\mathcal{W}_1' \rar["\psi^1"] & \mathcal{V}' \uar["\varphi_2"']
\end{tikzcd}
\quad\text{where}\quad
\begin{array}{rcl}
\mathcal{W}_2' & \in & \mathbf{G}(r,\mathcal{W}_2)(T), \\
\mathcal{V}'   & \in & \mathbf{G}(r-c,\mathcal{V})(T), \\
\mathcal{W}_1' & \in & \mathbf{G}(r-c,\mathcal{W}_1)(T),
\end{array}
\]
consisting of subbundles of \(\mathcal{W}_{2,T}\), \(\mathcal{V}_T\), and
\(\mathcal{W}_{1,T}\) of appropriate ranks, and which are mapped to one another
via the given maps \(\psi^1\) and \(\varphi_2\); here, the points of
\(\mathbf{G}(\varphi_2,m-r)\) are expressed in terms of subbundles.

The projection \(\pr_{\mathcal{V}} \colon \mathbf{G} \to \mathbf{G}_0\) to the
base of the fibre product exhibits \(\mathbf{G}\) as a product of
Grassmannian bundles over \(\mathbf{G}_0\). Namely, consider the commutative diagram
\begin{equation}\label{subquotient.tautological-on-G0}
\begin{tikzcd}[column sep=1em, row sep=1em]
0 \ar[rr]
&& \tilde{\mathcal{S}}_{\mathbf{G}_0} \ar[rr] \ar[dr,hook] \ar[dd, two heads]
&& \mathcal{W}_{\mathbf{G}_0} \ar[dr, two heads] \ar[rr]
&& \tilde{\mathcal{Q}}_{\mathbf{G}_0} \ar[rr]
&& 0 \\
&&& \mathcal{W}_{1,\mathbf{G}_0} \ar[ur,hook] \ar[dr, two heads]
&& \mathcal{W}_{2,\mathbf{G}_0} \ar[ur,two heads] \\
0 \ar[rr]
&& \mathcal{S}_{\mathbf{G}_0} \ar[rr]
&& \mathcal{V}_{\mathbf{G}_0} \ar[rr] \ar[ur,hook]
&& \mathcal{Q}_{\mathbf{G}_0} \ar[rr] \ar[uu,hook]
&& 0
\end{tikzcd}
\end{equation}
in which the bottom row is the tautological exact sequence of \(\mathbf{G}_0\),
\(\tilde{\mathcal{S}}_{\mathbf{G}_0}\) is the pullback in the left square, and
\(\tilde{\mathcal{Q}}_{\mathbf{G}_0}\) is the pushout of the right square; note that the top
row is also exact. Then
\parref{linear-projection-resolve}\ref{linear-projection-resolve.grassmannian} and
\parref{intersection-resolve}\ref{intersection-resolve.grassmannian} together imply:

\begin{Corollary}\label{subquotient-G-bundle}
\(\mathbf{G} \cong
\mathbf{G}(r-c,\tilde{\mathcal{S}}_{\mathbf{G}_0}) \times_{\mathbf{G}_0}
\mathbf{G}(\tilde{\mathcal{Q}}_{\mathbf{G}_0},m-r)\)
as schemes over \(\mathbf{G}_0\). \qed
\end{Corollary}

\subsection{The scheme \(\mathbf{H}\)}\label{subquotient-scheme-H}
To construct the scheme on which the rational map
\(\varphi_1^{n-r} \times \psi^2_r \colon \mathbf{H}_0 \dashrightarrow \mathbf{G}\)
is resolved. Consider the subfunctor \(\mathbf{H}\) of
\(\mathbf{G}(r,\mathcal{W}) \times_B \mathbf{G}\) whose points on a scheme
\(T\) is the set of diagrams
\[
\begin{tikzcd}
\mathcal{W}' \rar["\psi^2"'] & \mathcal{W}_2' \\
\mathcal{W}_1' \uar["\varphi_1"] \rar["\psi^1"] & \mathcal{V}' \uar["\varphi_2"']
\end{tikzcd}
\quad\text{where}\quad
\begin{array}{rcl}
\mathcal{W}'   & \in & \mathbf{G}(r,\mathcal{W})(T), \\
\mathcal{W}_2' & \in & \mathbf{G}(r,\mathcal{W}_2)(T), \\
\mathcal{V}'   & \in & \mathbf{G}(r-c,\mathcal{V})(T), \\
\mathcal{W}_1' & \in & \mathbf{G}(r-c,\mathcal{W}_1)(T),
\end{array}
\]
consisting of a subbundle of \(\mathcal{W}_T\) and a \(T\)-point of \(\mathbf{G}\)
which are compatible with the maps \(\varphi_1\) and \(\psi^2\). This is representable
by the closed subscheme of \(\mathbf{G}(r,\mathcal{W}) \times_B \mathbf{G}\)
given by the vanishing of the two maps
\begin{align*}
\pr_{\mathcal{W}}^*\mathcal{S}_{\mathbf{G}(r,\mathcal{W})} \subset
\mathcal{W}_{\mathbf{G}(r,\mathcal{W}) \times_B \mathbf{G}} & \xrightarrow{\psi^2}
\mathcal{W}_{2,\mathbf{G}(r,\mathcal{W}) \times_B \mathbf{G}} \twoheadrightarrow
\pr_{\mathcal{W}_2}^*\mathcal{Q}_{\mathbf{G}(r,\mathcal{W}_2)} \\
\pr_{\mathcal{W}_1}^*\mathcal{S}_{\mathbf{G}(r-c,\mathcal{W}_1)} \subset
\mathcal{W}_{1,\mathbf{G}(r,\mathcal{W}) \times_B \mathbf{G}} & \xrightarrow{\varphi_1}
\mathcal{W}_{\mathbf{G}(r,\mathcal{W}) \times_B \mathbf{G}} \twoheadrightarrow
\pr_{\mathcal{W}}^*\mathcal{Q}_{\mathbf{G}(r,\mathcal{W})}
\end{align*}
of tautological sheaves on the product.

In other words, \(\mathbf{H}\) is the subscheme of the product
\[
\mathbf{G}(r,\mathcal{W}) \times_B
\mathbf{G}(r,\mathcal{W}_2) \times_B
\mathbf{G}(r-c,\mathcal{W}_1) \times_B
\mathbf{G}(r-c,\mathcal{V})
\]
which is universal for the property that the morphisms \(\varphi_1\), \(\varphi_2\),
\(\psi^1\), and \(\psi^2\) in the subquotient situation \parref{subquotient-diagram}
induce a commutative diagram of tautological subbundles
\[
\begin{tikzcd}
\pr_{\mathcal{W}}^*\mathcal{S}_{\mathbf{G}(r,\mathcal{W})} \rar["\psi^2"']
& \pr_{\mathcal{W}_2}^*\mathcal{S}_{\mathbf{G}(r,\mathcal{W}_2)} \\
\pr_{\mathcal{W}_1}^*\mathcal{S}_{\mathbf{G}(r-c,\mathcal{W}_1)} \uar["\varphi_1"] \rar["\psi^1"]
& \pr_{\mathcal{V}}^*\mathcal{S}_{\mathbf{G}(r-c,\mathcal{V})} \uar["\varphi_2"']
\end{tikzcd}
\]
where \(\pr_{\mathcal{W}}\), \(\pr_{\mathcal{W}_2}\), \(\pr_{\mathcal{W}_1}\),
and \(\pr_{\mathcal{V}}\) are the projections of \(\mathbf{H}\) to the Grassmannian
factors in the fibre product. The following describes the structure of \(\mathbf{H}\).

\begin{Proposition}\label{subquotient-result}
In the above setting, there is a commutative diagram
\[
\begin{tikzcd}
& \mathbf{H} \ar[dr,"\pr_{\mathbf{G}}"] \ar[dl,"\pr_{\mathcal{W}}"'] \\
\mathbf{H}_0 \ar[rr,dashed,"\varphi_1^{n-r} \times \psi_r^2"] && \mathbf{G}
\end{tikzcd}
\]
of schemes over \(B\). Furthermore:
\begin{enumerate}
\item\label{subquotient-result.blowup}
\(\pr_{\mathcal{W}} \colon \mathbf{H} \to \mathbf{H}_0\) is projective and an
isomorphism away from the union of the indeterminacy loci of \(\varphi_1^{n-r}\)
and \(\psi^2_r\), and
\item\label{subquotient-result.grassmannian}
\(\pr_{\mathbf{G}} \colon \mathbf{H} \to \mathbf{G}\) is the Grassmannian
bundle \(\mathbf{G}(c,\bar{\mathcal{W}}) \to \mathbf{G}\), where
\(\bar{\mathcal{W}}\) is a locally free
subquotient of \(\mathcal{W}_{\mathbf{G}}\) of rank \(n-m+c\) which
fits into a short exact sequence
\[
0 \to
\pr_1^*\mathcal{Q}_{\mathbf{G}_1/\mathbf{G}_0} \to
\bar{\mathcal{W}} \to
\pr_2^*\mathcal{S}_{\mathbf{G}_2/\mathbf{G}_0} \to
0.
\]
\end{enumerate}
\end{Proposition}

That the diagram in question exists and is commutative follows from the
functorial descriptions of the schemes given in \parref{subquotient-scheme-G}
and \parref{subquotient-scheme-H}. The remaining statements will be established
in pieces below.

\subsection{}\label{subquotient-blowup-factor}
Let \(\mathbf{H}_1 \coloneqq \mathbf{G}(\varphi_1,n-r)\) and
\(\mathbf{H}_2 \coloneqq \mathbf{G}(r,\psi^2)\) be the schemes from
\parref{intersection} and \parref{linear-projection} which resolve
\(\varphi_1^{n-r}\) and \(\psi_2\), respectively. Then the fibre product
\(\mathbf{H}_1 \times_{\mathbf{H}_0} \mathbf{H}_2\) represents the functor
that sends a scheme \(T\) over \(B\) to the set of diagrams
\[
\begin{tikzcd}
\mathcal{W}' \rar["\psi^2"'] & \mathcal{W}_2' \\
\mathcal{W}_1' \uar["\varphi_1"]
\end{tikzcd}
\quad\text{where}\quad
\begin{array}{rcl}
\mathcal{W}'   & \in & \mathbf{G}(r,\mathcal{W})(T), \\
\mathcal{W}_2' & \in & \mathbf{G}(r,\mathcal{W}_2)(T), \\
\mathcal{W}_1' & \in & \mathbf{G}(r-c,\mathcal{W}_1)(T),
\end{array}
\]
consisting of subbundles of \(\mathcal{W}_T\), \(\mathcal{W}_{1,T}\), and
\(\mathcal{W}_{2,T}\) of appropriate ranks, and which are compatible under
the maps \(\varphi_1\) and \(\psi^2\). Comparing with \parref{subquotient-scheme-H}
shows that \(\pr_{\mathcal{W}} \colon \mathbf{H} \to \mathbf{H}_0\) factors as
\[ \mathbf{H} \to \mathbf{H}_1 \times_{\mathbf{H}_0} \mathbf{H}_2 \to \mathbf{H}_0. \]
The first map is birational, and the isomorphism locus is as follows:

\begin{Lemma}\label{subquotient-blowup-intermediate}
The morphism \(\mathbf{H} \to \mathbf{H}_1 \times_{\mathbf{H}_0} \mathbf{H}_2\)
is an isomorphism away from the closed subscheme cut out by the ideals
\begin{align*}
\mathcal{I}_1 & \coloneqq
\mathrm{Fitt}_{r-c}\big(
  \image(\psi^1 \colon
  \pr_{\mathcal{W}_1}^*\mathcal{S}_{\mathbf{G}(r-c,\mathcal{W}_1)} \to
  \mathcal{V}_{\mathbf{H}_1 \times_{\mathbf{H}_0} \mathbf{H}_2})\big), \\
\mathcal{I}_2 & \coloneqq
\mathrm{Fitt}_{m-r}\big(
  \coker(\varphi_2 \colon
  \mathcal{V}_{\mathbf{H}_1 \times_{\mathbf{H}_0} \mathbf{H}_2} \to
  \pr_{\mathcal{W}_2}^*\mathcal{Q}_{\mathbf{G}(r,\mathcal{W}_2)})\big).
\end{align*}
\end{Lemma}

\begin{proof}
Let \(T\) be a scheme over \(B\) and consider a \(T\)-point
\((\mathcal{W}_1',\mathcal{W}',\mathcal{W}_2',\mathcal{V}')\)
of \(\mathbf{H}\). In the case the point lies over the complement of
\(\mathrm{V}(\mathcal{I}_1)\), then \(\mathcal{V}'\) is determined as the rank
\(r - c\) locally free image of
\(\psi^1 \colon \mathcal{W}_1' \to \mathcal{V}_T\); similarly, if the point
were over the complement of \(\mathrm{V}(\mathcal{I}_2)\), then
\(\mathcal{V}'\) is determined as the kernel of \(\mathcal{V}_T \to \mathcal{W}_{2,T}/\mathcal{W}_2'\).
This shows that \(\mathbf{H} \to \mathbf{H}_1 \times_{\mathbf{H}_0} \mathbf{H}_2\)
is an isomorphism away from \(\mathrm{V}(\mathcal{I}_1) \cap \mathrm{V}(\mathcal{I}_2)\).
\end{proof}

\begin{proof}[Proof of \parref{subquotient-result}\ref{subquotient-result.blowup}]
It suffices to show that, for every scheme \(T\) over \(B\), the \(T\)-points
of the non-isomorphism locus
\[
\Set{(\mathcal{W}_1',\mathcal{W}',\mathcal{W}_2') |
\rank(\psi^1 \colon \mathcal{W}_1' \to \mathcal{V}_T) < r-c, \;\;
\rank(\varphi_2 \colon \mathcal{V}_T \to \mathcal{W}_{2,T}/\mathcal{W}_2') < m-r}
\]
from \parref{subquotient-blowup-intermediate} maps into the locus
\[
\Set{\mathcal{W}' \in \mathbf{H}_0(T) |
\rank(\psi^2 \colon \mathcal{W}' \to \mathcal{W}_{2,T}) < r\;\text{or}\;
\rank(\varphi_1 \colon \mathcal{W}_{1,T} \to \mathcal{W}_T/\mathcal{W}') < n-r}
\]
of indeterminacy of \(\psi^2\) and \(\varphi_1\). But when
\(\mathcal{W}_1' \to \mathcal{V}_T\) does not have full rank,
\(\mathcal{W}_1'\) intersects the kernel of \(\psi^1 \colon \mathcal{W}_{1,T} \to \mathcal{V}_T\).
Since \(\mathcal{W}_1'\) is a subbundle of \(\mathcal{W}'\), this implies that
\(\mathcal{W}'\) intersects the kernel of \(\psi^2 \colon \mathcal{W}_T \to \mathcal{W}_{2,T}\).
Whence \(\psi^2 \colon \mathcal{W}' \to \mathcal{W}_{2,T}\) does not have full rank.
\end{proof}

\subsection{The sheaf \(\bar{\mathcal{W}}\)}\label{subquotient-sheaf-Wbar}
The sheaf appearing in \parref{subquotient-result}\ref{subquotient-result.grassmannian}
parmeterizes the data of \(\mathcal{W}\) that remains upon taking into account
a point of \(\mathbf{G}\). It is constructed as follows: Pullback the
short exact sequence on the top row of \parref{subquotient.tautological-on-G0}
along the projection \(\pr_{\mathcal{V}} \colon \mathbf{G} \to \mathbf{G}_0\).
Juxtaposing the resulting sequence with the tautological bundles arising from
the identification \parref{subquotient-G-bundle} of \(\mathbf{G}\) as a product
of Grassmannian bundles over \(\mathbf{G}_0\) yields a commutative diagram
\[
\begin{tikzcd}
0 \rar
& \pr_{\mathcal{V}}^*\tilde{\mathcal{S}}_{\mathbf{G}} \rar
& \mathcal{W}_{\mathbf{G}} \rar
& \pr_{\mathcal{V}}^*\tilde{\mathcal{Q}}_{\mathbf{G}} \rar \dar[two heads]
& 0 \\
& \pr_1^*\mathcal{S}_{\mathbf{G}_1/\mathbf{G}_0} \rar[hook] \uar[hook]
& \mathcal{W}_{\mathbf{G}} \uar[equal] \rar[two heads]
& \pr_2^*\mathcal{Q}_{\mathbf{G}_2/\mathbf{G}_0}
\end{tikzcd}
\]
in which the bottom row is a complex which is not necessarily exact; let
\[
\bar{\mathcal{W}} \coloneqq
\mathcal{H}(
\pr_1^*\mathcal{S}_{\mathbf{G}_1/\mathbf{G}_0} \hookrightarrow
\mathcal{W}_{\mathbf{G}} \twoheadrightarrow
\pr_2^*\mathcal{Q}_{\mathbf{G}_2/\mathbf{G}_0})
\]
be the cohomology sheaf of this complex. Thus \(\bar{\mathcal{W}}\) is a
subquotient of \(\mathcal{W}_{\mathbf{G}}\) and there is an exact
commutative diagram given by
\[
\begin{tikzcd}
&& \pr_2^*\mathcal{Q}_{\mathbf{G}_2/\mathbf{G}_0} \rar[equal]
& \pr_{\mathcal{W}_2}^*\mathcal{Q}_{\mathbf{G}(\mathcal{W}_2,m-r)} \\
0 \rar
& \pr_1^*\mathcal{S}_{\mathbf{G}_1/\mathbf{G}_0} \rar
& \mathcal{W}_{\mathbf{G}} \rar \uar[two heads]
& \bar{\mathcal{W}}_1 \rar \uar[two heads]
& 0 \\
0 \rar
& \pr_{\mathcal{W}_1}^*\mathcal{S}_{\mathbf{G}(r-c,\mathcal{W}_1)} \rar \uar[equal]
& \bar{\mathcal{W}}_2 \rar \uar[hook]
& \bar{\mathcal{W}} \rar \uar[hook]
& 0
\end{tikzcd}
\]
where tautological bundles are identified as in
\parref{linear-projection-resolve}\ref{linear-projection-resolve.tautological} and
\parref{intersection-resolve}\ref{intersection-resolve.tautological}.

\begin{proof}[Proof of \parref{subquotient-result}\ref{subquotient-result.grassmannian}]
That \(\bar{\mathcal{W}}\) fits into the claimed short exact sequence follows
from taking the cokernel of the left column and taking the kernel of the
right column in the first diagram of \parref{subquotient-sheaf-Wbar}. This also shows
that \(\bar{\mathcal{W}}\) is locally free of rank \(n - m + c\). It remains
to identify \(\mathbf{H}\) and \(\mathbf{G}(c,\bar{\mathcal{W}})\). Construct
a morphism \(\mathbf{H} \to \mathbf{G}(c,\bar{\mathcal{W}})\) as follows.
Consider the pullback to \(\mathbf{H}\) of the tautological exact sequence
on \(\mathbf{G}(r,\mathcal{W})\):
\[
0 \to
\pr_{\mathcal{W}}^*\mathcal{S}_{\mathbf{G}(r,\mathcal{W})} \to
\mathcal{W}_{\mathbf{H}} \to
\pr_{\mathcal{W}}^*\mathcal{Q}_{\mathbf{G}(r,\mathcal{W})} \to
0.
\]
By the characterization of \(\mathbf{H}\) given in \parref{subquotient-scheme-H},
the inclusion of \(\pr_{\mathcal{W}_1}^*\mathcal{S}_{\mathbf{G}(r-c,\mathcal{W}_1)}\)
and the quotient to \(\pr_{\mathcal{W}_2}^*\mathcal{Q}_{\mathbf{G}(r,\mathcal{W}_2)}\)
from \(\mathcal{W}_{\mathbf{H}}\) factor through maps
\[
\varphi_1 \colon
\pr_{\mathcal{W}_1}^*\mathcal{S}_{\mathbf{G}(r-c,\mathcal{W}_1)} \hookrightarrow
\pr_{\mathcal{W}}^*\mathcal{S}_{\mathbf{G}(r,\mathcal{W})}
\quad\text{and}\quad
\psi^2 \colon
\pr_{\mathcal{W}}^*\mathcal{Q}_{\mathbf{G}(r,\mathcal{W})} \twoheadrightarrow
\pr_{\mathcal{W}_2}^*\mathcal{Q}_{\mathbf{G}(r,\mathcal{W}_2)}.
\]
Comparing with the subquotient diagram of \parref{subquotient-sheaf-Wbar} shows
that the sheaf
\[
\bar{\mathcal{S}} \coloneqq \coker\big(
\varphi_1 \colon
\pr_{\mathcal{W}_1}^*\mathcal{S}_{\mathbf{G}(r-c,\mathcal{W}_1)} \hookrightarrow
\pr_{\mathcal{W}}^*\mathcal{S}_{\mathbf{G}(r,\mathcal{W})}
\big)
\]
is a subbundle of \(\bar{\mathcal{W}}\) of rank \(c\), thereby defining a morphism
\(\mathbf{H} \to \mathbf{G}(c,\bar{\mathcal{W}})\).

To construct the inverse, consider the following diagram:
\[
\begin{tikzcd}[column sep=1em, row sep=1em]
0 \ar[rr]
&& \tilde{\mathcal{S}}_{\mathbf{G}(c,\bar{\mathcal{W}})} \ar[rr] \ar[dr,hook] \ar[dd, two heads]
&& \mathcal{W}_{\mathbf{G}(c,\bar{\mathcal{W}})} \ar[dr, two heads] \ar[rr]
&& \tilde{\mathcal{Q}}_{\mathbf{G}(c,\bar{\mathcal{W}})} \ar[rr]
&& 0 \\
&&& \bar{\mathcal{W}}_{2,\mathbf{G}(c,\bar{\mathcal{W}})} \ar[ur,hook] \ar[dr, two heads]
&& \bar{\mathcal{W}}_{1,\mathbf{G}(c,\bar{\mathcal{W}})} \ar[ur,two heads] \\
0 \ar[rr]
&& \mathcal{S}_{\mathbf{G}(c,\bar{\mathcal{W}})} \ar[rr]
&& \bar{\mathcal{W}}_{\mathbf{G}(c,\bar{\mathcal{W}})} \ar[rr] \ar[ur,hook]
&& \mathcal{Q}_{\mathbf{G}(c,\bar{\mathcal{W}})} \ar[rr] \ar[uu,hook]
&& 0
\end{tikzcd}
\]
in which the bottom row is the relative tautological sequence for
\(\mathbf{G}(c,\bar{\mathcal{W}}) \to \mathbf{G}\),
\(\tilde{\mathcal{S}}_{\mathbf{G}(c,\bar{\mathcal{W}})}\) is the
pullback of the left square, and
\(\tilde{\mathcal{Q}}_{\mathbf{G}(c,\bar{\mathcal{W}})}\) is the pushout in the
right square. Comparing with the subquotient diagram in
\parref{subquotient-sheaf-Wbar} shows that the composite
\[
\tilde{\mathcal{S}}_{\mathbf{G}(c,\tilde{\mathcal{W}})} \hookrightarrow
\mathcal{W}_{\mathbf{G}(c,\bar{\mathcal{W}})} \twoheadrightarrow
\pr_{\mathcal{W}_2}^*\mathcal{Q}_{\mathbf{G}(r,\mathcal{W}_2)}
\]
vanishes and that there is a short exact sequence
\[
0 \to
\pr_{\mathcal{W}_1}^*\mathcal{S}_{\mathbf{G}(r-c,\mathcal{W}_1)} \to
\tilde{\mathcal{S}}_{\mathbf{G}(c,\bar{\mathcal{W}})} \to
\mathcal{S}_{\mathbf{G}(c,\bar{\mathcal{W}})} \to
0.
\]
This implies that \(\tilde{\mathcal{S}}_{\mathbf{G}(c,\bar{\mathcal{W}})}\) is
a rank \(r\) subbundle of \(\mathcal{W}_{\mathbf{G}(c,\bar{\mathcal{W}})}\)
and it fits into a square
\[
\begin{tikzcd}
\tilde{\mathcal{S}}_{\mathbf{G}(c,\bar{\mathcal{W}})} \rar["\psi^2"']
& \pr_{\mathcal{W}_2}^*\mathcal{S}_{\mathbf{G}(r,\mathcal{W}_2)} \\
\pr_{\mathcal{W}_1}^*\mathcal{S}_{\mathbf{G}(r-c,\mathcal{W}_1)} \uar["\varphi_1"] \rar["\psi^1"]
& \pr_{\mathcal{V}}^*\mathcal{S}_{\mathbf{G}(r-c,\mathcal{V})} \uar["\varphi_2"']
\end{tikzcd}
\]
where here, \(\pr_{\mathcal{W}}\), \(\pr_{\mathcal{W}_1}\), \(\pr_{\mathcal{W}_1}\),
and \(\pr_{\mathcal{V}}\) denotes the projection from
\(\mathbf{G}(c,\bar{\mathcal{W}})\) to the corresponding Grassmannians.
Therefore, by the description its description as a subfunctor of the quadruple
product from \parref{subquotient-scheme-H}, yields a morphism
\(\mathbf{G}(c,\bar{\mathcal{W}}) \to \mathbf{H}\). That the two morphisms are
mutually inverse is because the two constructions described are mutually inverse:
some details omitted.
\end{proof}

\section{Extensions and projective bundles}\label{ext-and-PP}
An extension \(0 \to \mathcal{V}_1 \to \tilde{\mathcal{V}} \to \mathcal{V}_2 \to 0\)
of vector bundles on a scheme \(B\) induces an interesting vector bundle on
any product of Grassmannian bundles of the form
\(\pi \colon \mathbf{G}(\mathcal{V}_1,s) \times_B \mathbf{G}(r,\mathcal{V}_2) \to B\). Namely,
consider the homology sheaf
\[
\mathcal{V} \coloneqq
\mathcal{H}(
\mathcal{S}_{\pi_1} \hookrightarrow
\pi^*\tilde{\mathcal{V}} \twoheadrightarrow
\mathcal{Q}_{\pi_2})
\]
arising from the given short exact sequence and the tautological bundles on
the Grassmannian bundles. This construction arose, for instance, in the
Subquotient Situation above, see \parref{subquotient-result}. This Section is
concerned with the situation where \(\mathcal{V}_i = \mathcal{W}_i \oplus L_i\)
for a free \(\sO_B\)-module \(L_i\), and \(r = s = 1\) so that \(\pi\) is
a product of projective bundles and \(\mathcal{V}\) is of rank \(2\). The
hyperplane bundles \(\mathcal{W}_i \subset \mathcal{V}_i\) and
\(\mathcal{Q}_{\pi_1} \subset \mathcal{V}\) give natural affine bundles over
\(B\) and the goal of this Section is to describe the underlying
\(\sO_B\)-algebras. A particularly simple description is possible in the
special case when, say, \(\mathcal{W}_1\) is of rank \(1\), and this is
given in \parref{ext-and-PP-one-dimensional}; it is also this case that is
used in \parref{section-D}.

\subsection{Setting}\label{ext-and-PP-setting}
Let \(B\) be a scheme and suppose given a
short exact sequence
\[ 0 \to \mathcal{W}_1 \to \mathcal{W} \to \mathcal{W}_2 \to 0 \]
of finite locally free \(\sO_B\)-modules. Let \(L_1\) and \(L_2\) be free
\(\sO_B\)-modules of rank \(1\) and set
\[
\mathcal{V}_1 \coloneqq \mathcal{W}_1 \oplus L_1,
\quad
\mathcal{V}_2 \coloneqq \mathcal{W}_2 \oplus L_2,
\quad
\tilde{\mathcal{V}} \coloneqq \mathcal{W} \oplus L_1 \oplus L_2
\]
so that \(\mathcal{V}_1\), \(\mathcal{V}_2\), and
\(\tilde{\mathcal{V}}\) still fit into a short exact sequence as above. Let
\[
\PP_1 \coloneqq \mathbf{G}(\mathcal{V}_1,1) \cong \PP\mathcal{V}_1^\vee,
\quad
\PP_2 \coloneqq \PP\mathcal{V}_2,
\quad
\PP \coloneqq \PP_1 \times_B \PP_2,
\]
and let \(\pi_1\), \(\pi_2\), and \(\pi\) be their respective projections to
\(B\). For \(i = 1,2\), write \(\mathcal{S}_{\pi_i}\) and
\(\mathcal{Q}_{\pi_i}\) for the tautological sub and quotient bundles of
\(\pi_i \colon \PP_i \to B\) pulled up to \(\PP\). Thus there is a canonical
complex of finite locally free \(\sO_\PP\)-modules
\[ \mathcal{S}_{\pi_1} \to \pi^*\tilde{\mathcal{V}} \to \mathcal{Q}_{\pi_2}. \]
This complex is exact in all but the middle, and its homology sheaf
\(\mathcal{V}\) is a locally free \(\sO_\PP\)-module of rank \(2\) that
fits into a short exact sequence
\[
0 \to
\mathcal{Q}_{\pi_1} \to
\mathcal{V} \to
\mathcal{S}_{\pi_2} \to
0.
\]
Set \(\mathbf{Q} \coloneqq \PP\mathcal{V}\), let
\(\rho \colon \mathbf{Q} \to \PP\) be the structure morphism, and let
\(\varphi \coloneqq \pi \circ \rho \colon \mathbf{Q} \to B\).

Consider now the affine bundles over \(B\) given by
\[
\PP_1^\circ
\coloneqq \PP\mathcal{V}_1^\vee \setminus \PP\mathcal{W}_1^\vee
\cong \mathbf{A}(\mathcal{W}_1^\vee \otimes L_1)
\quad\text{and}\quad
\PP_2^\circ
\coloneqq \PP\mathcal{V}_2 \setminus \PP\mathcal{W}_2
\cong \mathbf{A}(\mathcal{W}_2 \otimes L_2^\vee)
\]
where the identification of the coordinate ring of the bundle comes from
\parref{bundles-affine-subs-split}. Set
\[
\PP^\circ \coloneqq \PP_1^\circ \times_B \PP_2^\circ
\quad\text{and}\quad
\mathbf{Q}^\circ \coloneqq
\PP^\circ \times_{\PP} (\PP\mathcal{V} \setminus \PP\mathcal{Q}_{\pi_1}).
\]
Then \(\rho \colon \mathbf{Q}^\circ \to \PP^\circ\) is also a bundle of affine
spaces, and the morphism \(\varphi \colon \mathbf{Q}^\circ \to B\) is affine.
The goal of this Section is describe the sheaves
\(\mathcal{A} \coloneqq \pi_*\sO_{\PP^\circ}\) and
\(\mathcal{B} \coloneqq \varphi_*\sO_{\mathbf{Q}^\circ}\) of \(\sO_B\)-algebras.
Their basic structure is as follows:

\begin{Lemma}\label{ext-and-PP-basic}
There is a canonical isomorphism of bigraded \(\sO_B\)-algebras
\[
\mathcal{A} \cong
\Sym^*(\mathcal{W}_1 \otimes L_1^\vee) \otimes \Sym^*(\mathcal{W}_2^\vee \otimes L_2)
\]
in which \(L_1^\vee\) has weight \((1,0)\) and \(L_2\) has weight \((0,1)\).
The sheaf \(\mathcal{B}\) is a bigraded \(\mathcal{A}\)-algebra with an
increasing \(\mathbf{Z}_{\geq 0}\) filtration with graded pieces
\[
\gr_i \mathcal{B} \coloneqq
\Fil_i\mathcal{B}/\Fil_{i-1}\mathcal{B} \cong
\mathcal{A} \otimes (L_1^\vee \otimes L_2)^{\otimes i}
\quad\text{for all}\; i \in \mathbf{Z}_{\geq 0}.
\]
\end{Lemma}

\begin{proof}
By \parref{bundles-affine-subs-split} and the discussion of
\parref{ext-and-PP-setting}, there is a canonical identification
\[
\PP^\circ \cong
\mathbf{A}(\mathcal{W}_1^\vee \otimes L_1) \times_B
\mathbf{A}(\mathcal{W}_2 \otimes L_2^\vee)
\]
yielding the identification of \(\mathcal{A}\). Note that the bigrading may
be constructed by considering the \(\mathbf{G}_m^2\)-equivariant structure
induced by a weight \((-1,0)\) action on \(L_1\), a weight \((0,1)\) action on
\(L_2\), and trivial action on \(\mathcal{W}_1\), \(\mathcal{W}_2\), and \(B\).

As for the sheaf \(\mathcal{B}\), that it is an \(\mathcal{A}\)-algebra is due
to the factorization \(\varphi = \pi \circ \rho\). For the filtration, consider
first the affine bundle
\(
\PP\mathcal{V}^\circ \coloneqq
\PP\mathcal{V} \setminus \PP\mathcal{Q}_{\pi_1}
\)
over \(\PP\). By \parref{bundles-affine-subs-algebra},
\[
\rho_*\sO_{\PP\mathcal{V}^\circ}
\cong
\colim_n \Sym^n(\mathcal{V}^\vee \otimes \mathcal{S}_{\pi_2}).
\]
There is a two step filtration on \(\mathcal{V}^\vee \otimes \mathcal{S}_{\pi_2}\)
induced by the short exact sequence
\[
0 \to
\sO_\PP \to
\mathcal{V}^\vee \otimes \mathcal{S}_{\pi_2} \to
\mathcal{Q}_{\pi_1}^\vee \otimes \mathcal{S}_{\pi_2} \to
0.
\]
This induces an \(n+1\) step filtration on each
\(\Sym^n(\mathcal{V}^\vee \otimes \mathcal{S}_{\pi_2})\). These
filtrations are compatible with the maps in the colimit, so there is
an induced filtration on \(\rho_*\sO_{\PP\mathcal{V}^\circ}\) and it
satisfies
\[
\gr_i(\rho_*\sO_{\PP\mathcal{V}^\circ})
\cong (\mathcal{Q}_{\pi_1}^\vee \otimes \mathcal{S}_{\pi_2})^{\otimes i}
\quad\text{for all}\; i \in \mathbf{Z}_{\geq 0}.
\]
Restricting to \(\PP^\circ\) and making the identifications
\(\mathcal{Q}_{\pi_1}^\vee \rvert_{\PP^\circ} \cong \pi^*L_1^\vee\)
and \(\mathcal{S}_{\pi_2}\rvert_{\PP^\circ} \cong \pi^*L_2\) via
\parref{bundles-affine-subs-tautological} then completes the statements
regarding \(\mathcal{B}\).
\end{proof}

The main goal is to determine the components of the bigraded decompositions
\[
\mathcal{A}
= \bigoplus\nolimits_{a,b \in \mathbf{Z}_{\geq 0}}
\mathcal{A}_{(a,b)} \otimes L_1^{\vee, \otimes a} \otimes L_2^{\otimes b}
\quad\text{and}\quad
\mathcal{B}
= \bigoplus\nolimits_{a,b \in \mathbf{Z}_{\geq 0}}
\mathcal{B}_{(a,b)} \otimes L_1^{\vee, \otimes a} \otimes L_2^{\otimes b}.
\]
The description of \(\mathcal{A}\) given in
\parref{ext-and-PP-basic} already gives:

\begin{Corollary}\label{ext-and-PP-A}
For each \(a,b \in \mathbf{Z}_{\geq 0}\), there is a canonical isomorphism
\[
\pushQED{\qed}
\mathcal{A}_{(a,b)} \cong
\Sym^a(\mathcal{W}_1) \otimes \Sym^b(\mathcal{W}_2^\vee).
\qedhere
\popQED
\]
\end{Corollary}

The pieces of \(\mathcal{B}\) are more complicated. To begin,
note that the filtration of \(\mathcal{B}\) restricts to a filtration on
each \(\mathcal{B}_{(a,b)}\). Their graded pieces are as follows:

\begin{Lemma}\label{ext-and-PP-gr-B}
For every \(a,b,i \in \mathbf{Z}_{\geq 0}\), there is a canonical isomorphism
\[
\gr_i\mathcal{B}_{(a,b)} \cong \begin{dcases*}
\Sym^{a-i}(\mathcal{W}_1) \otimes \Sym^{b-i}(\mathcal{W}_2^\vee) &
if \(0 \leq i \leq \min(a,b)\), and \\
0 & otherwise.
\end{dcases*}
\]
\end{Lemma}

\begin{proof}
For each \(i \in \mathbf{Z}_{\geq 0}\), \parref{ext-and-PP-basic} gives
an isomorphism
\(\gr_i\mathcal{B}_{(a,b)} \cong \mathcal{A}_{(a-i,b-i)}\), from which
the result follows from \parref{ext-and-PP-A}.
\end{proof}

In particular, this implies that, for all \(a,b \in \mathbf{Z}_{\geq 0}\),
\[
\mathcal{B}_{(a,0)} \cong \mathcal{A}_{(a,0)} \cong \Sym^a(\mathcal{W}_1)
\quad\text{and}\quad
\mathcal{B}_{(0,b)} \cong \mathcal{A}_{(0,b)} \cong \Sym^b(\mathcal{W}_2^\vee).
\]
Next, observe that \(\mathcal{B}_{(a,b)}\) can always be related to a diagonal
bigraded piece:

\begin{Lemma}\label{ext-and-PP-a-b-present}
For each \(a,b \in \mathbf{Z}_{\geq 0}\) with \(a \leq b\), the maps
\[
\mathcal{B}_{(a,a)} \otimes \mathcal{B}_{(0,b-a)} \to \mathcal{B}_{(a,b)},
\quad
\mathcal{B}_{(a,a)} \otimes \mathcal{B}_{(b-a,0)} \to \mathcal{B}_{(b,a)},
\quad
\Sym^a(\mathcal{B}_{(1,1)}) \to \mathcal{B}_{(a,a)}
\]
induced by multiplication are strict surjections of filtered \(\sO_B\)-modules.
\end{Lemma}

\begin{proof}
It follows from \parref{ext-and-PP-basic} that \(\mathcal{B}\) is generated in
degrees \((1,0)\), \((0,1)\), and \((1,1)\). This implies that each of the maps
induced by multiplication are surjective. Since multiplication is compatible
with the filtration, they are strict.
\end{proof}

The following gives a sort of dual to the presentation of \(\mathcal{B}_{(a,b)}\)
in \parref{ext-and-PP-a-b-present}:

\begin{Lemma}\label{ext-and-PP-a-b-filter}
For each \(a,b \in \mathbf{Z}_{\geq 0}\) with \(a \leq b\), the maps
\[
\mathcal{B}_{(a,b)} \otimes \mathcal{B}_{(b-a,0)} \to \mathcal{B}_{(b,b)}
\quad\text{and}\quad
\mathcal{B}_{(b,a)} \otimes \mathcal{B}_{(0,b-a)} \to \mathcal{B}_{(b,b)}
\]
are strict surjections onto
\(\Fil_a\mathcal{B}_{(b,b)} \subseteq \mathcal{B}_{(b,b)}\).
\end{Lemma}

\begin{proof}
Since the filtration of \(\mathcal{B}_{(a,b)}\) has \(a+1\) steps and since
multiplication is compatible with the filtration, the multiplication maps
factor through \(\Fil_a\mathcal{B}_{(b,b)}\). To see that they are surjective,
consider the induced maps on associated graded pieces. Using \parref{ext-and-PP-gr-B},
the map induced by
\(\mathcal{B}_{(a,b)} \otimes \mathcal{B}_{(b-a,0)} \to \mathcal{B}_{(b,b)}\)
on the \(i\)-th associated graded piece is the multiplication map
\[
\Sym^{a-i}(\mathcal{W}_1) \otimes \Sym^{b-i}(\mathcal{W}_2^\vee) \otimes \Sym^{b-a}(\mathcal{W}_1)
\to \Sym^{b-i}(\mathcal{W}_1) \otimes \Sym^{b-i}(\mathcal{W}_2^\vee).
\]
These are surjective for each \(0 \leq i \leq a\), showing that \(\mathcal{B}_{(a,b)} \otimes \mathcal{B}_{(b-a,0)} \to \mathcal{B}_{(b,b)}\)
surjects onto \(\Fil_a\mathcal{B}_{(b,b)}\). Similarly, the map induced by
\(\mathcal{B}_{(b,a)} \otimes \mathcal{B}_{(0,b-a)} \to \mathcal{B}_{(b,b)}\)
on the \(i\)-th associated graded piece is the multiplication map
\[
\Sym^{b-i}(\mathcal{W}_1) \otimes \Sym^{a-i}(\mathcal{W}_2^\vee) \otimes \Sym^{b-a}(\mathcal{W}_2^\vee)
\to \Sym^{b-i}(\mathcal{W}_1) \otimes \Sym^{b-i}(\mathcal{W}_2^\vee)
\]
and so this also surjects onto \(\Fil_a\mathcal{B}_{(b,b)}\).
\end{proof}

The crucial point is to determine \(\mathcal{B}_{(1,1)}\). This is done as
follows:

\begin{Proposition}\label{ext-and-PP-1-1}
The filtration on \(\mathcal{B}_{(1,1)}\) gives an exact sequence of
\(\sO_B\)-modules
\[
0 \to
\mathcal{W}_1 \otimes \mathcal{W}_2^\vee \to
\mathcal{B}_{(1,1)} \to
\sO_B \to
0
\]
whose extension class is sent to the extension class of the given sequence
\[
0 \to
\mathcal{W}_1 \to
\mathcal{W} \to
\mathcal{W}_2 \to
0
\]
under the identification
\(\Ext^1_B(\mathcal{W}_2,\mathcal{W}_1)
\cong \Ext^1_B(\sO_B,\mathcal{W}_1 \otimes \mathcal{W}_2^\vee)\).
\end{Proposition}

\begin{proof}
That the filtration on \(\mathcal{B}_{(1,1)}\) gives the claimed
sequence follows from \parref{ext-and-PP-gr-B}. It remains to identify the
extension class. Note that \(\mathcal{B}_{(1,1)}\) is contained in
\[
\Fil_1\mathcal{B}
\cong \pi_*(\Fil_1 \rho_*\sO_{\PP\mathcal{V}^\circ}\rvert_{\PP^\circ})
\cong \pi_*(\mathcal{V}^\vee \otimes \mathcal{S}_{\pi_2} \rvert_{\PP^\circ})
\cong \colim_n \pi_*(\mathcal{V}^\vee \otimes \mathcal{S}_{\pi_2} \otimes \mathcal{L}^{\otimes n})
\]
where
\(\mathcal{L} \coloneqq (\mathcal{Q}_{\pi_1} \otimes \pi^*L_1^\vee) \otimes (\mathcal{S}_{\pi_2}^\vee \otimes \pi^*L_2)\)
is the invertible \(\sO_\PP\)-module corresponding to the Cartier divisor
\(\PP \setminus \PP^\circ\), and where the second and third isomorphism follow
from \parref{bundles-affine-subs-algebra}. In fact, the colimit is compatible
with the total degree of \(L_1^\vee\) and \(L_2\), so \(\mathcal{B}_{(1,1)}\) is
a summand of the \(n=1\) piece of the colimit. The sheaf being pushed along
\(\pi\) sits in the short exact sequence obtained from that of \(\mathcal{V}\),
\[
0 \to
\mathcal{Q}_{\pi_1} \otimes \mathcal{S}_{\pi_2}^\vee  \to
\mathcal{V}^\vee \otimes \mathcal{Q}_{\pi_1} \to
\sO_\PP \to
0
\]
with an additional twist by \(\pi^*(L_1^\vee \otimes L_2)\). The sheaf
\(\mathcal{B}_{(1,1)}\) and its filtration is now obtained by taking the weight
\((1,1)\) component of the direct image of this sequence. The local-to-global
spectral sequence for \(\Ext\)-groups together with the Leray spectral sequence
gives a canonical isomorphism
\[
\Ext^1_\PP(\sO_\PP,\mathcal{Q}_{\pi_1} \otimes \mathcal{S}_{\pi_2}^\vee)
\cong \Ext^1_B(\sO_B,\mathcal{W}_1 \otimes \mathcal{W}_2^\vee)
\]
such that the extension class of \(\mathcal{V}^\vee \otimes \mathcal{Q}_{\pi_1}\)
corresponds to that of \(\mathcal{B}_{(1,1)}\).

Next, since the short exact sequence for
\(\mathcal{V}^\vee \otimes \mathcal{Q}_{\pi_1}\) is obtained from that of
\(\mathcal{V}\) by taking duals and twisting by \(\mathcal{Q}_{\pi_1}\), their
classes correspond under the isomorphism
\[
(-)^\vee \otimes \mathcal{Q}_{\pi_1} \colon
\Ext^1_\PP(\sO_\PP, \mathcal{Q}_{\pi_1} \otimes \mathcal{S}_{\pi_2}^\vee)
\cong
\Ext^1_\PP(\mathcal{S}_{\pi_2},\mathcal{Q}_{\pi_1}).
\]
Using the tautological short exact sequences and that
\[
\pi_*\mathcal{S}_{\pi_1} = \pi_*\mathcal{Q}_{\pi_2} = 0,
\quad
\pi_*\mathcal{Q}_{\pi_1} \cong \mathcal{W}_1,
\quad
\pi_*\mathcal{S}_{\pi_2}^\vee \cong \mathcal{W}_2^\vee,
\]
it follows that the natural map
\(\Ext^1_\PP(\pi^*\mathcal{W}_2,\pi^*\mathcal{W}_1) \to
\Ext^1_\PP(\mathcal{S}_{\pi_2}, \mathcal{Q}_{\pi_1})\) given by
\[
(\pi^*\mathcal{W}_2 \xrightarrow{\xi} \pi^*\mathcal{W}_1[1]) \mapsto
(\mathcal{S}_{\pi_2} \subset
\pi^*\mathcal{W}_2 \xrightarrow{\xi}
\pi^*\mathcal{W}_1[1] \twoheadrightarrow
\mathcal{Q}_{\pi_1}[1])
\]
is an isomorphism.

Finally, the local-to-global spectral sequence for \(\Ext\)-groups and the
Leray spectral sequence give canonical and compatible isomorphisms
\[
\Ext^1_B(\mathcal{W}_2,\mathcal{W}_1) \cong
\Ext^1_\PP(\pi^*\mathcal{W}_2,\pi^*\mathcal{W}_1) \cong
\Ext^1_\PP(\mathcal{S}_{\pi_2},\mathcal{Q}_{\pi_1})
\]
such that the classes of \(\mathcal{W}\), \(\pi^*\mathcal{W}\), and
\(\mathcal{V}\) correspond to one another. With the identifications above,
this shows that the extension class of \(\mathcal{B}_{(1,1)}\) corresponds
with that of \(\mathcal{V}\), as claimed.
\end{proof}

In the case when either \(\mathcal{W}_1\) or \(\mathcal{W}_2\) in the setting
\parref{ext-and-PP-setting} is of rank \(1\), all the pieces of \(\mathcal{B}\)
can be described quite explicitly. The following does so for when
\(\mathcal{W}_1\) is of rank \(1\):

\begin{Proposition}\label{ext-and-PP-one-dimensional}
In the setting of \parref{ext-and-PP-setting}, assume furthermore that
\(\mathcal{W}_1\) is of rank \(1\). Then there are isomorphisms of filtered
\(\sO_B\)-modules
\[
\mathcal{B}_{(a,b)} \cong
\Fil_a\Sym^b(\mathcal{W}^\vee \otimes \mathcal{W}_1) \otimes \mathcal{W}_1^{\otimes a-b}
\quad\text{for all}\; a,b \in \mathbf{Z}_{\geq 0}.
\]
\end{Proposition}

\begin{proof}
First, when \(\mathcal{W}_1\) is of rank \(1\), \parref{ext-and-PP-gr-B}
implies that
\(\mathcal{B}_{(1,1)} \cong \mathcal{W}^\vee \otimes \mathcal{W}_1\). Second,
observe that the surjection \(\Sym^a(\mathcal{B}_{(1,1)}) \to
\mathcal{B}_{(a,a)}\) from \parref{ext-and-PP-a-b-present} is an isomorphism
for all \(a \in \mathbf{Z}_{\geq 0}\). Indeed, the induced map on the \(i\)-th
associated graded piece is the canonical map
\[
\Sym^{a-i}(\mathcal{W}_1 \otimes \mathcal{W}_2^\vee) \to
\Sym^{a-i}(\mathcal{W}_1) \otimes \Sym^{a-i}(\mathcal{W}_2^\vee)
\quad\text{for each}\; 0 \leq i \leq a,
\]
where the pieces are identified via \parref{ext-and-PP-1-1} and
\parref{ext-and-PP-gr-B}. Since \(\mathcal{W}_1\) is invertible, this
is an isomorphism. Third, observe that multiplication by \(\mathcal{B}_{(d,0)} = \mathcal{W}_1^{\otimes d}\)
is injective for any \(d \in \mathbf{Z}_{\geq 0}\). This is because \(\mathcal{B}\)
is locally a polynomial algebra and \(\mathcal{B}_{(d,0)}\) is generated by a
single monomial. When \(a \leq b\), this together with
\parref{ext-and-PP-a-b-filter} implies that the
multiplication map \(\mathcal{B}_{(a,b)} \otimes \mathcal{B}_{(b-a,0)} \to \mathcal{B}_{(b,b)}\)
is an isomorphism onto \(\Fil_a\mathcal{B}_{(b,b)}\). Therefore
\[
\mathcal{B}_{(a,b)}
\cong \Fil_a\mathcal{B}_{(b,b)} \otimes \mathcal{B}_{(b-a,0)}^\vee
\cong \Fil_a\Sym^b(\mathcal{W}^\vee \otimes \mathcal{W}_1) \otimes \mathcal{W}_1^{\otimes a-b}.
\]
Similarly, when \(a \geq b\), injectivity of multiplication together with
\parref{ext-and-PP-a-b-present} implies that the multiplication map
\(\mathcal{B}_{(b,b)} \otimes \mathcal{B}_{(a-b,0)} \to \mathcal{B}_{(a,b)}\)
is an isomorphism, showing
\[
\mathcal{B}_{(a,b)}
\cong \mathcal{B}_{(b,b)} \otimes \mathcal{B}_{(a-b,0)}
\cong \Sym^b(\mathcal{W}^\vee \otimes \mathcal{W}_1) \otimes \mathcal{W}_1^{\otimes a-b}.
\]
Since \(a \geq b\), this coincides with the \(a\)-th filtered piece, whence
the result.
\end{proof}

%% file: representations.tex
\chapter{Representation Theory Computations}\label{chapter-representations}
This Appendix collects some facts and computations involving the positive
characteristic representation theory of the algebraic group \(\mathrm{SL}_n\)
and the finite special unitary groups \(\mathrm{SU}_n(q)\). The primary
references are \cite{Jantzen:RAGS, Humphreys}.

\section{Setting}\label{representations-setting}
Throughout this Appendix, \(\kk\) denotes a field of positive characteristic
\(p > 0\), \(V\) is a \(3\)-dimensional vector space over \(\kk\), and
\(\mathbf{SL}_3 = \mathbf{SL}(V)\) is the special linear group of automorphisms
on \(V\).

\subsection{Root data}\label{representations-root-data}
Choose a maximal torus and a Borel subgroup
\(\mathbf{T} \subset \mathbf{B} \subset \mathbf{SL}_3\). Let
\begin{align*}
\mathrm{X}(\mathbf{T}) & \coloneqq
\Hom(\mathbf{T},\mathbf{G}_m) \cong
\mathbf{Z}\{\epsilon_1, \epsilon_2, \epsilon_3\}/(\epsilon_1 + \epsilon_2 + \epsilon_3) \\
\mathrm{X}^\vee(\mathbf{T}) & \coloneqq \Hom(\mathbf{G}_m,\mathbf{T})
\cong \Set{a_1 \epsilon_1^\vee + a_2 \epsilon_2^\vee + a_3 \epsilon_3^\vee \in \mathbf{Z}\{\epsilon_1^\vee, \epsilon_2^\vee, \epsilon_3^\vee\} | a_1 + a_2 + a_3 = 0}
\end{align*}
be the lattices of characters and cocharacters of \(\mathbf{T}\); here, upon
conjugating \(\mathbf{T}\) to the diagonal matrices in \(\mathbf{SL}_3\), the
characters \(\epsilon_i\) extract the \(i\)-th diagonal entry, whereas the
cocharacters \(\epsilon_i^\vee\) include into the \(i\)-th diagonal entry. Let
\[
\langle -,- \rangle \colon
\mathrm{X}(\mathbf{T}) \times \mathrm{X}^\vee(\mathbf{T}) \to
\Hom(\mathbf{G}_m, \mathbf{G}_m) \cong \mathbf{Z}
\]
be the natural root pairing, so that
\(\langle \epsilon_i, \epsilon_j^\vee \rangle = \delta_{ij}\). Let
\[
\alpha_1 \coloneqq \epsilon_1 - \epsilon_2,\;\;
\alpha_2 \coloneqq \epsilon_2 - \epsilon_3,
\quad\text{and}\quad
\alpha_1^\vee \coloneqq \epsilon_1^\vee - \epsilon_2^\vee,\;\;
\alpha_2^\vee \coloneqq \epsilon_2^\vee - \epsilon_3^\vee,
\]
be the simple roots and coroots corresponding to the choice of \(\mathbf{B}\),
so that the positive roots are
\(\Phi^+ \coloneqq \set{\alpha_1, \alpha_2, \alpha_1 + \alpha_2}\). Let
\[
\varpi_1 \coloneqq \epsilon_1
\quad\text{and}\quad
\varpi_2 \coloneqq \epsilon_1 + \epsilon_2
\]
be the fundamental weights, characterized as the dual basis to
\(\{\alpha_1^\vee, \alpha_2^\vee\}\) under the character pairing. In particular,
the half sum of all positive roots is given by \(\rho = \varpi_1 + \varpi_2\).
The fundamental weights form a basis of the weight lattice and their nonnegative
combinations form the set of dominant weights, the set of which is denoted by
\[
\mathrm{X}_+(\mathbf{T}) =
\Set{a \varpi_1 + b \varpi_2 \in \mathrm{X}(\mathbf{T}) |
a,b \in \mathbf{Z}_{\geq 0}}.
\]
Highest weight theory gives a bijection between the set of simple
representations of \(\mathbf{SL}_3\) and the set of dominant weights; the
simple representation corresponding to the dominant weight
\(a\varpi_1 + b\varpi_2 \in \mathrm{X}_+(\mathbf{T})\) is denoted by \(L(a,b)\).

\subsection{Flag variety}\label{representations-flag-variety}
Let \(\Flag(V) \cong \mathbf{SL}_3/\mathbf{B}\) be the full flag variety of \(V\).
As \(V\) is \(3\)-dimensional, the wedge product pairing yields an isomorphism
\(\wedge^2 V \cong V^\vee\) as \(\mathbf{SL}_3\) modules. Therefore the
Pl\"ucker embedding exhibits \(\Flag(V)\) as the point-line incidence
correspondence in \(\PP V \times \PP V^\vee\); this is a divisor of bidegree
\((1,1)\). Thus
\[
\Pic(\Flag(V)) =
\Set{
  \sO_{\Flag(V)}(a,b) \coloneqq
  \sO_{\PP V}(a) \boxtimes \sO_{\PP V^\vee}(b)\rvert_{\Flag(V)} |
  a,b \in \mathbf{Z}
}
\cong \mathbf{Z}^{\oplus 2}.
\]

\begin{Lemma}\label{representations-flag-ab}
For \(a,b \in \mathbf{Z}\),
\[
\pr_{\PP V,*}(\sO_{\Flag(V)}(a,b)) =
\begin{dcases*}
\Sym^b(\mathcal{T}_{\PP V}(-1)) \otimes \sO_{\PP V}(a) & if \(b \geq 0\), \\
0 & if \(b < 0\).
\end{dcases*}
\]
\end{Lemma}

\begin{proof}
Write \(\sO(a,b) \coloneqq \sO_{\PP V}(a) \boxtimes \sO_{\PP V^\vee}(b)\)
for the line bundle on \(\PP V \times \PP V^\vee\).
Then \(\pr_{\PP V,*}\sO(a,b)\) vanishes if \(b < 0\) and is
\(\Sym^b(V) \otimes \sO_{\PP V}(a)\) otherwise. Since \(\Flag(V)\)
defined in \(\PP V \times \PP V^\vee\) by the trace section of
\(\mathrm{H}^0(\PP V \times \PP V^\vee, \sO(1,1)) = V^\vee \otimes V\),
\[
\pr_{\PP V,*}\sO_{\Flag(V)}(a,b) \cong
\coker\big(\pr_{\PP V,*}\sO(a-1,b-1) \to \pr_{\PP V,*}\sO(a,b)\big).
\]
This vanishes when \(b < 0\), and otherwise yields
\[
\pr_{\PP V,*}\sO_{\Flag(V)}(a,b) \cong
\coker\big(
\Sym^{b-1}(V) \otimes \sO_{\PP V}(a-1) \to
\Sym^b(V) \otimes \sO_{\PP V}(a)\big).
\]
The map in the cokernel arises by applying \(\Sym^b\) to the Euler sequence
\[
0 \to
\sO_{\PP V}(-1) \to
V \otimes \sO_{\PP V} \to
\mathcal{T}_{\PP V}(-1) \to
0
\]
together with a twist by \(\sO_{\PP V}(a)\). This gives the result.
\end{proof}

\subsection{}\label{representations-homogeneous}
There is an isomorphism of abelian groups
\(\mathrm{X}(\mathbf{T}) \to \Pic(\Flag(V))\) given by
\[
a \varpi_1 + b \varpi_2 \mapsto
\sO_{\Flag(V)}(a\varpi_1 + b\varpi_2) \coloneqq \sO_{\Flag(V)}(a,b).
\]
Following the conventions of \cite[II.2.13(1)]{Jantzen:RAGS}, the
\emph{Weyl module} corresponding to a dominant weight
\(a \varpi_1 + b \varpi_2 \in \mathrm{X}_+(\mathbf{T})\) is
\[
\Delta(a,b)
\coloneqq \mathrm{H}^0(\Flag(V), \sO_{\Flag(V)}(b,a))^\vee
\cong \mathrm{H}^0(\PP V, \Sym^a(\mathcal{T}_{\PP V}(-1)) \otimes \sO_{\PP V}(b))^\vee
\]
the isomorphism due to \parref{representations-flag-ab}.
For example, \(\Delta(a,0) = \Sym^a(V)^\vee\) and \(\Delta(0,b) = \Div^b(V)\).
Here, \(\Div^b(V)\) is the \(b\)-th divided power of \(V\), and is defined to
be the subspace of symmetric tensors in \(V^{\otimes b}\), and satisfies
\(\Div^b(V) \cong \Sym^b(V^\vee)^\vee\).

\section{Borel--Weil--Bott Theorem}
Let \(\mathbf{G}\) be a reductive linear algebraic group and let \(\mathbf{P}\)
be a parabolic subgroup. The classical Borel--Weil--Bott Theorem of
\cite{Bott, Demazure} determines the cohomology of homogeneous vector bundles
on the projective variety \(\mathbf{G}/\mathbf{P}\) associated with
completely reducible representations of \(\mathbf{P}\). The matter is rather
subtle in positive characteristic: see \cite[II.5.5]{Jantzen:RAGS}. Remarkably,
in the case \(\mathbf{G} = \SL_3\), a complete answer can be given.

\begin{Definition}\label{representation-BWB-definition}
Let \(\lambda \coloneqq a \varpi_1 + b \varpi_2 \in \mathrm{X}(\mathbf{T})\)
be a weight.
\begin{enumerate}
\item The weight \(\lambda\) is \emph{singular} if either \(a = -1\) or
\(b = -1\) or \(a + b = -2\).
\item If \(\lambda\) is not singular, then its \emph{index}
\(\idx(\lambda)\) is the number of negative integers in the set
\(\set{a+1, b+1, a+b+2}\).
\item The weight \(\lambda\) is said to \emph{satisfy BWB} if
\[
\mathbf{R}\Gamma(\Flag(V), \sO_{\Flag(V)}(\lambda)) =
\begin{dcases*}
0 & if \(\lambda\) is singular, and \\
\mathrm{H}^{\idx(\lambda)}(\Flag(V), \sO_{\Flag(V)}(\lambda))
& if \(\lambda\) is not singular.
\end{dcases*}
\]
\item The weight \(\lambda\) is said to be in the \emph{Griffith region} if \(a+1\)
and \(b+1\) have opposite signs, and there are positive integers \(\nu\)
and \(m < p\) such that
\[ m p^\nu < \abs{a+1}, \abs{b+1} < (m+1) p^\nu. \]
\end{enumerate}
\end{Definition}

In general, Kempf's Theorem \cite[Theorem 1 on p.586]{Kempf} shows that any
dominant weight satisfies BWB; by Serre duality, any
anti-dominant weight also satisfies BWB. Taking this into account, Griffith was
able to completely classify those weights which satisfy BWB:

\begin{Theorem}[{\cite[Theorem 1.3]{Griffith:BWB}}]\label{representations-BWB}
A weight \(\lambda \in \mathrm{X}(\mathbf{T})\) satisfies BWB if and only if
\(\lambda\) is not in the Griffith region. \qed
\end{Theorem}

This allows one to compute cohomology of certain homogeneous bundles on
\(\PP V = \PP^2\). Of particular use will be:

\begin{Corollary}\label{representations-BWB-PP2}
Let \(0 \leq b \leq p - 1\). Then
\begin{enumerate}
\item\label{representations-BWB-PP2.H0}
\(\mathrm{H}^0(\PP^2, \Sym^b(\mathcal{T}_{\PP^2}(-1))(a)) = 0\)
whenever \(a < 0\), and
\item\label{representations-BWB-PP2.H1}
\(\mathrm{H}^1(\PP^2, \Sym^b(\mathcal{T}_{\PP^2}(-1))(a)) = 0\)
whenever \(a < p\).
\end{enumerate}
\end{Corollary}

\begin{proof}
Consider the weight \(\lambda = a\varpi_1 + b\varpi_2\). Since
\(0 \leq b \leq p - 1\), there can be no positive integers \(\nu\) and
\(m < p\) such that \(m p^\nu < b + 1 < (m+1)p^\nu\). Therefore such
\(\lambda\) is never in the Griffith region, and so it satisfies BWB by
\parref{representations-BWB}. If \(a < 0\), then either \(a = -1\) and
\(\lambda\) is singular, or \(a + 1 < 0\) and the index of \(\lambda\)
is at least \(1\). If, furthermore,
\(a < p\), then either \(a+b=-2\) when \(a = -p-1\) and \(b = p-1\) and \(\lambda\)
is singular, or else \(a+b+2 < 0\) and the index of \(\lambda\) is \(2\).
Since, by \parref{representations-flag-ab},
\[
\mathbf{R}\Gamma(\PP^2,\Sym^b(\mathcal{T}_{\PP^2}(-1))(a))
= \mathbf{R}\Gamma(\Flag(V),\sO_{\Flag(V)}(a,b))
\]
this gives the result.
\end{proof}

\section{Some simple modules}\label{representations-some-simple-modules}
The Weyl modules \(\Delta(\lambda)\), as defined in
\parref{representations-homogeneous}, are generally not irreducible
representations in positive characteristic. Their simple composition factors
can sometimes be described using Jantzen's filtration together with a
remarkable nonnegative sum formula, as described in
\cite[II.8.19]{Jantzen:RAGS}. This is recalled in the following, and will be
used to determine some simple representations of \(\SL_3\). The Section ends
with some comments as to how these results apply to representations of the
finite special unitary group \(\mathrm{SU}_3(p)\).

\subsection{Jantzen filtration and sum formula}\label{representations-filtration}
Given a dominant weight \(\lambda \in \mathrm{X}_+(\mathbf{T})\), the Jantzen
filtration is a decreasing filtration
\[
\Delta(\lambda) =
\Delta(\lambda)^0 \supseteq
\Delta(\lambda)^1 \supseteq
\Delta(\lambda)^2 \supseteq \cdots
\]
such that \(L(\lambda) = \Delta(\lambda)/\Delta(\lambda)^1\). Furthermore,
there is the \emph{sum formula}:
\[
\sum\nolimits_{i > 0} \ch(\Delta(\lambda)^i) =
\sum\nolimits_{\alpha \in \Phi^+} \sum\nolimits_{m : 0 < mp < \langle \lambda + \rho, \alpha^\vee\rangle}
\nu_p(mp) \chi(s_{\alpha,mp} \cdot \lambda)
\]
where \(\ch\) extracts the \(\mathbf{T}\)-character of a module,
\(\nu_p \colon \mathbf{Z} \to \mathbf{Z}\) is the standard \(p\)-adic valuation,
\(s_{\alpha,mp}\) is the affine reflection on \(\mathrm{X}(\mathbf{T})\) given by
\[
s_{\alpha,mp}(\lambda) \coloneqq
\lambda + (mp - \langle \lambda, \alpha^\vee \rangle) \alpha,
\]
\(s_{\alpha,mp} \cdot \lambda \coloneqq s_{\alpha,mp}(\lambda + \rho) - \rho\)
is the dot action, and
\[
\chi(\lambda) \coloneqq
\sum\nolimits_{i \geq 0}
(-1)^i \big[\mathrm{H}^i(\Flag(V), \sO_{\Flag(V)}(\lambda))\big]
\]
is the Euler characteristic of \(\sO_{\Flag(V)}(\lambda)\) with values in the
representation ring of \(\mathbf{T}\). A simple application of this is:

\begin{Lemma}\label{representations-easy-simples}
If \(a,b \in \mathbf{Z}_{\geq 0}\) satisfy \(a + b \leq p - 2\), then
\(\Delta(a,b)\) is simple.
\end{Lemma}

\begin{proof}
Consider the Jantzen filtration of \(\Delta(a,b)\). Taking
\(\lambda = a\varpi_1 + b\varpi_2\), the root pairings appearing in the
sum formula are
\[
\langle \lambda + \rho, \alpha_1^\vee \rangle = a + 1, \quad
\langle \lambda + \rho, \alpha_2^\vee \rangle = b + 1, \quad
\langle \lambda + \rho, \alpha_1^\vee + \alpha_2^\vee \rangle = a + b + 2.
\]
Since \(a + b \leq p - 2\), each of these pairings is at most \(p\), so
the right hand side of the sum formula is empty. Thus \(\Delta(a,b)^1 = 0\)
and \(\Delta(a,b) = L(a,b)\) is simple.
\end{proof}

\begin{Lemma}\label{representations-divs}
The Weyl module of highest weight \(b\varpi_2\)
is \(\Delta(0,b) = \Div^b(V)\).
\begin{enumerate}
\item\label{representations-divs.simple}
If \(0 \leq b \leq p-1\), then \(\Delta(0,b) = L(0,b)\) is simple.
\item\label{representations-divs.not-simple}
If \(p \leq b \leq 2p-3\), then \(L(0,b) = \Fr^*(V) \otimes \Div^{b-p}(V)\) and
there is a short exact sequence
\[ 0 \to L(b-p+1,2p-2-b) \to \Delta(0,b) \to L(0,b) \to 0. \]
\end{enumerate}
\end{Lemma}

\begin{proof}
That \(\Delta(0,b) = \Div^b(V)\) follows immediately from their construction
in \parref{representations-homogeneous}.
Case \ref{representations-divs.simple} follows from
\parref{representations-easy-simples}. So suppose
\(p \leq b \leq 2p-3\). The Steinberg Tensor Product Theorem,
\cite[II.3.17]{Jantzen:RAGS}, together with \ref{representations-divs.simple}
gives
\[L(0,b) = \Fr^*(L(0,1)) \otimes L(0,b-p) = \Fr^*(V) \otimes \Div^{b-p}(V). \]
As for the short exact sequence, consider the Jantzen filtration on
\(\Delta(0,b)\). The sum formula of \parref{representations-filtration}
has two terms indexed by \((\alpha_2,p)\) and
\((\alpha_1+\alpha_2,p)\). The first is
\begin{align*}
\chi(s_{\alpha_2,p} \cdot b\varpi_2)
& = \chi((b-p+1)\varpi_1 + (2p-2-b)\varpi_2) \\
& = \ch(\Delta(b-p+1,2p-2-b))
\end{align*}
since \((b-p+1)\varpi_1 + (2p-2-b)\varpi_2\) is dominant. To express this
Express this in terms of simple characters by considering the Jantzen
filtration for \(\Delta(b-p+1,2p-b-2)\); the sum formula contains a single term
indexed by \((\alpha_1+\alpha_2,p)\), and
\[
s_{\alpha_1+\alpha_2,p} \cdot ((b-p+1)\varpi_1 + (2p-b-2)\varpi_2)
= (b-p)\varpi_1 + (2p-3-b)\varpi_2.
\]
This final weight is dominant and simple by \parref{representations-easy-simples},
so
\[\ch(\Delta(b-p+1,2p-2-b)) = \ch(L(b-p+1,2p-2-b)) + \ch(L(b-p,2p-3-b)). \]

Consider the second term in the sum formula for \(\Delta(0,b)\),
indexed by \((\alpha_1+\alpha_2,p)\). Observe that
\begin{align*}
s_{\alpha_1+\alpha_2,p} \cdot b\varpi_2
& = (p-b-2)\varpi_1 + (p-2)\varpi_2  \\
& = s_{\alpha_1} \cdot ((b-p)\varpi_1 + (2p-3-b)\varpi_2).
\end{align*}
Therefore, \cite[II.5.9]{Jantzen:RAGS} gives the first equality in
\begin{align*}
\chi(s_{\alpha_1+\alpha_2,p} \cdot b\varpi_2)
& = -\chi((b-p)\varpi_1 + (2p-3-b)\varpi_2) \\
& = -\ch(L(b-p,2p-3-b)).
\end{align*}
Putting everything together shows that
\[
\sum\nolimits_{i > 0} \ch(\Delta(0,b)^i) = \ch(L(b-p+1,2p-2-b)).
\]
Therefore the only composition factor in \(\Delta(0,b)^1\) can be \(L(b-p+1,2p-2-b)\),
so it is simple and \(\Delta(0,b)^2 = 0\). This gives the exact sequence
in \ref{representations-divs.not-simple}.
\end{proof}

The conclusion of \ref{representations-divs.not-simple} means that
\(\ker(\Div^b(V) \to \Fr^*(V) \otimes \Div^{b-p}(V))\)
is a simple \(\SL_3\) representation when \(p \leq b \leq 2p-3\);
dually, \(\Sym^b(V^\vee)/(\Fr^*(V^\vee) \otimes \Sym^{b-p}(V^\vee))\)
is simple for \(p \leq b \leq 2p-3\). This can be established in a more
elementary way, as done in \cite{Doty}.

\begin{Lemma}\label{representations-1-b}
The Weyl module of highest weight \(\varpi_1 + b \varpi_2\) is given by
\[
\Delta(1,b) \coloneqq
\ker\big(\operatorname{ev} \colon V^\vee \otimes \Div^b(V) \to \Div^{b-1}(V)\big).
\]
\begin{enumerate}
\item\label{representations-1-b.simple}
If \(0 \leq b \leq p - 3\), then \(\Delta(1,b) = L(1,b)\) is simple.
\item\label{representations-1-b.not-simple}
If \(b = p - 2\), then there is a short exact sequence
\[ 0 \to L(0,p-3) \to \Delta(1,p-2) \to L(1,p-2) \to 0. \]
\end{enumerate}
\end{Lemma}

\begin{proof}
Its definition in \parref{representations-homogeneous} together with the
Euler sequence on \(\PP V\) gives
\begin{align*}
\Delta(1,b)
& \cong \mathrm{H}^0(\PP V, \mathcal{T}_{\PP V}(b-1))^\vee \\
& \cong \ker\big(V^\vee \otimes \mathrm{H}^0(\PP V,\sO_{\PP V}(b))^\vee \to \mathrm{H}^0(\PP V,\sO_{\PP V}(b-1))^\vee\big) \\
& \cong \ker\big(\operatorname{ev} \colon V^\vee \otimes \Div^b(V) \to \Div^{b-1}(V)\big).
\end{align*}
Simplicity in case \ref{representations-1-b.simple} follows from \parref{representations-easy-simples}.
So consider case \ref{representations-1-b.not-simple} and consider the Janzten
filtration on \(\Delta(1,p-2)\). The sum formula of
\parref{representations-filtration} reads
\[
\sum\nolimits_{i > 0} \ch(\Delta(1,p-2)^i) =
\chi(s_{\alpha_1+\alpha_2,p} \cdot (\varpi_1 + (p-2)\varpi_2)) =
\ch(L(0,p-3)).
\]
Thus \(L(0,p-3)\) is the only composition factor in \(\Delta(1,p-2)^1\),
giving \ref{representations-1-b.not-simple}.
\end{proof}

The simple submodule \(L(0,p-3) \subset \Delta(1,p-2)\) in
\parref{representations-1-b}\ref{representations-1-b.not-simple} can be explicitly
identified as follows. Choose a basis \(V = \langle x,y,z \rangle\) and let
\(\partial_x, \partial_y, \partial_z \in V^\vee\) be the associated differential
operators. Then the image of the map
\(\Div^{p-3}(V) \to \Div^{p-2}(V) \otimes V^\vee\) given by
\[
f \mapsto
xf \otimes \partial_x + yf \otimes \partial_y + zf \otimes \partial_z
\]
lies in the kernel of the evaluation map, showing
\(\Div^{p-3}(V) \cong L(0,p-3)\).

\subsection{Finite unitary group}\label{representations-SU}
Let \((V,\beta)\) be a nondegenerate \(q\)-bic form over \(\kk\). The
\emph{special unitary group} of \(\beta\) is the finite \'etale group scheme
\(\mathrm{SU}(V,\beta) \coloneqq \mathrm{U}(V,\beta) \cap \SL(V)\) defined as
the subgroup of the unitary group, as in \parref{forms-aut-unitary}, contained
in \(\SL(V)\). Representations of \(\SL(V)\) restrict to
representations of \(\mathrm{SU}(V,\beta)\), are often irreducible, and all
irreducible representations are obtained in this way. In the case at hand:

\begin{Theorem}\label{representations-steinberg-restriction}
Let \(0 \leq a,b \leq p-1\). Then the restriction of the
simple \(\mathbf{SL}_3\)-modules \(L(a,b)\) to \(\mathrm{SU}_3(p)\) remain
simple, are pairwise nonisomorphic, and give all isomorphism classes of
simple \(\mathrm{SU}_3(p)\)-modules.
\end{Theorem}

\begin{proof}
This follows from Steinberg's Restriction Theorem \cite{Steinberg}; see also
\cite[Theorem 2.11]{Humphreys}.
\end{proof}

By abuse of notation, \(L(a,b)\) and \(\Delta(a,b)\) will be used to denote the
\(\mathrm{SU}_3(p)\)-modules obtained via restriction of the corresponding
\(\SL_3\)-modules. The Steinberg Restriction Theorem together with
\parref{representations-easy-simples} and \parref{representations-divs}
produces some simple modules for \(\mathrm{SU}_3(p)\).
The next statement identifies a certain \(\mathrm{SU}_3(p)\)-module as a
restriction of a \(\SL_3\)-module:

\begin{Lemma}\label{representations-F-weyl}
For each \(0 \leq b \leq p - 1\), the \(\mathrm{SU}_3(p)\) representation
\[
\ker(f \colon \Fr^*(V) \otimes \Div^b(V) \xrightarrow{\beta} V^\vee \otimes \Div^b(V) \xrightarrow{\mathrm{ev}} \Div^{b-1}(V))
\]
is isomorphic to the restriction of the Weyl module of \(\mathbf{SL}_3\) with
highest weight \(\varpi_1 + b\varpi_2\).
\end{Lemma}

\begin{proof}
By construction, there is a \(\mathrm{SU}_3(p)\)-equivariant commutative diagram
\[
\begin{tikzcd}
\Fr^*(V) \otimes \Div^b(V) \dar["\beta^\vee"'] \rar["f"']
& \Div^{b-1}(V) \dar[equal] \\
V^\vee \otimes \Div^b(V) \rar["\mathrm{ev}"]
& \Div^{b-1}(V)\punct{.}
\end{tikzcd}
\]
Thus the kernels of the two rows are isomorphic as representations of
\(\mathrm{SU}_3(p)\). Now the bottom row is equivariant for \(\mathbf{SL}_3\),
and by \parref{representations-1-b}, its kernel is precisely the Weyl module of
weight \(\varpi_1 + b\varpi_2\) for \(\mathbf{SL}_3\), as claimed.
\end{proof}

%% file: main.bbl
\newcommand{\etalchar}[1]{$^{#1}$}
\providecommand{\bysame}{\leavevmode\hbox to3em{\hrulefill}\thinspace}
\providecommand{\MR}{\relax\ifhmode\unskip\space\fi MR }
\providecommand{\MRhref}[2]{%
  \href{http://www.ams.org/mathscinet-getitem?mr=#1}{#2}
}
\providecommand{\href}[2]{#2}
\begin{thebibliography}{KKP{\etalchar{+}}21}

\bibitem[Agu19]{Aguglia:Threefolds}
A.~Aguglia, \emph{Characterizing {H}ermitian varieties in three- and
  four-dimensional projective spaces}, J. Aust. Math. Soc. \textbf{107} (2019),
  no.~1, 1--8.

\bibitem[AK77]{AK:Fano}
A.~B. Altman and S.~L. Kleiman, \emph{Foundations of the theory of {F}ano
  schemes}, Compositio Math. \textbf{34} (1977), no.~1, 3--47.

\bibitem[AP15]{AP:Superspecial}
J.~D. Achter and R.~Pries, \emph{Superspecial rank of supersingular abelian
  varieties and {J}acobians}, J. Th\'{e}or. Nombres Bordeaux \textbf{27}
  (2015), no.~3, 605--624.

\bibitem[AP20]{AP:Threefolds}
A.~Aguglia and F.~Pavese, \emph{On non-singular {H}ermitian varieties of {${\rm
  PG}(4,q^2)$}}, Discrete Math. \textbf{343} (2020), no.~1, 111634, 5.

\bibitem[BC66]{BC:Hermitian}
R.~C. Bose and I.~M. Chakravarti, \emph{Hermitian varieties in a finite
  projective space {${\rm PG}(N,\,q^{2})$}}, Canadian J. Math. \textbf{18}
  (1966), 1161--1182.

\bibitem[BDE{\etalchar{+}}13]{BDENY:Fermat}
T.~Bridges, R.~Datta, J.~Eddy, M.~Newman, and J.~Yu, \emph{Free and very free
  morphisms into a {F}ermat hypersurface}, Involve \textbf{6} (2013), no.~4,
  437--445.

\bibitem[BDH21]{BDH:Sorenson}
P.~Beelen, M.~Datta, and M.~Homma, \emph{A proof of {S}\o rensen's conjecture
  on {H}ermitian surfaces}, Proc. Amer. Math. Soc. \textbf{149} (2021), no.~4,
  1431--1441.

\bibitem[Bea77]{Beauville:Prym}
A.~Beauville, \emph{Vari\'{e}t\'{e}s de {P}rym et jacobiennes
  interm\'{e}diaires}, Ann. Sci. \'{E}cole Norm. Sup. (4) \textbf{10} (1977),
  no.~3, 309--391.

\bibitem[Bea90]{Beauville:Moduli}
\bysame, \emph{Sur les hypersurfaces dont les sections hyperplanes sont \`a
  module constant}, The {G}rothendieck {F}estschrift, {V}ol. {I}, Progr. Math.,
  vol.~86, Birkh\"{a}user Boston, Boston, MA, 1990, pp.~121--133.

\bibitem[BM76]{BM:III}
E.~Bombieri and D.~Mumford, \emph{Enriques' classification of surfaces in char.
  {$p$}. {III}}, Invent. Math. \textbf{35} (1976), 197--232.

\bibitem[Bon11]{Bonnafe:SL2}
C.~Bonnaf\'{e}, \emph{Representations of {${\rm SL}_2(\mathbf{F}_q)$}}, Algebra
  and Applications, vol.~13, Springer-Verlag London, Ltd., London, 2011.

\bibitem[Bot57]{Bott}
R.~Bott, \emph{Homogeneous vector bundles}, Ann. of Math. (2) \textbf{66}
  (1957), 203--248.

\bibitem[BPRS21]{BPRS}
A.~Brosowsky, J.~Page, T.~Ryan, and K.~E. Smith, \emph{Geometry of smooth
  extremal surfaces}, 2021.

\bibitem[BVdV79]{BVV:Fano}
W.~Barth and A.~Van~de Ven, \emph{Fano varieties of lines on hypersurfaces},
  Arch. Math. (Basel) \textbf{31} (1978/79), no.~1, 96--104.

\bibitem[CC10]{CC:StarPoints}
F.~Cools and M.~Coppens, \emph{Star points on smooth hypersurfaces}, J. Algebra
  \textbf{323} (2010), no.~1, 261--286.

\bibitem[CC11]{CC:StarPointsII}
\bysame, \emph{Singular hypersurfaces possessing infinitely many star points},
  Proc. Amer. Math. Soc. \textbf{139} (2011), no.~10, 3413--3422.

\bibitem[CCN{\etalchar{+}}85]{ATLAS}
J.~H. Conway, R.~T. Curtis, S.~P. Norton, R.~A. Parker, and R.~A. Wilson,
  \emph{{$\Bbb{ATLAS}$} of finite groups}, Oxford University Press, Eynsham,
  1985.

\bibitem[CG72]{CG}
C.~H. Clemens and P.~A. Griffiths, \emph{The intermediate {J}acobian of the
  cubic threefold}, Ann. of Math. (2) \textbf{95} (1972), 281--356.

\bibitem[Cha78]{Chang:Automorphisms}
H.~C. Chang, \emph{On plane algebraic curves}, Chinese J. Math. \textbf{6}
  (1978), no.~2, 185--189.

\bibitem[CK03]{CK:Ovoids}
A.~Cossidente and G.~Korchm\'{a}ros, \emph{Transitive ovoids of the {H}ermitian
  surface of {${\rm PG}(3,q^2)$}, {$q$} even}, J. Combin. Theory Ser. A
  \textbf{101} (2003), no.~1, 117--130.

\bibitem[CKT00]{CKT:HermitianCover}
A.~Cossidente, G.~Korchm\'{a}ros, and F.~Torres, \emph{Curves of large genus
  covered by the {H}ermitian curve}, Comm. Algebra \textbf{28} (2000), no.~10,
  4707--4728.

\bibitem[Col79]{Collino}
A.~Collino, \emph{Lines on quartic threefolds}, J. London Math. Soc. (2)
  \textbf{19} (1979), no.~2, 257--267.

\bibitem[Con06]{Conduche}
D.~Conduch\'e, \emph{Courbes rationnelles et hypersurfaces de l'espace
  projectif}, Ph.D. thesis, Universit\'e Louis Pasteur, 2006.

\bibitem[CR19]{CR:RationalCurves}
I.~Coskun and E.~Riedl, \emph{Normal bundles of rational curves on complete
  intersections}, Commun. Contemp. Math. \textbf{21} (2019), no.~2, 1850011,
  29.

\bibitem[CZ14]{CZ14}
Q.~Chen and Y.~Zhu, \emph{Very free curves on {F}ano complete intersections},
  Algebr. Geom. \textbf{1} (2014), no.~5, 558--572.

\bibitem[Deb01]{Debarre:HDAG}
O.~Debarre, \emph{Higher-dimensional algebraic geometry}, Universitext,
  Springer-Verlag, New York, 2001.

\bibitem[Dem68]{Demazure}
M.~Demazure, \emph{Une d\'{e}monstration alg\'{e}brique d'un th\'{e}or\`eme de
  {B}ott}, Invent. Math. \textbf{5} (1968), 349--356.

\bibitem[DG70]{DG}
M.~Demazure and P.~Gabriel, \emph{Groupes alg\'{e}briques. {T}ome {I}:
  {G}\'{e}om\'{e}trie alg\'{e}brique, g\'{e}n\'{e}ralit\'{e}s, groupes
  commutatifs}, Masson \& Cie, \'{E}diteurs, Paris; North-Holland Publishing
  Co., Amsterdam, 1970.

\bibitem[DL76]{DL}
P.~Deligne and G.~Lusztig, \emph{Representations of reductive groups over
  finite fields}, Ann. of Math. (2) \textbf{103} (1976), no.~1, 103--161.

\bibitem[DLR17]{DLR:Lines}
O.~Debarre, A.~Laface, and X.~Roulleau, \emph{Lines on cubic hypersurfaces over
  finite fields}, Geometry over nonclosed fields, Simons Symp., Springer, Cham,
  2017, pp.~19--51.

\bibitem[DM98]{DM:Fano}
O.~Debarre and L.~Manivel, \emph{Sur la vari\'{e}t\'{e} des espaces
  lin\'{e}aires contenus dans une intersection compl\`ete}, Math. Ann.
  \textbf{312} (1998), no.~3, 549--574.

\bibitem[Dot85]{Doty}
S.~R. Doty, \emph{The submodule structure of certain {W}eyl modules for groups
  of type {$A_n$}}, J. Algebra \textbf{95} (1985), no.~2, 373--383.

\bibitem[Dum95]{Dummigan:MW}
N.~Dummigan, \emph{The determinants of certain {M}ordell-{W}eil lattices},
  Amer. J. Math. \textbf{117} (1995), no.~6, 1409--1429.

\bibitem[EH99]{EH:Lines}
G.~L. Ebert and J.~W.~P. Hirschfeld, \emph{Complete systems of lines on a
  {H}ermitian surface over a finite field}, Des. Codes Cryptogr. \textbf{17}
  (1999), no.~1-3, 253--268.

\bibitem[Eke87]{Ekedahl:Supersingular}
T.~Ekedahl, \emph{On supersingular curves and abelian varieties}, Math. Scand.
  \textbf{60} (1987), no.~2, 151--178.

\bibitem[Eke04]{Ekedahl:CY}
\bysame, \emph{On non-liftable {C}alabi--{Y}au threefolds}, 2004.

\bibitem[FGI{\etalchar{+}}05]{FGAExplained}
B.~Fantechi, L.~G\"{o}ttsche, L.~Illusie, S.~L. Kleiman, N.~Nitsure, and
  A.~Vistoli, \emph{Fundamental algebraic geometry}, Mathematical Surveys and
  Monographs, vol. 123, American Mathematical Society, Providence, RI, 2005.

\bibitem[FGT97]{FGT:Maximal}
R.~Fuhrmann, A.~Garcia, and F.~Torres, \emph{On maximal curves}, J. Number
  Theory \textbf{67} (1997), no.~1, 29--51.

\bibitem[Ful98]{Fulton}
W.~Fulton, \emph{Intersection theory}, second ed., Ergebnisse der Mathematik
  und ihrer Grenzgebiete. 3. Folge. A Series of Modern Surveys in Mathematics,
  vol.~2, Springer-Verlag, Berlin, 1998.

\bibitem[GK03]{GK:Ovoids}
L.~Giuzzi and G.~Korchm\'{a}ros, \emph{Ovoids of the {H}ermitian surface in odd
  characteristic}, no. suppl., 2003, pp.~S49--S58.

\bibitem[Gop81]{Goppa:Codes}
V.~D. Goppa, \emph{Codes on algebraic curves}, Dokl. Akad. Nauk SSSR
  \textbf{259} (1981), no.~6, 1289--1290.

\bibitem[Gri80]{Griffith:BWB}
W.~L. Griffith, Jr., \emph{Cohomology of flag varieties in characteristic
  {$p$}}, Illinois J. Math. \textbf{24} (1980), no.~3, 452--461.

\bibitem[Gro58]{Grothendieck:Chern}
A.~Grothendieck, \emph{La th\'{e}orie des classes de {C}hern}, Bull. Soc. Math.
  France \textbf{86} (1958), 137--154.

\bibitem[GS95]{GS:Codes}
A.~Garc\'{\i}a and H.~Stichtenoth, \emph{A tower of {A}rtin-{S}chreier
  extensions of function fields attaining the {D}rinfeld-{V}l\u{a}du\c{t}
  bound}, Invent. Math. \textbf{121} (1995), no.~1, 211--222.

\bibitem[GSX00]{GSX:HermitianQuotients}
A.~Garcia, H.~Stichtenoth, and C.-P. Xing, \emph{On subfields of the
  {H}ermitian function field}, Compositio Math. \textbf{120} (2000), no.~2,
  137--170.

\bibitem[GSY15]{GSY:CanonicalRep}
B.~Gunby, A.~Smith, and A.~Yuan, \emph{Irreducible canonical representations in
  positive characteristic}, Res. Number Theory \textbf{1} (2015), Paper No. 3,
  25.

\bibitem[GT08]{GT:Maximal}
A.~Garcia and S.~Tafazolian, \emph{Certain maximal curves and {C}artier
  operators}, Acta Arith. \textbf{135} (2008), no.~3, 199--218.

\bibitem[Han92a]{Hansen:DL}
J.~P. Hansen, \emph{Deligne-{L}usztig varieties and group codes}, Coding theory
  and algebraic geometry ({L}uminy, 1991), Lecture Notes in Math., vol. 1518,
  Springer, Berlin, 1992, pp.~63--81.

\bibitem[Han92b]{Hansen:Codes}
\bysame, \emph{Deligne-{L}usztig varieties and group codes}, Coding theory and
  algebraic geometry ({L}uminy, 1991), Lecture Notes in Math., vol. 1518,
  Springer, Berlin, 1992, pp.~63--81.

\bibitem[Har77]{Hartshorne:AG}
R.~Hartshorne, \emph{Algebraic geometry}, Graduate Texts in Mathematics, No.
  52, Springer-Verlag, New York-Heidelberg, 1977.

\bibitem[HE70]{HE:Determinantal}
M.~Hochster and J.~A. Eagon, \emph{A class of perfect determinantal ideals},
  Bull. Amer. Math. Soc. \textbf{76} (1970), 1026--1029.

\bibitem[Hef85]{Hefez:Thesis}
A.~Hefez, \emph{Duality for projective varieties}, Ph.D. thesis, Massachusetts
  Institute of Technology, 1985.

\bibitem[Hef89]{Hefez:Nonreflexive}
\bysame, \emph{Nonreflexive curves}, Compositio Math. \textbf{69} (1989),
  no.~1, 3--35.

\bibitem[Hen78]{Henn}
H.-W. Henn, \emph{Funktionenk\"{o}rper mit grosser {A}utomorphismengruppe}, J.
  Reine Angew. Math. \textbf{302} (1978), 96--115.

\bibitem[HH16]{HH:Fermat}
T.~Hoai~Hoang, \emph{Degeneration of {F}ermat hypersurfaces in positive
  characteristic}, Hiroshima Math. J. \textbf{46} (2016), no.~2, 195--215.

\bibitem[Hir79]{Hirschfeld:Geometries}
J.~W.~P. Hirschfeld, \emph{Projective geometries over finite fields}, Oxford
  Mathematical Monographs, The Clarendon Press, Oxford University Press, New
  York, 1979.

\bibitem[Hir85]{Hirschfeld:Three}
\bysame, \emph{Finite projective spaces of three dimensions}, Oxford
  Mathematical Monographs, The Clarendon Press, Oxford University Press, New
  York, 1985.

\bibitem[Hir00]{Hirokado:Hirzebruch}
M.~Hirokado, \emph{Zariski surfaces as quotients of {H}irzebruch surfaces by
  1-foliations}, Yokohama Math. J. \textbf{47} (2000), no.~2, 103--120.

\bibitem[HJ90]{HJ:Crystalline}
B.~Haastert and J.~C. Jantzen, \emph{Filtrations of the discrete series of
  {${\rm SL}_2(q)$} via crystalline cohomology}, J. Algebra \textbf{132}
  (1990), no.~1, 77--103.

\bibitem[HK85]{HK:Notes}
A.~Hefez and S.~L. Kleiman, \emph{Notes on the duality of projective
  varieties}, Geometry today ({R}ome, 1984), Progr. Math., vol.~60,
  Birkh\"{a}user Boston, Boston, MA, 1985, pp.~143--183.

\bibitem[HK13]{HK:Bound}
M.~Homma and S.~J. Kim, \emph{An elementary bound for the number of points of a
  hypersurface over a finite field}, Finite Fields Appl. \textbf{20} (2013),
  76--83.

\bibitem[HK15]{HK:SurfaceBounds}
\bysame, \emph{Numbers of points of surfaces in the projective 3-space over
  finite fields}, Finite Fields Appl. \textbf{35} (2015), 52--60.

\bibitem[HK16]{HK:Surface}
\bysame, \emph{The characterization of {H}ermitian surfaces by the number of
  points}, J. Geom. \textbf{107} (2016), no.~3, 509--521.

\bibitem[HK17]{HK:EvenDimensional}
\bysame, \emph{Number of points of a nonsingular hypersurface in an
  odd-dimensional projective space}, Finite Fields Appl. \textbf{48} (2017),
  395--419.

\bibitem[HKT08]{HKT:Curves}
J.~W.~P. Hirschfeld, G.~Korchm\'{a}ros, and F.~Torres, \emph{Algebraic curves
  over a finite field}, Princeton Series in Applied Mathematics, Princeton
  University Press, Princeton, NJ, 2008.

\bibitem[HM78]{HM:TT-Lemma}
R.~Hotta and K.~Matsui, \emph{On a lemma of {T}ate-{T}hompson}, Hiroshima Math.
  J. \textbf{8} (1978), no.~2, 255--268.

\bibitem[Hom87]{Homma:FunnyCurves}
M.~Homma, \emph{Funny plane curves in characteristic {$p>0$}}, Comm. Algebra
  \textbf{15} (1987), no.~7, 1469--1501.

\bibitem[Hom89]{Homma:SoupedUp}
\bysame, \emph{A souped-up version of {P}ardini's theorem and its application
  to funny curves}, Compositio Math. \textbf{71} (1989), no.~3, 295--302.

\bibitem[HSTV91]{HSTV:Hermitian}
J.~W.~P. Hirschfeld, L.~Storme, J.~A. Thas, and J.~F. Voloch, \emph{A
  characterization of {H}ermitian curves}, J. Geom. \textbf{41} (1991),
  no.~1-2, 72--78.

\bibitem[HT16]{HT:Geometries}
J.~W.~P. Hirschfeld and J.~A. Thas, \emph{General {G}alois geometries},
  Springer Monographs in Mathematics, Springer, London, 2016.

\bibitem[Hum06]{Humphreys}
J.~E. Humphreys, \emph{Modular representations of finite groups of {L}ie type},
  London Mathematical Society Lecture Note Series, vol. 326, Cambridge
  University Press, Cambridge, 2006.

\bibitem[Huy22]{Huybrechts:Cubics}
D.~Huybrechts, \emph{The geometry of cubic hypersurfaces}, 2022, Available at
  {\small\url{https://www.math.uni-bonn.de/people/huybrech/Notes.pdf}}.

\bibitem[IIL20]{IIL:Rational}
K.~Ito, T.~Ito, and C.~Liedtke, \emph{Deformations of rational curves in
  positive characteristic}, J. Reine Angew. Math. \textbf{769} (2020), 55--86.

\bibitem[Jan03]{Jantzen:RAGS}
J.~C. Jantzen, \emph{Representations of algebraic groups}, second ed.,
  Mathematical Surveys and Monographs, vol. 107, American Mathematical Society,
  Providence, RI, 2003.

\bibitem[Kah21]{Kahn:Jacobian}
B.~Kahn, \emph{On the universal regular homomorphism in codimension 2}, Ann.
  Inst. Fourier (Grenoble) \textbf{71} (2021), no.~2, 843--848.

\bibitem[Kat17]{Katsura:Lefschetz}
T.~Katsura, \emph{Lefschetz pencils on a certain hypersurface in positive
  characteristic}, Higher dimensional algebraic geometry---in honour of
  {P}rofessor {Y}ujiro {K}awamata's sixtieth birthday, Adv. Stud. Pure Math.,
  vol.~74, Math. Soc. Japan, Tokyo, 2017, pp.~265--278.

\bibitem[Kem71]{Kempf:Schubert}
G.~R. Kempf, \emph{Schubert methods with an application to algebraic curves},
  Mathematisch Centrum, Amsterdam, 1971.

\bibitem[Kem76]{Kempf}
\bysame, \emph{Linear systems on homogeneous spaces}, Ann. of Math. (2)
  \textbf{103} (1976), no.~3, 557--591.

\bibitem[KKP{\etalchar{+}}21]{KKPSSW:F-Pure}
Z.~Kadyrsizova, J.~Kenkel, J.~Page, J.~Singh, K.~E. Smith, A.~Vraciu, and E.~E.
  Witt, \emph{Lower bounds on the {F}-pure threshold and extremal
  singularities}, 2021.

\bibitem[Kol96]{Kollar:RationalCurves}
J.~Koll\'{a}r, \emph{Rational curves on algebraic varieties}, Ergebnisse der
  Mathematik und ihrer Grenzgebiete. 3. Folge. A Series of Modern Surveys in
  Mathematics, vol.~32, Springer-Verlag, Berlin, 1996.

\bibitem[Kol15]{Kollar:ST}
\bysame, \emph{Szemer\'{e}di-{T}rotter-type theorems in dimension 3}, Adv.
  Math. \textbf{271} (2015), 30--61.

\bibitem[KP91]{KP:Gauss}
S.~L. Kleiman and R.~Piene, \emph{On the inseparability of the {G}auss map},
  Enumerative algebraic geometry ({C}openhagen, 1989), Contemp. Math., vol.
  123, Amer. Math. Soc., Providence, RI, 1991, pp.~107--129.

\bibitem[Kun69]{Kunz:Flat}
E.~Kunz, \emph{Characterizations of regular local rings of characteristic
  {$p$}}, Amer. J. Math. \textbf{91} (1969), 772--784.

\bibitem[Lan02]{Lang:Algebra}
S.~Lang, \emph{Algebra}, third ed., Graduate Texts in Mathematics, vol. 211,
  Springer-Verlag, New York, 2002.

\bibitem[Lan16]{Langer:BMY}
A.~Langer, \emph{The {B}ogomolov-{M}iyaoka-{Y}au inequality for logarithmic
  surfaces in positive characteristic}, Duke Math. J. \textbf{165} (2016),
  no.~14, 2737--2769.

\bibitem[Lan19]{Langer:Drinfeld}
\bysame, \emph{Birational geometry of compactifications of {D}rinfeld
  half-spaces over a finite field}, Adv. Math. \textbf{345} (2019), 861--908.

\bibitem[Laz04]{Lazarsfeld:PositivityI}
R.~Lazarsfeld, \emph{Positivity in algebraic geometry. {I}}, Ergebnisse der
  Mathematik und ihrer Grenzgebiete. 3. Folge. A Series of Modern Surveys in
  Mathematics, vol.~48, Springer-Verlag, Berlin, 2004.

\bibitem[Lus76]{Lusztig:Green}
G.~Lusztig, \emph{On the {G}reen polynomials of classical groups}, Proc. London
  Math. Soc. (3) \textbf{33} (1976), no.~3, 443--475.

\bibitem[Mil80]{Milne:EC}
J.~S. Milne, \emph{\'{E}tale cohomology}, Princeton Mathematical Series, No.
  33, Princeton University Press, Princeton, N.J., 1980.

\bibitem[Miy77]{Miyaoka:Inequality}
Y.~Miyaoka, \emph{On the {C}hern numbers of surfaces of general type}, Invent.
  Math. \textbf{42} (1977), 225--237.

\bibitem[MM64]{MM:Automorphisms}
H.~Matsumura and P.~Monsky, \emph{On the automorphisms of hypersurfaces}, J.
  Math. Kyoto Univ. \textbf{3} (1963/64), 347--361.

\bibitem[MO67]{MO:Automorphisms}
H.~Matsumura and F.~Oort, \emph{Representability of group functors, and
  automorphisms of algebraic schemes}, Invent. Math. \textbf{4} (1967), 1--25.

\bibitem[MTW05]{MTW:BS}
M.~Musta\c{t}\v{a}, S.~Takagi, and K.~Watanabe, \emph{F-thresholds and
  {B}ernstein-{S}ato polynomials}, European {C}ongress of {M}athematics, Eur.
  Math. Soc., Z\"{u}rich, 2005, pp.~341--364.

\bibitem[Mur72]{Murre:Prym}
J.~P. Murre, \emph{Algebraic equivalence modulo rational equivalence on a cubic
  threefold}, Compositio Math. \textbf{25} (1972), 161--206.

\bibitem[Mur83]{Murre:Jacobian-CR}
\bysame, \emph{Un r\'{e}sultat en th\'{e}orie des cycles alg\'{e}briques de
  codimension deux}, C. R. Acad. Sci. Paris S\'{e}r. I Math. \textbf{296}
  (1983), no.~23, 981--984.

\bibitem[Mur85]{Murre:Jacobian}
\bysame, \emph{Applications of algebraic {$K$}-theory to the theory of
  algebraic cycles}, Algebraic geometry, {S}itges ({B}arcelona), 1983, Lecture
  Notes in Math., vol. 1124, Springer, Berlin, 1985, pp.~216--261.

\bibitem[Nak87]{Nakajima:Auts}
S.~Nakajima, \emph{{$p$}-ranks and automorphism groups of algebraic curves},
  Trans. Amer. Math. Soc. \textbf{303} (1987), no.~2, 595--607.

\bibitem[Nom95]{Noma:Dual}
A.~Noma, \emph{Hypersurfaces with smooth dual varieties}, Amer. J. Math.
  \textbf{117} (1995), no.~6, 1507--1515.

\bibitem[Oji19]{Ojiro:Hermitian}
N.~Ojiro, \emph{Rational curves on a smooth {H}ermitian surface}, Hiroshima
  Math. J. \textbf{49} (2019), no.~1, 161--173.

\bibitem[Par86]{Pardini:Curves}
R.~Pardini, \emph{Some remarks on plane curves over fields of finite
  characteristic}, Compositio Math. \textbf{60} (1986), no.~1, 3--17.

\bibitem[Poo05]{Poonen:Automorphisms}
B.~Poonen, \emph{Varieties without extra automorphisms. {III}.
  {H}ypersurfaces}, Finite Fields Appl. \textbf{11} (2005), no.~2, 230--268.

\bibitem[PW15]{PW:EO}
R.~Pries and C.~Weir, \emph{The {E}kedahl-{O}ort type of {J}acobians of
  {H}ermitian curves}, Asian J. Math. \textbf{19} (2015), no.~5, 845--869.

\bibitem[Rod00]{Rodier:DL}
F.~Rodier, \emph{Nombre de points des surfaces de {D}eligne et {L}usztig}, J.
  Algebra \textbf{227} (2000), no.~2, 706--766.

\bibitem[RS94]{RS:Hermitian}
H.-G. R\"{u}ck and H.~Stichtenoth, \emph{A characterization of {H}ermitian
  function fields over finite fields}, J. Reine Angew. Math. \textbf{457}
  (1994), 185--188.

\bibitem[Seg65]{Segre:Hermitian}
B.~Segre, \emph{Forme e geometrie hermitiane, con particolare riguardo al caso
  finito}, Ann. Mat. Pura Appl. (4) \textbf{70} (1965), 1--201.

\bibitem[SGA7\textsubscript{II}]{SGAVII}
P.~Deligne and N.~Katz, \emph{S\'{e}minaire de {G}\'{e}om\'{e}trie
  {A}lg\'{e}brique du {B}ois-{M}arie 1967--1969 - {G}roupes de {M}onodromie en
  {G}\'{e}om\'{e}trie {A}lg\'{e}brique. {II} {(SGA 7 II)}}, Lecture Notes in
  Mathematics, Vol. 340, Springer-Verlag, Berlin-New York, 1973.

\bibitem[She12]{Shen:Fermat}
M.~Shen, \emph{Rational curves on {F}ermat hypersurfaces}, C. R. Math. Acad.
  Sci. Paris \textbf{350} (2012), no.~15-16, 781--784.

\bibitem[Shi74]{Shioda:Unirational}
T.~Shioda, \emph{An example of unirational surfaces in characteristic {$p$}},
  Math. Ann. \textbf{211} (1974), 233--236.

\bibitem[Shi77a]{Shioda:Supersingular}
\bysame, \emph{On unirationality of supersingular surfaces}, Math. Ann.
  \textbf{225} (1977), no.~2, 155--159.

\bibitem[Shi77b]{Shioda:Unirationality}
\bysame, \emph{Some results on unirationality of algebraic surfaces}, Math.
  Ann. \textbf{230} (1977), no.~2, 153--168.

\bibitem[Shi88]{Shioda:Automorphisms}
\bysame, \emph{Arithmetic and geometry of {F}ermat curves}, Algebraic
  {G}eometry {S}eminar ({S}ingapore, 1987), World Sci. Publishing, Singapore,
  1988, pp.~95--102.

\bibitem[Shi90]{Shimada:Cylinder}
I.~Shimada, \emph{On the cylinder isomorphism associated to the family of lines
  on a hypersurface}, J. Fac. Sci. Univ. Tokyo Sect. IA Math. \textbf{37}
  (1990), no.~3, 703--719.

\bibitem[Shi92]{Shimada:Supercuspidal}
\bysame, \emph{On supercuspidal families of curves on a surface in positive
  characteristic}, Math. Ann. \textbf{292} (1992), no.~4, 645--669.

\bibitem[Shi01]{Shimada:Lattices}
\bysame, \emph{Lattices of algebraic cycles on {F}ermat varieties in positive
  characteristics}, Proc. London Math. Soc. (3) \textbf{82} (2001), no.~1,
  131--172.

\bibitem[Sim91]{Simpson}
C.~T. Simpson, \emph{Nonabelian {H}odge theory}, Proceedings of the
  {I}nternational {C}ongress of {M}athematicians, {V}ol. {I}, {II} ({K}yoto,
  1990), Math. Soc. Japan, Tokyo, 1991, pp.~747--756.

\bibitem[Sin74]{Singh:Auts}
B.~Singh, \emph{On the group of automorphisms of function field of genus at
  least two}, J. Pure Appl. Algebra \textbf{4} (1974), 205--229.

\bibitem[SK79]{SK:Fermat}
T.~Shioda and T.~Katsura, \emph{On {F}ermat varieties}, Tohoku Math. J. (2)
  \textbf{31} (1979), no.~1, 97--115.

\bibitem[{Stacks}]{stacks-project}
The {Stacks Project Authors}, \emph{\textit{Stacks Project}},
  {\small\url{https://stacks.math.columbia.edu}}.

\bibitem[Ste63]{Steinberg}
R.~Steinberg, \emph{Representations of algebraic groups}, Nagoya Math. J.
  \textbf{22} (1963), 33--56.

\bibitem[Ste16]{Steinberg:Lectures}
\bysame, \emph{Lectures on {C}hevalley groups}, University Lecture Series,
  vol.~66, American Mathematical Society, Providence, RI, 2016.

\bibitem[Sti73a]{Stichtenoth:AutI}
H.~Stichtenoth, \emph{\"{U}ber die {A}utomorphismengruppe eines algebraischen
  {F}unktionenk\"{o}rpers von {P}rimzahlcharakteristik. {I}. {E}ine
  {A}bsch\"{a}tzung der {O}rdnung der {A}utomorphismengruppe}, Arch. Math.
  (Basel) \textbf{24} (1973), 527--544.

\bibitem[Sti73b]{Stichtenoth:AutII}
\bysame, \emph{\"{U}ber die {A}utomorphismengruppe eines algebraischen
  {F}unktionenk\"{o}rpers von {P}rimzahlcharakteristik. {II}. {E}in spezieller
  {T}yp von {F}unktionenk\"{o}rpern}, Arch. Math. (Basel) \textbf{24} (1973),
  615--631.

\bibitem[Sti09]{Stichtenoth:Codes}
\bysame, \emph{Algebraic function fields and codes}, second ed., Graduate Texts
  in Mathematics, vol. 254, Springer-Verlag, Berlin, 2009.

\bibitem[Tak91]{Takeda:Cuspidal}
Y.~Takeda, \emph{Fibrations with moving cuspidal singularities}, Nagoya Math.
  J. \textbf{122} (1991), 161--179.

\bibitem[Tat65]{Tate:Conjecture}
J.~T. Tate, \emph{Algebraic cycles and poles of zeta functions}, Arithmetical
  {A}lgebraic {G}eometry ({P}roc. {C}onf. {P}urdue {U}niv., 1963), Harper \&
  Row, New York, 1965, pp.~93--110.

\bibitem[Tha92]{Thas:Hermitian}
J.~A. Thas, \emph{A combinatorial characterization of {H}ermitian curves}, J.
  Algebraic Combin. \textbf{1} (1992), no.~1, 97--102.

\bibitem[Tia15]{Tian}
Z.~Tian, \emph{Separable rational connectedness and stability}, Rational
  points, rational curves, and entire holomorphic curves on projective
  varieties, Contemp. Math., vol. 654, Amer. Math. Soc., Providence, RI, 2015,
  pp.~155--159.

\bibitem[Tir17]{Tironi}
A.~L. Tironi, \emph{Hypersurfaces achieving the {H}omma-{K}im bound}, Finite
  Fields Appl. \textbf{48} (2017), 103--116.

\bibitem[TVN07]{TVN:Codes}
M.~Tsfasman, S.~Vl\u{a}du\c{t}, and D.~Nogin, \emph{Algebraic geometric codes:
  basic notions}, Mathematical Surveys and Monographs, vol. 139, American
  Mathematical Society, Providence, RI, 2007.

\bibitem[Wal56]{Wallace:Duality}
A.~H. Wallace, \emph{Tangency and duality over arbitrary fields}, Proc. London
  Math. Soc. (3) \textbf{6} (1956), 321--342.

\bibitem[Wei49]{Weil:Fermat}
A.~Weil, \emph{Numbers of solutions of equations in finite fields}, Bull. Amer.
  Math. Soc. \textbf{55} (1949), 497--508.

\bibitem[Wei52]{Weil:Grossenchar}
\bysame, \emph{Jacobi sums as ``{G}r\"{o}ssencharaktere''}, Trans. Amer. Math.
  Soc. \textbf{73} (1952), 487--495.

\bibitem[Zhu11]{Zhu}
Y.~Zhu, \emph{Fano hypersurfaces in positive characteristic}, 2011.

\end{thebibliography}
